\documentclass{article}
\usepackage{psfrag}
\usepackage[parfill]{parskip}
\usepackage{float, graphicx}
\usepackage[]{epsfig}
\usepackage{amsmath, amsthm, amssymb,}
\usepackage{epsfig}
\usepackage{verbatim}
\usepackage{multicol}
\usepackage{url}

\usepackage{latexsym}
\usepackage{mathrsfs}
\usepackage[colorlinks, bookmarks=true]{hyperref}
\usepackage{graphicx}
\usepackage{amsmath}
\usepackage{enumerate}
\usepackage[normalem]{ulem}
\usepackage{bm}
\usepackage{dsfont}
\usepackage{stmaryrd}
\usepackage{enumitem}
\usepackage[noadjust]{cite}

\usepackage{indentfirst}
\usepackage{color}
\usepackage{xcolor}
\usepackage{hyperref}
\usepackage{mathtools}
\usepackage[margin=1.1in]{geometry}
\usepackage{cleveref}

\makeatletter

\def\@settitle{\begin{center}%
    \bfseries
 \normalfont\LARGE\@title
  \end{center}%
}
\def\@setauthors{\begin{center}%
 \normalsize\@author
  \end{center}%
}

\makeatother

\numberwithin{equation}{section}

\renewcommand{\cal}{\mathcal}
\newcommand\cA{{\mathcal A}}
\newcommand\cB{{\mathcal B}}

\newcommand{\cC}{{\cal C}}
\newcommand{\cD}{{\cal D}}
\newcommand{\cZ}{{\cal Z}}
\newcommand{\cE}{{\cal E}}
\newcommand{\cF}{{\cal F}}
\newcommand{\cG}{{\cal G}}

\newcommand\cH{{\mathcal H}}
\newcommand{\cK}{{\cal K}}

\newcommand{\cM}{{\cal M}}
\newcommand{\cN}{{\cal N}}

\newcommand{\cP}{{\cal P}}

\newcommand{\cR}{{\mathcal R}}
\newcommand{\cS}{{\mathcal S}}
\newcommand{\cU}{{\mathcal U}}
\newcommand{\cV}{{\mathcal V}}
\newcommand\cW{{\mathcal W}}
\newcommand{\cX}{{\mathcal X}}
\newcommand{\cY}{{\mathcal Y}}

\newcommand{\sfJ}{{\mathsf J}}
\newcommand{\sfA}{{\mathsf A}}

\newcommand{\bfS}{{\bf S}}

\newcommand{\fa}{{\mathfrak a}}
\newcommand{\fb}{{\mathfrak b}}
\newcommand{\fc}{{\mathfrak c}}

\newcommand{\fo}{{\mathfrak o}}

\newcommand{\fp}{{\mathfrak p}}
\newcommand{\fq}{{\mathfrak q}}

\newcommand{\fg}{{\mathfrak g}}

\newcommand{\bmh}{{\bm{h}}}

\newcommand{\bmr}{{\bm{r}}}

\newcommand{\bmu}{{\bm{u}}}

\newcommand{\bmx}{{\bm{x}}}

\newcommand{\bmz} {{\bm {z}}}

\newcommand{\bms}{{\bm s}}

\newcommand{\bfi}{{\bf i}}
\newcommand{\fC}{{\mathfrak C}}

\newcommand{\be}{\begin{equation}}
\newcommand{\ee}{\end{equation}}

\newcommand{\rd}{{\rm d}}

\newcommand{\ri}{\mathrm{i}}

\newcommand{\bC}{{\mathbb C}}
\newcommand{\bE}{\mathbb{E}}
\newcommand{\bH}{\mathbb{H}}

\newcommand{\bP}{\mathbb{P}}

\newcommand{\bR}{{\mathbb R}}
\newcommand{\bX}{{\mathbb X}}

\newcommand{\bT}{\mathbb T}
\newcommand{\bV}{\mathbb V}
\newcommand{\bW}{\mathbb W}

\newcommand{\bZ}{\mathbb{Z}}

\newcommand{\bI}{\mathbb{I}}

\newcommand{\al}{\alpha}

\newcommand{\la}{\lambda}

\DeclareMathOperator{\Tr}{Tr}

\DeclareMathOperator{\dist}{dist}

\DeclareMathOperator{\Adm}{{Adm}}

\DeclareMathOperator{\Av}{\mathsf{Av}}

\DeclareMathOperator{\OO}{O}
\DeclareMathOperator{\oo}{o}
\DeclareMathOperator{\spec}{spec}

\renewcommand{\Re}{\mathop{\mathrm{Re}}}
\renewcommand{\Im}{\mathop{\mathrm{Im}}}

\newcommand{\ft}{\mathfrak t}
\newcommand{\fX}{{\mathfrak X}}
\newcommand{\deq}{\mathrel{\mathop:}=} 
 
\renewcommand{\leq}{\leqslant}
\renewcommand{\geq}{\geqslant}

\newcommand{\del}{\partial}

\newcommand{\wh}{\widehat}
\newcommand{\wt}{\widetilde}

\newcommand{\qq}[1]{[\![{#1}]\!]}

\newcommand{\beq}{\begin{equation}}
\newcommand{\eeq}{\end{equation}}

\newcommand{\h}[1]{\{{#1}\}}

\theoremstyle{plain} 
\newtheorem{theorem}{Theorem}[section]
\newtheorem*{theorem*}{Theorem}
\newtheorem{lemma}[theorem]{Lemma}
\newtheorem*{lemma*}{Lemma}
\newtheorem{corollary}[theorem]{Corollary}
\newtheorem*{corollary*}{Corollary}
\newtheorem{proposition}[theorem]{Proposition}
\newtheorem*{proposition*}{Proposition}

\newtheorem*{assumption*}{Assumption}

\newtheorem{definition}[theorem]{Definition}
\newtheorem*{definition*}{Definition}

\newtheorem*{example*}{Example}
\newtheorem{remark}[theorem]{Remark}

\newtheorem*{remark*}{Remark}
\newtheorem*{remarks*}{Remarks}

\newcommand{\col}{\vcentcolon}

\newcommand{\rhosc}{\rho_{\mathrm{sc}}}

\newcommand{\msc}{m_{\rm sc}}
\newcommand{\md}{m_d}

\newcommand{\sfS}{{\sf S}}

\newcommand{\sfF}{{\sf F}}

\def\author#1{\par
    {\centering{\authorfont#1}\par\vspace*{0.05in}}
}

\def\titlefont{\fontsize{13}{15}\bfseries\boldmath\selectfont\centering{}}
\def\authorfont{\fontsize{13}{15}}

\let\affiliationfont\rhfont

\def\address#1{\par
    {\centering{\affiliationfont#1\par}}\par\vspace*{11pt}
}

\def\body{
\setcounter{footnote}{0}
\def\thefootnote{\alph{footnote}}
\def\@makefnmark{{$^{\rm \@thefnmark}$}}
}

\def\title#1{
    \thispagestyle{plain}
    \vspace*{-14pt}
    \vskip 79pt
    {\centering{\titlefont #1\par}}%
    \vskip 1em
}

\def\rhosc{\rho_{\text{sc}}}

\usepackage{scalerel}
\newcommand{\cT}{{\mathcal T}}

\definecolor{forestgreen}{RGB}{34, 139, 34}

\newcommand{\bld}[1]{\boldsymbol{\mathrm{#1}}} 
\newcommand{\pbb}[1]{\biggl(#1\biggr)}

\newcommand{\GG}{{\mathcal G}}
\newcommand{\fR}{{\mathfrak R}}
\newcommand{\tG}{{\widetilde G}}
\newcommand{\tL}{{\widetilde L}}
\newcommand{\tcG}{{\widetilde \cG}}
\newcommand{\oOmega}{\overline{\Omega}}

\begin{document}

\title{Ramanujan Property and Edge Universality of Random Regular Graphs}

\vspace{1.2cm}

\noindent \begin{minipage}[c]{0.34\textwidth}
 \author{Jiaoyang Huang}
\address{University of Pennsylvania\\
   huangjy@wharton.upenn.edu}
 \end{minipage}
  \begin{minipage}[c]{0.32\textwidth}
 \author{Theo McKenzie}
\address{Stanford University\\
   theom@stanford.edu }
 \end{minipage}
\begin{minipage}[c]{0.32\textwidth}
 \author{Horng-Tzer Yau}
\address{Harvard University \\
   htyau@math.harvard.edu}

 \end{minipage}

\begin{abstract}
 We consider the normalized adjacency matrix of a random $d$-regular graph on $N$ vertices with any fixed degree $d\geq 3$ and denote its eigenvalues as $\lambda_1=d/\sqrt{d-1}\geq \la_2\geq\la_3\cdots\geq \la_N$. We establish the following two results as $N\rightarrow \infty$. (i) With high probability, all eigenvalues are optimally rigid, up to an additional $N^{\oo(1)}$ factor. 
Specifically, the fluctuations of bulk eigenvalues are bounded by $N^{-1+\oo(1)}$, and the fluctuations of edge eigenvalues are bounded by $N^{-2/3+\oo(1)}$. 
    (ii) Edge universality holds for random $d$-regular graphs. That is, the distributions of $\lambda_2$ and $-\lambda_N$ converge to the Tracy-Widom$_1$ distribution associated with the Gaussian Orthogonal Ensemble. 
   As a consequence, for sufficiently large $N$, approximately $69\%$ of $d$-regular graphs on $N$ vertices
are Ramanujan, meaning $\max\{\lambda_2,|\lambda_N|\}\leq 2$.
\end{abstract}

\bigskip
{
\hypersetup{linkcolor=black}
\setcounter{tocdepth}{1}
\tableofcontents
}

\newpage
\section{Introduction}\label{s:intro}

 Studying the spectrum of adjacency matrices of $d$-regular graphs is a fundamental problem in analysis and combinatorics, with numerous applications to computer science, statistical physics, and number theory \cite{alon1986eigenvalues,chung1997spectral, colin1998spectres, sarnak1990some}.
Equivalently, we can consider the normalized adjacency matrix denoted as $H$,
\begin{equation} \label{def_H}
H\deq \frac{A}{\sqrt{d-1}}.
\end{equation} Normalizing in this manner facilitates comparisons across different values of $d$ and with other random matrix ensembles. The eigenvalues of $H$ are denoted as $\lambda_1 \geq \lambda_2 \geq \cdots \geq \lambda_N$.

According to the Perron-Frobenius theorem, as all vertices have degree $d$, $\lambda_1=d/\sqrt{d-1}$ is always the largest modulus eigenvalue. The Alon-Boppana bound \cite{nilli1991second} asserts that the second largest eigenvalue $\lambda_2$ of \emph{any} infinite family of $d$-regular graphs satisfies $\lambda_2\geq 2-\oo(1)$ as the size of graph goes to infinity, as $2$ is the spectral radius of the adjacency operator on the infinite tree; see \cite{kesten1959symmetric}.
The precise location of the second largest eigenvalue is a crucial statistic with applications in computer science and number theory. For example, it governs the mixing time of the associated random walk \cite{hoory2006expander}.
Lubotzky, Phillips, and Sarnak introduced the concept of a \emph{Ramanujan graph}---a graph for which, besides the ``trivial'' eigenvalues $\lambda_1$ and, if the graph is bipartite, $\lambda_N=-\lambda_1$, all others lie in the interval $[-2,2]$
\cite{lubotzky1988ramanujan}. By the Alon-Boppana bound, this is the smallest possible interval containing these eigenvalues. 
Lubotzky, Phillips, and Sarnak \cite{lubotzky1988ramanujan} and Margulis \cite{margulis1988explicit} independently constructed infinite families of Ramanujan graphs with a fixed degree $d\geq 3$, where $d-1$ is prime. These constructions are Cayley graphs built using advanced algebraic techniques. 

A small number of other algebraic constructions have followed since then \cite{jordan1997ramanujan,morgenstern1994existence,pizer1990ramanujan,chiu1992cubic}, all of which apply only when $d-1$ is a prime or a prime power. Proving the existence of infinite families of Ramanujan graphs for any degree $d\geq 3$ has been a major open problem (see \cite[Problem 10.7.3]{lubotzky1994discrete} and \cite{sarnak2004expander}). Marcus, Spielman, and Srivastava fully solved the existence problem for \emph{bipartite} graphs, showing that for any $d\geq 3$ and $N$, there exists a bipartite $d$-regular graph of size $N$ with $\lambda_1=-\lambda_N=d/\sqrt{d-1}$, and $\lambda_2,|\lambda_{N-1}|\leq 2$ \cite{marcus2013interlacing,marcus2018interlacing}. Instead of using Cayley graphs, their approach involves analyzing families of expected characteristic polynomials associated with these graphs. However, the non-bipartite version of this question has remained open.  It is worth noting that any non-bipartite Ramanujan graph immediately yields a bipartite one, since the bipartite double cover of a $d$-regular non-bipartite Ramanujan graph is a $d$-regular bipartite Ramanujan graph.

In this paper, we resolve this question, proving the existence of infinite families of  non-bipartite Ramanujan graphs for any degree $d\geq 3$. We accomplish this by precisely characterizing the distributions of $\lambda_2$ and $\lambda_N$ within the \emph{random regular graph ensemble}, where a graph is chosen uniformly at random from all simple $d$-regular graphs on $N$ vertices.

Studying the typical spectrum of graphs sampled from the random regular graph ensemble is  a well-established and active area of research. Particularly, there are applications in constructing optimally expanding networks \cite{hoory2006expander}, analyzing graph $\zeta$-functions \cite{terras2010zeta}, and studying quantum chaos \cite{smilansky2013discrete}.
Alon conjectured that random regular graphs are ``almost Ramanujan'', meaning that with high probability over the random regular graph distribution, $\lambda_2\leq 2+\oo(1)$ \cite{alon1986eigenvalues}. Friedman proved this conjecture   \cite{friedman2008proof}, and later Bordenave gave an alternative proof \cite{Bordenave2015ANP}. Both proofs rely on bounding high moments of the adjacency matrix by classifying walks of different lengths and cycle types. Specifically, they show that $\lambda_2\leq 2+\OO((\log \log N/ \log N)^2)$.
In \cite{huang2024spectrum}, the first and third authors of the present paper, building upon joint work with Bauerschmidt \cite{bauerschmidt2019local}, presented a new proof of Friedman's Theorem with a polynomially small error, demonstrating $\lambda_2\leq 2+\OO(N^{-c})$, for some small constant $c>0$. This proof is based on a careful analysis of the Green's function and more generally establishes concentration of all eigenvalues, not just $\lambda_2$. More recently, Chen, Garza-Vargas, Tropp and van Handel introduced a new approach to strong convergence of random matrices, providing an alternative proof of Friedman's Theorem with $\lambda_2\leq 2+\OO((\log N/N)^{1/6})$ \cite{chen2024new,chen2024new2}.

 Given that random regular graphs are typically almost Ramanujan, it has long been suggested that analyzing their typical behavior could produce examples of Ramanujan graphs \cite[Problem 10.7.4]{lubotzky1994discrete}.
Numerical simulations have appeared to show that the distribution of the second largest eigenvalue of random $d$-regular graphs after normalizing is the
    same as that of the largest eigenvalue of a matrix sampled from the Gaussian Orthogonal Ensemble (GOE).
    Sarnak \cite{ sarnak2004expander} and Miller, Novikoff, and Sabelli \cite{miller2008distribution} conjectured that, for every $d\geq 3$, there exist constants $C_{1}^d$ and $C_{2}^d$ such that  $C_1^{d} N^{2/3} (\lambda_2-2)-C_2^{d} $ converges to the Tracy-Widom$_1$ distribution associated with the GOE (see \cite{tracy1996orthogonal}).
    A similar statement holds for $\lambda_N$. This would imply that the fluctuations of extreme eigenvalues are of order $N^{-2/3+\oo(1)}$, and a positive portion of  all $d$-regular graphs are Ramanujan.

In this work, we establish this edge universality conjecture for random $d$-regular graphs for every $d\geq 3$, with the numerical constants $C_1^d=(d(d-1)/(d-2)^2)^{2/3}$ and $C_2^d=0$. A direct consequence is that for sufficiently large $N$, approximately $69\%$ of $d$-regular graphs on $N$ vertices are Ramanujan. As an intermediate step, we derive an optimal bound (up to an additional $N^{\oo(1)}$ factor) on the fluctuations of all eigenvalues. Specifically, the fluctuations of bulk eigenvalues are bounded by $N^{-1+\oo(1)}$, while those of edge eigenvalues are bounded by $N^{-2/3+\oo(1)}$.

\subsection{Main Results}
Before stating our main results, we introduce some necessary notation. By local weak convergence, the empirical eigenvalue density of random $d$-regular graphs converges to that of the infinite $d$-regular tree, which is known as the Kesten-McKay distribution; see \cite{kesten1959symmetric,mckay1981expected}. This density is given by 
\begin{align}\label{e:KMdistribution}
\varrho_d(x):=\mathbf1_{x\in [-2,2]} \left(1+\frac1{d-1}-\frac {x^2}d\right)^{-1}\frac{\sqrt{4-x^2}}{2\pi}.
\end{align}
Note that close to the spectral edge $\pm 2$, the Kesten-Mckay distribution has square root behavior:
\begin{align}\label{e:edge_behavior}
x\rightarrow\pm 2,\quad \varrho_d(x)=\frac{\cA\sqrt{2\mp x}}{\pi} +\OO(|2\mp x|), \quad \cA:=\frac{d(d-1)}{(d-2)^2}.
\end{align}
We denote by $\md(z)$ the Stieltjes transform of the Kesten--McKay distribution $\varrho_d(x)$,
\begin{align*}
    \md(z)=\int_\bR \frac{\varrho_d(x)\rd x}{x-z},\quad z\in \bC^+:=\{w\in \bC \col \Im[w]>0\}.
\end{align*}
We recall the semicircle distribution $\varrho_{\rm sc}(x)$ and its Stieltjes transform $\msc(z)$:
\begin{align}\begin{split}\label{e:msc_equation}
 \varrho_{\rm sc}(x)=\bm1_{x\in[-2,2]}\frac{\sqrt{4-x^2}}{2\pi},
 \quad 
 \msc(z)=\int_\bR \frac{\varrho_{\rm sc}(x)\rd x}{x-z}=\frac{-z+\sqrt{z^2-4}}{2}.
\end{split}\end{align}
Explicitly, the Stieltjes transform of the Kesten--McKay distribution $\md(z)$ can be expressed in terms of the Stieltjes transform $\msc(z)$:
\begin{align}\label{e:md_equation}
    \md(z)=\frac{1}{-z-\frac{d}{d-1}\msc(z)},\quad \md(z)=(d-1)\frac{-(d-2)z+ d\sqrt{z^2-4}}{2(d^2-(d-1)z^2)}.
\end{align}
 
For $2\leq i \leq N$, we expect the $i$-th largest eigenvalue of the normalized adjacency matrix $H$ to be close to the classical eigenvalue locations $\gamma_i$, where $\gamma_i$ is defined so that 
\be\label{eq:gammadef}
\int_{\gamma_i}^2 \varrho_d(x)\rd x=\frac{i-1/2}{N-1},\quad 2\leq i\leq N.
\ee
Our first result gives optimal concentration for each eigenvalue of random $d$-regular graphs.

\begin{theorem}\label{thm:eigrigidity}
Fix $d\geq 3$, and an arbitrarily small $\fa>0$. There exists a positive integer $\omega_d \geq 1$, depending only on $d$, such that with probability $1 - N^{-(1-\fa)\omega_d}$, the eigenvalues $\lambda_1 = {d}/{\sqrt{d-1}} \geq \lambda_2 \geq \cdots \geq \lambda_N$ of the normalized adjacency matrix $H$ of random $d$-regular graphs satisfy:
\begin{align}\label{e:optimal_rigidity}
|\lambda_i-\gamma_i|\leq N^{-2/3+\fa}(\min\{i,N-i+1\})^{-1/3},
\end{align}
for every $2\leq i\leq N$ and $\gamma_i$ are the classical eigenvalue locations, as defined in \eqref{eq:gammadef}.
\end{theorem}

The aforementioned theorem implies that $\lambda_2,|\lambda_N|=2+N^{-2/3+\oo(1)}$. In the bulk, for fixed $\delta>0$ and $\delta N\leq i\leq (1-\delta)N$, the fluctuation of $\lambda_i$ is of order $N^{-1+\oo(1)}$. This level of fluctuation aligns with that observed for the eigenvalues of matrices from the GOE, and more broadly, Wigner matrices; see \cite{erdHos2012rigidity}.

Conditioned on the event 
$\oOmega$, which ensures that the $d$-regular graphs are locally tree-like (see \Cref{def:omegabar} for a precise definition), our proof of \Cref{thm:eigrigidity} implies that the optimal rigidity estimates \eqref{e:optimal_rigidity} hold with probability at least $1-\OO(N^{-\fC})$ for any $\fC>0$, provided $N$ is sufficiently large.

Our second result verifies the edge universality conjecture for random $d$-regular graphs by Sarnak \cite{sarnak2004expander} and Miller, Novikoff and Sabelli \cite{miller2008distribution}.
\begin{theorem}
\label{t:universality}

Fix $d\geq 3$, $k\geq 1$ and $s_1,s_2,\cdots, s_k \in \bR$, and recall $\cA=d(d-1)/(d-2)^2$ from \eqref{e:edge_behavior}. There exists a small $\fa>0$ such that the eigenvalues $\lambda_1 = {d}/{\sqrt{d-1}} \geq \lambda_2 \geq \cdots \geq \lambda_N$ of the normalized adjacency matrix $H$ of random $d$-regular graphs satisfy: 
\begin{align*}\begin{split}
 \bP_{H}\left( (\cA N)^{2/3} ( \lambda_{i+1} - 2 )\geq s_i,1\leq i\leq k \right)= \bP_{\mathrm{GOE}}\left( N^{2/3} ( \mu_i - 2  )\geq s_i,1\leq i\leq k \right) +\OO(N^{-\fa}),
\end{split}\end{align*}
where $\mu_1\geq \mu_2\geq \cdots \geq\mu_N$ are the eigenvalues of the GOE.
The analogous statement holds for the smallest eigenvalues $-\lambda_N,\ldots,-\lambda_{N-k+1}$.
\end{theorem}

When the degree $d$ grows with the size of the graph, edge universality for random $d$-regular graphs has been established previously for $N^{2/3 + \oo(1)} \leq d \leq N/2$, by He \cite{he2022spectral}, and for  $N^{\oo(1)} \leq d \leq N^{1/3 - \oo(1)}$, by the first and third authors of this paper \cite{huang2023edge}, which generalized a result for $N^{2/9 + \oo(1)} \leq d \leq N^{1/3 - \oo(1)}$, by Bauerschmidt, Knowles and the first and third authors of this paper \cite{bauerschmidt2020edge}.

As emphasized in \cite{miller2008distribution,sarnak2004expander},
the Tracy--Widom$_1$ distribution has  about $83\%$ of its mass on the set $\{x:x<0\}$.
Therefore \Cref{t:universality} implies $83\%$ of
$d$-regular graphs have the second eigenvalue \emph{less than} $2$.
The proof of \Cref{t:universality} can be extended to show 
that  the largest and smallest nontrivial eigenvalues converge in distribution to \emph{independent} Tracy--Widom$_1$ distributions; see Remarks \ref{rem:independence} and \ref{rem:ind2} below.
As a consequence, we have the following result.

\begin{corollary}\label{c:rate}
Fix $d\geq 3$ and $N$ sufficiently large.  With probability approximately $69\%$, a randomly sampled $d$-regular graph has $\max\{\lambda_2,|\lambda_N|\}\leq 2$, and is therefore Ramanujan. 
\end{corollary}

\subsection{Proof ideas}

We define the {Green's function} and the Stieltjes transform of the empirical eigenvalue distribution of the normalized adjacency matrix $H$ of a $d$-regular graph $\GG$ on $N$ vertices as
\begin{align*}\label{e:m}
  G(z) \deq  (H-z)^{-1},\quad   m_N(z) \deq \frac{1}{N} \sum_i G_{ii}(z)=\frac1N\sum_{i=1}^N\frac{1}{\la_i-z},\quad z\in \bC^+.
\end{align*}

We establish \Cref{thm:eigrigidity} and \Cref{t:universality} through a detailed analysis of the Green's function and its variations, a method that has been highly effective in studying the spectral properties of random matrices (see \cite{erdHos2017dynamical} for an overview of this approach). In the remainder of this section, we outline the key ideas behind the proofs of our main results.


\subsubsection{Loop Equations}
The loop equation, also known as the Dyson-Schwinger equation, is a fundamental tool for studying both macroscopic and microscopic properties of interacting particle systems, particularly in the context of random matrix theory and $\beta$-ensembles (see \cite{guionnet2019asymptotics} for an overview). At its core, the loop equation describes a recursive structure satisfied by the correlation functions of the system. Specifically, let $s_N(z)$ denote the Stieltjes transform of the empirical eigenvalue distribution of the GOE matrix of $N$, defined for $z\in \bC^+$. The first order loop equation associated with the GOE matrix is given by
\begin{align}
\bE\left[s_N^2(z)+zs_N(z)+1+\frac{\del_z s_N(z)}{N}\right]=0.
\end{align}
More generally, the $p$-th order loop equations state that for $z, z_1, z_2,\cdots, z_{p-1}\in \bC^+$
\begin{align}\begin{split}\label{e:loop0}
&\phantom{{}={}}\bE\left[\left(s_N(z)^2+zs_N(z)+1+\frac{\del_z s_N(z)}{N}\right)\prod_{j=1}^{ p-1}s_N(z_j)\right]+\frac{2}{N^2}\sum_{j=1}^{p-1}
    \bE\left[\del_{z_j} \frac{s_N(z)-s_N(z_j)}{z-z_j}\prod_{ i\neq j}s_N(z_i)\right]=0.
\end{split}\end{align}

On the macroscopic scale, the loop equation is crucial in deriving global properties of the system, such as the density of states and global fluctuations. For instance, loop equations have been employed to establish macroscopic central limit theorems for $\beta$-ensembles \cite{johansson1998fluctuations, borot2024asymptotic, borot2013asymptotic, shcherbina2013fluctuations}, demonstrating that fluctuations of the empirical spectral distribution around the equilibrium measure converge to Gaussian limits.

On intermediate (mesoscopic) scales, loop equations play a pivotal role in establishing rigidity. Rigidity refers to the phenomenon where particles (eigenvalues, in the context of random matrix ensembles) exhibit highly deterministic behavior, staying close to their classical locations as predicted by the equilibrium measure. Loop equations provide a robust framework for proving optimal rigidity. For the GOE matrix, loop equations \eqref{e:loop0} can be combined to estimate the higher moments of the self-consistent equation $s_N^2(z)+zs_N(z)+1$:
\begin{align}
   \bE[|s_N^2(z)+zs_N(z)+1|^{2p}]\lesssim \bE\left[\left(\frac{\Im[s_N(z)](\Im[s_N(z)]+|z+2s_N(z)|)}{(N\eta)^2}\right)^{p}\right]. 
\end{align}
Then the Markov's inequality together with a continuity or bootstrapping argument leads to an optimal estimate for $|s_N(z)-\msc(z)|$.
More generally, we refer the reader to results for $\beta$-ensembles \cite{bourgade2022optimal,bourgade2014universality,bourgade2012bulk,bourgade2014edge,guionnet2019rigidity} and generalizations of Wigner matrices \cite{lee2018local, he2018isotropic,ajanki2019stability}.

On the microscopic scale, the role of loop equations has only recently started to attract attention. It was observed in \cite{bourgade2022optimal} that local eigenvalue statistics satisfy a microscopic version of the loop equation. In particular, the Airy$_1$ point process, which describes the edge scaling limit of the GOE matrix, satisfies such loop equations. Let $S(w)$ denote the normalized Stieltjes transform of the Airy$_1$ point process, defined for $w\in \bC^+$. The first order loop equation for the Airy$_1$ point process is given by
\begin{align}\label{e:loop2}
    \bE\left[{(S(w)-\sqrt{w})^2+ 2\sqrt{w}(S(w)-\sqrt w) +\del_w S(w)}\right]=0,
\end{align}
which can be derived from \eqref{e:m} by examining a window of size $N^{-2/3}$ around the spectral edge. Specifically, $S(w)-\sqrt w=\lim_{N\rightarrow \infty}N^{1/3} (s_N(2+w N^{-2/3})-\msc(2+wN^{-2/3}))$,
where  $\msc$ denotes the Stieltjes transform of the semicircle law.


In this work, we show that random $d$-regular graphs satisfy the loop equations on the microscopic scale at the edge. As a consequence, we establish optimal eigenvalue rigidity. 
While satisfying the loop equations is a necessary condition for Airy statistics, it is unclear whether the loop equations \emph{uniquely} characterize the Airy$_1$ point process. 
To derive Airy statistics, we follow the three-step strategy developed in \cite{erdHos2010bulk,erdHos2012rigidity,erdHos2011universality,erdHos2017dynamical} and analyze the normalized adjacency matrix with a small Gaussian component $H(t)=H+\sqrt t Z$, where $Z$ is a GOE matrix constrained to have vanishing row and column sums.
Using eigenvalue rigidity as input, it was proven in \cite{landon2017edge,adhikari2020dyson} that for 
$t\geq N^{-1/3+\oo(1)}$, edge universality holds for 
$H+\sqrt t Z$.
A key observation is that the time derivative of the multi-point correlation functions of the Stieltjes transform  of $H(t)$ is \emph{precisely governed by the microscopic loop equations}, see \eqref{e:DBM_mt3}. 
We prove that the eigenvalue distributions of 
$H+\sqrt{t} Z$ satisfy the microscopic loop equations at the edge for small $t$, see \eqref{e:Qrefined_bound} and \eqref{e:key_cancel}. 
 These equations ensure that the multi-point correlation functions of the edge statistics remain invariant as $t$ evolves. 
 In particular, by setting $t\asymp N^{-1/3+\oo(1)}$, we conclude that $H$ and $H+\sqrt t Z$ have the same edge statistics, both of which are described by the Airy$_1$ point process. 

In our proof, the loop equation serves a dual purpose: it establishes both the rigidity of eigenvalue locations and edge universality. This novel perspective opens up possibilities for extending the proof of edge universality to other random matrix ensembles and related systems.

\subsubsection{Self-consistent equation}
 
To establish Theorem \ref{thm:eigrigidity}, we start with the self-consistent equation of a modified Green's function quantity, as introduced in \cite{bauerschmidt2019local, huang2024spectrum}. Given a $d$-regular graph $\GG$ on $N$ vertices, let $\cG^{(i)}$ denote the graph obtained by removing vertex $i$, and let  $G^{(i)}(z)$ represent its Green's function. We  consider the average of $G_{jj}^{(i)}(z)$ over all pairs of adjacent vertices $i\sim j$:
\begin{align}\label{e:Qsum0}
Q(z):=\frac{1}{Nd}\sum_{i\sim j}G_{jj}^{(i)}(z).
\end{align}
For $d$-regular graphs, while $Q(z)$ is slightly more intricate than the Stieltjes transform of the empirical eigenvalue distribution $m_N(z)$, it proves to be a more useful quantity. One key reason is that $Q(z)$ approximately satisfies a \emph{self-consistent equation}, which is exactly satisfied by $\msc(z)$, the Stieltjes transform of the semicircle distribution. 
To approximate the Green's function $G_{jj}^{(i)}(z)$ in \eqref{e:Qsum0}, we consider a  neighborhood of vertex $j$ in $\GG^{(i)}$ with a large radius $\ell\asymp\log_{d-1} N$. The Schur complement formula allows us to express $G^{(i)}_{jj}(z)$ as the Green's function of the radius-$\ell$ neighborhood, with weights added to boundary vertices. These boundary weights, derived  from the Green's function of the graph with the radius-$\ell$ neighborhood removed, can be approximated by the averaged quantity $Q(z)$. For almost all vertices, this radius-$\ell$ neighborhood is a depth-$\ell$ $(d-1)$-ary tree (a tree where the root has degree $d-1$, and all other vertices have degree $d$). In this case, the relevant function is $Y_\ell(Q(z), z)$, which is the Green's function at the root vertex of a truncated $(d-1)$-ary tree of depth $\ell$, with boundary weights $Q(z)$. This structure gives rise to the following self-consistent equation for 
$Q(z)$:
\begin{align}\label{e:selfeq}
Q(z)=\frac{1}{Nd}\sum_{i\sim j}G_{jj}^{(i)}(z)\approx Y_\ell(Q(z),z).
\end{align}

 The fixed point of the above self-consistent equation \eqref{e:selfeq} is given by the Stieltjes transform of the semicircle distribution $\msc(z)$. Indeed, $Y_\ell(\msc(z),z)=\msc(z)$ represents the Green's function at the root vertex of an infinite $(d-1)$-ary tree, whose spectral density is governed by the semicircle distribution. Consequently, the self-consistent equation \eqref{e:selfeq} was utilized in \cite{bauerschmidt2019local, huang2024spectrum} to demonstrate that $Q(z)$ is closely approximated by $\msc(z)$ with high probability:
\begin{align*}
  Q(z)\approx \msc(z).
\end{align*}
The Stieljes transform $m_N(z)$ of the empirical eigenvalue distribution of $H$ can also be recovered from the quantity $Q(z)$, through the following approximation
\begin{align}\label{e:mz_approx}
    m_N(z)=\frac{1}{N}\sum_{i=1}^N G_{ii}(z)\approx X_\ell(Q(z),z),
\end{align}
where $X_\ell(Q(z),z)$ denotes the Green's function at the root vertex of a truncated  $d$-regular tree of depth $\ell$, with boundary weights $Q(z)$.

Utilizing the self-consistent equations \eqref{e:selfeq} and \eqref{e:mz_approx}, \cite{bauerschmidt2019local, huang2024spectrum} established that with high probability, uniformly for any $z$ in the upper half-plane with $\Im[z]\geq N^{-1+\oo(1)}$, the Stieltjes transform of the empirical eigenvalue distribution closely approximates $\md(z)$:
\begin{align}\label{e:weak}
|m_N(z)-\md(z)|\leq N^{-\fb},
\end{align}
for some small $\fb>0$ and that eigenvectors are completely delocalized. Achieving optimal rigidity of eigenvalue locations, as stated in Theorem \ref{thm:eigrigidity}, necessitates an optimal error bound much stronger than $N^{-\fb}$ in \eqref{e:weak}.  This constitutes a standard problem in random matrix theory, which requires an optimal error bound for the high moments of the self-consistent equation:
\begin{align}\label{e:hmm}
\bE[|Q(z)-Y_\ell(Q(z),z)|^{2p}],
\end{align}
 for large integers $p>0$. Analogous estimates have been established for Wigner matrices \cite{erdHos2012rigidity,erdHos2013local, he2018isotropic}, Erd\H{o}s-R\'enyi graphs \cite{lee2018local,erdHos2013spectral,erdHos2012spectral, he2021fluctuations,huang2020transition,huang2022edge,lee2024higher}, $d$-regular graphs with growing degrees \cite{bauerschmidt2020edge, huang2023edge, he2022spectral}, and $\beta$-ensembles \cite{bourgade2022optimal}.

Unlike Wigner matrices, the primary challenge in analyzing the adjacency matrices of random $d$-regular graphs lies in the correlations between matrix entries, where row and column sums are fixed at $d$.  To address this constraint, the technique of local resampling has proven essential. This method was first applied to derive spectral statistics for random  $d$-regular graphs with $d\gg \log N$ in \cite{bauerschmidt2017local}, and was later extended in \cite{bauerschmidt2019local, huang2024spectrum} to establish the self-consistent equation \eqref{e:selfeq} and bound its high moments \eqref{e:hmm}. In this approach, local resampling randomizes the boundary of a radius-$\ell$ neighborhood, rather than randomizing edges near a vertex as in  \cite{bauerschmidt2017local}. Specifically, it achieves this by exchanging the edge boundary of a radius-$\ell$ neighborhood with randomly chosen edges elsewhere in the graph. A notable feature of this resampling method is its reversibility, meaning the law governing the random regular graph and its switched counterpart is exchangeable.

However, the error bound for the high moments of the self-consistent equation \eqref{e:hmm} provided in \cite{huang2024spectrum} is suboptimal, as are the results concerning eigenvalue rigidity. One of the primary contributions of this work is to establish an \emph{optimal error bound for the high moments of the self-consistent equation} \eqref{e:hmm}.

The self-consistent equation \eqref{e:selfeq} represents the leading order term in the loop equations. To derive the complete loop equation, it is essential to account for the next-order term. We establish that for $z$ near the spectral edge $\pm 2$, the following estimate holds 
\begin{align}\label{e:next_order}
\bE\left[\frac{\cA^2}{(\ell+1)}\left(Q(z)-Y_\ell(Q(z),z)\right) +\frac{\del_z m_N(z)}{N}\right]\ll N^{-2/3}.
\end{align}
In combination with a similar equation for $m_N(z)-X_\ell(Q(z),z)$, this allows us to rewrite the above expression in terms of the Stieltjes transform of the empirical eigenvalue distribution $m_N(z)$
\begin{align}\label{e:loop3}
\bE\left[(m_N(z)-\md(z)) ^2+ 2\cA \sqrt{z-2}(m_N(z)-\md(z)) +\frac{\del_z m_N(z)}{N}\right]\ll N^{-2/3}.
\end{align}
The {scaling limit $N^{1/3} (m_N(2+w/(\cA N)^{2/3})-\md(2+w/(\cA N)^{2/3}))/\cA^{2/3}$} converges to $S(w)-\sqrt w$ in the microscopic version of the loop equation \eqref{e:loop2}. For this convergence to hold, the error in \eqref{e:loop3} must be much smaller than $N^{-2/3}$, ensuring that its scaling limit vanishes.

The second contribution of this work is the \emph{identification of the next-order term for the self-consistent equation near the spectral edge}, as described in \eqref{e:next_order}, along with its multi-point analogue in \eqref{e:key_cancel}. These relations correspond to the loop equations at the microscopic scale near the spectral edge for random $d$-regular graphs.

\subsubsection{Iteration scheme}\label{s:newstrategy}

To illustrate the ideas that will lead to an optimal error bound for \eqref{e:hmm} and \eqref{e:next_order},  we begin with the first moment of the self-consistent equation: 
\begin{align}\label{e:sample}
\bE[Q(z)-Y_\ell(Q(z),z)]
    =\bE[(G_{oo}^{(i)}(z)-Y_\ell(Q(z),z)]
    =\bE[\widetilde G_{oo}^{(i)}(z)-Y_\ell(Q(z),z)],
\end{align}
where in the first statement we exploit the permutation invariance of vertices, so the expectation of $Q(z)$ is the same as the expectation of $G_{oo}^{(i)}(z)$ for any pair of adjacent vertices $i\sim o$. Given that almost all 
 vertices have large tree neighborhoods, we can focus on cases where $o$ has a large tree neighborhood.  The second statement arises from the local resampling process. Instead of computing the expectation of $G_{oo}^{(i)}(z)$, we perform a local resampling around vertex $o$, by switching the boundary edges of the radius-$\ell$ neighborhood $\cT=\cB_\ell(o, \cG)$ of vertex $o$, with randomly selected edges $\{(b_\al, c_\al)\}_{\al\in \qq{\mu}}$ from $\cG$, where $\mu\asymp (d-1)^\ell$ is the total number of boundary edges.
 We denote the resulting graph as $\tcG$, and its Green's function as $\widetilde G(z)$. The new boundary vertices of $\cT$ after local resampling is $\{c_\al\}_{\al\in\qq{\mu}}$, which are typically well spaced in terms of graph distance (see Section \ref{s:local_resampling} for further details). As local resampling is reversible, $G_{oo}^{(i)}(z)$ and $\widetilde G_{oo}^{(i)}(z)$ share the same law and expectation, which gives the last expression in \eqref{e:sample}.

To demonstrate the smallness of the final expression in \eqref{e:sample}, we expand  $\widetilde G_{oo}^{(i)}(z)$ using the Schur complement formula.
The radius-$\ell$ neighborhood of $o$ in $\tcG^{(i)}$ (where vertex $i$ is removed) is  $\cT^{(i)}$, a truncated $(d-1)$-ary tree at level $\ell$. 
The Schur complement formula states that $\widetilde G_{oo}^{(i)}(z)$ is the same as the Green's function of $\cT^{(i)}$ with boundary weights given by $\widetilde G_{c_\al c_\beta}^{(\bT)}(z)$, which are the Green's functions of $\tcG^{(\bT)}$ (with the vertex set $\bT$ of $\cT$ removed). With high probability, the new boundary vertices $\{c_\al\}_{\al\in\qq{\mu}}$ are far from each other, and have large tree neighborhoods. Consequently, the neighborhoods of $c_\al$ in $\cG^{(\bT)}$ are given by the truncated $(d-1)$-ary trees, and the boundary weights can be approximated by the Green's function of $\cG$ (the graph before switching) as
\begin{align}\label{e:GtoQ}
   \widetilde G_{c_\al c_\beta}^{(\bT)}(z)
   \approx  G_{c_\al c_\beta}^{(b_\al b_\beta)}(z)\approx \delta_{\al\beta} Q(z),\quad 1\leq \al, \beta\leq \mu.
\end{align}
As a result, the leading-order term of 
 $\widetilde G_{oo}^{(i)}(z)$ is given by $Y_\ell(Q(z),z)$, and \eqref{e:sample} is small. This strategy, utilized in  \cite{huang2024spectrum}, provided a weak bound for \eqref{e:hmm}, by bounding all approximation errors, such as those in\eqref{e:GtoQ}, to be  $N^{-\fb}$ for some small $\fb >0$. 

 To achieve optimal estimates for Green's functions, we need to analyze the approximation errors from the Schur complement formula \eqref{e:GtoQ} more carefully. These errors comprise weighted sums of terms involving factors such as:
 \begin{align}\label{e:error_factor}
     (G^{(b_\al)}_{c_\al c_\al}(z)-Q(z)),  \quad G_{c_\al c_\beta}^{(b_\al b_\beta)}(z),\quad Q(z)-\msc(z).
 \end{align}
For a more precise description of these error terms, we refer to \Cref{lem:diaglem}. Instead of simply bounding each term in \eqref{e:error_factor} by $N^{-\fb}$, we carefully examine all possible error terms. 
The crucial observation  is that the expectation of the first factor in \eqref{e:error_factor} with respect to the randomness of the simple switching data $\bfS$ is very small:
\begin{align*}
\mathbb{E}_{\mathbf{S}}[G^{(b_{\alpha})}_{c_{\alpha} c_{\alpha}}(z)-Q(z)]=\OO\left(\frac{1}{N^{1-\oo(1)}}\right);
\end{align*}
and, up to negligible error, the expectation of the second factor in \eqref{e:error_factor} can be expressed to include the first factor,
\begin{align*}
   \bE_\bfS[G_{c_\al c_\beta}^{(b_\al b_\beta)}(z)]=\frac{1}{d-1}\bE_\bfS\left[G_{b_\al b_\beta}(z) (G_{c_\al c_\al}^{(b_\al)}(z)-Q(z))(G_{c_\beta c_\beta}^{(b_\beta)}(z)-Q(z))\right]+\text{``negligible error"}.
\end{align*}
For more precise statements, we refer the reader to \Cref{p:upbb}.
In this way, we demonstrate that either the error term is negligible or contains one ``diagonal factor" akin to the first factor in \eqref{e:error_factor}. Such a factor, $(G^{(b_\al)}_{c_\al c_\al}(z)-Q(z))$, is in the same form as the initial expression \eqref{e:sample}. To evaluate the expectation, we will perform another local resampling around the vertex $c_\alpha$, chosen randomly from the last local switching step. 

Our new strategy to derive the optimal error bound for \eqref{e:hmm} is an iteration scheme. At each step, we perform local resampling and rewrite the Green's function of the switched graph in terms of the original graph. Next, we show that each term contains at least one ``diagonal factor" $G_{c_\alpha c_\alpha}^{(b_\alpha)}(z)$ where $(b_\alpha, c_\alpha)$ is an edge selected during local resampling; otherwise it is negligible. Then, we can perform a local resampling around $c_\alpha$, and repeat this procedure.
The terms in this iterative scheme are intricate, depending on the vertices involved in local resampling and their relative positions. To encode this information, we introduce a sequence of forests comprising the radius-$\ell$ neighborhoods of the vertices where local resampling is performed, along with all switching edges at each step. For a detailed description of these forests and the terms involved in the iteration scheme, we refer to \Cref{s:forest} and \Cref{t:admissible}. 
Crucially, we demonstrate that each iteration of local resampling introduces an additional multiplicative factor of size at most $(d-1)^{-\ell/2}$. This factor arises from the averaging effect across the $\mu\asymp (d-1)^\ell$ new boundary vertices introduced during local resampling.
After a finite number of steps, all error terms become negligible, leading to an optimal bound for the high moments of the self-consistent equation \eqref{e:hmm}. This iteration is detailed in \Cref{p:add_indicator_function}, \Cref{p:iteration} and \Cref{p:general}.

To derive the full loop equation, it is essential to identify the next-order correction term, specifically  $\del_z m_N(z)/N$ as in \eqref{e:next_order}. This correction arises from two primary sources. The first is the replacement  
$\widetilde G_{c_\al c_\beta}^{(\bT)}(z)- G_{c_\al c_\beta}^{(b_\al b_\beta)}(z)$ as described in \eqref{e:GtoQ}. 
By carefully analyzing the changes in the Green's function induced by local resampling, we demonstrate in \eqref{e:first_term} and \Cref{l:diffG1} that
\begin{align}\label{c1}
   \bE[\widetilde G_{c_\al c_\beta}^{(\bT)}(z)- G_{c_\al c_\beta}^{(b_\al b_\beta)}(z)]=\frac{1}{\cA^2 N}\left(\frac{d(d-1)^\ell}{d-2}-\frac{d}{d-2}\right)\bE[\del_z m_N(z)]+\text{``negligible error"}.
\end{align}
The second source stems from errors introduced by the Schur complement formula as in \eqref{e:error_factor}, which involves two ``off-diagonal factors" $(G_{c_\al c_\beta}^{(b_\al b_\beta)}(z))^2$. We show that the contribution of these terms can also be reduced to $\del_z m_N(z)$, as detailed in \Cref{l:error_term}
\begin{align}\label{c2}
    \bE[(G_{c_\al c_\beta}^{(b_\al b_\beta)}(z))^2]=\frac{1}{\cA^2 N}\bE[\del_z m_N(z)]+\text{``negligible error"}.
\end{align}
By combining all these corrections as in \eqref{c1} and \eqref{c2}, the coefficients cancel out to yield $1$. This results in the correction $\del_z m_N(z)/N$ in \eqref{e:next_order}.

\subsubsection{New Technical Ideas}

Similar to \eqref{e:sample}, we can compute higher moments of the self-consistent equation through local resampling: 
\begin{align}\begin{split}\label{e:sample2}
&\phantom{{}={}}\bE[(Q(z)-Y_\ell(Q(z),z))^{p}]
    =\bE[(G_{oo}^{(i)}(z)-Y_\ell(Q(z),z))(Q(z)-Y_\ell(Q(z),z))^{p-1}]\\
    &=\bE[(\wt G_{oo}^{(i)}(z)-Y_\ell( Q(z),z))(\wt Q(z)-Y_\ell(\wt Q(z),z))^{p-1}] +\text{``negligible error"}.
\end{split}\end{align}
To handle the new expression, we  must rewrite the Green's function of the switched graph in terms of the original graph. While the Schur complement formula suffices for expressions like $(\wt G_{oo}^{(i)}(z)-Y_\ell( Q(z),z))$, for higher moments, as in \eqref{e:sample2}, we also need to carefully track changes of the following form:
\begin{align*}
    ((\wt Q(z)-Y_\ell (\wt Q(z),z))-(Q(z)-Y_\ell(Q(z),z)))=(1-\del_1 Y_\ell(Q(z),z))(\wt Q(z)-Q(z))+\OO(|\wt Q(z)-Q(z)|^2).
\end{align*}
Since $\widetilde H-H$ (the difference of the normalized adjacency matrices of $\cG$ and $\tcG$) is low rank,  one can use the Ward identity \eqref{eq:wardex} to show that with high probability
\begin{align*}
    |\widetilde Q(z)-Q(z)|\leq \frac{N^{\oo(1)}\Im[m_N(z)]}{N\eta}, \quad \eta=\Im[z],
\end{align*}
as was done  in previous work \cite{huang2024spectrum}. However this bound is insufficient for achieving optimal rigidity.

To achieve improved accuracy for $\wt Q(z)-Q(z)$, it is essential  to compute the difference of the Green's functions $\wt G(z)$ and $G(z)$ for the graphs $\tcG$ and $\cG$. The local resampling around vertex $o$ can be represented by a matrix $\xi:=\wt H-H$. Thus, using the resolvent identity, we can write
\begin{align}\label{e:resolvent}
    \wt G(z)-G(z)=-G(z)\xi G(z) -\sum_{k\geq 1}G(z)\xi (-G(z)\xi)^k G(z).
\end{align}
The aforementioned expansion was utilized in \cite{huang2023edge} to prove the edge universality of random $d$-regular graphs, when the degree $d=N^{\oo(1)}$ grows with the size the graph. In this case, owing to our normalization, each entry of $\xi$ scales as $\OO(1/\sqrt{d})$. Therefore, the terms in \eqref{e:resolvent} exhibit exponential decay in $1/\sqrt d$. However, in our scenario where $d$ is fixed, this decay is too slow. 

For $d\geq 3$ with $d$ fixed, instead of relying on the  resolvent expansion \eqref{e:resolvent}, we introduce a novel new expansion based on the Woodbury formula, taking advantage of the local tree structure as detailed in \Cref{lem:woodbury}. Specifically, since the rank of the matrix $\xi=\wt H-H$ is at most $\OO((d-1)^\ell)$ (the number of edges involved in local resampling), we denote the low rank decomposition of $\xi$ as $\xi=UV^\top$, where the number of rows of $U$ and $V$ is $N$, and the number of columns is the rank of $\xi$. The Woodbury formula yields the  difference of the Green's functions $\wt G(z)-G(z)$ as 
\begin{align}\label{e:WB}
\wt G(z)-G(z)&=-G(z)U(\mathbb I+V^\top G(z)U)^{-1}V^\top G(z).
\end{align}
Here, the nonzero rows of $U$ and $V$ correspond to the vertices involved in local resampling. Therefore, the  term $V^\top G(z)U$ in \eqref{e:WB} depends solely on the Green's function entries restricted to the subgraph $\cF:=\cB_{\ell+1}(o,\cG)\cup \{(b_\al, c_\al)\}_{\al\in\qq{\mu}}$. With high probability, the randomly chosen edges $(b_\al, c_\al)$ have tree neighborhoods and are far apart from each other and the vertex $o$. Let $P(z)$ denote the Green's function of the tree extension of $\cF$ (extending each connected component to an infinite $d$-regular tree). Then $V^\top G(z)U-V^\top P(z)U$ is small, leading to the following expansion
\begin{align}\label{e:tG-G}
\wt G(z)-G(z) = G(z)F(z)G(z)+\sum_{k\geq 1}G(z)F(z)((G(z)-P(z))F(z))^kG(z).
\end{align}
Here, $F(z)=-U(\mathbb I+V^\top P(z)U)^{-1}V^\top$ is an explicit matrix. When restricted to the subgraph $\cF$, each entry of $G(z)-P(z)$ is smaller than $N^{-\fb}$ for some small $\fb>0$. Hence, the terms in \eqref{e:tG-G} exhibit exponential decay in $N^{-\fb}$ which is much faster than \eqref{e:resolvent}. Moreover, up to negligible error, we can truncate the expansion \eqref{e:tG-G} at some finite $\fp$. 

Another technical idea involves a Ward identity-type bound for the entries of the Green's function, which serves to constrain various error terms. The Ward identity plays a crucial role in mean-field random matrix theory. It states that the average over the  Green's function entries can be controlled by the Stieltjes transform of the empirical eigenvalue distribution, thus ensuring smallness:
\begin{align}\label{eq:wardex}
\frac{1}{N^2}\sum_{ij}|G_{ij}(z)|^2= \frac{\Im[m_N(z)]}{N\eta},\quad \eta=\Im[z].
\end{align}
Consequently, when selecting two vertices randomly from our graph, the Green's function entries are expected to have small modulus. While some error terms adhere to this form, we also encounter Green's function entries such as 
$G_{ij}^{(o)}$,
where $i,j$ are two adjacent vertices of a vertex $o$, i.e. $o\sim i,  o \sim j$. Namely, we take two vertices of distance two, delete their common neighbor, then take the Green's function. In the initial graph $\cG$, $i$ and $j$ have distance two, so $G_{ij}$ is not small.
In \Cref{lem:deletedalmostrandom}, we establish a Ward identity type result for the expectation of $|G_{ij}^{(o)}|^2$, in a similar way as  \eqref{eq:wardex}. The proof again leverages the idea of local resampling. By local resampling around vertex $o$, we reduce the computation to 
\begin{align}\label{e:Gijo}
   \bE[| G_{ij}^{(o)}(z)|^2]=\bE[|\widetilde G_{ij}^{(o)}(z)|^2],
\end{align}
for the switched graph. We then expand it using the Schur complement formula, in a similar way as \eqref{e:sample}. Crucially, we can bound \eqref{e:Gijo}, by itself times a small factor, and errors as in \eqref{eq:wardex}, leading to the desired bound given by the right-hand side of \eqref{eq:wardex}.


\subsection{Related Work}

The eigenvalue statistics of  random graphs have been intensively studied in the past decade. Thanks to a series of papers \cite{bauerschmidt2017bulk, huang2015bulk, lee2018local,erdHos2013spectral,erdHos2012spectral, he2021fluctuations,huang2020transition,huang2022edge,lee2024higher, bauerschmidt2020edge, huang2023edge, he2022spectral}, the bulk and edge statistics of Erd{\H o}s--R{\'e}nyi graphs $G(N,p)$ with $Np\geq N^{\oo(1)}$ and random $d$-regular graphs with $d\geq N^{\oo(1)}$ are now well understood. Universality holds; namely, after proper normalization and shifts, spectral statistics agree with those from GOE. 

The situation is dramatically different for very sparse Erd{\H o}s--R{\'e}nyi graphs. In the very sparse regime $Np=\OO(\ln N)$, for Erd{\H o}s--R{\'e}nyi graphs, there exists a critical value $b_*=1/(\ln 4-1)$ such that  if $Np\geq b_*\ln N/N$, the extreme eigenvalues of the normalized adjacency matrix converge to $\pm 2$ \cite{benaych2019largest, alt2021extremal, tikhomirov2021outliers, benaych2020spectral}, and all the eigenvectors are delocalized \cite{alt2022completely, erdHos2013spectral}. For $(\ln  \ln  N)^4\ll Np<b_*\ln N/N$, there exist outlier eigenvalues \cite{tikhomirov2021outliers, alt2021extremal}.
The spectrum splits into three phases: a delocalized phase
in the bulk, a fully localized phase near the spectral edge, and a semilocalized phase in between \cite{alt2023poisson,alt2021delocalization}. Moreover, the joint fluctuations of the eigenvalues near the spectral edges form a Poisson point process. For constant degree Erd\H{o}s-R\'enyi graphs, it was proven in \cite{hiesmayr2023spectral} that the largest eigenvalues are determined by small neighborhoods around vertices of close to maximal degree and the corresponding eigenvectors are localized.

Random $d$-regular graphs can also be constructed from $d$ copies of random perfect matchings, or random lifts of a base graph containing two vertices and $d$ edges between them. This class of random graphs obtained from random lifts and in  particular  their extremal eigenvalues have been extensively studied 
\cite{amit2002random, amit2006random,friedman2003relative, bilu2006lifts, friedman2014relativized, puder2015expansion,lubetzky2011spectra, bordenave2019eigenvalues, chen2024new}. It would be interesting to explore if the approach in this paper can be applied to analyze  extremal eigenvalues in this setting.

This paper focuses on the eigenvalue statistics. The eigenvectors of random $d$-regular graphs are also important. The complete delocalization of eigenvectors in sup norm was proven in previous works \cite{bauerschmidt2019local,huang2024spectrum}. For sparse random regular graphs, several results provide information on the structure of eigenvectors without relying on a local law.  For example, random regular graphs are quantum ergodic
\cite{anantharaman2015quantum}, their local eigenvector statistics converge to those of multivariate Gaussians \cite{backhausz2019almost}, and their high-energy eigenvalues have many nodal domains \cite{ganguly2023many}. Gaussian statistics have been conjectured in broad generality for chaotic systems in both the manifold and graph setting \cite{berry1977regular}. A rich line of research exists toward this idea. For an overview in the manifold setting, see the book \cite{zelditch2017eigenfunctions}, and for the graph setting, refer to \cite{smilansky2013discrete}.


\subsection{Outline of the Paper}
In \Cref{s:preliminary}, we introduce the Gaussian divisible ensemble, recall the concept of free convolution, and present results from \cite{huang2024spectrum} on local resampling and the estimation of the Green's function of random $d$-regular graphs. In \Cref{s:outline}, we state our main results, \Cref{t:recursion} and \Cref{t:correlation_evolution}, which concern the high moment estimates of the self-consistent equation and the microscopic version of the loop equations at the edge. Using \Cref{t:recursion} and \Cref{t:correlation_evolution} as inputs, we proceed to prove both \Cref{thm:eigrigidity} and \Cref{c:rigidity}. Additionally, we outline an iteration scheme for proving \Cref{t:recursion} where we iteratively express the Green's functions after local resampling in terms of the original Green's functions before resampling. In \Cref{sec:expansions}, we collect estimates on the differences between Green's functions before and after local resampling using either the Schur complement formula or the Woodbury formula. In \Cref{e:error_term}, we gather bounds on various error terms that arise in the iteration scheme. Finally, the crude bound for the high moment self-consistent equation in \Cref{t:recursion} is proven in \Cref{s:proof_main}, while the refined estimates identifying the next-order correction for the loop equations are established in \Cref{s:refined}.

\subsection{Notation}
We reserve letters in mathfrak mode, e.g. $\fb,  \fc,\fo,\cdots$, to represent universal constants, and $\fC$ for a large universal constant, which
may be different from line by line. We use letters in mathbb mode, e.g. $ \bT, \mathbb X$, to represent set of vertices. 
Given a graph $\cG$, we denote the graph distance as $\dist_\cG(\cdot, \cdot)$, and the radius-$r$ neighborhood of a vertex $i$ in $\cG$ as $\cB_r(i,\cG)$. Given a vertex set $\mathbb X$ of $\cG$, let $\cG^{(\bX)}$ denote the graph obtained from 
$\cG$ by removing the vertices in $\bX$. 
For two vertices $i$ and $j$ in $\cG$, we write $i\sim j$ to indicate that $i$ and $j$ are adjacent in $\cG$.
Given any matrix $A$ and index sets $\bX, \mathbb Y$, let $A^{(\bX)}$ denote the matrix obtained from $A$ by removing the rows and columns indexed by $\bX$. The restriction of 
$A$ to the submatrices corresponding to $\bX\times \bX, \bX\times \mathbb Y$ are denoted as $A_\bX, A_{\bX \mathbb Y}$ respectively.

For two quantities $X_N$ and $Y_N$ depending on $N$, 
we write that $X_N = \OO(Y_N )$ or $X_N\lesssim Y_N$ if there exists some universal constant such
that $|X_N| \leq \fC Y_N$ . We write $X_N = \oo(Y_N )$, or $X_N \ll Y_N$ if the ratio $|X_N|/Y_N\rightarrow 0$ as $N$ goes to infinity. We write
$X_N\asymp Y_N$ if there exists a universal constant $\fC>0$ such that $ Y_N/\fC \leq |X_N| \leq  \fC Y_N$. We remark that the implicit constants may depend on $d$. 
With a slight abuse of notation, for numbers independent of  $N$, we write $\fa\ll \fb$ to indicate that $\fa/\fb\leq 0.01$. 
We denote $\qq{a,b} = [a,b]\cap\bZ$ and $\qq{N} = \qq{1,N}$. We say an event $\Omega$ holds with high probability if $\bP(\Omega)\geq 1-\oo(1)$. We say an event $\Omega$ holds with overwhelmingly high probability, if for any $\fC>0$, 
$\bP(\Omega)\geq 1-N^{-\fC}$ holds provided $N$ is large enough. Given a random variable $Z$, we say an event $\Omega$ holds with overwhelmingly high probability over $Z$, if for any $\fC>0$, 
$\bP_{Z}(\Omega)\geq 1-N^{-\fC}$ holds provided $N$ is large enough. Given an event $\Omega_0$, we say conditioned on $\Omega_0$, an event $\Omega$ holds with overwhelmingly high probability, if for any $\fC>0$, 
$\bP(\Omega|\Omega_0)\geq 1-N^{-\fC}$ holds provided $N$ is large enough.

 \subsection*{Acknowledgements.}
The research of J.H. is supported by NSF grants DMS-2331096 and DMS-2337795,  and the Sloan Research Award. 
The research of T.M. is supported by NSF Grant DMS-2212881.
The research of H-T.Y. is supported by NSF grants DMS-1855509 and DMS-2153335. J.H. wishes to thank Paul Bourgade and Alice Guionnet for enlightening discussions on loop equations, as well as Peter Sarnak for valuable discussions on the edge universality of random $d$-regular graphs during his time as a member at IAS in 2019–2020.

\section{Preliminaries}\label{s:preliminary}

In this section, we will first introduce the Gaussian divisible ensemble in \Cref{s:GDE} and then recall the concept of free convolution in \Cref{s:fc}. We then present results from \cite{huang2024spectrum} on local resampling and its properties in \Cref{s:local_resampling}, Green's function extension with general weights in \Cref{s:pre},  and the estimation of the Green's function of random $d$-regular graphs in \Cref{s:local_law}.

In this paper we fix the parameters as follows
\begin{align}\label{e:parameters}
    0<\fo\ll \ft\ll \fb\ll  \fc\ll \fg\ll 1,
\end{align} 
set $\fR=(\fc/4)\log_{d-1}N$ and choose $\ell $ such that $\ft\ll \ell/\log_{d-1}N\ll \fb$.
Below, we describe their meanings and where they are introduced:
\begin{itemize}
    \item $\fo$ arises from the delocalization of eigenvectors \eqref{e:eig_delocalization}. Many estimates involve bounds containing $N^\fo$ factors, which are harmless.
    \item The time $t$ in the Gaussian division ensemble $H(t)$ from \eqref{e:Ht} satisfies $0\leq t\leq N^{-1/3+\ft}$. Throughout all sections except \Cref{s:universality}, the reader may treat $t$ as a fixed, small number.
    \item $\fb$ relates to the concentration of Green's function entries, with errors bounded by $N^{-\fb}$, see \eqref{eq:infbound0} and \eqref{eq:infbound}.
    \item For the spectral parameter $z\in \bC^+$ in Green's functions and Stieltjes transforms, we restrict it to $\Im[z]\geq N^{-1+\fg}$, see \eqref{e:D}.
    \item $\ell$ comes from local resampling in \Cref{s:local_resampling}, we resample  boundary edges of balls with radius $\ell$.
    \item $\fc$ defines $\fR$, and with high probability, random $d$-regular graphs are tree-like within radius $\fR$ neighborhoods, see \Cref{def:omegabar}.
\end{itemize}
We restrict our analysis to the spectral domain 
\begin{align} \label{e:D}
    \mathbf D=\{z\in \bC^+:  N^{-1+\fg}\leq \Im[z]\leq N^{-\fo}, |\Re[z]|\leq 2+N^{-\fo}  \}.
\end{align}
In the spectral domain $\bf D$,  we impose the conditions $\Im[z]\leq N^{-\fo}$ and $|\Re[z]|\leq 2+N^{-\fo}$. These constraints ensure that  $|\msc(z)|$ is close to $1$, specifically satisfying $||\msc(z)|-1|\lesssim N^{-\fo/2}$.

\subsection{Gaussian Divisible Ensemble}
\label{s:GDE}

We recall the constrained GOE matrix $Z\in \bR^{N\times N}$ as introduced in \cite[Section 2.1]{bauerschmidt2017bulk}. It may be viewed as the usual GOE matrix restricted to matrices with vanishing row and column sums. 
Formally, the GOE matrix $M\in \bR^{N\times N}$ is the centered Gaussian process on the space of symmetric matrices $\cal M \deq \h{A \in \bR^{N \times N} \col A = A^\top= 0}$ with inner product $\langle A, B\rangle = \Tr[AB]$. The covariances of the GOE matrix are given by $\bE[\langle A, M\rangle \langle B, M\rangle]=(2/N)\langle A, B\rangle$, for any $A,B\in \cM$.

The constrained GOE matrix $Z$ can be obtained from the projection of the GOE matrix $M$ on the subspace $\cal M_{\perp} \deq \h{A \in \bR^{N \times N} \col A = A^\top, A \bld 1 = 0}$.  
More concretely, for any $i,j\in \qq{N}$
 \begin{align}\label{e:GOEproject}
    Z_{ij} \stackrel{d}{=}M_{ij}- (g_i+g_j), \quad  g_i:=\frac{1}{N}\left(\sum_{j=1}^N M_{ij}-\frac{\sum_{j,k=1}^N M_{jk}}{2N}\right),
 \end{align}
where $Z$ and $\{g_1, g_2,\cdots, g_N\}$ are independent. 
The covariances of $Z$ are given by
\begin{equation} \label{W_ibp}
\bE [Z_{ij} Z_{kl}] = \frac{1}{N} \pbb{\delta_{ik} - \frac{1}{N}} \pbb{\delta_{jl} - \frac{1}{N}}
+ \frac{1}{N} \pbb{\delta_{il} - \frac{1}{N}} \pbb{\delta_{jk} - \frac{1}{N}}.
\end{equation}


The following result is a straightforward consequence of \eqref{W_ibp} and Gaussian integration by parts.

\begin{lemma}\label{p:intbypart}
For the constrained GOE matrix $Z$, and any bounded differentiable function $F:\bR^{N\times N}\mapsto \bR$, we have the integration by parts formula: for $i,j\in \qq{N}$
\begin{align}\begin{split}\label{e:intbypart}
&\phantom{{}={}}\bE[Z_{ij}F(Z)]
=\sum_{k,l\in \qq{N}}\bE[Z_{ij}Z_{kl}]\bE[\del_{Z_{kl}}F(Z)]
=\frac{1}{N}\bE[\del_{Z_{ij}} F(Z)+\del_{Z_{ji}} F(Z)]\\
&-\frac{1}{N^2}\sum_{k\in \qq{N}}\bE[\del_{Z_{ik}} F(Z)+\del_{Z_{ki}} F(Z)+\del_{Z_{jk}} F(Z)+\del_{Z_{kj}} F(Z)]+\frac{2}{N^3}\sum_{k.l\in \qq{N}}\bE[\del_{Z_{kl}} F(Z)].
\end{split}\end{align}
\end{lemma}

Next, we define the matrix-valued process
\begin{align}\label{e:Ht}
H(t) \deq H+\sqrt{t} \, Z,
\end{align}
where $H$ was defined in \eqref{def_H} as the normalized adjacency matrix of the $d$-regular graph. Thus, $H(0)=H$. The matrix $H(t)$ has a trivial largest eigenvalue $\la_1(t)=d/\sqrt{d-1}$ with eigenvector $\bmu_1(t)={\bf1}/\sqrt{N}$. We denote the remaining eigenvalues of $H(t)$ by $\la_2(t)\geq \la_{3}(t)\geq \cdots \la_{N-1}(t)\geq \la_N(t)$, and corresponding normalized eigenvectors $\bmu_2(t), \bmu_3(t),\cdots, \bmu_N(t)$.

For $z \in \bC^+ =\{z\in \bC\col \Im[z]>0\}$, we define the \emph{time-dependent Green's function} by
\begin{align}\label{e:Gt}
  G(z,t) \deq  (H(t)-z)^{-1}=\sum_{\al=1}^N \frac{\bmu_\al(t)\bmu_\al(t)^\top}{\la_\al(t)-z},
\end{align}
For any vertex set $\bX\subset \qq{N}$, we denote the Green's function with vertex set $\bX$ removed as $G^{(\bX)}(z,t)=(H^{(\bX)}(t)-z)^{-1}$.
We denote the Stieltjes transform of the empirical eigenvalue distribution of $H(t)$ by $m_t(z)$,
\begin{align} \label{def_mtz}
m_t(z)
=\frac{1}{N}\Tr G(z,t)=\frac{1}{N}\sum_{\al=1}^N\frac{1}{\la_\al(t)-z}.
\end{align}

\subsection{Free convolution with semicircle distribution}\label{s:fc}
We recall the Kesten-McKay distribution $\varrho_d(x)$ and semicircle distribution $\varrho_{\rm sc}(x)$ from \eqref{e:KMdistribution} and \eqref{e:msc_equation}. 
The asymptotic empirical eigenvalue distribution of $H$ converges to the Kesten--McKay distribution $\rho_d(x)$ from \eqref{e:KMdistribution}, and the asymptotic empirical eigenvalue distribution of $ Z$ is given by the semicircle distribution from $\varrho_{\rm sc}(x)$ \eqref{e:msc_equation}. 
The asymptotic empirical eigenvalue distribution of the matrix $H(t)=H+\sqrt t Z$ can be described by the free additive convolution, from free probability theory.

In this section, we recall some properties of measures obtained by the free convolution with a semicircle distribution from \cite{MR1488333}. We denote the semicircle distribution of variance $t$ as $t^{-1/2}\rhosc(t^{-1/2}x)$.
Given a probability measure $\mu$ on $\bR$, we denote its free convolution with a semicircle distribution of variance $t$ by $\mu_t:=\mu\boxplus t^{-1/2}\rhosc(t^{-1/2}x)$. The Stieltjes transforms of $\mu$ and $\mu_t$ are given by 
\begin{align*}
s_\mu(z)=\int (x-z)^{-1}\rd \mu(x),\quad s_{\mu_t}(z)=\int (x-z)^{-1}\rd \mu_{t}(x).
\end{align*} 
Then the following holds.
\begin{lemma}[{\cite[Lemma 4 and Proposition 2]{MR1488333}}]\label{l:free_convolution}
 We denote the set $U_t = \{z\in\bC^+: \int |z-x|^{-2}\rd \mu(x)<t^{-1}\}$. Then $z \mapsto z-t s_\mu(z)$ is a homeomorphism from $\overline{U}_t$ to $\bC^+\cup \bR$ and conformal from $U_t$ to $\bC^+$. Moreover, the Stieltjes transform of $\mu_t$ is characterized by 
\begin{align}\label{e:Gfreeconvolution}
s_\mu(z)=s_{\mu_t}(z-ts_\mu(z)), \text{ for any } z\in U_t.
\end{align}
We can also rearrange \eqref{e:Gfreeconvolution} as
\begin{align}\label{e:Gfreeconvolution2}
s_\mu(z+ts_{\mu_t}(z))=s_{\mu_t}(z),\text{ for any } z\in \bC^+.
\end{align}
\end{lemma}

The asymptotic eigenvalue distribution of $H$ converges to the Kesten--McKay distribution $\rho_d( x)$, and the asymptotic eigenvalue distribution of $\sqrt{t} Z$ is given by the semicircle distribution of variance $t$. Thus, the asymptotic eigenvalue distribution of $H(t)=H+\sqrt{t}Z$ is the free convolution of the Kesten--McKay distribution $\rho_d( x)$ and the semicircle distribution of variance $t$. We denote this density and its Stieltjes transform by 
\begin{align}\label{e:defrhodt}
    \rho_d(x,t):=\varrho_d(x) \boxplus t^{-1/2}\rhosc(t^{-1/2}x),\quad \md(z,t) = \int \frac{\rho_d(x,t)}{x - z} \rd x,\quad z\in \bC^+.
\end{align}
We deduce from \Cref{l:free_convolution} that
\begin{align}\label{e:selffc}
\md(z-t\md(z),t)=\md(z),
\end{align}
where
$z-t\md(z)$
is a homeomorphism from the set $\{z\in \bC^+: \int |x-z|^{-2}\rho_d(x)\rd x\leq t^{-1}\}$ to $\bC^+\cup \bR$.  From \Cref{l:free_convolution},  $z-t\md(z)$ maps $\{z\in \bC^+: \int |x-z|^{-2}\rho_d(x)\rd x=t^{-1}\}$ to the support of the measure $\rho_d(x,t)$. Since $\varrho_d$ is symmetric, we denote by $z=\pm \xi_{t} \in\bR$  the largest and smallest real solutions to
\begin{align}\label{e:edgeeqn1}
\int \frac{\rho_d(x)}{|x-z|^2}\rd x=\frac{1}{t}\quad \Leftrightarrow \quad  m_d'(z)=\frac{1}{t}.
\end{align}
As a consequence,  the right and left edges of the measure $\rho_d(x,t)$ are given by
\begin{align}\label{e:edgeeqn2}
\pm E_{t} =\pm \xi_t -t\md(\pm\xi_t).
\end{align}
If we differentiate with respect to $t$ on both sides of \eqref{e:edgeeqn2}, 
\begin{align}\label{e:Etshift}
    \pm\del_t E_t=\pm\del_t \xi_t-\md(\pm\xi_t)\mp t\md'(\pm \xi_t)\del_t \xi_t= -\md(\pm\xi_t)=-\md(\pm E_t,t),
\end{align}
where we used \eqref{e:edgeeqn1} that $\md'(\pm \xi_t)=1/t$.

By rearranging the expression of $\msc(z)$ and $m_d(z)$ from \eqref{e:msc_equation} and \eqref{e:md_equation}, we have for $|z-2|\ll 1$,
\begin{align}\begin{split}\label{e:medge_behavior}
&\msc(z)=-1+\sqrt{z-2}+\OO(|z-2|), \quad 1-\msc(z)^2=2\sqrt{z-2}+\OO(|z-2|),
\\
&m_d(z)=-\frac{d-1}{d-2}+\cA\sqrt{z-2}+\OO(|z-2|),\quad \md'(z)=\frac{\cA}{2\sqrt{z-2}}+\OO(1),
\end{split}\end{align}
where $\cA=d(d-1)/(d-2)^2$ is from \eqref{e:edge_behavior}.
We can then solve for $\xi_t$ and the edges $E_t$ of $\varrho_d(x,t)$ using \eqref{e:edgeeqn1}, \eqref{e:edgeeqn2} and \eqref{e:medge_behavior},
\begin{align}\label{e:xi_behavior}
\xi_t=2+\frac{\cA^2 t^2}{4}+\OO(t^3),\quad E_t=2+\frac{d-1}{d+2}t-\frac{\cA^2 t^2}{4}+\OO(t^3).
\end{align}
For any $z\in \bC^+$, we denote 
$
    z_t=z+t\md(z,t),
$
then we can rewrite \eqref{e:selffc} as 
\begin{align}\label{e:mdzt}
    \md(z,t)=\md(z_t).
\end{align}
In the following lemma, we present some estimates on $z_t$, which will be used in subsequent sections. The proof is deferred to \Cref{app:Green}.
\begin{lemma}\label{l:relation_zt_z}
 For any $z\in \bC^+$, we denote $z_t=z+t\md(z,t)$. If $|z-E_t|\ll 1$, then 
 \begin{align}\begin{split}\label{e:relation_zt_z}
    \sqrt{z_t-2}&=\sqrt{\xi _t-2}+\sqrt{z-E _t}+\OO\left(t\sqrt{|z-E _t|}+t^2\right)\\
    &=\frac{\cA t}{2}+\sqrt{z-E _t}+\OO\left(t\sqrt{|z-E _t|}+t^2\right)
    =\OO\left(t+\sqrt{|z-E_t |}\right),
\end{split}\end{align}
and 
\begin{align}\label{e:square_root_behavior}
    \Im[\sqrt{z_t-2}]\asymp \Im[\sqrt{z-E_t}].
\end{align}
 Similar statements hold for when $z$ is close to the left edge $-E_t$. As a consequence, let $z\in \bf D$ with $\eta:=\Im[z]$ and $\kappa=\min\{|z-E_t|, |z+E_t|\}$, then
\begin{align}\begin{split}\label{eq:square_root_behave}
&\Im[\msc(z_t)],\Im[m_d(z_t)]\asymp\left\{\begin{array}{ll}
\sqrt{\kappa+\eta}&\text{ for }  |\Re[z]|\leq E_t,\\
\eta/\sqrt{\kappa+\eta}&\text{ for }  |\Re[z]|\geq E_t.
\end{array}
\right. 
\end{split}
\end{align}
\end{lemma}

\subsection{Local Resampling}  
\label{s:local_resampling}

In this section, we recall the local resampling and its properties. This gives us the framework to talk about resampling from the random regular graph distribution as a way to get an improvement in our estimates of the Green's function.

For any graph $\cG$, we denote the set of unoriented edges by $E(\cG)$,
and the set of oriented edges by $\vec{E}(\cG):=\{(u,v),(v,u):\{u,v\}\in E(\cG)\}$.
For a subset $\vec{S}\subset \vec{E}(\cG)$, we denote by $S$ the set of corresponding non-oriented edges.
For a subset $S\subset E(\cG)$ of edges we denote by $[S] \subset \qq{N}$ the set of vertices incident to any edge in $S$.
Moreover, for a subset $\bV\subset\qq{N}$ of vertices, we define $E(\cG)|_{\bV}$ to be the subgraph of $\cG$ induced by $\bV$.

\begin{figure}
    \centering
    \includegraphics[scale=0.7]{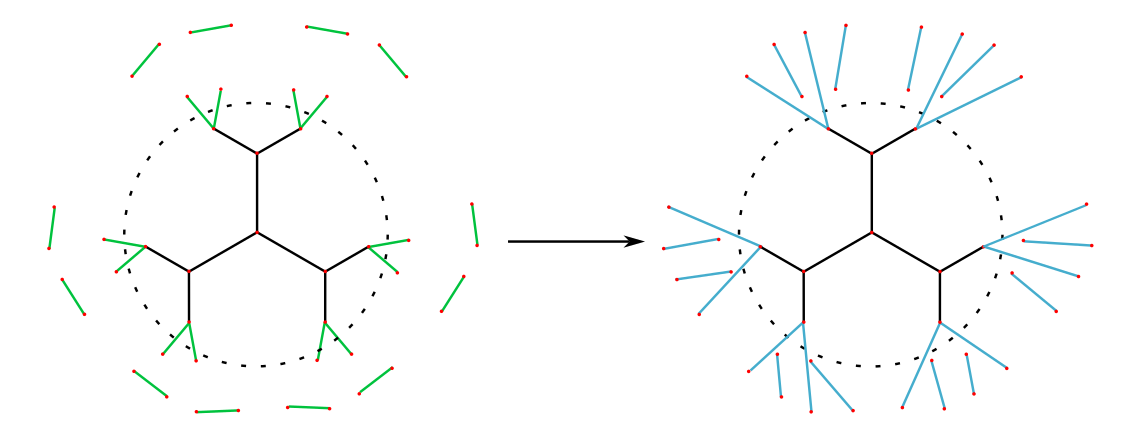}
    \caption{
    An example of the local resampling performed on the graph is as follows. We replace the green edges, located on the boundary of the radius-$\ell$ neighborhood of a vertex $o$, with randomly chosen edges from the graph. Together, these edges constitute the resampling data, denoted by $\mathbf{S}$. This operation creates new blue edges and establishes a new boundary.}
    \label{fig:switchingproc}
\end{figure}

\begin{definition}
A (simple) switching is encoded by two oriented edges $\vec S=\{(v_1, v_2), (v_3, v_4)\} \subset \vec{E}$.
We assume that the two edges are disjoint, i.e.\ that $|\{v_1,v_2,v_3,v_4\}|=4$.
Then the switching consists of
replacing the edges $\{v_1,v_2\}, \{v_3,v_4\}$ with the edges $\{v_1,v_4\},\{v_2,v_3\}$.
We denote the graph after the switching $\vec S$ by $T_{\vec S}(\cG)$,
and the new edges $\vec S' = \{(v_1,v_4), (v_2,v_3)\}$ by
$
  T(\vec S) = \vec S'
$.
\end{definition}

The local resampling involves a fixed center vertex, which we now assume to be vertex $o$,
and a radius $\ell$.
Given a $d$-regular graph $\cG$, we write $\cT\deq\cB_{\ell}(o,\cG)$ to denote the radius-$\ell$ neighborhood of $o$ (which may not necessarily be a tree) and write $\bT$ for its vertex set.\index{$\cT, \bT$}
The edge boundary $\del_E \cT$ of $\cT$ consists of the edges in $\cG$ with one vertex in $\bT$ and the other vertex in $\qq{N}\setminus\bT$.
We enumerate the edges of $\del_E \cT$ as $ \del_E \cT = \{e_1,e_2,\dots, e_\mu\}$, where $e_\al=\{l_\al, a_\al\}$ with $l_\al \in \bT$ and $a_\al \in \qq{N} \setminus \bT$. We orient the edges $e_\al$ by defining $\vec{e}_\al=(l_\al, a_\al)$.
We notice that $\mu$ and the edges $e_1,e_2, \dots, e_\mu$ depend on $\cG$. The edges $e_\al$ are distinct, but
the vertices $a_\al$ are not necessarily distinct and neither are the vertices $l_\al$. Our local resampling switches the edge boundary of $\cT$ with randomly chosen edges in $\cG^{(\bT)}$
if the switching is admissible (see below), and leaves them in place otherwise.
To perform our local resampling, see \Cref{fig:switchingproc}, we choose $(b_1,c_1), \dots, (b_\mu,c_\mu)$ to be independent, uniformly chosen oriented edges from the graph $\cG^{(\bT)}$, i.e.,
the oriented edges of $\cG$ that are not incident to $\bT$,
and define 
\begin{equation}\label{e:defSa}
  \vec{S}_\al= \{\vec{e}_\al, (b_\al,c_\al)\},
  \qquad
  {\bf S}=(\vec S_1, \vec S_2,\dots, \vec S_\mu).
\end{equation}
The sets $\bf S$ will be called the \emph{resampling data} for $\cG$. We remark that repetitions are allowed in the data $(b_1, c_1), (b_2, c_2),\cdots, (b_\mu, c_\mu)$.
We define an indicator that will be crucial to the definition of the switch.

\begin{definition}
For $\al\in\qq{\mu}$,
we define the indicator functions
$I_\al \equiv I_\al(\cG,{\bf S})=1$\index{$I_\alpha$} 
\begin{enumerate}
\item
 the subgraph $\cB_{\fR/4}(\{a_\al, b_\al, c_\al\}, \cG^{(\bT)})$ after adding the edge $\{a_\al, b_\al\}$ is a tree;
\item 
and $\dist_{\cG^{(\bT)}}(\{a_\al,b_\al,c_\al\}, \{a_\beta,b_\beta,c_\beta\})> {\fR/4}$ for all $\beta\in \qq{\mu}\setminus \{\al\}$.
\end{enumerate}
\end{definition}
 The indicator function $I_\alpha$ imposes two conditions. The first one is a ``tree" condition, which ensures that 
 $a_\al$ and $\{b_\al, c_\al\}$ are far away from each other, and their neighborhoods are trees. 
The second one imposes an ``isolation" condition, which ensures that we only perform simple switching when the switching pair is far away from other switching pairs. In this way, we do not need to keep track of the interaction between different simple switchings. 

We define the \emph{admissible set}
\begin{align}\label{Wdef}
{\mathsf W}_{\bf S}:=\{\al\in \qq{\mu}: I_\al(\cG,{\bf S}) \}.
\end{align}
We say that the index $\al \in \qq{\mu}$ is \emph{switchable} if $\al\in {\mathsf W}_{\bf S}$. We denote the set $\bW_{\bf S}=\{b_\al:\al\in {\mathsf W}_{\bf S}\}$\index{$\bW_{\bf S}$}. Let $\nu:=|{\mathsf W}_{\bf S}|$ be the number of admissible switchings and $\al_1,\al_2,\dots, \al_{\nu}$
be an arbitrary enumeration of ${\mathsf W}_{\bf S}$.
Then we define the switched graph by
\begin{equation} \label{e:Tdef1}
T_{\bf S}(\cG) := \left(T_{\vec S_{\al_1}}\circ \cdots \circ T_{\vec S_{\al_\nu}}\right)(\cG),
\end{equation}
and the resampling data by
\begin{equation} \label{e:Tdef2}
  T({\bf S}) := (T_1(\vec{S}_1), \dots, T_\mu(\vec{S}_\mu)),
  \quad
  T_\al(\vec{S}_\al) \deq
  \begin{cases}
    T(\vec{S}_\al) & (\al \in {\mathsf W}_{\bf S}),\\
    \vec{S}_\al & (\al \not\in {\mathsf W}_{\bf S}).
  \end{cases}
\end{equation}

To make the structure more clear, we introduce an enlarged probability space.
Equivalent to the definition above, the sets $\vec{S}_\al$ as defined in \eqref{e:defSa} are uniformly distributed over 
\begin{align*}
{\sf S}_{\al}(\cG)=\{\vec S\subset \vec{E}: \vec S=\{\vec e_\al, \vec e\}, \text{$\vec{e}$ is not incident to $\cT$}\},
\end{align*}
i.e., the set of pairs of oriented edges in $\vec{E}$ containing $\vec{e}_\al$ and another oriented edge in $\cG^{(\bT)}$.
Therefore ${\bf S}=(\vec S_1,\vec S_2,\dots, \vec S_\mu)$ is uniformly distributed over the set
${\sf S}(\cG)=\sf S_1(\cG)\times \cdots \times \sf S_\mu(\cG)$.

We introduce the following notation on the probability and expectation with respect to the randomness of the $\bfS\in \sf S(\cG)$.
\begin{definition}\label{def:PS}
    Given any $d$-regular graph $\cG$, we 
    denote $\bP_\bfS(\cdot)$ the uniform probability measure on ${\sf S}(\cG)$;
 and $\bE_\bfS[\cdot]$ the expectation  over the choice of $\bfS$ according to $\bP_\bfS$. 
\end{definition}

 The following claim from \cite[Lemma 7.3]{huang2024spectrum} states that this switch is invariant under the random regular graph distribution.

\begin{lemma}[{\cite[Lemma 7.3]{huang2024spectrum}}] \label{lem:exchangeablepair}
Fix $d\geq 3$. We recall the operator $T_\bfS$ from \eqref{e:Tdef1}. Let $\cG$ be a random $d$-regular graph  and $\bfS$ uniformly distributed over $\sfS(\cG)$, then the graph pair $(\cG, T_{\bf S}(\cG))$ forms an exchangeable pair:
\begin{align*}
(\cG, T_{\bf S}(\cG))\stackrel{law}{=}(T_{\bf S}(\cG), \cG).
\end{align*}
\end{lemma}

\subsection{Green's function extension with general weights}\label{s:pre}
Fix degree $d\geq 3$, we recall the integer $\omega_d\geq 1$ from \cite[Definition 2.6]{huang2024spectrum}, which represents the maximum number of cycles a $d$-regular graph can have while still ensuring that its Green's function exhibits exponential decay. Instead of delving into the technical definition, it suffices to note two key properties: $\omega_d\geq 1$ and $\omega_d$ is nondecreasing with respect to $d$.


\begin{definition}\label{def:omegabar}
Fix $d\geq 3$ and a sufficiently small $0<\fc<1$, $\fR=(\fc /4)\log_{d-1}N$ as in \eqref{e:parameters}. We define the event $\oOmega$,  where the following occur: 
    \begin{enumerate}
        \item 
The number of vertices that do not have a  tree neighborhood of radius $\fR$ is at most $N^{\fc}$.
        \item 
        The radius $\fR$ neighborhood of each vertex has an excess (i.e., the number of independent cycles) of at most $\omega_d$. 
    \end{enumerate}
\end{definition}

The event $\oOmega$ is a typical event. The following proposition from \cite[Proposition 2.1]{huang2024spectrum} states that $\oOmega$ holds with high probability. 
\begin{proposition}[{\cite[Proposition 2.1]{huang2024spectrum}}]\label{lem:omega}
$\oOmega$ occurs with probability $1-\OO(N^{-(1-\fc)\omega_d})$.
\end{proposition}

As we will see, for graphs $\cG\in \oOmega$,  their Green's functions can be approximated by tree extensions with overwhelmingly high probability. 
For the infinite $d$-regular tree and the infinite $(d-1)$-ary tree (trees where the root has degree $d-1$ and all other vertices have degree $d$),
the following proposition computes their Green's function explicitly.
\begin{proposition}[{\cite[Proposition 2.2]{huang2024spectrum}}]\label{greentree}
Let $\cX$ be the infinite $d$-regular tree.
For all $z \in \bC^+$, its Green's function is
\begin{equation} \label{e:Gtreemkm}
  G_{ij}(z)=m_{d}(z)\left(-\frac{\msc(z)}{\sqrt{d-1}}\right)^{\dist_{\cX}(i,j)}.
\end{equation}
Let $\cY$ be the infinite $(d-1)$-ary tree with root vertex $o$.
Its Green's function is
\begin{equation} \label{e:Gtreemsc}
  G_{ij}(z)=m_{d}(z)\left(1-\left(-\frac{\msc(z)}{\sqrt{d-1}}\right)^{2{\rm anc}(i,j)+2}\right)\left(-\frac{\msc(z)}{\sqrt{d-1}}\right)^{\dist_{\cY}(i,j)},
\end{equation}
where ${\rm anc}(i,j)$ is the distance from the common ancestor of the vertices $i,j$ to the root $o$. 
In particular,
\begin{align}\label{e:Gtreemsc2}
G_{oi}(z)=\msc(z)\left(-\frac{\msc(z)}{\sqrt{d-1}}\right)^{\dist_{\cY}(o,i)}.
\end{align}
\end{proposition}

Next, we recall the idea of a Green's function extension with general weight $\Delta$ from {\cite[Section 2.3]{huang2024spectrum}}.

\begin{definition}\label{def:pdef}
    Fix degree $d\geq 3$, and a graph $\cT$ with degrees bounded by $d$. We define the function $P(\cT,z,\Delta)$ as follows. We denote $A(\cT)$  the adjacency matrix of $\cT$, $D(\cT)$ the diagonal matrix of degrees of $\cT$, and $\bI$ the diagonal matrix indexed by the vertex set $\bT$ of $\cT$. Then 
    \begin{align}\label{e:defP}
    P(\cT,z,\Delta):=\frac{1}{-z+A(\cT)/\sqrt{d-1}-(d\mathbb I -D(\cT))\Delta/(d-1)}.
    \end{align}
\end{definition}
The matrix $P(\cT,z,\Delta)$ is the Green's function of the matrix obtained from $A(\cT)/\sqrt{d-1}$ by attaching to each vertex $i\in \cT$ a weight $-(d-D_{ii}(\cT))\Delta/(d-1)$. When $\Delta=\msc(z)$, \eqref{e:defP} is the Green's function of the tree extension of $\cT$, i.e. extending $\cT$ by attaching copies of infinite $(d-1)$-ary trees to $\cT$ to make each vertex degree $d$. If $\cT$ is a tree, then in this case, the Green's function agrees with the Green's function of the infinite $d$-regular tree, as in \eqref{e:Gtreemkm}. 
For any vertex set $\bX$ in $\cT$, we define the following Green's function with vertex $\bX$ removed. Let $A^{(\bX)}(\cT)$ and $D^{(\bX)}(\cT)$ denote the matrices obtained from $A(\cT), D(\cT)$ by removing the row and column associated with vertex $\bX$. The Green's function is then defined as:
 \begin{align}\label{e:defPi}
    P^{(\bX)}(\cT,z,\Delta):=\frac{1}{-z+A^{(\bX)}(\cT)/\sqrt{d-1}-(d\mathbb I-D^{(\bX)}(\cT))\Delta/(d-1)}.
    \end{align}

We recall from \cite[Proposition 2.12]{huang2024spectrum}, the following estimates on the Green's function extension with general weights $\Delta$, which are sufficiently close to $\msc(z)$.
\begin{proposition}[{\cite[Proposition 2.12]{huang2024spectrum}}]
If the graph $\cT$ has excess at most $\omega_d$, and ${\rm diam}(\cT)|\Delta-\msc(z)|\ll 1$, then for any vertices $i,j\in \cT$
\begin{align}\label{e:Pijbound}
    |P_{ii}(\cT, z, \Delta)|\asymp 1,
    \quad |P_{ij}(\cT, z, \Delta)|\lesssim \left(\frac{|\msc(z)|}{\sqrt{d-1}}\right)^{\dist_\cT(i,j)}.
\end{align}
\end{proposition}

For any integer $\ell\geq 1$, we define the functions $X_\ell(\Delta,z), Y_\ell(\Delta,z)$ as
\begin{align}\label{def:Y}
X_\ell(\Delta,z)=P_{oo}(\cB_\ell(o,\cX),z,\Delta),\quad Y_\ell(\Delta,z)=P_{oo}(\cB_\ell(o,\cY),z,\Delta),
\end{align}
where $\cX$ is the infinite $d$-regular tree with root vertex $o$, and $\cY$ is the infinite $(d-1)$-ary tree with root vertex $o$. 
Then $\msc(z)$ is a fixed point of the function $Y_\ell$, i.e. $Y_\ell(\msc(z),z)=\msc(z)$. Also, $X_\ell(\msc(z),z)=\md(z)$. 
The following proposition states that if $\Delta$ is sufficiently close to $\msc(z)$ and $w$ is close to $z$, then $Y_\ell(\Delta,w)$ is close to $\msc(z)$, and $X_\ell(\Delta,w)$ is close to $\md(z)$. We postpone its proof to \Cref{app:Green}.

\begin{proposition}\label{p:recurbound}
Given $z, w, \Delta\in \bC^+$ such that $\ell^2|w-z|, \ell^2|\Delta-\msc(z)|\ll 1$, then the derivatives of $Y_\ell(\Delta,w)$ satisfies
\begin{align}\label{e:Yl_derivative}
   |\del_1 Y_\ell(\Delta,  w)|, |\del_2 Y_\ell(\Delta,  w)|\lesssim \ell, 
   \quad |\del_1^2 Y_\ell(\Delta,  w)|, |\del_1\del_2 Y_\ell(\Delta,  w)|,|\del_2^2 Y_\ell(\Delta,  w)|\lesssim \ell^3,
\end{align}
and the same statement holds for $X_\ell(\Delta,w)$.
Moreover, the functions $X_\ell(\Delta,w), Y_\ell(\Delta,w)$ 
satisfy
\begin{align}\begin{split}\label{e:recurbound}
Y_\ell(\Delta,w)-\msc(z)
&=\msc^{2\ell+2}(z)(\Delta-\msc(z))+(\msc^2(z)+\msc^4(z)+\cdots+\msc^{2\ell+2}(z))(w-z)\\
&+\msc^{2\ell+2}(z)\md(z)\left(\frac{1-\msc^{2\ell+2}(z)}{d-1}+\frac{d-2}{d-1}\frac{1-\msc^{2\ell+2}(z)}{1-\msc^2(z)}\right)(\Delta-\msc(z))^2\\
&+\OO(\ell^5(|w-z|^2 +|\Delta-\msc(z)||w-z|+|\Delta-\msc(z)|^3)).
\end{split}\end{align}
and
\begin{align}\begin{split}\label{e:Xrecurbound}
X_\ell(\Delta,w)-\md(z)
&=\frac{d}{d-1}\md^2(z)\msc^{2\ell}(z)(\Delta-\msc(z))+\OO\left(\ell^3(|w-z|+|\Delta-\msc(z)|^2)\right).
\end{split}\end{align}
\end{proposition}

\subsection{Local Law of $H(t)$}
\label{s:local_law}

In this section we collect some estimates on the Green's function $G(z,t)$ and the Stieltjes transform $m_t(z)$ of $H(t)$. For any $z\in \bC^+$, we recall the following relations from \eqref{e:mdzt}
\begin{align}\label{e:defw}
    \md(z,t)=\md(z_t), \quad z_t=z_t(z)=z+t\md(z,t). 
\end{align} 
We also introduce $Q_t(z)$, a quantity that was first defined in \cite{bauerschmidt2019local, huang2024spectrum}, and plays a crucial role in the proof of local law. The quantity is the average of $G_{jj}^{(i)}(z,t)$ over all pairs of adjacent vertices $i\sim j$:
\begin{align}\label{e:Qsum}
Q_t(z):=\frac{1}{Nd}\sum_{i\sim j}G_{jj}^{(i)}(z,t).
\end{align}


For any vertex set $\bX\subset \qq{N}$, and integer $r\geq 1$, we denote
$\cB_r(\bX,\cG)=\{i\in \qq{N}: \dist_\cG(i, \bX)\leq r\}$ the ball of radius-$r$ around vertices $\bX$ in $\GG$. 
For $t=0$, the weak local law of $H=H(0)$ has been proven in {\cite[Theorem 4.2]{huang2024spectrum}}, which is recalled below (by taking $(\fa, \fb,\fc, \mathfrak r)=(12,300,\fc,\fc/32)$). We denote $G(z)=G(z,0)=(H-z)^{-1}$, $m_N(z)=m_0(z)$ and $Q(z)=Q_0(z)$.

\begin{theorem}[{\cite[Theorem 4.2]{huang2024spectrum}}] \label{thm:prevthm0}
Fix any sufficiently small $0<\fb\ll\fc<1$, $\fR=(\fc/4)\log_{d-1}(N)$ and any $z\in \bC^+$ , we define $\eta(z)=\Im[z], \kappa(z)=\min\{|\Re[z]-2|, |\Re[z]+2|\}$, and the error parameters
\begin{align}\label{e:defepsilon}
\varepsilon'(z):=(\log N)^{100}\left(N^{-10\fb} +\sqrt{\frac{\Im[m_d(z)]}{N\eta(z)}}+\frac{1}{(N\eta(z))^{2/3}}\right),\quad \varepsilon(z):=\frac{\varepsilon'(z)}{\sqrt{\kappa(z)+\eta(z)+\varepsilon'(z)}}.
\end{align}
For any $\fC\geq 1$ and $N$ large enough, with probability at least $1-\OO(N^{-\fC})$ with respect to the uniform measure on $\oOmega$, 
\be\label{eq:infbound0}
|G_{ij}(z)-P_{ij}(\cB_{\fR/100}(\{i,j\},\cG),z,\msc(z))|,\quad |Q(z)-\msc(z)|,\quad |m_N(z)-m_d(z)|\lesssim \varepsilon(z).
\ee
uniformly for every $i,j\in \qq{N}$, and any $z\in \bC^+$ with $\Im[z]\geq (\log N)^{300}/N$. We denote the event that \eqref{eq:infbound0} holds as $\Omega$. 
\end{theorem}

The following theorem gives bounds for the Green's functions and the Stieltjes transform of the empirical eigenvalue distribution of $H(t)$, for $t>0$. 
Recall from \eqref{e:Ht}, $H(t)$ is a Gaussian divisible ensemble. Using \Cref{thm:prevthm0} as input, the claims follow from applying {\cite[Lemma 4.2]{bauerschmidt2017bulk}}, \cite[Theorem 3.1, Corollary 3.2]{huang2019rigidity} and \cite[Theorem 2.1]{bourgade2017eigenvector}.
We postpone the detailed proof to \Cref{app:Green}.

\begin{theorem} \label{thm:prevthm}
Recall parameters $\fo\ll\ft\ll \fb\ll\fc\ll \fg$ from \eqref{e:parameters}, and fix $0\leq t\leq N^{-1/3+\ft}$.  Condition on $\cG\in \Omega$, for $N$ large enough,  with overwhelmingly high probability over $Z$, the following holds
\begin{enumerate}
\item The eigenvectors of $H(t)$ are delocalized
\begin{align}\label{e:eig_delocalization}
    \|\bmu_\al(t)\|^2_\infty \lesssim N^{\fo/2-1}, \quad 1\leq \al \leq N.
\end{align}

\item For any $z\in \bC^+$ we denote $z_t=z_t(z)=z+t\md(z,t)$. Then uniformly for every $i,j\in \qq{N}$, and any $z\in \bC^+$ with $|z|\leq 1/\fg, \Im[z]\geq N^{-1+\fg}$, we have the following bounds: 
\begin{align}\begin{split}\label{eq:infbound}
&|G_{ij}(z,t)-P_{ij}(\cB_{\fR/100}(\{i,j\},\cG),z_t,\msc(z_t)|\lesssim N^{-2\fb} ,\\
 &|Q_t(z)-\msc(z_t)|,\quad |m_t(z)-m_d(z_t)|\lesssim N^{-2\fb} .
\end{split}\end{align}
\end{enumerate}
\end{theorem}

\begin{remark}
    We remark the relations \eqref{eq:infbound} are essentially the same as those in \eqref{eq:infbound0}, but with $z$ in $P_{ij}(\cB_{\fR/100}(\{i,j\},\cG),z,\msc(z))$, $\msc(z)$, and $m_d(z)$ replaced by $z_t=z+t\md(z,t)$.
\end{remark}

Fix an edge $\{i,o\}\subset \cG$, we recall the resampling data ${\bf S}=\{(l_\al, a_\al), (b_\al, c_\al)\}_{\al\in\qq{\mu}}$ around $o$ from \Cref{s:local_resampling}, denote  $\wt \cG=T_\bfS \cG$. In the remainder of the paper, we denote by $\wt H$ the normalized adjacency matrix of $\tcG$. Its Green's function and the Stieltjes transform of its empirical eigenvalue distribution are denoted as follows:
\begin{align}\label{e:tGtm}
    \wt G(z,t)=(\widetilde H+\sqrt t Z-z)^{-1}, \quad \wt m_t(z)=\frac{1}{N}\Tr\wt G(z,t).
\end{align}

In the rest of this section we collect some basic estimates of the Green's functions of $G(z,t), \wt G(z,t)$. Their proofs follow from analyzing Green's functions using the Schur complement formula \eqref{e:Schur1}. We postpone their proofs to \Cref{app:Green}.   
\begin{lemma}\label{l:basicG}
Recall parameters $\fo\ll\ft\ll \fb\ll \fc\ll \fg$ from \eqref{e:parameters}, and fix $0\leq t\leq N^{-1/3+\ft}$.
Fix a $d$-regular graph $\cG\in \Omega$ (recall from \Cref{thm:prevthm}),  edges $\{i,o\}, \{b,c\}, \{b',c'\}$ in $\cG$, and let $\cT=\cB_\ell(o,\cG)$ with vertex set $\bT$. We recall the resampling data ${\bf S}=\{(l_\al, a_\al), (b_\al, c_\al)\}_{\al\in\qq{\mu}}$ around $o$ from \Cref{s:local_resampling}, and let  $\bW=\{b_\al\}_{\al \in \qq{\mu}}$.
Assume the following holds
\begin{align}\label{e:indicator0}
        A_{io}A_{bc}A_{b'c'}\prod_{\al\in \qq{\mu}}A_{b_\al c_\al}\prod_{x\neq y\in \{o,c,c'\}\cup\{c_\al\}_{\al\in \qq{\mu}}}\bm1(\cB_{\fR/2}(x,\cG) \text{is a tree})\bm1(\dist_\cG(x,y)\geq \fR/2)=1.
    \end{align}
    Let $z\in \bC^+$ with $|z|\leq 1/\fg, \eta:=\Im[z]\geq N^{-1+\fg}$, and denote $z_t=z_t(z)=z+t\md(z,t)$. 
    
    Then for any vertex set $\bX\subset \bW\cup \{b,b'\}$, or $\bX=\bT\cup \bX'$ with $\bX'\subset \bW\cup \{b,b'\}$, the following holds with overwhelmingly high probability over $Z$:
\begin{align}\begin{split}\label{eq:local_law}
&|G^{(\bX)}_{xy}(z,t)-P^{(\bX)}_{xy}(\cB_{\fR/100}(\{x,y\}\cup \bX, \GG),z_t,\msc(z_t))|\lesssim N^{-\fb} ,
\end{split}\end{align}
and 
    \begin{align}\label{e:Gest}
        &\max_{x,y\not\in \bX}|\Im[ G^{(\bX)}_{xy}(z,t)]\lesssim N^\fo \Im[ m_t(z)],\quad \frac{1}{N}\sum_{x\notin \bX}|G_{xy}^{(\bX)}|^2\lesssim \frac{N^\fo \Im[m_t(z)]}{N\eta},\\
   \label{e:Trace-change}
       & \left|\frac{1}{N}\Tr[G^{(\bX)}(z,t)]-m_t(z)\right|\lesssim  \frac{(d-1)^\ell N^\fo \Im[m_t(z)]}{N\eta},\\
    &\label{e:Ghao0}
        \sum_{x=1}^NG_{xy}(z,t)=\frac{1}{d/\sqrt{d-1}-z}, \quad 
        \left|\sum_{x\not\in \bX}G_{xy}^{(\bX)}(z,t)\right|\lesssim \ell,\\
        &\label{e:use_Ward}
 \frac{1}{Nd}\sum_{v\sim u\not\in \bX}|G_{u y}^{(v \bX)}(z,t)|^2\lesssim \frac{ N^\fo \Im[m_t(z)]}{N\eta}.
    \end{align}
    
\end{lemma}

\begin{lemma}
    \label{c:expectationbound}
Adopt the notation and assumptions from \Cref{l:basicG}, and further assume $\tcG\in \Omega$. 
 Then for any vertex set $\bX\subset \bW\cup \{b,b'\}$, or $\bX=\bT\cup \bX'$ with $\bX'\subset \{b,b'\}$, the following holds with overwhelmingly high probability over $Z$:
\begin{align}\label{e:tmmdiff}
    |\wt m_t(z)-m_t(z)|\lesssim \frac{(d-1)^\ell N^\fo \Im[m_t(z)]}{N\eta},
\end{align}
and 
\begin{align}\begin{split}\label{e:use_Ward2}
 &\frac{1}{N}\sum_{x\notin\bX} |\wt G^{(\bX)}_{xy}(z,t)|^2\lesssim \frac{ N^\fo \Im[m_t(z)]}{N\eta}, \quad   \left|\frac{1}{N}\Tr[\wt G^{(\bX)}(z,t)]-m_t(z)\right|\lesssim  \frac{(d-1)^\ell N^\fo \Im[m_t(z)]}{N\eta},\\
    &
        \sum_{x=1}^N\wt G_{xy}(z,t)=\frac{1}{d/\sqrt{d-1}-z}, \quad 
        \left|\sum_{x\not\in \bX}\wt G_{xy}^{(\bX)}(z,t)\right|\lesssim \ell.
\end{split}\end{align}
\end{lemma}

In the following lemma, we gather some estimates related to the constrained GOE matrix $Z$. Its proof follows from standard Gaussian concentration inequalities. We postpone their proofs to \Cref{app:Green}.

\begin{lemma}\label{p:WtGbound}
For any $x,y\in \qq{N}$, $| Z_{xy}|\leq N^\fo /\sqrt N$ holds with overwhelmingly high probability over $Z$.

Adopt the notation and assumptions from \Cref{l:basicG}.
Then for any vertex set $\bX\subset \bW\cup \{b,b'\}$ or $\bX=\bT\cup \bX'$ with $\bX'\subset \bW\cup \{b,b'\}$, the following holds with overwhelmingly high probability over $Z$:
\begin{align}\begin{split}\label{e:Werror}
&
| (Z G^{(\bX)}(z,t))_{x y} |\lesssim N^{\fo}\sqrt{\frac{\Im[m_t(z)]}{N\eta}} , \text{ for } x\in \bX, y\not\in\bX; \\ 
&|(Z G^{(\bX)}(z,t)Z)_{x y}- m_t(z)\delta_{xy}|\leq 
N^\fo\sqrt{\frac{\Im[m_t(z)]}{N\eta}}, \text{ for } x,y\in \bX;  \\
&\sum_{y\not\in \bX} (Z G^{(\bX)}(z,t))_{x y}\lesssim N^\fo,\text{ for }  x\in \bX;\quad \frac{1}{N}\sum_{y\not\in \bX}|(Z G^{(\bX)}(z,t))_{x y} |^2
\lesssim
N^{2\fo}\frac{\Im[m_t(z)]}{N\eta}, \text{ for }  x\in \qq{N}.
\end{split}\end{align}
We further assume $\tcG\in \Omega$. Then for any vertex set $\bX=\bT\cup \bX'$ with $\bX'\subset \{b,b'\}$, the following holds with overwhelmingly high probability over $Z$:
\begin{align}\begin{split}\label{e:Werror2}
&
| (Z \wt G^{(\bX)}(z,t))_{x y} |\lesssim N^{\fo}\sqrt{\frac{\Im[m_t(z)]}{N\eta}}, \text{ for } x\in \bX, y\not\in\bX;\\
&
|(Z \wt G^{(\bX)}(z,t)Z)_{x y}- m_t(z)\delta_{xy}|\lesssim 
N^\fo\sqrt{\frac{\Im[m_t(z)]}{N\eta}}, \text{ for } x,y\in \bX;  \\
&\sum_{y\not\in \bX} (Z \wt G^{(\bX)}(z,t))_{x y}\lesssim N^\fo,\text{ for }  x\in \bX.
\end{split}\end{align}
\end{lemma}

\section{Main Results and Proof Outlines}\label{s:outline}

In this section, we present our main results: \Cref{t:recursion} and \Cref{t:correlation_evolution}, which address the high-moment estimates of the self-consistent equation and the microscopic version of the loop equations at the edge. Using these results as inputs, we proceed to prove both \Cref{thm:eigrigidity} (optimal eigenvalue rigidity) and \Cref{t:universality} (edge universality).

The proof of optimal eigenvalue rigidity (\Cref{thm:eigrigidity}), based on \Cref{t:recursion}, follows a standard argument in random matrix theory, which we defer to \Cref{s:proofrigidity}. For edge universality (\Cref{t:universality}), we provide the proof in \Cref{s:universality}, leveraging \Cref{t:correlation_evolution} as the key input.

The proof of \Cref{t:recursion} is developed through an iteration scheme. Foundational concepts are introduced in \Cref{s:forest} and \Cref{t:admissible}, while the detailed proof is outlined in \Cref{s:proofoutline}.

\subsection{Setting and notation}\label{s:setting3}

Fix $d\geq 3$. We recall the spectral domain $\bf D$ from \eqref{e:D}, and parameters $\fo\ll \ft\ll\fb\ll\fc\ll\fg$ from \eqref{e:parameters}. Fix time $t\leq N^{-1/3+\ft}$. We recall $\varrho_d(x,t)$, $\md(z,t)$ and $E_t$ from \eqref{e:defrhodt} and \eqref{e:edgeeqn2}. For any parameter $z\in \bf D$, we denote $\eta=\Im[z]$, $\kappa=\min\{|\Re[z]-E_t|, |\Re[z]+E_t|\}$, and $z_t=z+t\md(z,t)=z+t\md(z_t)$ (recall from \eqref{e:defw}). 
We also recall the matrix $H(t)$, its Green's function $G(z,t)$, its Stieltjes transform $m_t(z)$, the functions $Y_\ell, X_\ell$ and the quantity  $Q_t(z)$ from \eqref{e:Ht} , \eqref{e:Gt}, \eqref{def_mtz},  \eqref{def:Y} and \eqref{e:Qsum}. For simplicity of notation, we write 
\begin{align}\label{e:defYt}
 Y_t(z)=Y_\ell(Q_t(z), z+tm_t(z)),\quad X_t(z)=X_\ell(Q_t(z), z+tm_t(z)).
\end{align}
We recall the sets $\Omega$ and $\oOmega$ of $d$-regular graphs from \Cref{def:omegabar} and \Cref{thm:prevthm0}. Recall that for $\cG\in \Omega$,  the statements in \Cref{thm:prevthm} hold with overwhelmingly high probability over $Z$. 

Next, we introduce some more error terms. We notice that they depend on the graph $\cG$ and thus are random quantities. 
\begin{definition}\label{def:phidef}
For any $z\in \bC^+$ in the upper half plane with $\eta=\Im[z]$, we introduce the control parameter 
\begin{align}\begin{split}\label{e:defPhi}
 \Phi(z):&=\frac{\Im[m_t(z)]}{N\eta}+\frac{1}{N^{1-2\fc}},\\ 
 \Upsilon(z):&=|1-\del_1Y_\ell(Q_t(z), z+tm_t(z))|+t|\del_2 Y_\ell(Q_t(z), z+tm_t(z))|+(d-1)^{8\ell}\Phi(z).
\end{split}\end{align}
For any $p\geq 1$, we define the error parameter 
\begin{align}\begin{split}\label{eq:phidef}
&\Psi_p(z):= \bm1(\cG\in \Omega)
\Bigg[\frac{|Q_t(z)-Y_t(z)|}{N^{\fb/8}}(|Q_t(z)-Y_t(z)|+N^\fb\Upsilon(z)\Phi(z))^{p-1}\\
&+\Phi(z)\left(|Q_t(z)-Y_t(z)|+(1+N^\fb \Upsilon(z))\Phi(z)
+\frac{\Upsilon(z) }{N\eta}\right)(|Q_t(z)-Y_t(z)|+N^\fb\Upsilon(z)\Phi(z))^{p-2}
 \Bigg],
\end{split} \end{align} 
and more generally for any function $F_t(z)=F(Q_t(z), m_t(z),t)$, we define the error parameter
\begin{align}\begin{split}
\Xi_p(z;F_t(z))\label{def:Xip}
:= \bm1(\cG\in \Omega)\Bigg[&\frac{|Q_t(z)-Y_t(z)|}{N^{\fb/8}}(|F_t(z)|+N^\fb\Phi(z))^{p-1}\\
&+\Phi(z)\left(|F_t(z)|+(1+N^\fb)\Phi(z)
+\frac{1 }{N\eta}\right)(|F_t(z)|+N^\fb\Phi(z))^{p-2}\Bigg].
\end{split}\end{align}
 \end{definition}
 \begin{remark}
The quantity $\Phi(z)$ will be used to bound errors as in \Cref{l:basicG}, \Cref{c:expectationbound} and \Cref{p:WtGbound}.  The quantity $\Upsilon(z)$ will be used to bound the derivatives of $Y_\ell(Q_t(z), z+tm_t(z))$ with respect to $Q_t(z)$ and $m_t(z)$. The quantity $\Xi_p(z;F)$ in \eqref{def:Xip} is derived from $\Psi_p(z)$ in \eqref{eq:phidef} by replacing $|Q_t(z)-Y_t(z)|$ and  $\Upsilon(z)$ with $|m_t(z)-X_t(z)|$ and $1$, except for the first $|Q_t(z)-Y_t(z)|/N^{\fb/8}$ factor. The error parameters $\Psi_p(z)$ and $\Xi_p(z;F_t(z))$ each consists of two terms. The first term (containing the factor $|Q_t(z)-Y_t(z)|/N^{\fb/8}$) essentially arises from the error introduced in our iteration process, where, at the end of each step, a copy of $Q_t(z)$ is replaced by $Y_t$. For further discussion, see \Cref{r:QtoY}. The remaining errors are bound by the second term.
 \end{remark}

 \subsection{Main Theorems}\label{s:mainresult}
The following is our main result on the high moments estimate of the self-consistent equation.
\begin{theorem}
\label{t:recursion}Adopt the notation of \Cref{s:setting3}. Fix an integer $p\geq 1$, a spectral parameter $z\in \bf D$, and $z_1, z_2, \cdots, z_{p-1}\in \{z,\overline{z}\}$, and let $\fX_p=\Xi_p(z; m_t(z)-X_t(z))$ (recall from \eqref{def:Xip}). Then  the following holds
\begin{align}
\label{e:QY}&\bE\left[\bm1(\cG\in \Omega)(Q_t(z)-Y_t(z))\prod_{1\leq j\leq p-1}(Q_t(z_j)-Y_t(z_j))\right]\lesssim (d-1)^{2\ell}\bE[\Psi_p(z)],\\
\label{e:mz} &\bE\left[\bm1(\cG\in \Omega)(m_t(z)-X_t(z))\prod_{1\leq j\leq p-1}(m_t(z_j)-X_t(z_j))\right]\lesssim (d-1)^{2\ell}\bE[\fX_p(z)].
\end{align}
Moreover, if we further assume that $|z-E_t|\leq N^{-\fg}$ we have a refined estimate for $Q_t-Y_t$
\begin{align}\begin{split}\label{e:Qrefined_bound}
  & \phantom{{}={}}\frac{\cA^2}{\ell+1}\bE\left[\bm1(\cG\in \Omega)(Q_t(z)-Y_t(z))\prod_{1\leq j\leq p-1}(Q_t(z_j)-Y_t(z_j))\right]\\
  &+\bE\left[\bm1(\cG\in \Omega)\frac{\del_z m_t(z)}{N}\prod_{1\leq j\leq p-1}(Q_t(z_j)-Y_t(z_j))\right] \\
&+   \sum_{1\leq j\leq p-1}\frac{2}{ N^2}\bE\left[\left(\frac{1-\del_1Y_\ell(Q_t(z_j),z_j+tm_t(z_j))}{\cA }-t\del_2 Y_\ell(Q_t(z_j),z_j+tm_t(z_j))\right)\right.\\
&\left.\times\bm1(\cG\in \Omega)\del_{z_j}\left(\frac{m_t(z)-m_t(z_j)}{z-z_j} \right)\prod_{1\leq i\neq j\leq p-1}(Q_t(z_i)-Y_t(z_i))\right]=\OO\left(\frac{N^\fo\bE[\Psi_p(z)]}{(d-1)^{\ell/2}}\right),
\end{split}\end{align}
where the constant $\cA$ is from \eqref{e:edge_behavior}. If $|z+E_t|\leq N^{-\fg}$, an analogous statement holds after multiplying the first term in \eqref{e:Qrefined_bound} by $-1$. 
\end{theorem}

For simplicity of notation, in this article, we prove \Cref{t:recursion} only for the case where $z=z_1=z_2=\cdots=z_{p-1}$. The general case can be proven using the same argument and is therefore omitted.

\begin{remark}
    Inside the bulk of spectrum, the size of the quantities in \Cref{def:phidef} and $\fX_p=\Xi_p(z; m_t(z)-X_t(z))$ are given by
\begin{align*}
    \Phi(z), |Q_t(z)-Y_t(z)|, |m_t(z)-X_t(z)|\lesssim \frac{N^{\oo(1)}}{N\eta}, \quad \Upsilon(z)\lesssim 1, \quad \Psi_p(z), \fX_p(z)\lesssim \left(\frac{N^{\oo(1)}}{N\eta}\right)^p.
\end{align*}
They imply $|Q_t(z)-\msc(z_t)|, |m_t(z)-\md(z_t)|\lesssim N^{\oo(1)}/N\eta$.

Close to the spectral edge, i.e. $|z\pm E_t|\lesssim N^{-2/3+\oo(1)}$ with $\Im[z]\gtrsim N^{-2/3+\oo(1)}$, the sizes of the quantities defined in \Cref{def:phidef} are given by:
\begin{align}\begin{split}\label{e:edge}
    &\Phi(z), |Q_t(z)-Y_t(z)|\lesssim \frac{N^{\oo(1)}}{N^{2/3}},  \quad |m_t(z)-X_t(z)|\lesssim \frac{N^{\oo(1)}}{N^{1/2}},\\
    &\Im[m_t(z)], \Upsilon(z)\lesssim \frac{N^{\oo(1)}}{N^{1/3}}, \quad \Psi_p(z)\lesssim \left(\frac{N^{\oo(1)}}{N^{2/3}}\right)^p, \quad \fX_p(z)\lesssim \left(\frac{N^{\oo(1)}}{N^{1/2}}\right)^p.
\end{split}\end{align}
The correction terms in \eqref{e:Qrefined_bound} are also of size $\OO((N^{\oo(1)}/N^{2/3})^p)$.
The estimates for $|m_t(z)-X_t(z)|$ and $\fX_p(z)$ in \eqref{e:edge} are worse than those for $|Q_t(z)-Y_t(z)|$ and $\Psi_p(z)$, because $\fX_p(z)$ (from \eqref{def:Xip}) does not include the additional $\Upsilon(z)$ factor, which is small near the spectral edge. Nevertheless, these estimates are sufficient to derive optimal bounds for  $Q_t(z)-\msc(z_t)$ and $m_t(z)-\md(z_t)$. Because the self-consistent equation is not singular with respect to $m_t-\md(z_t)$ at the spectral edge.
\end{remark}

\begin{remark}
For any $F_t(z)=F(Q_t(z), m_t(z),t)$ (recall from \eqref{def:Xip}) provided its derivatives satisfy bounds similar to those in \eqref{e:Yl_derivative}, our argument for \Cref{t:recursion} can be applied to show: 
\begin{align}\begin{split}\label{e:general_form}
&\bE\left[\bm1(\cG\in \Omega)(Q_t(z)-Y_t(z))\prod_{1\leq j\leq p-1}F_t(z_j)\right]\lesssim (d-1)^{2\ell}\bE[\Xi_p(z; F_t(z))],\\
&\bE\left[\bm1(\cG\in \Omega)(m_t(z)-X_t(z))\prod_{1\leq j\leq p-1}F_t(z_j)\right]\lesssim (d-1)^{2\ell}\bE[\Xi_p(z; F_t(z))].
\end{split}\end{align}
The statement \eqref{e:mz} is a special case of \eqref{e:general_form}, obtained by taking $F_t(z)=m_t(z)-X_t(z)$. We emphasize that near the spectral edge \eqref{e:QY} is stronger than simply setting $F_t(z)=Q_t(z)-Y_t(z)$ in \eqref{e:general_form}. This is because we can leverage the fact that the derivatives of $Q_t(z)-Y_t(z)$ with respect to $Q_t(z)$ and $m_t(z)$ are small, which gives the $\Upsilon(z)$ factors in \eqref{eq:phidef}.
\end{remark}

The following corollary provides optimal concentration of the Stieltjes transform $m_t(z)$ of the emprical eigenvalue distribution of $H(t)$. Given \Cref{t:recursion}, the proof of
\Cref{c:rigidity}
and the proof of optimal eigenvalue rigidity (\Cref{thm:eigrigidity}) based on \Cref{c:rigidity} follow a standard argument in random matrix theory, which we defer to \Cref{s:proofrigidity}.

\begin{corollary}\label{c:rigidity}
Adopt the notation of \Cref{s:setting3}. We recall the parameters $\fo\ll \ft\ll\fc\ll\fg$ from  \eqref{e:parameters}, and fix time $t\leq N^{-1/3+\ft}$.
Conditioned on $\cG\in \Omega$ (recall from \Cref{thm:prevthm0}), with overwhelmingly high probability the following holds:
For any $z\in \bf D$ (recall from \eqref{e:D}), with $\eta=\Im[z]$,
\begin{align}\label{e:mbond}
|Q_t(z)-\msc(z_t)|, |m_t(z)-m_d(z_t)|\lesssim \frac{N^{8\fc}}{N\eta},
\end{align}
and \begin{align}\label{e:equation_est}
|m_t(z)-X_t(z)|\lesssim  \frac{N^{2\fc}(\kappa+\eta)^{1/4}}{N\eta}+\frac{N^{6\fc}}{(N\eta)^{3/2}}.
\end{align}
If we further assume that $\min\{|z-E_t|, |z+E_t|\}\leq N^{-\fg}$, we have the following improved estimates
\begin{align}\label{e:mtmddiff}
&|Q_t(z)-\msc(z_t)|, |m_t(z)-m_d(z_t)|\lesssim   \frac{N^{8\fo}}{N\eta},
\end{align}
and 
\begin{align}\label{e:QQYY}
|Q_t(z)-Y_t(z)|\lesssim  \frac{N^{2\fo}(\kappa+\eta)^{1/2}}{N\eta}+\frac{N^{10\fo}}{(N\eta)^2}.
\end{align}
Moreover, if $|\Re[z]|\geq E_t$, and $\kappa=\min\{|\Re[z]-E_t|, |\Re[z]+E_t|\}$ satisfies  $N\eta\sqrt{\kappa+\eta}\geq N^{6\fo}$, then we have
\begin{align}\label{e:mtimproved}
&|Q_t(z)-\msc(z_t)|, |m_t(z)-m_d(z_t)|\lesssim   \frac{N^{2\fo}}{\sqrt{\kappa+\eta}}\left(\frac{1}{N\sqrt{\eta}}+\frac{1}{(N\eta)^2}\right).
\end{align}
\end{corollary}


\begin{remark}
    We recall from \eqref{e:parameters} that $\fo\ll \ft\ll \ell/\log_{d-1} N\ll \fb\ll \fc\ll\fg$. While both statements \eqref{e:mbond} and \eqref{e:mtmddiff} can be interpreted as $|m_t(z)-m_d(z_t)|\leq N^{\oo(1)}/(N\eta)$, the improved version \eqref{e:mtmddiff} asserts that the $\oo(1)$ quantity in the exponent can be made arbitrarily small, in particularly, independent from the time $t$ and the radius $\ell$.  
\end{remark}

Finally, we record the following microscopic version of the loop equations for random $d$-regular graphs near the spectral edge.
As we will show in \Cref{l:Ftbiaoda}, for $z$ close to the edge $E_t$,  the term $(\cA^2 /(\ell+1))(Q_t(z)-Y_t(z))$ in \eqref{e:key_cancel} is, up to a negligible error,  equal to the quadratic 
\begin{align}
    (m_t(z)-\md(z,t)) ^2+ 2(m_d(z,t)-m_d(E_t,t))(m_d(z)-\md(z,t)).
\end{align}
In the scaling limit, the rescaled Stieltjes transform $N^{1/3} (m_t(2+w/(\cA N)^{2/3})-\md(2+w/(\cA N)^{2/3}),t)/\cA^{2/3}$ converges to $S(w)-\sqrt w$ in the microscopic loop equation \eqref{e:loop2}.

\begin{theorem}\label{t:correlation_evolution}
Adopt the notation of \Cref{s:setting3}. We recall the parameter $\ft$ from  \eqref{e:parameters}, and fix time $t\leq N^{-1/3+\ft}$.
We introduce the microscopic window:
 \begin{align}\label{e:micro_window}
     {\bf M}:=\{w\in \bC:N^{-2/3-\ft}\leq |\Im[w]|\leq N^{-2/3+\ft}, -N^{-2/3+\ft}\leq \Re[w]\leq N^{-2/3+\ft}\}.
 \end{align}
  Take any $p\geq 1$, $w, w_1, w_2,\cdots, w_{p-1}\in {\bf M}$, $z= E_t+w$ and $z_j=\pm E_t+w_j$ for $1\leq j\leq p-1$, the following holds
\begin{align}\begin{split}\label{e:key_cancel}
    & \phantom{{}={}}\frac{\cA^2 }{\ell+1}\bE\left[\bm1(\cG\in \Omega)(Q_t(z)-Y_t(z))\prod_{1\leq j\leq p-1}(m_t(z_j)-\md(z_j,t))\right]\\
    &
+\bE\left[\bm1(\cG\in \Omega)\frac{\del_z m_t(z)}{N} \prod_{1\leq j\leq p-1} (m_t(z_j)-\md(z_j,t))\right]\\
&+  \sum_{1\leq j\leq p-1}\frac{2}{N^2 }\bE\left[\bm1(\cG\in \Omega)\del_{z_j}\left(\frac{m_t(z)-m_t(z_j)}{z-z_j}\right)\prod_{1\leq i\neq j\leq p-1}(m_t(z_i)-\md(z_i,t))\right]\\
&=\OO\left(\frac{N^{2(p+1)\ft}}{N^{(p+1)/3}(d-1)^{\ell/2}}\right).
\end{split}\end{align}
If $z=-E_t+w$, an analogous statement holds after multiplying the first term in \eqref{e:key_cancel} by $-1$. 
\end{theorem}

The equation \eqref{e:key_cancel} is essentially the same as \eqref{e:Qrefined_bound}, except for replacing $(Q_t(z_j)-Y_t(z_j))$ with $(m_t(z_j)-\md(z_j,t))$. Additionally, the  factor 
$(1-\del_1Y_\ell(Q_t(z_j),z_j+tm_t(z_j)))/\cA -t\del_2 Y_\ell(Q_t(z_j),z_j+tm_t(z_j))$ in \eqref{e:Qrefined_bound}, which arises from derivatives of $Q_t(z_j)-Y_t(z_j)$ with respect to $Q_t(z_j)$ and $m_t(z_j)$, is absent in \eqref{e:key_cancel}. This is because, for $(m_t(z_j)-\md(z_j,t))$, such derivatives are simply $1$. The proofs of \eqref{e:key_cancel} and \eqref{e:Qrefined_bound} are identical, we provide a sketch of the proof for  \eqref{e:key_cancel} in \Cref{s:proofoutline}.

\subsection{Edge Universality}
\label{s:universality}
In this section we prove edge universality as stated in \Cref{t:universality}, using \Cref{c:rigidity} and \Cref{t:correlation_evolution} as input.

For sufficiently regular initial data, it has been proven in \cite{landon2017edge}, after short time the eigenvalue statistics at the spectral edge of \eqref{e:Ht} agree with GOE. A modified version of this theorem was proven in \cite{adhikari2020dyson}, which assumes that the initial data is sufficiently close to a nice profile. To use these results, we need to restrict  $H(0)$ to a subset, on which the optimal rigidity holds.
We denote $\cM$ to be the set of normalized adjacency matrices $H$, such that \eqref{e:mtmddiff} and \eqref{e:mtimproved} hold:
\begin{align}\label{e:defA}
\cM\deq \{H: \text{\eqref{e:mtmddiff} and \eqref{e:mtimproved} hold for $t=0$.} \}.
\end{align}
By Theorem \ref{c:rigidity}, we know that, conditioned on $\cG\in \Omega$, the event that our randomly sampled matrix lies in $\cM$ holds with overwhelmingly high probability.

The following theorem from \cite[Theorem 2.2]{landon2017edge} states that for time $t\gg N^{-1/3}$ the fluctuations of extreme eigenvalues of $H(t)$ conditioning on $H(0)\in \cM$ are given by the Tracy-Widom$_1$ distribution.

\begin{theorem}
\label{t:universalityHt}
Adopt the notation of \Cref{s:setting3}. We recall the parameter $\ft$ from  \eqref{e:parameters}, fix time $t= N^{-1/3+\ft}$, and recall $E_t$ from \eqref{e:edgeeqn2}.
Let $H(t)=H+\sqrt{t}Z$ be as in \eqref{e:Ht}, with $H\in \cM$ (recall from \eqref{e:defA}).   Fix $k\geq 1$ and $s_1,s_2,\cdots, s_k \in \bR$. There exists a small $\fa>0$, such that  the eigenvalues  $\la_1(t)=d/\sqrt{d-1}\geq \la_2(t)\geq \la_{3}(t)\geq \cdots \la_{N-1}(t)\geq \la_N(t)$ of $H(t)$ satisfy
\begin{align*}\begin{split}
&\phantom{{}={}} \bP_{H(t)}\left( (\cA N)^{2/3} ( \lambda_{i+1}(t) - E_t )\geq s_i,1\leq i\leq k \right)\\
&= \bP_{\mathrm{GOE}}\left( N^{2/3} ( \mu_i - 2  )\geq s_i,1\leq i\leq k \right) +\OO(N^{-\fa}),
\end{split}\end{align*}
where $\mu_1\geq \mu_2\geq \cdots \geq\mu_N$ are the eigenvalues of the GOE.
The analogous statement holds for the smallest eigenvalues.
\end{theorem}

\begin{remark} \label{rem:independence}
By an appropriate modification of the analysis of Dyson Brownian motion from~\cite{landon2017edge, adhikari2020dyson}, Proposition~\ref{t:universalityHt} also holds for the joint distribution of the $k$ largest and smallest eigenvalues. In particular, this implies that, under the same assumption as in the previous proposition, the asymptotic joint distribution of $(\cA N)^{2/3}(\lambda_2(t) - E_t, -\lambda_N(t) + E_t)$ is a pair of independent Tracy--Widom$_1$ distributions. 
\end{remark}

In this section we prove the following short-time comparison result for the edge eigenvalue statistics of $H(t)$. 
\begin{proposition}\label{thm:comp}
Adopt the notation of \Cref{s:setting3}.
We recall the parameter $\ft$ from  \eqref{e:parameters}, fix time $t= N^{-1/3+\ft}$, and recall $E_t$ from \eqref{e:edgeeqn2}.
Let $H(t)=H+\sqrt{t}Z$ be as in \eqref{e:Ht}.   Fix $k\geq 1$ and $s_1,s_2,\cdots, s_k \in \bR$. There exists a small $\fa>0$, such that  the eigenvalues  $\la_1(t)=d/\sqrt{d-1}\geq \la_2(t)\geq \la_{3}(t)\geq \cdots \la_{N-1}(t)\geq \la_N(t)$ of $H(t)$ satisfy
\begin{align}\begin{split}\label{e:comp1}
&\phantom{{}={}} \bP_{H(0)}\left( (\cA N)^{2/3} ( \lambda_{i+1}(0) - 2 )\geq s_i,1\leq i\leq k \right)\\
&= \bP_{H(t)}\left( (\cA N)^{2/3} ( \lambda_{i+1}(t) - E_t )\geq s_i,1\leq i\leq k \right) +\OO(N^{-\fa}).
\end{split}\end{align}
 The analogous statement holds for the smallest eigenvalues.
\end{proposition}

\begin{proof}[Proof of  \Cref{t:universality}]
Combine \Cref{t:universalityHt} and \Cref{thm:comp}.
\end{proof}

\begin{remark}\label{rem:ind2}
The proof of Proposition \ref{thm:comp} can be modified to show that the joint distributions of the $k$ largest and smallest eigenvalues of $H(0)$ and $H(t)$ are asymptotically the same. Also, Corollary \ref{c:rate} follows from combining with Remark \ref{rem:independence}.
\end{remark}

Proposition \ref{thm:comp} follows from the following comparison theorem for the multi-point correlation functions of edge statistics (equivalently the product of Stieltjes transform on microscopic scale), see \cite[Section 17]{erdHos2017dynamical}.

A key observation is that the time derivative of the multi-point correlation functions of the Stieltjes transform  of $H(t)$ is precisely governed by the microscopic loop equations, see \eqref{e:DBM_mt3}.

 \begin{proposition}\label{p:small_change}
 Adopt the notation of \Cref{s:setting3}. We recall the parameters $\fo\ll \ft$ from  \eqref{e:parameters}, take time $t\leq N^{-1/3+\ft}$, and recall $E_t$ from \eqref{e:edgeeqn2}. 
  For any $p\geq 1$,  $w_1, w_2,\cdots, w_p\in {\bf M}$ (recall from \eqref{e:micro_window}), and $z_j=E_t+w_j$ for $1\leq j\leq p$, the following holds
   \begin{align}\label{e:small_change}
     \del_t \bE\left[\bm1(\cG\in \Omega)\prod_{1\leq j\leq p} N^{1/3}(m_t(z_j)- \md(z_j,t))\right]\lesssim \frac{N^{2(p+2)\ft} N^{1/3}}{(d-1)^{\ell/2}}.
 \end{align}
By integrating \eqref{e:small_change} from $0$ to $t=N^{-1/3+\ft}$ we have 
 \begin{align*}
    \left.\bE\left[\bm1(\cG\in \Omega)\prod_{1\leq j\leq p} N^{1/3}(m_t(z_j)- \md(z_j,t))\right]\right|_{t=0}^{t=N^{-1/3+\ft}}\lesssim {(d-1)^{-\ell/4}},
 \end{align*}
provided that $(d-1)^{\ell/4}\gg N^{2(p+3)\ft}$.
 \end{proposition}

Towards proving \Cref{p:small_change}, we make some preliminary observations. The non-trivial eigenvalues of $H(t)$ have the law of Dyson Brownian motion $(\lambda_2(t)\geq \lambda_3(t)\geq \cdots\geq  \lambda_N(t))$ starting from the empirical eigenvalue of $H(0)$:
\begin{align*}
    \rd \lambda_\al(t)=\sqrt{\frac{2}{N}} \rd B_\al(t)+\frac{1}{N}\sum_{\beta\in \qq{2,N} \setminus\{\al\}}\frac{\rd t}{\lambda_\al(t)-\lambda_\beta(t)},\quad 2\leq \al\leq N,
\end{align*}
and its Stieltjes transfrom $m_t(z)$ satisfies the following stochastic differential equation
\begin{align}\label{e:eq1}
\rd m_t(z)=-\sqrt{\frac{2}{N^3}}\sum_{\al\in\qq{2, N}}\frac{\rd B_\al(t)}{(\la_\al(t)-z)^2}+\frac{1}{2}\del_z \left(m_t^2(z) +\frac{\del_z m_t(z)}{N}\right)\rd t.
\end{align}
We also recall the free convolution equation \eqref{e:selffc} for $\md(z,t)$. By taking time derivative on both sides, we get the following complex Burgers equation
\begin{align}\label{e:eq2}
\del_t \md(z,t)=\frac{1}{2}\del_z (\md(z,t)^2).
\end{align}

By taking the difference of \eqref{e:eq1} and \eqref{e:eq2}, we get
\begin{align}\label{e:DBM_mt}
\rd (m_t(z)- \md(z,t))
&=-\sqrt{\frac{2}{N^3}}\sum_{\al\in\qq{2,N}}\frac{\rd B_\al(z)}{(\la_\al(t)-z)^2}+\frac{1}{2}\del_z\left(m^2_t(z)-\md^2(z,t)+ \frac{\del_z m_t(z)}{N}\right)\rd t.
\end{align}
Let $z=E_t+w$ with $w\in \bf M$ (recall from \eqref{e:micro_window}). We also recall from \eqref{e:Etshift} that the spectral edge $E_t$  satisfies the differential equation $\del_t E_t=-\md(E_t,t)$. By plugging $z=E_t+w$ into \eqref{e:DBM_mt}, we get
\begin{align}\begin{split}\label{e:DBM_mt2}
&\rd (m_t(z)- \md(z,t))
=-\sqrt{\frac{2}{N^3}}\sum_{\al\in\qq{2,N}}\frac{\rd B_\al(t)}{(\la_\al(t)-z)^2}+\frac{1}{2}\del_z F_t(z_j)\rd t,\\
&F_t(z):=(m_t(z)-\md(z,t))^2+2(\md(z,t)-\md(E_t,t))(m_t(z)-\md(z,t))+ \frac{\del_z m_t(z)}{N}.
\end{split}\end{align}
The following lemma rewrites $F_t(z)$ in terms of $Q_t(z)-Y_t(z)$, which will be used later in \Cref{p:small_change}. 
\begin{lemma}\label{l:Ftbiaoda}
Adopt the notation of \Cref{s:setting3}. We recall the parameters $\fo\ll \ft\ll \fc$ from  \eqref{e:parameters}, take time $t\leq N^{-1/3+\ft}$, and recall $E_t$ from \eqref{e:edgeeqn2}.  For any $z=E_t+w$ with $w\in \bf M$ (recall from \eqref{e:micro_window}), conditioned on $ \Omega$, the following holds with overwhelmingly high probability:
    \begin{align}\label{e:Ftbiaoda}
    F_t(z)=
    \frac{\cA^2}{\ell+1}(Q_t(z)-Y_t(z))+\frac{\del_z m_t(z)}{ N}+\OO(N^{-5/6+10\fc}).
\end{align}
\end{lemma}
\begin{proof}
Denote $z_t=z+t\md(z,t)$. Thanks to \eqref{e:relation_zt_z} and $|z-E_t|=|w|\leq N^{-2/3+\ft}$, we have 
\begin{align}\begin{split}\label{e:mthaha0}
    \sqrt{z_t-2}&=\sqrt{\xi_t-2}+\sqrt{z-E_t}+\OO\left(t\sqrt{|z-E_t|}+t^2\right)\\
    &=\sqrt{\xi_t-2}+\sqrt{z-E_t}+\OO(N^{-2/3+2\ft})=\OO(N^{-1/3+\ft/2}),
\end{split}\end{align}
where in the last statement we used that $\xi_t-2=\cA^2 t^2/2+\OO(t^3)$ from \eqref{e:xi_behavior}. We also recall from \eqref{e:medge_behavior}
\begin{align}\label{e:mthaha}
    \msc(z_t)=-1+\OO(\sqrt{|z_t-2|}),\;
    1-\msc^2(z_t)=\OO(\sqrt{|z_t-2|}),\;
    \md(z_t)=-\frac{d-1}{d-2}+\OO(\sqrt{|z_t-2|}),
\end{align}
and, conditioned on $\Omega$, the following estimates from \eqref{e:equation_est} and \eqref{e:mtmddiff} hold with overwhelmingly high probability,
\begin{align}\begin{split}\label{e:crude}
&|m_t(z)-X_t(z)|\leq \frac{N^{2\fc}\sqrt{\kappa+\eta}}{N\eta}+\frac{N^{6\fc}}{(N\eta)^{3/2}}\leq \frac{N^{8\fc}}{N^{1/2}},\\
&|Q_t(z)-\msc(z_t)|, |m_t(z)-m_d(z_t)|\leq  \frac{N^{8\fo}}{N\eta}\lesssim \frac{N^{2\ft}}{N^{1/3}}.
\end{split}\end{align}
We also recall the expansion of $Y_t$ from \Cref{p:recurbound} (by taking $(\Delta, w,z)$ as $(Q_t(z), z+tm_t(z),z+t\md(z_t))$, 
\begin{align}\begin{split}\label{e:Q-YQ6}
&\phantom{{}={}}Q_t-Y_t=(Q_t-\msc(z_t))-(Y_t-\msc(z_t))\\
&= 
(1-\msc^{2}(z_t))(\ell+1)(Q_t-\msc(z_t))+(\ell+1)(Q_t-\msc(z_t))^2-(\ell+1)t (m_t-\md(z_t))\\
&+\OO(\ell^5(|z_t-2||Q_t-\msc(z_t)|+\sqrt{|z_t-2|}(t|m_t-\md(z_t)|+|Q_t-\msc(z_t)|^2))\\
&+\OO(\ell^5(|Q_t-\msc(z_t)|^3 + t^2|m_t-\md(z_t)|^2+ t|m_t-\md(z_t)||Q_t-\msc(z_t)|))\\
&= 
(\ell+1)\left((1-\msc^{2}(z_t))(Q_t-\msc(z_t))+(Q_t-\msc(z_t))^2-t (m_t-\md(z_t))\right)+\OO(N^{-1+8\ft}),
\end{split}
\end{align}
where we used \eqref{e:mthaha0}, \eqref{e:mthaha} and \eqref{e:crude} to bound the errors. 

Thanks to the expansion of $X_t$ from \eqref{e:Xrecurbound}, together with \eqref{e:mthaha}, we have
\begin{align*}
X_t-\md(z_t)&=X_t-\md(z,t)= \cA(Q_t-\msc(z_t)) 
\\
&+\OO(\ell^3(t|m_t-\md(z_t)|+\sqrt{|z_t-2|}|Q_t-\msc(z_t)|+|Q_t-\msc(z_t)|^2)).
\end{align*}
It follows by rearranging that
\begin{align}\begin{split}\label{e:mQrelation}
m_t-\md(z,t) &=m_t-X_t(Q)+\cA(Q_t-\msc(z_t)) +\OO(N^{-2/3+6\ft})\\
&=\cA(Q_t-\msc(z_t)) +\OO(N^{-1/2+8\fc}),
\end{split}\end{align}
where we used \eqref{e:mthaha0} and \eqref{e:crude} to bound the errors. 
By plugging \eqref{e:mQrelation} into \eqref{e:Q-YQ6}, we conclude
\begin{align}\label{e:AQ-Y}
    \frac{\cA^2(Q_t-Y_t)}{\ell+1}
    =(m_t-m_d(z,t))^2+(\cA (1-\msc^2(z_t))-t\cA^2)(m_t-m_d(z,t))+\OO(N^{-5/6+10\fc}).
\end{align}
For the factor in front of $(m_t-\md(z,t))$, we claim we can replace it with $2(\md(z_t)-\md(E_t,t))$. Indeed,
\begin{align}\begin{split}\label{e:xishu}
&\phantom{{}={}}2(\md(z_t)-\md(E_t,t))-\cA(1-\msc(z_t)^2)+t\cA^2\\
&=2(\md(z_t)-\md(\xi_t))-\cA(1-\msc(z_t)^2)+t\cA^2\\
&=2\cA(\sqrt{z_t-2}-\sqrt{\xi_t-2})-\cA 2\sqrt{z_t-2}+t\cA^2+\OO(|z_t-2|+|\xi_t-2|)\\
&=\OO(|z_t-2|+|\xi_t-2|)=\OO(N^{-2/3+2\ft}),
\end{split}\end{align}
where the third line follows from \eqref{e:mthaha}, and for the last line we used that $\xi_t-2=\cA^2 t^2/4+\OO(t^3)$ from \eqref{e:xi_behavior}.
By plugging \eqref{e:xishu} into \eqref{e:AQ-Y}, we conclude
\begin{align*}
    \frac{\cA^2(Q_t(z)-Y_t(z))}{\ell+1}=(m_t(z)-\md(z,t))^2+2(\md(z,t)-\md(E_t,t))(m_t(z)-\md(z,t))+\OO\left(\frac{N^{10\fc}}{N^{5/6}}\right),
\end{align*}
and the claim \eqref{e:Ftbiaoda} follows. 
\end{proof}

 \begin{proof}[Proof of \Cref{p:small_change}]
 Thanks to \eqref{e:DBM_mt2} and It{\^ o}'s formula, we have
\begin{align}\begin{split}\label{e:DBM_mt3}
   &\phantom{{}={}} \del_t \bE\left[\bm1(\cG\in \Omega)\prod_{1\leq j\leq p} (m_t(z_j)- \md(z_j,t))\right]\\
    &=\frac{1}{2}\sum_{1\leq i\leq p} \bE\left[\bm1(\cG\in \Omega)\del_z F_t(z_i)\prod_{j:j\neq i}(m_t(z_j)- \md(z_j,t))\right]\\
    &+\sum_{i\neq j}\bE\left[\frac{\bm1(\cG\in \Omega)}{N^3}\sum_{\al\in \qq{N}} \frac{1}{(\la_\al(t)-z_i)^2(\la_\al(t)-z_j)^2}\prod_{k:k\neq i,j} (m_t(z_k)- \md(z_k,t))\right].
\end{split}\end{align}
We notice that 
\begin{align*}
    \del_{z_j}\del_{z_i}\frac{m_t(z_i)-m_t(z_j)}{z_i-z_j}=\frac{1}{N}\sum_{\al\in \qq{N}} \frac{1}{(\la_\al(t)-z_i)^2(\la_\al(t)-z_j)^2}.
\end{align*}
The right-hand side of \eqref{e:DBM_mt3} is a sum for $1\leq i\leq p$ of the following quantity
\begin{align}\begin{split}\label{e:dz_exp}
&\del_{z_i}\bE\left[\bm1(\cG\in \Omega)F_t(z_i)\prod_{j:j\neq i}(m_t(z_j)- \md(z_j,t))\right.+\\
&\left.+\bm1(\cG\in \Omega)\sum_{j:j\neq i}\frac{2}{N^3}\del_{z_j}\frac{m_t(z_i)-m_t(z_j)}{z_i-z_j}\prod_{k:k\neq i,j} (m_t(z_k)- \md(z_k,t))\right].
\end{split}\end{align}

Thanks to \Cref{l:Ftbiaoda} and our main result \Cref{t:correlation_evolution} (with $z$ taken to be $z_i$), we get
\begin{align}\begin{split}\label{e:dz_exp2}
&\del_{z_i}\bE\left[\bm1(\cG\in \Omega)F_t(z_i)\prod_{j:j\neq i}(m_t(z_j)- \md(z_j,t))\right.+\\
&\left.+\bm1(\cG\in \Omega)\sum_{j:j\neq i}\frac{2}{N^3}\del_{z_j}\frac{m_t(z_i)-m_t(z_j)}{z_i-z_j}\prod_{k:k\neq i,j} (m_t(z_k)- \md(z_k,t))\right]\\
&\lesssim \frac{N^{10\fc}}{N^{5/6}}\bE\left[\bm1(\cG\in \Omega)\prod_{j:j\neq i}|m_t(z_j)- \md(z_j,t)|\right]+\OO\left(\frac{N^{2(p+1)\ft}}{N^{(p+1)/3}(d-1)^{\ell/2}}\right)\\
&=\OO\left(\frac{N^{2(p+1)\ft}}{N^{(p+1)/3}(d-1)^{\ell/2}}\right),
\end{split}\end{align}
where in the last statement we used \eqref{e:crude}.
To estimate \eqref{e:dz_exp}, we can take a small contour $\cC_i$ of radius $N^{-2/3-\ft}/10$ around $z_i$ and perform a contour integral on both sides of \eqref{e:dz_exp2}. It follows that
\begin{align}\label{e:dz_exp3}
    |\eqref{e:dz_exp}|\lesssim \frac{1}{2\pi\ri}\oint_{\cC_i}\frac{N^{2(p+1)\ft}}{N^{(p+1)/3}(d-1)^{\ell/2}}\frac{|\rd z|}{|z-z_i|^2}\lesssim \frac{N^{2(p+2)\ft}}{N^{(p-1)/3}(d-1)^{\ell/2}}.
\end{align}
The claim \eqref{e:small_change} follows from summing over \eqref{e:dz_exp3} for $1\leq i\leq p$. 
\end{proof}

\subsection{Switching Edges and the Forest}\label{s:forest}
As explained in the introduction \Cref{s:newstrategy}, our new strategy to prove \Cref{t:recursion} is an iteration scheme. At each iteration, we perform a local resampling and express the Green's function of the switched graph in terms of the original graph. 
Because almost all neighborhoods in the graph $\cG$ have no cycles, with high probability, the edges involved in local resamplings have large tree neighborhoods, and are typically far from each other. Therefore, we think of the edges involved as a forest. In this section we formalize this iterative scheme using a sequence of forests, see \Cref{fig:forest}, which encode all the edges involved in local resamplings. 

\begin{figure}
\centering
\includegraphics[scale=0.15, trim=12cm 0cm 0cm 0cm, clip]{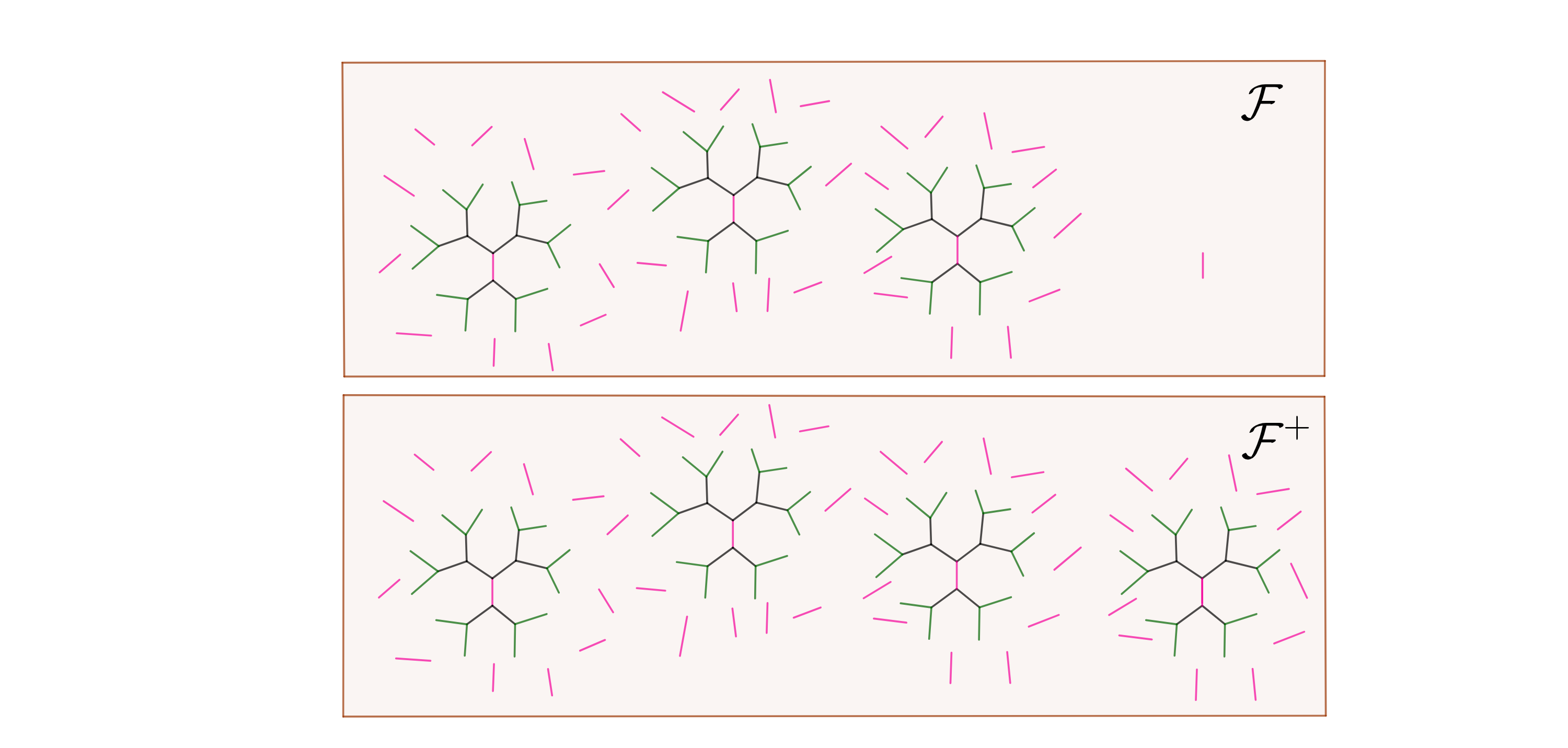}
\caption{Top Panel:
In the forest \(\mathcal{F}\), red edges represent core edges \(\mathcal{C}\). The used core edges belong to radius-\((\ell+1)\) balls, while each unused core edge \(\mathcal{C}^\circ\) forms its own connected component. Together, the red and green edges constitute the switching edges \(\mathcal{K}\).  Bottom Panel: We construct \(\mathcal{F}^+\) from \(\mathcal{F}\) by selecting an unused core edge (the rightmost red edge), expanding it into a radius-\((\ell+1)\) ball, and adding \(\mu\) new switching edges.
\label{fig:forest}}
\end{figure}

Fix $\mu:=d(d-1)^\ell$, which is the number of boundary edges for a $d$-regular tree truncated at level $\ell+1$. We start from a graph $\cF_0:=(\bfi_0=\{i_0,o_0\}, E_0:=\{i_0,o_0\})$.

We then construct $\cF_1$ from $\cF_0$, by 
\begin{itemize}
    \item extending the (directed) edge $(i_0,o_0)$ to a $d$-regular tree truncated at level $\ell$, $\cT_\ell(o_0)$ (with $o_0$ as the root vertex);
    \item adding $\mu$ boundary edges $\{ e'_\al\}_{\al\in\qq{\mu}}$ to $\cT_\ell(o_0)$, and get $\cT_{\ell+1}(o_0)$, which is the truncated $d$-regular tree at level $\ell+1$;
    \item adding $\mu$ new (directed) edges $\{ e_\al\}_{\al\in\qq{\mu}}$.
\end{itemize}
We denote $\cM'_1=\{ e'_\al\}_{\al\in\qq{\mu}}$,
$\cM_1=\{e_\al\}_{\al\in\qq{\mu}}$, so
\begin{align}\label{eq:forestdef1}
    \cF_1:=\cT_{\ell+1}(o_0)\cup \cM_1, \quad \cT_\ell(o_0)\bigcup\cM'_1=\cT_{\ell+1}(o_0).
\end{align}

In general, given the forest $\cF_s=(\bfi_s,E_s)$, with edge sets $\cM'_s, \cM_s$ constructed in last step, 
we construct $\cF_{s+1}=(\bfi_{s+1},E_{s+1})$ by 
\begin{itemize}
    \item picking one of the edges (from last step) $e=(i_s, o_s)\in \cM_s$  and extending $e$ to a $d$-regular tree  truncated at level $\ell$, $\cT_\ell(o_s)$ (with $o_s$ as the root vertex);
    \item adding $\mu$ boundary edges $\cM_{s+1}'=\{ e_\al'\}_{\al\in\qq{\mu}}$ to $\cT_\ell(o_s)$, and get $\cT_{\ell+1}(o_s)$, which is the truncated $d$-regular tree at level $\ell+1$;
    \item adding another $\mu$ new (directed) edges $\cM_{s+1}=\{e_\al\}_{\al\in\qq{\mu}}$.
\end{itemize} 
Explicitly, the new forest $\cF_{s+1}=(\bfi_{s+1}, E_{s+1})$ is 
\begin{align}\begin{split}\label{eq:forestdef2}
    &\cF_{s+1}=\cF_s\cup \cT_{\ell+1}(o_s)\cup \cM_{s+1},
\end{split}\end{align}
where $\bfi_{s+1}, E_{s+1}$ are its vertex set and edge set, respectively.

In the context of local resampling as described in \Cref{s:local_resampling}, the procedure above involves performing a local resampling around the edge $(i_{s}, o_s)$. Here, $\cT_\ell(o_s)$ represents the radius-$\ell$ neighborhood of $o_s$. The boundary edges are defined by
$\cM_{s+1}'=\{l_\al, a_\al\}_{\al\in \qq{\mu}}$,
and the new edges are given by
$\cM_{s+1}'=\{(b_\al, c_\al)\}_{\al\in \qq{\mu}}$.

The forest $\cF_s$ encodes the local resamplings up to the $s$-th step. Along with this, we denote
\begin{enumerate}
\item 
the set of switching edges $\cK_s=\{(i_0,o_0)\}\cup \cM_1\cup  \cM'_1\cup \cdots \cup \cM_s\cup  \cM'_s$; 
\item the
core edges of the forest $\cF_s$ as $\cC_s=(i_0,o_0)\cup \cM_1\cup \cM_2\cdots \cup\cM_s$. Each connected component of $\cF_s$ contains exactly one core edge;
\item the set of used core edges $\{(i_0,o_0), \cdots, (i_{s-1},o_{s-1})\}$ and the set of unused core edges
$\cC_s^\circ=\cC_s\setminus \{(i_0,o_0), \cdots, (i_{s-1},o_{s-1})\}$. Connected components containing a used core edge are truncated $d$-regular trees with radius $\ell + 1$; Other connected components consist of a single unused core edge.
\end{enumerate}

For the most of the later analysis, we will focus on one step, and simplify the notation as 
\begin{align*}\begin{split}
    &\cF=\cF_s, \quad \cK=\cK_s, \quad \cC=\cC_s,\quad \cC^\circ =\cC^\circ_s,\\
    &\cF^+=\cF_{s+1}, \quad \cK^+=\cK_{s+1}, \quad \cC^+=\cC_{s+1},\quad (\cC^\circ)^+ =\cC^\circ_{s+1}.
\end{split}\end{align*}
 More generally, see \Cref{fig:forest}, a forest $\cF$ associated with core edges $\cC$, and unused core edges $\cC^\circ$ is
\begin{align}\label{e:cFtocF+}
    \cF=\{e\}_{e\in \cC^\circ}\bigcup_{(i',o')\in \cC\setminus \cC^\circ}\cT_{\ell+1}(o'), \quad \cK=\cC\cup \bigcup_{(i',o')\in \cC\setminus \cC^\circ}(\cT_{\ell+1}(o')\setminus\cT_{\ell}(o')).
\end{align}
We construct $\cF^+$ from $\cF$ by picking an unused core edge $(i,o)\in \cC^\circ$, expanding it to $\cT_{\ell+1}(o)=\cT_\ell(o)\cup \cM'$, where $\cM'=\{(l_\al, a_\al)\}_{\al\in\qq{\mu}}$ represents $\mu$ boundary edges of $\cT_\ell(o)$, and adding $\mu$ new  edges $\cM=\{(b_\al, c_\al)\}_{\al\in\qq{\mu}}$. Then 
\begin{align}\label{e:cF++}
   \cC^+=\cC\cup \cM,\quad  (\cC^\circ)^+ =\cC^\circ\cup \cM\setminus \{(i,o)\},\quad \cF^+=\cF\cup \cT_{\ell+1}(o)\cup \cM, \quad
   \cK^+=\cK\cup \cM\cup \cM'. 
\end{align}

The forest $\cF$ encodes the edges involved in all previous local resamplings. To ensure these edges are sufficiently spaced apart and have large tree neighborhoods, the following indicator functions can be utilized.

\begin{definition}
    \label{def:indicator}
We consider a forest $\cF=(\bfi, E)$ (as in \eqref{e:cFtocF+}) with core edges $\cC$, viewed as a subgraph of a $d$-regular graph $\cG$. We write  $I(\cF, \cG)=1$ to denote the indicator function on the event that vertices close to core edges have radius $\fR$ tree neighborhoods, and core edges are distance $3\fR$ away from each other. Explicitly,  $I(\cF, \cG)$ is given by
\begin{align*}
   I(\cF, \cG):= \prod_{(x,y)\in \cF}A_{xy}\prod_{(b,c)\in \cC}\left(\prod_{x\in \cB_\ell(c,\cG)}\bm1(\cB_{\fR}(x, \cG) \text{ is a tree})\right) \prod_{(b,c)\neq (b',c')\in \cC}\bm1(\dist_\cG(c,c')\geq 3\fR).
\end{align*}
\end{definition}

Given a forest $\cF=(\bfi, E)$, we view it as a subgraph of a $d$-regular graph $\cG\in \Omega$. Then we perform a local resampling around an unused core edge $(i,o)$. 
In the following lemma, we show that with high probability with respect to the randomness of $\bfS$, the randomly selected edges $(b_\al, c_\al)$ are far away from each other, and have large tree neighborhood. In particular $\wt \cG=T_\bfS \cG\in \oOmega$.
\begin{lemma}\label{lem:configuration}
Adopt the notation of \Cref{s:setting3}.
Fix a $d$-regular graph $\cG\in \oOmega$, and a forest $\cF=(\bfi, E)$ (as in \eqref{e:cFtocF+}) viewed as a subgraph of $\cG$. 
Assume that $I(\cF,\cG)=1$ and $|\bfi|\leq N^{\fc/2}$. 
We consider the local resampling around an unused core edge $(i,o)\in \cC^\circ$, with resampling data $\{(l_\al, a_\al), (b_\al, c_\al)\}_{\al\in \qq{\mu}}$. 
We denote the set of resampling data $\sfF(\cG)\subset \sfS(\cG)$ such that the following holds 
\begin{enumerate}
    \item 
for any $\al\neq \beta\in \qq{\mu}$, $\dist_\cG(\{b_\al, c_\al\}\cup \bfi,\{b_\beta, c_\beta\} )\geq3\fR$; 
\item 
for any  $v\in \cB_{\ell}(\{b_\al, c_\al\}_{\al\in \qq{\mu}}, \cG)$, the radius $\fR$ neighborhood of $v$ is a tree.
\end{enumerate}
Then $\bP_\bfS(\sfF(\cG))\geq 1-N^{-1+2\fc}$ (where $\bP_{\bfS}(\cdot)$ is the probability with respect to the randomness of $\bfS$ as in \Cref{def:PS}). Also, for $\bfS\in \sfF(\cG)$ the following holds
\begin{enumerate}
    \item $\mu=d(d-1)^\ell$, ${\mathsf W}_{\bf S}=\qq{\mu}$ (recall from \eqref{Wdef}), and $\tcG=T_\bfS(\cG)\in \oOmega$; 
\item $I(\cF^+, \cG)=1$ and $I(\cF, \wt\cG)=1$.
\end{enumerate}
\end{lemma}

\begin{proof}[Proof of \Cref{lem:configuration}]
    We sequentially select $(b_\al, c_\al)$ uniformly random from $\cG^{(\bT)}$. For any fixed $\alpha$, we consider all edges that would break the requirements of the lemma. For the first requirement, we have
    \begin{align}\label{e:small1}
        \bP_{\bfS}(\dist_\cG(\bfi\cup_{1\leq \beta\leq \al-1}\{b_\beta, c_\beta\},\{b_\al, c_\al\} )\leq3\fR)\lesssim N^{-1}(|\bfi|+2\mu)d(d-1)^{3\fR}\leq N^{-1+3\fc/2}.
    \end{align}
    For the second requirement, we recall that  $\cG\in \oOmega$, in which all vertices except for $N^\fc$ many have radius $\fR$ tree neighborhood. Thus  
    \begin{align}\label{e:small2}
        \bP_{\bfS}(\cB_\fR(v,\cG) \text{ is not a tree for some } v\in \cB_\ell(\{b_\al, c_\al\}, \cG))\leq N^{-1} N^\fc d(d-1)^\ell\leq N^{-1+3\fc/2}.
    \end{align}
   The claim $\bP_\bfS(\sfF(\cG))\geq 1-N^{-1+2\fc}$ follows from union bounding over all $\alpha$ using \eqref{e:small1} and \eqref{e:small2}. 
   
    Under our assumption $I(\cF,\cG)=1$, the radius $\fR$ neighborhood of $o$ is a tree. Thus $\mu=d(d-1)^\ell$. Moreover, the neighborhoods $\cB_{3\fR/2}(o, \cG)$, and $\cB_{\fR}(\{b_\al, c_\al\}, \cG)$ for $\al\in \qq{\mu}$ are disjoint. It follows that $\dist_{\cG^{(\bT)}}(\{a_\al,b_\al,c_\al\}, \{a_\beta,b_\beta,c_\beta\})> {\fR/4}$ for all $\al\neq \beta\in \qq{\mu}$, and the subgraph $\cB_{\fR/4}(\{a_\al, b_\al, c_\al\}, \cG^{(\bT)})$ after adding the edge $\{a_\al, b_\al\}$ is a tree for all $\al \in \qq{\mu}$. We conclude that ${\mathsf W}_{\bf S}=\qq{\mu}$.

    Next we show that for any vertex $v\in \qq{N}$, the excess of $\cB_{\fR}(v, \tcG)$ is no bigger than that of $\cB_{\fR}(v, \cG)$. Then it follows that $\tcG\in \oOmega$. 
    If $\dist(v, \{a_\al, b_\al, c_\al\}_{\al\in \qq{\mu}})\geq \fR$, then $\cB_{\fR}(v, \tcG)=\cB_{\fR}(v, \cG)$, and the statement follows.
    Otherwise either $v\in \cB_{\fR}(\{a_\al\}_{\al\in \qq{\mu}}, \cG)\subset\cB_{3\fR/2}(o, \cG)$ or $v\in \cB_{\fR}(\{b_\al, c_\al\}, \cG)$ for some $\al\in \qq{\mu}$. We will discuss the first case. The second case can be proven in the same way, so we omit its proof. 
    If $v\in \cB_\ell(o, \cG)$, we denote $\min_{\al \in \qq{\mu}}\dist_\cG(v, \{l_\al\})=r\leq \fR$. Then $\cB_{\fR}(v, \tcG)$ is a subgraph of 
    $\cB_{\fR}(v, \cG)\cup_{\al\in \qq{\mu}} \cB_{\fR-r-1}(c_\al,\cG)$ after removing $\{(b_\al,c_\al)\}_{\al\in \qq{\mu}}$ and adding $\{(l_\al, c_\al)\}_{\al\in \qq{\mu}}$. By our construction of $\sfF(\cG)$, $\cB_{\fR-r-1}(c_\al,\cG)$ are disjoint trees. We conclude that $\cB_{\fR}(v, \tcG)$ is a tree.
    If $v\not\in \cB_\ell(o, \cG)$,
    we denote $\min_{\al \in \qq{\mu}}\dist_\cG(v, a_\al)=r\leq \fR$, then $\cB_{\fR}(v, \tcG)$ is a subgraph of 
    $\cB_{\fR}(v, \cG)\cup_{\al\in \qq{\mu}} 
    \cB_{\fR-r-1}(b_\al,\cG)$ 
    after removing $\{(b_\al,c_\al)\}_{\al\in \qq{\mu}}$ and adding $\{(a_\al, b_\al)\}_{\al\in \qq{\mu}}$. 
   Again by our construction of $\sfF(\cG)$, $\cB_{\fR-r-1}(c_\al,\cG)$ are disjoint trees, we conclude the excess of $\cB_{\fR}(v, \tcG)$ is at most that of $\cB_{\fR}(v, \cG)$.

The claim $I(\cF^+, \cG)=1$ follows from the construction of $\sfF(\cG)$. It also follows from  the above discussion that $\cB_\fR(v,\tcG)$ is a tree for any $v\in \cB_\ell(o, \cG)$. One can then check that $I(\cF,\tcG)=1$. This finishes the proof of the second statement in  \Cref{lem:configuration}.

\end{proof}

From \eqref{e:cFtocF+}, each connected component of $\cF$ is either an unused core edge or a radius-$(\ell+1)$ ball corresponding to a used core edge. The following proposition states that the total number of embeddings where $I(\cF,\cG)=1$ is approximately equal to that of choosing each connected component independently.

\begin{proposition}\label{p:sumA}
Given a forest $\cF=(\bfi, E)$ with core edges $\cC$ and unused core edges $\cC^\circ$ as in \eqref{e:cFtocF+}, as well as a $d$-regular graph $\cG\in \oOmega$, we have 
\[
\sum_{\bfi} I(\cF,\cG)=Z_{\cF}\left(1+\OO\left(\frac{1}{N^{1-2\fc}}\right)\right),
\]
where
\begin{align}\label{e:sumA2}\begin{split}
Z_{\cF}:=(Nd)^{|\cC|}\left([(d-1)!]^{1+d+d(d-1)+\cdots+d(d-1)^{\ell-1}}\right)^{|\cC\setminus\cC^\circ|}.
\end{split}\end{align}
Here $|\cC|$ is the number of core edges; and $|\cC\setminus\cC^\circ|$ is the number of used core edges. We remark that $Z_\cF$ depends only on the forest $\cF$ but not $\cG$.
\end{proposition}

\begin{proof}
We notice that $|\cC|$ is also the number of connected components of $\cF$, and $|\cC\setminus\cC^\circ|$ is the number of connected components in $\cF$ which are balls of radius $\ell+1$.

We can prove \eqref{e:sumA2} by 
induction on the number of connected components. If $\cF$ consists of a single edge $\cF=\{b,c\}$ which is an unused core edge, then 
\begin{align}\label{e:single_edge}
    \sum_{\bfi}I(\cF,\cG)=\sum_{b,c} A_{bc}\prod_{v\in \cB_\ell(c,\cG)}\bm1(\cB_{\fR}(v, \cG) \text{ is a tree}) )=Nd\left(1+\OO\left(\frac{1}{N^{1-3\fc/2}}\right)\right),
\end{align}
where we used the definition of $\overline\Omega$ from \Cref{def:omegabar}. 
If $\cF$ consists of a radius $(\ell+1)$-ball, corresponding to one used core edge, then we can also first sum over its core edge. The number of choices of this is the same as \eqref{e:single_edge}.  Then we sum over the remaining vertices. Each interior vertex of the radius-$(\ell+1)$ ball contributes a factor $(d-1)!$, since there are $(d-1)!$ ways to embed its children vertices. We get
\begin{align}\label{e:ball}
    \sum_{\bfi}I(\cF,\cG)=Nd[(d-1)!]^{1+d+d(d-1)+\cdots+d(d-1)^{\ell-1}}\left(1+\OO\left(\frac{1}{N^{1-3\fc/2}}\right)\right).
\end{align}

If the statement holds for $\cF$ with $\theta$ connected components, next we show it for $\cF$ with $\theta+1$ connected components. We can first sum over the indices corresponding to a connected component, fixing the other indices. 
 If it is a single edge, we get a factor similar to \eqref{e:single_edge}; if it is a radius-$(\ell+1)$ ball, we get a factor similar to \eqref{e:ball}. Next we can sum over the remaining $\theta$ connected components of $\cF$, which gives \eqref{e:sumA2}.

\end{proof}

\subsection{Admissible Functions} \label{t:admissible}
As discussed in \Cref{s:forest}, at each iteration, we have a forest $\cF = (\bfi, E)$ (recall from \eqref{e:cFtocF+}) consisting of switching edges $\cK$, core edges $\cC$, and unused core edges $\cC^\circ$. We need to estimate expectations of the following form: \begin{align}\label{e:R_ifirst}
\bE\left[I(\cF,\cG)\bm1(\cG\in \Omega) (G_{oo}^{(i)}-Y_t)R_\bfi\right],
\end{align}
where ${(i,o)} \in \cC^\circ$ is an unused core edge, and $R_\bfi$ is a monomial of (averaged) Green's function terms that depend on the forest $\cF$. In this section, we introduce the set of admissible functions in \Cref{def:pgen}, which classifies all possible terms $R_\bfi$ as in \eqref{e:R_ifirst}.

We also define the set of special edges 
\begin{align}\label{e:special_edge}
    \cV = \{\{u_1, v_1\}, \{u_2, v_2\}, \cdots, \{u_{p-1}, v_{p-1}\}\},
\end{align} 
where the elements are dummy variables. In \Cref{def:pgen}, $R_\bfi$ contains $p-1$ factors, each involves a summation over $u_i \sim v_j$ (i.e., a summation over all the edges of $\cG$).

Before stating \Cref{def:pgen} on admissible functions, we first introduce the local Green's function and the averaged Green's function.

\begin{definition}[Local Green's Function]
  We introduce the local Green's function $L(z,t,\cF,\cG)$: for $w,w'\in \qq{N}$, $z_t=z+t\md(z,t)$
\begin{align}\begin{split}\label{e:local_Green}
    &L_{ww'}(z,t,\cF,\cG)=P_{ww'}(\cB_{\fR}(\cF,\cG),z_t,\msc(z_t))\bm1(w,w'\in \cB_{\fR}(\cF,\cG)).
\end{split}\end{align}
 We also denote the centered version of the Green's function as
\begin{align}\label{e:G-L}
G^\circ(z,t)=G(z,t)-L(z,t,\cF,\cG).
\end{align}
When the context is clear, we will simply write $G^\circ(z,t), L(z,t,\cF,\cG)$ as $G^\circ, L$ for simplicity.

\end{definition}
\begin{remark}

 At each step, we expand the forest $\cF$ to a new forest $\cF^+$ by including local resampling data. This change also affects the local Green's function. However, given the event $I(\cF^+,\cG)=1$, the local Green's functions are compatible 
     \begin{align*}
        L_{sw}(z,t,\cF,\cG)=L_{sw}(z,t,\cF^+,\cG),\quad s\in \cB_\fR(\cF,\cG), \quad w\in \qq{N}.
    \end{align*}

Later, we need the local Green's function with one vertex removed: Let $(i,o)\in \cC\setminus\cC^\circ$, and recall $P^{(i)}$ from \eqref{e:defPi},
\begin{align}\label{e:defLi}
    L^{(i)}_{ww'}:= L^{(i)}_{ww'}(z,t,\cF, \cG):=P^{(i)}_{ww'}(\cB_{\fR}(\cF,\cG),z_t,\msc(z_t))\bm1(w,w'\in \cB_{\fR}(\cF,\cG)).
\end{align}
We also write the local Green's function of the switched graph, 
\begin{align}\label{e:deftL}
    \wt L_{ww'}:= L_{ww'}(z,t,\cF^+,\wt \cG)=P_{ww'}(\cB_{\fR}(\cF^+,\wt\cG),z_t,\msc(z_t))\bm1(w,w'\in \cB_{\fR}(\cF^+,\wt\cG)).
\end{align}
\end{remark}
\begin{definition}[Averaged Green's Function]
\label{def:av_Green}
For any used core edge $(i',o')\in \cC\setminus \cC^\circ$ with switching data $\{(l'_\al, a'_\al)\}_{\al\in \qq{\mu}}$,  $\sfA_{i'}=\{\al: \dist_\cG(l'_\al, i')=\ell+1\}$ (these are the indices for which, in $\cB^{(i')}(o',\cG)$, $l'_\alpha$ is connected to $o'$), and $w\in \qq{N}$, we define  $(\Av G^{\circ})_{o'w}$ as one of the following expressions:
\begin{align}\label{e:defAvGL}
    \sum_{\al\in \qq{\mu}}\frac{G^{\circ}_{l'_\al w}}{(d-1)^{\ell/2}},\quad  \sum_{\al\in \qq{\mu}}\frac{G^{\circ}_{a'_\al w}}{(d-1)^{\ell/2}},\quad \sum_{\al\in \sfA_{i'}}\frac{G^{\circ}_{l'_\al w}}{(d-1)^{\ell/2}},
    \quad \sum_{\al\in \sfA_{i'}}\frac{G^{\circ}_{l'_\al w}}{(d-1)^{\ell/2}}.
\end{align}
Similarly, we define $(\Av L)_{o'w}$ as one of the following:
\begin{align}\label{e:defAvL}
    \sum_{\al\in \qq{\mu}}\frac{L_{l'_\al w}}{(d-1)^{\ell/2}}, \quad\sum_{\al\in \qq{\mu}}\frac{L_{a'_\al w}}{(d-1)^{\ell/2}}, \quad \sum_{\al\in \sfA_{i'}}\frac{L_{l'_\al w}}{(d-1)^{\ell/2}},\quad
    \sum_{\al\in \sfA_{i'}}\frac{L_{l'_\al w}}{(d-1)^{\ell/2}}.
\end{align}
\end{definition}
\begin{remark}
As $\mu=\OO((d-1)^\ell)$, and $|\sfA_{i'}|=\OO((d-1)^\ell)$, the quantities in \eqref{e:defAvGL} and \eqref{e:defAvL} are not true averages,  as the denominators are $(d-1)^{\ell/2}$ rather than $(d-1)^{\ell}$. Instead, they should be interpreted as $\OO((d-1)^{\ell/2})$-weighted sums of (local) Green's functions. Fortunately, thanks to the local tree-like structure, these quantities are of size $\OO(1)$ (see \Cref{lem:boundaryreduction}).
\end{remark}
\begin{definition}[Admissible Function]
    \label{def:pgen}
Fix a large integer $p\geq 1$. Consider a forest $\cF=(\bfi, E)$ as defined in \eqref{e:cFtocF+}, with switching edges $\cK$, core edges $\cC$, unused core edges $\cC^\circ$, and special edges $\cV=\{\{u_j, v_j\}\}_{1\leq j\leq p-1}$. For any nonnegative integers $r\geq0$, we consider the set of vectors 
$\bmr=[r_{jk}]_{1\leq j\leq p-1, 0\leq k\leq 2}$, such that for any $1\leq j\leq p-1$
\begin{align}\label{e:defr}
    (r_{j0},r_{j1},r_{j2})\in\{(1,0,0), (3,0,0)\}\cup \{(2,r_1,r_2): r_1+r_2\geq 2 \text{ is even}\}.
\end{align}
We denote the set of admissible functions $\Adm(r, \bmr, \cF,\cG)$ where a function $R_{\bfi}\in \Adm(r,\bmr,\cF,\cG)$ is given by the formula 
\begin{align}\label{e:defRi'}
    R_\bfi=R
    \prod_{j=1}^{p-1}\cW_j.
\end{align}
Here is a breakdown of the components:
\begin{enumerate}
\item 
$R$ contains $r$ factors of the form 
\begin{align}\begin{split}\label{e:defcE1}
    &\{(G_{c c}^{(b)}-Q_t)\}_{(b,c)\in \cC^\circ}, \quad  \{ G_{c c'}^{(bb')}, G_{b c'}^{(b')}, G_{bb'},G_{cb'}\}_{(b,c)\neq (b',c')\in \cC^\circ},\\
&\{ G^{\circ}_{ss'}\}_{s,s'\in \cK},\quad \{(\Av G^{\circ})_{o'w}\}_{(i',o')\in \cC\setminus\cC^\circ, w\in\cK},\quad (Q_t-\msc(z_t)),\quad t(m_t-\md(z_t)).
\end{split}\end{align}

\item For $1\leq j\leq p-1$,  $\cW_j$ is associated with the special edge $\{u_j,v_j\}$ (as defined in \eqref{e:special_edge}), and there are three cases to consider: 
\begin{enumerate}
    \item 
If $r_{j0}=1$, then $\cW_j=Q_t-Y_t$.

\item If $r_{j_0}=2$, then $\cW_j$ is given by
\begin{align}\label{e:Scase2}
    \cW_j=\{1-\del_1 Y_\ell, -t\del_2 Y_\ell\}\times \frac{1}{Nd}\sum_{u_j\sim v_j\in\qq{N}}F_{u_j v_j},
\end{align}
where $F_{u_j v_j}$ is a product of $r_{j1}$ factors of the form 
\begin{align}\label{e:defcE0}
    \{G^{\circ}_{su_j},G^{\circ}_{sv_j}\}_{s\in \cK},\quad  \{(\Av G^{\circ})_{o'u_j},(\Av G^{\circ})_{o'v_j}\}_{(i',o')\in \cC\setminus \cC^\circ}, 
\end{align} 
$r_{j2}$ factors of the form 
\begin{align}\label{e:defLerror}
    \{L_{su_j},L_{sv_j}\}_{s\in \cK}, \quad \{(\Av L)_{o'u_j},(\Av L)_{o'v_j}\}_{(i',o')\in \cC\setminus \cC^\circ}, 
\end{align} 
and an arbitrary number of factors of the form $G_{u_jv_j}, 1/G_{u_ju_j}$.
\item 
If $r_{j0}=3$, then 
\begin{align}\label{e:Scase3}
    \cW_j= \frac{\{1-\del_1 Y_\ell, -t\del_2 Y_\ell\}\times (d-1)^\ell}{N}.
\end{align}
\end{enumerate}
 \end{enumerate}
\end{definition}  
\begin{remark}

   The three cases of $\cW_j$ in \Cref{def:pgen} correspond to $r_{j0}=1,2,3$. 
   Furthermore, except for the case $r_{j0} = 2$, the values $r_{j1} = r_{j2} = 0$ do not have any effect, nor does the special edge ${u_j, v_j}$. However, we can still interpret them as an average over the special edges.
   \begin{align*}
       \cW_j=\frac{1}{Nd}\sum_{u_j\sim v_j} (Q_t-Y_t) \text{ or } \cW_j=\frac{1}{Nd}\sum_{u_j\sim v_j} \frac{\{1-\del_1 Y_\ell, -t\del_2 Y_\ell\}(d-1)^\ell }{N}.
   \end{align*}
\end{remark}

\begin{remark}
    \label{r:r_fenlei}
    The array $\bmr$ in \Cref{def:pgen} falls into one of the following three categories:  
    \begin{enumerate}
        \item  For all $1\leq j\leq p-1$, we have $r_{j0}=1$ and $r_{j1}=r_{j2}=0$. In this case $\cW_j=(Q_t-Y_t)$ for each $1\leq j\leq p-1$.
        \item There exists some $1\leq j'\leq p-1$ such that 
        \begin{align}\label{e:1_expand}
           r_{j'0}=2, \quad  (r_{j'1}, r_{j'2})=(2,0), \quad  r_{j0}=1 \text{ for } j\in \qq{p-1}\setminus\{j'\}.
        \end{align}
        \item In all other cases, $\bmr$ satisfies \begin{align}\begin{split}\label{e:feasible}
            &\phantom{\text{ or }}|\{j\in\qq{p-1}: r_{j0}=3\}|\geq 1, \\&\text{ or } |\{j\in\qq{p-1}: r_{j0}=2\}|\geq 2,\\
            &\text{ or } |\{j\in\qq{p-1}: r_{j0}=2, (r_{j1},r_{j2})\neq (2,0)\}|\geq 1.
        \end{split}\end{align}
    \end{enumerate}
    
\end{remark}

We now give the general ways of bounding the terms involved in the admissible functions (recall from \Cref{def:pgen}). We postpone its proof to \Cref{s:adm_bound}.
\begin{lemma}\label{l:adm_term_bound}
Adopt the notation of \Cref{s:setting3}. Condition on that $\cG\in \Omega $ and $I(\cF,\cG)=1$, the following holds with overwhelmingly probability over $Z$
\begin{enumerate}
    \item Let $B$ be any factor in \eqref{e:defcE1} or \eqref{e:defcE0}, then
    \begin{align}\label{e:Bsmall}
        |B|\lesssim N^{-\fb}.
    \end{align} 
    \item For any $s\in \cK$:
    \begin{align}\label{e:naive-Ward}
        \frac{1}{N}\sum_{w\in \qq{N}} | G^{\circ}_{sw}|^2\lesssim N^\fo \Phi,\quad
        \frac{1}{N}\sum_{w\in \qq{N}} |L_{sw}|^2\lesssim \frac{\mathfrak R}{N}.
    \end{align}
    For any used core edge $(i',o')\in \cC\setminus \cC^\circ$,
    \begin{align}\label{e:av_naive-Ward}
        \frac{1}{N}\sum_{w\in \qq{N}} |(\Av G^{\circ})_{o'w}|^2\lesssim N^\fo \Phi,\quad
        \frac{1}{N}\sum_{w\in \qq{N}} |(\Av L)_{o'w}|^2\lesssim \frac{\mathfrak R}{N}.
    \end{align}
    \item For each $\cW_j$ as in \Cref{def:pgen}, the following holds: 
    
\begin{enumerate}
    \item If $r_{j0}=1$, then $r_{j1}=0$, and  $|\cW_j|=|Q_t-Y_t|$. 
    \item If $r_{j0}=2$, then
\begin{align}\label{e:Sjbound}
  |\cW_j|
   &\lesssim N^\fo \Upsilon\times
   \left\{
    \begin{array}{ll}
     (d-1)^{-(r_{j1}-2)\fb} \Phi & \text{ if } r_{j2}=0, \\
       (d-1)^{-\max\{r_{j1}-1,0\}\fb} \sqrt{\Phi/N}  & \text{ if } r_{j2}\geq 1. 
    \end{array}
    \right.
\end{align}
\item If $r_{j0}=3$, then $r_{j1}=0$, and  $|\cW_j|\lesssim (d-1)^\ell\Upsilon/N$.
  \end{enumerate}  
The cases above can be summarized as
\begin{align}\label{e:1_Sbound}
    (d-1)^{3\ell r_{j1}}|\cW_j|
    \lesssim \left\{\begin{array}{ll}
      |Q_t-Y_t|& \text{if } r_{j0}=1;\\
         (d-1)^{8\ell}\Upsilon\Phi& \text{if }(r_{j0},r_{j1}, r_{j2})=(2,2,0);\\
         N^{-\fb}\Upsilon\Phi& \text{remaining cases}. 
    \end{array} 
     \right.
\end{align}

\end{enumerate}
\end{lemma}

\begin{remark}
    At each step, we expand the forest $\cF$ to a new forest $\cF^+$ by including local resampling data. This change also affects the admissible set of functions, which now expands as follows:
    \begin{align*}
        \Adm(r,\bmr,\cF,\cG)\subset \Adm(r,\bmr,\cF^+,\cG).
    \end{align*}
\end{remark}


\subsection{Proof outline for \Cref{t:recursion}} \label{s:proofoutline}
As discussed in \Cref{s:forest}, at each iteration, we estimate \eqref{e:R_ifirst} by performing a local resampling around $(i,o)$. We will show that the expectation breaks down into an $\OO(1)$-weighted sum of terms in the same form, in the following sense.

\begin{definition}\label{def:O1sum}
    We say $\cU$ is an $\OO(1)$-weighted sum of terms in the set $\cR$, if 
    \begin{align*}
    \cU=\sum_{j\geq 1} \fc_j R_j, \quad R_j\in \cR,
    \end{align*}
    and the total weights $\sum_{j\geq 1}|\fc_j|=\OO(1)$.
\end{definition}

We begin with a weighted version of \eqref{e:R_ifirst}, as presented on the left-hand side of \eqref{e:maint} in the following proposition. Here, the additional factor $(d-1)^{(6r + 3\sum_{j=1}^{p-1} r_{j1})\ell}$ depends on the admissible function $R_\bfi$. The reader can interpret this as follows: each term in \eqref{e:defcE1} (contributing \(1\) to \(r\)) is associated with a factor \((d-1)^{6\ell}\), and each term in \eqref{e:defcE0} (contributing \(1\) to \(r_{j1}\)) is associated with a factor \((d-1)^{3\ell}\). Thanks to \Cref{l:adm_term_bound}, even with these factors, the size of the terms remains small.
These factors are introduced to ensure that all the expansions in this paper are \(\OO(1)\)-weighted sums of terms, as defined in \Cref{def:O1sum}. Specifically, combinatorial factors are absorbed into \((d-1)^{(6r + 3\sum_{j=1}^{p-1} r_{j1})\ell}\). In most cases, a factor of \((d-1)^{(3r + 3\sum_{j=1}^{p-1} r_{j1})\ell}\) would suffice. However, for \eqref{e:tG-Gexp1}, a factor in \eqref{e:defcE0} may transform into a term in \eqref{e:defcE1} and introduce an additional factor of \((d-1)^{3\ell}\).

The proposition below expresses the expectation of Green's functions of the graph $\cG$ in terms of the quantities of the new graph $\widetilde \cG$ after local resampling.

\begin{proposition}\label{p:add_indicator_function}
    Adopt the notation of \Cref{s:setting3}. Consider a forest $\cF=(\bfi, E)$ as in \eqref{e:cFtocF+} and a function $(G_{oo}^{(i)}-Y_t)R_\bfi$ with $R_\bfi\in \Adm(r,\bmr,\cG,\cF)$. We perform a local resampling around $(i, o) \in \cF$ using the resampling data ${\bf S}=\{(l_\al, a_\al), (b_\al, c_\al)\}_{\al\in\qq{\mu}}$, denoting the new graph as $\widetilde \cG = T_\bfS(\cG)$, with its corresponding Green's function $\widetilde G$. Then 
    \begin{align}\begin{split}\label{e:maint}
    &\phantom{{}={}}\frac{(d-1)^{(6r+3\sum_{j=1}^{p-1} r_{j1})\ell}}{Z_\cF}\sum_{\bfi}\bE\left[I(\cF,\cG)\bm1(\cG\in \Omega) (G_{oo}^{(i)}-Y_t)R_\bfi\right]\\
    &=\frac{(d-1)^{(6r+3\sum_{j=1}^{p-1} r_{j1})\ell}}{Z_{\cF^+}}\sum_{\bfi^+}\bE\left[I(\cF^+,\cG)\bm1(\cG,\tcG\in \Omega)(\widetilde G_{oo}^{(i)}-Y_t) \wt R_\bfi\right]+\OO(N^{-\fb/2}\bE[\Psi_p]).
\end{split}\end{align}
Here, $\widetilde R_\bfi$ is obtained by computing $R_\bfi$ for the graph $\widetilde \cG$.
\end{proposition}

The right-hand side of \eqref{e:maint} involves the Green's function of the switched graph $\tcG$. The following two propositions help evaluate them, and express them as $\OO(1)$-weighted sums of terms involving only the Green's function of the original graph $\cG$, with negligible error. More importantly, these terms match the structure of the left-hand side of \eqref{e:maint}. 

\Cref{p:iteration} addresses the first iteration step, stating that \eqref{e:IFIF} can be rewritten as an $\OO(1)$-weighted sum of three types of terms, as shown in \eqref{e:case1}, \eqref{e:case2} and \eqref{e:case3}.
For subsequent iteration steps, we start from one of terms from \eqref{e:case1}, \eqref{e:case2} and \eqref{e:case3}. \Cref{p:general} handles each case separately, as outlined in \eqref{e:higher_case1}, \eqref{e:higher_case2} and \eqref{e:higher_case3}. It shows that after further expansion, both \eqref{e:higher_case1} and \eqref{e:higher_case2} either maintain the same form with an additional factor of $(d-1)^{-\ell/2}$, or reduce to \eqref{e:case3}.
Importantly,  after further expansion,  \eqref{e:higher_case3} remains in the same form, either gaining a factor of  $(d-1)^{-\ell/2}$ factor, or an extra term in the form of \eqref{e:defcE1}. Therefore, with each iteration step, we either gain an additional factor of $(d-1)^{-\ell/2}$,  or introduce a term of the form \eqref{e:defcE1}, which are bounded by $N^{-\fb}$ thanks to \eqref{eq:infbound} and \eqref{eq:local_law}. After a finite number of iterations, all terms become negligible.

\begin{proposition}\label{p:iteration}
    Adopt the notation of \Cref{s:setting3}. 
 Given a forest $\cF=(\bfi=\{i,o\}, E=\{\{i,o\}\})$, $\cK=\cC=\{(i,o)\}$. We construct $\cF^+=(\bfi^+, E^+)$ (as given by \eqref{e:cF++}) by performing a local resampling around  $(i,o)\in \cF$ with resampling data ${\bf S}=\{(l_\al, a_\al), (b_\al, c_\al)\}_{\al\in\qq{\mu}}$, and denote $\wt \cG=T_\bfS(\cG)$. Then 
    \begin{align}\label{e:IFIF}
        \frac{1}{Z_{\cF^+}}\sum_{\bfi^+}\bE\left[I(\cF^+,\cG)\bm1(\cG,\tcG\in \Omega)(\widetilde G_{oo}^{(i)}-Y_t) (\wt Q_t-\wt Y_t)^{p-1}\right]
        =I_1+I_2+I_3+\cE,
    \end{align}
    where $|\cE|=\OO((d-1)^{2\ell}\bE[\Psi_p])$, and   
  \begin{enumerate}
  \item  $I_1$ is an $\OO(1)$-weighted sum of terms of the following form
    \begin{align}\label{e:case1}
       \frac{1}{(d-1)^{\fq^+\ell/2}Z_{\cF^+}}\sum_{\bfi^+}  \bE[\bm1(\cG\in \Omega)I(\cF^+, \cG)(G_{c_\al c_\al}^{(b_\al)}-Y_t)R_{\bfi^+}],
    \end{align}
    where $\fq^+\geq 0$ and $R_{\bfi^+}=(G_{c_\al c_\al}^{(b_\al)}-Q_t)(Q_t-Y_t)^{p-1}$; 
     \item  $I_2$ is an $\OO(1)$-weighted sum of terms of the following form
    \begin{align}\label{e:case2}
       \frac{1}{(d-1)^{\fq^+\ell/2}Z_{\cF^+}}\sum_{\bfi^+}  \bE[\bm1(\cG\in \Omega)I(\cF^+, \cG)(G_{c_\al c_\al}^{(b_\al)}-Y_t)R_{\bfi^+}],
    \end{align}
    where $\fq^+\geq 0$ and $R_{\bfi^+}$ is of the form: for $s,s'\in \{b_\al,c_\al\}$, $j\in\qq{p-1}$ and $w,w'\in \{u_j, v_j\}$ 
    \begin{align*}
        \frac{\{1-\del_1 Y_\ell, \del_2 Y_\ell\}}{Nd}\sum_{u_j\sim v_j\in \qq{N}} G^{\circ}_{s w} G^{\circ}_{s' w'}\times (Q_t-Y_t)^{p-2}.
    \end{align*}
    
      \item $I_3$ is an $\OO(1)$-weighted sum of terms of the following form
    \begin{align}\label{e:case3}
       \frac{(d-1)^{(6r^++3\sum_{j=1}^{p-1} r^+_{j1})\ell}}{(d-1)^{\fq^+\ell/2}Z_{\cF^+}}\sum_{\bfi^+}  \bE[\bm1(\cG\in \Omega)I(\cF^+, \cG)(G_{c_\al c_\al}^{(b_\al)}-Y_t)R_{\bfi^+}].
    \end{align}
Here, $R_{\bfi^+}\in \Adm(r^+,\bmr^+,\cF^+,\cG)$, which satisfies $\fq^+\geq 0$ and one of the following conditions:
    (a) $r^+\geq 2$ and $\{j: r^+_{j0}=2,3\}=0$; 
    or (b) $r^+\geq 1$ and $\bmr^+$ satisfies \eqref{e:1_expand};
    or (c) $\bmr^+$ satisfies \eqref{e:feasible}.

  \end{enumerate}
\end{proposition}

\begin{proposition}\label{p:general}
 Adopt the notation of \Cref{s:setting3}. Given a forest $\cF=(\bfi, E)$ and a function $(G_{oo}^{(i)}-Y_t)R_\bfi$ with $R_\bfi\in \Adm(r,\bmr,\cF,\cG)$. We construct $\cF^+=(\bfi^+, E^+)$ (as given by \eqref{e:cF++}) by performing a local resampling around  $(i,o)\in \cF$ with resampling data ${\bf S}=\{(l_\al, a_\al), (b_\al, c_\al)\}_{\al\in\qq{\mu}}$, and denote $\wt \cG=T_\bfS(\cG)$. 
 \begin{enumerate}
    
\item Let $R_\bfi=( G_{oo}^{(i)}- Q_t)( Q_t-Y_t)^{p-2}$. Then, up to an error of size $\OO((d-1)^{-\fq\ell/2}N^\fo\bE[\Psi_p])$,  
\begin{align}\label{e:higher_case1}
       \frac{1}{(d-1)^{\fq\ell/2}Z_{\cF^+}}\sum_{\bfi}  \bE[I(\cF^+,\cG)\bm1(\cG,\tcG\in \Omega)(\wt G_{oo}^{(i)}-Y_t)\wt R_\bfi]
    \end{align}
 can be rewritten as an $\OO(1)$-weighted sum of terms in the form \eqref{e:case3},  or \eqref{e:case2} with $\fq^+\geq \fq+1$.
 \item  
Let $R_{\bfi}$ be a term of the following form: for $s,s'\in \{i,o\}$, $j\in \qq{p-1}$ and $w, w'\in \{u_j, v_j\}$,
    \begin{align*}
        \frac{\{1-\del_1 Y_\ell, \del_2 Y_\ell\}}{Nd}\sum_{u_j\sim v_j\in \qq{N}} G^{\circ}_{s w} G^{\circ}_{s' w'}(Q_t-Y_t)^{p-2}.
    \end{align*}
Then, up to an error of size $\OO((d-1)^{-\fq\ell/2}N^\fo\bE[\Psi_p])$,  
 \begin{align}\label{e:higher_case2}
       \frac{1}{(d-1)^{\fq\ell/2}Z_{\cF^+}}\sum_{\bfi}  \bE[I(\cF^+,\cG)\bm1(\cG,\tcG\in \Omega)(\wt G_{oo}^{(i)}-Y_t)\wt R_{\bfi}]
    \end{align}
with $\fq\geq 0$ can be rewritten as an $\OO(1)$-weighted sum of terms in the form \eqref{e:case3}, or \eqref{e:case2} with $\fq^+\geq \fq+1$. 

 \item Let  $(r,\bmr)$ satisfy one of the conditions (a) $r\geq 2$ and $\{j: r_{j0}=2,3\}=0$;
 or (b) $r\geq 1$ and $\bmr$ satisfies \eqref{e:1_expand}; or (c) $\bmr$ satisfies \eqref{e:feasible}. Take $R_{\bfi}\in \Adm(r,\bmr,\cF,\cG)$. Then, up to an error of size $\OO(N^{-\fb/4}\bE[ \Psi_p])$,  
    \begin{align}\label{e:higher_case3}
       \frac{(d-1)^{(6r+3\sum_{j=1}^{p-1} r_{j1})\ell}}{(d-1)^{\fq \ell/2}Z_{\cF^+}}\sum_{\bfi^+}\bE\left[I(\cF^+,\cG)\bm1(\cG,\tcG\in \Omega)(\widetilde G_{oo}^{(i)}-Y_t) \wt R_\bfi\right]
    \end{align}
    can be rewritten as an $\OO(1)$-weighted sum of terms in the following form
     \begin{align}\label{e:case3_copy}
       \frac{(d-1)^{(6r^+ +3\sum_{j=1}^{p-1} r^+_{j1})\ell}}{(d-1)^{\fq^+\ell/2}Z_{\cF^+}}\sum_{\bfi^+}  \bE[\bm1(\cG\in \Omega)I(\cF^+, \cG)(G_{c_\al c_\al}^{(b_\al)}-Y_t)R_{\bfi^+}],
    \end{align}
    where $R_{\bfi^+}\in \Adm(r^+,\bmr^+,\cF^+,\cG)$, where either $\fq^+\geq \fq+1$, $r^+\geq r$; or  $\fq^+\geq \fq$, $r^+\geq r+1$. Also, for $1\leq j\leq p-1$, either  $r^+_{j0}>r_{j0}$ or $r^+_{j0}=r_{j0}$, $r^+_{j1}+r^+_{j2}\geq r_{j1}+r_{j2}$ and $r^+_{j2}\geq r_{j2}$.
 \end{enumerate}
\end{proposition}

The statement \eqref{e:QY} in \Cref{t:recursion} follows from iterating \Cref{p:add_indicator_function}, \Cref{p:iteration} and \Cref{p:general}.  To prove \eqref{e:Qrefined_bound}, we need to identify the leading order error terms from  \Cref{p:iteration} and \Cref{p:general}. These refined estimates are presented in the following three propositions.

\begin{proposition}\label{p:track_error1}
Adopt the notation and assumptions in \Cref{p:iteration}, and define the index set $\sfA_i := \{ \alpha \in \qq{\mu} : \dist_{\cT}(i, l_\al) = \ell+1 \}$ (these are the indices for which, in $\cT^{(i)}$, $l_\alpha$ is connected to $o$). We recall the local Green's functions $L$ and $ L^{(i)}$ (with vertex $i$ removed) from \eqref{e:local_Green} and \eqref{e:defLi}.
$I_1$ in \eqref{e:IFIF} is as follows
\begin{align}\label{e:case1_term}
    I_1=\sum_{\al\in \sfA_i}\sum_{\bfi^+}\frac{\msc^{2\ell}(z_t)L_{l_\al l_\al}^{(i)}}{(d-1)^{\ell+2}Z_{\cF^+}}\bE\left[I(\cF^+,\cG)\bm1(\cG\in \Omega) (G_{c_\al c_\al}^{(b_\al)}-Y_t)(G_{c_\al c_\al}^{(b_\al)}-Q_t)(Q_t-Y_t)^{p-1}\right].
\end{align}
$I_2$ in \eqref{e:IFIF} is as follows
\begin{align}\begin{split}\label{e:case2_term}
     I_2&=\sum_{\al\in\sfA_i}\sum_{\bfi^+}\frac{\msc^{2\ell}(z_t)}{(d-1)^{\ell+1} Z_{\cF^+}}\bE\left[I(\cF^+,\cG)\bm1(\cG\in \Omega)(G_{c_\al c_\al}^{(b_\al)}-Q_t)(Q_t-Y_t)^{p-2}\times \right.\\ &\times \left.\frac{1}{Nd}\sum_{u\sim v}\left((1-\del_1Y_\ell)\left(F_{vv}^{(\al)}-\frac{2G_{uv}}{G_{uu}}F^{(\al)}_{uv}+\frac{G^2_{uv}}{G^2_{uu}}F^{(\al)}_{uu}\right)+(-t\del_2 Y_\ell) F_{vv}^{(\al)} \right)\right],
\end{split}\end{align}
where 
\begin{align}\begin{split}\label{e:Fuv}
    F_{uv}^{(\al)}
    &=\left(\frac{1}{\sqrt{d-1}}+\frac{2\md(z_t) \msc(z_t)}{(d-1)\sqrt{d-1}}\right)( G^{\circ}_{uc_\al} G^{\circ}_{vb_\alpha}+ G^{\circ}_{ub_\al} G^{\circ}_{vc_\alpha})\\
    &+\frac{2\md(z_t)}{d-1}\left( G^{\circ}_{ub_\al} G^{\circ}_{ub_\alpha}+ G^{\circ}_{uc_\al} G^{\circ}_{uc_\alpha}\right).
\end{split}\end{align}
We have the following refined expression for the error term $\cE$ in \eqref{e:IFIF}
\begin{align}\label{e:daerrorE}
    \cE=\eqref{e:first_term0}+\eqref{e:second_term0}+\eqref{e:third_term0}+\OO(N^{-\fb/4} \Psi_p),
\end{align}
where 
\begin{align}
     \eqref{e:first_term0}   &\label{e:first_term0}=\sum_{\al,\beta\in \sfA_i}\sum_{\bfi^+}\frac{\msc^{2\ell}(z_t)}{(d-1)^{\ell+1}Z_{\cF^+}}\bE\left[I(\cF^+,\cG)\bm1(\cG,\wt \cG\in \Omega)(\wt G_{c_\al c_\beta}^{(\bT)}-G_{c_\al c_\beta}^{(b_\al b_\beta)}) (Q_t-Y_t)^{p-1}\right],\\
       \eqref{e:second_term0} &\label{e:second_term0}=\sum_{\al\in \sfA_i,\beta\in \qq{\mu}\atop 
       \al\neq  \beta}\sum_{\bfi^+}\frac{\msc^{2\ell}(z_t)(L_{l_\beta l_\beta}^{(i)}+L_{l_\al l_\beta}^{(i)})}{(d-1)^{\ell+2}Z_{\cF^+}}\bE\left[I(\cF^+,\cG)\bm1(\cG\in \Omega)(G_{c_\al c_\beta}^{(b_\al b_\beta)})^2(Q_t-Y_t)^{p-1}\right],
\end{align}
and
\begin{align}
        \begin{split}
\eqref{e:third_term0}&\label{e:third_term0}=\sum_{\al\neq \beta\in \sfA_i}\sum_{\bfi^+}\frac{2(p-1)\msc^{2\ell}(z_t)L_{l_\al l_\beta}}{ (d-1)^{\ell+2}Z_{\cF^+}} \Bigg[ I(\cF^+,\cG)\bm1(\cG\in \Omega)
 G_{c_\al c_\beta}^{(b_\al b_\beta)}(Q_t-Y_t)^{p-2}\\
 &\times \frac{1}{Nd}\sum_{u\sim v}\left((-t\del_2 Y_\ell)\left( G^{\circ}_{u c_\al}+\frac{\msc(z_t)  G^{\circ}_{u b_\al}}{\sqrt{d-1}}\right)\left( G^{\circ}_{u c_\beta}+\frac{\msc(z_t)  G^{\circ}_{u b_\beta}}{\sqrt{d-1}}\right)
    \right.\\
    &+(1-\del_1 Y_\ell)\left( G^{\circ}_{v c_\al}+\frac{\msc(z_t)  G^{\circ}_{v b_\al}}{\sqrt{d-1}}-\frac{G_{uv}}{G_{uu}}\left( G^{\circ}_{u c_\al}+\frac{\msc(z_t)  G^{\circ}_{u b_\al}}{\sqrt{d-1}}\right)\right) \\
    &\times\left.\left( G^{\circ}_{v c_\beta}+\frac{\msc(z_t)  G^{\circ}_{v b_\beta}}{\sqrt{d-1}}-\frac{G_{uv}}{G_{uu}}\left( G^{\circ}_{u c_\beta}+\frac{\msc(z_t)  G^{\circ}_{u b_\beta}}{\sqrt{d-1}}\right)\right)\right)\Bigg].
    \end{split}
    \end{align}
\end{proposition}

\begin{proposition}\label{p:track_error2}
Adopt the notation and assumptions in \Cref{p:general}. The error from expanding $I_1$ (from \eqref{e:case1_term}) as in \eqref{e:higher_case1}
is given by 
\begin{align}\begin{split}\label{e:track_error2}
      &\left(1-\left(\frac{\msc(z_t)}{\sqrt{d-1}}\right)^{2\ell+2}\right)   \frac{2\md(z_t)\msc^{6\ell}(z_t)}{(d-1)^{2\ell+3}}\times\\
      &\times \sum_{\al\neq\beta\in \sfA_i}\sum_{\bfi^+} \frac{1}{Z_{\cF^+}} \bE[\bm1(\cG\in \Omega)I(\cF^+, \cG)(G_{c_\al c_\beta}^{(b_\al b_\beta)})^2 (Q_t-Y_t)^{p-1}]
      +\OO(N^{-\fb/4} \bE[\Psi_p]).
\end{split}\end{align}
The error from expanding $I_2$ (from \eqref{e:case2_term}) as in \eqref{e:higher_case2}
is bounded by $\OO(N^{-\fb/4} \bE[\Psi_p])$. 
\end{proposition}

For $z$ close to the spectral edge $\pm E_t$, the following proposition gives refined estimates for the error terms in \eqref{e:first_term0}, \eqref{e:second_term0}, \eqref{e:third_term0} and \eqref{e:track_error2}.
\begin{proposition}\label{l:first_term}
Adopt the notation and assumptions in \Cref{p:iteration}, and recall $\cA$ from \eqref{e:defA}. For $z\in \bf D$ and $|z-E_t|\leq N^{-\fg}$, we have the following estimates for the terms involved in the error \eqref{e:daerrorE}: 
\begin{align}\label{e:first_term}
   \eqref{e:first_term0}&=\left(\frac{d(d-1)^{\ell}}{d-2} -\frac{d}{d-2}\right)\frac{1}{\cA^2}\bE\left[\bm1(\cG\in \Omega)\frac{\del_z m_t(z)}{N}(Q_t-Y_t)^{p-1}\right]+\OO\left(\frac{\bE[\Psi_p]}{(d-1)^\ell} \right),\\
\label{e:refine_Gccerror}
    \eqref{e:second_term0}&=\left(\frac{d+2}{d-2}-\frac{d(d-1)^{\ell}}{d-2}-(\ell+1)\right)\frac{1}{\cA^2}\bE\left[\bm1(\cG\in \Omega)\frac{\del_z m_t(z)}{N}(Q_t-Y_t)^{p-1}\right]+\OO\left(\frac{\bE[\Psi_p]}{(d-1)^\ell} \right),\\
\label{e:GIGG2}
\eqref{e:third_term0}&=-  \frac{ (p-1)(\ell+1)}{\cA^2}\bE\left[\left(\frac{1-\del_1Y_\ell}{\cA}-t\del_2 Y_\ell\right)\bm1(\cG\in \Omega)\frac{\del^2_z m_t}{N^2} (Q_t-Y_t)^{p-2}\right]+\OO\left(\frac{\bE[\Psi_p]}{(d-1)^\ell} \right).
\end{align}
Moreover, the error \eqref{e:track_error2} satisfies
\begin{align}\label{e:GIGG3}
       \eqref{e:track_error2}= -\frac{2}{d-2}\frac{1}{\cA^2}\bE\left[\bm1(\cG\in \Omega)\frac{\del_z m_t(z)}{N}(Q_t-Y_t)^{p-1}\right]+\OO\left(\frac{\bE[\Psi_p]}{(d-1)^\ell} \right).
\end{align}
If $|z+E_t|\leq N^{-\fg}$, analogous statements hold after multiplying the right-hand sides by $-1$. 
\end{proposition}


\begin{proof}[Proof of \Cref{t:recursion}]
We will only prove the estimate for $\bE[{\bm1(\cG\in \Omega)}(Q_t-Y_t)^{p}]$ which is from \eqref{e:QY} by taking $z=z_1=z_2=\cdots=z_{p-1}$. The general case, and the estimate for $(m_t(z)-X_t(z))$ can be proven in the same way, so we omit its proof.  
Denote $\cF_0=(\bfi_0=(i,o), E=\{(i,o)\})$, and recall $I(\cF_0, \cG)$ from \Cref{def:indicator}. $I(\cF_0, \cG)=1$ if every vertex $v\in \cB_\ell(o,\cG)$ has a radius $\fR$ tree neighborhood in $\cG$.  Then we split
\begin{align}\begin{split} \label{e:tt1}
\bE[{\bm1(\cG\in \Omega)}(Q_t-Y_t)^{p}]
&=\frac{1}{Nd}\bE\left[\sum_{i,o}A_{oi}(G_{oo}^{(i)}-Y_t)(Q_t-Y_t)^{p-1}{\bm1(\cG\in \Omega)}\right]\\
&=\frac{1}{Nd}\sum_{\bfi_0}\bE\left[I(\cF_0,\cG)(G_{oo}^{(i)}-Y_t)(Q_t-Y_t)^{p-1}{\bm1(\cG\in \Omega)}\right]\\
&+\OO\left(\frac{1}{N^{1-3\fc/2}}\right)\bE\left[|Q_t-Y_t|^{p-1}{\bm1(\cG\in \Omega)}\right],
\end{split}\end{align}
where we used that for $\cG\in \Omega$, $A_{oi}|G_{oo}^{(i)}|,|Y_\ell (Q)|\lesssim 1$ from \eqref{eq:infbound}, and $I(\cF_0,\cG)=1$ except for $\OO(d(d-1)^{\ell+\fR} N^{\fc})=\OO(N^{3\fc/2})$ vertices from \Cref{def:omegabar}.

We denote 
\begin{align*}
    R_{\bfi_0}=(Q_t-Y_t)^{p-1},
\end{align*}
which satisfies the conditions in \Cref{def:pgen} with $r=0$ and $r_{j0}=1$ for $1\leq j\leq p-1$. With this notation, we can rewrite \eqref{e:tt1} as 
\begin{align}\label{e:IIR}
   \bE[(Q_t-Y_t)^{p}{\bm1(\cG\in \Omega)}]
   =\frac{1}{Z_{\cF_0}}\sum_{\bfi_0}\bE\left[I(\cF_0,\cG){\bm1(\cG\in \Omega)}(G_{oo}^{(i)}-Y_t)R_{\bfi_0}\right]+\OO(N^{-\fb}\bE[\Psi_p]),
\end{align}
where $Z_{\cF_0}=Nd$.
The above expression aligns with the form of \Cref{p:add_indicator_function}, allowing us to apply \Cref{p:add_indicator_function}, \Cref{p:iteration} and \Cref{p:general} to begin the iteration process. After expanding \eqref{e:IIR} using \Cref{p:iteration}, we obtain three types of terms, as described in \eqref{e:case1}, \eqref{e:case2}, \eqref{e:case3}. We can then further expand these terms using \Cref{p:general}. The result in \Cref{p:general} essentially states that, after further expansion, \eqref{e:case1} and \eqref{e:case2} either maintain the same form with an additional factor of $(d-1)^{-\ell/2}$, or they transform into \eqref{e:case3}.
Similarly, after expansion, \eqref{e:case3} remains in the same form, either with an additional  $(d-1)^{-\ell/2}$ factor, or an extra term in the form of \eqref{e:defcE1}, which is bounded by $N^{-\fb}\ll (d-1)^{-\ell/2}$.
Therefore, after finitely many steps, namely $\OO(4p\log_{d-1}(N)/\ell)$, all terms are bounded by $\OO(N^{-2p})=\OO(\bE[\Psi_p]/N)$. Meanwhile the errors from \Cref{p:add_indicator_function}, \Cref{p:iteration} and \Cref{p:general} are all bounded by $\OO((d-1)^{2\ell}\bE[\Psi_p])$. This gives \eqref{e:QY}.

We will prove \eqref{e:Qrefined_bound} only for $|z-E_t|\leq N^{-\fg}$, the other case $|z+E_t|\leq N^{-\fg}$, can be established in exactly the same way. 
To prove \eqref{e:Qrefined_bound}, we must track the errors from the iteration process more carefully.  These refined error estimates are presented in \Cref{p:track_error1}, \Cref{p:track_error2} and \Cref{p:track_error2}.
By adding \eqref{e:first_term},\eqref{e:refine_Gccerror} and \eqref{e:GIGG2},  the error $\cE$ from \Cref{p:iteration} is given by
\begin{align}\begin{split}\label{e:final_error1}
&-\left(\ell+1-\frac{2}{d-2}\right)\frac{1}{\cA^2}\bE\left[\bm1(\cG\in \Omega)\frac{\del_z m_t(z)}{N}(Q_t-Y_t)^{p-1}\right]\\
&- (p-1)\frac{(\ell+1)}{\cA^2}\bE\left[\left(\frac{1-\del_1Y_\ell}{\cA}-t\del_2 Y_\ell\right)\bm1(\cG\in \Omega)\frac{\del^2_z m_t(z)}{N^2} (Q_t-Y_t)^{p-2}\right]+\OO\left(\frac{\bE[\Psi_p]}{(d-1)^\ell} \right).
\end{split}\end{align}

For \Cref{p:track_error2}, the error from expanding \eqref{e:higher_case3} is small, i.e. bounded by $\OO(N^{-\fb/4}\bE[\Psi_p]$.
The errors from expanding \eqref{e:higher_case1} or \eqref{e:higher_case2} with $\fq\geq 1$ are bounded by $\OO((d-1)^{-\ell/2}N^\fo\bE[\Psi_p])$. For $\fq=0$, the error from expanding \eqref{e:higher_case1} is given in \eqref{e:track_error2} and \eqref{e:GIGG3}
\begin{align}\label{e:final_error2}
 -\frac{2}{d-2}\frac{1}{\cA^2}\bE\left[\bm1(\cG\in \Omega)\frac{\del_z m_t(z)}{N}(Q_t-Y_t)^{p-1}\right]+\OO\left(\frac{\bE[\Psi_p]}{(d-1)^\ell} \right).
\end{align}
Thanks to \Cref{p:track_error2}, for $\fq=0$, the error from expanding \eqref{e:higher_case2} is bounded by $\OO(N^{-\fb/4}\bE[\Psi_p])$. 
The correction terms in the microscopic loop equation \eqref{e:Qrefined_bound} is obtained by summing the refined errors from \eqref{e:final_error1} and \eqref{e:final_error2}, and noticing that $2\lim_{z_j\rightarrow z}(m_t(z)-m_t(z_j))/(z-z_j)=\del^2_z m_t(z)$.

\end{proof}

\begin{proof}[Proof of \Cref{t:correlation_evolution}]
We will only sketch the proof of \eqref{e:key_cancel} when $z=z_1=z_2=\cdots=z_{p-1}=E_t+w$ with $w\in \bf M$ (recall from \eqref{e:micro_window}). Then $\eta=\Im[z]\in[N^{-2/3-\ft}, N^{-2/3+\ft}]$ and $\kappa=|\Re[z]-E_t|\leq N^{-2/3+\ft}$.

We introduce the following error parameter by taking $F_t(z)=m_t(z)-\md(z;t)$ in \eqref{def:Xip}
\begin{align}\begin{split}\label{e:defXi2}
&\fX_p=\Xi_p(z;m_t(z)-\md(z;t))
= \bm1(\cG\in \Omega)\Bigg[\frac{|Q_t(z)-Y_t(z)|}{N^{\fb/8}}(|m_t(z)-\md(z;t)|+N^\fb\Phi(z))^{p-1}\\
&+\Phi(z)\left(|m_t(z)-\md(z;t)|+(1+N^\fb)\Phi(z)
+\frac{1 }{N\eta}\right)(|m_t(z)-\md(z;t)|+N^\fb\Phi(z))^{p-2}\Bigg].\\
\end{split}\end{align}
Then we can repeat the same argument as for \eqref{e:Qrefined_bound}, with $Q_t-Y_t$ replaced by $m_t-\md(z_t)$, and obtain the following estimate
\begin{align}\begin{split}\label{e:expbound}
    & \phantom{{}={}}\frac{\cA^2}{\ell+1}\bE\left[\bm1(\cG\in \Omega)(Q_t-Y_t)(m_t-\md(z_t))^{p-1}\right]
+\bE\left[\bm1(\cG\in \Omega)\frac{\del_z m_t(z)}{N}  (m_t-\md(z_t))^{p-1}\right]\\
&+  (p-1)\bE\left[\bm1(\cG\in \Omega)\frac{\del^2_z m_t(z)}{N^2}(m_t-\md(z_t))^{p-2}\right]=\OO\left((d-1)^{-\ell/2}N^\fo\bE[\fX_p]\right).
\end{split}\end{align}
Thanks to \eqref{e:mtmddiff} and \eqref{e:QQYY}, conditioned on the $\cG\in\Omega$, we have with overwhelmingly high probability 
\begin{align}\begin{split}\label{e:mtmddiff_refine}
    &|m_t-\md(z_t)|\lesssim \frac{N^{2\fo}}{N\eta}\lesssim \frac{N^{2\ft}}{N^{1/3}},\quad 
|Q_t(z)-Y_t(z)|\lesssim  \frac{N^{2\fo}(\kappa+\eta)^{1/2}}{N\eta}+\frac{N^{10\fo}}{(N\eta)^2}\lesssim \frac{N^{3\ft}}{N^{2/3}},\\
&\Phi(z)\leq \frac{\Im[\md(z_t)]+|m_t-\md(z_t)|}{N\eta}+\frac{1}{N^{1-2\fc}} \lesssim  \frac{N^{3\ft/2}}{N^{2/3}},
\end{split}\end{align}
where in the third statement, we used  \eqref{eq:square_root_behave} that $\Im[\md(z_t)]\lesssim \sqrt{\kappa+\eta}\lesssim N^{-1/3+\ft/2}$. By plugging \eqref{e:mtmddiff_refine} into \eqref{e:defXi2} we conclude the following holds with overwhelmingly high probability 
\begin{align}\label{e:Xibound}
    N^\fo\fX_p\lesssim \frac{N^{2(p+1)\ft}}{N^{(p+1)/3}}.
\end{align}
The claim \eqref{e:key_cancel} follows from plugging \eqref{e:Xibound} into \eqref{e:expbound}.
\end{proof}

\section{Expansions of Green's Function Differences}\label{sec:expansions}
In this section, we gather estimates on the difference in Green's functions before and after local resampling. For Green's functions related to the center of the local resampling, we employ Schur complement formulas, with results detailed in \Cref{sec:schurlemmas}. This methodology has been previously utilized in similar contexts \cite{huang2024spectrum,bauerschmidt2019local} to establish the local law of random $d$-regular graphs.

For Green's function terms away from the center of the local resampling, we develop a novel expansion using the Woodbury formula, as stated in \Cref{lem:woodbury}. This expansion represents a reorganization of the resolvent identity. In prior research \cite{huang2023edge}, resolvent identities played a pivotal role in analyzing the changes induced by simple switching in Green's functions, yielding an expansion where the terms exhibit exponential decay in $1/\sqrt d$. This decay rate proves adequate when $d$ scales with the size of the graph; however, in our scenario, where $d$ remains fixed, the decay is too slow. Notably, in the new expansion introduced in \Cref{lem:woodbury}, the terms decay exponentially at a rate of $1/N^{\fb}$. 

\subsection{Setting and notation}
\label{s:setting}
In this section, let $d\geq 3$, and $\cG$ is a $d$-regular graph on $N$ vertices. Let $\cF = (\bfi, E)$ be a forest as in \eqref{e:cFtocF+}, with switching edges $\cK$, core edges $\cC$ and unused core edges $\cC^\circ$. We view $\cF$ as a subgraph of $\cG$. We construct $\cF^+ = (\bfi^+, E^+)$ (as in \eqref{e:cF++}) by performing a local resampling around $(i, o) \in \cC^\circ$ with resampling data ${\bf S}=\{(l_\al, a_\al), (b_\al, c_\al)\}_{\al\in \qq{\mu}}$ where $\mu=d(d-1)^{\ell}$. We denote $\cT=\cB_\ell(o,\cG)$ with vertex set $\bT$. Let the switched graph be $\widetilde \cG = T_\bfS(\cG)$. We recall the set $\Omega$ of $d$-regular graphs from \Cref{thm:prevthm0},
and the indicator function $I(\cdot,\cdot)$ from \eqref{def:indicator}. 
In this section we assume
\begin{align}\label{e:tGI}
    \cG, \widetilde \cG\in \Omega, \quad I(\cF^+,\cG)=1.
\end{align}
The second statement in \eqref{e:tGI} implies the following: by considering $\cF^+$ as a subgraph of $\cG$, each connected component of $\cF^+$ has a radius $\fR$-tree neigbhorhood, is separated from the others by a distance of at least $3\fR$. Moreover, \eqref{e:tGI} also implies that the statements in \Cref{thm:prevthm}, \Cref{l:basicG}, \Cref{c:expectationbound} and \Cref{p:WtGbound} all hold with overwhelmingly high probability over $Z$.

We recall the spectral domain $\bf D$ from \eqref{e:D}, and parameters $\fo\ll \ft\ll\fb\ll\fc\ll\fg$ from \eqref{e:parameters}. Fix time $t\leq N^{-1/3+\ft}$. We recall $\varrho_d(x,t)$, $\md(z,t)$ and $E_t$ from \eqref{e:defrhodt} and \eqref{e:edgeeqn2}. For any parameter $z\in \bf D$ we denote $\eta=\Im[z]$, $\kappa=\min\{|\Re[z]-E_t|, |\Re[z]+E_t|\}$, and $z_t=z+t\md(z,t)=z+t\md(z_t)$ (recall from \eqref{e:defw}).

We recall the matrix $H(t)$, its Green's function $G(z,t)$, its Stieltjes transform $m_t(z)$, the quantities  $Q_t(z)$,  $Y_t(z)=Y_\ell(Q_t(z),z+tm_t(z))$, and   $X_t(z)=X_\ell(Q_t(z),z+tm_t(z))$ from \eqref{e:Ht} , \eqref{e:Gt}, \eqref{def_mtz}, \eqref{e:Qsum}, and  \eqref{e:defYt}. 
We recall the control parameters $\Phi(z), \Upsilon(z)$ and $\Psi_p(z)$ from \eqref{e:defPhi} and \eqref{eq:phidef}. We recall the local Green's function $L(z,t)=L(z,t,\cF^+,\cG)$ from \eqref{e:local_Green}, and $G^\circ(z,t)=G(z,t)-L(z,t)$ from \eqref{e:G-L}. Assumption \eqref{e:tGI} implies that restricted on $\cB_{\fR}(\cF^+,\cG)$, $L$ agrees with the Green's function of the infinite tree as defined in \eqref{e:Gtreemkm}. We denote the corresponding quantities for the switched graph $\tcG$ as $\wt H(t)$, $\wt G(z,t)$, $\wt m_t(z)$, $\wt Q_t(z)$,  $\wt Y_t(z)$, $\wt X_t(z)$,$\wt \Phi(z)$, $\wt \Upsilon(z)$,$\wt \Psi_p(z)$, $\wt L(z,t)$ and $\wt G^\circ(z,t)$. If the context is clear, we may omit the dependence on $z$ and $t$. 

\subsection{Switching using 
the Schur complement formula}\label{sec:schurlemmas}
In this section we will use the Schur complement formula to study the Green's function after local resampling. 
We recall the local resampling and related notation from \Cref{s:local_resampling}. We also introduce the following S-Product term.
\begin{definition}[S-Product term]\label{d:S-product}
    Fix $r\geq 0$, we define $R_r$ to be a \emph{S-product term} of order $r$ (where ``\emph{S}" indicates that these terms arise from expansions using the Schur complement formula) if it is a product of $r$ factors in the following forms: 
\begin{align*}
   (G_{c_\al c_\al}^{(b_\al)}-Q_t), \quad  G_{c_\al c_{\beta}}^{(b_\al b_\beta)},\quad  (Q_t-\msc(z_t)), \quad t(m_t-\md(z_t)),\quad \al\neq \beta\in \qq{\mu}.
\end{align*} 
\end{definition}

In the following proposition, we derive an expansion for factors that are Green's function entries with both indices in $\{i,o\}$.
\begin{proposition}\label{lem:diaglem}
Adopt the notation and assumptions that $\cG, \widetilde \cG\in \Omega,  I(\cF^+,\cG)=1$ in \Cref{s:setting}, and define the index set $\sfA_i := \{ \alpha \in \qq{\mu} : \dist_{\cT}(i, l_\al) = \ell+1 \}$. The following holds with overwhelmingly high probability over $Z$:
\begin{enumerate}
    \item  $\widetilde G_{oo}^{(i)}-Y_t$ can be rewritten as a weighted sum  
\begin{align}\begin{split}\label{e:G-Y}
   \widetilde G_{oo}^{(i)}-Y_t
   &= \frac{\msc(z_t)^{2\ell+2}} {(d-1)^{\ell+1}}\sum_{\al\in\sfA_i}(G^{(b_\alpha)}_{c_\alpha c_\alpha}-Q_t)+\frac{\msc(z_t)^{2\ell+2}} {(d-1)^{\ell+1}}\sum_{\al\neq \beta\in\sfA_i}G^{(b_\alpha b_\beta)}_{c_\alpha c_\beta}
   + \cU+\cZ+\cE.
\end{split}\end{align}
where 
$\cU$ is an $\OO(1)$-weighted sum of terms in the form $(d-1)^{3(r-1)\ell}R_r$, where $r\geq 2$, and $R_r$ is an $S$-product term (see \Cref{d:S-product}) which contains $G_{c_\al c_\al}^{(b_\al)}-Q_t$ or $G_{c_\al c_\beta}^{(b_\al b_\beta)}$; and the error $\cE$ is given by
\begin{align}\begin{split}\label{e:defCE}
 \cE&= \frac{\fc} {(d-1)^{\ell+1}}\sum_{\al,\beta\in\sfA_i}(\wt G^{(\bT)}_{c_\alpha c_\beta}-G^{(b_\alpha b_\beta)}_{c_\alpha c_\beta})+\OO\left(\frac{1}{N^{\fb/2} } \sum_{\al, \beta\in \qq{\mu}} |\wt G^{(\bT)}_{c_\al c_\beta}-G^{(b_\al b_\beta)}_{c_\al c_\beta}| +\frac{1}{N^2}\right)\\
 &=\OO\left(\sum_{\al, \beta\in \qq{\mu}} |\wt G^{(\bT)}_{c_\al c_\beta}-G^{(b_\al b_\beta)}_{c_\al c_\beta}| +\frac{1}{N^2}\right).
\end{split}\end{align}

\item For $s,s'\in\{i,o\}$, $\widetilde G^\circ_{ss'}=(\widetilde G-\widetilde L)_{ss'}$ can be rewritten as a weighted sum, 
\begin{align}\begin{split}\label{e:G-Ydouble}
    \widetilde G^\circ_{ss'}
    &=\frac{1} {(d-1)^\ell}\sum_{\al\in \qq{\mu}}\fc_1(\bm1(\al\in\sfA_i))(G^{(b_\alpha)}_{c_\alpha c_\alpha}-Q_t)
    \\
   &+\frac{1} {(d-1)^\ell}\sum_{\al\neq \beta\in \qq{\mu}} \fc_2(\bm1(\al\in\sfA_i),\bm1(\beta\in\sfA_i)) G^{(b_\alpha b_\beta)}_{c_\alpha c_\beta}
   +\cU+ \cZ+\cE,
\end{split}\end{align}
where $|\fc_1(\cdot)|, |\fc_2(\cdot, \cdot)|\lesssim 1$; 
$\cU$ is an $\OO(1)$-weighted sum of terms of the form $\{Q_t-\msc(z_t), t(m_t-\md(z_t))\}$ or $(d-1)^{3(r-1)\ell}R_r$, where $r\geq 2$, and $R_r$ is an $S$-product term (see \Cref{d:S-product}); and the error $\cE$ is bounded by
\begin{align*}
 |\cE|\lesssim\sum_{\al, \beta\in \qq{\mu}} |\wt G^{(\bT)}_{c_\al c_\beta}-G^{(b_\al b_\beta)}_{c_\al c_\beta}| +\frac{1}{N^2}.
\end{align*}
\end{enumerate}
In both cases, $\cZ$ in \eqref{e:G-Y} and \eqref{e:G-Ydouble} is an $\OO(1)$-weighted sum of terms of the form $(d-1)^{3(r+r')\ell}R_r R_{r'}'$, where $r\geq 0, r'\geq 1$, $R_r$ is an $S$-product term, and $R_{r'}'$ is a product of $r'$ factors of the following forms: 
\begin{align}\label{e:W-term}
 \sqrt t Z_{xy}, \quad
  \sqrt t( Z\wt G^{(\bT)})_{xc_\al}, \quad t(( Z\wt G^{(\bT)}Z)_{xy}-m_t \delta_{xy}), \quad \al\in \qq{\mu},\quad 
   x,y\in \bT.
\end{align}
\end{proposition}

In the following proposition, we derive a similar expansion for factors that are Green's function entries with at most one index in $\{i,o\}$.
\begin{proposition}\label{lem:offdiagswitch}
Adopt the notation and assumptions that $\cG, \widetilde \cG\in \Omega,  I(\cF^+,\cG)=1$ in \Cref{s:setting}, and define the index set $\sfA_i := \{ \alpha \in \qq{\mu} : \dist_{\cT}(i, l_\al) = \ell+1 \}$. Then for any unused core edges $(b,c)\neq (b',c')\in \cC^\circ\setminus\{ (i,o)\}$, the following holds with overwhelmingly high probability over $Z$:
\begin{enumerate}
    \item  $\widetilde G^{(ib)}_{oc}$ and $\widetilde G^{(b)}_{ic}$ can be rewritten as a weighted sum 
\begin{align}\label{e:tGoc_exp}
  \frac{1} {(d-1)^{\ell/2}}\sum_{\al\in \qq{\mu}}\fc_1(\bm1(\al\in \sfA_i))G_{c_\al c}^{(b_\al b)}
   +\cU + \cZ +\cE,
\end{align}
where $|\fc_1(\cdot)|\lesssim 1$; $\cU$ is an $\OO(1)$-weighted sum of terms of the form $(d-1)^{3r\ell}R_r G_{c_\al c}^{(b_\al b)}$, for $R_r$ an S-product term (see \Cref{d:S-product}) with $r\geq 1$; $\cZ$ is an $\OO(1)$-weighted sum of terms in the forms $(d-1)^{3(r+r')\ell} R_r R_{r'}'G_{c_\al c}^{(b_\al b)}$ or $(d-1)^{3(r+r')\ell} R_r R_{r'}'$, where $r\geq 0, r'\geq 1$, $R_r$ is an S-product term, and $R_{r'}'$ is a product of $r'$ factors of the forms
\begin{align}\label{e:Werror3}
   \sqrt t Z_{xy},\quad \sqrt t( Z\wt G^{(\bT b)})_{xc}, \quad 
   \sqrt t( Z\wt G^{(\bT b)})_{xc_\al},\quad t(( Z\wt G^{(\bT b)}Z)_{xy}-m_t \delta_{xy}),\quad \al\in \qq{\mu},\quad  
   x,y\in \bT;
\end{align}
and the error $\cE$ is bounded by
\begin{align}\label{e:tGoc_cE}
    |\cE|\lesssim \sum_{\al\in \qq{\mu}} |\wt G_{c_\al c}^{(\bT b)}-G_{c_\al c}^{(b_\al b)}|+\sum_{\al, \beta\in \qq{\mu}}|\wt G_{c_\al c_\beta}^{(\bT)}-G^{(b_\al b_\beta)}_{c_\al c_\beta}|+N^{-2}.
\end{align}
  \item  $\widetilde G_{ib}, \widetilde G_{ob}$ and $\wt G_{ob}^{(i)}$ can be rewritten as a weighted sum 
\begin{align*}
  \frac{1} {(d-1)^{\ell/2}}\sum_{\al\in \qq{\mu}}\fc_1(\bm1(\al\in \sfA_i))G^{(b_\al)}_{c_\al b }
   +\cU + \cZ +\cE,
\end{align*}
where $|\fc_1(\cdot)|\lesssim 1$; $\cU$ is an $\OO(1)$-weighted sum of terms of the form $(d-1)^{3r\ell}R_r G_{c_\al b}^{( b_\al)}$, where $R_r$ is an S-product term (see \Cref{d:S-product}) with $r\geq 1$; $\cZ$ is an $\OO(1)$-weighted sum of terms of the form $(d-1)^{3(r+r')\ell}R_r R_{r'}'G_{c_\al b}^{( b_\al)}$ or $(d-1)^{3(r+r')\ell} R_r R_{r'}'$, where $r\geq 0, r'\geq 1$, $R_r$ is an S-product term, and $R_{r'}'$ is a product of $r'$ factors of the following forms
\begin{align*}
   \sqrt t Z_{xy}, \quad  \sqrt t( Z\wt G^{(\bT b)})_{xc_\al},\quad
   \sqrt t( Z\wt G^{(\bT)})_{xb},\quad t(( Z\wt G^{(\bT)}Z)_{xy}-m_t \delta_{xy}),\quad
   \al\in\qq{\mu}, \quad x,y\in \bT;
\end{align*}
and the error $\cE$ is bounded by
\begin{align*}
    |\cE|\lesssim \sum_{\al\in \qq{\mu}} |\wt G_{c_\al b}^{(\bT )}-G_{c_\al b}^{(b_\al )}|+\sum_{\al, \beta\in \qq{\mu}}|\wt G_{c_\al c_\beta}^{(\bT)}-G^{(b_\al b_\beta)}_{c_\al c_\beta}|+N^{-2}.
\end{align*}
\item  $\wt G_{cc}^{(b)}, \wt G_{cc'}^{(bb')},\wt G_{bc'}^{(b')}$ and $ \wt G_{bb'}$ can be rewritten as
\begin{align*}
&\wt G_{cc}^{(b)}=G_{cc}^{(b)}+\cE,\quad 
    |\cE|\lesssim |\wt G^{(\bT b)}_{cc}-G^{( b)}_{cc}|
    +(d-1)^\ell\sum_{\al\in \qq{\mu}} |\wt G_{c_\al c}^{(\bT b)}|^2+(d-1)^{2\ell}N^{2\fo}t\Phi,
\\
&\wt G_{cc'}^{(bb')}=G_{cc'}^{(bb')}+\cE,\quad
    |\cE|\lesssim |\wt G^{(\bT bb')}_{cc'}-G^{( bb')}_{cc'}|
    +(d-1)^\ell\sum_{\al\in \qq{\mu}} (|\wt G_{c_\al c}^{(\bT bb')}|^2+|\wt G_{c'c_\al}^{(\bT bb')}|^2)+(d-1)^{2\ell}N^{2\fo}t\Phi,
    \\
&\wt G_{bc'}^{(b')}=G_{bc'}^{(b')}+\cE,\quad
    |\cE|\lesssim |\wt G^{(\bT b')}_{bc'}-G^{( b')}_{bc'}|
    +(d-1)^\ell\sum_{\al\in \qq{\mu}} (|\wt G_{bc_\al}^{(\bT b')}|^2+|\wt G_{c'c_\al}^{(\bT b')}|^2)+(d-1)^{2\ell}N^{2\fo}t\Phi,
\\
&\wt G_{bb'}=G_{bb'}+\cE,\quad
    |\cE|\lesssim |\wt G^{(\bT)}_{bb'}-G_{bb'}|
    +(d-1)^\ell\sum_{\al\in \qq{\mu}} (|\wt G_{bc_\al}^{(\bT)}|^2+|\wt G_{b'c_\al}^{(\bT)}|^2)+(d-1)^{2\ell}N^{2\fo}t\Phi.
\end{align*}
\end{enumerate}
\end{proposition}

We recall from \eqref{e:G-Y}, $\cU$ is an $\OO(1)$-weighted sum of terms in the form $(d-1)^{3(r-1)\ell}R_r$ and $R_r$ is an $S$-product term (see \Cref{d:S-product}), which contains $G_{c_\al c_\al}^{(b_\al)}-Q_t$ or $G_{c_\al c_\beta}^{(b_\al b_\beta)}$. 
We now gather the following computations related to \eqref{e:G-Y}, which will be used later. 
\begin{lemma}\label{l:coefficient}
    Adopt the same notation and assumptions as in \Cref{lem:diaglem}. In \eqref{e:G-Y}, the first few terms of $\wt G_{oo}^{(i)}-Y_t$ are explicitly given by
    \begin{align}\begin{split}\label{e:Uterm}
        &\wt G_{oo}^{(i)}-Y_t
        =
        \frac{\msc^{2\ell}(z_t)}{(d-1)^{\ell+1}}\sum_{\al \in \sfA_i}(G_{c_\al c_\al}^{(b_\al)}-Q_t)
        +\frac{\msc^{2\ell}(z_t)}{(d-1)^{\ell+1}}\sum_{\al\neq\beta \in \sfA_i}G_{c_\al c_\beta}^{(b_\al b_\beta)}+\cU+\cZ+\cE\\
        &\cU=\frac{\msc^{2\ell}(z_t)}{(d-1)^{\ell+2}}\sum_{\al \in \sfA_i}L^{(i)}_{l_\al l_\al}(G_{c_\al c_\al}^{(b_\al)}-Q_t)^2
        + \frac{\msc^{2\ell}(z_t)}{(d-1)^{\ell+2}}\sum_{ \al\in \sfA_i,\beta \in \qq{\mu}\atop \al\neq \beta}(L^{(i)}_{l_\beta l_\beta}+L^{(i)}_{l_\al l_\beta})(G_{c_\al c_\beta}^{(b_\al b_\beta)})^2\\
        &+
     \sum_{\al\neq\beta\in\qq{\mu}}  \fc_1(\al,\beta)  (G_{c_\al c_\al}^{(b_\al)}-Q_t)(G_{c_{\beta} c_{\beta}}^{(b_{\beta})}-Q_t)
    +
   \sum_{\al\in\qq{\mu}\atop \al'\neq\beta'\in\qq{\mu}}  \fc_2(\al,\al', \beta')  (G_{c_\al c_\al}^{(b_\al)}-Q_t)G_{c_{\al'} c_{\beta'}}^{(b_{\al'} b_{\beta'})}\\
    &+\sum_{\al\neq \beta\in \qq{\mu},\al'\neq \beta'\in \qq{\mu}\atop \{\al,\beta\}\neq \{\al',\beta'\}}\fc_3(\al,\beta,\al',\beta')  G_{c_\al c_\beta}^{(b_\al b_\beta)}G_{c_{\al'} c_{\beta'}}^{(b_{\al'} b_{\beta'})}+ \sum_{\al\in\qq{\mu}}\fc_4(\al)  (G_{c_\al c_\al}^{(b_\al)}-Q_t)(Q_t-\msc(z_t))\\
    &+ \sum_{\al\in\qq{\mu}}\fc_5(\al)  (G_{c_\al c_\al}^{(b_\al)}-Q_t)t(m_t-\md(z_t))+ \sum_{\al\neq \beta\in\qq{\mu}}\fc_6(\al,\beta)  G_{c_\al c_\beta}^{(b_\al b_\beta)}(Q_t-\msc(z_t))\\
     &+ \sum_{\al\neq \beta\in\qq{\mu}}\fc_7(\al,\beta)  G_{c_\al c_\beta}^{(b_\al b_\beta)}t(m_t-\md(z_t))
     +\textnormal{terms of the form }\{(d-1)^{3(r-1)\ell} R_r\}_{r\geq 3},
    \end{split}\end{align}
    where the total sum of the coefficients $\fc_1, \fc_2, \ldots, \fc_7$ is bounded by $\OO((d-1)^3)$, and 
    \begin{align}\label{e:2defPi}
        L^{(i)}=P^{(i)}(\cT, z_t, \msc(z_t))=\frac{1}{H_{\bT}^{(i)}-z_t- \msc(z_t)\mathbb I^{\del}},\quad \mathbb I^{\del}_{xy}=\bm1(\dist_\cT(x,o)=\ell)\delta_{xy}.
    \end{align}
\end{lemma}
Here we slightly abuse notation: $L^{(i)}$ in \eqref{e:2defPi} is consistent with the local Green's function with vertex $i$ removed (as defined in \eqref{e:defLi}), and we use the same symbols to represent them.

We start with the Schur complement formulas which will be used to prove \Cref{lem:diaglem} and \Cref{lem:offdiagswitch}. We recall the notation from \Cref{s:setting}. 
Let $\wt B$ be the normalized adjacency matrix of the directed edges $\{(c_\al, l_\al)\}_{\al \in \qq{\mu}}$. Then the adjacency matrices $\widetilde H^{(i)}, Z^{(i)}$ are in the block form
\begin{align*}
    \widetilde H^{(i)}=
    \left[
    \begin{array}{cc}
        H^{(i)}_{\bT} & \wt B^\top\\
        \wt B & \widetilde H_{\bT^\complement}
    \end{array}
    \right],\quad
   Z^{(i)}=
    \left[
    \begin{array}{cc}
        Z^{(i)}_{\bT} & Z^{(i)}_{\bT\bT^\complement}\\
        Z^{(i)}_{\bT^\complement\bT} & Z^{(i)}_{\bT^\complement}
    \end{array}
    \right].
\end{align*}
We also denote the Green's function of $\cG^{(\bT)}$ and $\wt\cG^{(\bT)}$ as $ G^{(\bT)}$ and $\wt G^{(\bT)}$ respectively.  

We collect some estimates below, which will be used later. Recall from \Cref{greentree}, for $x,y\in \bT\setminus\{i\}$, $|L_{xy}^{(i)}|\lesssim (d-1)^{-\dist_\cT(x,y)/2}$. It follows that
\begin{align}\begin{split}\label{e:sum_Pbound}
    &\sum_{x\in \bT\setminus\{i\}} |L^{(i)}_{ox}|\lesssim \sum_{r=0}^{\ell}(d-1)^{r/2}\lesssim (d-1)^{\ell/2},\\
    &\sum_{x,y\in \bT\setminus\{i\}} |L^{(i)}_{xy}|\lesssim
    \sum_{r=0}^{\ell}(d-1)^{\ell-r}\left(\sum_{r'=0}^r (d-1)^{r'/2}+\sum_{r'=r+1}^{2\ell-r}(d-1)^{r/2}\right)\lesssim \ell(d-1)^{\ell},
\end{split}\end{align}
where for the first sum, we used that $|\{x\in \bT: \dist_\cT(o,x)=r\}|=\OO((d-1)^{r})$ for $0\leq r\leq \ell$. For the second sum, we consider $\{x\in \bT: \dist_\cT(o,x)=\ell-r\}$. Note that there are $\OO((d-1)^{\ell-r})$ such vertices. Given such $x$, for any $y\in \bT$, $0\leq \dist_\cT(x,y)\leq 2\ell-r$.
There are two cases depending on $\dist_\cT(x,y)=r'$. If $0\leq r'\leq r$, we have $|\{y\in \bT: \dist_\cT(x,y)=r'\}|=\OO((d-1)^{r'})$, and if $r+1\leq r'\leq 2\ell-r$ we have $|\{y\in \bT: \dist_\cT(x,y)=r'\}|=\OO((d-1)^{(r+r')/2})$. The second statement in \eqref{e:sum_Pbound} follows from summing over $r,r'$ and using $|L_{xy}^{(i)}|\lesssim (d-1)^{-r'/2}$.

\begin{proof}[Proof of \Cref{lem:diaglem}]

Thanks to the Schur complement formula \eqref{e:Schur1}, we have 
\begin{align*}
\widetilde G_{oo}^{(i)}=\left(\frac{1}{H_{\bT}^{(i)}+\sqrt t Z_\bT^{(i)}-z-(\widetilde  B+\sqrt t Z^{(i)}_{\bT^\complement \bT})^\top\widetilde G^{(\bT)}(\widetilde  B+\sqrt t Z^{(i)}_{\bT^\complement \bT})}\right)_{oo}.
\end{align*}
Since $o$ has a radius $\fR$ tree neighborhood, $H^{(i)}_\bT$ is the normalized adjacency matrix of a truncated $(d-1)$-ary tree, and $L^{(i)}$ from \eqref{e:2defPi} agrees with the Green's function of $(d-1)$-ary tree (see \eqref{e:Gtreemsc2}),
\begin{align*}
    &\msc(z_t)=L^{(i)}_{oo}=P_{oo}^{(i)}(\cT, z_t, \msc(z_t))=\left(\frac{1}{H_{\bT}^{(i)}-z_t- \msc(z_t)\mathbb I^{\partial}}\right)_{oo}=\left(\frac{1}{H_{\bT}^{(i)}-z_t-\widetilde B^\top \msc(z_t)\widetilde B}\right)_{oo}, 
\end{align*}
By taking the difference of the two above expressions, we have
\begin{equation*}
\widetilde G_{oo}^{(i)}-\msc(z_t)=\left(\frac{1}{H_{\bT}^{(i)}-z_t-\widetilde B^\top \msc(z_t)\widetilde B-\cD}-\frac{1}{H_{\bT}^{(i)}-z_t-\widetilde B^\top \msc(z_t)\widetilde B}\right)_{oo},
\end{equation*}
where we recall $z_t=z+t\md(z,t)=z+tm_d(z_t)$, and
\begin{align}\begin{split}\label{e:defcE}
\cD
&=\widetilde  B^\top(Q_t-\msc(z_t))  \widetilde  B+t(m_t-\md(z_t))\mathbb I+\cD_1+\cD_2, \\
\cD_1&= B^\top(\widetilde G^{(\bT)}-Q_t  )\widetilde  B,\\
\cD_2&
=-\sqrt t Z_\bT^{(i)}+ \sqrt t  Z_{\bT \bT^\complement}^{(i)}\widetilde G^{(\bT)}\widetilde  B+ \sqrt t \widetilde  B^\top\widetilde G^{(\bT)}Z^{(i)}_{\bT^\complement \bT }
+t (Z^{(i)}_{\bT \bT^\complement}\widetilde G^{(\bT)}Z^{(i)}_{\bT^\complement \bT }- m_t\mathbb I),
\end{split}\end{align}
are matrices indexed by $(\bT\setminus\{i\})\times(\bT\setminus\{i\})$. We remark that each entry of $\cD_2$ is in the form \eqref{e:W-term}.

Thanks to \eqref{eq:infbound} and \eqref{e:Werror2}, with overwhelmingly high probability over $Z$, we have 
\begin{align*}
     |(\cD_1)_{xy}|, |(\cD_2)_{xy}|, |\cD_{xy}|\leq N^{-\fb}, \text{ for }x,y\in \bT\setminus\{i\}.
\end{align*} 
Thus, for some sufficiently large constant $\fp$, we have
\begin{align}\begin{split}\label{eq:resolventexp}
    \widetilde G_{oo}^{(i)}-\msc(z_t)&=\left(\frac{1}{H_{\bT}^{(i)}-z_t-\widetilde B^\top \msc(z_t)  \widetilde B-\cD}-\frac{1}{H_{\bT}^{(i)}-z_t-\widetilde B^\top \msc(z_t)  \widetilde B}\right)_{oo}\\
    &=\left(L^{(i)}\sum_{k=1}^\fp \left(\cD L^{(i)}\right)^k\right)_{oo}+\OO(N^{-2}).
\end{split}
\end{align}
Recall $Y_t=Y_\ell(Q_t, z+tm_t)$ from \eqref{def:Y} and \eqref{e:defYt}:
\begin{align*}
Y_t=\left(\frac{1}{H_{\bT}^{(i)}-z-tm_t-\widetilde B^\top Q_t  \widetilde B}\right).
\end{align*}
By the same argument as in \eqref{eq:resolventexp} we also have that
\begin{align}\label{e:YQ-m}
    Y_t-\msc(z_t)=\left(L^{(i)}\sum_{k=1}^\fp \left((\widetilde B^\top (Q_t-\msc(z_t))\widetilde B +t(m_t-m_d(z_t))\mathbb I)L^{(i)}\right)^k\right)_{oo}+\OO(N^{-2}).
\end{align}
By taking the difference of \eqref{eq:resolventexp} and \eqref{e:YQ-m}, up to error $\OO(N^{-2})$, we get that the difference $\widetilde G_{oo}^{(i)}-Y_t$ is given as
\begin{align}\begin{split}\label{e:GooY}
(L^{(i)} \left(\cD_1+\cD_2\right)L^{(i)})_{oo}
&+ \left(L^{(i)}\sum_{k=2}^\fp \left(\cD L^{(i)}\right)^k\right)_{oo}\\
&-\left(L^{(i)}\sum_{k=2}^\fp \left((\widetilde B^\top (Q_t-\msc(z_t)) \widetilde B +t(m_t-m_d(z_t))\mathbb I )L^{(i)}\right)^k\right)_{oo}.
\end{split}\end{align}

If $\al \in \sfA_i$, then $\dist_\cT(i, l_\al)=\ell+1$, and \Cref{greentree} gives
\begin{align}\label{e:Piol}
  L^{(i)}_{ol_\al}=\msc(z_t)\left(-\frac{{\msc(z_t)}}{\sqrt{d-1}}\right)^{\dist_\cT(o,l_\al)},\quad  |L^{(i)}_{ol_\al}|\lesssim (d-1)^{-\dist_{\cT}(o,l_\al)/2}=(d-1)^{-\ell/2}.
\end{align}
Otherwise if $\al \in\qq{\mu}\setminus \sfA_i$, then $o,l_\al$ are in different connected components of $\cT^{(i)}$, and $|L^{(i)}_{ol_\al}|=0$.
Thus the first term in \eqref{e:GooY} can be computed as, 
\begin{align}\label{e:G-Y2}
    (L^{(i)}\cD_1 L^{(i)})_{oo}=\frac{\msc(z_t)^{2\ell+2}} {(d-1)^{\ell+1}}\sum_{\al\in\sfA_i}(\wt G^{(\bT)}_{c_\alpha c_\alpha}-Q_t)
   +\frac{\msc(z_t)^{2\ell+2}} {(d-1)^{\ell+1}}\sum_{\al\neq \beta\in\sfA_i} \wt G^{(\bT)}_{c_\alpha c_\beta},
\end{align}
and $(d-1)^{-3\ell}L^{(i)}\cD_2 L^{(i)}$ is a weighted sum of terms of the form \eqref{e:W-term}, with total weights bounded by 
\begin{align*}
    \frac{1}{(d-1)^{3\ell}}\sum_{x,y\in \bT\setminus\{i\}}|L^{(i)}_{ox}||L^{(i)}_{yo}|\lesssim 1,
\end{align*}
where we used \eqref{e:sum_Pbound}.

We obtain the first two terms in \eqref{e:G-Y}, after replacing $\widetilde G_{c_\al c_\al}^{(\bT)}-Q_t,\widetilde G_{c_\al c_\beta}^{(\bT)}$ in \eqref{e:G-Y2} by $G_{c_\al c_\al}^{(b_\al)}, G_{c_\al c_\beta}^{(b_\al b_\beta)}$. We collect the difference in the error term $\cE$ (as in \eqref{e:defCE}).

For the terms $k\geq 2$, in general for any matrix $V$ defined on $(\bT\setminus\{i\})\times (\bT\setminus\{i\})$, 
$L^{(i)}(V L^{(i)})^k$ is given as a sum of terms in the following form
\begin{align}\label{e:PUP}
   \sum_{x_1, x_2, \cdots, x_{2k}\in \bT\setminus\{i\}}L^{(i)}_{o x_1} V_{x_1 x_2} L^{(i)}_{x_2 x_3} V_{x_3 x_4} 
   L^{(i)}_{x_4 x_5}\cdots V_{x_{2k-1} x_{2k}}L^{(i)}_{x_{2k} o}.
\end{align}
We can reorganize \eqref{e:PUP} in the following way
\begin{align}\begin{split}\label{e:totalsum0}
    &\phantom{{}={}}\sum_{x_1, x_2, \cdots, x_{2k}\in \bT\setminus\{i\}}L^{(i)}_{o x_1} L^{(i)}_{x_2 x_3}\cdots L^{(i)}_{x_{2k} o}  V_{x_1 x_2}  V_{x_3 x_4} 
   \cdots V_{x_{2k-1} x_{2k}}.\\
   &=(d-1)^{3(k-1)\ell}\sum_{x_1, x_2, \cdots, x_{2k}\in \bT\setminus\{i\}}(d-1)^{-3(k-1)\ell}L^{(i)}_{o x_1} L^{(i)}_{x_2 x_3}\cdots L^{(i)}_{x_{2k} o} V_{x_1x_2} V_{x_3 x_4} 
   \cdots  V_{x_{2k-1}x_{2k}}\\
   &=:(d-1)^{3(k-1)\ell}\sum_{x_1, x_2, \cdots, x_{2k}\in \bT\setminus\{i\}}\fc_{\bmx}V_{x_1x_2}V_{x_3 x_4}
   \cdots  V_{x_{2k-1}x_{2k}},
\end{split}\end{align}
where the weights $\fc_{\bmx}=(d-1)^{-3(k-1)\ell}L^{(i)}_{o x_1} \cdots L^{(i)}_{x_{2k} o} $, and the total weights are bounded as
\begin{align*}
\sum_{x_1, x_2, \cdots, x_{2k}\in \bT\setminus\{i\}} |\fc_{\bmx}|
&=(d-1)^{-3(k-1)\ell}\sum_{x_1\in \bT\setminus\{i\}}|L^{(i)}_{o x_1}| \sum_{x_2, x_3\in \bT\setminus\{i\}}|L^{(i)}_{x_2 x_3}|\cdots \sum_{x_{2k}\in \bT\setminus\{i\}}|L^{(i)}_{x_{2k} o}|\\
&\lesssim (d-1)^{-3(k-1)\ell} \ell^{k-1} (d-1)^{k\ell}= (d-1)^{-(2k-3)\ell} \ell^{k-1}\lesssim 1,
\end{align*}
where to get the second line we used \eqref{e:sum_Pbound}; in the last inequality, we used that $k\geq 2$.

To compute the difference for $k\geq 2$ in \eqref{e:GooY}, we consider two possible forms for $V$:
$V=\widetilde B^\top (Q_t-\msc(z_t))\widetilde B +t(m_t-m_d(z_t)\mathbb I $ or $V=\widetilde B^\top (Q_t-\msc(z_t))\widetilde B +t(m_t-m_d(z_t))\mathbb I+\cD_1+\cD_2 $. As discussed above (see \eqref{e:totalsum0}), terms in \eqref{e:GooY} with $k\geq 2$ break down to an $\OO(1)$-weighted sum of terms in the form $(d-1)^{3(k-1)\ell}\wt R_k$. Here $\wt R_k$ is a product of $k$ factors,  each taking the form of \eqref{e:W-term} or one of the following:
\begin{align*}
    (\wt G_{c_\al c_\al}^{(\bT)}-Q_t), \quad  \wt G_{c_\al c_{\beta}}^{(\bT)},\quad  (Q_t-\msc(z_t)), \quad  t(m_t-\md(z_t)), \quad \al\neq \beta\in \qq{\mu}.
\end{align*} 
Moreover, $\wt R_k$ contains at least one factor in \eqref{e:W-term} (arising from $\cD_2$) or at least one factor of the form $\{\wt G_{c_\al c_\al}^{(\bT)}-Q_t, \wt G_{c_\al c_\beta}^{(\bT)}\}_{\al\neq \beta\in \qq{\mu}}$ (arising from  $\cD_1$). Otherwise, the terms from the difference in \eqref{e:GooY} cancel out.

For each $\wt R_k$ terms, we get $R_k$ by replacing $\widetilde G_{c_\al c_\al}^{(\bT)}-Q_t,\widetilde G_{c_\al c_\beta}^{(\bT)}$ with $G_{c_\al c_\al}^{(b_\al)}-Q_t, G_{c_\al c_\beta}^{(b_\al b_\beta)}$, respectively. As $k\geq 2$  and each factor of $R_k, \wt R_k$ is bounded by $N^{-\fb} $ with overwhelmingly high probability over $Z$, the replacement error is bounded by
\begin{align*}
    |\wt R_k-R_k|\lesssim \frac{1}{N^{\fb/2} } \sum_{\al, \beta\in \qq{\mu}} |\wt G^{(\bT)}_{c_\al c_\beta}-G^{(b_\al b_\beta)}_{c_\al c_\beta}|.
\end{align*}
We collect the above error in $\cE$ (as in \eqref{e:defCE}).

For the $\OO(1)$-weighted sum of terms in the form $(d-1)^{3(k-1)\ell} R_k$, we can regroup it as $\cU+\cZ$, depending if $R_k$ contains a factor in \eqref{e:W-term} or not. This finishes  the first statement in \Cref{lem:diaglem}.

Now, we generalize this argument to a term of the form $\wt G_{ss'}^\circ=(\wt G-\wt L)_{ss'}$ with $s,s'\in\{i,o\}$. We notice that $\widetilde L_{ss'}=L_{ss'}$. By the exact same expansion as in  \eqref{eq:resolventexp}, we have
\begin{align*}
&\phantom{{}={}}(\wt G-\wt L)_{ss'}=\left(L\sum_{k=1}^\fp \left(\cD L\right)^k\right)_{ss'}+\OO(N^{-2}),
\end{align*}
where $\cD=\widetilde  B^\top(Q_t-\msc(z_t))\widetilde  B+t(m_t-\md(z_t))\mathbb I+\cD_1+\cD_2 $ is defined similarly as in \eqref{e:defcE} without removing the vertex $i$.
Using the above expansion, the statement follows from the same argument as the first statement.  
\end{proof}

\begin{proof}[Proof of \Cref{lem:offdiagswitch}]
The proof closely follows the argument used in \Cref{lem:diaglem}, with a few minor modifications. For the first statement, we demonstrate the result for $G_{oc}^{(ib)}$; the proof for $G_{ic}^{(b)}$ is analogous and will be omitted. We expand $\wt G_{oc}^{(ib)}$ using the Schur complement formula \eqref{e:Schur1} as
\begin{align}\begin{split}\label{eq:offdiagonalexp}
\wt G^{(ib)}_{oc}
&=-\sum_{x\in \bT \setminus \{i\}} \wt G^{(ib)}_{o x}((\wt B^\top+\sqrt t Z)\wt G^{(\bT b)})_{x c}.
\end{split}\end{align}
Here, we slightly abuse notation, letting $(Z\wt G^{(\bT b)})_{x c}=\sum_{y\not\in \bT b}Z_{xy}G^{(\bT b)}_{y c}$.

For $\wt G^{(ib)}_{o x}$ in \eqref{eq:offdiagonalexp}, we can do the same expansion as in \eqref{e:defcE} and \eqref{eq:resolventexp}, writing
\begin{align}\begin{split}\label{eq:offdiagonalexp2}
   &\phantom{{}={}}\sum_{x\in \bT \setminus \{i\}} \wt G^{(ib)}_{o x}((\wt B^\top+\sqrt t Z)\wt G^{(\bT b)})_{x c}
   =\sum_{x\in \bT \setminus \{i\}}L_{o x}^{(i)}((\wt B^\top+\sqrt t Z)\wt G^{(\bT b)})_{x c}\\
   &+ \sum_{x\in \bT \setminus \{i\}}  \left(L^{(i)}\sum_{k=1}^\fp \left(\cD L^{(i)}\right)^k\right)_{o x}((\wt B^\top+\sqrt t Z)\wt G^{(\bT b)})_{x c}+\OO(N^{-2}),
\end{split}\end{align}
where 
\begin{align}\begin{split}\label{e:def_cD2}
&\cD:=\widetilde  B^\top(Q_t-\msc(z_t))\widetilde  B+t(m_t-\md(z_t))\mathbb I+\cD_1+\cD_2,\\
&\cD_1:= B^\top(\widetilde G^{(\bT b)}-Q_t)\widetilde  B
\\
&\cD_2
=-\sqrt t Z_\bT^{(i)}+ \sqrt t  Z_{\bT \bT^\complement}^{(i)}\widetilde G^{(\bT b)}\widetilde  B+ \sqrt t \widetilde  B^\top\widetilde G^{(\bT b)}Z^{(i)}_{\bT^\complement \bT}
+t (Z_{\bT \bT^\complement}^{(i)}\widetilde G^{(\bT b)}Z_{\bT^\complement \bT}^{(i)}- m_t\mathbb I).
\end{split}\end{align}
are matrices index by $(\bT\setminus \{i\})\times (\bT \setminus \{i\})$. We remark that each entry of $\cD_2$ is in the form \eqref{e:Werror3}.

For the first term on the right-hand side of \eqref{eq:offdiagonalexp2}, using \eqref{e:Piol}, we can rewrite it as
\begin{align}\label{e:tGoc_exp2}
  \frac{\fc_1}{(d-1)^{\ell/2}}\sum_{\al\in\sfA_i}\wt G^{(\bT b)}_{c_\alpha c}+\sum_{x\in \bT \setminus \{i\}} L_{o x}^{(i)}\sqrt t (Z\wt G^{(\bT b)})_{x c},  \quad \fc_1=\OO(1).
\end{align}
We obtain the first term in \eqref{e:tGoc_exp}, after replacing $\widetilde G_{c_\al c}^{(\bT b)}$ in \eqref{e:tGoc_exp2} by $G_{c_\al c}^{(b_\al b)}$. We collect the difference in the error term $\cE$ (as in \eqref{e:tGoc_cE}). The second term in \eqref{e:tGoc_exp2} is a weighted sum of terms of the form $(d-1)^{3\ell}\sqrt t (Z\wt G^{(\bT b)})_{x c}$, with total weights bounded by 
$(d-1)^{-3\ell}\sum_{x\in \bT \setminus \{i\}}|L^{(i)}_{ox}|\lesssim 1$, which follows from \eqref{e:sum_Pbound}.

For the terms with $k\geq 1$ on the right-hand side of \eqref{eq:offdiagonalexp2}, by the same argument as in \eqref{e:totalsum0}, 
we can reorganize them in the following way:
\begin{align}\begin{split}\label{e:total_Goc}
   (d-1)^{3k\ell}\sum_{x_1, x_2, \cdots, x_{2k+1}\in \bT \setminus \{i\}} \fc_{\bmx}\cD_{x_1x_2}\cD_{x_3 x_4}
   \cdots  \cD_{x_{2k-1}x_{2k}} ((\wt B^\top+\sqrt t Z)\wt G^{(\bT b)})_{x_{2k+1} c},
\end{split}\end{align}
where the weights $\fc_{\bmx}=(d-1)^{-3k\ell}L^{(i)}_{o x_1} \cdots L^{(i)}_{x_{2k} x_{2k+1}} $, and using \eqref{e:sum_Pbound}, the total weights is bounded  
\begin{align*}
\sum_{x_1, x_2, \cdots, x_{2k+1}\in \bT \setminus \{i\}}  |\fc_{\bmx}|
&=(d-1)^{-3k\ell}\sum_{x_1\in \bT \setminus \{i\}}|L^{(i)}_{o x_1}| \sum_{x_2, x_3\in \bT \setminus \{i\}}|L^{(i)}_{x_2 x_3}|\cdots \sum_{x_{2k}, x_{2k+1}\in \bT \setminus \{i\}} |L^{(i)}_{x_{2k} x_{2k+1}}|\\
&\lesssim (d-1)^{-3k\ell} \ell^{k} (d-1)^{(k+1/2)\ell}= (d-1)^{-(2k-1/2)\ell} \ell^{k}\lesssim 1. 
\end{align*}

Recalling $\cD$ from \eqref{e:def_cD2}, \eqref{e:total_Goc} further breaks down to an $\OO(1)$-weighted sum of terms of the form $(d-1)^{3k\ell}\wt R_k \wt G_{c_\al c}^{(\bT b)}$ or $ (d-1)^{3k\ell}\wt R_{k} \sqrt t(Z \wt G^{(\bT b)})_{x c}$. Here $\wt R_k$ is a product of $k$ factors, each taking the form of \eqref{e:Werror3} or one of the following:
\begin{align*}
    \wt G_{c_\al c_\al}^{(\bT)}-Q_t, \quad  \wt G_{c_\al c_{\beta}}^{(\bT)},\quad  (Q_t-\msc(z_t)), \quad  t(m_t-\md(z_t)), \quad \al\neq \beta\in \qq{\mu}.
\end{align*} 
By the same argument as in the proof of \Cref{lem:diaglem}, we can replace $\widetilde G_{c_\al c_\al}^{(\bT)}-Q_t,\widetilde G_{c_\al c_\beta}^{(\bT)}, \widetilde G_{c_\al c}^{(\bT b)}$ by $G_{c_\al c_\al}^{(b_\al)}-Q_t, G_{c_\al c_\beta}^{(b_\al b_\beta)}, G_{c_\al c}^{(b_\al b)}$, and the errors are bounded by 
\begin{align*}
\sum_{\al\in \qq{\mu}} |\wt G_{c_\al c}^{(\bT b)}-G_{c_\al c}^{(b_\al b)}|+\sum_{\al, \beta\in \qq{\mu}}|\wt G_{c_\al c_\beta}^{(\bT)}-G^{(b_\al b_\beta)}_{c_\al c_\beta}|.
\end{align*}
Then we can regroup them as $\cU+\cZ$,  depending if these terms contain a factor in \eqref{e:Werror3} or not. This finishes  the first statement in \Cref{lem:offdiagswitch}.

The second statement in \Cref{lem:offdiagswitch} can be proven using the same approach as the first statement; therefore, the proof is omitted.

For the third statement, we provide a proof for $\wt G_{cc}^{(b)}$. The remaining cases can be proven similarly and are therefore omitted. Using the Schur complement formula \eqref{e:Schur1}, we have
\begin{align}\label{e:tGccb}
\wt G^{(b)}_{cc}
=\wt G^{(\bT b)}_{cc}
+\left(\wt G^{(\bT b)}(\wt B+\sqrt tZ_{\bT^\complement \bT}) \widetilde G^{( b)}(\wt B^\top+\sqrt tZ_{\bT^\complement \bT})\wt G^{(\bT b)}\right)_{cc}.
\end{align}
By the Cauchy-Schwarz inequality, we can bound the second term in \eqref{e:tGccb} as
\begin{align}\begin{split}\label{e:tGccb2}
    &\phantom{{}={}}\sum_{x,y\in \bT}|((\wt B^\top+\sqrt tZ)\wt G^{(\bT b)})_{xc}| \widetilde G^{( b)}_{xy}||((\wt B^\top+\sqrt tZ)\wt G^{(\bT b)})_{yc}|
    \lesssim  \left(\sum_{x\in \bT}|((\wt B^\top+\sqrt tZ)\wt G^{(\bT b)})_{xc}|\right)^2\\
    &\lesssim (d-1)^{\ell}\left(\sum_{\al\in \qq{\mu}} |\wt G_{c_\al c}^{(\bT b)}|^2+\sum_{x\in \bT} t|(Z\wt G^{(\bT b)})_{xc}|^2\right)
    \lesssim (d-1)^{\ell}\sum_{\al\in \qq{\mu}} |\wt G_{c_\al c}^{(\bT b)}|^2 +(d-1)^{2\ell}N^{2\fo}t\Phi,
\end{split}\end{align}
where in the first statement, we used $|\widetilde G_{xy}^{(b)}|\lesssim 1$ from \eqref{eq:infbound}; in the last statement, we used \eqref{e:Werror2}.
It follows from combining \eqref{e:tGccb} and \eqref{e:tGccb2} that $\wt G_{cc}^{(b)}= G^{(b)}_{cc}+\cE$, where
\begin{align*}
    |\cE|\lesssim |\wt G^{(\bT b)}_{cc}-G^{( b)}_{cc}|
    +(d-1)^{\ell}\sum_{\al\in \qq{\mu}} |\wt G_{c_\al c}^{(\bT b)}|^2+(d-1)^{2\ell}N^{2\fo}t\Phi.
\end{align*}
\end{proof}

\begin{proof}[Proof of \Cref{l:coefficient}]
The decomposition of $\wt G_{oo}^{(i)}-Y_t$ in \eqref{e:Uterm} is from \eqref{e:G-Y}, where 
$\cU$ is an $\OO(1)$-weighted sum of terms in the form $(d-1)^{3(r-1)\ell}R_r$, where $r\geq 2$, and $R_r$ is an $S$-product term (see \Cref{d:S-product}) which contains $G_{c_\al c_\al}^{(b_\al)}-Q_t$ or $G_{c_\al c_\beta}^{(b_\al b_\beta)}$. For the expression of $\cU$ in \eqref{e:G-Y}, we list all possibility of $R_2$ which contains $G_{c_\al c_\al}^{(b_\al)}-Q_t$ or $G_{c_\al c_\beta}^{(b_\al b_\beta)}$, and the total sum of coefficients is bounded by $\OO((d-1)^3)$. The precise coefficients in front of  $(G_{c_\al c_\al}^{(b_\al)}-Q_t)^2$ and $(G_{c_\al c_\beta}^{(b_\al b_\beta)})^2$ are from the expansion \eqref{e:GooY}.  More precisely, in $\cU$ the coefficient of $(G_{c_\al c_\al}^{(b_\al)}-Q_t)^2$ is given by 
\begin{align*}
\frac{1}{(d-1)^2}L_{o l_\al}^{(i)}L_{ l_\al  l_\al}^{(i)}L_{ l_\al o}^{(i)}
= \frac{\msc^{2\ell}(z_t)}{(d-1)^{\ell+2}}L_{ l_\al  l_\al}^{(i)}.
 \end{align*}
Similarly, in $\cU$ the coefficient of $(G_{c_\al c_\beta}^{(b_\al b_\beta)})^2$ is given by 
\begin{align*}
    \frac{1}{(d-1)^2}L_{ol_\al}^{(i)}L_{l_\beta l_\beta}^{(i)} L_{l_\al o}^{(i)} 
    + \frac{1}{(d-1)^2}L_{ol_\al}^{(i)}L_{l_\beta l_\al}^{(i)} L_{l_\beta o}^{(i)} 
    =\frac{\msc^{2\ell}(z_t)}{(d-1)^{\ell+2}}(L_{ l_\al  l_\beta}^{(i)}+L_{ l_\beta  l_\beta}^{(i)})\bm1(\al\in \sfA_i).
\end{align*}
\end{proof}

\subsection{An Expansion using the Woodbury formula}\label{sec:fanalysis}
 We recall the notation and assumptions from \Cref{s:setting}. In this section, we introduce a novel expansion based on the Woodbury formula \eqref{e:woodbury}.
We compare the normalized adjacency matrix of the switched graph to that of the original graph, $\wt H-H$.
We denote the adjacency matrices of our switching as
\begin{align*}
   \wt H-H=-\sum_{\al\in\qq{\mu} }\xi_\al,\quad  \xi_\al:=\frac{1}{\sqrt{d-1}}\left(\Delta_{l_\al a_\al}
    +\Delta_{b_\al c_\al}
    -\Delta_{l_\al c_\al}-\Delta_{a_\al b_\al}\right).
\end{align*}
We denote the rank of this difference as $r=\OO((d-1)^\ell)$, and rewrite 
\begin{align*}
    \wt H-H=UV^\top,
\end{align*}
where $U, V$ are $N\times r$ matrices, and their nonzero rows correspond to the vertices $\{l_\al, a_\al, b_\al, c_\al\}_{\al\in \qq{\mu}}$.
Then, the Woodbury formula \eqref{e:woodbury} gives us 
\begin{align}\label{e:tGG}
\tG-G=(H-z+UV^\top)^{-1}-(H-z)^{-1}=-GU(\mathbb I+V^\top G U)^{-1}V^\top G.
\end{align}

We recall $\cF^+$ as in \eqref{eq:forestdef2}, and denote by $\wt \cF^+$ the switched version of it
\begin{align*}
  \cF^+:=\cF\cup \cB_{\ell}(o,\cG)\cup\bigcup_{\al=1}^\mu \{(l_\alpha,a_\alpha), (b_\alpha, c_\alpha)\},\quad   \wt \cF^+:=\cF\cup \cB_{\ell}(o,\cG)\cup\bigcup_{\al=1}^\mu \{(l_\alpha,c_\alpha), (a_\alpha, b_\alpha)\}.
\end{align*} 
We view $\cF^+, \wt\cF^+$ as subgraphs of $\cG, \wt\cG$ respectively. Their radius $\fR$ neighborhoods, $\cB_\fR(\cF^+, \cG), \cB_\fR(\wt \cF^+,\wt \cG)$ consist of the same vertices. We will analyze \eqref{e:tGG} using the matrices \begin{align}\label{e:defPtP}
   L:=P(\cB_\fR(\cF^+, \cG),z_t,\msc(z_t)), \quad
   \widetilde L:=P(\cB_\fR(\wt\cF^+, \wt\cG),z_t,\msc(z_t))=P(\cB_\fR(\cF^+, \wt\cG),z_t,\msc(z_t)),
\end{align} 
as was defined in \Cref{def:pdef}. Here we slightly abuse notation: $L,\widetilde L$ in \eqref{e:defPtP} are consistent with the local Green's functions  (as defined in \eqref{e:local_Green} and \eqref{e:deftL}) on $\cB_{\fR}(\cF^+,\cG)$, and we use the same symbols to represent them. 

Notice that when restricted to the vertex set $\cB_{\fR}(\cF^+,\cG)$, 
\begin{align}\label{e:inverseL}
    \widetilde L^{-1}-L^{-1}=\widetilde H-H=-\sum_{\al\in\qq{\mu} }\xi_\al=UV^\top.
\end{align} 
We can use the Woodbury formula on $ L, \widetilde L$ as well, giving
\begin{align}\label{e:tP-P}
    \widetilde L-L=-LU(\mathbb I+V^\top L U)^{-1}V^\top L.
\end{align}

A crucial observation is that the quantity $-U(\mathbb I+V^\top L U)^{-1}V^\top$ in \eqref{e:tGG} and \eqref{e:tP-P} take very simple form.
\begin{lemma}\label{l:defF2}
We introduce the following matrix $F$, which is nonzero on the vertex set $\{l_\al, a_\al, b_\al, c_\al\}_{\al\in \qq{\mu}}$,
\begin{align}\label{e:defF}
F:=\sum_{\al \in \qq{\mu}} \xi_{\al}+\sum_{\al, \beta\in \qq{\mu}} \xi_{\al}\tL \xi_{\beta}.
\end{align}
Then 
\begin{align}\label{e:defF2g}
    F=-U(\mathbb I+V^\top L U)^{-1}V^\top.
\end{align}
\end{lemma}

\begin{proof}[Proof of \Cref{l:defF2}]
The nonzero rows of $U, V$ are parametrized by $\{l_\alpha, a_\alpha, b_\alpha, c_\alpha\}_{\al\in \qq{\mu}}$.
By rearranging the above expression \eqref{e:tP-P} (we view all the matrices as restricted on the vertex set $\cB_{\fR}(\cF^+,\cG)$), we get
\begin{align}\label{e:Fz}
L^{-1}\tL L^{-1}-L^{-1}=-U(\mathbb I+V^\top L U)^{-1}V^\top.
\end{align}
We can reorganize \eqref{e:Fz} as
\begin{align}\begin{split}\label{e:defF2}
    &\phantom{{}={}}-U(\mathbb I+V^\top L U)^{-1}V^\top =L^{-1}\tL L^{-1}-L^{-1}=L^{-1}\tL \tL^{-1}+L^{-1}\tL (L^{-1}-\tL^{-1})-L^{-1}\\
    &=L^{-1}\tL (L^{-1}-\tL^{-1})
    =(L^{-1}-\tL^{-1})\tL (L^{-1}-\tL^{-1})+\tL^{-1}\tL (L^{-1}-\tL^{-1})\\
     &=(L^{-1}-\tL^{-1})+(L^{-1}-\tL^{-1})\tL (L^{-1}-\tL^{-1})=\sum_{\al \in \qq{\mu}} \xi_{\al}+\sum_{\al, \beta\in \qq{\mu}} \xi_{\al}\tL \xi_{\beta}=F,
\end{split}\end{align}
where in the last statement we used  \eqref{e:inverseL}.
\end{proof}

Our next lemma attempts to expand $\tG-G$ in terms of $\tL-L$.

\begin{lemma}\label{lem:woodbury} 
Adopt the notation and assumptions that $\cG, \widetilde \cG\in \Omega,  I(\cF^+,\cG)=1$ in \Cref{s:setting}, and recall $G^\circ|_{\cB_\fR(\cF^+, \cG)}=(G-L)|_{\cB_\fR(\cF^+, \cG)}$. Then with overwhelmingly high probability over $Z$, we have:  
\begin{align}\label{e:tG-Gdiff}
&\tG-G=\sum_{k\geq 0} GF(G^\circ F)^{k}G,
\end{align}
and if we restrict to the vertex set of ${\cB_\fR(\cF^+, \cG)}$ the following holds
\begin{align}\begin{split}\label{e:tG-Gdiff2}
&\phantom{{}={}}\left.(\tG^\circ -G^\circ)\right|_{\cB_\fR(\cF^+, \cG)}=\sum_{k\geq 1} L(FG^\circ)^{k}+(G^\circ F)^{k}L+(G^\circ F)^{k}G^\circ+LF(G^\circ F)^{k}L.
\end{split}\end{align} 
\end{lemma}

\begin{proof}

Thanks to \eqref{eq:infbound}, with overwhelmingly high probability over $Z$, $|G^\circ_{xy}|\lesssim N^{-\fb}$,   uniformly for $x,y\in \{l_\al, a_\al, b_\al, c_\al\}_{\al \in \qq{\mu}}$.
We can then expand \eqref{e:tGG} using the resolvent identity \eqref{e:resolv} and  \eqref{e:defF2g} to conclude that
\begin{align*}
\tG-G
&=-GU(\mathbb I+V^\top G U)^{-1}V^\top G
=-GU(\mathbb I+V^\top L U+V^\top G^\circ U)^{-1}V^\top G\\
&
=-GU\left((\mathbb I+V^\top L U)^{-1}\sum_{k\geq 0}(-1)^k(V^\top G^\circ U (\mathbb I+V^\top L U)^{-1})^k \right)V^\top G\\
&
=\sum_{k\geq 0}(-1)^{k+1} GU(\mathbb I+V^\top L U)^{-1}(V^\top G^\circ U (\mathbb I+V^\top L U)^{-1})^k V^\top G\\
&=\sum_{k\geq 0} GF(G^\circ F)^{k}G
.
\end{align*}
This gives \eqref{e:tG-Gdiff}. The claim \eqref{e:tG-Gdiff2} follows by taking the difference between \eqref{e:tG-Gdiff} and \eqref{e:tP-P}, then substituting according to \eqref{e:defF2g}.
\end{proof}

Finally we gather the following estimates, which will be used later. 
\begin{lemma}
    Recall $\wt L, L$ from \eqref{e:defPtP}. The matrix $F$ from \eqref{e:defF} has nonzero entries only on the vertices $\{l_\al, a_\al, b_\al, c_\al\}_{\al\in \qq{\mu}}$, and 
\begin{align}\begin{split}\label{e:Fbound}
    \sum_{s,s'\in \{l_\al, a_\al, b_\al, c_\al\}_{\al\in \qq{\mu}}}|F_{ss'}|
    \lesssim \ell (d-1)^\ell.
\end{split}\end{align}
Moreover, for $s\in \{i,o\}$, and $\sfJ\in \{l, a,b,c\}$
\begin{align}\label{e:PF_bound}
    (LF)_{s \sfJ_\al}=\frac{1}{\sqrt{d-1}}\sum_{x\sim \sfJ_\al} \wt L_{s x} -z_t\wt L_{s\sfJ_\al}=\OO\left(\frac{1}{(d-1)^{\ell/2}}\right).
\end{align}
\end{lemma}
    
\begin{proof}
    It is easy to see from the expression \eqref{e:defF} that $F$
 has nonzero entries only on the vertices $\{l_\al, a_\al, b_\al, c_\al\}_{\al\in \qq{\mu}}$. The estimate \eqref{e:Pijbound} gives
\begin{align}\label{e:Piolha}
  |\tL_{{\sfJ}_\al \sfJ'_{\al'}}|\lesssim (d-1)^{-\dist_{\cT}(l_\al,l_{\al'})/2},\text{ for } \sfJ, \sfJ'\in \{l,a,b,c\},\quad \al, \al'\in \qq{\mu}.
\end{align}
 By plugging \eqref{e:Piolha} into \eqref{e:defF}, we get
\begin{align*}\begin{split}
    &|F_{{\sfJ}_\al \sfJ'_{\al'}}|\lesssim (d-1)^{-\dist_{\cT}(l_\al, l_{\al'})/2},
   \text{ for }  \sfJ, \sfJ'\in \{l,a,b,c\},\quad \al, \al'\in \qq{\mu},\\
    &\sum_{s,s'\in \{l_\al, a_\al, b_\al, c_\al\}_{\al\in \qq{\mu}}}|F_{ss'}|
    \lesssim \sum_{\al, \al'\in \qq{\mu}}(d-1)^{-\dist_{\cT}(l_\al, l_{\al'})/2}\lesssim \ell (d-1)^\ell,
\end{split}\end{align*}
where in the last inequality we used that $\{\al'\in \qq{\mu}: \dist_{\cT}(l_\al, l_{\al'})=2r\}=\OO((d-1)^{r})$ for $0\leq r\leq \ell$.

Next, we show \eqref{e:PF_bound}. We recall from \eqref{e:defF2} that $F=L^{-1} \wt L L^{-1} -L^{-1}$, and from \eqref{e:defP} and \eqref{e:defPtP}, we have $L^{-1}_{x\sfJ_\al}=H_{x\sfJ_\al}-z_t$ for any vertex $x$ in $\cB_\fR(\cF^+,\cG)$ (since $\sfJ_\al$ is not on the boundary of $\cB_\fR(\cF^+,\cG)$).
It follows
\begin{align*}
    (LF)_{s\sfJ_\al}= (L(L^{-1} \wt L L^{-1} -L^{-1}))_{s\sfJ_\al}=(\wt L L^{-1})_{s\sfJ_\al}
    =(\wt L H)_{s \sfJ_\al }-z_t \wt L_{s \sfJ_\al}.
\end{align*}
The claim \eqref{e:PF_bound} follows by noticing that $|\wt L_{sx}|, |\wt L_{s \sfJ_\al}|\lesssim (d-1)^{-\ell/2}$.
\end{proof}

\subsection{Switching using 
the Woodbury formula}\label{sec:woodbury_lemmas}

In this section, we will use the expansion formula \Cref{lem:woodbury} to study the Green's function after local resampling. We recall the local Green's function $L$ from \eqref{e:local_Green}, and extend it to be associated with $\cF^+$. We also introduce the following W-Product term.

\begin{definition}[W-Product term]\label{d:W-product}
    Fixing $r\geq 0$, we define $R_r$ as a \emph{W-product term} of order $r$ (where ``\emph{W}" indicates that these terms arise from expansions using the Woodbury formula) if it is a product of $r$ factors in the following forms: $G^{\circ}_{xy}=(G-L)_{xy}$ for ${x,y\in \cK^+}$. 
\end{definition}

In the following proposition, we gather estimates on the difference in Green's functions before and after local resampling, using \Cref{lem:woodbury}.
\begin{proposition}
\label{lem:generalQlemma}
Adopt the notation and assumptions that $\cG, \widetilde \cG\in \Omega,  I(\cF^+,\cG)=1$ in \Cref{s:setting}, and define the index set $\sfA_i := \{ \alpha \in \qq{\mu} : \dist_{\cT}(i, l_\al) = \ell+1 \}$. The following holds with overwhelmingly high probability over $Z$, where in all cases, the error $\cE$ satisfies
\begin{align*}
  |\cE|\leq N^{-2}.
\end{align*}
\begin{enumerate}

\item For $w,w'\in \qq{N}$ with $\dist_{\cG}(w,w')\leq 1$, we have
\begin{align}\begin{split}\label{e:Gsw_exp}
    (\tG-G)_{ww'}+\md(z_t) \msc(z_t)(\wt H_{ww'}- H_{ww'}) =( \wt G^{\circ}- G^{\circ})_{ww'}=\cU+\cE,
\end{split}\end{align}
where
$\cU$ is an $\OO(1)$-weighted sum of terms of the forms 
$(d-1)^{3\ell}\times \{ G^{\circ}_{wx}  G^{\circ}_{yw'}, L_{wx} G^{\circ}_{yw'},  G^{\circ}_{wx}L_{yw'}\}$ and
$\{ G^{\circ}_{wx}, L_{xw}\}\times \{ G^{\circ}_{yw'}, L_{yw'}\}\times (d-1)^{3r\ell} R_r$, where $x,y\in \{l_\al, a_\al, b_\al, c_\al\}_{\al\in \qq{\mu}}$ and $R_r$ is a W-product term (see \Cref{d:W-product}) with $r\geq 1$. 

\item    For any $w\in \qq{N}$, we have
\begin{align*}
   (1/\widetilde G_{ww})-(1/G_{ww})=\cU+\cE,
\end{align*}
where $\cU$ is an $\OO(1)$-weighted sum of terms of the form $(d-1)^{3(r+r_1)\ell}R'_{r_1,r_2}R_{r}$. Here $R'_{r_1,r_2}$ contains $r_1$ factors of the form $ G^{\circ}_{xw}$, $r_2$ factors of the form $L_{xw}$ with $x\in\{l_\al, a_\al,b_\al, c_\al\}_{\al\in \qq{\mu}}$ and an arbitrary number of factors $1/G_{ww}$; $R_r$ is a W-product term (see \Cref{d:W-product}).  Moreover, $r_1+r_2\geq 2$ is even, and $r+r_1\geq 1$.

\item 
For $s\in \cK$ and $w\in \qq{N}$,  we have 
\begin{align}\begin{split}\label{e:tG-Gexp1}
 \wt G^{\circ}_{sw}&=G^{\circ}_{sw}+  \cU+\cE, \quad s\in \cK\setminus\{i,o\},\\
 \wt G^{\circ}_{sw}&=  G^{\circ}_{sw}+ \frac{\sum_{\sfJ\in\{b,c\}}\sum_{\al\in \qq{\mu}} \fc(\sfJ, \bm1(\al\in \sfA_i))  G^{\circ}_{\sfJ_\al w}}{(d-1)^{\ell/2}}+(\Av G^{\circ})_{ow}+\cU+\cE,\quad  s\in \{i,o\}.
\end{split}\end{align}
where $|\fc(\cdot, \cdot)|\lesssim 1$; $(\Av G^{\circ})_{ow}$ represents an $\OO(1)$-weighted sum of terms of the form \eqref{e:defAvGL}; and $\cU$ is an $\OO(1)$-weighted sum of terms of the form $\{ G^{\circ}_{xw}, L_{xw}\}\times (d-1)^{3r\ell}R_{r}$, with $x\in \{l_\al, a_\al, b_\al, c_\al\}_{\al\in \qq{\mu}}$ and $R_r$ is a W-product term (see \Cref{d:W-product}) with $r\geq 1$.

Similarly, for $s\in \cK$ and $w\in \qq{N}$, we have
\begin{align}\begin{split}\label{e:t}
    &\wt L_{sw}=L_{sw}, \quad s\in \cK\setminus\{i,o\},\\
    &\wt L_{sw}=L_{sw}+\frac{ \sum_{\sfJ\in \{b,c\}}\sum_{\al\in \qq{\mu}} \fc(\sfJ, \bm1(\al\in \sfA_i)) L_{\sfJ_\al w}}{(d-1)^{\ell/2}}+(\Av L)_{ow}, \quad s\in \{i,o\},
    \end{split}\end{align}
where $|\fc(\cdot,\cdot)|\lesssim 1$; and $(\Av L)_{ow}$ represents an $\OO(1)$-weighted sum of terms of the form \eqref{e:defAvL}.

\item 
For  $(i',o')\in \cC\setminus \cC^\circ$ and $w\in \qq{N}$,  we have 
\begin{align}\begin{split}\label{e:tAvG-AvGexp1}
&
(\Av \wt G^{\circ})_{o'w}=(\Av G^{\circ})_{o'w}+ \cU+\cE,
\end{split}\end{align}
where $\cU$ is an $\OO(1)$-weighted sum of terms of the forms $ \{ G^{\circ}_{xw}, L_{xw}\}\times (d-1)^{3r\ell} R_r$ where $x\in \{l_\al, a_\al, b_\al, c_\al\}_{\al\in \qq{\mu}}$ and $R_r$ is a W-product term (see \Cref{d:W-product}) with $r\geq 1$.

Similarly, for  $(i',o')\in \cC\setminus \cC^\circ$ and $w\in \qq{N}$, $(\Av \wt L)_{o'w}=(\Av L)_{o'w}$.

\end{enumerate}

\end{proposition}

In the following proposition, we gather estimates on the difference of $Q_t, m_t$ before and after local resampling, using \Cref{lem:woodbury}.

\begin{proposition}
\label{c:Qmchange}
Adopt the same notation and assumptions as in \Cref{lem:generalQlemma}. The following holds with overwhelmingly high probability over $Z$: 
\begin{enumerate}
   \item  
$\wt Q_t-Q_t$ can be rewritten as a weighted sum
\begin{align}\begin{split}\label{e:tQ-Qexp}
  \wt Q_t-Q_t&=\frac{1}{Nd}\sum_{u\sim v}\cU^{(u,v)}+  \frac{(d-1)^\ell}{N}\cU+\cE,
\end{split}\end{align}
where $\cU^{(u,v)}$ is an $\OO(1)$-weighted sum of terms $(d-1)^{3(r+r_1)\ell}R'_{r_1,r_2}R_r$. Here $R'_{r_1, r_2}$ contains $r_1$ factors of the form $ G^{\circ}_{xu},  G^{\circ}_{xv}$, $r_2$ factors of the form $L_{xu}, L_{xv}$, with $x\in\{l_\al, a_\al,b_\al, c_\al\}_{\al\in \qq{\mu}}$ and an arbitrary number of factors $G_{uv}, 1/G_{uu}$; $R_r$ is a W-product term (see \Cref{d:W-product}). Moreover, $ r_1+r_2\geq 2$ is even and $r+r_1\geq 1$;
$\cU$ is an $\OO(1)$-weighted sum of terms $(d-1)^{3r\ell}R_{r}$, where $R_r$ is a W-product term (see \Cref{d:W-product}) with $r\geq 0$.

\item $\wt m_t-m_t$ can be rewritten as a weighted sum
\begin{align}\begin{split}\label{e:tm-mexp}
  \wt m_t-m_t=\frac{1}{dN}\sum_{u\sim v}\cU^{(u)}+\cE,
\end{split}\end{align}
where $\cU^{(u)}$ is an $\OO(1)$-weighted sum of terms of the form
$(d-1)^{3\ell} \{ G^{\circ}_{ux}L_{y u},L_{ux} G^{\circ}_{y u},  G^{\circ}_{ux} G^{\circ}_{y u}\}$, and 
$\{  G^{\circ}_{ux},L_{ux}\}\times\{ G^{\circ}_{y u}, L_{y u}\} \times (d-1)^{3r\ell}R_r$
where $x,y\in \{l_\al, a_\al, b_\al, c_\al\}_{\al\in \qq{\mu}}$, and $R_r$ is a W-product term (see \Cref{d:W-product}) with $r\geq 1$.
\end{enumerate}
In both cases the error $\cE$ satisfies $|\cE|\leq N^{-2}$. Moreover, we have the following bounds
\begin{align}\label{e:Qtmtbound}
|\wt Q_t-Q_t|, |\wt m_t-m_t|\lesssim (d-1)^{6\ell}N^\fo \Phi.
\end{align}
\end{proposition}

In the following lemma, we gather estimates on the difference of $Q_t-Y_t$, $\del_1 Y_\ell(Q_t, z+t m_t)$ and $\del_2 Y_\ell(Q_t, z+t m_t)$ before and after local resampling, using \Cref{lem:woodbury}.
\begin{proposition}
    \label{c:Q-Ylemma}
Adopt the same notation and assumptions as in \Cref{lem:generalQlemma}, The following holds with overwhelmingly high probability over $Z$: 
\begin{enumerate}
    \item 
$(\wt Q_t-\wt Y_t)$ can be rewritten as
\begin{align}\begin{split}\label{e:tQY-QYexp}
  &\phantom{{}={}}(Q_t-Y_t)+\frac{1}{Nd}\sum_{u\sim v} (1-\del_1 Y_\ell(Q_t, z+tm_t))\cU^{(u,v)}+(-t\del_2 Y_\ell)(Q_t, z+tm_t))\cU^{(u)} \\
  &+\frac{(d-1)^\ell (1-\del_1 Y_\ell(Q_t, z+tm_t))}{N}\cU
  +\cE,
\end{split}\end{align}
where $\cU^{(u,v)}, \cU$ are as in \eqref{e:tQ-Qexp}; $\cU^{(u)}$ is as in \eqref{e:tm-mexp}; and $|\cE|\lesssim \ell^3((d-1)^{6\ell}N^{\fo}\Phi)^2$. 
\item $\del_1 Y_\ell(\wt Q_t, z+t\wt m_t)$ and $\del_2 Y_\ell(\wt Q_t, z+t\wt m_t)$ can be written as
\begin{align}\begin{split}\label{e:dtY-dYexp}
    &\del_1 Y_\ell(\wt Q_t, z+t\wt m_t)=\del_1 Y_\ell( Q_t, z+t m_t)+\cE,\quad |\cE|\lesssim \ell^3(d-1)^{6\ell}N^{\fo}\Phi,\\
    &\del_2 Y_\ell(\wt Q_t, z+t\wt m_t)=\del_2 Y_\ell(Q_t, z+t m_t)+\cE,\quad |\cE|\lesssim \ell^3(d-1)^{6\ell}N^{\fo}\Phi.
\end{split}\end{align}
\end{enumerate}
Moreover, we have the following estimates \begin{align}\label{e:tilde_bound}
    \wt \Phi\lesssim \Phi,\quad  \wt \Upsilon\lesssim \Upsilon, \quad |\wt Q_t-\wt Y_t|\lesssim | Q_t- Y_t|+ (d-1)^{8\ell}\Upsilon\Phi, 
\end{align}
\end{proposition}

We gather the following computations related to \eqref{e:Gsw_exp}, \eqref{e:tG-Gexp1} and \eqref{e:tQY-QYexp}, which will be used later. 
\begin{lemma}\label{l:coefficientW}
Adopt the same notation and assumptions as in \Cref{lem:generalQlemma}.  Recall $G^\circ=G-L$, and $F_{uv}^{(\al)}$ from \eqref{e:Fuv}.
The first few terms of $\cU$ in \eqref{e:Gsw_exp} (by taking $(w,w')=(u,v)$) are given by
    \begin{align}\begin{split}\label{e:coeff2}
    \cU= \sum_{\al\in \qq{\mu}} F_{uv}^{(\al)}+\sum_{\al\neq \beta\in\qq{\mu}} \frac{L_{l_\al l_\beta}}{d-1}&\left( G^{\circ}_{u c_\al}+\frac{\msc(z_t)  G^{\circ}_{u b_\al}}{\sqrt{d-1}}\right)\left( G^{\circ}_{v c_\beta}+\frac{\msc(z_t)  G^{\circ}_{v b_\beta}}{\sqrt{d-1}}\right)+\cdots.
\end{split}\end{align}
The leading order terms in \eqref{e:tG-Gexp1} are given by
    \begin{align}\label{e:fengkai}
        \frac{\sum_{\sfJ\in\{b,c\}}\sum_{\al\in \qq{\mu}} \fc(\sfJ, \bm1(\al\in \sfA_i))  G^{\circ}_{\sfJ_\al w}}{(d-1)^{\ell/2}}=\sum_{\al\in \qq{\mu}} -\frac{L_{s l_\al}}{\sqrt{d-1}}\left( G^{\circ}_{c_\al w}+\frac{ G^{\circ}_{b_\al w}}{\sqrt{d-1}}\right).
    \end{align}
  The first few terms of $(\wt Q_t-\wt Y_t)-(Q_t-Y_t)$ in \eqref{e:tQY-QYexp} are given by
    \begin{align}\begin{split}\label{e:refine_QYdiff}
   &\phantom{{}={}}\frac{1}{Nd} \sum_{u\sim v}\sum_{\al\in \qq{\mu}} (1-\del_1 Y_\ell)\left(F_{vv}^{(\al)}-\frac{2G_{uv}}{G_{uu}}F_{uv}^{(\al)}+\left(\frac{G_{uv}}{G_{uu}}\right)^2F_{uu}^{(\al)}\right)+\del_2 Y_\ell F_{vv}^{(\al)}\\
&+\frac{1}{Nd} \sum_{u\sim v}\sum_{\al\neq \beta\in\qq{\mu}}(1-\del_1 Y_\ell)
    \frac{L_{l_\al l_\beta}}{d-1}\left( G^{\circ}_{v c_\al}+\frac{\msc(z_t)  G^{\circ}_{v b_\al}}{\sqrt{d-1}}-\frac{G_{uv}}{G_{uu}}\left( G^{\circ}_{u c_\al}+\frac{\msc(z_t)  G^{\circ}_{u b_\al}}{\sqrt{d-1}}\right)\right)\\
    &\times\left( G^{\circ}_{v c_\beta}+\frac{\msc(z_t)  G^{\circ}_{v b_\beta}}{\sqrt{d-1}}-\frac{G_{uv}}{G_{uu}}\left( G^{\circ}_{u c_\beta}+\frac{\msc(z_t)  G^{\circ}_{u b_\beta}}{\sqrt{d-1}}\right)\right)\\
& +\frac{1}{Nd}\sum_{u\sim v}\sum_{\al\neq \beta\in\qq{\mu}}(-t\del_2 Y_\ell)\frac{L_{l_\al l_\beta}}{d-1}\left( G^{\circ}_{u c_\al}+\frac{\msc(z_t)  G^{\circ}_{u b_\al}}{\sqrt{d-1}}\right)\left( G^{\circ}_{u c_\beta}+\frac{\msc(z_t)  G^{\circ}_{u b_\beta}}{\sqrt{d-1}}\right)+\cdots.
\end{split}\end{align}
\end{lemma}

Before proving \Cref{lem:generalQlemma}, we collect the following lemma on the difference of Green's functions before and after local resampling, which is an easy consequence of \eqref{e:tP-P} and \eqref{e:tG-Gdiff}.
\begin{lemma}\label{l:Gdiff}
    Adopt the assumptions and notation of \Cref{lem:woodbury}. 
    For $w,w'\in \qq{N}$,  
    \begin{align}\begin{split}\label{e:tG-Gdiff35}
        ( \wt G^{\circ}- G^{\circ})_{ww'}=\sum_{k\geq 1} \left(( G^{\circ}F)^{k}L+( G^{\circ}F)^{k} G^{\circ}+L(F G^{\circ})^{k}+LF( G^{\circ}F)^{k}L\right)_{ww'}.
    \end{split}\end{align}
    and 
    \begin{align}\begin{split}\label{e:tL-Ldiff3}
        (\wt L-L)_{ww'}=(LFL)_{ww'}.
    \end{split}\end{align} 
    If we further assume $\dist_\cG(w,w')\leq 1$, then
    \begin{align}\label{e:tG-Gdiff3}
        ( \wt G^{\circ}- G^{\circ})_{ww'}=(\tG-G)_{ww'}+\md(z_t) \msc(z_t)(\wt H_{ww'}- H_{ww'})=\eqref{e:tG-Gdiff35}.
    \end{align}
\end{lemma}
\begin{proof}[Proof of \Cref{l:Gdiff}]
We recall that $L,\wt L$ in \eqref{e:defPtP} are the restrictions of the local Green's functions on $\cB_\fR(\cF^+,\cG)$, Thus for $w,w'$ in $\cB_\fR(\cF^+,\cG)$, the claim \eqref{e:tG-Gdiff35} follows from \eqref{e:tG-Gdiff2}. 
     If $\dist_\cG(\cF^+, w')>\fR$, then $L_{x w'}=0$ for any $x\in \qq{N}$. Thus, $( \wt G^{\circ}- G^{\circ})_{ww'}=( \wt G- G)_{ww'}$, and only the second and the third terms in \eqref{e:tG-Gdiff35} are nonzero, and the claim \eqref{e:tG-Gdiff35} reduces to \eqref{e:tG-Gdiff}
    \begin{align*}
       \sum_{k\geq 1}\left(( G^{\circ}F)^{k} G^{\circ}+L(F G^{\circ})^{k}\right)_{ww'} 
       &= \sum_{k\geq 1}\left(GF( G^{\circ}F)^{k-1} G^{\circ}\right)_{ww'}\\
       &=\sum_{k\geq 1}\left(GF( G^\circ F)^{k-1}G\right)_{ww'},
    \end{align*}
    where in the last line we used $L_{xw'}=0$. By the same argument, if $\dist_\cG(\cF^+, w)>\fR$, \eqref{e:tG-Gdiff35} follows from \eqref{e:tG-Gdiff}. 

    By the same argument \eqref{e:tL-Ldiff3} follows from \eqref{e:tP-P} and \eqref{e:defF2g}. For \eqref{e:tG-Gdiff3}, if $\dist_\cG(\cF^+, w)>\fR$ or $\dist_\cG(\cF^+, w')>\fR$, then $\wt L_{w w'}=L_{ww'}=0$, and $\wt H_{ww'}-H_{ww'}=0$, so \eqref{e:tG-Gdiff3} holds.
    If $\dist_\cG(\cF^+, w)\leq \fR, \dist_\cG(\cF^+, w')\leq \fR$, then $\wt L, L$ agree with the Green's function of infinite $d$-regular tree \eqref{e:Gtreemkm} and $-\wt L_{ww'}+L_{ww'}=(\wt H_{ww'}- H_{ww'})\md(z_t) \msc(z_t)$, and \eqref{e:tG-Gdiff3} also holds.
\end{proof}

\begin{proof}[Proof of \Cref{lem:generalQlemma}]
The first statement in \eqref{e:Gsw_exp} follows from \eqref{e:tG-Gdiff3}. We start with the second statement of \eqref{e:Gsw_exp}. Thanks to \eqref{eq:infbound},
$| G^{\circ}_{xy}|\leq N^{-\fb} $ uniformly for $x,y\in \{l_\al, a_\al, b_\al, c_\al\}_{\al \in \qq{\mu}}$. We can truncate the summation \eqref{e:tG-Gdiff35} as
     \begin{align}\begin{split}\label{e:tG-Gdiff4}
        \sum_{k=1}^\fp \left(( G^{\circ}F)^{k}L+( G^{\circ}F)^{k} G^{\circ}+L(F G^{\circ})^{k}+LF( G^{\circ}F)^{k}L\right)_{ww'}+\OO(N^{-2}).
    \end{split}\end{align}


The first three terms in \eqref{e:tG-Gdiff4} are in the following form 
\begin{align}\begin{split}\label{eq:fexpansion3}
\sum_{\bms}
V_{ws_1}
F_{s_1s_2}
V_{s_2s_3}
F_{s_3s_4}
\cdots
V_{s_{2k-2}s_{2k-1}}
F_{s_{2k-1}s_{2k}}V_{s_{2k}w'},
\end{split}\end{align}
where the summation is over $\bms=(s_1,s_2,\cdots,s_{2k})\in (\{l_\al, a_\al, b_\al, c_\al\}_{\al\in \qq{\mu}})^{2k}$, $V_{ws_1}\in \{ G^{\circ}_{ws_1}, L_{ws_1}\}$, $V_{s_{2j}s_{2j+1}}=G^\circ_{s_{2j} s_{2j+1}}$ for $1\leq j\leq k-1$,  $V_{s_{2k}w'}\in \{ G^{\circ}_{s_{2k}w'}, L_{s_{2k}w'}\}$

For $k=1$ in \eqref{eq:fexpansion3}, we have
\begin{align}\label{e:biaodaling}
    \sum_{s_1 s_2}V_{ws_1}
F_{s_1s_2}V_{s_2 w'}
=\sum_{\sfJ, \sfJ'\in\{l, a,b,c\}}\sum_{\al\in \qq{\mu}} F_{\sfJ_\al \sfJ'_\al} V_{w \sfJ_\al}V_{\sfJ'_\al w'} + \sum_{\sfJ, \sfJ'\in\{l, a,b,c\}}\sum_{\al\neq \beta\in\qq{\mu}}F_{\sfJ_\al \sfJ'_\beta}  V_{w \sfJ_\al}V_{\sfJ'_\beta w'}.
\end{align}
Thanks to \eqref{e:Fbound},  the above expression is an $\OO(1)$-weighted sum of terms of the form $(d-1)^{3\ell}\times \{ G^{\circ}_{wx}L_{yw'}, L_{wx} G^{\circ}_{yw'},  G^{\circ}_{wx} G^{\circ}_{yw'}\}$, with $x,y\in \{l_\al, a_\al, b_\al, c_\al\}_{\al\in \qq{\mu}}$.

In general, for $k\geq 2$ in \eqref{eq:fexpansion3}, we can reorganize terms in the following way
\begin{align}\label{e:biaodayi}
    (d-1)^{3(k-1)\ell}\sum_{\bms}\fc_\bms V_{ws_1}V_{s_2s_3}\cdots V_{s_{2k-2}s_{2k-1}} V_{s_{2k}w'},\quad \fc_\bms=(d-1)^{-3(k-1)\ell}F_{s_1s_2}F_{s_3 s_4}\cdots F_{s_{2k-1}s_{2k}}.
\end{align}
Thanks to \eqref{e:Fbound}, the summation over the weights can be bounded as
\begin{align}\begin{split}\label{e:csxiangjia}
   \sum_{\bms}|\fc_{\bms}|
  &\lesssim (d-1)^{-3(k-1)\ell}\sum_{s_1, s_2}|F_{s_1s_2}|\sum_{s_3, s_4}|F_{s_3s_4}|\cdots 
  \sum_{s_{2k-1}, s_{2k}}|F_{s_{2k-1}s_{2k}}|\\
   &\leq (d-1)^{-3(k-1)\ell}\ell^k (d-1)^{k\ell}=(d-1)^{-(2k-2)\ell}\ell^k \lesssim (d-1)^{-\ell}. 
\end{split}\end{align}

For the last term on the right-hand side of \eqref{e:tG-Gdiff4}, we reorganize terms as
\begin{align}\label{e:biaodaer}
    (d-1)^{3k\ell}\sum_{\bms}\fc_\bms V_{w s_1}V_{s_2s_3} \cdots V_{s_{2k}s_{2k+1}} V_{s_{2k+2}w'},\quad \fc_\bms=(d-1)^{-3k\ell}F_{s_1s_2}\cdots F_{s_{2k+1}s_{2k+2}},
\end{align}
where the summation is over $\bms=(s_1,s_2,\cdots,s_{2k+2})\in (\{l_\al, a_\al, b_\al, c_\al\}_{\al\in \qq{\mu}})^{2k+2}$, $V_{ws_1}= L_{ws_1}$, $V_{s_{2j}s_{2j+1}}=G^\circ_{s_{2j} s_{2j+1}}$ for $1\leq j\leq k$,  $V_{s_{2k+2}w'}=L_{s_{2k+2}w'}$. By the same argument as in \eqref{e:csxiangjia}, the total weight is $\sum_\bms |\fc_\bms|\lesssim (d-1)^{-\ell}$. 

These terms in \eqref{e:biaodayi} and \eqref{e:biaodaer} are $\OO(1)$-weighted sums of terms of the form $\{ G^{\circ}_{sx}, L_{sx}\}\times \{ G^{\circ}_{yw}, L_{yw}\}\times (d-1)^{3 r\ell}R_r$, where $x,y\in \{l_\al, a_\al, b_\al, c_\al\}_{\al\in \qq{\mu}}$ and $R_r$ is a W-product term (see \Cref{d:W-product}) with $r\geq 1$. This finishes the proof of the first statement of \Cref{lem:generalQlemma}.

Next we prove the second statement in \Cref{lem:generalQlemma}. 
From the first statement in \Cref{lem:generalQlemma}, $\wt G_{ww}-G_{ww}$ is an $\OO(1)$-weighted sum of terms of the forms 
$(d-1)^{3\ell}\times \{ G^{\circ}_{wx}  G^{\circ}_{yw}, L_{wx} G^{\circ}_{yw},  G^{\circ}_{wx}L_{yw}\}$ and
$\{ G^{\circ}_{wx}, L_{xw}\}\times \{ G^{\circ}_{yw}, L_{yw}\}\times (d-1)^{3r\ell} R_r$, where $x,y\in \{l_\al, a_\al, b_\al, c_\al\}_{\al\in \qq{\mu}}$ and $R_r$ is a W-product term (see \Cref{d:W-product}) with $r\geq 1$. 
Notice that thanks to \eqref{eq:infbound}, $| G^{\circ}_{xw}|, | G^{\circ}_{xy}|\lesssim N^{-\fb} $ for $x,y\in \{l_\al, a_\al, b_\al, c_\al\}_{\al\in \qq{\mu}}$, thus $|\wt G_{ww}- G_{ww}|\lesssim (d-1)^{3\ell}N^{-\fb} \lesssim N^{-\fb/2}$.
By Taylor expansion we have
\begin{align}\label{e:texp}
    \frac{1}{\widetilde G_{ww}}-  \frac{1}{ G_{ww}}
    =\sum_{1\leq j\leq \fp} \frac{(G_{ww}-\wt G_{ww})^j}{G_{ww}^{j+1}}+\OO(N^{-2}).
\end{align}

From the first statement in \Cref{lem:generalQlemma}, the right-hand side of \eqref{e:texp}  is an $\OO(1)$-weighted sum of terms $(d-1)^{3(r+r_1)\ell}R'_{r_1, r_2}R_{r}$, where $R'_{r_1,r_2}$ contains $r_1$ factors of the form $ G^{\circ}_{xw}$, $r_2$ factors of the form $L_{xw}$, with $x\in\{l_\al, a_\al,b_\al, c_\al\}_{\al\in \qq{\mu}}$ and an arbitrary number of factors $1/G_{ww}$; $R_r$ is a W-product term (as in \Cref{d:W-product}). Moreover, 
the triple $(r, r_1,r_2)=(0,1,1), (0,2,0)$ or $r_1+r_2\geq 2$ is even and $r\geq 1$. They satisfy
$r+r_1\geq 1$. Then the second statement in \Cref{lem:generalQlemma} follows.

Next we prove the third statement in \Cref{lem:generalQlemma}. Similar to \eqref{e:tG-Gdiff4}, we can truncate the summation \eqref{e:tG-Gdiff35} as
     \begin{align}\begin{split}\label{e:tG-Gdiff45}
        \sum_{k=1}^\fp \left(( G^{\circ}F)^{k}L+( G^{\circ}F)^{k} G^{\circ}+L(F G^{\circ})^{k}+LF( G^{\circ}F)^{k}L\right)_{sw}+\OO(N^{-2}).
    \end{split}\end{align}
We first prove the third statement for $s\notin\{i,o\}$. In this case,
since $I(\cF^+,\cG)=1$, we have $L_{sx}=0$ for any $x\in\{l_\al, a_\al, b_\al, c_\al\}_{\al\in \qq{\mu}}$. Thus the last two terms in \eqref{e:tG-Gdiff45} vanish. Similar to \eqref{e:biaodaling} and \eqref{e:biaodayi}, the first two terms in \eqref{e:tG-Gdiff4}
are an $\OO((d-1)^{-\ell})$-weighted sum of terms of the forms 
\begin{align}\label{e:zhongjian}
     G^{\circ}_{sx}\times\{ G^{\circ}_{yw},L_{yw}\}\times  (d-1)^{3(r+1)\ell}R_r,\quad x,y\in \{l_\al, a_\al, b_\al, c_\al\}_{\al\in \qq{\mu}},
\end{align} 
where $x,y\in \{l_\al, a_\al, b_\al, c_\al\}_{\al\in \qq{\mu}}$ and $R_r$ is a W-product term (see \Cref{d:W-product}) with $r\geq 0$. This finishes the proof of the first statement in \eqref{e:tG-Gexp1}.

For $s\in\{i,o\}$, we need to understand the last two terms in \eqref{e:tG-Gdiff45}. For the third term on the right-hand side of \eqref{e:tG-Gdiff4}, with $k=1$, namely, the term 
\begin{align}\label{e:sums1s2}
    (LF G^{\circ})_{sw}=\sum_{s_1,s_2}L_{ss_1}F_{s_1 s_2} G^{\circ}_{s_2 w}=\sum_{s_2}(LF)_{s s_2} G^{\circ}_{s_2 w},
\end{align}
for $s_1,s_2\in\{l_\al, a_\al, b_\al, c_\al\}_{\al\in \qq{\mu}}$.
 Fix $s_2=\sfJ_{\al_2}$ for some $\sfJ\in \{l,a,b,c\}$ and $1\leq \alpha_2\leq \mu$. By \eqref{e:PF_bound} 
\begin{align}\label{e:sum_s1}
   (LF)_{s s_2}=(LF)_{s \sfJ_{\al_2}}=\fc(\sfJ, \bm1(\al_2\in\sfA_i))/(d-1)^{\ell/2},
\end{align}
where $\sfA_i=\{\al\in \qq{\mu}: \dist_\cT(i,l_\al)=\ell+1\}$ and the coefficients $|\fc(\cdot, \cdot)|\lesssim 1$ depend on $s\in\{i,o\}$.
Then we can further sum over $s_2$ in \eqref{e:sums1s2},
\begin{align*}\begin{split}
    &\phantom{{}={}}\sum_{s_1,s_2}L_{ss_1}F_{s_1 s_2} G^{\circ}_{s_2 w}
    =\sum_{s_2=\sfJ_{\al_2}}\fc(\sfJ, \bm1(\al_2\in\sfA_i)) \frac{G^{\circ}_{s_2 w}}{(d-1)^{\ell/2}}\\
     &=\frac{ \sum_{\sfJ}\sum_{\al\in \qq{\mu}} \fc(\sfJ,\bm1(\al\in \sfA_i))  G^{\circ}_{\sfJ_\al w}}{(d-1)^{\ell/2}}
     =\frac{ \sum_{\sfJ\in\{b,c\}}\sum_{\al\in \qq{\mu}} \fc(\sfJ,\bm1(\al\in \sfA_i))  G^{\circ}_{\sfJ_\al w}}{(d-1)^{\ell/2}}+(\Av G^{\circ})_{ow}.
\end{split}\end{align*}
where in the last equation, we collect the summation over $\sfJ\in\{l,a\}$ in $(\Av G^{\circ})_{ow}$. The above expression gives the leading terms in \eqref{e:tG-Gexp1}.

For the third term on the right-hand side of \eqref{e:tG-Gdiff45}, with $k\geq 2$, we reorganize them as
\begin{align}\label{e:biaodasan}
    (d-1)^{3(k-1)\ell}\sum_{\bms}\fc_\bms V_{s_2s_3} \cdots V_{s_{2k-2}s_{2k-1}}  G^{\circ}_{s_{2k}w},\quad \fc_\bms=(d-1)^{-3(k-1)\ell}L_{ss_1}F_{s_1s_2}\cdots F_{s_{2k-1}s_{2k}},
\end{align}
where the summation is over $\bms=(s_1,s_2,\cdots,s_{2k})\in (\{l_\al, a_\al, b_\al, c_\al\}_{\al\in \qq{\mu}})^{2k}$, and  $V_{s_{2j}s_{2j+1}}=  G^{\circ}_{s_{2j} s_{2j+1}}$ for $1\leq j\leq k-1$.
Thanks to \eqref{e:Fbound} and \eqref{e:sum_s1}, the summation over the weights can be bounded as
\begin{align}\begin{split}\label{e:youyige}
   \sum_{\bms}|\fc_{\bms}|
  &\lesssim (d-1)^{-3(k-1)\ell}\sum_{ s_2}|(LF)_{ss_2}|\sum_{s_3, s_4}|F_{s_3s_4}|\cdots 
  \sum_{s_{2k-1}, s_{2k}}|F_{s_{2k-1}s_{2k}}|\\
   &\leq (d-1)^{-3(k-1)\ell}\ell^k (d-1)^{(k-1/2)\ell}\lesssim \ell^k (d-1)^{-(2k-5/2)}\lesssim 1. 
\end{split}\end{align}

For the last term on the right-hand side of \eqref{e:tG-Gdiff4}, we reorganize them as
\begin{align}\label{e:biaodasi}
    (d-1)^{3k\ell}\sum_{\bms}\fc_\bms  V_{s_2s_3} \cdots V_{s_{2k}s_{2k+1}} L_{s_{2k+2}w},\quad \fc_\bms=(d-1)^{-3k\ell}L_{ss_1} F_{s_1s_2}\cdots F_{s_{2k+1}s_{2k+2}},
\end{align}
where $V_{s_{2j}s_{2j+1}}=  G^{\circ}_{s_{2j} s_{2j+1}}$ for $1\leq j\leq k$.
The same as in \eqref{e:youyige}, the summation over the weights can be bounded as $\sum_{\bms}|\fc_{\bms}|\lesssim 1$.

These terms in \eqref{e:biaodasan} with $k\geq 2$ and \eqref{e:biaodasi} are $\OO(1)$-weighted sum of terms of the form $ \{ G^{\circ}_{yw}, L_{yw}\}\times (d-1)^{3r\ell}R_r$, where $y\in \{l_\al, a_\al, b_\al, c_\al\}_{\al\in \qq{\mu}}$ and $R_r$ is a W-product term (see \Cref{d:W-product}) with $r\geq 1$. This finishes the proof of the second statement in \eqref{e:tG-Gexp1}.

Using \eqref{e:tL-Ldiff3} as input, the result for $\wt L_{sw}-L_{sw}$ can be proven in a similar manner, and is therefore omitted.

Next we prove the last statement in \Cref{lem:generalQlemma}. From the definition \eqref{e:defAvGL}, $(\Av G^{\circ})_{o'w}$ is an $\OO((d-1)^{\ell/2})$-weighted sum of terms of the form $ G^{\circ}_{sw}$ for $s\in \cK\setminus\{i,o\}$. By the same argument as for the third statement (see \eqref{e:tG-Gdiff45} and \eqref{e:zhongjian}), we conclude that $(\Av \wt G^{\circ})_{o'w}-(\Av G^{\circ})_{o'w}$ is an $\OO(1)$-weighted sum of terms of the forms 
$ G^{\circ}_{sx}\times\{ G^{\circ}_{yw},L_{yw}\}\times  (d-1)^{3(r+1)\ell}R_r$,
 where $x,y\in \{l_\al, a_\al, b_\al, c_\al\}_{\al\in \qq{\mu}}$ and $R_r$ is a W-product term (see \Cref{d:W-product}) with $r\geq 0$. This finishes the proof of \eqref{e:tAvG-AvGexp1}.  The statement for $(\Av\wt L)_{o'w}-(\Av L)_{o'w}$ can be proven in a similar manner using \eqref{e:tL-Ldiff3} as input.
\end{proof}

\begin{proof}[Proof of \Cref{c:Qmchange}]
The difference $\wt Q_t-Q_t$ can be rewritten as
\begin{align}\begin{split}\label{e:tQ-Qdiff}
  &\phantom{{}={}}\frac{1}{Nd}\sum_{ \{u,v\}\notin \{\{l_\al, a_\al\},\{b_\al, c_\al\}\}_{\al\in\qq{\mu}} }A_{uv}(\widetilde G_{vv}^{(u)}-G_{vv}^{(u)})+\frac{1}{Nd}\sum_{\{u,v\}\in \{\{l_\al, c_\al\},\{a_\al, b_\al\}\}_{\al\in\qq{\mu}}}(\wt G^{(u)}_{vv}-G_{vv}^{(u)})\\
    &+\frac{1}{Nd}\sum_{\{u,v\}\in \{\{l_\al, c_\al\},\{a_\al, b_\al\}\}_{\al\in\qq{\mu}}}G_{vv}^{(u)}
    -\frac{1}{Nd}\sum_{\{u,v\}\in \{\{l_\al, a_\al\},\{b_\al, c_\al\}\}_{\al\in\qq{\mu}}}G^{(u)}_{vv}. 
\end{split}\end{align}
For $G_{vv}^{(u)}, \wt G_{vv}^{(u)}$ in \eqref{e:tQ-Qdiff}, we can rewrite them using the Schur complement formula \eqref{e:Schurixj},
\begin{align}\label{e:schurexp}
    G_{vv}^{(u)}=G_{vv}-\frac{G_{uv}^2}{G_{uu}},\quad
    \wt G_{vv}^{(u)}=\wt G_{vv}-\frac{\wt G_{uv}^2}{\wt G_{uu}}.
\end{align}

By \eqref{e:Gsw_exp}, for $w,w'\in \{u,v\}$,
$\widetilde G_{ww'}-G_{ww'}$ can be rewritten as a weighted sum 
\begin{align}\label{e:diftG-G}
    \widetilde G_{ww'}-G_{ww'}=\md(z_t) \msc(z_t) (H_{ww'}-\wt H_{ww'})+\cV^{(w,w')}+\cE,
\end{align}
where $\cV^{(w,w')}$ is an $\OO(1)$-weighted sum of terms of the forms 
$(d-1)^{3\ell}\times \{ G^{\circ}_{wx}  G^{\circ}_{yw'}, L_{wx} G^{\circ}_{yw'},  G^{\circ}_{wx}L_{yw'}\}$ and
$\{ G^{\circ}_{wx}, L_{xw}\}\times \{ G^{\circ}_{yw'}, L_{yw'}\}\times (d-1)^{3r\ell} R_r$, where $x,y\in \{l_\al, a_\al, b_\al, c_\al\}_{\al\in \qq{\mu}}$ and $R_r$ is a W-product term (see \Cref{d:W-product}) with $r\geq 1$;  here $|\cE|\lesssim N^{-2}$.

For the difference $\widetilde G_{vv}^{(u)}- G_{vv}^{(u)}$, using \eqref{e:schurexp} and \eqref{e:texp}, we can rewrite it as
\begin{align}\begin{split}\label{e:Guuvdiff}
  \widetilde G_{vv}^{(u)}- G_{vv}^{(u)}
  &=  (\wt G_{vv}-G_{vv})
  - \frac{G^2_{uv}}{G_{uu}}\sum_{1\leq j\leq \fp }\left(\frac{ G_{uu}-\wt G_{uu}}{G_{uu}}\right)^j\\
  &-\frac{2G_{uv}(\wt G_{uv} -G_{uv})+(\wt G_{uv} -G_{uv})^2}{G_{uu}}\sum_{0\leq j\leq \fp }\left(\frac{ G_{uu}-\wt G_{uu}}{G_{uu}}\right)^j+\OO(N^{-2}).
\end{split}\end{align}
Then we can replace the differences $\wt G_{vv}-G_{vv}$, $\wt G_{uu}-G_{uu}$ and $\wt G_{uv}-G_{uv}$ by $\cV^{(v,v)}$, $\cV^{(u,u)}$, and $\cV^{(u,v)}$ from \eqref{e:diftG-G}, giving
\begin{align}\begin{split}\label{e:tG-Gexpa}
   &\widetilde G_{vv}^{(u)}- G_{vv}^{(u)}= \cU^{(u,v)}+\cE,\quad \{u,v\}\notin \{\{l_\al, a_\al\}, \{b_\al, c_\al\},\{l_\al, c_\al\},\{a_\al, b_\al\}\}_{\al\in \qq{\mu}},
\end{split}\end{align}
where $\cU^{(u,v)}$ is an $\OO(1)$-weighted sum of terms $(d-1)^{3(r+r_1)\ell}R'_{r_1,r_2}R_r$, where $R'_{r_1,r_2}$ contains $r_1$ factors of the form $ G^{\circ}_{xu},  G^{\circ}_{xv}$, $r_2$ factors of the form $L_{xu}, L_{xv}$ for $x\in\{l_\al, a_\al,b_\al, c_\al\}_{\al\in \qq{\mu}}$, and an arbitrary number of factors $G_{uv}, 1/G_{uu}$; $R_r$ is a W-product term (see \Cref{d:W-product}). Moreover, $r_1+r_2\geq 2$ is even and $r_1+r\geq 1$. 

For $\{u,v\}\in \{\{l_\al, c_\al\},\{a_\al, b_\al\}\}_{\al\in \qq{\mu}}$, we have $\md(z_t) \msc(z_t) (H_{uv}-\wt H_{uv})=-\md(z_t)\msc(z_t)/\sqrt{d-1}$, and $\cV^{(u,u)}$, $\cV^{(v,v)}$ and $\cV^{(u,v)}$ are $\OO(1)$-weighted sum of terms of the form  $(d-1)^{3r\ell}R_{r}$, where $R_r$ is a W-product term (see \Cref{d:W-product}), and $r\geq 0$. Plugging \eqref{e:diftG-G} into \eqref{e:tG-Gexpa}, we conclude 
\begin{align}\label{e:tG-Gexpa2}
   &\widetilde G_{vv}^{(u)}- G_{vv}^{(u)}= \overline\cU^{(u,v)}+\cE,\quad \{u,v\}\in \{\{l_\al, c_\al\},\{a_\al, b_\al\}\}_{\al\in \qq{\mu}},
\end{align}
where $\overline \cU^{(u,v)}$ is an  $\OO(1)$-weighted sum of terms $(d-1)^{3r\ell}R_{r}$, where $R_r$ is a W-product term (see \Cref{d:W-product}), and $r\geq 0$. 

Thanks to \eqref{e:tG-Gexpa} and \eqref{e:tG-Gexpa2}, up to error $\OO(N^{-2})$, we can rewrite \eqref{e:tQ-Qdiff} as
\begin{align}\begin{split}\label{e:tQ-Qdiff2}
  &\phantom{{}={}}\frac{1}{Nd}\sum_{ u\sim v} \cU^{(u,v)}-\frac{1}{Nd}\sum_{ \{u,v\}\notin \{\{l_\al, a_\al\},\{b_\al, c_\al\}\}_{\al\in\qq{\mu}} }\cU^{(u,v)}+\frac{1}{Nd}\sum_{\{u,v\}\in \{\{l_\al, c_\al\},\{a_\al, b_\al\}\}_{\al\in\qq{\mu}}}\overline \cU^{(u,v)}\\
    &+\frac{1}{Nd}\sum_{\{u,v\}\in \{\{l_\al, c_\al\},\{a_\al, b_\al\}\}_{\al\in\qq{\mu}}}G_{vv}^{(u)}
    -\frac{1}{Nd}\sum_{\{u,v\}\in \{\{l_\al, a_\al\},\{b_\al, c_\al\}\}_{\al\in\qq{\mu}}}G^{(u)}_{vv}.
\end{split}\end{align}

For the last line of \eqref{e:tQ-Qdiff2}, we notice that for $\{u,v\}\in \{\{l_\al, a_\al\}, \{b_\al, c_\al\},\{l_\al, c_\al\},\{a_\al, b_\al\}\}_{\al\in \qq{\mu}}$.
\begin{align*}
   G_{vv}^{(u)}&=L_{vv}-\frac{L_{uv}^2}{L_{uu}}+ G^{\circ}_{vv}
  +\frac{2L_{uv} G^{\circ}_{uv} }{L_{uu}}\sum_{0\leq j\leq \fp }\frac{ (-G^{\circ}_{uu})^j}{L^j_{uu}}+\frac{ (G^{\circ})^2_{uv}}{L_{uu}}\sum_{0\leq j\leq \fp }\frac{ (-G^{\circ}_{uu})^j}{L^j_{uu}}+\OO(N^{-2}),
\end{align*}
which is an $\OO(1)$-weighted sum of terms $(d-1)^{3r\ell}R_{r}$, where $R_r$ is a W-product term (see \Cref{d:W-product}), and $r\geq 0$. The last four terms in \eqref{e:tQ-Qdiff2} give $\cU$ in \eqref{e:tQ-Qexp}.

This concludes the proof of the first statement in \Cref{c:Qmchange}. The second statement follows from the same reasoning, so we omit its proof.

Next we prove \eqref{e:Qtmtbound} for $|\wt Q_t-Q_t|$, the statement for $|\wt m_t-m_t|$ follows from \eqref{e:tmmdiff}. Thanks to \eqref{e:Gsw_exp}
\begin{align}\label{e:tQ-Qdiffcopy}
|\wt Q_t-Q_t|\leq \frac{1}{Nd}\sum_{u\sim v}|\cU^{(u,v)}|+  \frac{(d-1)^\ell}{N}|\cU|+\OO(N^{-2}).
\end{align}

Thanks to \eqref{eq:infbound}, with overwhelmingly high probability,  factors involved in W-product term (see \Cref{d:W-product}) are all bounded by $N^{-\fb} $. In \eqref{e:tQ-Qdiffcopy}, $\cU$ is an $\OO(1)$-weighted sum of terms $(d-1)^{3r\ell}R_{r}$, where $R_r$ is a W-product term with $r\geq 0$. It is bounded as $|(d-1)^{3r\ell}R_{r}|\lesssim (d-1)^{3r\ell} N^{-r\fb} \lesssim 1$.

In \eqref{e:tQ-Qdiffcopy}, $\cU^{(u,v)}$ is an $\OO(1)$-weighted sum of terms $(d-1)^{3(r+r_1)\ell}R'_{r_1,r_2}R_r$, where $R'_{r_1, r_2}$ contains $r_1$ factors of the form $ G^{\circ}_{xu},  G^{\circ}_{xv}$, $r_2$ factors of the form $L_{xu}, L_{xv}$, with $x\in\{l_\al, a_\al,b_\al, c_\al\}_{\al\in \qq{\mu}}$ and an arbitrary number of factors $G_{uv}, 1/G_{uu}$; $R_r$ is a W-product term (see \Cref{d:W-product}). Moreover, $ r_1+r_2\geq 2$ is even and $r+r_1\geq 1$. We can bound it as
\begin{align*}
    |(d-1)^{3(r+r_1)\ell}R'_{r_1,r_2}R_r|
\lesssim (d-1)^{3(r+r_1)\ell}N^{-r\fb} |R'_{r_1,r_2}|\lesssim (d-1)^{3r_1\ell} |R'_{r_1,r_2}|.
\end{align*}
Then we can sum over $u\sim v$, by the same argument as for \eqref{e:1_Sbound}, using a Cauchy-Schwarz inequality and \eqref{e:naive-Ward},
\begin{align*}
    \frac{1}{Nd}\sum_{u\sim v}(d-1)^{3r_1\ell} |R'_{r_1,r_2}|
    \lesssim (d-1)^{3r_1\ell}N^{\fo} 
    \left\{
    \begin{array}{cc}
      N^{-(r_{1}-2)\fb} \Phi  & r_{2}=0 \\
       N^{-\max\{r_{1}-1,0\}\fb} \sqrt{\Phi/N}  & r_{2}\geq 1, 
    \end{array}
    \right.
\end{align*}
which is bounded by $(d-1)^{6\ell} N^\fo \Phi$, by noticing $\Phi\gtrsim 1/N$. The claim \eqref{e:Qtmtbound}  follows from plugging these above estimates into \eqref{e:tQ-Qdiffcopy}.
\end{proof}

\begin{proof}[Proof of \Cref{c:Q-Ylemma}]
Thanks to \eqref{e:Yl_derivative}, by Taylor expansion, we have 
\begin{align}\begin{split}\label{e:tQY-QYleading_term}
    (\wt Q_t-\wt Y_t)-(Q_t-Y_t)
    &= (1-\del_1 Y_\ell(Q_t, z+tm_t))(\wt Q_t-Q_t)\\
    &+(-t\del_2 Y_\ell)(Q_t, z+tm_t))(\wt m_t-m_t)
  +\OO(\ell^3(|\wt Q_t-Q_t|^2+t^2|\wt m_t-m_t|^2)).
\end{split}\end{align}
The claim \eqref{e:tQY-QYexp} follows from bounding $\wt Q_t-Q_t$ and $\wt m_t-m_t$ using \eqref{e:Qtmtbound}, which gives $|\cE|\lesssim \ell^3((d-1)^{6\ell}N^\fo \Phi)^2$.

The statement \eqref{e:dtY-dYexp} also follows from performing a Taylor expansion using \eqref{e:Yl_derivative} and \eqref{e:Qtmtbound}
\begin{align*}\begin{split}
    &\del_1 Y_\ell(\wt Q_t, z+t\wt m_t)=\del_1 Y_\ell( Q_t, z+t m_t)+\OO(\ell^3(|\wt Q_t-Q_t|+t|\wt m_t-m_t|)),\\
    &\del_2 Y_\ell(\wt Q_t, z+t\wt m_t)=\del_2 Y_\ell(Q_t, z+t m_t)+\OO(\ell^3(|\wt Q_t-Q_t|+t|\wt m_t-m_t|)).
\end{split}\end{align*}

Next we prove \eqref{e:tilde_bound}. We recall from \eqref{e:tmmdiff}, for $z\in \bf D$ (recall  from \eqref{e:D})
\begin{align}\label{e:tmmdiffcopy}
    |\wt m_t(z)-m_t(z)|\lesssim \frac{(d-1)^\ell N^\fo \Im[m_t(z)]}{N\Im[z]}\ll \Im[m_t(z)],
\end{align}
where we used that for $z\in \bf D$, $\Im[z]\geq N^{-1+\fg}$ and $(d-1)^{6\ell}N^{\fo}/N\Im[z]\ll1 $. We conclude from \eqref{e:tmmdiffcopy}, and the definition of $\Phi$ from \eqref{e:defPhi}
\begin{align}\label{e:Immt_change}
   \wt \Phi=\frac{\Im[\wt m_t]}{N\Im[z]}+\frac{1}{N^{1-2\fc}}\lesssim \frac{\Im[ m_t]}{N\Im[z]}+\frac{1}{N^{1-2\fc}}= \Phi,
\end{align}
 and the expansion \eqref{e:tQY-QYleading_term} gives
\begin{align*}\begin{split}
    &\phantom{{}={}}|(\wt Q_t-\wt Y_t)-(Q_t- Y_t)| \\
    &\lesssim |1-\del_1 Y_\ell||\wt Q_t-Q_t|+t|\del_2 Y_\ell||\wt m_t-m_t|+\OO(\ell^3(|\wt Q_t-Q_t|+t|\wt m_t-m_t|)^2)\\
    &\lesssim (d-1)^{6\ell}N^{\fo}(|1-\del_1 Y_\ell|+t|\del_2 Y_\ell|) \Phi+\ell^3((d-1)^{6\ell}N^{\fo}\Phi)^2\lesssim (d-1)^{8\ell}\Upsilon\Phi,
\end{split}\end{align*}
where in the last line we used \eqref{e:Qtmtbound}, and the definition of $\Upsilon$ from \eqref{e:defPhi}. The last statement of \eqref{e:tilde_bound} follows.
Moreover, thanks to \eqref{e:dtY-dYexp}, we have 
\begin{align*}
    \wt \Upsilon=|1-\del_1 \wt Y_\ell|+t|\del_2 \wt Y_\ell|+(d-1)^{8\ell}\wt \Phi
    \leq (|1-\del_1 Y_\ell|+t|\del_2 Y_\ell|)  +\OO((d-1)^{8\ell}\Phi)\lesssim\Upsilon.
\end{align*}
\end{proof}

 \begin{proof}[Proof of \Cref{l:coefficientW}]
We recall $L,\wt L$ from \eqref{e:defPtP}, and the matrix $F$ from \eqref{e:defF}. Notice that $L,\wt L$ agree with the Green's function of infinite $d$-regular tree \eqref{e:Gtreemkm}, we have
\begin{align}\begin{split}\label{e:tP_exp}
    &\wt L_{s c_\al}=-\frac{\msc(z_t) }{\sqrt{d-1}}\wt L_{s l_\al},\quad \wt L_{s b_\al}=\wt L_{s a_\al}=0,\text{ for } s\in\{i,o\}, \quad \al\in\qq{\mu},\\
   &\wt L_{l_\al c_\beta}=
   \wt L_{c_\al l_\beta}=-\frac{\msc(z_t)}{\sqrt{d-1}}L_{l_\al l_\beta},\quad 
   \wt L_{c_\al c_\beta}=\frac{\msc^2(z_t)}{d-1}L_{l_\al l_\beta},\text{ for }  \al\neq \beta\in \qq{\mu}.
\end{split}\end{align}

  The two terms in \eqref{e:coeff2} are from the expansion of $(G^\circ F G^\circ )_{uv}$ as in \eqref{e:tG-Gdiff4}, which is given by
  \begin{align}\begin{split}\label{e:all_index}
\sum_{\al,\beta\in\qq{\mu}}\frac{ G^{\circ}_{uc_\al}F_{c_\al c_\beta} G^{\circ}_{vc_\beta}}{d-1}
  +\frac{ G^{\circ}_{uc_\al}F_{c_\al b_\beta} G^{\circ}_{v b_\beta}}{d-1}
  +\frac{ G^{\circ}_{u b_\al}F_{b_\al c_\beta} G^{\circ}_{vc_\beta}}{d-1}
  +\frac{ G^{\circ}_{ub_\al}F_{b_\al b_\beta} G^{\circ}_{vb_\beta}}{d-1}.
\end{split}\end{align}
For $\al\neq \beta$ and $\sfJ,\sfJ'\in\{b,c\}$, using \eqref{e:defF} we have
$F_{\sfJ_\al \sfJ'_\beta}=(\xi_\al \wt L \xi_\beta)_{\sfJ_\al \sfJ'_\beta}$,
and we can rewrite \eqref{e:all_index} as
\begin{align}\begin{split}\label{e:different_index}
&\phantom{{}={}}\sum_{\al\neq \beta\in\qq{\mu}}\frac{ G^{\circ}_{uc_\al}\widetilde L_{l_\al l_\beta} G^{\circ}_{vc_\beta}}{d-1}
  -\frac{ G^{\circ}_{uc_\al}\widetilde L_{l_\al c_\beta} G^{\circ}_{v b_\beta}}{d-1}
  -\frac{ G^{\circ}_{u b_\al}\widetilde L_{c_\al l_\beta} G^{\circ}_{vc_\beta}}{d-1}
  +\frac{ G^{\circ}_{ub_\al}\widetilde L_{c_\al c_\beta} G^{\circ}_{vb_\beta}}{d-1}\\
&= \sum_{\al\neq \beta\in\qq{\mu}} \frac{L_{l_\al l_\beta}}{d-1}\left( G^{\circ}_{uc_\al} G^{\circ}_{vc_\beta}
  +\frac{\msc(z_t)  G^{\circ}_{uc_\al} G^{\circ}_{v b_\beta}}{\sqrt{d-1}}
  +\frac{\msc(z_t)  G^{\circ}_{u b_\al} G^{\circ}_{vc_\beta}}{\sqrt{d-1}}
  +\frac{\msc^2(z_t)  G^{\circ}_{ub_\al} G^{\circ}_{vb_\beta}}{d-1}\right)\\
  &=\sum_{\al\neq \beta\in\qq{\mu}} \frac{L_{l_\al l_\beta}}{d-1}\left( G^{\circ}_{u c_\al}+\frac{\msc(z_t)  G^{\circ}_{u b_\al}}{\sqrt{d-1}}\right)\left( G^{\circ}_{v c_\beta}+\frac{\msc(z_t)  G^{\circ}_{v b_\beta}}{\sqrt{d-1}}\right),
\end{split}\end{align}
where we used \eqref{e:tP_exp} in the second line.

For $\al=\beta$ and $\sfJ,\sfJ'\in\{b,c\}$, using \eqref{e:defF} we have
$F_{\sfJ_\al \sfJ'_\al}=(\xi_\al+\xi_\al \wt L \xi_\al)_{\sfJ_\al \sfJ'_\al}$, and we can rewrite \eqref{e:all_index} as
\begin{align}\begin{split}\label{e:same_index}
&\phantom{{}={}}\sum_{\al\in \qq{\mu}} ( G^{\circ}_{uc_\al} G^{\circ}_{vb_\alpha}+ G^{\circ}_{ub_\al} G^{\circ}_{vc_\alpha})\frac{(\xi_\al+\xi_\al\wt L\xi_\al)_{c_\al b_\al}}{d-1}+ G^{\circ}_{ub_\al} G^{\circ}_{vb_\alpha}\frac{(\xi_\al\wt L\xi_\al)_{b_\al b_\al}}{d-1}+ G^{\circ}_{uc_\al} G^{\circ}_{vc_\alpha}\frac{(\xi_\al\wt L\xi_\al)_{c_\al c_\al}}{d-1}\\
&=\sum_{\al\in \qq{\mu}} ( G^{\circ}_{uc_\al} G^{\circ}_{vb_\alpha}+ G^{\circ}_{ub_\al} G^{\circ}_{vc_\alpha})\left(\frac{1}{\sqrt{d-1}}+\frac{2\md(z_t) \msc(z_t)}{(d-1)\sqrt{d-1}}\right)+\frac{2\md(z_t)}{d-1}\left( G^{\circ}_{ub_\al} G^{\circ}_{vb_\alpha}+ G^{\circ}_{uc_\al} G^{\circ}_{vc_\alpha}\right).
\end{split}\end{align}
The claim \eqref{e:coeff2} follows from combining \eqref{e:different_index} and \eqref{e:same_index}.

The left-hand side of \eqref{e:fengkai} is from \eqref{e:sum_s1}, 
 \begin{align}\label{e:leadfengkai}
        \frac{\sum_{\sfJ\in\{b,c\}}\fc(\sfJ, \bm1(\al\in \sfA_i))  G^{\circ}_{\sfJ_\al w}}{(d-1)^{\ell/2}}=
        (LF)_{s c_\al} G^{\circ}_{c_\al w} + (LF)_{sb_\al} G^{\circ}_{b_\al w}.
    \end{align}
Using \eqref{e:PF_bound} and \eqref{e:tP_exp}, explicit computation leads to    
\begin{align*}\begin{split}
    (LF)_{s c_\al}&=\frac{1}{\sqrt{d-1}}\sum_{x\sim c_\al} \wt L_{s x} -z_t\wt L_{s c_\al}
    =-\msc(z_t) \wt L_{s c_\al} -z_t\wt L_{s c_\al}\\
    &=\frac{\msc(z_t) (\msc(z_t)+z_t) }{\sqrt{d-1}}L_{s l_\al}
    =-\frac{L_{s l_\al} }{\sqrt{d-1}}.
    \\
    (LF)_{s b_\al}&=\frac{1}{\sqrt{d-1}}\sum_{x\sim b_\al} \wt L_{s x} -z_t\wt L_{s b_\al}
    =\frac{1}{\sqrt{d-1}}\wt L_{s c_\al}
    =-\frac{\msc(z_t)}{d-1} L_{s l_\al}.
\end{split}\end{align*}
 The claim \eqref{e:fengkai} follows from plugging the above formulas into \eqref{e:leadfengkai}.

  The leading terms in $\wt Q_t- Q_t$ are from the leading terms in \eqref{e:Guuvdiff}, which are given by 
    \begin{align}\label{e:Guuvdiff_leading}
       (\wt G_{vv}-G_{vv})-\frac{2G_{uv}(\wt G_{uv} -G_{uv})}{G_{uu}}
  + \frac{G^2_{uv}(\wt G_{uu}-G_{uu})}{G^2_{uu}}.
    \end{align}
    By plugging \eqref{e:coeff2} into \eqref{e:Guuvdiff_leading}, we get
    that the first few terms in $\wt Q_t-Q_t$ are given by
    \begin{align}\begin{split}\label{e:tQ-Q_leading_term}
        &\phantom{{}={}}\frac{1}{Nd} \sum_{\al\in \qq{\mu}}\sum_{u\sim v}\left(F_{vv}^{(\al)}-\frac{2G_{uv}}{G_{uu}}F_{uv}^{(\al)}+\left(\frac{G_{uv}}{G_{uu}}\right)^2F_{uu}^{(\al)}\right)
        \\
        &+\frac{1}{Nd} \sum_{\al\neq \beta\in\qq{\mu}} \sum_{u\sim v}\frac{L_{l_\al l_\beta}}{d-1}\left(\left( G^{\circ}_{v c_\al}+\frac{\msc(z_t)  G^{\circ}_{v b_\al}}{\sqrt{d-1}}\right)\left( G^{\circ}_{v c_\beta}+\frac{\msc(z_t)  G^{\circ}_{v b_\beta}}{\sqrt{d-1}}\right)\right.\\
        &-\frac{2G_{uv}}{G_{uu}}\left( G^{\circ}_{u c_\al}+\frac{\msc(z_t)  G^{\circ}_{u b_\al}}{\sqrt{d-1}}\right)\left( G^{\circ}_{v c_\beta}+\frac{\msc(z_t)  G^{\circ}_{v b_\beta}}{\sqrt{d-1}}\right)\\
        &\left.+\left(\frac{G_{uv}}{G_{uu}}\right)^2\left( G^{\circ}_{u c_\al}+\frac{\msc(z_t)  G^{\circ}_{u b_\al}}{\sqrt{d-1}}\right)\left( G^{\circ}_{u c_\beta}+\frac{\msc(z_t)  G^{\circ}_{u b_\beta}}{\sqrt{d-1}}\right)\right).
    \end{split}\end{align}
The second term in \eqref{e:tQ-Q_leading_term} factorizes as
\begin{align}\begin{split}\label{e:tQ-Q_leading_term2} 
    &\left( G^{\circ}_{v c_\al}+\frac{\msc(z_t)  G^{\circ}_{v b_\al}}{\sqrt{d-1}}-\frac{G_{uv}}{G_{uu}}\left( G^{\circ}_{u c_\al}+\frac{\msc(z_t)  G^{\circ}_{u b_\al}}{\sqrt{d-1}}\right)\right)\times\\
    &\times\left( G^{\circ}_{v c_\beta}+\frac{\msc(z_t)  G^{\circ}_{v b_\beta}}{\sqrt{d-1}}-\frac{G_{uv}}{G_{uu}}\left( G^{\circ}_{u c_\beta}+\frac{\msc(z_t)  G^{\circ}_{u b_\beta}}{\sqrt{d-1}}\right)\right).
\end{split}\end{align}
By the same argument, the first few terms in $\wt m_t-m_t$ are given by
\begin{align}\label{e:tm-m_leading_term} 
\frac{1}{Nd} \sum_{\al\in \qq{\mu}}\sum_{u\sim v}F_{vv}^{(\al)}+\frac{1}{Nd}\sum_{u\sim v}\sum_{\al\neq \beta\in\qq{\mu}}\frac{L_{l_\al l_\beta}}{d-1}\left( G^{\circ}_{u c_\al}+\frac{\msc(z_t)  G^{\circ}_{u b_\al}}{\sqrt{d-1}}\right)\left( G^{\circ}_{u c_\beta}+\frac{\msc(z_t)  G^{\circ}_{u b_\beta}}{\sqrt{d-1}}\right).
\end{align}
The claim \eqref{e:refine_QYdiff} follows from plugging \eqref{e:tQ-Q_leading_term}, \eqref{e:tQ-Q_leading_term2} and \eqref{e:tm-m_leading_term} into \eqref{e:tQY-QYleading_term}.

\end{proof}

\section{Bounds on Error Terms}\label{e:error_term}
In \Cref{sec:expansions}, we introduced various expansions for the differences in Green's functions resulting from local resampling. In this section, we demonstrate that many of these terms are negligibly small, a critical aspect of the iteration scheme described in \Cref{s:proofoutline}.

In \Cref{s:Green_est}, we provide a collection of estimates for the Green's function, leveraging the fact that $\bm1$ is a trivial eigenvector and utilizing the tree structure.
In \Cref{s:change_est}, we present bounds on the error terms $\cE$ derived in \Cref{lem:diaglem} and \Cref{lem:offdiagswitch}.
In \Cref{s:adm_bound}, we gather estimates for the admissible functions as defined in \Cref{def:pgen}.
Finally, in \Cref{s:Z_term}, we focus on estimates involving the constrained GOE matrix  $Z$.

\subsection{Setting and notation}
\label{s:setting5}
In this section, let $d\geq 3$ and $\cG$ be a $d$-regular graph on $N$ vertices. Let $\cF = (\bfi, E)$ be a forest as in \eqref{e:cFtocF+}, with switching edges $\cK$, core edges $\cC$ and unused core edges $\cC^\circ$. We view $\cF$ as a subgraph of $\cG$. We construct $\cF^+ = (\bfi^+, E^+)$ (as in \eqref{e:cF++}) by performing a local resampling around $(i, o) \in \cC^\circ$ with resampling data ${\bf S}=\{(l_\al, a_\al), (b_\al, c_\al)\}_{\al\in \qq{\mu}}$ where $\mu=d(d-1)^{\ell}$. We denote $\cT=\cB_\ell(o,\cG)$ with vertex set $\bT$. Let the switched graph be $\widetilde \cG = T_\bfS(\cG)$.
We recall the sets $\Omega$ and $\oOmega$ of $d$-regular graphs from \Cref{def:omegabar} and \Cref{thm:prevthm0}, and 
and the indicator function $I(\cF,\cG)$ from \eqref{def:indicator}.

Fix $d\geq 3$. We recall the spectral domain $\bf D$ from \eqref{e:D}, and parameters $\fo\ll \ft\ll\fb\ll\fc\ll\fg$ from \eqref{e:parameters}. Fix time $t\leq N^{-1/3+\ft}$. We recall $\varrho_d(x,t)$, $\md(z,t)$ and $E_t$ from \eqref{e:defrhodt} and \eqref{e:edgeeqn2}. For any parameter $z\in \bf D$ we denote $\eta=\Im[z]$, $\kappa=\min\{|\Re[z]-E_t|, |\Re[z]+E_t|\}$, and $z_t=z+t\md(z,t)=z+t\md(z_t)$ (recall from \eqref{e:defw}). 

We recall the matrix $H(t)$, its Green's function $G(z,t)$, its Stieltjes transform $m_t(z)$, the quantities  $Q_t(z)$,  $Y_t(z)=Y_\ell(Q_t(z),z+tm_t(z))$, and   $X_t(z)=X_\ell(Q_t(z),z+tm_t(z))$ from \eqref{e:Ht} , \eqref{e:Gt}, \eqref{def_mtz}, \eqref{e:Qsum}, and  \eqref{e:defYt}. 
We recall the control parameters $\Phi(z), \Upsilon(z)$ and $\Psi_p(z)$ from \eqref{e:defPhi} and  \eqref{eq:phidef}. We recall the local Green's function $L(z,t)=L(z,t,\cF^+,\cG)$ from \eqref{e:local_Green}, and $G^\circ(z,t)=G(z,t)-L(z,t)$ from \eqref{e:G-L}. We denote the corresponding quantities for the switched graph $\tcG$ as $\wt H(t)$, $\wt G(z,t)$, $\wt m_t(z)$, $\wt Q_t(z)$,  $\wt Y_t(z)$, $\wt X_t(z)$,$\wt \Phi(z)$, $\wt \Upsilon(z)$,$\wt \Psi_p(z)$, $\wt L(z,t)$ and $\wt G^\circ(z,t)$. If the context is clear, we may omit the dependence on $z$ and $t$. 

We recall the array $\bmr=[r_{jk}]_{1\leq j\leq p-1, 0\leq k\leq 2}$ used in the definition of the admissible function \Cref{def:pgen}: for any $1\leq j\leq p-1$,
\begin{align}\label{e:defr_copy}
    (r_{j0},r_{j1},r_{j2})\in\{(1,0,0), (3,0,0)\}\cup \{(2,r_1,r_2): r_1+r_2\geq 2 \text{ is even}\}.
\end{align}

\subsection{Green's Function Estimates}
\label{s:Green_est}
In \Cref{p:upbb}, we provide a collection of estimates for the Green's function, leveraging the fact that $\bm1$ is a trivial eigenvector. It states that the expectation of $G_{c_\al c_\al}^{(b_\al)}-Q_t$ and $G_{c_\al c_\beta}^{(b_\al b_\beta)}$, with respect to the randomness of the local resampling, are small. 
Before stating \Cref{p:upbb}, let us first introduce some indicator functions. 
\begin{definition}\label{def:Ic}
    Given a $d$-regular graph $\cG$ on $N$ vertices, we introduce indicator functions 
   $I_c, I_{cc'}, I_{cc'c''}$ for $c,c',c''\in \qq{N}$ as follows: 
\begin{enumerate}
    \item $I_c=1$ if and only if the vertex $c$ has a radius $\fR/2$ tree neighbhorhood. Note that if $\GG\in \oOmega$, $\sum_{c\in \qq{N}} |1-I_c|\leq N^{3\fc/2}$.
    \item $I_{cc'}=I_c I_{c'} \widehat{I}_{cc'}$ where $I_c, I_{c'}$ are as above. $\widehat I_{cc'}=1$ if and only if 
    $\dist_{\cG}(c,c')> \fR/2$. Note that if $\GG\in \oOmega$, $\sum_{c\in \qq{N}} |1-\widehat I_{cc'}|, \sum_{c'\in \qq{N}} |1-\widehat I_{cc'}|\leq N^{3\fc/2}$.
    \item $I_{cc'c''}=I_c I_{c'} I_{c''}\widehat{I}_{cc'}\widehat{I}_{cc''}\widehat{I}_{c'c''}$ where $I_c, I_{c'},I_{c''},\widehat{I}_{cc'},\widehat{I}_{cc''},\widehat{I}_{c'c''}$ are as above. 
\end{enumerate}
\end{definition}

\begin{remark}
   Later we will use indicator functions $I_c, I_{cc'}$ in \Cref{def:Ic} to decompose the indicator function $I(\cF,\cG)$ from \Cref{def:indicator}. More precisely, assume $\cG\in \Omega$, and let $(b,c)\neq  (b',c')$ be two unused core edges of $\cF=(\bfi, E)$. 
   
   If we define the forest $\widehat \cF$ from $\cF$ by removing $\{(b,c)\}$,
$
    \widehat \cF=(\widehat \bfi, \widehat E)=\cF\setminus  \{(b,c)\},\quad \widehat \bfi=\bfi\setminus\{b,c\}
$, then 
$I(\cF,\cG)=A_{bc}I_{c}I(\widehat \cF, \cG)$,
 where $I_{c}=\bm1(c\not\in \bX)$ and 
$\bX$ is the collection of vertices $v$ such that either there exists some $u\in \cB_\ell(v,\cG)$ such that $\cB_\fR(u,\cG)$ is not a tree; or $\dist(v,c'')<3\fR$ for some $(b'',c'')\in \cC\setminus\{(b, c)\}$. $I_c$ satisfies the requirements in \Cref{def:Ic}.

Alternatively, if we define the forest $\widehat \cF$ from $\cF$ by removing $\{(b,c), (b',c')\}$,
$
    \widehat \cF=(\widehat \bfi, \widehat E)=\cF\setminus  \{(b,c), (b',c')\},\quad \widehat \bfi=\bfi^+\setminus\{b,c,b',c'\}
$, then $
I(\cF,\cG)=A_{bc }A_{b'c'} I_{cc'}I(\widehat \cF, \cG)$, 
 where $I_{cc'}=\bm1(c \not\in \bX)\bm1(c' \not\in \bX)\bm1(\dist_\cG(c,c')\geq 3\fR)$ and 
$\bX$ is the collection of vertices $v$ such that either there exists some $u\in \cB_\ell(v,\cG)$ such that $\cB_\fR(u,\cG)$ is not a tree; or $\dist(v,c'')<3\fR$ for some $(b'',c'')\in \cC\setminus\{(b, c), (b', c')\}$. $I_{cc'}$ satisfies the requirements in \Cref{def:Ic}.

\end{remark}

\begin{proposition}
\label{p:upbb}
Adopt the notation of \Cref{s:setting5}, recall the indicator functions $I_c, I_{cc'}$ from \Cref{def:Ic}, and assume $z\in {\bf D}, \cG\in\Omega$. 
Then the following holds with overwhelming probability over $Z$:
 \begin{align}\label{e:single_index_sum}
        \frac{1}{Nd}\sum_{b \sim c }(G_{c  c }^{(b )}-Q_t)I_c, 
        \quad 
       \frac{1}{Nd}\sum_{b\sim  c }G_{b  x}I_c 
        \quad 
       \frac{1}{Nd}\sum_{b \sim c } G_{c  x}I_c 
       =\OO\left(\frac{1}{N^{1-3\fc/2} }\right),
    \end{align}
    and 
    \begin{align}\begin{split}\label{e:Gccbb}
      &\phantom{{}={}}\frac{A_{b'c'}}{Nd}\sum_{b\sim c } \left(G_{c  c'}^{(b  b')}-\frac{G_{b  c'}^{(b')}(G_{c  c }^{(b)}-Q_t)}{\sqrt{d-1}}\right)I_{cc'} 
    =\frac{A_{b'c'}}{Nd}\sum_{b\sim c}\OO\left(|G_{b c'}^{(b')}|\sum_{x\sim b, x\neq c}|G_{c  x}^{(b )}| +\frac{ \Phi}{N^{\fb/2}}\right)I_{cc'}, \\
    &\phantom{{}={}}\frac{1}{(Nd)^2}\sum_{b \sim c \atop  b'\sim c'}\left(G_{c  c'}^{(b   b')}-\frac{G_{b   b'}(G_{c  c }^{(b )}-Q_t) (G_{c' c'}^{( b')}-Q_t)}{d-1}\right)I_{cc'}\\
&=\frac{N^{-\fb/2} }{(Nd)^2}\sum_{b \sim c \atop  b'\sim c'} \OO\left(|G_{b   b'}|\left(\sum_{x\sim b , x\neq c }|G_{c  x}^{(b )}|+\sum_{x\sim  b', x\neq c'}|G_{c' x}^{( b')}|\right)+ \Phi\right)I_{cc'}.
    \end{split}\end{align}
Moreover, the following holds 
\begin{align}\begin{split}\label{e:Gccbb_youyige}
&\left|\frac{1}{Nd}\sum_{b \sim c } G_{c  x}^{(b )}I_{c} \right|    \lesssim \frac{N^{-\fb/2} }{Nd}\sum_{b \sim c } (|G_{b  x}| +\Phi)I_{c},
\\
&\left|\frac{A_{b'c'}}{Nd}\sum_{b \sim c } G_{c  c'}^{(b   b')}I_{cc'}   \right|  \lesssim \frac{N^{-\fb/2} A_{b'c'}}{Nd}\sum_{b \sim c } (|G^{( b')}_{b  c'}| +\Phi)I_{cc'}. 
\end{split}\end{align}
\end{proposition}

\begin{proof}[Proof of \Cref{p:upbb}]
For the first statement in \eqref{e:single_index_sum}, we have
    \begin{align*}
         \frac{1}{Nd}\sum_{b \sim c }(G_{c  c }^{(b )}-Q_t)I_c  
        =\frac{1}{Nd}\sum_{b \sim c }(G_{c  c }^{(b )}-Q_t) + \OO\left(\frac{\sum_{c\in\qq{N}}|1-I_c|}{Nd}\right)
        =\OO\left(\frac{1}{N^{1-3\fc/2}}\right),
    \end{align*}
    where for the first equality we used $|G_{cc}^{(b)}|, |Q_t|\lesssim 1$ thanks to \eqref{eq:local_law}; for the second equality, we used the definition \eqref{e:Qsum} of $Q_t$, and  $\sum_{c\in\qq{N}}|1-I_c|\lesssim N^{3\fc/2}$ from \Cref{def:Ic}.

    To prove the last two relations of \eqref{e:single_index_sum}, we recall from \eqref{e:Ghao0} that the Green's function $G$ satisfies 
    \begin{align}\label{e:Ghao}
        \sum_{x\in\qq{N}} G_{xy}=\frac{1}{d/\sqrt{d-1} -z}\lesssim 1,
    \end{align}
    for $z\in \bf D$ from \eqref{e:D}. It follows that
    \begin{align}\label{e:Gbs}
       \frac{1}{Nd}\sum_{b \sim c}G_{b x}I_c
        =\frac{1}{Nd}\sum_{b\sim c}G_{b x}+ \OO\left(\frac{\sum_{c\in\qq{N}}|1-I_c|}{Nd}\right)
          =\OO\left(\frac{1}{N^{1-3\fc/2}}\right).
    \end{align}
where we used \eqref{e:Ghao} and $|G_{bx}|\lesssim 1$ thanks to \eqref{eq:infbound}. The same estimate holds for $G_{cx}$.

To prove \eqref{e:Gccbb}, thanks to the Schur complement formula \eqref{e:Schurixj}, we have
\begin{align}\begin{split}\label{e:Gccbb1}
    G_{c c'}^{(b b')}
    &=G_{c c'}^{( b')}+(G^{(b b')}H(t))_{cb}G^{(b')}_{bc'}\\
    &=G_{c c'}^{(b')}+G_{b c'}^{(b')}
   \left(
   \frac{1}{\sqrt{d-1}}\sum_{x\sim b}G_{c x}^{(b b')}
   +\sqrt t (ZG^{(b b')})_{b c}
   \right)\\
   &=G_{c c'}^{( b')}+G_{b c'}^{(b')}
   \left(
   \frac{1}{\sqrt{d-1}}\sum_{x\sim b}\left(G_{c  x}^{(b )}-\frac{G_{c  b'}^{(b)}G_{b' x}^{(b )}}{G_{b'b'}^{(b )}}\right)
   +\sqrt t (ZG^{(b  b')})_{b  c }
   \right).
\end{split}\end{align}
For $G_{c c'}^{(b')}$ in \eqref{e:Gccbb1}, by the Schur complement formula \eqref{e:Schurixj}, we have
\begin{align}\begin{split}\label{e:Gbs2}
   &\phantom{{}={}}\frac{1}{Nd}\sum_{b \sim c } G_{c c'}^{(b')}I_{cc'}  =\frac{1}{Nd}\sum_{c \sim b }\left(G_{c c'}-\frac{G_{c  b'}G_{b'c'}}{G_{b'b'}}\right)I_c\widehat I_{cc'}I_{c'}\\
   &=\frac{1}{Nd}\sum_{c \sim b }\left(G_{cc'}-\frac{G_{c  b'}G_{b'c'}}{G_{b'b'}}\right)I_{c'}+\OO\left(\frac{\sum_{c\in\qq{N}}|1-I_{c}\widehat I_{cc'}|I_c'}{N}\right)=\OO\left(\frac{I_{c'}}{N^{1-3\fc/2}}\right)=\OO\left(\frac{\Phi I_{c'}}{N^{\fb}}\right), 
\end{split}\end{align}
where we used that $\sum_{c\in\qq{N}}|1-I_{c}\widehat I_{cc'}|\leq \sum_{c\in\qq{N}}(|1-I_{c}|+|1-\widehat I_{cc'}|)\lesssim N^{3\fc/2}$, \eqref{e:Gbs},  and $\Phi\geq 1/N^{1-2\fc}$ (recall from \eqref{e:defPhi}).
Moreover, the same statement holds for $G_{b  c'}^{(b')}$. For the last term in \eqref{e:Gccbb1}, thanks to \eqref{e:Werror},  with overwhelmingly high probability over $Z$ it holds that
\begin{align}\label{e:Wxybound0}
    (ZG^{(b  b')})_{b  c }\lesssim N^\fo \sqrt{\Phi}.
\end{align}
Moreover, for $I_{cc'}=1$, \eqref{eq:local_law} implies that $|G_{b'b'}^{(b)}|\gtrsim 1$ and $|G_{b'x}^{(b)}|\lesssim N^{-\fb} $ for $b\sim x$. Then  \eqref{e:Gest} and \eqref{e:use_Ward} imply
\begin{align}\label{e:warduse1}
    \frac{A_{b'c'}}{Nd}\sum_{b \sim c } \sum_{x\sim b }\left|\frac{G_{b  c'}^{(b')}G_{c  b'}^{(b )}G_{b' x}^{(b )}I_{cc'}  }{G_{b'b'}^{(b )}}\right|\lesssim  \frac{A_{b'c'} N^{-\fb} }{N}\sum_{b \sim c }|G_{b  c'}^{(b')}G_{c  b'}^{(b )}|I_{c'}\lesssim  N^{-\fb/2} \Phi A_{b'c'} I_{c'}.
\end{align}
By plugging \eqref{e:Gbs2}, \eqref{e:Wxybound0} and \eqref{e:warduse1} into \eqref{e:Gccbb1}, and averaging over the edges $(b , c )$ we conclude that
\begin{align}\begin{split}\label{e:Gbbcc0}
    \frac{A_{b'c'}}{Nd}\sum_{b \sim c } 
 G_{c c'}^{(b  b')}I_{cc'} 
 &=\frac{A_{b'c'}}{Nd}\sum_{b \sim c } \frac{G_{b  c'}^{(b')}}{\sqrt{d-1}}\left((G_{c  c }^{(b )}-Q_t) +\sum_{x\sim b , x\neq c }G_{c  x}^{(b )}\right)I_{cc'} +\OO\left(\frac{A_{b'c'} I_{c'}\Phi}{N^{\fb/2}}\right)\\
    &= \frac{A_{b'c'}}{Nd}\sum_{b \sim c } \left(\frac{G_{b  c'}^{(b')}}{\sqrt{d-1}}\left((G_{c  c }^{(b )}-Q_t) +\sum_{x\sim b , x\neq c }G_{c  x}^{(b )}\right)+\OO(N^{-\fb/2} \Phi)\right)I_{cc'}.
\end{split}\end{align}
Moreover, the first claim in \eqref{e:Gccbb} follows from rearranging the above expression.



To prove the second claim in \eqref{e:Gccbb}, we can then average over edges $( b', c')$  in \eqref{e:Gbbcc0}. Similar to \eqref{e:Gccbb1}, we can first replace $G_{b  c'}^{( b')}$ as 
\begin{align}\begin{split}\label{e:Gccbb2}
    G_{b  c'}^{( b')}
    &=G_{b  c'}+G_{ b b' }(H(t)G^{( b')})_{b'c'}\\
   &=G_{b  c'}+G_{b   b'}
   \left(
   \frac{1}{\sqrt{d-1}}\sum_{x\sim  b'}G^{( b')}_{c' x}
   +\sqrt t (ZG^{( b')})_{ b' c'}
   \right),
\end{split}\end{align}
and conclude
\begin{align*}\begin{split}
\frac{1}{(Nd)^2}\sum_{b \sim c \atop  b'\sim c'}G_{c  c'}^{(b   b')}I_{cc'}
    &=\frac{1}{(Nd)^2}\sum_{b \sim c \atop  b'\sim c'}\frac{G_{b   b'}}{d-1}\left(\sum_{x\sim b }G_{c  x}^{(b )} -Q_t\right)\left( \sum_{x\sim  b'}G_{c' x}^{( b')}-Q_t\right)I_{cc'} +\OO(N^{-\fb/2} \Phi)\\
&=\frac{1}{(Nd)^2}\sum_{b \sim c \atop  b'\sim c'}\frac{G_{b   b'}}{d-1}(G_{c  c }^{(b )}-Q_t) (G_{c' c'}^{( b')}-Q_t)I_{cc'} \\
&+\frac{N^{-\fb/2} }{(Nd)^2}\sum_{b \sim c \atop  b'\sim c'} \OO\left(|G_{b   b'}|\left(\sum_{x\sim b , x\neq c }|G_{c  x}^{(b )}|+\sum_{x\sim  b', x\neq c'}|G_{c' x}^{( b')}|\right)+ \Phi\right)I_{cc'},
\end{split}\end{align*}
where for the last equality we bound $|G_{c x}^{( b)}|\lesssim N^{-\fb} $ for $x\sim b,x\neq c$, and $|G_{c' x}^{( b')}|\lesssim N^{-\fb} $ for $x\sim b',x\neq c'$ using \eqref{eq:local_law}.

Finally the estimates in \eqref{e:Gccbb_youyige} follow from the first statement in \eqref{e:Gccbb}, and noticing 
$|G_{cc}^{(b)}-Q_t|\lesssim N^{-\fb}$, and 
$|G_{c x}^{( b)}|\lesssim N^{-\fb} $ for $x\sim b,x\neq c$ using \eqref{eq:local_law}.
\end{proof}

In the following proposition, we provide an estimate related to the Green's function, leveraging the fact that the spectral norm of $G$ is bounded by $1/\eta$.

\begin{proposition}
Adopt the notation of \Cref{s:setting5}, and assume $z\in {\bf D}$ and $\cG\in\Omega$. Take a forest $\cF=\{\{b,c\}, \{b',c'\}\}$ consisting of two unused core edges, and denote $I(\cF,\cG)=I_{cc'}$ (which satisfies the requirements in \Cref{def:Ic}).
For  $w\in \{u,v\}$, $U^{(u,v)}$ a product of $G_{uv}, 1/G_{uu}$, $s\in \{b,c\}$,
   and any vector $(V^{(b,c)})_{b\sim c\in\qq{N}}$,
 the following holds with overwhelmingly high probability over $Z$:
    \begin{align}\label{e:single_index_term2}
       \frac{1}{(Nd)^2} \sum_{b\sim c,b'\sim c'}\sum_{u\sim v} I_{cc'} V^{(b,c)}  G^{\circ}_{s  w} G^{\circ}_{c'u}U^{(u,v)} =\OO\left( \frac{N^{3\fc/2+\fo}\sqrt{\sum_{b\sim c} |V^{(b,c)}|^2} }{N^{3/2}\eta}\right).
    \end{align}
\end{proposition}
\begin{proof}
We assume $w=v$. The other case that $w=u$ can be proven in the same way, so we omit its proof.

We notice that $ G^{\circ}_{sv} G^{\circ}_{c'u}=G_{sv}G_{c'u}-G_{sv}L_{c'u}-L_{sv}G_{c'u}+L_{sv}L_{c'u}$. First we show that up to negligible error, we can replace $ G^{\circ}_{sv} G^{\circ}_{c'u}$ in \eqref{e:single_index_term2} by $G_{sv}G_{c'u}$.  
Since $s\in \{b,c\}$ and $c'$ belong to two different core edges, for $u\sim v$, $L_{sv} L_{c'u}=0$ .
For $G_{sv}L_{c'u}$, using $\|G\|_{\rm spec}\leq 1/\eta$ and  \eqref{e:naive-Ward}
\begin{align}\begin{split}\label{e:trianglesvcu0}
    &\phantom{{}={}}\frac{1}{(Nd)^2}\sum_{b'\sim c'}\left|\sum_{ b\sim c}\sum_{u\sim v}
I_{cc'} V^{(b,c)}  G_{s  w} L_{c'u}U^{(u,v)}\right|\\
&\lesssim \frac{1}{(Nd)^2\eta}\sum_{b' \sim c'} 
\sqrt{\sum_{b\sim c}|V^{(b,c)}|^2I_{c c'}}\sqrt{\sum_{u\sim v} |L_{c'u}|^2}\lesssim \frac{\sqrt{\sum_{b\sim c}|V^{(b,c)}|^2}\sqrt{\fR/N}}{N\eta}
\lesssim \frac{N^{\fc}\sqrt{\sum_{b\sim c} |V^{(b,c)}|^2} }{N^{3/2}\eta}.
\end{split}\end{align}
For $L_{sv}G_{c'u}$, we have exactly the same bound as in \eqref{e:trianglesvcu0}, by first summing over $b'\sim c', u\sim v$. 

Since $(Nd)^{-2}\sum_{b'\sim c'}|1-I_{cc'}|\leq N^{3\fc/2-2}$, we can remove the indicator function $I_{cc'}$ in \eqref{e:single_index_term2}, and the error is bounded by
 \begin{align}\begin{split}\label{e:removeIcc}
     &\phantom{{}={}}\frac{1}{(Nd)^2}\sum_{b\sim c \atop b'\sim c'}
|1-I_{c c'}||V^{(b,c)} |\sum_{u\sim v} | G_{s  w} ||G_{c'u}|\lesssim \frac{N^\fo \Phi}{N^{1-3\fc/2}}\sum_{b\sim c}|V^{(b,c)}|\\
&\lesssim 
\frac{N^{3\fc/2+\fo}\Phi\sqrt{\sum_{b\sim c} |V^{(b,c)}|^2} }{N^{1/2}}
\lesssim \frac{N^{3\fc/2+\fo}\sqrt{\sum_{b\sim c} |V^{(b,c)}|^2} }{N^{3/2}\eta},
\end{split} \end{align}
where in the first statement we used \eqref{e:Gest}; and in the second statement we used Cauchy inequality; in the third statement we used $\Phi\lesssim 1/N\eta$.

The estimates \eqref{e:trianglesvcu0} and \eqref{e:removeIcc} together reduces  \eqref{e:single_index_term2} to the following 
\begin{align*}
    &\phantom{{}={}}\frac{1}{(Nd)^2}\sum_{v\in\qq{N}}\sum_{b\sim c}V^{(b,c)} G_{sv}\left(\sum_{u: u\sim v}\sum_{b' \sim c' }  G_{c'u}U^{(u,v)}\right)\\
    &\lesssim \frac{1}{N^2\eta} \sqrt{\sum_{b\sim c} |V^{(b,c)}|^2}\sqrt{N}
    \lesssim \frac{\sqrt{\sum_{b\sim c} |V^{(b,c)}|^2} }{N^{3/2}\eta},
\end{align*}
where we used that $\|G\|_{\spec}\leq 1/\Im[z]$ and  \eqref{e:Ghao} which gives
$\sum_{c'\in\qq{N}}   G_{c'u}\lesssim 1$.
This gives \eqref{e:single_index_term2}.
\end{proof}

If vertex $o$ has a tree neighborhood, the following lemma gives a simple formula for the average of the Green's function over the radius-$\ell$ ball, which will be used to bound $(\Av G^{\circ})_{o'v}$ from \eqref{e:defAvGL}.

\begin{lemma}\label{lem:boundaryreduction}
Adopt the notation of \Cref{s:setting5} and assume $z\in {\bf D}$, $\cG\in\Omega$ and $I(\cF,\cG)=1$.
For any used core edge $(i',o')\in \cC\setminus \cC^\circ$ with switching data $\{(l'_\al, a'_\al)\}_{\al\in \qq{\mu}}$,  $\sfA_{i'}=\{\al: \dist_\cG(l'_\al, i')=\ell+1\}$, $\sfJ\in\{l',a'\}$ and $w\in \qq{N}$, the following holds with overwhelmingly high probability over $Z$:
\begin{align}
\label{e:aveall}
&\frac{\sum_{\alpha\in \qq{\mu}} G_{ \sfJ_\alpha w}}{(d-1)^{\ell/2}}=\fc_1(\sfJ) G_{ow} +\cE_1^{(w)},\quad \frac{\sum_{\alpha \in\sfA_{i'}} G_{\sfJ_\alpha w }}{(d-1)^{\ell/2}}=\fc_2(\sfJ)G_{iw}+\fc_3(\sfJ) G_{ow}+\cE_2^{(w)},
\end{align}
where the constants $|\fc_1(\cdot)|, |\fc_2(\cdot)|, |\fc_3(\cdot)|\lesssim 1$, and the error terms satisfy
\begin{align*}
    \frac{1}{N}\sum_{w\in\qq{N}}|\cE_1^{(w)}|^2,\frac{1}{N}\sum_{w\in\qq{N}}|\cE_2^{(w)}|^2 \lesssim (d-1)^{2\ell}t\Phi+\frac{\ell}{N}.
\end{align*}
Similar statements hold for the local Green's function $L=L(z,t,\cF,\cG)$:
\begin{align}
\label{e:aveall2}
&\frac{\sum_{\alpha\in \qq{\mu}} L_{ \sfJ_\alpha w}}{(d-1)^{\ell/2}}=\fc_1(\sfJ) L_{ow} +\widetilde\cE_1^{(w)},\quad \frac{\sum_{\alpha \in\sfA_{i'}} L_{\sfJ_\alpha w }}{(d-1)^{\ell/2}}=\fc_2(\sfJ)L_{iw}+\fc_3(\sfJ) L_{ow}+\widetilde \cE_2^{(w)},
\end{align}
and the error terms satisfy $ \sum_{w\in\qq{N}}|\wt \cE_1^{(w)}|^2, \sum_{w\in\qq{N}}|\wt \cE_2^{(w)}|^2 \lesssim \ell$.
\end{lemma}

\begin{proof}
We will only prove the second statement in \eqref{e:aveall}, and the proof for the first statement follows from essentially the same argument. 

By our assumption that $I(\cF,\cG)=1$, vertex $o'$ has radius $\fR$ tree neighborhood. We denote $\cS_0=\{o'\}$, and $\cS_k=\{y\in \qq{N}: \dist_\cG(y,o')=k, \dist_\cG(y,i')=k+1\}$ for $0\leq k\leq \ell$. $\cS_k$ denotes the collection of vertices at a distance $k$ from the root vertex $o'$, excluding the descendants of $i'$.
By using the equation $((H(t)-z)G)_{o'w}=\delta_{o'w}$, we have that
\begin{align}\label{e:Gvo_relation}
\frac1{\sqrt{d-1}}\sum_{y\in \cS_1} G_{yw}=zG_{o'w}-\sqrt t(ZG)_{o'w}+\delta_{o'w}-\frac1{\sqrt{d-1}}G_{i'w}.
\end{align}
Recall $z_t=z+t\md(z,t)$. We can rewrite \eqref{e:Gvo_relation} as
\begin{align*}
\frac1{\sqrt{d-1}}\sum_{y\in \cS_1} G_{yw}=z_tG_{o'w}-\frac{G_{i'w}}{\sqrt{d-1}}+\cE_{o'},\quad \cE_{o'}:=-t\md(z,t)G_{o'w}-\sqrt{t}(ZG)_{o'w}+\delta_{o'w},
\end{align*}

More generally for $k\geq 1$, by summing over $((H(t)-z)G)_{yw}=\delta_{yw}$ for $y\in \cS_k$ we get
\begin{align}\begin{split}\label{e:Gsum_recursion}
    0&=\sum_{y\in \cS_k}((H(t)-z)G)_{yw}-\delta_{yw}=-z\sum_{y\in \cS_k}G_{yw}+\sum_{y\in \cS_k}\sqrt t(ZG)_{yw}-\delta_{yw}+\sum_{y\in \cS_k}(HG)_{yw}\\
    &=-z\sum_{y\in \cS_k}G_{yw}+\sum_{y\in \cS_k}\sqrt t(ZG)_{yw}-\delta_{yw}+\sqrt{d-1}\sum_{y\in \cS_{k-1}}G_{yw}+\frac{1}{\sqrt{d-1}}\sum_{y\in \cS_{k+1}}G_{yw}\\
    &=-z_t\sum_{y\in \cS_k}G_{yw}+\sqrt{d-1}\sum_{y\in \cS_{k-1}}G_{yw}+\frac{1}{\sqrt{d-1}}\sum_{y\in \cS_{k+1}}G_{yw}-\sum_{y\in \cS_k}\cE_y,
\end{split}\end{align}
where $\cE_y:=-t\md(z,t)G_{yw}-\sqrt{t}(ZG)_{yw}+\delta_{yw}$.
We can convert this into an equation of the sum $\sum_{y\in \cS_k}G_{yw}$, for $k\geq 1$,
\begin{align*}
&\phantom{{}={}}\sum_{y\in \cS_{k+1}}\left(\frac{1}{\sqrt{d-1}}\right)^{k+1}G_{yw}=z_t\sum_{y\in \cS_{k}}\left(\frac{1}{\sqrt{d-1}}\right)^{k}G_{yw}-\sum_{y\in \cS_{k-1}}\left(\frac{1}{\sqrt{d-1}}\right)^{k-1}G_{yw}+\frac{\sum_{y\in \cS_k}\cE_y}{(d-1)^{k/2}}.
\end{align*}

 The recursive equation gives
\begin{align}\begin{split}\label{e:recursion}
&\phantom{{}={}}\left(
\begin{array}{c}
(d-1)^{-\ell/2}\sum_{y\in \cS_\ell}G_{yw}\\
(d-1)^{-(\ell+1)/2}\sum_{y\in \cS_{\ell+1}} G_{yw}
\end{array}\right)
\\
&=
\left(
\begin{array}{cc}
0&1\\
-1&z_t
\end{array}
\right) 
\left(
\begin{array}{c}
(d-1)^{-(\ell-1)/2}\sum_{y\in \cS_{\ell-1}}G_{yw}\\
(d-1)^{-\ell/2}\sum_{y\in \cS_{\ell}} G_{yw}
\end{array}\right)
+
\left(
\begin{array}{c}
0\\
\sum_{y\in \cS_\ell}\cE_y/(d-1)^{\ell/2}
\end{array}\right)
\\
&=
\left(\begin{array}{cc}
0&1\\
-1&z_t
\end{array}
\right)^{\ell+1} \left(
\begin{array}{c}
G_{i'w}/\sqrt{d-1}\\
 G_{o'w}
\end{array}
\right) +\sum_{k=1}^\ell\left(\begin{array}{cc}
0&1\\
-1&z_t
\end{array}
\right)^{\ell-k}
\left(
\begin{array}{c}
0\\
\sum_{y\in \cS_k}\cE_y/(d-1)^{k/2}
\end{array}\right).
\end{split}\end{align}

The eigenvalues of the transfer matrix are given by $\msc(z_t)$ and $1/\msc(z_t)$. Therefore $(d-1)^{-\ell/2}\sum_{y\in \cS_\ell}G_{yw}$, 
$(d-1)^{-(\ell+1)/2}\sum_{y\in \cS_{\ell+1}} G_{yw}$ can be written as linear combination of $G_{i'w}, G_{o'w}$, with bounded coefficients, using that for $z\in \mathbf D$, $|\msc(z_t)|^\ell\asymp 1$.
The second claim in \eqref{e:aveall} follows by noticing that 
\begin{align*}
    &\sum_{\alpha \in\sfA_{i'}}(d-1)^{-\ell/2} G_{ l'_\alpha w}
    =(d-1)^{-(\ell-2)/2}\sum_{y\in \cS_{\ell}}G_{yw},\\
    &\sum_{\alpha \in\sfA_{i'}}(d-1)^{-\ell/2} G_{ a'_\alpha w}
    =(d-1)^{-\ell/2}\sum_{y\in \cS_{\ell+1}}G_{yw},
\end{align*}
where we recall that $\cS_k$ denotes the collection of vertices at a distance $k$ from the root vertex $o'$, excluding the descendants of $i'$. Consequently, the multiset $\{l'_\alpha\}_{\alpha \in \sfA_{i'}}$ contains each vertex in $\cS_\ell$ exactly $(d-1)$ times, and $\{a'_\al\}_{\al \in \sfA_{i'}}=\cS_{\ell+1}$.

The error term $\cE_2^{(w)}$ in \eqref{e:aveall} is given by the last term in \eqref{e:recursion}, and recall $\cE_y=-t\md(z,t)G_{yw}-\sqrt{t}(ZG)_{yw}+\delta_{yw}$. We have
\begin{align*}
    |\cE_2^{(w)}|\lesssim \sum_{k=0}^\ell \sum_{y\in \cS_k} \frac{|\cE_y|}{(d-1)^{k/2}}\lesssim \sum_{k=0}^\ell \left(\sum_{y\in \cS_k}\frac{t|G_{yw}|+\sqrt t|(ZG)_{yw}|}{(d-1)^{k/2}}+\frac{\delta_{yw}}{(d-1)^{k/2}}\right).
\end{align*}
Let $\bT'=\{y\in\qq{N}: \dist_\cG(y,o')\leq \ell\}$, then we can square both sides of the above estimate, and average over $w$, giving
\begin{align*}
    \frac{1}{N}\sum_{w\in \qq{N}}|\cE_2^{(w)}|^2
    &\lesssim \frac{1}{N}\sum_{w\in \qq{N}} \left(\sum_{y\in \bT'}\ell\left(t^2|G_{yw}|^2+t|(ZG)_{yw}|^2\right)+\frac{\bm1(w\in \bT')}{(d-1)^{\dist_\cT(w,o)}}\right)\\
    &\lesssim \ell (d-1)^\ell tN^\fo \Phi+\frac{\ell}{N}\lesssim (d-1)^{2\ell}t\Phi+\frac{\ell}{N}.
\end{align*}
where in the second inequality, we used $|\bT'|\lesssim (d-1)^\ell$, \eqref{e:Gest} and \eqref{e:Werror}.

By our assumption that $I(\cF,\cG)=1$, vertex $o'$ has radius $\fR$ tree neighborhood.  The local Green's function $L$ satisfies $(H-z_t)L)_{yw}=\delta_{yw}$  for $y\in \bT'$. The same as in \eqref{e:Gsum_recursion}, we have
\begin{align}\label{e:Lsum_recursion}
    0=-z_t\sum_{y\in \cS_k}L_{yw}+\sqrt{d-1}\sum_{y\in \cS_{k-1}}L_{yw}+\frac{1}{\sqrt{d-1}}\sum_{y\in \cS_{k+1}}L_{yw}-\sum_{y\in \cS_k}\wt\cE_y,\quad \wt \cE_y=\delta_{yw}.
\end{align}
Using \eqref{e:Lsum_recursion} as input, the claims in \eqref{e:aveall2} follows from the same argument as for \eqref{e:aveall}.
\end{proof}

\subsection{Error from Local Resampling}
\label{s:change_est}

Later, we will encounter Green's function entries such as 
$G_{ij}^{(o)}$,
where $i,j$ are two adjacent vertices of a vertex $o$, i.e. $o\sim i,  o \sim j$. In the original graph $\cG$, $i$ and $j$ have distance two, so $G_{ij}$ is not small. However, in the modified graph $\cG^{(o)}$ obtained by removing the vertex $o$, the vertices $i,j$ are far apart. 
In the following proposition, we establish a Ward identity type result for the expectation of $|G_{ij}^{(o)}|^2$, analogous to \eqref{eq:wardex}.
This result shows that the expected value of $|G_{ij}^{(o)}|^2$ is as small as it would be if $i,j$ were chosen randomly. The proof of \Cref{lem:deletedalmostrandom} is based on local resampling around the vertex
$o$. We defer this proof to \Cref{s:removeonevertex}.

To describe this result, we first introduce the control parameter $\Pi(z)$, defined as follows. In most applications, we set $\Pi(z)$ to $\Pi(z)=\bm1(\cG\in \Omega)(|Q_t(z)-Y_t(z)|+(d-1)^{8\ell}\Upsilon(z)\Phi(z))^{p-1}$, which is an upper bound for admissible functions from \Cref{def:pgen}. We denote $\widetilde \Pi(z)$ as the corresponding quantity associated with the switched graph $\widetilde \cG$.
\begin{definition}\label{def:Pi}
 Adopt the notation from \Cref{s:setting5}.
Let $\Pi(z)=\bm1(\cG\in \Omega)\widehat \Pi(z)$ be a control parameter, where $\widehat \Pi(z)$ depends on $Q_t(z), Y_t(z), \Upsilon(z), \Phi(z)$. We assume there exists a large constant $\fC$ such that $N^{-\fC}\leq \widehat\Pi(z)\leq N^{\fC}$. Furthermore,
with overwhelmingly high probability over $Z$, the following holds
\begin{align}\label{e:sametilde}
  \bm1(\cG,\wt \cG\in \Omega) \wt \Pi(z)
   \lesssim  \bm1(\cG,\wt \cG\in \Omega)\Pi(z).
\end{align}
\end{definition}

\begin{proposition}
    \label{lem:deletedalmostrandom}
    Adopt the notation of \Cref{s:setting5}, and recall $\Pi$ from \Cref{def:Pi}. We take $z\in {\bf D}$, and recall the indicator functions $I_c, I_{cc'}, I_{cc'c''}$ from \Cref{def:Ic}. 
    Then the following holds for $i,j\sim o$ and $i\neq j$, 
    \begin{align}\label{e:sameasdisconnect}
        \frac{1}{N}\sum_{o\in \qq{N}}\bE[I_o |G^{(o)}_{ij}|^2 \Pi]\lesssim \bE[N^\fo \Phi \Pi],
    \end{align}
and let $\al\neq \beta\in \qq{\mu}$,
\begin{align}\begin{split}\label{e:GiGi}
&\phantom{{}={}}\frac{1}{N}\sum_{o\in \qq{N}}\bE[I_o\bm1(\cG\in \Omega)(G_{i j}^{(o)})^2(Q_t-Y_t)^{p-1}]\\
&=\frac{\msc^{4\ell}(z_t)}{Z_{\cF^+}}\sum_{\bfi^+}\bE[\bm1(\cG\in \Omega)I(\cF^+,\cG) (G^{(b_\al b_\beta)}_{c_\al c_\beta})^2(Q_t-Y_t)^{p-1}]+\OO\left(N^{-\fb/2}\bE[\Psi_p]\right).
\end{split}\end{align}

    For $i,j,k\sim o$ and $i\neq j, i\neq k$, it holds
\begin{align}\label{e:GiGjGk}
&\frac{1}{N}\sum_{o\in \qq{N}}\bE[I_o\bm1(\cG\in \Omega)(G_{i j}^{(o)}G_{ik}^{(o)})(Q_t-Y_t)^{p-1}]=\OO(N^{-\fb/2} \bE[\Psi_p]).
\end{align}
Moreover, for $i,j\sim o$ with $i\neq j$, $b\sim c$ and $b'\sim c'$, it holds
\begin{align}\begin{split}\label{e:GiGj}
&\frac{1}{N^2}\sum_{o,c\in \qq{N}}\bE[I_{oc}\bm1(\cG\in \Omega)(G_{i c}^{(ob)}G_{jc}^{(ob)})(Q_t-Y_t)^{p-1}]=\OO(N^{-\fb/2} \bE[\Psi_p]),\\
&\frac{1}{N^2}\sum_{o,c,c'\in \qq{N}}\bE[I_{occ'}\bm1(\cG\in \Omega)(G_{i c}^{(ob)}G_{jc'}^{(ob')})(Q_t-Y_t)^{p-1}]=\OO(N^{-\fb/2} \bE[\Psi_p]).
\end{split}\end{align}
\end{proposition}

As a consequence of \Cref{lem:deletedalmostrandom}, the following proposition states that during the local resampling, the errors $\cE$ from \Cref{lem:diaglem} and \Cref{lem:offdiagswitch} (after averaging) are negligible. 
\Cref{lem:task2} follows from \Cref{lem:deletedalmostrandom} and Schur complement formula \eqref{e:Schur1}. We defer this proof to \Cref{s:removeonevertex}.

\begin{proposition}
\label{lem:task2}
Adopt the notation from \Cref{s:setting5}, and recall $\Pi$ from \Cref{def:Pi}. We take $z\in {\bf D}$, unused core edges $(b,c), (b',c')\in \cC^\circ\setminus\{ (i,o)\}$, and indices $\al,\beta\in \qq{\mu}$, the following holds
\begin{align}\begin{split}\label{eq:task2}
    &\frac{1}{Z_{\cF^+}}\sum_{\bfi^+}\bE[I(\cF^+,\cG)\bm1(\cG\in \Omega)|\widetilde G^{(\bT)}_{c_\alpha c_\beta}-G_{c_\alpha c_\beta}^{(b_\alpha b_\beta)}|\Pi]\lesssim(d-1)^{\ell}\bE[N^\fo\Phi \Pi],\\
    &\frac{1}{Z_{\cF^+}}\sum_{\bfi^+}\bE[I(\cF^+,\cG)\bm1(\cG\in \Omega)|\widetilde G^{(\bT\cup \mathbb X)}_{c_\alpha c}-G_{c_\alpha c}^{(b_\alpha\cup \mathbb X)}|\Pi]\lesssim (d-1)^{\ell}\bE[N^\fo\Phi \Pi], \quad \bX\in\{\emptyset, b\},\\
   &\frac{1}{Z_{\cF^+}}\sum_{\bfi^+}\bE[I(\cF^+,\cG)\bm1(\cG\in \Omega)|\widetilde G^{(\bT \cup \bX)}_{c c'}-G^{(\bX)}_{c c'}|\Pi]\lesssim (d-1)^{\ell}\bE[N^\fo\Phi \Pi],\quad \mathbb X\in  \{\emptyset, \{b\}, \{b'\}, \{bb'\}\}.
   \end{split}
    \end{align} 
 
\end{proposition}

\subsection{Estimates of Admissible Functions}
\label{s:adm_bound}

In this section, we gather estimates for the admissible functions as defined in \Cref{def:pgen}. We start with the proof of \Cref{l:adm_term_bound}.

\begin{proof}[Proof of \Cref{l:adm_term_bound}]
For $\cG\in \Omega$ and $I(\cF,\cG)=1$, thanks to \eqref{eq:infbound} and \eqref{eq:local_law}, each factor in \eqref{e:defcE1} and \eqref{e:defcE0} is bounded by $N^{-\fb} $ with overwhelmingly high probability over $Z$. 

For the second statement in \eqref{e:naive-Ward}, by \eqref{e:local_Green} and \eqref{e:Pijbound}, we have
\begin{align}\label{e:qiuhe1}
    \sum_{w\in \qq{N}}|L^2_{sw}|\lesssim \sum_{\dist_\cG(s,w)\leq 2\fR}\frac{1}{(d-1)^{\dist_\cG(s,w)}}\lesssim \fR. 
\end{align}
The first statement in \eqref{e:naive-Ward} follows from \eqref{e:Gest} and \eqref{e:qiuhe1}
\begin{align}\label{e:qiuhe2}
    \frac{1}{N}\sum_{w\in \qq{N}} | G^{\circ}_{sw}|^2\lesssim \frac{1}{N}\sum_{w\in \qq{N}} |G_{sw}|^2+|L_{sw}|^2\lesssim N^\fo \Phi+ \frac{\fR}{N}\lesssim N^\fo \Phi.
\end{align}

For the two statements in \eqref{e:av_naive-Ward}, we recall the definition of $\Av$ from \Cref{def:av_Green}. Thanks to \eqref{e:aveall} and \eqref{e:aveall2}, we have
\begin{align}\label{e:avG_avL}
    |(\Av G)_{o'w}|\lesssim |G_{i'w}|+|G_{o'w}|+|\cE^{(w)}|,
    \quad 
    |(\Av L)_{o'w}|\lesssim |L_{i'w}|+|L_{o'w}|+|\wt \cE^{(w)}|,
\end{align}
where the errors satisfy
\begin{align*}
\frac{1}{N}\sum_{w\in \qq{N}} |\cE^{(w)}|^2\lesssim (d-1)^{2\ell}t\Phi+\frac{\ell}{N}, \quad  \frac{1}{N}\sum_{w\in \qq{N}} |\wt \cE^{(w)}|^2\lesssim \frac{\ell}{N}.
\end{align*}
Using \eqref{e:avG_avL} as input, the claim \eqref{e:av_naive-Ward} follows by the same argument as in  \eqref{e:qiuhe2} and \eqref{e:qiuhe1}.

For the following discussion, we recall $\fc\gg \fb\gg \ell/\log_{d-1} N\gg \fo$ from \eqref{e:parameters}, and $\Phi\geq N^{2\fc}/N\geq N^{4\fb}/N$ from \eqref{e:defPhi}.
By the definition of admissible functions (as in \Cref{def:pgen}), for each $\cW_j$ if $r_{j0}=1$,  $|\cW_j|=|Q_t-Y_t|$, and if $r_{j0}=3$, we have  $|\cW_j|\lesssim (d-1)^\ell\Upsilon/N\lesssim \Upsilon N^{-\fb}\Phi$. In the following we discuss the case that $r_{j0}=2$. In this case, we have
\begin{align*}
    |\cW_j|&\lesssim \frac{\Upsilon}{Nd}\sum_{u_j\sim v_j\in \qq{N}}\prod_{k=1}^{r_{j1}}|R_{k}|\prod_{k'=1}^{r_{j2}}|R'_{k'}|,
\end{align*}
where $R_k$ are factors of the form \eqref{e:defcE0}, and $R_k'$ are factors of the form \eqref{e:defLerror}.

If $r_{j2}=0$, then $r_{j1}\geq 2$. Using that $|R_{k}|\lesssim N^{-\fb} $, $|R'_{k'}|\lesssim 1$, \eqref{e:naive-Ward} and \eqref{e:av_naive-Ward}, we have
\begin{align}\label{e:k1}
    |\cW_j|&\lesssim \frac{\Upsilon}{Nd}N^{-(r_{j1}-2)\fb} \sqrt{\sum_{u_j\sim v_j\in \qq{N}}|R_{1}|^2\sum_{u_j\sim v_j\in \qq{N}}|R_{2}|^2}\lesssim N^\fo N^{-(r_{j1}-2)\fb}  \Upsilon \Phi.
\end{align}
If $r_{j2}=1$, then $r_{j1}\geq 1$. Using that $|R_{k}|\lesssim N^{-\fb} , |R'_{k'}|\lesssim 1$, \eqref{e:naive-Ward} and \eqref{e:av_naive-Ward}, we have
\begin{align}\label{e:k2}
    |\cW_j|&\lesssim \frac{\Upsilon}{Nd}\sum_{u_j\sim v_j\in \qq{N}}N^{-(r_{j1}-1)\fb}  \sqrt{\sum_{u_j\sim v_j\in \qq{N}}|R_{1}|^2\sum_{u_j\sim v_j\in \qq{N}}|R'_{1}|^2}\lesssim N^\fo N^{-(r_{j1}-1)\fb} \Upsilon \sqrt{\frac{\Phi}{N}}.
\end{align}
Finally if $r_{j2}\geq 2$, then using that $|R_{k}|\lesssim N^{-\fb} , |R'_{k'}|\lesssim 1$, \eqref{e:naive-Ward} and \eqref{e:av_naive-Ward}, we have
\begin{align}\label{e:k3}
    |\cW_j|&\lesssim \frac{\Upsilon}{Nd}\sum_{u_j\sim v_j\in \qq{N}}N^{-r_{j1}\fb}  \sqrt{\sum_{u_j\sim v_j\in \qq{N}}|R'_{1}|^2\sum_{u_j\sim v_j\in \qq{N}}|R'_{2}|^2}\lesssim \frac{N^\fo \Upsilon}{N}\lesssim  N^\fo \Upsilon \sqrt{\frac{\Phi}{N}}.
\end{align}
The estimates \eqref{e:k1}, \eqref{e:k2} and \eqref{e:k3} together lead to \eqref{e:Sjbound}.

If $(r_{j0}, r_{j1}, r_{j2})=(2,2,0)$, then \eqref{e:k1} gives $(d-1)^{3r_{j1}\ell}|\cW_j|\lesssim (d-1)^{6\ell}N^\fo \Phi\lesssim (d-1)^{8\ell}\Upsilon \Phi$. If $r_{j0}=2, r_{j2}=0$, and $r_{j1}\geq 4$, then \eqref{e:k1} implies $(d-1)^{3r_{j1}\ell}|\cW_j|\lesssim (d-1)^{3r_{j1}\ell}N^\fo N^{-(r_{j1}-2)\fb}\Phi\lesssim N^{-\fb}\Upsilon \Phi$. 
If $r_{j0}=2, r_{j2}\geq 1$, then \eqref{e:k2} and \eqref{e:k3} implies $(d-1)^{3r_{j1}\ell}|\cW_j|\lesssim (d-1)^{3r_{j1}\ell}N^\fo N^{-\max\{r_{j1}-1,0\}\fb}\sqrt{\Phi/N}\lesssim (d-1)^{3\ell}N^{\fo}\sqrt{\Phi/N}\lesssim N^{-\fb}\Upsilon \Phi$.  This finishes the proof of \eqref{e:1_Sbound} and \Cref{l:adm_term_bound}.
\end{proof}

We recall the three cases of $\bmr$ from \Cref{r:r_fenlei}. Next, we introduce the following error parameter, which depends on whether $\bmr$ satisfies  \eqref{e:1_expand} or \eqref{e:feasible}.
\begin{align}\begin{split}\label{e:defUG}
 \Pi(z,\bmr):&=\bm1(\cG\in \Omega)\left(|Q_t(z)-Y_t(z)|+(d-1)^{8\ell} \Upsilon(z)\Phi(z)\right)^{p-1}, \;\text{if }\eqref{e:1_expand},\eqref{e:feasible}, \text{ do not hold;}\\
 \Pi(z,\bmr):&=\bm1(\cG\in \Omega)(d-1)^{8\ell} \Upsilon(z)\Phi(z)\left(|Q_t(z)-Y_t(z)|+(d-1)^{8\ell} \Upsilon(z)\Phi(z)\right)^{p-2}, \;  \text{if }\eqref{e:1_expand} \text{ holds};\\ \Pi(z,\bmr):&=\bm1(\cG\in \Omega)\bigg(N^{-\fb} \Upsilon(z)\Phi(z)\left(|Q_t(z)-Y_t(z)|+(d-1)^{8\ell} \Upsilon(z)\Phi(z)\right)^{p-2}\\
    &+((d-1)^{8\ell} \Upsilon(z)\Phi(z))^2\left(|Q_t(z)-Y_t(z)|+(d-1)^{8\ell} \Upsilon(z)\Phi(z)\right)^{p-3}\bigg), 
    \;\text{if }\eqref{e:feasible} \text{ holds}.
\end{split}\end{align}
The above error parameters will be used to upper bound admissible functions in \Cref{l:Sbound}. 

We remark that the sizes of the three $\Pi(z,\bmr)$ terms in \eqref{e:defUG}  decrease, and in particular in all three cases $\Pi(z,\bmr)\lesssim \bm1(\cG\in \Omega)\left(|Q_t(z)-Y_t(z)|+(d-1)^{8\ell} \Upsilon(z)\Phi(z)\right)^{p-1}$,
and 
\begin{align}\label{e:bPi1}
    \Phi\Pi(z,\bmr)\lesssim \Psi_p.
\end{align}
Moreover, if $\bmr$ satisfies \eqref{e:feasible}, then we have
\begin{align}\begin{split}\label{e:bPi2}
&\left(|Q_t-Y_t|+\frac{N^{-7\fb/8}}{N\eta}+\Phi\right)\Pi(z,\bmr)\lesssim N^{-3\fb/4}\Psi_p, \text { if $\bmr$ satisfies \eqref{e:1_expand} },\\
&\left(|Q_t-Y_t|+\frac{1}{N\eta}+\Phi\right)\Pi(z,\bmr)\lesssim N^{-3\fb/4}\Psi_p, \text { if $\bmr$ satisfies \eqref{e:feasible}}.
\end{split}\end{align}
In the following propositions, we gather estimates on admissible functions $R_\bfi$ before and after local resampling, in terms of the error parameter $\Pi(z,\bmr)$ from \eqref{e:defUG}.

\begin{lemma}\label{l:Sbound}
  Adopt the notation of \Cref{s:setting5}. We take $z\in {\bf D}$ and an array $\bmr=[r_{jk}]_{1\leq j\leq p-1, 0\leq k\leq 2}$ satisfying \eqref{e:defr_copy}. 
  Assume $\cG,\tcG\in \Omega$ and $I(\cF^+,\cG)=1$. Then
   for any admissible function $R_\bfi=R\prod_{j=1}^{p-1} \cW_j\in \Adm(r,\bmr,\cF,\cG)$ as in \Cref{def:pgen}, 
   the following holds with overwhelmingly high probability over $Z$:
      \begin{align}\label{e:Sbound}
    (d-1)^{3\ell\sum_{j=1}^{p-1} r_{j1}}\left(\prod_{j=1}^{p-1}  |\cW_j|+\prod_{j=1}^{p-1} |\wt \cW_j|\right)\lesssim \Pi(z,\bmr).
\end{align}
\end{lemma}

\begin{proposition}\label{p:small_Ri}
     Adopt the notation of \Cref{s:setting5}. We take $z\in {\bf D}$, an array $\bmr=[r_{jk}]_{1\leq j\leq p-1, 0\leq k\leq 2}$ satisfying \eqref{e:defr_copy}, and $R_\bfi\in \Adm(r,\bmr,\cF,\cG)$ as in \Cref{def:pgen}. Then the following holds with overwhelmingly high probability over $Z$:
    \begin{enumerate}
        \item \label{i1:small} If $R_\bfi$ contains two terms of the form $\{G_{cc'}^{(bb')},G_{bc'}^{(b')}, G_{bb'}, G_{cb'}\}$ with $(c,b)\neq (c',b')\in \cC^\circ$ or $ G^{\circ}_{ss'}$ with $s,s'\in \cK$ in different connected components of $\cF$, then
    \begin{align}\label{e:refined_bound}
        \frac{(d-1)^{6\ell r+  3\ell\sum_{j=1}^{p-1} r_{j1}}}{Z_\cF}\sum_{\bfi}\bm1(\cG\in \Omega)I(\cF, \cG)R_\bfi =\OO\left(\frac{(d-1)^{6\ell r}N^\fo}{N^{(r-2)\fb} } \Phi\Pi(z,\bmr)\right).
    \end{align}
     If we further assume $r\geq 3$, or $\bmr$ satisfies \eqref{e:feasible}, then by \eqref{e:bPi1} and \eqref{e:bPi2}, the expectation of \eqref{e:refined_bound} is bounded by  $\OO(N^{-\fb/2}\bE[\Psi_p])$;
        \item \label{i2:small} If $R_\bfi$ contains exact one term of the following form: $\{G_{cc'}^{(bb')},G_{bc'}^{(b')}, G_{bb'}, G_{cb'}\}$ with $(b,c)\neq (b',c')\in \cC^\circ$ or $ G^{\circ}_{ss'}$ with $s,s'\in \cK$ in different connected components of $\cF$, and $|\{j\in\qq{p-1}: r_{j0}=2,3\}|\geq 1$ then
    \begin{align}\label{e:refined_bound2}
        \frac{(d-1)^{6\ell r+ 3\ell\sum_{j=1}^{p-1} r_{j1}}}{Z_\cF}\sum_{\bfi}\bm1(\cG\in \Omega)I(\cF, \cG)R_\bfi =\OO\left(\left( \frac{1}{N\eta}+\Phi\right)\frac{(d-1)^{6\ell r} \Pi(z,\bmr)}{ N^{(r-1)\fb} }\right).
    \end{align}
     If we further assume $r\geq 2$, or $\bmr$ satisfies \eqref{e:feasible}, then by \eqref{e:bPi1} and \eqref{e:bPi2}, the expectation of \eqref{e:refined_bound2} is bounded by   $\OO(N^{-\fb/2}\bE[\Psi_p])$.

     We also record the following estimate, which is a special case of \eqref{e:refined_bound2}. For $(b,c)\neq (b',c')\in \cC^\circ$, $x\in \{b,c\}$, $x'\in \{b',c'\}$, and $w,w'\in \{u,v\}$, then
\begin{align}\begin{split}\label{e:refined_bound2_special}
          &\phantom{{}={}}\frac{1}{Z_\cF}\sum_{\bfi}\bm1(\cG\in \Omega)I(\cF, \cG)G_{cc'}^{(bb')}\times \frac{\{1-\del_1 Y_\ell, -t\del_2 Y_\ell\}}{Nd}\sum_{u\sim v}G^{\circ}_{xw}G^{\circ}_{x'w'} (Q_t-Y_t)^{p-2}\\
          &\lesssim \bm1(\cG\in \Omega)\frac{N^\fo \Upsilon \Phi}{N\eta}|Q_t-Y_t|^{p-2}
          \lesssim N^\fo\Psi_p.
     \end{split}\end{align}
       \end{enumerate}
\end{proposition}

Before proving \Cref{l:Sbound} and \Cref{p:small_Ri}, we present the following lemma, which will be used to estimate $G_{cc'}^{(bb')}$.
\begin{lemma}Adopt the notation of \Cref{p:small_change}.  Assume $\cG\in \Omega$ and $I(\cF,\cG)=1$. Then for any  $(c,b)\neq (c',b')\in \cC^\circ$, the following holds with overwhelmingly high probability over $Z$:
\begin{align}\label{e:rough_bb1}
    |G_{bb}|,|G_{b'b'}|,|G_{cc}|,|G_{c'c'}|, |G_{bc}|, |G_{b'c'}|\asymp 1, \quad 
    |G_{bb'}|,|G_{cc'}|, |G_{bc'}|,|G_{cb'}|\lesssim N^{-\fb} ,
\end{align}
and 
\begin{align}\begin{split}\label{e:Gbbcc}
 G_{cc'}^{(bb')}&=   G_{cc'}-\frac{G_{cb}G_{bc'}}{G_{bb}}-\frac{G_{cb'}G_{b'c'}}{G_{b'b'}} +\frac{G_{cb}G_{bb'}G_{b'c'}}{G_{bb}G_{b'b'}}  \\ 
&+\OO\left(|G_{bb'}|^2+|G_{cc'}|^2+|G_{bc'}|^2+|G_{cb'}|^2+\frac{1}{N^2}\right).
\end{split}\end{align}
\end{lemma}
\begin{proof}
 For $\cG\in \Omega, I(\cF,\cG)=1$, the claim \eqref{e:rough_bb1} follows from \Cref{thm:prevthm}.
To prove \eqref{e:Gbbcc}, we start with the Schur complement formula \eqref{e:Schur1}
\begin{align}\label{e:scf}
 G_{cc'}^{(bb')}=G_{cc'}-(G (G|_{\{bb'\}})^{-1} G)_{cc'},
\end{align}
where
$(G|_{\{bb'\}})^{-1}$ is given by
\begin{align}\begin{split}\label{e:rough_bb2}
    (G|_{\{bb'\}})^{-1}
    &=\frac{1}{G_{bb}G_{b'b'} -G_{bb'}^2}
    \left[
    \begin{array}{cc}
    G_{b'b'} & -G_{bb'}\\
    -G_{bb'} & G_{bb}
    \end{array}
    \right]\\
    &=\left(\sum_{j=0}^\fp \frac{G_{bb'}^{2j}}{(G_{bb}G_{b'b'})^{j+1} } +\OO\left(\frac{1}{N^2}\right)\right)
    \left[
    \begin{array}{cc}
    G_{b'b'} & -G_{bb'}\\
    -G_{bb'} & G_{bb}
    \end{array}
    \right].
\end{split}\end{align}

 By substituting \eqref{e:rough_bb2} into \eqref{e:scf}, we obtain an expansion of $G_{cc'}^{(bb')}$. Terms in this expansion containing at least two factors of the form $\{G_{bb'},G_{cc'}, G_{bc'},G_{cb'}\}$ can be bounded by $\OO(|G_{bb'}|^2+|G_{cc'}|^2+|G_{bc'}|^2+|G_{cb'}|^2)$. The remaining terms in the expansion are given by
\begin{align*}
    G_{cc'}-\frac{G_{cb}G_{b'b'}G_{bc'}}{G_{bb}G_{b'b'}} +\frac{G_{cb}G_{bb'}G_{b'c'}}{G_{bb}G_{b'b'}} -\frac{G_{cb'}G_{bb}G_{b'c'}}{G_{bb}G_{b'b'}} .
\end{align*} 
This gives \eqref{e:Gbbcc}.
\end{proof}

\begin{proof}[Proof of \Cref{l:Sbound}]

The estimate \eqref{e:1_Sbound} implies the following bound for $(d-1)^{3\ell\sum_{j=1}^{p-1} r_{j1}}\prod_{j=1}^{p-1} \cW_j$,
\begin{align}\begin{split}\label{e:Eexp}
    &\phantom{{}={}}(d-1)^{3\ell\sum_{j=1}^{p-1} r_{j1}}\prod_{j=1}^{p-1} |\cW_j|\\
    &\lesssim  \prod_{j:r_{j0}=1}|Q_t-Y_t|\prod_{j:r_{j0}=2, r_{j2}=0}(d-1)^{8\ell} \Upsilon \Phi
    \times\prod_{j:r_{j0}=2, r_{j2}\geq1}N^{-\fb}\Upsilon \Phi
    \prod_{j:r_{j0}=3}N^{-\fb}\Upsilon \Phi\\
    &\lesssim \Pi(z,\bmr),
\end{split}\end{align}
where the last statement follows from the definition of $\Pi(z,\bmr)$ as in \eqref{e:defUG}.
Thanks to \eqref{e:Eexp}, we have
\begin{align*}
     (d-1)^{3\ell\sum_{j=1}^{p-1} r_{j1}}\prod_{j=1}^{p-1} |\wt\cW_j|\lesssim \widetilde \Pi(z,\bmr)\lesssim \Pi(z,\bmr),
\end{align*}
where the last statement follows from \eqref{e:tilde_bound}, and noticing $\Pi(z,\bmr)$ is a function of $\Phi, \Upsilon$ and $|Q_t-Y_t|+(d-1)^{8\ell}\Upsilon \Phi$. 
\end{proof}

\begin{proof}[Proof of \Cref{p:small_Ri}]
We prove \eqref{e:refined_bound} when $R_\bfi$ contains the factor $(G_{cc'}^{(bb')})^2$. The other cases can be proven in the same way, so we omit their proofs.
Define the forest $\widehat \cF$ from $\cF$ by removing $\{(b,c), (b',c')\}$,
\begin{align}\label{e:hatcF}
    \widehat \cF=(\widehat \bfi, \widehat E)=\cF\setminus \left( \{(b,c), (b',c')\}\right),\quad \widehat \bfi=\bfi\setminus\{b, c, b', c'\},\quad 
  I(\cF,\cG)=A_{bc }A_{b'c'} I_{cc'}I(\widehat \cF, \cG),  
\end{align}
 where $I_{cc'}=\bm1(c \not\in \bX)\bm1(c' \not\in \bX)\bm1(\dist_\cG(c,c')\geq 3\fR)$ and 
$\bX$ is the collection of vertices $v$ such that either there exists some $u\in \cB_\ell(v,\cG)$ such that $\cB_\fR(u,\cG)$ is not a tree; or $\dist(v,c'')<3\fR$ for some $(b'',c'')\in \cC^+\setminus\{(b, c), (b', c')\}$. 
For $\cG\in \Omega \subset \oOmega$, $|\bX|=\OO((d-1)^{\fR+\ell}N^\fc)=\OO(N^{3\fc/2})$, and
$I_{cc'}$ satisfies 
\begin{align}\label{e:indicator}
     \frac{\bm1(\cG\in \Omega)}{(Nd)^2}\sum_{b\sim c \atop b'\sim c'}|1-I_{cc'}|
     \lesssim \frac{N(|\bX|+(d-1)^{3\fR})}{N^{2}}\lesssim \frac{1}{N^{1-2\fc}}.
\end{align}

The statements \eqref{e:rough_bb1} and \eqref{e:Gbbcc} together imply
\begin{align}\label{e:Gcc'bb'}
|G_{cc'}^{(bb')}|\lesssim |G_{bb'}|+|G_{bc'}|+|G_{b'c}|+|G_{cc'}|,
\end{align}
and \eqref{e:Gest} leads to 
\begin{align}\begin{split}\label{e:ccbb}
   &\phantom{{}={}}\frac{\bm1(\cG\in \Omega)}{(Nd)^2} \sum_{c\sim b\atop c'\sim b'}I(\cF,\cG) |G_{cc'}^{(bb')}|^2\\
   &\lesssim 
  \frac{\bm1(\cG\in \Omega)}{(Nd)^2}I(\widehat \cF,\cG)\sum_{c\sim b\atop c'\sim b'}I_{cc'}(|G_{bb'}|^2+|G_{bc'}|^2+|G_{b'c}|^2+|G_{cc'}|^2)\lesssim I(\widehat\cF,\cG)\bm1(\cG\in \Omega)N^\fo\Phi.
\end{split}\end{align}
The claim \eqref{e:refined_bound} then follows from \eqref{e:ccbb}. In fact, by \eqref{e:Bsmall}, we have
\begin{align}\begin{split}\label{e:Eexp2}
&\phantom{{}={}}\left|\frac{(d-1)^{  3\ell\sum_{j=1}^{p-1} r_{j1}}}{Z_\cF}\sum_{\bfi}I(\cF, \cG)\bm1(\cG\in \Omega)R_\bfi\right|\\
&\lesssim\frac{(d-1)^{ 3\ell\sum_{j=1}^{p-1} r_{j1}}}{Z_\cF}\sum_{\bfi} I(\cF, \cG)\bm1(\cG\in \Omega)N^{-(r-2)\fb} |G_{cc'}^{(bb')}|^2\prod_{j=1}^{p-1}|\cW_j|.
\end{split}\end{align}
We can first bound $(d-1)^{  3\ell\sum_{j=1}^{p-1} r_{j1}}\prod_{j=1}^{p-1}|\cW_j|$ by \eqref{e:Sbound}, then the average over $b,b', c,c'$ using \eqref{e:ccbb}. This gives \eqref{e:refined_bound}.


Next, we prove \eqref{e:refined_bound2}  when $R_\bfi$ contains a term $G^{(b b')}_{c c'}$ with $(b,c)\neq (b',c')\in \cC^\circ$
\begin{align}\label{e:Riexp}
    R_\bfi=G^{(b b')}_{c c'}\prod_{j=1}^{r-1}B_j\prod_{j=1}^{p-1} \cW_j\in \Adm(r,\bmr, \cF,\cG),
\end{align}
where for each $1\leq j\leq r-1$, $B_j$ is a factor of the form \eqref{e:defcE1}, and each $B_j$ does not contain indices from both $\{b,c\}$ and $\{b',c'\}$.  The other cases can be proven similarly, so we omit the details. 

In this case, $(b,c)\neq (b',c')\in \cC^\circ$ are two distinct connected components of $\cF$.  We recall the expression of $G_{cc'}^{(bb')}$ from \eqref{e:Gbbcc}. For the error terms $\OO(|G_{bb'}|^2+|G_{b'b'}|^2+|G_{bc'}|^2+|G_{cb'}|^2+N^{-2})$, we apply the same argument as for \eqref{e:refined_bound}. For the four main terms in \eqref{e:Gbbcc},  each of them contains a unique factor that involves indices from both $\{b,c\}$ and $\{b',c'\}$, and they can be analyzed similarly. Therefore, we focus on the term $G_{cc'}$. More explicitly, let $\widehat R_\bfi=G_{c c'}\prod_{j=1}^{r-1}B_j\prod_{j=1}^{p-1} \cW_j$, which is obtained from $R_\bfi$ by replacing $G_{cc'}^{(bb')}$ with $G_{cc'}$, we will show
\begin{align}\label{e:refined_bound3}
        \frac{(d-1)^{3\ell\sum_{j=1}^{p-1} r_{j1}}}{Z_\cF}\sum_{\bfi}I(\cF, \cG)\bm1(\cG\in \Omega)\widehat R_\bfi =\OO\left(\left( \frac{1}{N\eta}+\Phi\right)\frac{\Pi(z,\bmr)}{N^{(r-1)\fb} } \right),
    \end{align}
and the claim \eqref{e:refined_bound2} follows.
    
We first show that we can replace the indicator function $I(\cF,\cG)$ by $I(\widehat\cF, \cG)A_{bc}A_{b'c'}$ (recall from \eqref{e:hatcF}) up to negligible error: 
\begin{align}\begin{split}\label{e:relaxbc}
    &\phantom{{}={}}\frac{(d-1)^{  3\ell\sum_{j=1}^{p-1} r_{j1}}}{Z_\cF}\sum_{\bfi}I(\cF, \cG)\bm1(\cG\in \Omega)\widehat R_\bfi \\
    &= \frac{(d-1)^{  3\ell\sum_{j=1}^{p-1} r_{j1}}}{Z_\cF}\sum_{\bfi}\bm1(\cG\in \Omega)I(\widehat\cF, \cG)A_{bc}A_{b'c'} \widehat R_\bfi +\OO(N^{-2\fc+1}N^{-r\fb} \Pi(z,\bmr)).
\end{split}\end{align}
Indeed, thanks to \eqref{e:hatcF}, we have  $|I(\cF,\cG)-I(\widehat\cF, \cG)A_{bc}A_{b'c'}|=I(\widehat\cF, \cG)A_{bc}A_{b'c'}|1-I_{cc'}|$.
Then 
\begin{align*}\begin{split}
    &\phantom{{}={}}\frac{(d-1)^{  3\ell\sum_{j=1}^{p-1} r_{j1}} }{Z_\cF}\sum_{\bfi}|I(\cF,\cG)-I(\widehat\cF, \cG)A_{bc}A_{b'c'}|\bm1(\cG\in \Omega)|\widehat R_\bfi|\\
    &=
   \frac{(d-1)^{  3\ell\sum_{j=1}^{p-1} r_{j1}}(Nd)^2}{Z_\cF}\sum_{\widehat\bfi}\frac{I(\widehat \cF,\cG)}{(Nd)^2}\sum_{b\sim c \atop b'\sim c'}|1-I_{cc'}|\bm1(\cG\in \Omega)|\widehat R_\bfi|\\
   &\lesssim
   \frac{(Nd)^2}{Z_\cF}\sum_{\widehat\bfi}\frac{I(\widehat \cF,\cG)}{(Nd)^2}\sum_{b\sim c \atop b'\sim c'}|1-I_{cc'}| N^{-r\fb}\Pi(z,\bmr)\lesssim N^{-2\fc+1}N^{-r\fb} \Pi(z,\bmr),
\end{split}\end{align*}
where in the second statement we used \eqref{e:Sbound} and \eqref{e:Bsmall}; in the third statement, we used
\eqref{e:indicator}.

With the above reductions, we will prove \eqref{e:refined_bound3} by categorizing factors of $\prod_{j=1}^{r-1}B_j\prod_{j=1}^{p-1} \cW_j$ based on their relationship with $\{b, c\}$ and $\{b', c'\}$. 
We will now decompose \eqref{e:Riexp} as follows
\begin{align}\label{e:bfi}
    R_{\bfi}=\frac{1}{(Nd)^{|\bV_0|}} \sum_{u_j\sim v_j: j\in \bV_0} \left(U_{bc}V_{b'c'}G_{c c'}R'_{\widehat\bfi}\right),
\end{align}
where $\{u_j, v_j\}_{1\leq j\leq p-1}$ are the special edges from the definition of $\cW_j$; $\bV_0$ is the collection of indices $j\in \{1,2,\cdots, p-1\}$ such that $\cW_j$ involves both $\{b,c\}$, and $\{b',c'\}$,  and $U_{bc}=\widehat U_{bc} \widehat U'_{bc}$, $V_{b'c'}=\widehat V_{b'c'} \widehat V'_{b'c'}$, $R'_{\widehat\bfi}$ are given as below 
\begin{enumerate}
    \item For any factor $B_j$ with $1\leq j\leq r-1$, by our assumption, it does not contain indices from both $\{b,c\}$ and $\{b',c'\}$.
    \begin{enumerate}
        \item If $B_j$ does not involve $\{b,c,b',c'\}$, we place it in $R'_{\widehat\bfi}$. 
        \item If $B_j$ involves $b,c$, we place it in $\widehat U_{b c}$.
        \item If $B_j$ involves $b',c'$, we place it in $\widehat V_{b' c'}$.
    \end{enumerate}
    \item We divide $\{j: 1\leq j\leq p-1\}$ into two subsets $\bV_0\cup\bV_1$.
    For any special edge $\{u_j, v_j\}$ with $r_{j0}\in\{1,3\}$, the factor $\cW_j$ is placed in $R'_{\wh\bfi}$ and  $j$ is placed in $\bV_1$. 
    Otherwise $r_{j0}=2$. We recall from \Cref{def:pgen} that
     \begin{align*}
         \cW_j=\frac{1}{Nd}\sum_{u_j\sim v_j\in \qq{N}} \{1-\del_1 Y_\ell, -t\del_2 Y_\ell\}\times F_{u_j v_j},
     \end{align*}
     and $F_{u_j v_j}$ contains at least two factors of the form $L_{sw},  G^{\circ}_{s w}, (\Av G^{\circ})_{o'w}, (\Av L)_{o'w}$ with $s\in \cK$, $(i',o')\in \cC\setminus \cC^\circ$ and $w\in \{u_j, v_j\}$, and possibly some factors of the form $G_{u_j v_j}, 1/G_{v_j v_j}$.
     \begin{enumerate}
         \item 
         If these factors $L_{sw},  G^{\circ}_{s w}$ do not involve $\{b,c,b', c'\}$, then $\cW_j$ is placed in $R'_{\wh\bfi}$, and $j$ is placed in $\bV_1$.
    
        \item Otherwise, if all factors $L_{sw},  G^{\circ}_{s w}$ satisfy $s\not\in \{b', c'\}$,  $j$ goes to $\bV_1$, and $\cW_j$ is placed in $\widehat U_{b c}$. If all factors $L_{sw},  G^{\circ}_{s w}$ satisfy  $s\not\in \{b, c\}$,
         $j$ goes to $\bV_1$, and $\cW_j$ is placed in $\widehat V_{b' c'}$
    \item The remaining indices $j$ are collected in $\bV_0$, factors $L_{sw},  G^{\circ}_{s w}$  with $s\in \{b, c\}$, are placed in $\widehat U'_{b c}$;  and remaining factors, namely $1-\del_1 Y_\ell, -t\del_2 Y_\ell, (\Av G^{\circ})_{o'w}, (\Av L)_{o'w}, G_{u_j v_j}, 1/G_{v_j v_j}$ and $L_{sw},  G^{\circ}_{s w}$ with $s\in \{b',c'\}$, are placed in $\widehat V'_{b' c'}$.
\end{enumerate}
   \end{enumerate}

In this way, $U_{bc}$ and $ V_{b'c'}$ collect factors involving $b,c$ and $b', c'$ respectively. $R'_{\widehat\bfi}$ only depends on the indices $\widehat \bfi$, and does not depend on indices $b, c, b', c'$. We have the following relation
\begin{align}\begin{split}\label{e:UVbound}
    &|U_{bc}||V_{b'c'}||R'_{\wh \bfi}|
     \lesssim N^{-(r-1)\fb} |\widehat U'_{bc}|  |\widehat V'_{b'c'}| \prod_{j\in \bV_1}|\cW_j|.
\end{split}\end{align}

With the above notation, we can rewrite the right-hand side of \eqref{e:relaxbc} as 
\begin{align}\begin{split}\label{e:sumbound1}
&\phantom{{}={}}\frac{(d-1)^{  3\ell\sum_{j=1}^{p-1} r_{j1}}}{Z_\cF}\sum_{\bfi}\bm1(\cG\in \Omega)I(\widehat\cF, \cG)A_{bc}A_{b'c'} \widehat R_\bfi \\
  &=\frac{(d-1)^{  3\ell\sum_{j=1}^{p-1} r_{j1}}}{Z_{\cF}} \sum_{\widehat\bfi}\frac{\bm1(\cG\in \Omega)I(\widehat\cF, \cG)}{(Nd)^{|\bV_0|}} \sum_{u_j\sim v_j:j\in \bV_0} R'_{\widehat\bfi}\sum_{b\sim c\atop b'\sim c'} U_{bc}G_{c c'}  V_{b'c'}.
\end{split}\end{align}
For the summation over $b\sim c, b'\sim c'$ in \eqref{e:sumbound1}, we can view $U_{bc}$, $V_{b'c'}$ as two vectors indexed by $(b,c), (b',c')$ respectively.
Because $\|G\|_{\spec}\leq\eta^{-1}$, then 
\begin{align}\label{eq:spectralbound}
    \left|\sum_{b\sim c\atop b'\sim c'}U_{bc} G_{cc'}V_{b'c'}\right|
    \lesssim 
    \frac{1}{\eta} \sqrt{\sum_{b\sim c}|U_{bc}|^2 \sum_{b'\sim c'} |V_{b'c'}|^2}.
\end{align}

By plugging \eqref{eq:spectralbound} into \eqref{e:sumbound1}, we get
\begin{align}\begin{split}\label{e:sumbound2}
&\phantom{{}={}}\left|\frac{\bm1(\cG\in \Omega)I(\widehat\cF, \cG)}{(Nd)^{|\bV_0|}} \sum_{u_j\sim v_j:j\in \bV_0} R'_{\widehat\bfi}\sum_{b\sim c\atop b'\sim c'} U_{bc}G_{c c'}  V_{b'c'}\right|\\
 &\lesssim    \frac{\bm1(\cG\in \Omega)I(\widehat\cF, \cG)}{(Nd)^{|\bV_0|}}\sum_{u_j\sim v_j: j\in \bV_0}|R'_{\wh\bfi}|\frac{1}{\eta}\sqrt{\sum_{b\sim c}|U_{bc}|^2 \sum_{b'\sim c'} |V_{b'c'}|^2}\\
 &\lesssim    \frac{1}{\eta}\frac{N^{-(r-1)\fb} \bm1(\cG\in \Omega)I(\widehat\cF, \cG) }{(Nd)^{|\bV_0|}}\prod_{j\in \bV_1}|\cW_j|\sum_{u_j\sim v_j: j\in \bV_0}\sqrt{\sum_{b\sim c}|\widehat U'_{bc}|^2 \sum_{b'\sim c'} |\widehat V'_{b'c'}|^2},
\end{split}\end{align}
where in the last line we used \eqref{e:UVbound}. For the last line of \eqref{e:sumbound2}, a Cauchy-Schwarz inequality implies
\begin{align}\begin{split}\label{e:sumbound3}
    &\phantom{{}={}}\sum_{u_j\sim v_j: j\in \bV_0}\sqrt{\sum_{b\sim c}|\widehat U'_{bc}|^2 \sum_{b'\sim c'} |\widehat V'_{b'c'}|^2}
    \lesssim 
    \sqrt{\sum_{b\sim c}\sum_{u_j\sim v_j: j\in \bV_0}|\widehat U'_{bc}|^2\sum_{b'\sim c'}\sum_{u_j\sim v_j: j\in \bV_0}|\widehat V'_{b'c'}|^2}\\
    &\lesssim 
    (Nd)^{|\bV_0|+1}\prod_{j\in \bV_0, r_{j2}=0}N^{-(r_{j1}-2)\fb}N^\fo\Upsilon \Phi\prod_{j\in \bV_0, r_{j2}\geq1}{N^{-\max\{r_{j1}-1,0\}\fb} }N^\fo\Upsilon\sqrt{\frac{\Phi}{N}},
\end{split}\end{align}
where the last line follows from \eqref{e:k1}, \eqref{e:k2} and \eqref{e:k3}. Also,
\begin{align}\begin{split}\label{e:sumbound4}
    \prod_{j\in \bV_1}|\cW_j|&\lesssim \prod_{j:r_{j0}=1}|Q_t-Y_t|\prod_{j\in \bV_1:r_{j0}=2, r_{j2}=0}N^{-(r_{j1}-2)\fb} \Upsilon N^\fo\Phi\\
    &\times\prod_{j\in \bV_1:r_{j0}=2, r_{j2}\geq1}{N^{-\max\{r_{j1}-1,0\}\fb}\Upsilon N^\fo}\sqrt{\frac{\Phi}{N}}
    \prod_{j:r_{j0}=3}\frac{(d-1)^\ell\Upsilon}{N}.
\end{split}\end{align}
By plugging \eqref{e:sumbound2} and \eqref{e:sumbound3} and \eqref{e:sumbound4} into \eqref{e:sumbound1}, we conclude that
\begin{align*}
    &\phantom{{}={}}\frac{(d-1)^{3\ell\sum_{j=1}^{p-1} r_{j1}}}{Z_\cF}\sum_{\bfi}\bm1(\cG\in \Omega)I(\widehat\cF, \cG)A_{bc}A_{b'c'} \widehat R_\bfi\lesssim \frac{\bm1(\cG\in \Omega)}{N\eta} \frac{1}{N^{(r-1)\fb}} \prod_{j:r_{j0}=1}|Q_t-Y_t|\\
    &\times\prod_{j:r_{j0}=2, r_{j2}=0}(d-1)^{8\ell} \Upsilon \Phi\prod_{j:r_{j0}=2, r_{j2}\geq1}N^{-\fb}\Upsilon \Phi
    \prod_{j:r_{j0}=3}N^{-\fb}\Upsilon \Phi
    \lesssim \frac{\bm1(\cG\in \Omega)}{N\eta}  \frac{\Pi(z,\bmr)}{N^{(r-1)\fb}},
\end{align*}
where in the last statement we used \eqref{e:Eexp}.
This together with \eqref{e:relaxbc} finish the proof of \eqref{e:refined_bound2}. The claim \eqref{e:refined_bound2_special} follows from the same argument, so we omit its proof.
\end{proof}

\subsection{Estimates of terms involving $Z$}
\label{s:Z_term}

In this section, we prove the following lemma, which addresses the error terms involving the constrained  GOE matrix $Z$ from \eqref{e:Werror}.
\begin{proposition}\label{p:WRbound}
 Adopt the notation of \Cref{s:setting5}. We take $z\in {\bf D}$, an array $\bmr=[r_{jk}]_{1\leq j\leq p-1, 0\leq k\leq 2}$ satisfying \eqref{e:defr_copy}, and $R_\bfi\in \Adm(r,\bmr,\cF,\cG)$ as in \Cref{def:pgen}. Assume $\cG,\tcG\in \Omega$ and $I(\cF^+,\cG)=1$.
 Then the following expectation, with respect to the constrained GOE matrix $Z$, holds:
    \begin{align}\begin{split}\label{e:WRbound}
       &(d-1)^{(6r+3\sum_{j=1}^{p-1} r_{j1})\ell}|\bE_Z[\sqrt t Z_{xz} R_\bfi]|\leq N^{-\fb}\bE_Z[\Psi_p],\text{ for }x,z\in \bT,\\ 
       &(d-1)^{(6r+3\sum_{j=1}^{p-1} r_{j1})\ell}|\bE_Z[\sqrt t  (Z \widetilde G^{(\bT)})_{x z} R_\bfi]|\leq N^{-\fb}\bE_Z[\Psi_p],\text{ for }x\in \bT, z\not\in \bT,\\
       &(d-1)^{(6r+3\sum_{j=1}^{p-1} r_{j1})\ell}|\bE_Z[t  ((Z\widetilde G^{(\bT)}Z)_{x z}-m_t\delta_{xz}) R_\bfi]|\leq N^{-\fb}\bE_Z[\Psi_p], \text{ for }x,z\in \bT.
    \end{split}\end{align}
\end{proposition}
\begin{proof}
We prove the last two statements in \eqref{e:WRbound}. The first statement can be proven in the same way, so we omit its proof. 
    Thanks to the integration by parts formula \Cref{p:intbypart},
\begin{align*}
   \bE_Z[ (Z \widetilde G^{(\bT)})_{x z} R_\bfi]
=\sum_{y\not\in \bT}\sum_{x'y'\in \qq{N}}\bE_Z[Z_{xy}Z_{x'y'}]\bE_Z[\del_{Z_{x'y'}}(\wt G_{yz}^{(\bT)}R_\bfi)].
\end{align*}
There are two cases, if the derivative hits $\widetilde G_{yz}^{(\bT)}$, we get the following
\begin{align}\begin{split}\label{e:WG1}
&\phantom{{}={}}\sum_{yx'y'\not\in \bT}\bE_Z[Z_{xy}Z_{x'y'}]\del_{Z_{x'y'}}  \wt G_{yz}^{(\bT)}
=
    -\sqrt t\sum_{yx'y'\not\in \bT}\bE_Z[Z_{xy}Z_{x'y'}]  \wt G_{yx'}^{(\bT)}  \wt G_{y'z}^{(\bT)}\\
&=\frac{\OO(\sqrt t)}{N^2}\left(\sum_{yx'\not\in \bT}   \wt G_{yx'}^{(\bT)}  \wt G_{yz}^{(\bT)} +\sum_{yy'\not\in \bT}   \wt G_{yy}^{(\bT)}  \wt G_{y'z}^{(\bT)}\right)+\frac{\OO(\sqrt t)}{N^3}\sum_{yx'y'\not\in \bT}  \wt G_{yx'}^{(\bT)}  \wt G_{y'z}^{(\bT)}\\
&\lesssim N^\fo\sqrt t \Phi+\frac{\OO(\sqrt t)}{N^2}\sum_{yy'\not\in \bT}   \wt G_{yy}^{(\bT)}  \wt G_{y'z}^{(\bT)}\lesssim N^\fo\sqrt t \Phi,
\end{split}\end{align}
where  the first two terms on the second line are from the cases $y'=y$ and $x'=y$; to get the last line we used  \eqref{e:use_Ward2}; in the last inequality, we used from \eqref{e:use_Ward2} that $|\sum_{y'\not\in \bT} G_{y'z}^{(\bT)}|\lesssim \ell$. We conclude from \eqref{e:WG1}, \eqref{e:Sbound}, and \eqref{e:bPi1} that
\begin{align}\begin{split}\label{e:ttt1}
    &\phantom{{}+={}}(d-1)^{(6r+3\sum_{j=1}^{p-1} r_{j1})\ell}\sqrt t\sum_{yx'y'}\bE_Z[Z_{xy}Z_{x'y'}]\bE_Z[\del_{Z_{x'y'}}(\widetilde G_{yz}^{(\bT)})R_\bfi]\\
    &\lesssim (d-1)^{(6r+3\sum_{j=1}^{p-1} r_{j1})\ell} t\bE_Z[N^\fo \Phi|R_\bfi|]\lesssim (d-1)^{6r\ell}N^{\fo -r\fb} t \bE_Z[   \Phi\Pi(z,\bmr)]\leq N^{-\fb}\bE_Z[\Psi_p].
\end{split}\end{align}

In the following we focus on the case that the derivative hits $R_\bfi$. Denote
\begin{align*}
    R_\bfi=\prod_{j=1}^r B_j\prod_{j=1}^{p-1} \cW_j\in \Adm(r,\bmr,\cF_+,\cG),
\end{align*}
where the $B_j$ are factors of the form 
\begin{align}\begin{split}\label{e:allterm}
    &\{(G_{c c}^{(b)}-Q_t)\}_{(b,c)\in \cC^\circ}, \quad  \{ G_{c c'}^{(bb')}, G_{b c'}^{(b')}, G_{bb'},G_{cb'}\}_{(b,c)\neq (b',c')\in \cC^\circ},\\
&\{  G^{\circ}_{ss'}\}_{s,s'\in \cK},\quad (Q_t-\msc),\quad t(m_t-\md(z_t)).
\end{split}\end{align}

There are several cases for the derivative $\del_{Z_{x'y'}}R_\bfi$. Next we show that if $\del_{Z_{x'y'}}$ hits a term $B_j$ from \eqref{e:allterm}, then
\begin{align}\label{e:dWB0}
   \sum_{y\not\in \bT}\sum_{x'y'\in \qq{N}}\bE_Z[Z_{xy}Z_{x'y'}]\wt G_{yz}^{(\bT)}\del_{Z_{x'y'}}(B_j)\lesssim N^\fo\sqrt t\Phi.
\end{align}
We start with the case that $B_j=G_{cc'}^{(bb')}$: $\del_{Z_{x'y'}}G_{cc'}^{(bb')}=-\sqrt t G_{cx'}^{(bb')}G_{y'c'}^{(bb')}$, and  
\begin{align}\begin{split}\label{e:WGR1}
    &\phantom{{}={}}\sum_{y\not\in\bT}\sum_{x'y'\not\in\{b,b'\}}\bE_Z[Z_{xy}Z_{x'y'}]\wt G_{yz}^{(\bT)}G_{cx'}^{(bb')}G_{y'c'}^{(bb')}
    =\frac{\OO(1)}{N}\sum_{y\not\in\bT}\wt G_{yz}^{(\bT)}G_{cx}^{(bb')}G_{yc'}^{(bb')}\\
    &+\frac{\OO(1)}{N^2}\sum_{y\not\in\bT}\sum_{x'\not\in\{b,b'\}}(\wt G_{yz}^{(\bT)}G_{cx'}^{(bb')}G_{xc'}^{(bb')}+\wt G_{yz}^{(\bT)}G_{cx'}^{(bb')}G_{yc'}^{(bb')})\\
    &+\frac{\OO(1)}{N^3}\sum_{y\not\in\bT}\sum_{x'y'\not\in\{b,b'\}}\wt G_{yz}^{(\bT)}G_{cx'}^{(bb')}G_{y'c'}^{(bb')}\lesssim N^\fo\Phi.
\end{split}\end{align}
where in the last statement we use \eqref{e:Gest} and \eqref{e:use_Ward2}. 
In all cases of \eqref{e:dWB0}, $\wt G_{yz}^{(\bT)}\del_{Z_{x'y'}}(B_j)$ contains at least two off-diagonal Green's function terms, meaning they can be proven by the same way as in \eqref{e:WGR1}, so we omit the proofs. We conclude from \eqref{e:dWB0}, \eqref{e:Sbound} and \eqref{e:bPi1} that
\begin{align}\begin{split}\label{e:ttt2}
   &\phantom{{}={}} (d-1)^{(6r+3\sum_{j=1}^{p-1} r_{j1})\ell}\sqrt t\sum_{y\not\in \bT}\sum_{x'y'\in\qq{N}}\bE_Z[Z_{xy}Z_{x'y'}]\bE_Z\left[\wt G_{yz}^{(\bT)}\del_{Z_{x'y'}}\left(\prod_{j=1}^r B_j\right) \prod_{j=1}^{p-1}\cW_j\right]\\
    &\lesssim 
    \frac{(d-1)^{6r\ell}N^\fo t}{N^{(r-1)\fb}}\bE_Z\left[ \Phi(d-1)^{3\sum_{j=1}^{p-1} r_{j1}\ell} \prod_{j=1}^{p-1}|\cW_j|\right]
     \lesssim (d-1)^{6\ell}N^{\fo}t\bE_Z[  \Phi\Pi(z,\bmr)]\leq N^{-\fb}\bE_Z[\Psi_p].
\end{split}\end{align}

If $\del_{Z_{x'y'}}$ hits $\cW_j=Q_t-Y_t$ with $r_{j0}=1$, we have 
\begin{align}\label{e:dzQ-Y}
    \del_{Z_{x'y'}}(Q_t-Y_t)= (1-\del_1 Y_\ell)\del_{Z_{x'y'}}Q_t-t\del_2 Y_\ell \del_{Z_{x'y'}}m_t.
\end{align}
Next we show that 
\begin{align}\label{e:dQ-Y}
    \sum_{y\in \bT}\sum_{x'y'\in \qq{N}}\bE_Z[Z_{xy}Z_{x'y'}]\wt G_{yz}^{(\bT)}\times \{\del_{Z_{x'y'}} Q_t,\del_{Z_{x'y'}} m_t \}\lesssim \frac{N^\fo\sqrt t \Phi}{N\eta}.
\end{align}
It follows from combining \eqref{e:dzQ-Y} and \eqref{e:dQ-Y} that 
\begin{align}\label{e:dS1c1}
    \left| \sum_{y\not\in \bT}\sum_{x'y'\in\qq{N}}\bE_Z[Z_{xy}Z_{x'y'}]\wt G_{yz}^{(\bT)} \del_{Z_{x'y'}}(Q_t-Y_t)\right|\lesssim \frac{ \sqrt t }{N\eta}(\Upsilon N^\fo \Phi).
\end{align}
We prove the statement \eqref{e:dQ-Y} for $m_t$, the statement for $Q_t$ can be proven in the same way, so we omit its proof. Notice that $\del_{Z_{x'y'}} m_t=-(\sqrt t /N)\sum_{w\in \qq{N}}G_{wx'}G_{y'w}$, thus
\begin{align}
\begin{split}\label{e:WGR2}
     &\phantom{{}={}}
     \frac{1}{N}\sum_{y\not\in \bT}\sum_{wx'y'\in\qq{N}}\bE_Z[Z_{xy}Z_{x'y'}]\wt G_{yz}^{(\bT)}G_{wx'}G_{y'w}
     =\frac{\OO(1)}{N^2}\sum_{y\not\in \bT}\sum_{w\in\qq{N}}\wt G_{yz}^{(\bT)}G_{wx}G_{yw}\\
    &+\frac{\OO(1)}{N^3}\sum_{y\not\in \bT}\sum_{wx'\in\qq{N}}(\wt G_{yz}^{(\bT)}G_{wx'}G_{xw}+\wt G_{yz}^{(\bT)}G_{wx'}G_{yw})
    +\frac{\OO(1)}{N^4}\sum_{y\not\in \bT}\sum_{wx'y'\in\qq{N}}\wt G_{yz}^{(\bT)}G_{wx'}G_{y'w}.
\end{split}
\end{align}
All four terms on the right-hand side of \eqref{e:WGR2} can be estimated in the same way; we will only explain how to bound the first one. Using $\|G\|_{\spec}\leq 1/\eta$, we have 
\begin{align*}
    \frac{1}{N^2}\sum_{y\not\in \bT}\sum_{w\in\qq{N}}\wt G_{yz}^{(\bT)}G_{wx}G_{yw}\lesssim  \frac{1}{N^2\eta } \sqrt{\sum_{w\in \qq{N}} |G_{wx}|^2\sum_{y\not\in \bT}|\wt G_{yz}^{(\bT)}|^2 }\lesssim \frac{N^\fo \Phi}{N\eta},
\end{align*}
where we used \eqref{e:Gest} and \eqref{e:use_Ward2}. The other three terms on the right-hand side of \eqref{e:WGR2} can all be bounded in the same way, giving \eqref{e:dQ-Y}.

Similarly, if $\del_{Z_{x'y'}}$ hits $\cW_j=(d-1)^\ell\times \{(1-\del_1Y_\ell), \del_2 Y_\ell\}/N$ with $r_{j0}=3$, then 
\begin{align}\begin{split}\label{e:dSj3}
    &\phantom{{}={}}\frac{(d-1)^\ell}{N}\sum_{y\not\in \bT}\sum_{x'y'\in\qq{N}}\bE_Z[Z_{xy}Z_{x'y'}]\wt G_{yz}^{(\bT)}\times \{-\del_1\del_{Z_{x'y'}}Y_\ell, -\del_2 \del_{Z_{x'y'}}Y_\ell\}\\
    &=\frac{(d-1)^\ell}{N}\sum_{y\not\in \bT}\sum_{x'y'\in\qq{N}}\bE_Z[Z_{xy}Z_{x'y'}]\wt G_{yz}^{(\bT)}\times\{\OO(\ell^3) \del_{Z_{x'y'}} Q_t+\OO(\ell^3) \del_{Z_{x'y'}} m_t \}\\
    &\lesssim \frac{(d-1)^\ell\ell^3 \sqrt t \Phi}{N^2\eta}\lesssim \frac{ \sqrt t }{N\eta}(\Upsilon N^\fo \Phi),
\end{split}\end{align}
where in the first statement we used \eqref{e:Yl_derivative}; in the second statement we used \eqref{e:dQ-Y}; in the last statement we used \eqref{e:defPhi}, $\Upsilon\geq \Phi\geq 1/N^{1-2\fc}$.

The next situation is if  $\del_{Z_{x'y'}}$ hits $\cW_j$ with $r_{j0}=2$, which is associated to the special edge $\{u_j,v_j\}$. We recall from \eqref{e:Scase2}
\begin{align*}
    \cW_j=\{1-\del_1 Y_\ell, -t\del_2 Y_\ell\}\times \frac{1}{Nd}\sum_{u_j\sim v_j\in \qq{N}}F_{u_j v_j},
\end{align*}
where $F_{u_j v_j}$ is a product of $r_{j1}$ factors of the form 
   $ \{G^{\circ}_{su_j},G^{\circ}_{sv_j}\}_{s\in \cK},\quad  \{(\Av G^{\circ})_{o'u_j},(\Av G^{\circ})_{o'v_j}\}_{(i',o')\in \cC\setminus \cC^\circ}$;
$r_{j2}$ factors of the form 
  $ \{L_{su_j},L_{sv_j}\}_{s\in \cK}, \quad \{(\Av L)_{o'u_j},(\Av L)_{o'v_j}\}_{(i',o')\in \cC\setminus \cC^\circ}$, 
and an arbitrary number of factors of the form $G_{u_jv_j}, 1/G_{u_ju_j}$.
We will show that
\begin{align}\label{e:dWB}
    (d-1)^{3r_{j1}\ell}\sum_{y\not\in \bT}\sum_{x'y'\in\qq{N}}\bE_Z[Z_{xy}Z_{x'y'}]\wt G_{yz}^{(\bT)}\del_{Z_{x'y'}}(\cW_j)\lesssim \frac{N^{3\fc} \sqrt t}{N\eta}\times (d-1)^{8\ell}\Upsilon\Phi.
\end{align}

There are three cases:
\begin{enumerate}
    \item If $\del_{Z_{x'y'}}$ hits $1-\del_1 Y_\ell$ or $-t\del_2 Y_\ell$, then using \eqref{e:dSj3}, we can bound the term by
    \begin{align*}
       &\phantom{{}={}}(d-1)^{3r_{j1}\ell}\sum_{y\not\in \bT}\sum_{x'y'\in\qq{N}}\bE_Z[Z_{xy}Z_{x'y'}]\wt G_{yz}^{(\bT)}\times \frac{\del_{Z_{x'y'}}\{1-\del_1 Y_\ell, -t\del_2 Y_\ell\}}{Nd}\sum_{u_j\sim v_j\in \qq{N}}F_{u_j v_j}\\
       &\lesssim   \frac{\ell^3 \sqrt t \Phi}{N\eta}
       \times \frac{(d-1)^{3r_{j1}\ell}}{Nd}\sum_{u_j\sim v_j\in \qq{N}}|F_{u_j v_j}|\lesssim \frac{\ell^3 \sqrt t \Phi}{N\eta}\times (d-1)^{8\ell}\Upsilon \Phi,
    \end{align*}
    where in the last statement we use \eqref{e:1_Sbound}.
\item 
We now consider if $\del_{Z_{x'y'}}$ hits a factor $ G^{\circ}_{sw}$ for some $w\in \{u_j,v_j\}$ in $F_{u_j v_j}=G^{\circ}_{sw}\widehat F_{u_j v_j}$, where $\widehat F_{u_jv_j}$ contains $r_{j1}-1$ factors of the form \eqref{e:defcE0}, and $r_{j2}$ factors of the form \eqref{e:defLerror}. Notice that $\del_{Z_{x'y'}}G^{\circ}_{sw}=-\sqrt t G_{sx'}G_{y'w}$. Then,
\begin{align}
\begin{split}\label{e:WGR3}
     &\phantom{{}={}}\sum_{y\not\in \bT}\sum_{x'y'\in\qq{N}}\bE_Z[Z_{xy}Z_{x'y'}]\wt G_{yz}^{(\bT)}G_{sx'}G_{y'w}
     =\frac{\OO(1)}{N}\sum_{y\not\in \bT}\wt G_{yz}^{(\bT)}(G_{sx}G_{yw}+G_{sy}G_{xw})\\
    &+\frac{\OO(1)}{N^2}\sum_{y\not\in \bT}\sum_{x'\in\qq{N}}(\wt G_{yz}^{(\bT)}G_{sx'}G_{xw}+\wt G_{yz}^{(\bT)}G_{sx'}G_{yw}+\wt G_{yz}^{(\bT)}G_{sx}G_{x'w}+\wt G_{yz}^{(\bT)}G_{sy}G_{x'w})\\
    &+\frac{\OO(1)}{N^3}\sum_{y\not\in \bT}\sum_{x'y'\in\qq{N}}\wt G_{yz}^{(\bT)}G_{sx'}G_{y'w}
    = \frac{\OO(1)}{N}\sum_{y\not\in \bT}\wt G_{yz}^{(\bT)}G_{sx}G_{yw}
    +\OO\left(\frac{\sqrt{N^\fo\Phi}}{N} +N^\fo \Phi  |G_{xw}|\right),
\end{split}
\end{align}
where in the last statement we used \eqref{e:Gest}, \eqref{e:use_Ward2}, and $\sum_{x\in\qq{N}}G_{xy}=\OO(1)$ from \eqref{e:Ghao0}. 

By plugging \eqref{e:WGR3} into \eqref{e:dWB}, we get
\begin{align}\begin{split}\label{e:extra_term}
    &\phantom{{}={}}\frac{\OO(\sqrt t\Upsilon)}{dN}\sum_{u_j\sim v_j\in \qq{N}}\left(\frac{1}{N}\sum_{y\not\in \bT}\wt G_{yz}^{(\bT)}G_{sx}G_{yw}+\OO\left(\frac{\sqrt{N^\fo \Phi}}{N}+N^\fo \Phi  |G_{xw}|\right)\right)\widehat F_{u_j v_j}\\
&\lesssim \frac{\sqrt t\Upsilon}{N\eta}\sqrt{\sum_{y\not\in \bT} \frac{|\wt G_{yz}^{(\bT)}|^2}{N^2} \sum_{u_j\sim v_j\in \qq{N}}|\widehat F_{u_j v_j}|^2}\\
&+\frac{\sqrt t\Upsilon}{N} \sqrt{\sum_{u_j\sim v_j\in \qq{N}}\left(\frac{\sqrt{N^\fo \Phi}}{N}+N^\fo \Phi  |G_{xw}|\right)^2 \sum_{u_j\sim v_j\in \qq{N}}|\widehat F_{u_j v_j}|^2}\\
&\lesssim \frac{\sqrt t\Upsilon}{N}\left(\sqrt{\frac{N^{\fo} \Phi}{N\eta^2} \sum_{u_j\sim v_j\in \qq{N}}|\widehat F_{u_j v_j}|^2}+\sqrt{\left(\frac{N^\fo \Phi}{N}+N(N^\fo \Phi)^3  \right) \sum_{u_j\sim v_j\in \qq{N}}|\widehat F_{u_j v_j}|^2}\right)\\
&\lesssim \frac{\sqrt t\Upsilon }{N}\frac{ N^{2\fc+2\fo}\sqrt{ \Phi}}{\eta}\sqrt{ \frac{1}{N}\sum_{u_j\sim v_j\in \qq{N}}|\widehat F_{u_j v_j}|^2},
\end{split}\end{align}
where in the first statement we used that $\|G\|_{ \spec}\leq 1/\eta$ and Cauchy-Schwarz inequality; in the second statement we used \eqref{e:Gest} and \eqref{e:use_Ward2}; in the last statement we used that $\Phi=\Im[m_t]/(N\eta)+1/N^{1-2\fc}\lesssim N^{2\fc}/(N\eta)$.

If $r_{j2}=0$, then $\widehat F_{u_j v_j}$ contains $r_{j1}-1\geq 1$ factors of the form \eqref{e:defcE0}. Thanks to \eqref{e:naive-Ward} and \eqref{e:av_naive-Ward}, we have 
$(1/N)\sum_{u_j\sim v_j\in \qq{N}}|F_{u_j v_j}|^2\lesssim N^{-2(r_{j1}-2)\fb} N^\fo \Phi$;
If $r_{j2}\geq 1$, similarly, we have 
$(1/N)\sum_{u_j\sim v_j\in \qq{N}}|F_{u_j v_j}|^2\lesssim N^{-2\max\{r_{j1}-1,0\}\fb} (\fR/N)$. Thus in both cases, we have the bound $((d-1)^{3r_{j1}\ell}/N)\sum_{u_j\sim v_j\in \qq{N}}|F_{u_j v_j}|^2\lesssim (d-1)^{8\ell}\Phi$. By plugging this estimate into \eqref{e:extra_term}, we conclude 
\begin{align}\label{e:dWB1}
   (d-1)^{3r_{j1}\ell}\times \eqref{e:extra_term} \lesssim
    \frac{\sqrt t N^{3\fc} }{N\eta}\times (d-1)^{8\ell}\Upsilon\Phi.
\end{align}
\item If $\del_{Z_{x'y'}}$ hits a factor $(\Av G^{\circ})_{o'w}$ for some $w\in \{u_j,v_j\}$, the same estimate \eqref{e:dWB1} holds.
\end{enumerate}
All three cases discussed above together give \eqref{e:dWB}.

The three estimates \eqref{e:dS1c1}, \eqref{e:dSj3} and \eqref{e:dWB} together with \eqref{e:Bsmall} imply 
\begin{align}\begin{split}\label{e:ttt3}
    &\phantom{{}={}}(d-1)^{(6r+3\sum_{j=1}^{p-1} r_{j1})\ell}\sqrt t\sum_{yx'y'}\bE_Z[Z_{xy}Z_{x'y'}]\bE_Z\left[\prod_{j=1}^r B_j \wt G_{yz}^{(\bT)}\del_{Z_{x'y'}}\left( \prod_{j=1}^{p-1}\cW_j\right)\right]\\
    &\lesssim (d-1)^{6r\ell}N^{-r\fb}t \bE_Z\left[\frac{N^{3\fc}\sqrt t }{N\eta} ((d-1)^{8\ell}\Upsilon\Phi)  (|Q_t-Y_t|+(d-1)^{8\ell}\Upsilon \Phi)^{p-2}\right]\lesssim N^{-\fb}\bE_Z[\Psi_p].
\end{split}\end{align}

The second statement in \eqref{e:WRbound} follows from combining \eqref{e:ttt1}, \eqref{e:ttt2} and \eqref{e:ttt3}.

In the following, we explain the proof for the third statement in \eqref{e:WRbound} when $x=z\in \bT$. Thanks to the integration by parts formula in \Cref{p:intbypart},  
\begin{align}\begin{split}\label{e:WGW-m}
 &\phantom{{}={}}\bE_Z[ ((Z\widetilde G^{(\bT)}Z)_{x x}-m_t) R_\bfi]=\sum_{y\not\in\bT}\bE_Z[ Z_{xy}(\widetilde G^{(\bT)}Z)_{y x}R_\bfi]-\bE_Z[ m_t R_\bfi]\\
 &=\sum_{y\not\in \bT}\sum_{x'y'\in\qq{N}}\bE_Z[Z_{xy}Z_{x'y'}](\bE_Z[(\del_{Z_{x'y'}}(\widetilde G^{(\bT)}Z)_{yx}R_\bfi]
 +\bE_Z[(\widetilde G^{(\bT)}Z)_{yx}\del_{Z_{x'y'}}R_\bfi])-\bE_Z[ m_t R_\bfi].
\end{split}    
\end{align}
If the derivative hits $(\wt G^{(\bT)}Z )_{yx}$, we get
\begin{align}\begin{split}\label{e:WGW1}
    &\phantom{{}={}}\sum_{y\not\in \bT}\sum_{x'y'\in\qq{N}}\bE_Z[Z_{xy}Z_{x'y'}](\del_{Z_{x'y'}}((\wt G^{(\bT)}Z)_{yx})\\
    &=\sum_{yx'\not\in \bT}\bE_Z[Z_{xy}Z_{x'x}]\wt G_{yx'}^{(\bT)}-\sqrt t\sum_{yx'y'\not\in \bT}\bE_Z[Z_{xy}Z_{x'y'}]\wt G_{yx'}^{(\bT)}(\wt  G^{(\bT)}Z)_{y'x}\\
    &=\frac{1}{N}\sum_{y\not\in \bT}   \wt G_{yy}^{(\bT)}+\frac{\OO(1)}{N^2}\sum_{yx'\not\in \bT}  \wt G_{yx'}^{(\bT)}+\frac{\OO(\sqrt t)}{N^2}\left(\sum_{yx'\not\in \bT}  \wt G_{yx'}^{(\bT)}(\wt  G^{(\bT)}Z)_{yx}+\sum_{yy'\not\in \bT}   \wt G_{yy}^{(\bT)}(\wt  G^{(\bT)}Z)_{y'x}\right)+\\
    &+\frac{\OO(\sqrt t)}{N^3}\sum_{yx'y'\not\in \bT}   \wt G_{yx'}^{(\bT)}( \wt G^{(\bT)}Z)_{y'x}
    = \frac{1}{N}\sum_{y\not\in \bT}   \wt G_{yy}^{(\bT)}+\frac{\OO(\sqrt t)}{N^2}\sum_{yy'\not\in \bT}    \wt G_{yy}^{(\bT)}( \wt G^{(\bT)}Z)_{y'x}+\OO\left(\frac{\ell}{N}\right)
    \\
    &= m_t+\frac{\OO(\sqrt t)}{N}\sum_{y'\not\in \bT} m_t (  \wt G^{(\bT)}Z)_{y'x}+\OO((d-1)^\ell N^\fo\Phi).
\end{split}\end{align}
where the terms on the third line are from the cases $x'=y$ or $y'=y$; for third statement we used \eqref{e:Werror} that $\sum_{x'\not\in \bT}\wt G_{yx'}^{(\bT)}\lesssim \ell$; for the fourth statement we used \eqref{e:use_Ward2}. By plugging \eqref{e:WGW1} into the first term on the right-hand side of \eqref{e:WGW-m}, we conclude
\begin{align}\begin{split}\label{e:firstbb}
   &\phantom{{}={}}(d-1)^{(6r+3\sum_{j=1}^{p-1} r_{j1})\ell} t\left(\sum_{y\not\in \bT}\sum_{x'y'\in \qq{N}}\bE_Z[Z_{xy}Z_{x'y'}]\bE_Z[\del_{Z_{x'y'}}(\widetilde G^{(\bT)}Z)_{yx}R_\bfi]-\bE_Z[ m_t R_\bfi]\right)\\
    &=  (d-1)^{(6r+3\sum_{j=1}^{p-1} r_{j1})\ell} t\left(\frac{\OO(\sqrt t)}{N}\sum_{y'\not\in\bT}\bE_Z[m_t (  Z\wt G^{(\bT)})_{xy'} R_\bfi]+\OO(\bE_Z[(d-1)^\ell N^\fo |R_\bfi|]\right).
\end{split}\end{align}
 Here, although the first term  on the right-hand side of \eqref{e:firstbb} contains an extra $m_t$ factor (compared with the second statement in \eqref{e:WRbound}), the same argument gives the bound $N^{-\fb}\bE_Z[\Psi_p]$; the second term on the right-hand side of \eqref{e:firstbb} is also bounded by $N^{-\fb}\bE_Z[\Psi_p]$ by the same argument as in  \eqref{e:ttt1}. We conclude that 
 \begin{align*}
     \eqref{e:firstbb}\lesssim N^{-\fb}\bE_Z[\Psi_p].
 \end{align*}

The second term 
in the last line of \eqref{e:WGW-m} can be bounded by the same argument as that of \eqref{e:ttt2} and \eqref{e:ttt3}. 
More precisely, using \eqref{e:Werror2} as input, one can check that \eqref{e:dWB0}, \eqref{e:dS1c1}, \eqref{e:dSj3},  \eqref{e:dWB} and  \eqref{e:dWB1} hold if we replace $\wt G_{yz}^{(\bT)}$ by $(\widetilde G^{(\bT)}Z)_{yx}$.  This finishes the proof of the third statement in \eqref{e:WRbound}.
\end{proof}

\section{Proof of the Iteration Scheme}\label{s:proof_main}
In this section we prove \Cref{p:add_indicator_function}, \Cref{p:iteration},  \Cref{p:general} and \Cref{p:track_error1}.
\subsection{Proof of \Cref{p:add_indicator_function}}

\begin{proof}[Proof of \Cref{p:add_indicator_function}]
We recall the set $\overline{\Omega}$ from \eqref{def:omegabar}. First, we show that up to negligible error terms, we can assume the set $\cF$ remains well behaved in the switched graph (namely $I(\cF,\wt\cG)=1$). We split the left-hand side of \eqref{e:maint} into two terms
\begin{align}\begin{split}\label{e:maint0}
&\phantom{{}={}}\bE\left[I(\cF,\cG){\bm1(\cG\in \Omega)}(G_{oo}^{(i)}-Y_t)R_\bfi\right] \\ 
&=\bE\left[I(\cF,\cG)I(\cF,\wt\cG){\bm1(\cG\in \Omega)}\bm1(\tcG\in \oOmega)(G_{oo}^{(i)}-Y_t)R_\bfi\right]\\
&+\bE\left[I(\cF,\cG)\bE_\bfS[1-I(\cF,\wt\cG)\bm1(\tcG\in \oOmega)]{\bm1(\cG\in \Omega)}(G_{oo}^{(i)}-Y_t)R_\bfi\right] \\
&=\bE\left[I(\cF,\cG)I(\cF,\wt\cG){\bm1(\cG\in \Omega)\bm1(\tcG\in \oOmega)}(G_{oo}^{(i)}-Y_t)R_\bfi\right]\\
&+\OO\left(N^{-1+2\fc}\bE\left[I(\cF,\cG){\bm1(\cG\in \Omega)}|(G_{oo}^{(i)}-Y_t)||R_\bfi|\right] \right),
\end{split}\end{align}
where in the third line we used \Cref{lem:configuration}, that 
\begin{align*}
    \bE_\bfS[1-I(\cF,\tcG)\bm1(\tcG\in \oOmega) ]\leq  \bE_\bfS[1-\bm1(\tcG\in \oOmega) ]+\bE_\bfS[1-I(\cF,\tcG) ]\leq N^{-1+2\fc}.
\end{align*} 
After averaging over $\bfi$, the last term in \eqref{e:maint0} is sufficiently small. Namely, 
\begin{align*}
    &\phantom{{}={}}\frac{(d-1)^{(6r+3\sum_{j=1}^{p-1} r_{j1})\ell}}{Z_\cF}\sum_{\bfi}N^{-1+2\fc}\bE\left[I(\cF,\cG){\bm1(\cG\in \Omega)}|(G_{oo}^{(i)}-Y_t)||R_\bfi|\right]\\
    &\lesssim 
     \frac{(d-1)^{6r\ell}}{Z_\cF}\sum_{\bfi}N^{-1+2\fc}\bE\left[I(\cF,\cG){\bm1(\cG\in \Omega)}N^{-(r+1)\fb}(| Q_t- Y_t|+ (d-1)^{8\ell}\Upsilon\Phi)^{p-1}\right]
    =\OO\left(\frac{\bE[\Psi_p]}{N^\fb}\right),
\end{align*}
where in the first statement we used \eqref{e:Bsmall} and \eqref{e:Sbound}.

For the other term on the right-hand side of \eqref{e:maint0}, we now show we can further condition that $\wt{G}\in \Omega$.
\begin{align}\begin{split}\label{e:maint1}
&\phantom{{}={}}\bE\left[I(\cF,\cG)I(\cF,\wt\cG)\bm1(\cG\in \Omega) \bm1(\tcG\in \oOmega) (G_{oo}^{(i)}-Y_t)R_\bfi\right]\\
&=
\bE\left[I(\cF,\cG)I(\cF,\wt\cG){\bm1(\cG,\tcG\in \Omega)}(G_{oo}^{(i)}-Y_t)R_\bfi\right]\\
&+\bE\left[I(\cF,\cG)I(\cF,\wt\cG){\bm1(\cG\in\Omega)\bm1(\tcG\in \oOmega\setminus \Omega)} (G_{oo}^{(i)}-Y_t)R_\bfi\right].
\end{split}\end{align}
We can bound the second term on the right-hand side of \eqref{e:maint1} by \eqref{e:Bsmall}, \eqref{e:Sbound}, and H{\" o}lder's inequality 
\begin{align}\begin{split}\label{e:maint2}
&\phantom{{}={}}\frac{(d-1)^{(6r+3\sum_{j=1}^{p-1} r_{j1})\ell}}{Z_\cF}\sum_{\bfi}\bE\left[I(\cF,\cG)I(\cF,\wt\cG){\bm1(\cG\in\Omega)\bm1(\tcG\in \oOmega\setminus \Omega)} |(G_{oo}^{(i)}-Y_t)||R_\bfi|\right]\\
&\lesssim N^{-\fb} \bE\left[ {\bm1(\cG\in\Omega)\bm1(\tcG\in \oOmega\setminus \Omega)}(| Q_t- Y_t|+ (d-1)^{8\ell}\Upsilon\Phi)^{p-1}\right]\\
&\lesssim 
N^{-\fb} \bE[\bm1(\tcG\in \oOmega\setminus \Omega)]^{\frac{1}{p}}\bE\left[ \bm1(\cG\in\Omega)(| Q_t- Y_t|+ (d-1)^{8\ell}\Upsilon\Phi)^{p}\right]^{\frac{p-1}{p}}
\lesssim N^{-\fb}\bE[\Psi_p],
\end{split}\end{align}
where in the last statement we used  \Cref{thm:prevthm0} that $\bP(\oOmega\setminus \Omega)\leq N^{-\fC'}$ for any $\fC'>0$, provided $N$ is large enough.

The main term is the first term on the right-hand side of \eqref{e:maint1}, we can rewrite it as
\begin{align}\begin{split}\label{e:towrite}
     \bE\left[I(\cF,\cG)I(\cF,\wt\cG)\bm1(\cG,\tcG\in \Omega)G_{oo}^{(i)} R_\bfi\right]-\bE\left[I(\cF,\cG)I(\cF,\wt\cG)\bm1(\cG,\tcG\in \Omega)Y_t R_\bfi\right].
\end{split}\end{align}
Thanks to \Cref{lem:exchangeablepair}, $\cG$ and $\tcG=T_{\bf S}(\cG)$ form an exchangeable pair. We can rewrite the first term of \eqref{e:towrite} as
\begin{align}\begin{split}\label{e:towrite2}
    \bE\left[I(\cF,\cG)I(\cF,\wt\cG)\bm1(\cG,\tcG\in \Omega)G_{oo}^{(i)} R_\bfi\right]=\bE\left[I(\cF,\cG)I(\cF,\wt\cG)\bm1(\cG,\tcG\in \Omega)\wt G_{oo}^{(i)} \wt R_\bfi\right].
\end{split}\end{align}
Similarly, we can rewrite the second term on the right-hand side of \eqref{e:towrite} as
\begin{align}\begin{split}\label{e:towrite3}
    &\phantom{{}={}}\bE\left[I(\cF,\cG)I(\cF,\wt\cG)\bm1(\cG,\tcG\in \Omega)Y_t R_\bfi\right]=\bE\left[I(\cF,\cG)I(\cF,\wt\cG)\bm1(\cG,\tcG\in \Omega)\wt Y_t \wt R_\bfi\right]\\
    &=\bE\left[I(\cF,\cG)I(\cF,\wt\cG)\bm1(\cG,\tcG\in \Omega)Y_t \wt R_\bfi\right]+\bE\left[I(\cF,\cG)I(\cF,\wt\cG)\bm1(\cG,\tcG\in \Omega)(Y_t-\wt Y_t) R_\bfi\right].
\end{split}\end{align}
For the last term on the right-hand side of \eqref{e:towrite3}, we recall that $R_\bfi\in \Adm(r, \bmr, \cF,\cG)$. If $r\geq 1$ or $|\{j: r_{j0}=2,3\}|\geq 1$, thanks to \eqref{e:Sbound}, we can bound it as
\begin{align}\begin{split}\label{e:towrite4}
    &\phantom{{}={}}\left|\frac{(d-1)^{(6r+3\sum_{j=1}^{p-1} r_{j1})\ell}}{Z_\cF}\sum_{\bfi}\bE\left[I(\cF,\cG)I(\cF,\wt\cG)\bm1(\cG,\tcG\in \Omega)(Y_t-\wt Y_t) R_\bfi\right]\right|\\
    &\lesssim
    \frac{(d-1)^{6r\ell}}{Z_\cF}\sum_{\bfi}\bE\left[\bm1(\cG,\tcG\in \Omega)(d-1)^{8\ell}\Phi N^{-r\fb} \Pi(z,\bmr)\right]\lesssim N^{-\fb/2}\bE[\Psi_p],
\end{split}\end{align}
where we used  \eqref{e:bPi2} and the following bound from \eqref{e:Yl_derivative} and \eqref{e:Qtmtbound} 
\begin{align}\label{e:Ytdiff}
    &|\wt Y_t-Y_t|\lesssim |\del_1 Y_\ell||\wt Q_t-Q_t|+|\del_2 Y_\ell||\wt m_t-m_t|)\lesssim \ell (d-1)^{6\ell} N^\fo\Phi\lesssim (d-1)^{8\ell}\Phi.
\end{align}

Otherwise, $R_\bfi=(Q_t-Y_t)^{p-1}$. Because $\cG, \wt \cG$ form an exchangeable pair
\begin{align}\begin{split}\label{e:Ridouble}
    &\phantom{{}={}}\frac{1}{Z_\cF}\sum_{\bfi}\bE\left[I(\cF,\cG)I(\cF,\wt\cG)\bm1(\cG,\tcG\in \Omega)(Y_t-\wt Y_t) R_\bfi\right]\\
    &=\frac{1}{2Z_\cF}\sum_{\bfi}\bE\left[I(\cF,\cG)I(\cF,\wt\cG)\bm1(\cG,\tcG\in \Omega)(Y_t-\wt Y_t) (R_\bfi-\wt R_\bfi)\right].
\end{split}\end{align}
Thanks to \eqref{e:tilde_bound}
\begin{align}\begin{split}\label{e:Rtdiff}
    &\phantom{{}={}}|R_\bfi-\wt R_\bfi|
    \lesssim |(\wt Q_t-\wt Y_t)-(Q_t-Y_t)|(|\wt Q_t-\wt Y_t|+|Q_t-Y_t|)^{p-2}\\
    &\lesssim (d-1)^{8\ell}\Upsilon \Phi
    (|Q_t-Y_t|+(d-1)^{8\ell}\Upsilon \Phi)^{p-2}.
\end{split}\end{align}
By plugging \eqref{e:Ytdiff} and \eqref{e:Rtdiff} into \eqref{e:Ridouble}, we conclude that $\eqref{e:Ridouble}=\OO(N^{-\fb/2}\bE[ \Psi_p])$.

Therefore, in \eqref{e:towrite}, we are left to take the difference of the right-hand side of \eqref{e:towrite2} and the first term of the right-handside of \eqref{e:towrite3}. We claim we can replace the indicator functions $I(\cF, \cG)I(\cF, \wt \cG)$ by $I(\cF^+, \cG)$, and \begin{align}\begin{split}\label{e:fhatF}
    &\phantom{{}={}}\frac{(d-1)^{(6r+3\sum_{j=1}^{p-1} r_{j1})\ell}}{Z_\cF}\sum_{ \bfi}\bE\left[I(\cF,\cG)I(\cF,\wt\cG)\bm1(\cG,\tcG\in \Omega) (\wt G_{oo}^{(i)}-Y_t) \wt R_\bfi\right]\\
    &=\frac{(d-1)^{(6r+3\sum_{j=1}^{p-1} r_{j1})\ell}}{Z_\cF}\sum_{ \bfi}\bE\left[I(\cF^+,\cG)\bm1(\cG,\tcG\in \Omega) (\wt G_{oo}^{(i)}-Y_t) \wt R_\bfi\right]+\OO(N^{-\fb/2}\bE[\Psi_p]).
 \end{split}\end{align}   
This is true as by \Cref{lem:configuration}, the error term is bounded by 
\begin{align*}
    &\phantom{{}={}}\frac{(d-1)^{(6r+3\sum_{j=1}^{p-1} r_{j1})\ell}}{Z_\cF}\sum_{ \bfi}\bE\left[|I(\cF,\cG)I(\cF,\wt\cG)-I(\cF^+,\cG)|\bm1(\cG,\tcG\in \Omega) |\wt G_{oo}^{(i)}-Y_t| |\wt R_\bfi|\right]\\
     &\leq\frac{(d-1)^{6r\ell}}{Z_\cF}\sum_{ \bfi}\bE\left[I(\cF,\cG)\bE_\bfS[|I(\cF,\wt\cG)-I(\cF^+,\cG)|]\bm1(\cG\in \Omega) N^{-(r+1)\fb}  (|Q_t-Y_t|+(d-1)^{8\ell}\Upsilon \Phi)^{p-1}\right]\\
     &=\OO(N^{-\fb/2}\bE[\Psi_p]),
\end{align*}
where we used \eqref{e:Sbound}.
Up to negligible error in \eqref{e:fhatF}, we can pass from the average over $\bfi$ to an average over all possible embeddings of $\cF^+$ in $\cG$,
\begin{align*}
      \eqref{e:fhatF}=\frac{(d-1)^{(6r+3\sum_{j=1}^{p-1} r_{j1})\ell}}{Z_{\cF^+}}\sum_{ \bfi^+}\bE\left[I(\cF^+,\cG)\bm1(\cG,\tcG\in \Omega) (\wt G_{oo}^{(i)}-Y_t) \wt R_\bfi\right]+\OO(N^{-\fb/2}\bE[\Psi_p]),
\end{align*}
this finishes the proof of \Cref{p:add_indicator_function}.

\end{proof}

\subsection{Iteration First Step: Proof of \Cref{p:iteration} and \Cref{p:track_error1}}
\label{s:first_step}

In this section, we prove \Cref{p:iteration} and \Cref{p:track_error1}. We denote
\begin{align*}
    \cW_j=Q_t-Y_t, \quad R_\bfi=\prod_{j=1}^{p-1}\cW_j, \quad 1\leq j\leq p-1.
\end{align*}
Although the values of $\cW_j$ are identical, we associate each $\cW_j$ with a special edge $\{u_j, v_j\}$ (as defined in \eqref{e:special_edge}).
Then we can write \eqref{e:IFIF} as
\begin{align}\begin{split}\label{e:freplace0}
    &\phantom{{}={}}{Z^{-1}_{\cF^+}}\sum_{\bfi^+}\bE\left[I(\cF^+,\cG)\bm1(\cG,\tcG\in \Omega)(\wt G_{oo}^{(i)}-Y_t) (\wt Q_t-\wt Y_t)^{p-1}\right]\\
    &={Z^{-1}_{\cF^+}}\sum_{\bfi^+}\bE\left[I(\cF^+,\cG)\bm1(\cG,\tcG\in \Omega) (\wt G_{oo}^{(i)}-Y_t)\prod_{j=1}^{p-1}  \wt \cW_j\right],
\end{split}\end{align}
where for the factors $\cW_j$, we have denoted the corresponding term for the Green's function of the resampled graph as $\wt \cW_j$. Our next goal is to rewrite the product $(\wt G_{oo}^{(i)}-Y_t)\prod_{j=1}^{p-1}\wt \cW_j$ which depends on the Green's functions of the switched graph $\tcG$, in terms of the Green's functions of the original graph 
$\cG$, see \Cref{fig:first_step}. This is achieved as follows.

\begin{enumerate}

\item \label{i:case1}
By \Cref{lem:diaglem}, we can decompose $\wt G_{oo}^{(i)}-Y_t=\widehat B_0+\cZ_0+\cE_0$, where $\cZ_0, \cE_0$ are $\cZ, \cE$ in \eqref{e:G-Y} and
$\wh B_0$ is given by
\begin{align}\label{e:first_W0}
\wh B_0= \frac{\msc^{2\ell}(z_t)} {(d-1)^{\ell+1}}\sum_{\al\in\sfA_i}(G^{(b_\alpha)}_{c_\alpha c_\alpha}-Q_t)+\frac{\msc^{2\ell}(z_t)} {(d-1)^{\ell+1}}\sum_{\al\neq \beta\in\sfA_i}G^{(b_\alpha b_\beta)}_{c_\alpha c_\beta}
   + \cU_0. 
\end{align}
Here, $\cU_0$ is an $\OO(1)$-weighted sum of terms of the form $(d-1)^{3(h-1)\ell}R_h$, where $h\geq 2$, and $R_h$ is an $S$-product term (see \Cref{d:S-product}) containing at least one factor of the form $( G_{c_\al c_{\al}}^{(b_\al)}-Q_t)$ or $G_{c_\al c_{\beta}}^{(b_\al b_\beta)}$; 

\item \label{i:case2}\label{i:c2}
  For $\cW_j=Y_t-Q_t$, by \Cref{c:Q-Ylemma} we have 
$\wt \cW_j=\widehat \cW_j+\widehat\cE_j$, where $|\widehat\cE_j|\lesssim ((d-1)^{8\ell}\Phi)^2$ and  $\wh \cW_j=(Q_t-Y_t)+\wh \cU_j$, where 
$\wh \cU_j$ is an $\OO(1)$-weighted sum of terms of the following forms
\begin{enumerate}
\item For some $h\geq 0$, and $h_1+h_2\geq 2$ even,
     \begin{align}\label{e:first_special_term1}
     &(d-1)^{3h\ell}R_{h}\times \{(1-\del_1 Y_\ell), -t\del_2 Y_\ell\}\frac{(d-1)^{3h_1\ell} }{Nd}\sum_{u_j\sim v_j} R'_{h_1,h_2} , 
\end{align} 
   where $R_h$ is a W-product term (see \Cref{d:W-product}); $R'_{h_1, h_2}$ contains $h_1$ factors of the form $ G^{\circ}_{su_j},  G^{\circ}_{sv_j}$, $h_2$ factors of the form $L_{su_j}, L_{sv_j}$, with $s\in\{l_\al, a_\al,b_\al, c_\al\}_{\al\in \qq{\mu}}$ and an arbitrary number of factors $G_{u_jv_j}, 1/G_{u_ju_j}$;
  \item  Alternatively, for $h\geq 0$
\begin{align}
\label{e:first_special_term2}
     (d-1)^{3h\ell}R_h\times  \frac{(d-1)^\ell (1-\del_1 Y_\ell)}{N}.
\end{align} 
 where $R_h$ is a W-product term (see \Cref{d:W-product}). 
 \end{enumerate}
\end{enumerate}

\begin{figure}
    \centering
    \includegraphics[scale=0.7]{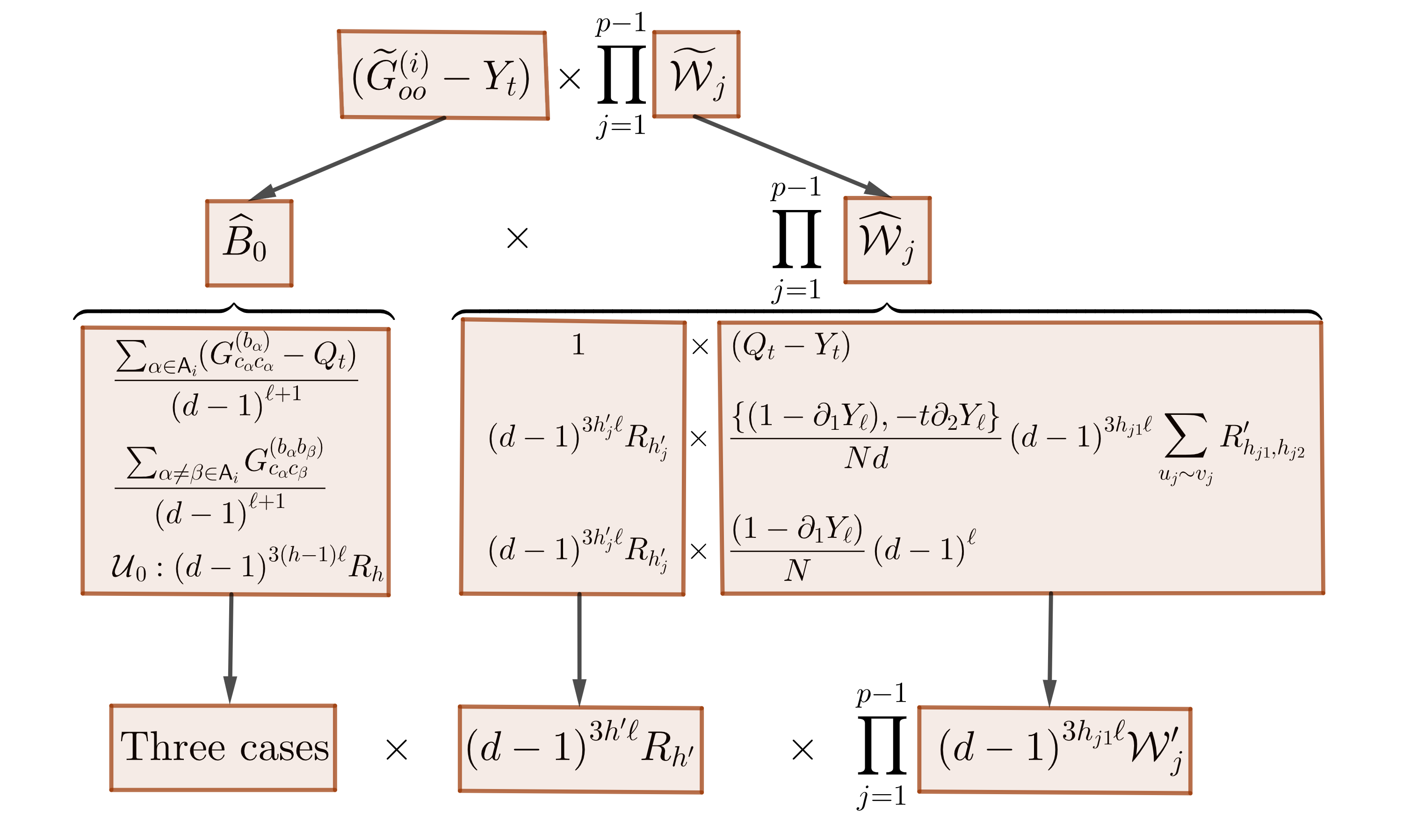}
    \caption{
This figure illustrates the three steps involved in the proof of \Cref{p:iteration}. 
First, we rewrite the product 
$(\wt G_{oo}^{(i)} - Y_t) \prod_{j=1}^{p-1} \wt \cW_j$,
which depends on the Green's functions of the switched graph $\tcG$, in terms of the Green's functions of the original graph $\cG$. This reformulation is expressed as 
$\widehat{B}_0 \prod_{j=1}^{p-1} \widehat{\cW}_j$
in \Cref{l:fyibu}. Second, we regroup the term $\prod_{j=1}^{p-1} \widehat{\cW}_j$ as $(d-1)^{3h'\ell} R_{h'}$ multiplied by an admissible function, see \eqref{e:W_decompose}. Finally, \Cref{p:iteration} follows by discussing the three cases in the decomposition \eqref{e:first_W0} of $\widehat{B}_0$.
    \label{fig:first_step}}
\end{figure}

The errors $(\wt G_{oo}^{(i)}-Y_t)-\widehat B_0=\cZ_0+\cE_0$ and $\wt \cW_j-\widehat \cW_j=\widehat \cE_j$ in \Cref{i:case1} and \Cref{i:case2} are small. The following lemma states that we can replace $(\wt G_{oo}^{(i)}-Y_t)\prod_{j=1}^{p-1}\wt \cW_j$ with $\widehat B_0\prod_{j=1}^{p-1}\widehat \cW_j$, with the overall error from this substitution being negligible. The proof is deferred to \Cref{sec:iterationproofs}.
\begin{lemma}\label{l:fyibu}
Adopt assumptions and notation of \Cref{p:iteration}. We can rewrite \eqref{e:IFIF} as
\begin{align}
\label{e:fhatRi}
  \eqref{e:IFIF}=  \frac{1}{Z_{\cF^+}}\sum_{\bfi^+}\bE\left[I(\cF^+ ,\cG)\bm1(\cG\in \Omega)  \widehat B_0\prod_{j=1}^{p-1}\widehat \cW_j\right]+\cE,
\end{align}
where let $\sfA_i=\{\al\in\qq{\mu}: \dist_{\cG}(i,l_\al)=\ell+1\}$ and 
\begin{align}\begin{split}\label{e:error_E}
    \cE&=\frac{1}{Z_{\cF^+}}\sum_{\al,\beta\in \sfA_i}\sum_{\bfi^+}\bE\left[I(\cF^+,\cG)\bm1(\cG,\wt \cG\in \Omega)\frac{\msc^{2\ell}(z_t)(\wt G_{c_\al c_\beta}^{(\bT)}-G_{c_\al c_\beta}^{(b_\al b_\beta)}) }{(d-1)^{\ell+1}} (Q_t-Y_t)^{p-1}\right]\\
    &+\OO(N^{-\fb/4}\bE[ \Psi_p])
    =\OO((d-1)^{2\ell}\bE[\Psi_p]).
\end{split}\end{align}
\end{lemma}

We can further decompose $\widehat B_0\prod_{j=1}^{p-1}\widehat \cW_j$ in \eqref{e:fhatRi} as an $\OO(1)$-weighted sum of terms. This corresponds to the decomposition of $\widehat B_0$
  as described in \eqref{e:first_W0}, and 
$\widehat \cW_j$
  as either 
$Q_t-Y_t$
  or terms from \eqref{e:first_special_term1} or \eqref{e:first_special_term2}.

There are two cases; the first is that all $\widehat \cW_j$ are replaced by $(Q_t-Y_t)$. Otherwise some $\widehat \cW_j$ is replaced by \eqref{e:first_special_term1} or \eqref{e:first_special_term2}. 
We analyze the two cases separately.

\noindent \textbf{Case 1.} All $\widehat \cW_j$ are replaced by $(Q_t-Y_t)$.  This case is further divided into three subcases, based on the decomposition of $\widehat{B}_0$ as described in \eqref{e:first_W0}. 
\begin{enumerate}
\item The term from replacing $\widehat{B}_0$ by $\msc^{2\ell}(z_t) {(d-1)^{-(\ell+1)}}\sum_{\al\in\sfA_i}(G^{(b_\alpha)}_{c_\alpha c_\alpha}-Q_t)$ is given by
\begin{align}\label{e:ftt1}
&\frac{\OO(1)}{(d-1)^{\ell}} \sum_{\al\in \sfA_i}   \sum_{\bfi^+} \frac{1}{Z_{\cF^+}} \bE\left[I(\cF^+,\cG)\bm1(\cG\in \Omega)(G_{c_\al c_\al}^{(b_\al)}-Q_t)(Q_t-Y_t)^{p-1}\right].
\end{align}

We can estimate \eqref{e:ftt1} using the following more general result:
Given  $h\geq 0$, $\bmh=[h_{jk}]_{1\leq j\leq p-1, 0\leq k\leq 2}$ satisfying \eqref{e:defr}, consider any   
$\overline R_{\bfi^+} \in \Adm(h,\bmh,\cF^+,\cG)$ (as defined in \eqref{def:pgen}) such that for some specific $\alpha\in\qq{\mu}$, $\overline R_{\bfi^+}$ does not depend on $\{b_\al, c_\al\}$. Then we claim that
 \begin{align}\label{e:huanG-Q}
         \frac{(d-1)^{(6h+3\sum_{j=1}^{p-1} h_{j1})\ell}}{Z_{\cF^+}}\sum_{\bfi^+} \bE\left[I(\cF^+,\cG)\bm1(\cG\in \Omega)(G_{c_\al c_\al}^{(b_\al)}-Q_t)\overline R_{\bfi^+}\right]=\OO(N^{-\fb/2}\bE[ \Psi_p]).
    \end{align}

Assuming \eqref{e:huanG-Q}, by taking $\overline R_{\bfi^+}=(G_{c_\al c_\al}^{(b_\al)}-Q_t)(Q_t-Y_t)^{p-1}$, $h=0$ and $h_{j0}=1$ for all $1\leq j\leq p-1$ in \eqref{e:huanG-Q}, we conclude that
\begin{align*}
    \eqref{e:ftt1}=\OO(N^{-\fb/2}\bE[\Psi_p]).
\end{align*}

    To prove \eqref{e:huanG-Q}, 
      we recall the definition of the indicator function $I(\cF^+,\cG)$ from \eqref{def:indicator}, and 
      define $
    \widehat \cF=(\widehat \bfi, \widehat E)\deq\cF^+\setminus  \{(b_\al,c_\al)\},\quad \widehat \bfi=\bfi^+\setminus\{b_\al, c_\al\}$.
Then 
\begin{align}\label{e:Idecompose}
 I(\cF^+,\cG)=I(\widehat \cF, \cG)A_{c_\al b_\al}I_{c_\al}, 
\end{align}
where $I_{c_\al}=\bm1(c_\al\not\in \bX)$ for 
$\bX$ the set of vertices $v$ such that either there is some $u\in \cB_\ell(v,\cG)$ such that $\cB_\fR(u,\cG)$ is not a tree, or $\dist(v,c)<3\fR$ for some $(b_\al, c_\al)\neq (b,c)\in \cC^+$. Then $I_{c_\al}$ satisfies requirements in \Cref{def:Ic}, and the first statement in  \eqref{e:single_index_sum} (by taking $(b,c)=(b_\al, c_\al)$) gives
\begin{align}\begin{split}\label{e:huanGcb00}
     &\phantom{{}={}}\frac{\bm1(\cG\in \Omega)}{ Z_{\cF^+}}\left|\sum_{b_\al,c_\alpha\in\qq{N}} I(\cF^+,\cG)(G_{c_\al c_\al}^{(b_\al)}-Q_t)\right|\\
     &=
      \frac{\bm1(\cG\in \Omega)I(\widehat \cF,\cG)}{Z_{\cF^+}}\left|\sum_{b_\al,c_\alpha\in\qq{N}}A_{b_\alpha c_\alpha}(G_{c_\al c_\al}^{(b_\al)}-Q_t)I_{c_\al}\right|\lesssim \frac{\bm1(\cG\in \Omega)I(\widehat \cF,\cG)}{Z_{\cF^+}}\frac{Nd}{N^{1-3\fc/2}}\\
      &\lesssim \frac{\bm1(\cG\in \Omega)I(\widehat \cF,\cG)}{Z_{\cF^+}}\frac{\sum_{b_\al,c_\alpha\in \qq{N}} A_{b_\alpha c_\alpha}I_{c_\al}}{N^{1-3\fc/2}}=\frac{\bm1(\cG\in \Omega)}{N^{1-3\fc/2} Z_{\cF^+}}\sum_{b_\al, c_\alpha\in \qq{N}} I(\cF^+,\cG),
\end{split}\end{align} 
and we can bound the left-hand side of  \eqref{e:huanG-Q} as 
\begin{align*}\begin{split}
   &\phantom{{}={}}\frac{(d-1)^{(6h+3\sum_{j=1}^{p-1} h_{j1})\ell}}{N^{1-3\fc/2}  Z_{\cF^+}}\sum_{ \bfi^+}\bE\left[I(\cF^+,\cG)\bm1(\cG\in \Omega)|\overline R_{\bfi^+}|\right]\\
   &\lesssim  \frac{(d-1)^{6h\ell}}{N^{1-3\fc/2}}\bE\left[\bm1(\cG\in \Omega) N^{-h\fb} (|Q_t-Y_t|+\Upsilon(d-1)^{8\ell}\Phi)^{p-1}\right]\lesssim N^{-\fb/2}\bE[ \Psi_p],
    \end{split}\end{align*}
    where we used  $\overline R_{\bfi^+}\in \Adm(h,\bmh,\cF^+,\cG)$ and \eqref{e:Bsmall} for the first statement. This finishes the proof of \eqref{e:huanG-Q}.

     \item 
     The term from replacing $\widehat{B}_0$ by $\msc^{2\ell}(z_t) {(d-1)^{-(\ell+1)}}\sum_{\al\neq \beta\in \sfA_i} G_{c_\al c_\beta}^{(b_\al b_\beta)}$ is given by
\begin{align}\label{e:ftt2}
\frac{\OO(1)}{(d-1)^{2\ell}} \sum_{\al\neq \beta\in \sfA_i}  (d-1)^\ell \times \sum_{\bfi^+}  \frac{1}{Z_{\cF^+}}\bE\left[I(\cF^+,\cG)\bm1(\cG\in \Omega) G_{c_\al c_\beta}^{(b_\al b_\beta)}(Q_t-Y_t)^{p-1}\right].
\end{align}

We can estimate \eqref{e:ftt2} using the following more general result:
Given  $h\geq 0$, $\bmh=[h_{jk}]_{1\leq j\leq p-1, 0\leq k\leq 2}$ satisfying \eqref{e:defr} and  
$\overline R_{\bfi^+} \in \Adm(h,\bmh,\cF^+,\cG)$ (as defined in \eqref{def:pgen}). Suppose $\overline R_{\bfi^+}=\overline R_{\bfi^+}(\bm\chi)$
depends on $\{b_\chi, c_\chi\}_{\chi \in \bm\chi}$ for some subset ${\bm\chi}\subset\qq{\mu}$, and for some $\al\neq \beta\in \qq{\mu}$ with $f_1\deq |\{\al, \beta\}\setminus \bm\chi|\geq 1$. We claim that
\begin{align}\begin{split}\label{e:huanG}
       &\phantom{{}={}}\frac{(d-1)^{(6h+3\sum_{j=1}^{p-1} h_{j1})\ell}}{Z_{\cF^+}}\sum_{\bfi^+}  \bE\left[I(\cF^+,\cG)\bm1(\cG\in \Omega) G_{c_{\al} c_{\beta}}^{(b_{\al} b_{\beta})}\overline R_{\bfi^+}\right]\\
       &=\frac{(d-1)^{(6h+3\sum_{j=1}^{p-1} h_{j1})\ell}}{Z_{\cF^+}}\sum_{\bfi^+}  \bE\left[I(\cF^+,\cG)\bm1(\cG\in \Omega) \frac{\overline R'_{\bfi^+}}{(d-1)^{f_1/2}}\right]+\OO(N^{-\fb/4} \bE[ \Psi_p]),
   \end{split} \end{align}
where $\overline R'_{\bfi^+}\in \Adm(h+f_1, \bmh,\cF^+,\cG)$ is obtained from $G_{c_{\al} c_{\beta}}^{(b_{\al} b_{\beta})}\overline R_{\bfi^+}$ by making the following substitutions: 
\begin{align}\begin{split}\label{e:final_replace1}
 &   G_{c_{\al} c_{\beta}}^{(b_{\al} b_{\beta})}
    \rightarrow 
 G_{b_{\al} c_{\beta}}^{ (b_{\beta})}(G_{c_{\al} c_{\al}}^{(b_{\al})}-Q_t) , 
\quad \al\notin \bm\chi,\quad \beta \in\bm\chi \\
 &   G_{c_{\al} c_{\beta}}^{(b_{\al} b_{\beta})}
    \rightarrow 
 G_{c_{\al} b_{\beta}}^{ (b_{\al})}(G_{c_{\beta} c_{\beta}}^{(b_{\beta})}-Q_t) , 
\quad \al\in \bm\chi,\quad \beta\notin \bm\chi\\
&G_{c_{\al} c_{\beta}}^{(b_{\al} b_{\beta})}
\rightarrow 
G_{b_{\al} b_{{\beta}}}(G_{c_{\al} c_{\al}}^{(b_{\al})}-Q_t) (G_{c_{{\beta}} c_{{\beta}}}^{(b_{{\beta}})}-Q_t) , 
\quad \al, \beta\notin \bm\chi.
\end{split}\end{align} 
Before proving \eqref{e:final_replace1}, let us first use it to show that, up to a negligible error, \eqref{e:ftt2} can be expressed as an $\OO(1)$-weighted sum of terms in the form of \eqref{e:case3}.

    By taking $\overline R_{\bfi^+}=(Q_t-Y_t)^{p-1}$ in \eqref{e:huanG} (in which case $h=0$,  $h_{j0}=1$ for all $1\leq j\leq p-1$, and $f_1=2$), we conclude that \eqref{e:ftt2} is an $\OO(1)$-weighted sum of terms in the form
    \begin{align}\begin{split}\label{e:fhatRiGcc1} 
          &\phantom{{}={}}\frac{(d-1)^\ell}{ Z_{\cF^+}}\sum_{ \bfi^+}\bE\left[I(\cF^+,\cG)\bm1(\cG\in \Omega)G_{c_\al c_\beta}^{(b_\al b_\beta)}(Q_t-Y_t)^{p-1}\right]\\
          &=\frac{(d-1)^\ell}{ (d-1)Z_{\cF^+}}\sum_{ \bfi^+}\bE\left[I(\cF^+,\cG)\bm1(\cG\in \Omega)R'_{\bfi^+}\right]+\OO(N^{-\fb/4} \bE[\Psi_p]).
          \end{split}\end{align}
      where $R'_{\bfi^+}=G_{b_\al b_\beta}(G_{c_\al c_\al}^{(b_\al)}-Q_t) (G_{c_\beta c_\beta}^{(b_\beta)}-Q_t)(Q_t-Y_t)^{p-1}$.
Let $R_{\bfi^+}=(G_{c_\beta c_\beta}^{(b_\beta)}-Q_t)G_{b_\al b_\beta}(Q_t-Y_t)^{p-1}$, we can further rewrite the leading term on the right-hand side of \eqref{e:fhatRiGcc1} as
\begin{align}\label{e:newterm2}
 \frac{(d-1)^{\ell}}{(d-1)Z_{\cF^+}}\sum_{ \bfi^+}\bE\left[I(\cF^+,\cG)\bm1(\cG\in \Omega) (G_{c_\al c_\al}^{(b_\al)}-Q_t) R_{\bfi^+}\right].
 \end{align}
We can further change $(G_{c_\al c_\al}^{(b_\al)}-Q_t)$ in \eqref{e:newterm2} to $(G_{c_\al c_\al}^{(b_\al)}-Y_t)$, 
\begin{align}\begin{split}
\label{e:changeQtoY}
&\phantom{{}={}}\frac{(d-1)^{\ell}}{(d-1)Z_{\cF^+}}\sum_{\bfi^+}\bE\left[I(\cF^+,\cG)\bm1(\cG\in \Omega)  (G_{c_\al c_\al}^{(b_\al)}-Q_t) R_{\bfi^+}\right]\\
&=\frac{(d-1)^{\ell}}{(d-1)Z_{\cF^+}}\sum_{\bfi^+}\bE\left[I(\cF^+,\cG)\bm1(\cG\in \Omega)  \left((G_{c_\al c_\al}^{(b_\al)}-Y_t) R_{\bfi^+}+\OO(|(Q_t-Y_t) R_{\bfi^+}|)\right)\right]
\\
&=\frac{1}{d-1}\frac{(d-1)^{6\times 2\ell}}{(d-1)^{\fq^+\ell/2}Z_{\cF^+}}\sum_{\bfi^+}\bE\left[I(\cF^+,\cG)\bm1(\cG\in \Omega)  (G_{c_\al c_\al}^{(b_\al)}-Y_t) R_{\bfi^+}\right] +\OO\left(N^{-\fb/4} \bE[\Psi_p]\right).
\end{split}\end{align}
These terms are of the form \eqref{e:case3} with $r^+=2, \bmr^+=0, \fq^+=22$. 

The three cases in \eqref{e:final_replace1} can be proven in the same way, so we will only prove \eqref{e:final_replace1} for $\al,\beta\not\in \bm\chi$.  We create the forest $\widehat \cF$ from $\cF$ by removing $\{(b_\al,c_\al), (b_{\beta}, c_{\beta})\}$, so 
$
    \widehat \cF\deq (\widehat \bfi, \widehat E)=\cF^+\setminus  \{(b_\al,c_\al), (b_{\beta}, c_{\beta})\},\quad \widehat \bfi=\bfi^+\setminus\{b_\al, c_\al, b_{\beta}, c_{\beta}\}
$. Then
\begin{align}\label{e:Idecompose2}
I(\cF^+,\cG)=I(\widehat \cF, \cG)A_{b_\al c_\al }A_{b_{\beta} c_{\beta} }I_{c_\al c_\beta},
\end{align}
 where $I_{c_\al c_\beta}=\bm1(c_\al, c_{\beta} \not\in \bX)\bm1(\dist_\cG(c_\al, c_{\beta})\geq 3\fR)$ and 
$\bX$ is the collection of vertices $v$ such that either there exists some $u\in \cB_\ell(v,\cG)$ such that $\cB_\fR(u,\cG)$ is not a tree or $\dist(v,c)<3\fR$ for some $(b,c)\in \cC^+\setminus\{(b_\al, c_\al), (b_{\beta}, c_{\beta})\}$. Then $I_{c_\al c_\beta}$ satisfies the requirements of \Cref{def:Ic}, and the second statement in \eqref{e:Gccbb} (by taking $(b,c,b',c')=(b_\al, c_\al, b_\beta, c_\beta)$) gives
\begin{align}\begin{split}\label{e:huanGcc}
   &\phantom{{}={}}\frac{\bm1(\cG\in \Omega)}{Z_{\cF^+}}\sum_{b_\al,c_\al\atop b_{\beta},c_{\beta}}I(\cF^+,\cG) G_{c_\al c_{\beta}}^{(b_\al b_{\beta})}I_{c_\al c_\beta} 
   =\frac{I(\widehat\cF,\cG)\bm1(\cG\in \Omega)}{Z_{\cF^+}}\sum_{b_\al\sim c_\al\atop b_{\beta}\sim c_{\beta}}G_{c_\al c_{\beta}}^{(b_\al b_{\beta})}I_{c_\al c_\beta} \\
     &=\frac{I(\widehat\cF,\cG)\bm1(\cG\in \Omega)}{Z_{ \cF^+}}\sum_{b_\al\sim c_\al\atop b_{\beta}\sim c_{\beta}}\left(\frac{G_{b_\al b_{\beta}}(G_{c_\al c_\al}^{(b_\al)}-Q_t) (G_{c_{\beta} c_{\beta}}^{(b_{\beta})}-Q_t)}{d-1}+\OO(\cE_{\al\beta}) \right)I_{c_\al c_\beta}\\
      &=\frac{\bm1(\cG\in \Omega)}{Z_{ \cF^+}}\sum_{b_\al, c_\al\atop b_{\beta}, c_{\beta}}I(\cF^+,\cG)\left(\frac{G_{b_\al b_{\beta}}(G_{c_\al c_\al}^{(b_\al)}-Q_t) (G_{c_{\beta} c_{\beta}}^{(b_{\beta})}-Q_t)}{d-1}+\OO(\cE_{\al\beta}) \right),
\end{split}\end{align}     
   where 
   \begin{align*}
       |\cE_{\al \beta}|\lesssim  N^{-\fb/2}  |G_{b_\al b_{\beta}}|\left(\sum_{x\sim b_\al, x\neq c_\al}|G_{c_\al x}^{(b_\al)}|+\sum_{x\sim b_{\beta}, x\neq c_{\beta}}|G_{c_{\beta} x}^{(b_{\beta})}|\right)+N^{-\fb/2} \Phi.
   \end{align*}

We now attempt to write \eqref{e:huanGcc} in the form of \eqref{e:huanG}. The leading term in \eqref{e:huanGcc} gives the first term on the right-hand side of \eqref{e:huanG}, where $\overline R_{\bfi^+}'$ is obtained from $G_{c_\al c_\beta}^{(b_\al b_\beta)}\overline R_{\bfi^+}$ by replacing $G_{c_\al c_\beta}^{(b_\al b_\beta)}$ by $G_{b_{\al} b_{{\beta}}}(G_{c_{\al} c_{\al}}^{(b_{\al})}-Q_t) (G_{c_{{\beta}} c_{{\beta}}}^{(b_{{\beta}})}-Q_t)$.
By \eqref{e:Bsmall}, \eqref{e:Gest} and \eqref{e:sameasdisconnect} (with $\Pi=\bm1(\cG\in \Omega)(|Q_t-Y_t|+\Upsilon(d-1)^{8\ell}\Phi)^{p-1}$), the error term in \eqref{e:huanGcc} is negligible:
\begin{align*}\begin{split}
       &\phantom{{}={}}\frac{(d-1)^{(6h+3\sum_{j=1}^{p-1} h_{j1})\ell}}{Z_{\cF^+}}\sum_{\bfi^+}\bE\left[I(\cG\in \Omega)\bm1(\cF^+,\cG)||\cE_{\al\beta}||\overline R_{\bfi^+}|\right]\\
&\lesssim 
    \frac{(d-1)^{6h\ell} N^{-h\fb}}{Z_{\cF^+}}\sum_{\bfi^+}\bE\left[I(\cG\in \Omega)\bm1(\cG\in \Omega) |\cE_{\al\beta}|(|Q_t-Y_t|+\Upsilon(d-1)^{8\ell}\Phi)^{p-1}\right]
   \lesssim \frac{\bE[\Psi_p]}{N^{\fb/4}}.
\end{split}\end{align*}  
This finishes the proof of \eqref{e:huanG}.

\item In the remaining cases, $\widehat B_0$ is replaced by  $\cU_0$  as specified in the decomposition  \eqref{e:first_W0}.
Recalling the first few terms of $\cU_0$ from  \eqref{e:Uterm}, we can write
\begin{align}\label{e:decomposeI1I2I3}
    &\frac{1}{Z_{\cF^+}}\sum_{\bfi^+} \bE\left[I(\cF^+,\cG)\bm1(\cG\in \Omega)\cU_0 (Q_t-Y_t)^{p-1}\right]
    =J_1+J_2+J_3,
\end{align}
where
\begin{align}
 &\label{e:defI1}J_1=\sum_{\al \in \sfA_i}\sum_{\bfi^+} \frac{\msc^{2\ell}(z_t)L^{(i)}_{l_\al l_\al}}{(d-1)^{\ell+2}Z_{\cF^+}} \bE\left[I(\cF^+,\cG)\bm1(\cG\in \Omega)(G_{c_\al c_\al}^{(b_\al)}-Q_t)^2 (Q_t-Y_t)^{p-1}\right],\\
   &\label{e:defI2}J_2= \sum_{ \al \in \sfA_i,\beta\in\qq{\mu}\atop\al\neq \beta}\sum_{\bfi^+} \frac{\msc^{2\ell}(z_t)(L^{(i)}_{l_\beta l_\beta}+L^{(i)}_{l_\al l_\beta})}{(d-1)^{\ell+2}Z_{\cF^+}}
\bE\left[I(\cF^+,\cG)\bm1(\cG\in \Omega)(G_{c_\al c_\beta}^{(b_\al b_\beta)})^2 (Q_t-Y_t)^{p-1}\right].
\end{align}
And $J_3$ is an $\OO(1)$-weighted sum of terms in the form 
\begin{align}\label{e:fhatRi03}
  \sum_{\bfi^+} \frac{(d-1)^{3(h-1)}}{Z_{\cF^+}} \bE\left[I(\cF^+,\cG)\bm1(\cG\in \Omega)R_h(Q_t-Y_t)^{p-1}\right],
\end{align}
where $R_h$ is an $S$-product term of order $h$ (see \Cref{d:S-product}), which contains $G_{c_\al c_\al}^{(b_\al)}-Q_t$ or $G_{c_\al c_\beta}^{(b_\al b_\beta)}$. Moreover, either $h\geq 3$, or $h=2$ and  $R_h$ (recall from \eqref{e:Uterm}) is one of the following  terms 
\begin{align}\begin{split}\label{e:Rhform}
 &(G_{c_\al c_\al}^{(b_\al)}-Q_t) (G_{c_\beta c_\beta}^{(b_\beta)}-Q_t), \quad   (G_{c_\al c_\al}^{(b_\al)}-Q_t)G_{c_{\al'} c_{\beta'}}^{(b_{\al'} b_{\beta'})},\quad G_{c_{\al} c_{\beta}}^{(b_{\al} b_{\beta})}G_{c_{\al'} c_{\beta'}}^{(b_{\al'} b_{\beta'})},\\
&\{G_{c_\al c_\al}^{(b_\al)}-Q_t, G_{c_\al c_\beta}^{(b_\al b_\beta)}\}\times\{(Q_t-\msc(z_t)), t(m_t-\md(z_t))\},
\end{split} \end{align}
where $\alpha\neq \beta\in\qq{\mu}, \al'\neq \beta'\in \qq{\mu}$ and $ \{\al, \beta\}\neq \{\al', \beta'\}$.

In the following we discuss the three terms $J_1, J_2, J_3$ one by one. For $J_1$, by the same argument as in \eqref{e:changeQtoY}, we can change one copy of $(G_{c_\al c_\al}^{(b_\al)}-Q_t)$ in \eqref{e:defI1} to $(G_{c_\al c_\al}^{(b_\al)}-Y_t)$, and the error is bounded by $\OO(N^{-\fb/2} \bE[\Psi_p])$. After such replacement, we get
\begin{align}\label{e:newterm3}
    J_1=\sum_{\al \in \sfA_i}\sum_{\bfi^+}\frac{\msc^{2\ell}(z_t)L^{(i)}_{l_\al l_\al}}{(d-1)^{\ell+2}Z_{\cF^+}}  \bE\left[I(\cF^+,\cG)\bm1(\cG\in \Omega)(G_{c_\al c_\al}^{(b_\al)}-Y_t) (G_{c_\al c_\al}^{(b_\al)}-Q_t)(Q_t-Y_t)^{p-1}\right],
\end{align}
which is an $\OO(1)$-weighted sum of terms in the form \eqref{e:case1} with $\fq^+=0$.

For $J_2$ as in \eqref{e:defI2}, \eqref{e:refined_bound} gives 
\begin{align}\begin{split}\label{e:Gccerror}
    |J_2|\lesssim \left|\sum_{\al\neq \beta}\frac{1}{(d-1)^\ell Z_{\cF^+}}\sum_{\bfi^+}\bE[I(\cF^+,\cG)\bm1(\cG\in \Omega)(G_{c_\al c_\beta}^{(b_\al b_\beta)})^2(Q_t-Y_t)^{p-1}]\right|\lesssim (d-1)^{2\ell}\bE[\Psi_p],
    \end{split}\end{align}
    meaning it can be incorporated into $\cE$ in \eqref{e:IFIF}.
  
Finally, for $J_3$, we recall that $R_h$ contains either $G_{c_\al c_\al}^{(b_\al)}-Q_t$ or $G_{c_\al c_\beta}^{(b_\al b_\beta)}$. By repeatedly applying statements \eqref{e:huanG-Q} and \eqref{e:huanG}, we obtain that either $\eqref{e:fhatRi03}=\OO(N^{-\fb/2}\bE[ \Psi_p])$, or 
\begin{align}\begin{split}\label{e:huanGcopy}
       \eqref{e:fhatRi03} &=\frac{(d-1)^{3(h-1)\ell}}{(d-1)^{(h'-h)/2}Z_{\cF^+}}\sum_{\bfi^+}  \bE\left[I(\cF^+,\cG)\bm1(\cG\in \Omega) R'_{\bfi^+}\right]+\OO(N^{-\fb/4}\bE[ \Psi_p])\\
        &=\frac{1}{(d-1)^{(h'-h)/2}}\frac{(d-1)^{6(h'-1)\ell}}{(d-1)^{\fq^+\ell/2}Z_{\cF^+}}\sum_{\bfi^+}  \bE\left[I(\cF^+,\cG)\bm1(\cG\in \Omega) R'_{\bfi^+}\right]+\OO(N^{-\fb/4}\bE[  \Psi_p]),
   \end{split} \end{align}
where $R'_{\bfi^+}\in \Adm(h', \bmh,\cF^+,\cG)$ with $h'\geq h\geq 2$, and $\fq^+=6(h'-h)+6(h'-1)\geq 12$. Furthermore, if $b_\al$ or  $c_\al$ appears in $R'_{\bfi^+}$, then they appear in at least two terms within $R'_{\bfi^+}$.

If $h=2$, $R_h$ is given in \eqref{e:Rhform}, either $\eqref{e:fhatRi03}=\OO(N^{-\fb/2}\bE[\Psi_p])$ by applying \eqref{e:huanG-Q}; or we can apply \eqref{e:final_replace1} at least once for $G_{c_{\al'}c_{\beta'}}^{(b_{\al'} b_{\beta'})}$. We conclude that in \eqref{e:huanGcopy}, $h'\geq h+1\geq 3$, and $R'_{\bfi^+}$ contains a factor in the form $G_{c_\al c_\al}^{(b_\al)}-Q_t$. The same as in \eqref{e:changeQtoY}, we can change $(G_{c_\al c_\al}^{(b_\al)}-Q_t)$ to $(G_{c_\al c_\al}^{(b_\al)}-Y_t)$. After such replacement, \eqref{e:huanGcopy} is of the form \eqref{e:case3} with $r^+= h'-1\geq 2, \bmr^+=0, \fq^+=6(h'-h)+6(h'-1)\geq 12$.

If $h\geq 3$, then $h'\geq h\geq 3$. Either $R'_{\bfi^+}$ contains a factor $G_{c_\al c_\al}^{(b_\al)}-Q_t$, or $R'_{\bfi^+}$ contains two factors in the form $G_{c_\al c_\beta}^{(b_\al b_\beta)}$. In the first case, by the same argument as above, \eqref{e:huanGcopy} leads to \eqref{e:case3} with $r^+= h'-1\geq 2, \bmr^+=0, \fq^+=6(h'-h)+6(h'-1)\geq 12$. For the second case,  
 thanks to \eqref{e:refined_bound}, $\eqref{e:huanGcopy}=\OO((d-1)^{6(h'-1)\ell}N^{-(h'-2)\fb} \bE[\Psi_p])=\OO(N^{-\fb/2}\bE[\Psi_p])$.

\end{enumerate}

\noindent \textbf{Case 2.} Some $\widehat \cW_j$ is replaced by \eqref{e:first_special_term1} or \eqref{e:first_special_term2}. Before discussing the second case, we record the following decomposition of  $\prod_{j=1}^{p-1} \widehat \cW_j$ as an $\OO(1)$-weighted sum of terms. This decomposition corresponds to the breakdown of $\widehat W_j$ as either $Q_t-Y_t$ or terms from \eqref{e:first_special_term1} or \eqref{e:first_special_term2}. These three cases are given by
\begin{enumerate}
    \item The term
    \begin{align}\label{e:diyi}
        \cW_j':=Y_t-Q_t.
    \end{align}
    We set 
      $(h_{j0}, h_{j1}, h_{j2})=(1,0,0)$ and $h'_{j}=0$.
    \item 
    A term of the form from \eqref{e:first_special_term1}: for $h_j'\geq 0$, $h_{j1}+h_{j2}\geq 2$ even,
     \begin{align}\label{e:dier}
     &(d-1)^{3h_j'\ell}R_{h_j'}\times (d-1)^{3h_{j1}\ell}\cW_j',\quad \cW_j':= \frac{\{(1-\del_1 Y_\ell), -t\del_2 Y_\ell\} }{Nd}\sum_{u_j\sim v_j\in \qq{N}} R'_{h_{j1},h_{j2}} , 
\end{align} 
   where $R_{h_j'}$ is a W-product term (see \Cref{d:W-product}). Here, $R'_{h_{j1},h_{j2}}$ contains $h_{j1}$ factors of the form $ G^{\circ}_{su_j},  G^{\circ}_{sv_j}$, $h_{j2}$ factors of the form $L_{su_j}, L_{sv_j}$, with $s\in\{l_\al, a_\al,b_\al, c_\al\}_{\al\in \qq{\mu}}$ and some number (that we do not need to track) of factors $G_{u_jv_j}, 1/G_{u_ju_j}$.
   We set $h_{j0}=2$, and note that $\cW_j'$ is in the form of \eqref{e:Scase2} with $(r_{j0}, r_{j1}, r_{j2})=(h_{j0}, h_{j1}, h_{j2})$.  
   
    \item A term of the form from \eqref{e:first_special_term2}: for $h_j'\geq 0$,
\begin{align}\label{e:disan0}
     (d-1)^{3h_j'\ell}R_{h_j'}\times \cW_j', \quad \cW_j':= \frac{(d-1)^\ell (1-\del_1 Y_\ell)}{N},
\end{align} 
 where $R_{h_j'}$ is a W-product term (see \Cref{d:W-product}). We set $(h_{j0}, h_{j1}, h_{j2})=(3,0,0)$ and note that $\cW_j'$ is in the form of \eqref{e:Scase3} with $(r_{j0}, r_{j1}, r_{j2})=(h_{j0}, h_{j1}, h_{j2})$.
\end{enumerate}

The discussion above gives a decomposition of $\prod_{j=1}^{p-1} \widehat \cW_j$ as an $\OO(1)$-weighted sum of terms in the following form: for $h'\geq 0$, and $\bmh=[h_{jk}]_{1\leq j\leq p-1, 0\leq k\leq 2}$ satisfying \eqref{e:defr}, 
\begin{align}\label{e:W_decompose}
(d-1)^{(3h'+3\sum_{j=1}^{p-1} h_{j1})\ell}\widehat R_{\bfi^+}.
\end{align}
Here $\widehat R_{\bfi^+}$ satisfies
\begin{align}\label{e:Rhath}
&\widehat R_{\bfi^+}
=R_{h'}\prod_{i=1}^{p-1}\cW_j'\in \Adm(h',\bmh,\cF^+,\cG), 
\end{align}
where  
 $\cW_j'$ and $\bmh$ are as defined in \eqref{e:diyi}, \eqref{e:dier} or \eqref{e:disan0}. The term $R_{h'}$ includes all other factors $R_{h_j'}$ from  \eqref{e:dier} and \eqref{e:disan0}, with a total count of $h'=\sum_{j=1}^{p-1} h_j'$. 

 In the second case, there exists some $1\leq j'\leq p-1$ such that $\widehat \cW_{j'}$ is replaced by \eqref{e:dier} or \eqref{e:disan0}. Then $h_{j'0}\in\{2,3\}$. 
Again this case can be further divided into three subcases,
depending on the decomposition of $\widehat B_0$ as described in \eqref{e:first_W0}.

\begin{enumerate}
 \item 
The terms from replacing $\widehat{B}_0$ by $\msc^{2\ell}(z_t) {(d-1)^{-(\ell+1)}}\sum_{\al\in\sfA_i}(G^{(b_\alpha)}_{c_\alpha c_\alpha}-Q_t)$ are
in the forms:
for $h'\geq 0$, $\bmh=[h_{jk}]_{1\leq j\leq p-1, 0\leq k\leq 2}$ satisfying \eqref{e:defr} and 
\begin{align}\label{e:fhatRi01}
&\frac{1}{(d-1)^{\ell}} \sum_{\al\in \sfA_i}   \sum_{\bfi^+} \frac{(d-1)^{(3h'+3\sum_{j=1}^{p-1} h_{j1})\ell}}{Z_{\cF^+}} \bE\left[I(\cF^+,\cG)\bm1(\cG\in \Omega)(G_{c_\al c_\al}^{(b_\al)}-Q_t)\widehat R_{\bfi^+}\right],
\end{align}
where $\widehat R_{\bfi^+}\in \Adm(h',\bmh,\cF^+,\cG)$ is as in \eqref{e:Rhath}. 


In \eqref{e:fhatRi01}, by \eqref{e:huanG-Q} (taking $\overline R_{\bfi^+}$ as $\widehat R_{\bfi^+}$) any term $\widehat R_{\bfi^+}$ which does not contain $\{b_\alpha, c_\alpha\}$, has contribution $\OO(N^{-\fb/2}\bE[\Psi_p])$.
Thus up to error $\OO(N^{-\fb/2}\bE[\Psi_p])$, \eqref{e:fhatRi01} can be expressed as an $\OO(1)$-weighted sum of terms where $\{b_\al, c_\al\}$ appear in $\widehat R_{\bfi+}$:
\begin{align}\label{e:Gcc-QR}
    \frac{ (d-1)^{3(h'+\sum_{j=1}^{p-1} h_{j1})\ell}}{(d-1)^{\ell}Z_{\cF^+}}\sum_{\bfi^+}\bE\left[I(\cF^+,\cG)\bm1(\cG\in \Omega) (G_{c_\al c_\al}^{(b_\al)}-Q_t)\widehat R_{\bfi^+}\right].
\end{align}

If $\bmh$ satisfies \eqref{e:feasible} or $h'\geq 1$, then after replacing $(G_{c_\al c_\al}^{(b_\al)}-Q_t)$ in \eqref{e:Gcc-QR} with $(G_{c_\al c_\al}^{(b_\al)}-Y_t)$, \eqref{e:Gcc-QR} is of form \eqref{e:case3} with $r^+=h', \bmr=\bmh, \fq^+\geq 2$. Thanks to \eqref{e:Sbound} and \eqref{e:bPi2}, the error in making this replacement is bounded by 
\begin{align}\begin{split}\label{e:a_error}
        &\phantom{{}={}}\frac{ (d-1)^{3(h'+\sum_{j=1}^{p-1} h_{j1})\ell}}{(d-1)^{\ell}Z_{\cF^+}}\sum_{\bfi^+}\bE\left[I(\cF^+,\cG)\bm1(\cG\in \Omega) |Y_t-Q_t||\widehat R_{\bfi^+}|\right]\\
        &\lesssim 
         (d-1)^{3h'\ell} N^{-h'\fb}\bE\left[\bm1(\cG\in \Omega) |Y_t-Q_t|\Pi(z,\bmh)\right]
        \lesssim N^{-\fb/2}\bE[ \Psi_p].
\end{split}\end{align}

Otherwise $h'=0$ and $\bmh$ satisfies \eqref{e:1_expand} with $h_{j'0}=2, h_{j'1}=2, h_{j'2}=0$, and $\cW_{j'}'$ is from the decomposition of $(\wt Q_t-\wt Y_t)- (Q_t-Y_t)$ given in \eqref{e:first_special_term1}. We recall the first few terms of $(\wt Q_t-\wt Y_t)- (Q_t-Y_t)$ from \eqref{e:refine_QYdiff}. Then in this case, corresponding terms in \eqref{e:fhatRi01} are of the form 
\begin{align}
     &\label{e:aterm}\phantom{{}={}}\sum_{\al\in \qq{\mu}}\sum_{\bfi^+}\frac{ \msc^{2\ell}(z_t)}{(d-1)^{\ell+1}Z_{\cF^+}}\bE\left[I(\cF^+,\cG)\bm1(\cG\in \Omega) (G_{c_\al c_\al}^{(b_\al)}-Q_t)(Q_t-Y_t)^{p-2}\cU_\al\right]\\
     &\label{e:abterm}+\sum_{\al\neq \beta\in \qq{\mu}}\sum_{\bfi^+}\frac{2(p-1)\msc^{2\ell}(z_t)L_{l_\al l_\beta}}{(d-1)^{\ell+2}Z_{\cF^+}}\bE\left[  I(\cF^+,\cG)\bm1(\cG\in \Omega) (G_{c_\al c_\al}^{(b_\al)}-Q_t)(Q_t-Y_t)^{p-2}\cU_{\al\beta}\right],
\end{align}
where we recall $F_{uv}^{(\al)}$ from \eqref{e:Fuv}, and 
\begin{align*}\begin{split}
   \cU_\al&\deq \frac{1}{Nd} \sum_{u\sim v} (1-\del_1 Y_\ell)\left(F_{vv}^{(\al)}-\frac{2G_{uv}}{G_{uu}}F_{uv}^{(\al)}+\left(\frac{G_{uv}}{G_{uu}}\right)^2F_{uu}^{(\al)}\right)+\frac{1}{Nd} \sum_{u\sim v}\del_2 Y_\ell F_{vv}^{(\al)}\\
\cU_{\al\beta}&\deq\frac{1}{Nd} \sum_{u\sim v}(1-\del_1 Y_\ell)
   \left( G^{\circ}_{v c_\al}+\frac{\msc(z_t)  G^{\circ}_{v b_\al}}{\sqrt{d-1}}-\frac{G_{uv}}{G_{uu}}\left( G^{\circ}_{u c_\al}+\frac{\msc(z_t)  G^{\circ}_{u b_\al}}{\sqrt{d-1}}\right)\right)\times\\
    &\times\left( G^{\circ}_{v c_\beta}+\frac{\msc(z_t)  G^{\circ}_{v b_\beta}}{\sqrt{d-1}}-\frac{G_{uv}}{G_{uu}}\left( G^{\circ}_{u c_\beta}+\frac{\msc(z_t)  G^{\circ}_{u b_\beta}}{\sqrt{d-1}}\right)\right)\\
& +\frac{1}{Nd}\sum_{u\sim v}(-t\del_2 Y_\ell)\left( G^{\circ}_{u c_\al}+\frac{\msc(z_t)  G^{\circ}_{u b_\al}}{\sqrt{d-1}}\right)\left( G^{\circ}_{u c_\beta}+\frac{\msc(z_t)  G^{\circ}_{u b_\beta}}{\sqrt{d-1}}\right).
\end{split}\end{align*}

For \eqref{e:aterm}, we replace $(G_{c_\al c_\al}^{(b_\al)}-Q_t)$ by $(G_{c_\al c_\al}^{(b_\al)}-Y_t)$. The same as in \eqref{e:a_error}, by \eqref{e:bPi2}, the error is bounded by $\OO( N^{-\fb/2}\bE[ \Psi_p])$. After such replacement, we get  
\begin{align}\label{e:getI2}
   I_2:=\sum_{\al\in \qq{\mu}}\frac{ \msc^{2\ell}(z_t)}{(d-1)^{\ell+1}Z_{\cF^+}}\sum_{\bfi^+}\bE\left[I(\cF^+,\cG)\bm1(\cG\in \Omega) (G_{c_\al c_\al}^{(b_\al)}-Q_t)(Q_t-Y_t)^{p-2}\cU_\al\right],
\end{align}
 which is an $\OO(1)$-weighted sum of terms in the form \eqref{e:case2}.

Next we estimate \eqref{e:aterm} using the following more general result:
Consider  $h\geq 0$, $\bmh=[h_{jk}]_{1\leq j\leq p-1, 0\leq k\leq 2}$ satisfying \eqref{e:defr} and  
$\overline R_{\bfi^+}=\overline R_h \prod_{j=1}^{p-1}\overline{\cW}_j \in \Adm(h,\bmh,\cF^+,\cG)$ (as defined in \eqref{def:pgen}). Suppose for some $1\leq {j'}\leq p-1$, $h_{{j'}0}=2$, and \begin{align}\label{e:Gsw_single_s}
    \overline\cW_{j'}'=\frac{\{1-\del_1 Y_\ell, -t\del_2 Y_\ell\}}{Nd}\sum_{u_{j'}\sim v_{j'}\in \qq{N}}G_{sw}^\circ \overline R_{h_{{j'}1}-1 h_{{j'}2}}',\quad s\in \cK^+, \quad w\in \{u_{j'}, v_{j'}\},
\end{align}
where $\overline R_{h_{{j'}1}-1 h_{{j'}2}}'$ contains $h_{{j'}1}-1$ factors of the form $ G^{\circ}_{su_{j'}},  G^{\circ}_{sv_{j'}}$, $h_{{j'}2}$ factors of the form $L_{su_{j'}}, L_{sv_{j'}}$, with $s\in \cK^+$ and an arbitrary number of factors $G_{u_{j'}v_{j'}}, 1/G_{u_{j'}u_{j'}}$. Furthermore, we assume that $\overline R_h \prod_{j\neq j'}\overline \cW_{j}'$ (as recalled from \eqref{e:Rhath}) and $\overline R'_{h_{{j'}1}-1 h_{{j'}2}}$ do not contain any index $s'\in \cK^+$ that belongs to the same connected component as $s$ in $\cF^+$. Then we claim that
\begin{align}\begin{split}\label{e:huanGS}
       &\phantom{{}={}}\frac{(d-1)^{(6h+3\sum_{j=1}^{p-1} h_{j1})\ell}}{Z_{\cF^+}}\sum_{\bfi^+}  \bE\left[I(\cF^+,\cG)\bm1(\cG\in \Omega) \overline R_{\bfi^+}\right]=\OO(N^{-\fb} \bE[ \Psi_p]).
   \end{split} \end{align}

Each term in \eqref{e:abterm} is in the form of \eqref{e:Gsw_single_s}, thus \eqref{e:huanGS} implies
\begin{align*}
    \eqref{e:abterm}\lesssim \sum_{\al\neq \beta\in \qq{\mu}}\frac{|L_{l_\al l_\beta}|}{(d-1)^\ell}\OO(N^{-\fb}\bE[ \Psi_p])=\OO( N^{-\fb/2}\bE[ \Psi_p]).
\end{align*}

We now prove \eqref{e:huanGS}. We assume that in \eqref{e:Gsw_single_s} $s=c_\al$ for some $\al\in \qq{\mu}$; other cases can be proven in the same way, so we omit their proofs. 
    We let  the forest $\widehat \cF$ and the indicator function $I_{c_\al}$ be as in \eqref{e:Idecompose}. There are two cases, either $h_{{j'}1}+h_{{j'}2}\geq 4$ or $h_{{j'}1}+h_{{j'}2}=2$.

    \noindent \textbf{Case 1.} If $h_{{j'}1}+h_{{j'}2}\geq 4$, we have 
\begin{align}\begin{split}\label{e:huanGcb0}
     &\phantom{{}={}}\frac{(d-1)^{3\ell \sum_{j=1}^{p-1}h_{j1}}\bm1(\cG\in \Omega)}{ Z_{\cF^+}}\left|\sum_{b_\al, c_\al\in\qq{N}} I(\cF^+,\cG)\prod_{j=1}^{p-1}\overline \cW_j\right|\\
     &\lesssim 
     \frac{(d-1)^{3\ell\sum_{j=1}^{p-1}h_{j1}}\bm1(\cG\in \Omega)I(\widehat \cF,\cG)}{Z_{\cF^+}}\left|\sum_{b_\al\sim c_\al}G^\circ_{c_\al w}I_{c_\al} \right|\frac{\Upsilon}{Nd}\sum_{u_{j'}\sim v_{j'}\in \qq{N}} |\overline R'_{h_{{j'}1}-1, h_{{j'}2}}|\prod_{j\neq j'}^{p-1}|\overline \cW_{j}|\\
      &\lesssim  \frac{\bm1(\cG\in \Omega)I(\widehat \cF,\cG)}{Z_{\cF^+}}\frac{Nd}{N^{1-3\fc/2}} (d-1)^{8\ell} \Upsilon \Phi (|Q_t-Y_t|+(d-1)^{8\ell} \Upsilon \Phi )^{p-2}\\
      &\lesssim   \frac{\bm1(\cG\in \Omega)}{N^{1-3\fc/2}Z_{\cF^+}}(d-1)^{8\ell} \Upsilon \Phi (|Q_t-Y_t|+(d-1)^{8\ell} \Upsilon \Phi )^{p-2}\sum_{b_\al\sim c_\al} I_{c_\al} I(\widehat\cF,\cG) \\
      &\lesssim \frac{1 } {N^{1-3\fc/2} Z_{\cF^+}}\sum_{b_\al, c_\al\in\qq{N}} \bm1(\cG\in \Omega)I(\cF^+,\cG)(|Q_t-Y_t|+(d-1)^{8\ell} \Upsilon \Phi )^{p-1}.
\end{split}\end{align}
The first inequality follows from the decomposition \eqref{e:Gsw_single_s}. For the second inequality,  we used \eqref{e:single_index_sum}, then \eqref{e:1_Sbound} to bound the summation of $\overline R'_{h_{{j'}1}-1 h_{{j'}2}}$ and $\overline \cW_{j}$.
The last two statements follow from a rearrangement of terms. Thanks to \eqref{e:huanGcb0},  we can bound the left-hand side of \eqref{e:huanGS} as
\begin{align*}
   &\phantom{{}={}}   \frac{(d-1)^{(6h+3\sum_{j=1}^{p-1} h_{j1})\ell}}{Z_{\cF^+}}\sum_{ \bfi^+}\bE\left[I(\cF^+,\cG)\bm1(\cG\in \Omega)|\overline R_{\bfi^+}|\right]\\
   &\lesssim\frac{(d-1)^{6h\ell} N^{-h\fb}}{ N^{1-3\fc/2} Z_{ \cF^+}}\sum_{\bfi^+}\bE\left[\bm1(\cG\in \Omega)I(\cF^+,\cG) (|Q_t-Y_t|+(d-1)^{8\ell} \Upsilon \Phi )^{p-1}\right]\lesssim N^{-\fb}\bE[ \Psi_p].
\end{align*}


 

\textbf{Case 2.}  If $h_{{j'}1}+h_{{j'}2}=2$, then
    \begin{align*}\begin{split}
    &\overline R'_{h_{{j'}1}-1, h_{{j'}2}}
    \in \{ G^{\circ}_{sw'}, L_{sw'}, (\Av G^{\circ})_{o'w'},(\Av L)_{o'w'} \}_{w'\in\{u_{j'}, v_{j'}\},s\in \cK^+}\times U^{(u_{j'},v_{j'})},
    \end{split}\end{align*}
    where $U^{(u_{j'},v_{j'})}$ is a product of $G_{u_{j'} v_{j'}}, 1/G_{u_{j'} u_{j'}}$. 
   Without loss of generality, we  give the prove assuming $\overline R'_{h_{{j'}1}-1, h_{{j'}2}}$ is of the form $ G^{\circ}_{s'w'} U^{(u_{j'},v_{j'})}$ with $s'\in \{b,c\}\in(\cC^\circ)^+$, the other cases can be proven in the same way, so we omit their proofs. Let $\widehat \cF=\{\{b,c\}, \{b_\al, c_\al\}\}$ consists of two unused core edges, and define
\begin{align}\begin{split}\label{e:Rsboundal}
 V^{(b,c)}&:=(Nd) \times \frac{(d-1)^{(6h+3\sum_{j\neq j'} h_{j1})\ell}}{Z_{\cF^+}}\sum_{\bfi/\{b,c,b_\al, b_\beta\}}\frac{ I(\cF^+,\cG)}{Z_{\cF^+}}\overline R_{h}\prod_{j\neq j'}\overline \cW_{j}\\
  &\lesssim (Nd) \times (d-1)^{6h\ell} N^{-h\fb}(|Q_t-Y_t|+(d-1)^{8\ell}\Upsilon \Phi)^{p-2}\sum_{\bfi/\{b,c,b_\al, b_\beta\}}\frac{ I(\cF^+,\cG)}{Z_{\cF^+}}\\
  &\lesssim \frac{(|Q_t-Y_t|+(d-1)^{8\ell}\Upsilon \Phi)^{p-2}}{N}I(\widehat \cF,\cG),
\end{split}\end{align}
where the first inequality follows from \eqref{e:Bsmall} and \eqref{e:1_Sbound}, and the second inequality follows as $I(\cF^+,\cG)$ contains the factor $I(\wh\cF,\cG)$.     By plugging in \eqref{e:Rsboundal} and noticing $h_{j'1}=2$, we can rewrite the left-hand side of \eqref{e:huanGS} as  
    \begin{align}\begin{split}\label{e:yigeJal}
     \eqref{e:huanGS} &=\frac{ (d-1)^{6\ell}}{Nd} \sum_{b,c, b_\al, c_\al\in\qq{N}}\bE\left[ \bm1(\cG\in \Omega)I(\wh \cF,\cG)V^{(b,c)}\overline \cW_{j'}\right].
    \end{split}\end{align}
For $\cG\in \Omega$, the summation in \eqref{e:yigeJal} can be bounded by \eqref{e:single_index_term2}
\begin{align}\begin{split}\label{e:Wjsum}
 &\phantom{{}={}}\frac{ (d-1)^{6\ell}}{Nd} \sum_{b,c, b_\al, c_\al\in\qq{N}}I(\wh \cF,\cG)V^{(b,c)}\overline \cW_{j'}\\
 &=
  \frac{\{1-\del_1 Y_\ell, \del_2 Y_\ell\}(d-1)^{6\ell}}{(Nd)^2}\sum_{b,c, b_\al, c_\al\in\qq{N}}I(\widehat \cF,\cG)\sum_{u_{j'}\sim v_{j'}}V^{(b,c)} G^{\circ}_{c_\al w} G^{\circ}_{s'w'} U^{(u_{j'},v_{j'})} \\
  &\lesssim \frac{ \Upsilon (d-1)^{6\ell} N^{3\fc/2+\fo}\sqrt{\sum_{b\sim c} |V^{(b,c)}|^2}}{N^{3/2}\eta}\\
  &\lesssim \frac{ \Upsilon (d-1)^{6\ell} N^{3\fc/2+\fo}(|Q_t-Y_t|+(d-1)^{8\ell}\Upsilon \Phi)^{p-2}}{N^{2}\eta}\lesssim  N^{-\fb}\Psi_p,
\end{split}\end{align}
where in the last line we used \eqref{e:Rsboundal}.
The claim \eqref{e:huanGS} follows.

 \item   The terms from replacing $\widehat{B}_0$ by $\msc^{2\ell}(z_t)(d-1)^{-(\ell+1)}\sum_{\al\neq \beta} G_{c_\al c_\beta}^{(b_\al b_\beta)}$ are in the following form: 
for $h'\geq 0$, and $\bmh=[h_{jk}]_{1\leq j\leq p-1, 0\leq k\leq 2}$ satisfying \eqref{e:defr},
\begin{align}
\label{e:fhatRi02}\frac{1}{(d-1)^{2\ell}} \sum_{\al\neq \beta\in \sfA_i}  (d-1)^\ell \times \sum_{\bfi^+}  \frac{(d-1)^{(3h'+3\sum_{j=1}^{p-1} h_{j1})\ell}}{Z_{\cF^+}}\bE\left[I(\cF^+,\cG)\bm1(\cG\in \Omega) G_{c_\al c_\beta}^{(b_\al b_\beta)}\widehat R_{\bfi^+}\right],
\end{align}
where $\widehat R_{\bfi^+}\in \Adm(h',\bmh,\cF^+,\cG)$ is as in \eqref{e:Rhath}. 

In \eqref{e:fhatRi02}, if $\{b_\al, c_\al\}$  do not appear in $\widehat R_{\bfi^+}$, then thanks to \eqref{e:huanG} (taking $(h,\bmh)$ as $(h',\bmh)$ and $f_1\geq 1$),
    \begin{align}\begin{split}\label{e:fhatRiGcc} 
          &\phantom{{}={}}(d-1)^\ell\frac{(d-1)^{3(h'+\sum_{j=1}^{p-1} h_{j1})\ell}}{Z_{\cF^+}}\sum_{ \bfi^+}\bE\left[I(\cF^+,\cG)\bm1(\cG\in \Omega)G_{c_\al c_\beta}^{(b_\al b_\beta)}\widehat R_{\bfi^+}\right]\\
          &=(d-1)^\ell\frac{(d-1)^{3(h'+\sum_{j=1}^{p-1} h_{j1})\ell}}{(d-1)^{f_1/2}Z_{\cF^+}}\sum_{ \bfi^+}\bE\left[I(\cF^+,\cG)\bm1(\cG\in \Omega)R'_{\bfi^+}\right]+\OO(N^{-\fb/4} \bE[\Psi_p])\\
          &=\frac{1}{(d-1)^{f_1/2}}\frac{(d-1)^{(6(h'+f_1)+\sum_{j=1}^{p-1} h_{j1})\ell}}{(d-1)^{\fq^+\ell/2}Z_{\cF^+}}\sum_{ \bfi^+}\bE\left[I(\cF^+,\cG)\bm1(\cG\in \Omega)R'_{\bfi^+}\right]+\OO(N^{-\fb/4} \bE[\Psi_p]),
          \end{split}\end{align}
      where $R'_{\bfi^+}\in \Adm(h'+f_1,\bmh,\cF^+,\cG)$ with $f_1\geq 1$, and $\fq^+=2(3h'+6f_1-1)\geq 10$. 
      Moreover, from \eqref{e:final_replace1}, $R_{\bfi^+}'$ contains $(G_{c_\al c_\al}^{(b_\al)}-Q_t)$.
We can further change $(G_{c_\al c_\al}^{(b_\al)}-Q_t)$ to $(G_{c_\al c_\al}^{(b_\al)}-Y_t)$ as in \eqref{e:changeQtoY}. Then \eqref{e:fhatRiGcc} leads to \eqref{e:case3} with $r^+= h'+f_1\geq h'+1\geq 1$, $\bmr^+=\bmh$ satisfying \eqref{e:1_expand} or \eqref{e:feasible}, and  $\fq^+\geq 10$.

If $\{b_\beta, c_\beta\}$ do not appear in $\widehat R_{\bfi^+}$,  the same statement holds. 
In the remaining cases, both $\{b_\al, c_\al\}$ and $\{b_\beta, c_\beta\}$ appear in $\widehat R_{\bfi^+}$. These terms in  \eqref{e:fhatRi02} form an $\OO(1)$-weighted sum of terms in the form 
\begin{align}\label{e:fhatRiGcc0} 
          &\phantom{{}={}}(d-1)^\ell\frac{(d-1)^{3(h'+\sum_{j=1}^{p-1} h_{j1})\ell}}{Z_{\cF^+}}\sum_{ \bfi^+}\bE\left[I(\cF^+,\cG)\bm1(\cG\in \Omega)G_{c_\al c_\beta}^{(b_\al b_\beta)}\widehat R_{\bfi^+}\right].
\end{align}
If $\bmh$ satisfies \eqref{e:feasible} or $h'\geq 1$, then thanks to \eqref{e:refined_bound2}, we have $|\eqref{e:fhatRiGcc0}|\lesssim N^{-\fb/2} \bE[\Psi_p]$.
Otherwise, $h'=0$ and $\bmh$ satisfies \eqref{e:1_expand}. We do a similar expansion as from \eqref{e:abterm}; in this case, terms in \eqref{e:fhatRi02} are of the form
\begin{align}\label{e:errorLL}
    \sum_{\al\neq \beta\in \qq{\mu}}\frac{2(p+1)\msc^{2\ell}(z_t)}{(d-1)^{\ell+2}Z_{\cF^+}}\sum_{\bfi^+}\bE\left[L_{l_\al l_\beta} I(\cF^+,\cG)\bm1(\cG\in \Omega) G_{c_\al c_\beta}^{(b_\al b_\beta)}(Q_t-Y_t)^{p-2}\cU_{\al\beta}\right].
\end{align}
We notice that $|L_{l_\al l_\beta}|\lesssim 1$. Each term in \eqref{e:errorLL} is in the following form: for $\al\neq \beta\in \qq{\mu}, x\in \{b_\al, c_\al\}, x'\in \{b_\beta, c_\beta\}, w,w'\in\{u,v\}$,
 \begin{align}\label{e:sumtt0}
  \sum_{\bfi^+}  \bE\left[\frac{\{1-\del_1 Y_\ell, \del_2 Y_\ell\}}{(Nd)Z_{\cF^+} }I(\cF^+,\cG)\bm1(\cG,\tcG\in \Omega)\sum_{u\sim v} G_{c_\al c_\beta}^{(b_\al b_\beta)} G^{\circ}_{x w} G^{\circ}_{x' w'}(Q_t-Y_t)^{p-2}\right].
    \end{align}
Thanks to \eqref{e:refined_bound2_special}, the expression above can be bounded by $\OO(N^\fo\bE[\Psi_p])$. We conclude
\begin{align}\label{e:Triangle_cerror}
     &\phantom{{}={}}\eqref{e:errorLL}\lesssim \sum_{\al\neq \beta\in \qq{\mu}}\frac{1}{(d-1)^\ell}\OO(N^\fo\bE[\Psi_p])\lesssim (d-1)^{2\ell}\bE[\Psi_p].
\end{align}

\item In the remaining cases, $\widehat B_0$ is replaced by  $\cU_0$  as specified in the decomposition  \eqref{e:first_W0}. Recall that $\cU_0$ is an $\OO(1)$-weighted sum of terms in the form $(d-1)^{3(h-1)\ell}R'_{h}$, where $R'_{h}$ is an $S$-product term of order $h$ (see \Cref{d:S-product}) with $h\geq 2$, and contains $G_{c_\al c_\al}^{(b_\al)}-Q_t$ or $G_{c_\al c_\beta}^{(b_\al b_\beta)}$. 
We get an $\OO(1)$-weighted sum of terms in the following form: for $h'\geq 0$, $h\geq 2$ and $\bmh=[h_{jk}]_{1\leq j\leq p-1, 0\leq k\leq 2}$ satisfying \eqref{e:defr},  
\begin{align}\label{e:fhatRi03copy}
  \sum_{\bfi^+} \frac{(d-1)^{(3(h+h'-1)+3\sum_{j=1}^{p-1} h_{j1})\ell}}{Z_{\cF^+}} \bE\left[I(\cF^+,\cG)\bm1(\cG\in \Omega)R'_{h}\widehat R_{\bfi^+}\right],
\end{align}
where where $\widehat R_{\bfi^+}\in \Adm(h',\bmh,\cF^+,\cG)$ is as in \eqref{e:Rhath}.

If $R'_{ h}$ contains a factor $G_{c_\al c_\beta}^{(b_\al b_\beta)}$, since $h\geq 2$, then \eqref{e:refined_bound2} implies that $|\eqref{e:fhatRi03copy}|\lesssim N^{-\fb/2}\bE[\Psi_p]$. 
Otherwise, $R'_{h}$ contains a factor $G_{c_\al c_\al}^{(b_\al)}-Q_t$. In this case, 
\begin{align}\label{e:last_case}
\eqref{e:fhatRi03copy}=\frac{(d-1)^{(6(h+h'-1)+3\sum_{j=1}^{p-1} h_{j1})\ell}}{(d-1)^{\fq^+\ell/2}Z_{\cF^+}}\sum_{\bfi^+}\bE\left[I(\cF^+,\cG)\bm1(\cG\in \Omega)  R'_{\bfi^+}\right], \;\fq^+=6(h+h'-1).
\end{align}
The same as in \eqref{e:changeQtoY}, we can change $(G_{c_\al c_\al}^{(b_\al)}-Q_t)$ in \eqref{e:last_case} to $(G_{c_\al c_\al}^{(b_\al)}-Y_t)$, and the error is bounded by $\OO(N^{-\fb/4}\bE[\Psi_p])$. \eqref{e:last_case} leads to \eqref{e:case3} with $r^+=h'+h-1$, $\bmr^+=\bmh$ satisfying \eqref{e:1_expand} or \eqref{e:feasible}, and $\fq^+=6(h'+h-1)\geq 6$.
 \end{enumerate}

\begin{remark}\label{r:QtoY}
   We note that, to obtain the final expressions in \eqref{e:case1}, \eqref{e:case2}, and \eqref{e:case3}, the last step of the proof of \Cref{p:iteration} requires replacing a copy of $Q_t$ with $Y_t$. The errors introduced by this replacement are bounded by the first term in $\Psi_p$ as defined in \eqref{e:defPhi}. All other errors are bounded by the second term in $\Psi_p$.
\end{remark}
\subsection{Iteration General Case: Proof of \Cref{p:general}}\label{s:general}

We denote $R_\bfi$ in \Cref{p:general}, as $R_\bfi=R_\bfi'\prod_{j=1}^{p-1}\cW_j\in \Adm(r,\bmr,\cF,\cG)$ (as in \Cref{def:pgen}), and $R_\bfi'$ contains $r$ factors in the form of \eqref{e:defcE1}. In the following lemma, we demonstrate that $\wt\cW_j$ is close to $\wh \cW_j$, which depends only on the Green's function of the original graph $\cG$. Furthermore, we provide a decomposition of $\wh \cW_j$, analogous to \eqref{e:diyi}, \eqref{e:dier}, and \eqref{e:disan0}. The proof is deferred to \Cref{sec:iterationproofs}.

\begin{lemma}
    \label{p:tS-Sdiff}
 $(d-1)^{3r_{j1}\ell}\wt \cW_j=(d-1)^{3r_{j1}\ell}\widehat \cW_j+\OO(((d-1)^{8\ell}\Phi)^2)$, where
 \begin{enumerate}
     \item If $r_{j0}=1$, then $r_{j1}=r_{j2}=0$, $\cW_j=Y_t-Q_t$, and $\wh \cW_j=(Y_t-Q_t)+\wh \cU_j$. Here $\wh \cU_j$ is an $\OO(1)$-weighted sum of terms in one of the two following forms:
\begin{enumerate} 
\item For $h_j'\geq 0$,
\begin{align}\label{e:diyicopy}
     (d-1)^{3h_j'\ell}R_{h_j'}\times \cW_j', \quad \cW_j':= \frac{(d-1)^\ell (1-\del_1 Y_\ell)}{N}.
\end{align} 
 Here $R_{h_j'}$ is a W-product term (see \Cref{d:W-product}). We set $(h_{j0}, h_{j1}, h_{j2})=(3,0,0)$ and note that $\cW_j'$ is in the form of \eqref{e:Scase3}.
 
 \item 
For $h_j'\geq 0$, $h_{j1}+h_{j2}\geq 2$ even,
     \begin{align}\label{e:hatFdef}
     &(d-1)^{3h_j'\ell}R_{h_j'}\times (d-1)^{3h_{j1}\ell}\cW_j',\quad \cW_j':= \frac{\{(1-\del_1 Y_\ell), -t\del_2 Y_\ell\} }{Nd}\sum_{u_j\sim v_j\in \qq{N}} R'_{h_{j1},h_{j2}}. 
\end{align} 
   Here $R_{h_j'}$ is a W-product term (see \Cref{d:W-product}). $R'_{h_{j1},h_{j2}}$ contains $h_{j1}$ factors of the form $ G^{\circ}_{xu_j},  G^{\circ}_{xv_j}$, $h_{j2}$ factors of the form $L_{xu_j}, L_{xv_j}$, with $x\in\cK^+$ and some number (that we do not need to track) of factors $G_{u_jv_j}, 1/G_{u_ju_j}$.
   We set $h_{j0}=2$, and note that $\cW_j'$ is in the form of \eqref{e:Scase2}.  
  \end{enumerate}

    \item If $r_{j0}=2$, then $(d-1)^{3r_{j1}\ell}\wh \cW_j$ is an $\OO(1)$-weighted sum of terms of one of the two following forms: 
    \begin{enumerate}
    \item 
    For $h_j'\geq r_{j_1}$,
\begin{align}\label{e:erbufen}
     (d-1)^{6h_j'\ell}R_{h_j'}\times \cW_j', \quad \cW_j':= \frac{(d-1)^\ell (1-\del_1 Y_\ell)}{N},
\end{align} 
 where $R_{h_j'}$ is a W-product term (see \Cref{d:W-product}). We set $(h_{j0}, h_{j1}, h_{j2})=(3,0,0)$ and note that $\cW_j'$ is in the form of \eqref{e:Scase3}.
   
\item 
For $h'_{j}\geq 0$, $h_{j2}\geq r_{j_2}$ and $h_{j1}+h_{j2}\geq r_{j1}+r_{j2}$ even, $g_{j1}\leq h_{j1}, g_{j2}\leq h_{j2}$
    \begin{align}\begin{split}\label{e:yibufen}
        & (d-1)^{6h_j'\ell}R_{h_j'}\times \frac{(d-1)^{3h_{j1}\ell}}{(d-1)^{(g_{j1}+g_{j2})\ell/2}}\sum_{\bm\theta_j} \cW_j'(\bm\theta_j),\\
        &\cW_j'(\bm\theta_j):=\frac{\{1-\del_1 Y_\ell, -t\del_2 Y_\ell\}}{Nd}\sum_{u_j\sim v_j\in \qq{N}}R'_{h_{j1}, h_{j2}},\\
        &R'_{h_{j1}, h_{j2}}=R''_{h_{j1}-g_{j1}, h_{j2}-g_{j2}}\prod_{m=1}^{g_{j1}+g_{j2}} D_{\theta_m},
    \end{split}\end{align}
where $R''_{h_{j1}-g_{j1}, h_{j2}-g_{j2}}$ contains $h_{j1}-g_{j1}$ factors of the following form: $\{ (\Av G^{\circ})_{o'w},   G^{\circ}_{xw}\}$, and 
$h_{j2}-g_{j2}$ factors of the form 
$\{(\Av L)_{o'w},  L_{xw}\}$ with ${ x\in\cK^+, (i',o')\in \cC\setminus \cC^\circ, w\in \{u_j,v_j\}}$; and
\begin{align}\begin{split}\label{e:defD}
  &D_{\theta_m}\in \left\{   G^{\circ}_{b_{\theta_m} w}, G^{\circ}_{c_{\theta_m} w}\right\}_{w\in\{u_j,v_j\}}, \quad 1\leq m\leq g_{j1},\\
  &D_{\theta_m}\in \left\{  L_{b_{\theta_m} w},L_{c_{\theta_m} w}\right\}_{w\in\{u_j,v_j\}}, \quad g_{j1}+1\leq m\leq g_{j1}+g_{j2},
\end{split}\end{align}
the summation for $ \bm\theta_j=(\theta_1, \theta_2,\cdots, \theta_{g_{j1}+g_{j2}})$ is over each $\theta_m$ in one of the sets $\qq{\mu}$, $\sfA_i$ or $\qq{\mu}\setminus \sfA_i$; $R_{h_j'}$ is a product of $h_j'$ factors of the form $\{ G^{\circ}_{ss'}\}_{s,s'\in \cK^+}$. 
 We set $h_{j0}=2$, and note that $\cW_j'=\cW_j'(\bm\theta_j)$ is in the form of \eqref{e:Scase2}.  

 \end{enumerate}

     \item If $r_{j0}=3$, then $r_{j1}=r_{j2}=0$,  $\cW_j= (d-1)^\ell \{1-\del_1Y_\ell, -t\del_2 Y_\ell\}/N$, and $\widehat \cW_j:=\cW_j':=\cW_j$. We set $(h_{j0}, h_{j1}, h_{j2})=(3,0,0)$ and note that $\cW_j'$ is in the form of \eqref{e:Scase3}.
 \end{enumerate}
 For $\cW_j'(\bm\theta_j)$ in \eqref{e:yibufen}, let $f_j$ denote the number of distinct values for $\theta_1, \theta_2,\cdots, \theta_{g_{j1}+g_{j2}}$, then
    \begin{align}\label{e:S'bound2}
    \left|\frac{(d-1)^{(f_j+3h_{j1})\ell}}{(d-1)^{(g_{j1}+g_{j2})\ell/2}}\cW_j'(\bm\theta_j)\right|\lesssim (d-1)^{8\ell}\Upsilon\Phi.
    \end{align}
We observe that $\cW_j'$ from \eqref{e:hatFdef} is a special case of $\cW_j'(\bm\theta_j)$ from \eqref{e:yibufen}, obtaining by setting $\bm\theta_j=\emptyset$. Therefore, \eqref{e:S'bound2} also applies to 
$\cW_j'$
  in \eqref{e:hatFdef}.
\end{lemma}

\Cref{p:tS-Sdiff} gives a decomposition of $\prod_{j=1}^{p-1} (d-1)^{3r_{j1}\ell}\widehat \cW_j$ as an $\OO(1)$-weighted sum of terms in the following form: for $h'\geq 0$, $\bmh=[h_{jk}]_{1\leq j\leq p-1, 0\leq k\leq 2}$ satisfying \eqref{e:defr} and for each $1\leq j\leq p-1$ either $h_{j0}>r_{j0}$ or $h_{j0}=r_{j0}$ and $h_{j1}+h_{j2}\geq r_{j1}+r_{j2}, h_{j2}\geq r_{j2}$, 
\begin{align}\begin{split}\label{e:decomposeWj}
& (d-1)^{6h'\ell}R_{h'}\times \frac{1}{(d-1)^{k_4 \ell /2}}\sum_{\bm\theta}    (d-1)^{3\sum_{j=1}^{p-1} h_{j1}\ell}\prod_{j=1}^{p-1}\cW_j',
\end{split}\end{align}
where  
 $\cW_j'$ and $\bmh$ are as defined in \Cref{p:tS-Sdiff}; the term $R_{h'}$ with $h'=\sum_{j=1}^{p-1} h_j'$ includes the $R_{h_j'}$ factors in \eqref{e:diyicopy}, \eqref{e:hatFdef}, \eqref{e:erbufen} and \eqref{e:yibufen}; The array $\bm\theta$ represents the concatenation of all $\bm\theta_j$ terms from \eqref{e:yibufen}, with a total length of $k_4=\sum_{j=1}^{p-1}(g_{j1}+g_{j2})$.

In the following we prove \Cref{p:general} assuming $\fq=0$. The case with $\fq\geq 1$ follows from simply multiplying $(d-1)^{-\fq \ell/2}$. 
In order to use our various propositions from Section \ref{sec:expansions} and \Cref{p:tS-Sdiff} that allow us to reduce $(\wt G_{oo}^{(i)}-Y_t) \wt R'_\bfi \prod_{j=1}^{p-1}\wt \cW_j$ to terms of the unswitched graph $\cG$, we need to classify the factors of $R'_\bfi$ based on their dependence on $i,o$. Let 
\begin{align}\label{e:introB_j}
     B_0=(G_{oo}^{(i)}-Y_t),\quad R'_{\bfi}=\prod_{j=1}^{r}B_j, \quad R_\bfi=R_\bfi'\prod_{j=1}^{p-1}\cW_j,
\end{align}
where $B_j$ is of one of the following forms
\begin{align}\begin{split}\label{e:Blist}
&\{G_{cc}^{(b)}-Q_t\}_{(b,c)\in \cC^\circ},\quad \{G^\circ_{ss'}\}_{s,s'\in \{i,o\}},\quad
\{G_{oc}^{(ib)}, G_{ic}^{(b)}, G_{ib}, G_{ob}, G_{ob}^{(i)}\}_{(i,o)\neq (b,c)\in \cC^\circ},\\
&\{ G_{cc'}^{(bb')}, G_{bc'}^{(b')}, G_{bb'}, G_{cb'}\}_{(i,o)\neq (b,c), (b',c')\in \cC^\circ},\quad
\{G^\circ_{sw}\}_{ s\in \cK,w\in \cK\setminus\{i,o\}},\\
 &\{\Av G^{\circ})_{o'w} \}_{(i',o')\in \cC\setminus \cC^\circ, w\in \cK\setminus\{i,o\}},\quad  (Q_t-\msc(z_t)),\quad t(m_t-\msc(z_t)).
\end{split}\end{align}
We note that the term $\{\Av G^{\circ})_{o'w} \}_{(i',o')\in \cC\setminus \cC^\circ, w\in \cK}$ appears only in the expansions in \Cref{i5} and \Cref{i6} below. Specifically, the terms $\{\Av G^{\circ})_{o'w} \}_{(i',o')\in \cC\setminus \cC^\circ, w\in \{o,i\}}$ never appear.

We can write \eqref{e:higher_case3} (with $\fq=0$) as
\begin{align}\begin{split}\label{e:sreplace0}
    &\phantom{{}={}}\frac{(d-1)^{(6r+3\sum_{j=1}^{p-1} r_{j1})\ell}}{Z_{\cF^+}}\sum_{\bfi^+}\bE\left[I(\cF^+,\cG)\bm1(\cG,\tcG\in \Omega)(\wt G_{oo}^{(i)}-Y_t) \wt R_\bfi\right]\\
    &=\frac{(d-1)^{(6r+3\sum_{j=1}^{p-1} r_{j1})\ell}}{Z_{\cF^+}}\sum_{\bfi^+}\bE\left[I(\cF^+,\cG)\bm1(\cG,\tcG\in \Omega) (\wt G_{oo}^{(i)}-Y_t)\prod_{j=1}^r \wt B_j \prod_{j=1}^{p-1}  \wt \cW_j\right].
\end{split}\end{align}
In the following we discuss the terms $B_j$ and $\cW_j$  as in \eqref{e:introB_j} after the local
resampling.  In \Cref{i1}-\Cref{i7} below, for any $h\geq 0$, we denote $R_h$ a product of $h$ factors of the form 
\begin{align}\label{e:list_factor}
    (G_{cc}^{(b)}-Q_t),  G_{cc'}^{(bb')}, G_{bc'}^{(b')}, G_{bb'},G_{cb'},\quad  (Q_t-\msc(z_t)), t(m_t-\md(z_t)), \quad G^{\circ}_{ss'},\quad  (\Av G^{\circ})_{ow},
\end{align}
where $(b,c)\neq (b',c')\in \cK^+$,  $s,s'\in \cK^+$ and $w\in \cK\setminus\{i,o\}$.
They are terms from \Cref{lem:diaglem}, \Cref{lem:offdiagswitch}, \Cref{lem:generalQlemma} and \Cref{c:Qmchange}.
\begin{enumerate}
\item \label{i1}
By \Cref{lem:diaglem}, $\wt G_{oo}^{(i)}-Y_t=\widehat B_0+\cZ_0+\cE_0$, where $\wh B_0$ is given by
\begin{align}\label{e:fcase1W0}
\wh B_0= \frac{\msc^{2\ell}(z_t)} {(d-1)^{\ell+1}}\sum_{\al\in\sfA_i}(G^{(b_\alpha)}_{c_\alpha c_\alpha}-Q_t)+\frac{\msc^{2\ell}(z_t)} {(d-1)^{\ell+1}}\sum_{\al\neq \beta\in\sfA_i}G^{(b_\alpha b_\beta)}_{c_\alpha c_\beta}
   + \cU_0. 
\end{align}
Here, $\cU_0$ is an $\OO(1)$-weighted sum of terms of the form $(d-1)^{3(h-1)\ell}R_h$ with $h\geq 2$, and each term contains at least one factor of the form  $( G_{c_\al c_{\al}}^{(b_\al)}-Q_t)$ or $G_{c_\al c_{\beta}}^{(b_\al b_\beta)}$; $\cZ_0, \cE_0$ are $\cZ, \cE$ in the first statement of \Cref{lem:diaglem}.

    \item \label{i2}
For $B_j=(G_{oo}^{(i)}-Q_t)$  we can first rewrite $\wt B_j$ as
\begin{align*}
    \wt B_j=(\wt G_{oo}^{(i)}-\wt Q_t)=(\wt G_{oo}^{(i)}-Y_t)+(Y_t-\wt Q_t),
\end{align*}
then expand according to \Cref{lem:diaglem}. This gives that $\wt B_j=\wh B_j+\cZ_j+\cE_j$, where
\begin{align}\label{e:zhankai1}
\wh B_j&=\frac{\fc} {(d-1)^{\ell+1}}\sum_{\al\in\sfA_i}(G^{(b_\alpha)}_{c_\alpha c_\alpha}-Q_t)+\frac{\fc} {(d-1)^{\ell+1}}\sum_{\al\neq \beta\in\sfA_i}G^{(b_\alpha b_\beta)}_{c_\alpha c_\beta}+\cU_j.
\end{align}
Here $\cU_j$ is an $\OO(1)$-weighted sum of terms of the form $(d-1)^{3(h-1)\ell}R_h$ with $h\geq 2$; 
$\cZ_j, \cE_j$ are $\cZ, \cE+( (Y_t-Q_t) +
    (Q_t-\wt Q_t))$ in the first statement of \Cref{lem:diaglem}.

For $B_j\in \{G^\circ_{ss'}\}_{s,s'\in \{i,o\}}$, we have a similar expansion $\wt B_j=\wh B_j+\cZ_j+\cE_j$, where
\begin{align}\begin{split}\label{e:zhankai2}
   \wh B_j
    &=\frac{1}{(d-1)^\ell}\sum_{\al\in \qq{\mu}}\fc_1(\bm1(\al \in \sfA_i))(G^{(b_\alpha)}_{c_\alpha c_\alpha}-Q_t)\\
    &+\frac{1}{(d-1)^\ell}\sum_{\al\neq \beta\in\qq{\mu}}\fc_2(\bm1(\al \in \sfA_i), \bm1(\beta\in \sfA_i)) G^{(b_\alpha b_\beta)}_{c_\alpha c_\beta}
  +\cU_j.
\end{split}\end{align}
Here $\cU_j$ is an $\OO(1)$-weighted sum of terms of the form $(Q_t-\msc)$ or $(d-1)^{3(h-1)\ell}R_h$ with $h\geq 2$; $\cZ_j, \cE_j$ are $\cZ, \cE$ in the second statement of \Cref{lem:diaglem}.

\item \label{i3}
For $B_j\in \{G_{oc}^{(ib)}, G_{ic}^{(b)}\}_{(i,o)\neq (b,c)\in \cC^\circ}$, by \Cref{lem:offdiagswitch}, we have $\wt B_j=\wh B_j+\cZ_j+\cE_{j}$, where
\begin{align}\label{e:zhankai3}
  \wh B_j=(d-1)^{-\ell/2}\sum_{\al\in \qq{\mu}}\fc_1(\bm1(\al\in \sfA_i)) G_{c_\al c}^{(b_\al b)}+\cU_j.
\end{align}
For $B_j\in \{ G_{ib}, G_{ob}, G_{ob}^{(i)}\}_{(i,o)\neq (b,c)\in \cC^\circ}$, we have a similar expansion $\wt B_j=\wh B_j+\cZ_j+\cE_{j}$, where
\begin{align}\label{e:zhankai4}
  \wh B_j=(d-1)^{-\ell/2}\sum_{\al\in \qq{\mu}}\fc_1(\bm1(\al \in \sfA_i)) G_{c_\al b}^{(b_\al )}+\cU_j.
\end{align}
In both cases, $\cU_j$ is an $\OO(1)$-weighted sum of terms of the form $(d-1)^{3(h-1)\ell}R_h$ with $h\geq 2$; $\cZ_j, \cE_j$ are $\cZ, \cE$ in the first and second statements of \Cref{lem:offdiagswitch}.

\item \label{i4}
For $B_j\in \{ G_{cc}^{(b)}, G_{cc'}^{(bb')}, G_{bc'}^{(b')}, G_{bb'}\}_{(b,c)\neq  (b',c')\in \cC^\circ}$, by \Cref{lem:offdiagswitch}, we have  $\widetilde B_j=\widehat B_j+\cE_j$, where 
 $   \widehat B_j=\cU_j=B_j$,
and $\cE_j$ is $\cE$ in the third statements of \Cref{lem:offdiagswitch}.

\item \label{i5}
For $B_j\in\{ G^\circ_{sw}\}_{s\in \cK, w\in \cK\setminus\{i,o\}}$, by the third statement in \Cref{lem:generalQlemma}, we have $\wt B_j=\widehat B_j+\cE_j $, where 
\begin{align}\begin{split}\label{e:tG-Gexp1c}
\widehat B_j&=\cU_j, \quad s\in \cK\setminus\{i,o\},\\
\widehat B_j&= \frac{\sum_{\sfJ\in\{b,c\}}\sum_{\al\in \qq{\mu}} \fc(\sfJ, \bm1(\al\in \sfA_i))  G^{\circ}_{\sfJ_\al w}}{(d-1)^{\ell/2}}+\cU_j,\quad  s\in \{i,o\}.
\end{split}\end{align}
Here  $\cU_j$ is an $\OO(1)$-weighted sum of terms of the form $(d-1)^{3(h-1)\ell}R_h$ with $h\geq 1$; and $|\cE_j|\lesssim N^{-2}$.

\item\label{i6} For 
$B_j\in\{(\Av G^{\circ})_{o'w} \}_{(i',o')\in \cC\setminus \cC^\circ, w\in \cK\setminus\{i,o\}}$, by \Cref{lem:generalQlemma}, we have $\wt B_j=\widehat B_j+\cE_j $, where
$
    \widehat B_j=\cU_j,
$
and  $\cU_j$ is an $\OO(1)$-weighted sum of terms of the form $(d-1)^{3(h-1)\ell}R_h$ with $h\geq 1$;  and $|\cE_j|\lesssim N^{-2}$.

\item \label{i7}
For $B_j\in \{(Q_t-\msc(z_t)), t(m_t-\msc(z_t))\}$, we have by the fourth statement in \Cref{c:Qmchange},
$\wt B_j=\widehat B_j+\cE_j$, where 
$
\widehat B_j=\cU_j=B_j
$,
and $|\cE_j|\lesssim (d-1)^{6\ell}N^{\fo}\Phi$.

\item \label{i8} For $\cW_j$, by \Cref{p:tS-Sdiff} we have 
$\wt \cW_j=\widehat \cW_j+\widehat\cE_j$, where  $|\widehat\cE_j|\lesssim ((d-1)^{8\ell}\Phi)^2$ and $\widehat \cW_j$ are as given in \Cref{p:tS-Sdiff}. 
\end{enumerate}

Analogous to \Cref{l:fyibu},  we have the following lemma, which  states that we can replace $(\wt G_{oo}^{(i)}-Y_t)\prod_{j=1}^r \wt B_j\prod_{j=1}^p\wt \cW_j$ in \eqref{e:sreplace0} with $\prod_{j=0}^r \widehat B_j\prod_{j=1}^p\widehat \cW_j$, with the overall error from this substitution being negligible. The proof is deferred to \Cref{sec:iterationproofs}.
\begin{lemma}\label{l:yibu}
Adopt the notation and assumptions in \Cref{p:general}, we can rewrite \eqref{e:sreplace0} as
\begin{align}
\label{e:hatRi}
  \eqref{e:sreplace0}=  \frac{(d-1)^{(6r+3\sum_{j=1}^{p-1} r_{j1})\ell}}{Z_{\cF^+}}\sum_{\bfi^+}\bE\left[I(\cF^+ ,\cG)\bm1(\cG\in \Omega) \prod_{j=0}^r \widehat B_j\prod_{j=1}^p\widehat \cW_j\right]+\OO(N^{-\fb/2}\bE[ \Psi_p]).
\end{align}
\end{lemma}

We can further decompose $(d-1)^{(6r+3\sum_{j=1}^{p-1} r_{j1})\ell}\prod_{j=0}^{r} \widehat B_j\prod_{j=1}^{p-1}\widehat \cW_j^{p-1}$ in \eqref{e:hatRi}. The decomposition of $(d-1)^{3\sum_{j=1}^{p-1} r_{j1}\ell}\prod_{j=1}^{p-1}\widehat \cW_j^{p-1}$ is given in \eqref{e:decomposeWj}. The discussions in \Cref{i1}--\Cref{i7} provide the decomposition of $(d-1)^{6r\ell}\prod_{j=0}^{r} \widehat B_j$ as an $\OO(1)$-weighted sum of terms of the following form:
\begin{align}\begin{split}\label{e:oneterm0}
&\frac{(d-1)^{6h''\ell}}{(d-1)^{(k_1 +(k_2+k_3)/2)\ell}} \sum_{\bm\al, \bm\beta, \bm\gamma}  \prod_{m=1}^{k_1}  A_{\al_  m}  \prod_{  m=1}^{k_2}  B_{\beta_m} 
\prod_{m=1}^{k_3/2}  C_{\gamma_{2m-1} \gamma_{2m}}  R_{h''+1-k_1-k_2-k_3/2}.
\end{split}\end{align}
Here the summation for $\bm\al, \bm\beta, \bm\gamma$ runs through each $\al_m, \beta_m, \gamma_m$ in one of the sets $\qq{\mu}$, $\sfA_i$ or $\qq{\mu}\setminus \sfA_i$.
The factors in \eqref{e:oneterm0} are defined as follows:
\begin{align}\begin{split}\label{e:defABC}
  &A_\al=   ( G_{c_\al c_\al}^{(b_\al)}-Q_t), \\
  &B_\beta\in \left\{ G_{c_\beta c}^{(b_\beta b)},  G_{c_\beta b}^{(b_\beta)},   G^{\circ}_{b_\beta s}, G^{\circ}_{c_\beta s}\right\},\quad (b,c)\in \cC^\circ,\quad s\in \cK\setminus\{i,o\},\\
   &C_{\gamma \gamma'}=  G_{c_\gamma c_{\gamma'}}^{(b_\gamma b_{\gamma'})}.
\end{split}\end{align}
 Here $A_\al$ originates from $G_{c_\al c_\al}^{(b_\al)}-Q_t$ in \eqref{e:fcase1W0}, \eqref{e:zhankai1}, \eqref{e:zhankai2}. $B_\beta$ comes from $G_{c_\al c}^{(b_\al b)}, G_{c_\al c}^{(b_\al)},  G^{\circ}_{\sfJ_\al s'}$ in \eqref{e:zhankai3}, \eqref{e:zhankai4},\eqref{e:tG-Gexp1c}.  $C_{\gamma \gamma'}$ are from $G_{c_\al  c_\al}^{(b_\al b_\beta)}$ in \eqref{e:fcase1W0}, \eqref{e:zhankai1}, \eqref{e:zhankai2}. $R_{h''+1-k_1-k_2-k_3/2}$ collects all $R_h$ factors from $\cU_j$ as in \Cref{i1}--\Cref{i7}, with a total count of ${h''+1-k_1-k_2-k_3/2}$. Thus, the summand in \eqref{e:oneterm0} consists of $h''+1\geq r+1$ factors in total. Finally the coefficient $(d-1)^{6h'\ell}$ arises from the following crucial observation: in the replacements outlined in \Cref{i1}--\Cref{i7}, each term $(d-1)^{6\ell}B_j$ is replaced by one of the terms $(d-1)^{6\ell}A_\al, (d-1)^{6\ell}A_\beta, (d-1)^{6\ell}C_{\gamma\gamma'}$ as in \eqref{e:defABC}, or a factor $R_h$ (as in \eqref{e:list_factor}) for $h\geq 1$ with coefficient at most $(d-1)^{6h\ell}$.

The decomposition of $(d-1)^{6r\ell}\prod_{j=0}^{r} \widehat B_j$ in \eqref{e:oneterm0}, the decompositon of $\prod_{j=1}^{p-1}(d-1)^{3r_{j1}\ell} \widehat \cW_j$ in \eqref{e:decomposeWj}, and the explicit expressions of $\cW_j'$ from \eqref{e:yibufen} and \eqref{e:defD} together mean that \eqref{e:hatRi} can be written as an $\OO(1)$-weighted sum of terms of the following form:
For $\widehat h\geq r$, and $\bmh=[h_{jk}]_{1\leq j\leq p-1, 0\leq k\leq 2}$ satisfying \eqref{e:defr} and for each $1\leq j\leq p-1$ either $h_{j0}>r_{j0}$ or $h_{j0}=r_{j0}$ and $h_{j1}+h_{j2}\geq r_{j1}+r_{j2}, h_{j2}\geq r_{j2}$, 
\begin{align}\begin{split}\label{e:oneterm}
&\frac{1}{(d-1)^{(k_1 +(k_2+k_3+k_4)/2)\ell}} \sum_{\bm\al, \bm\beta, \bm\gamma, \bm\theta}   \sum_{\bfi^+} \frac{(d-1)^{(6\widehat h+3\sum_{j=1}^{p-1} h_{j1})\ell}}{Z_{\cF^+}}\times \bE\left[I(\cF^+,\cG)\bm1(\cG\in \Omega) \widehat R_{\bfi^+}\right],\\
&\widehat R_{\bfi^+}
=R_{\widehat h+1}\prod_{j=1}^{p-1}\cW_j'
=\frac{1}{(Nd)^{p-1}}\sum_{u_1\sim v_1,\cdots, u_{p-1}\sim v_{p-1}} \left(\prod_{m=1}^{k_1}  A_{\al_  m}  \prod_{m=1}^{k_2}  B_{\beta_m} 
\prod_{m=1}^{k_3/2}  C_{\gamma_{2m-1} \gamma_{2m}} \prod_{m=1}^{k_4}  D_{\theta_{m}}  E(\bm\chi)\right).
\end{split}\end{align}
Here $\widehat h=h'+h''$ is the sum of  $h'$ from 
\eqref{e:decomposeWj} and $h''$ from \eqref{e:oneterm0},  $R_{\widehat h+1}$ is the product of $R_{h'}$ from \eqref{e:decomposeWj} and the summand in \eqref{e:oneterm0}, 
$\cW_j'$ and $\bmh$ are from \eqref{e:decomposeWj}. 
In this way $\widehat R_{\bfi^+}\in \Adm(\widehat h+1,\bmh,\cF^+,\cG)$.
The summation for $\bm\al, \bm\beta, \bm\gamma,\bm\theta$ is over each $\al_m, \beta_m, \gamma_m,\theta_m$ in one of the sets $\qq{\mu}$, $\sfA_i$ or $\qq{\mu}\setminus \sfA_i$;
 and the factors in \eqref{e:oneterm} are given by \eqref{e:defABC}, 
\begin{align*}
  D_\theta\in \left\{   G^{\circ}_{b_\theta w}, G^{\circ}_{c_\theta w},L_{b_\theta w},L_{c_\theta w}\right\}, \quad w\in \cV,
\end{align*}
are from $D_{\theta_m}$ in \eqref{e:yibufen} and \eqref{e:defD}; and $E(\bm\chi)$ is a product of the remaining terms in $R_{\widehat h+1}\prod_{j=1}^{p-1}\cW_j'$; $E(\bm\chi)$ depends on $\{b_\chi, c_\chi\}_{\chi \in \bm\chi}$ for some ${\bm\chi}\subset\qq{\mu}$.

The summation over $\bm\al, \bm\beta, \bm\gamma,\bm\theta$ in \eqref{e:oneterm} results in $(d-1)^{(k_1+k_2+k_3+k_4)\ell}$ terms in the form \eqref{e:case3_copy}. Note that the number of these terms is much larger than the normalization factor $(d-1)^{(k_1 +(k_2+k_3+k_4)/2)\ell}$ in the denominator. However, as we will show in \Cref{l:erbu}, most of these terms can be shown almost immediately to be negligible. 
Before stating \Cref{l:erbu}, we need to introduce some notation. 

We view $\bm\al, \bm\beta, \bm\gamma,\bm\theta, \bm\chi$ as words, which are sequences of indices in $\qq{\mu}$. In particular $(\bm\al, \bm\beta, \bm\gamma,\bm\theta)$ is a word with length $k_1+k_2+k_3+k_4$.
Given $\bm\chi$, we partition words $ \bm\omega\in \qq{\mu}^{k_1+k_2+k_3+k_4}$ into equivalence classes.  Two words $\bm\omega\sim \bm\omega'$ are equivalent if there is a bijection on $\qq{\mu}$ which preserves $\bm\chi$ and maps $\bm\omega$ to $\bm\omega'$. We remark that the expectation in \eqref{e:oneterm} depends only on the equivalence class of $(\bm\al, \bm\beta, \bm\gamma,\bm\theta)$. 

For any $f_0, f_1\geq 0$, let $\mathsf W(f_0, f_1)$ denote a set of representatives for equivalence classes of $\qq{\mu}^{k_1+k_2+k_3+k_4}$. Here, for a word $\bm\omega\in \mathsf W(f_0, f_1)$, $f_0$ is the number of distinct indices (ignoring multiplicity)  that do not appear in 
$\bm\chi$, and $f_1$ is the number of these indices appearing exactly once in $\bm\omega$. The length of $\bm\omega$ is $k_1+k_2+k_3+k_4$, and $f_0-f_1$ of these distinct indices appear at least twice in $\bm\omega$. This implies
\begin{align}\label{e:f0f1}
    k_1+k_2+k_3+k_4 -f_1\geq  2(f_0-f_1)\Rightarrow k_1+k_2+k_3+k_4\geq 2f_0-f_1.
\end{align} 

The expectation in \eqref{e:oneterm} depends only on the equivalence class of $(\bm\al, \bm\beta, \bm\gamma,\bm\theta)$. Moreover, for fixed $f_0$ and $f_1$, the summation of $(\bm\al, \bm\beta,\bm\gamma,\bm\theta)\sim \bm\omega$, contains $\OO((d-1)^{f_0\ell})$ terms. Thus for the summation over $(\bm\al, \bm\beta, \bm\gamma,\bm\theta)$ in \eqref{e:oneterm}, we can first sum over the equivalence classes. For this, we will write our sum over
\begin{align}\label{e:partition2}
   \sum_{\bfi^+} \frac{(d-1)^{(6\widehat h+3\sum_{j=1}^{p-1} h_{j1})\ell}}{Z_{\cF^+}}\bE\left[I(\cF^+,\cG)\bm1(\cG\in \Omega) \widehat R_{\bfi^+}(\bm\al, \bm\beta, \bm\gamma, \bm\theta,\bm\chi)\right].
\end{align}
We then have
\begin{align}\begin{split} \label{e:partition}
    \eqref{e:oneterm}&=\frac{1}{(d-1)^{(k_1 +(k_2+k_3+k_4)/2)\ell}} \sum_{\bm\al, \bm\beta, \bm\gamma, \bm\theta}   \eqref{e:partition2}\\
    &=\frac{1}{(d-1)^{(k_1 +(k_2+k_3+k_4)/2)\ell}} \sum_{f_0,f_1}\sum_{\bm\omega\in \mathsf W(f_0,f_1)}\sum_{(\bm\al, \bm\beta, \bm\gamma, \bm\theta)\sim \bm\omega}   \eqref{e:partition2},\\
    &=\sum_{f_0,f_1}\sum_{(\bm\al, \bm\beta, \bm\gamma, \bm\theta)\in \mathsf W(f_0,f_1)}\frac{\OO((d-1)^{f_0\ell})}{(d-1)^{(k_1 +(k_2+k_3+k_4)/2)\ell}} \times   \eqref{e:partition2}.
\end{split}\end{align}

The following lemma states that for given $\bm\al, \bm\beta, \bm\gamma,\bm\theta$, the summands in the last term of \eqref{e:partition} is either negligible, or it can be reduced to a term as in \eqref{e:oneterm2}, where each index in $\bm\al, \bm\beta, \bm\gamma,\bm\theta$ appears at least twice in $R_{\bfi^+}'$.

Let $f_0'$ denote the number of distinct values in $\bm\theta$. Then $f_0-f_0'\leq k_1+k_2+k_3$ and \eqref{e:S'bound2} implies
 \begin{align}\label{e:S'bound3_main}
     \frac{(d-1)^{f_0'\ell+3\ell\sum_{j=1}^{p-1}h_{j1}}}{(d-1)^{k_4 \ell/2}}\prod_{j=1}^{p-1}|\cW_j'|\lesssim (|Q_t-Y_t|+(d-1)^{8\ell}\Upsilon \Phi)^{p-1}.
 \end{align}
By using \eqref{e:S'bound3_main} as input, the proof of  \Cref{l:erbu} is parallel to those of the statements \eqref{e:huanG-Q}, \eqref{e:huanG} and \eqref{e:huanGS}. We defer its proof to \Cref{s:tree_computation}.

\begin{lemma}\label{l:erbu}
Fix $0\leq f_1\leq f_0$ satisfying \eqref{e:f0f1}, and a word $(\bm\al, \bm\beta, \bm\gamma, \bm\theta)\in \mathsf W(f_0, f_1)$.
Let $\mathsf I_{\rm single}\subset \qq{\mu}$ with $|\mathsf I_{\rm single}|=f_1$, denote the set of indices that appear only once among 
$(\bm\al, \bm\beta, \bm\gamma,\bm\theta)$, and do not appear in $\bm\chi$. Then 
\begin{align}\begin{split}\label{e:oneterm1}
&\frac{(d-1)^{f_0\ell}}{(d-1)^{(k_1 +(k_2+k_3+k_4)/2)\ell}}   \sum_{\bfi^+} \frac{(d-1)^{(6\widehat h+3\sum_{j=1}^{p-1} h_{j1})\ell}}{Z_{\cF^+}}\times \bE\left[I(\cF^+,\cG)\bm1(\cG\in \Omega) \widehat R_{\bfi^+}(\bm\al, \bm\beta, \bm\gamma, \bm\theta,\bm\chi)\right]
\end{split}\end{align}
satisfies
\begin{enumerate}
    \item If there exists $\al_m\in \mathsf I_{\rm single}$, or $\theta_m\in  \mathsf I_{\rm single}$, or $B_{\beta_m}\in \{ G^{\circ}_{b_{\beta_m} s},  G^{\circ}_{c_{\beta_m} s}\}_{s\in \cK\setminus\{i,o\}}$ with $\beta_m\in  \mathsf I_{\rm single}$, then $\eqref{e:oneterm1}=\OO(N^{-\fb/4}\bE[\Psi_p])$.
\item If none of the above cases apply, then 
\begin{align}\begin{split}\label{e:oneterm2}
\eqref{e:oneterm1}=\frac{(d-1)^{(6(\widehat h+f_1)+3\sum_{j=1}^{p-1} h_{j1})\ell}}{(d-1)^{\fq^+\ell/2}Z_{\cF^+}}
\sum_{\bfi^+}  \bE\left[I(\cF^+,\cG)\bm1(\cG\in \Omega)  (d-1)^{-f_1/2}R'_{\bfi^+}\right]+\OO(N^{-\fb/4}\bE[\Psi_p]),
\end{split}\end{align}
where $\fq^+=k_1\ell/2+(k_1 +k_2+k_3+k_4 -2 f_0+12f_1)\geq 0$, and $R'_{\bfi^+}\in \Adm(\widehat h+f_1+1, \bmh,\cF^+,\cG)$ is obtained from $\widehat R_{\bfi^+}$ by making the following substitutions: 
\begin{align}\begin{split}\label{e:final_replace}
&G_{c_{\beta_m} c}^{(b_{\beta_m} b)}\rightarrow G_{b_{\beta_m} c}^{(b)}(G_{c_{\beta_m} c_{\beta_m}}^{(b_{\beta_m})}-Q_t),
\quad 
G_{c_{\beta_m} b}^{(b_{\beta_m})}\rightarrow G_{b_{\beta_m} b}(G_{c_{\beta_m} c_{\beta_m}}^{(b_{\beta_m})}-Q_t), \quad \beta_m\in \mathsf I_{\rm single},
\\
 &   G_{c_{\gamma_{2m-1}} c_{\gamma_{2m}}}^{(b_{\gamma_{2m-1}} b_{\gamma_{2m}})}
    \rightarrow 
 G_{b_{\gamma_{2m-1}} c_{\gamma_{2m}}}^{ (b_{\gamma_{2m}})}(G_{c_{\gamma_{2m-1}} c_{\gamma_{2m-1}}}^{(b_{\gamma_{2m-1}})}-Q_t) , 
\quad \gamma_{2m-1}\in \mathsf I_{\rm single},\quad \gamma_{2m}\not\in \mathsf I_{\rm single},\\
 &   G_{c_{\gamma_{2m-1}} c_{\gamma_{2m}}}^{(b_{\gamma_{2m-1}} b_{\gamma_{2m}})}
    \rightarrow 
 G_{c_{\gamma_{2m-1}} b_{\gamma_{2m}}}^{ (b_{\gamma_{2m-1}})}(G_{c_{\gamma_{2m}} c_{\gamma_{2m}}}^{(b_{\gamma_{2m}})}-Q_t) , 
\quad \gamma_{2m-1}\notin \mathsf I_{\rm single},\quad \gamma_{2m}\in \mathsf I_{\rm single},\\
&G_{c_{\gamma_{2m-1}} c_{\gamma_{2m}}}^{(b_{\gamma_{2m-1}} b_{\gamma_{2m}})}
\rightarrow 
G_{b_{\gamma_{2m-1}} b_{{\gamma_{2m}}}}(G_{c_{\gamma_{2m-1}} c_{\gamma_{2m-1}}}^{(b_{\gamma_{2m-1}})}-Q_t) (G_{c_{{\gamma_{2m}}} c_{{\gamma_{2m}}}}^{(b_{{\gamma_{2m}}})}-Q_t), 
\quad \gamma_{2m-1}, \gamma_{2m}\in \mathsf I_{\rm single}.
\end{split}\end{align}
\end{enumerate}

\end{lemma}


We remark that in  \eqref{e:oneterm2}, if $\fq^+\geq 1$, we gain an additional factor of  $(d-1)^{-\ell/2}$. We do not obtain this extra factor only if $k_1=0$ and each index in $(\bm\al, \bm\beta, \bm\gamma, \bm\theta)$ appears exactly twice without appearing in $\bm\chi$.

\begin{proof}[Proof of \eqref{e:higher_case3} in \Cref{p:general}] 
Up to a negligible error, the expression \eqref{e:higher_case3} can be rewritten as an $\OO(1)$-weighted sum of terms in the form of  \eqref{e:oneterm1}. We also refer back to the more explicit expression given in \eqref{e:oneterm}. 
If the assumptions in the first statement in \Cref{l:erbu} hold, there is nothing to prove. So in the rest of the proof we can focus on the second case \eqref{e:oneterm2}.

There are several cases in which we can apply \eqref{e:oneterm2}, based on the decomposition of  $\widehat B_0$ (as in \eqref{e:fcase1W0}) using terms $\msc^{2\ell}(z_t)(d-1)^{-(\ell+1)}\sum_{\al\in \sfA_i} (G_{c_\al c_\al}^{(b_\al)}-Q_t)$, $\msc^{2\ell}(z_t)(d-1)^{-(\ell+1)}\sum_{\al\neq\beta\in \sfA_i} (G_{c_\al c_\beta}^{(b_\al b_\beta)}-Q_t)$ or a term in $\cU_0$. We treat each of these separately.

\begin{enumerate}

\item  Assume $\widehat R_{\bfi^+}$ in \eqref{e:oneterm} contains the factor $\msc^{2\ell}(z_t)(d-1)^{-(\ell+1)}\sum_{\al\in\sfA_i} (G_{c_\al c_\al}^{(b_\al)}-Q_t)$ from the decomposition of $\widehat  B_0$, then in \eqref{e:oneterm2}, $k_1\geq 1$ and $R'_{\bfi^+}$ contains a factor $(G_{c_\al c_\al}^{(b_\al)}-Q_t)$. Let $R'_{\bfi^+}= (G_{c_\al c_\al}^{(b_\al)}-Q_t)R_{\bfi^+}$, then $R_{\bfi^+}\in\Adm(\widehat h+f_1, \bmh, \cF^+,\cG)$.  We claim that replacing $Q_t$ with $Y_t$ yields a negligible error. To see this, we write 
    \begin{align}\begin{split}\label{e:finaleq}
       &\phantom{{}={}}\frac{(d-1)^{(6\widehat h+6f_1+3\sum_{j=1}^{p-1} h_{j1})\ell}}{(d-1)^{\fq^+\ell/2} Z_{\cF^+}}
       \sum_{\bfi^+}\bE\left[I(\cF^+,\cG)\bm1(\cG\in \Omega)  (G_{c_\al c_\al}^{(b_\al)}-Q_t)R_{\bfi^+}\right]\\
       &=\frac{(d-1)^{(6\widehat h+6f_1+3\sum_{j=1}^{p-1} h_{j1})\ell}}{(d-1)^{\fq^+\ell/2} Z_{\cF^+}}\sum_{\bfi^+}\bE\left[I(\cF^+,\cG)\bm1(\cG\in \Omega)  (G_{c_\al c_\al}^{(b_\al)}-Y_t)R_{\bfi^+}\right]\\
       &+
       \frac{(d-1)^{(6\widehat h+6f_1+3\sum_{j=1}^{p-1} h_{j1})\ell}}{(d-1)^{\fq^+\ell/2} Z_{\cF^+}}\sum_{\bfi^+}\bE\left[I(\cF^+,\cG)\bm1(\cG\in \Omega)  (Q_t-Y_t)R_{\bfi^+}\right].
    \end{split}\end{align}
    Note that $k_1\geq 1$, so $\fq^+\geq 1$, and the second term on the right-hand side of \eqref{e:finaleq} is bounded as
\begin{align*}\begin{split}
        &\phantom{{}={}}\frac{(d-1)^{(6\widehat h+6f_1+3\sum_{j=1}^{p-1} h_{j1})\ell}}{(d-1)^{\fq^+\ell/2} Z_{\cF^+}}\sum_{\bfi^+}\bE\left[I(\cF^+,\cG)\bm1(\cG\in \Omega)  |Q_t-Y_t||R_{\bfi^+}|\right]\\
        &\lesssim 
         \frac{(d-1)^{6(\widehat h+f_1)\ell}N^{-(\widehat h+f_1)\fb}}{(d-1)^{\fq^+\ell/2} }\bE\left[\bm1(\cG\in \Omega)  |Q_t-Y_t|\Pi(z,\bmh)\right]\lesssim N^{-\fb/2}\bE[ \Psi_p],
    \end{split}\end{align*}
    where in the first inequality we used \eqref{e:Bsmall} and \eqref{e:Sbound}; in the second inequality we used \eqref{e:bPi2}.
    The first term on the right-hand side of \eqref{e:finaleq} is in the form of \eqref{e:case3_copy}, by setting $r^+= \widehat h+f_1\geq r$ and $\bmr^+=\bmh$.

\item \label{ii:offab} If $\widehat R_{\bfi^+}$ in \eqref{e:oneterm} contains the factor $\msc^{2\ell}(z_t)(d-1)^{-(\ell+1)}\sum_{\al\neq \beta\in\sfA_i} G_{c_\al c_{\beta}}^{(b_\al b_{\beta})}$ from the decomposition of $\widehat  B_0$, then $k_3\geq 2$ in \eqref{e:oneterm2}. Recall that by our assumptions, we either have $\bmr$ satisfies \eqref{e:feasible}, $r\geq 1$ and $\bmr$ satisfies \eqref{e:1_expand}, or $r\geq 2$. There are several sub-cases:
\begin{enumerate}
    \item If $\bmh$ satisfies \eqref{e:feasible}, then \eqref{e:refined_bound2} implies that \eqref{e:oneterm2} is bounded by $\OO(N^{-\fb/2}\bE[ \Psi_p])$.
    \item If $\bmh$ satisfies \eqref{e:1_expand}, then $\bmr$ does not satisfy \eqref{e:feasible}, and $r\geq 1$. Since $\widehat h+f_1\geq r\geq 1$, \eqref{e:refined_bound2} again implies that \eqref{e:oneterm2} is bounded by $\OO(N^{-\fb/2}\bE[ \Psi_p])$. 
    
    \item In the remaining cases, $\widehat h+f_1\geq r\geq 2$ and $h_{j0}=1$ for all $1\leq j\leq p-1$. In this case, the reduction of $\widehat R_{\bfi^+}$ only uses expansions from \Cref{i1}, 
     \eqref{e:zhankai1}, \Cref{i3}, 
     \Cref{i4} and \Cref{i7}. In particular, $\widehat R_{\bfi^+}$ does not contain $G^\circ$ or $L$.
    
    If in  $\widehat R_{\bfi^+}$ (from \eqref{e:oneterm}) $\al\in \mathsf I_{\rm single}$, then $f_1 \geq 1$ in \eqref{e:oneterm2}. Also, \eqref{e:final_replace} implies that $R'_{\bfi^+}$ contains $G_{c_\al c_\al}^{(b_\al)}-Q_t$. By the same argument as in \eqref{e:finaleq}, this leads to \eqref{e:case3_copy} by setting $r^+= \widehat h+f_1\geq r+1$ and $\bmr^+=\bmh$. The same conclusion holds if $\beta\in \mathsf I_{\rm single}$.

    In the remaining cases $\al, \beta\not\in \mathsf I_{\rm single}$. If \eqref{e:oneterm} contains at least two terms in the form $\{G_{cc'}^{(bb')}, G_{c b'}^{(b)}, G_{bb'}, G_{cb'}\}_{(b,c)\neq (b',c')\in \cK^+}$, so does \eqref{e:oneterm2}. Then  \eqref{e:refined_bound} and $\widehat h+f_1\geq 2$ implies that \eqref{e:oneterm2} is bounded by $\OO(N^{-\fb/2}\bE[\Psi_p])$. 

    Otherwise \eqref{e:oneterm} contains factors $(G_{c_\al c_\al}^{(b_\al)}-Q_t), (G_{c_{\beta} c_{\beta}}^{(b_{\beta})}-Q_t)$. Then they are either contained in $E(\bm\chi)$ or $k_1\geq 1$. In both cases $\fq^+\geq 1$ in \eqref{e:oneterm2}. The same argument as in \eqref{e:finaleq} leads to  \eqref{e:case3_copy}, by setting $r^+= \widehat h+f_1\geq r$ and $\bmr^+=\bmh$.
\end{enumerate}

\item In the remaining case, $\widehat R_{\bfi^+}$ in \eqref{e:oneterm} contains a factor $(d-1)^{3(h-1)\ell}R_{h}$ from $\cU_0$ in the decomposition  \eqref{e:fcase1W0} of $\widehat  B_0$. Here $h\geq 2$, $R_{h}$ is an $S$-product term (recall from \Cref{d:S-product}), and it contains at least one factor of the form  $( G_{c_\al c_{\al}}^{(b_\al b_\al)}-Q_t)$ or $G_{c_\al c_{\beta}}^{(b_\al b_\beta)}$. Then in this case the factor $R_{h}$ is included in $E(\bm\chi)$ in \eqref{e:oneterm}, and $\widehat h\geq r+h-1\geq r+1\geq 1$.

If $R_{h}$ contains at least one term of the form $(G_{c_\al c_\al}^{(b_\al)}-Q_t)$, by the same argument as in \eqref{e:finaleq}, \eqref{e:oneterm2} leads to \eqref{e:case3_copy} with $r^+=\widehat h+f_1\geq \widehat h\geq r+1$. 

In the other cases, $R_{h}$ contains at least one term of the form $G_{c_\al c_\beta}^{(b_\al b_\beta)}$. By the same argument as in \Cref{l:erbu}, if $\{b_\al, c_\al\}$ do not appear in other terms of $\widehat R_{\bfi^+}$, we can replace $G_{c_\al c_\beta}^{(b_\al b_\beta)}$ by $
      G_{b_\al c_{\beta}}^{(b_{\beta})}(G_{c_\al c_\al}^{(b_\al)}-Q_t)/\sqrt{d-1} 
$ (as in \eqref{e:final_replace}); and if $\{b_\al, c_\al, b_\beta, c_\beta\}$ do not appear in other terms of $\widehat R_{\bfi^+}$, we can replace it by $
      G_{b_\al b_{\beta}}(G_{c_\al c_\al}^{(b_\al)}-Q_t) (G_{c_{\beta} c_{\beta}}^{(b_{\beta})}-Q_t) /\sqrt{d-1}$. Moreover, the errors from such replacements are bounded by $\OO(N^{-\fb/4}\bE[\Psi_p])$. 
Then by the same argument as in \Cref{ii:offab}, either \eqref{e:oneterm2} is bounded by $\OO(N^{-\fb/4}\bE[\Psi_p])$, or \eqref{e:oneterm2} leads to  \eqref{e:case3_copy} with $r^+\geq \widehat h+f_1\geq \widehat h\geq r+1$.
\end{enumerate}
\end{proof}

\begin{proof}[Proof of \eqref{e:higher_case1} in \Cref{p:general}] 
Up to a negligible error, the expression \eqref{e:higher_case1} can also be rewritten as an $\OO(1)$-weighted sum of terms in the form of  \eqref{e:oneterm1}. We also refer back to the more explicit expression given in \eqref{e:oneterm}. 

We recall that in \eqref{e:higher_case1}, $r=1$ and $B_1=(G_{oo}^{(i)}-Q_t)$. If $\widehat h+f_1\geq 2$ (from \eqref{e:oneterm2}), or $h_{j0}\in \{2,3\}$ for some $1\leq j\leq p-1$ in \eqref{e:oneterm2}, we can proceed in exactly the same manner as in the proof of \eqref{e:higher_case3}.
Otherwise, $\widehat h=1$, $f_1=0$, and $h_{j0}=1$ for all $1\leq j\leq p-1$. We assume this scenario in the following discussion.
\begin{enumerate}
    \item 

Assume $\widehat R_{\bfi^+}$ in \eqref{e:oneterm} contains $\msc^{2\ell}(z_t)(d-1)^{-(\ell+1)}\sum_{\al\in\sfA_i} (G_{c_\al c_{\al}}^{(b_\al)}-Q_t)$ from the decomposition \eqref{e:fcase1W0} of $\widehat  B_0$. 
Since $\widehat h=1, f_1=0$, $\widehat R_{\bfi^+}$ also contains $\msc^{2\ell}(z_t)(d-1)^{-(\ell+1)}\sum_{\al\in\sfA_i} (G_{c_\al c_{\al}}^{(b_\al)}-Q_t)$ from the decomposition of $\widehat  B_1$ (we recall the precise coefficients from \eqref{e:Uterm}). Then \eqref{e:oneterm2} is of the form 
       \begin{align}\label{e:main1}
       \frac{\msc^{4\ell}(z_t)}{(d-1)^{2(\ell+1)}}\sum_{\al\in \sfA_i}\sum_{\bfi^+} \frac{1}{Z_{\cF^+}} \bE[\bm1(\cG\in \Omega)I(\cF^+, \cG)(G_{c_\al c_\al}^{(b_\al)}-Q_t)^2 (Q_t-Y_t)^{p-1}].
    \end{align}
After replacing a copy of $(G_{c_\al c_\al}^{(b_\al)}-Q_t)$ by $(G_{c_\al c_\al}^{(b_\al)}-Y_t)$, \eqref{e:main1} is an $\OO(1)$-weighted sum of terms in the form \eqref{e:case1} with $\fq^+=2$, and the error is bounded by $\OO(N^{-\fb/2}\bE[\Psi_p])$. 

\item If $\widehat R_{\bfi^+}$ in \eqref{e:oneterm}  contains $\msc^{2\ell}(z_t)(d-1)^{-(\ell+1)}\sum_{\al\neq \beta\in\sfA_i} G_{c_\al c_{\beta}}^{(b_\al b_{\beta})}$ from the decomposition \eqref{e:fcase1W0} of $\widehat B_0$. Since $\widehat h=1, f_1=0$, $\widehat R_{\bfi^+}$ also contains $\msc^{2\ell}(z_t)(d-1)^{-(\ell+1)}\sum_{\al\neq \beta\in\sfA_i} G_{c_\al c_{\al}}^{(b_\al b_\beta)}$ from the decomposition of $\widehat  B_1$ (we recall the precise coefficients from \eqref{e:Uterm}). Then \eqref{e:oneterm2} is from 
 \begin{align}\label{e:main2}
       \frac{2\msc^{4\ell}(z_t)}{(d-1)^{2(\ell+1)}}\sum_{\al\neq\beta\in \sfA_i}\sum_{\bfi^+} \frac{1}{Z_{\cF^+}} \bE[\bm1(\cG\in \Omega)I(\cF^+, \cG)(G_{c_\al c_\beta}^{(b_\al b_\beta)})^2 (Q_t-Y_t)^{p-1}].
    \end{align}
Thanks to \eqref{e:use_Ward}, \eqref{e:main2} is bounded by $\OO(N^\fo\bE[\Psi_p])$.

\item
In the remaining case, $\widehat R_{\bfi^+}$ in \eqref{e:oneterm} contains a factor $(d-1)^{3(h-1)\ell}R_{h}$ with $h\geq 2$ from $\cU_0$ in the decomposition \eqref{e:fcase1W0} of $\widehat  B_0$. Then $\widehat h\geq r+1\geq 2$. 
\end{enumerate}

\end{proof}

\begin{proof}[Proof of \eqref{e:higher_case2} in \Cref{p:general}] 
Up to a negligible error, the expression \eqref{e:higher_case2} can also be rewritten as an $\OO(1)$-weighted sum of terms in the form of  \eqref{e:oneterm1}. We also refer back to the more explicit expression given in \eqref{e:oneterm}.

We recall that in \eqref{e:higher_case2} $r=0$. If in \eqref{e:oneterm2}, $\widehat h+f_1\geq 1$, or $\bmh$ satisfies \eqref{e:feasible} in \eqref{e:oneterm2}, we can proceed in exactly the same manner as in the proof of \eqref{e:higher_case3}.
Otherwise, $\widehat h=f_1=0$, $f_1=0$, and $\bmh$ satisfies \eqref{e:1_expand}. Therefore, we assume this scenario in the following discussion.

\begin{enumerate}
    \item 

Suppose $\widehat R_{\bfi^+}$ in \eqref{e:oneterm} contains $\msc^{2\ell}(z_t)(d-1)^{-(\ell+1)}\sum_{\al\in\sfA_i} (G_{c_\al c_{\al}}^{(b_\al)}-Q_t)$ from the decomposition \eqref{e:fcase1W0} of $\widehat  B_0$. 
Given that $\widehat h=f_1=0$ and $\bmh$ satisfies \eqref{e:1_expand}, \eqref{e:oneterm2} is obtained from \eqref{e:higher_case2} by making the following substitutions. For reference, we recall the precise coefficients from \eqref{e:Uterm} and \eqref{e:fengkai}. For $\al \in \sfA_i$,
\begin{align}\begin{split}\label{e:sub0}
    &\wt G_{oo}^{(i)}-Y_t \rightarrow \frac{\msc^{2\ell}(z_t)}{(d-1)^{\ell+1}}(G_{c_\al c_\al}^{(b_\al)}-Q_t), \\
    & \wt G^{\circ}_{sw}\rightarrow  -\frac{L_{s l_\al}}{\sqrt{d-1}}\left( G^{\circ}_{c_\al w}+\frac{ G^{\circ}_{b_\al w}}{\sqrt{d-1}}\right), \quad s\in \{i,o\}, \quad w\in \{u_j,v_j\}.
\end{split}\end{align}
In this way $f_1=0$. We recall that $|L_{s l_\al}|\lesssim (d-1)^{\ell/2}$ from \eqref{e:Gtreemkm}.  After replacing $(G_{c_\al c_\al}^{(b_\al)}-Q_t)$ by $(G_{c_\al c_\al}^{(b_\al)}-Y_t)$, terms obtained after the substitution \eqref{e:sub0} form an $\OO(1)$-weighted sum of terms in the form \eqref{e:case2} with $\fq^+=2$.

\item

Suppose $\widehat R_{\bfi^+}$ in \eqref{e:oneterm} contains $\msc^{2\ell}(z_t)(d-1)^{-(\ell+1)}\sum_{\al\neq \beta\in\sfA_i} G_{c_\al c_{\beta}}^{(b_\al b_{\beta})}$ from the decomposition of $\widehat  B_0$. 
Given that $\widehat h=f_1=0$ and $\bmh$ satisfies \eqref{e:1_expand}, then \eqref{e:oneterm2} results from making the following substitutions in \eqref{e:higher_case2}. For reference, we recall the precise coefficients from \eqref{e:Uterm} and \eqref{e:fengkai}. For $\al\neq \beta\in \sfA_i$,
\begin{align}\begin{split}\label{e:sub1}
    &\wt G_{oo}^{(i)}-Y_t \rightarrow \frac{\msc^{2\ell}(z_t)}{(d-1)^{\ell+1}}G_{c_\al c_\beta}^{(b_\al b_\beta)}  \text{ or }
    \frac{\msc^{2\ell}(z_t)}{(d-1)^{\ell+1}}G_{c_\beta c_\al}^{(b_\beta b_\al)}, \\
    & \wt G^{\circ}_{sw}\rightarrow  -\frac{L_{s l_\al}}{\sqrt{d-1}}\left( G^{\circ}_{c_\al w}+\frac{ G^{\circ}_{b_\al w}}{\sqrt{d-1}}\right), \quad s\in\{i,o\},\quad w\in \{u,v\},\\
    &\wt G^{\circ}_{s'w'}\rightarrow  -\frac{L_{o l_\beta}}{\sqrt{d-1}}\left( G^{\circ}_{c_\beta w'}+\frac{ G^{\circ}_{b_\beta w'}}{\sqrt{d-1}}\right), \quad s'\in\{i,o\},\quad w'\in \{u,v\}.
\end{split}\end{align}
In this way $f_1=0$. We recall that $|L_{s l_\al}|,|L_{o l_\beta}|\lesssim (d-1)^{\ell/2}$ from \eqref{e:Gtreemkm}. Terms obtained after the substitution \eqref{e:sub1} are $\OO(1)$-weighted sum of terms of the following form: for $\al\neq \beta\in \qq{\mu}, x\in \{b_\al, c_\al\}, x'\in \{b_\beta, c_\beta\}, w,w'\in\{u,v\}$,
 \begin{align}\label{e:sumtt}
      \sum_{\bfi}  \bE\left[\frac{\{1-\del_1 Y_\ell, \del_2 Y_\ell\}}{(Nd) Z_{\cF^+}}I(\cF^+,\cG)\bm1(\cG,\tcG\in \Omega)\sum_{u\sim v\in \qq{N}} G_{c_\al c_\beta}^{(b_\al b_\beta)} G^{\circ}_{x w} G^{\circ}_{x' w'}(Q_t-Y_t)^{p-2}\right].
    \end{align}
By \eqref{e:refined_bound3}, the expression in \eqref{e:sumtt} can be bounded by $\OO(N^\fo\bE[\Psi_p])$.

\item 
In the remaining case, $\widehat R_{\bfi^+}$ in \eqref{e:oneterm} contains a factor $(d-1)^{3(h-1)\ell}R_{h}$ from $\cU_0$ in the decomposition of $\widehat  B_0$. Then $\widehat h\geq r+1\geq 1$. 
\end{enumerate}
\end{proof}

\begin{proof}[Proof of \Cref{p:track_error2}]

We recall $I_1$ from \eqref{e:case1_term}. Conditioned on $I(\cF^+,\cG)=1$, the expectation in \eqref{e:case1_term} does not depend on $\al$.  Moreover, $|\sfA_i|=(d-1)^{\ell+1}$, and by \eqref{e:Gtreemsc}, $L_{l_\al l_\al}^{(i)}=\md(z_t)(1-(-\msc(z_t)/\sqrt{d-1})^{2\ell+2})$.
We denote $(i,o)=(b_\al,c_\al)$ and $(\cF, \bfi)=(\cF^+, \bfi^+)$, and rewrite $I_1$ from \eqref{e:case1_term} as
\begin{align}\label{e:I1_rewrite}
\left(1-\left(\frac{\msc(z_t)}{\sqrt{d-1}}\right)^{2\ell+2}\right)\frac{\md(z_t)\msc^{2\ell}(z_t) }{(d-1)Z_{\cF}}\sum_{\bfi}\bE\left[I(\cF,\cG)\bm1(\cG\in \Omega) (G_{oo}^{(i)}-Y_t)(G_{oo}^{(i)}-Q_t)(Q_t-Y_t)^{p-1}\right],
\end{align}
which is in the form of \eqref{e:higher_case1}, up to the constant. 

 From the proof of \eqref{e:higher_case1} in \Cref{p:general}, the errors from expanding \eqref{e:higher_case1} are either bounded by $\OO(N^{-\fb/4} \bE[\Psi_p])$, or given by \eqref{e:main2}. 
 Thus, the error from expanding \eqref{e:I1_rewrite} 
is given by 
\begin{align*}
   &\left(1-\left(\frac{\msc(z_t)}{\sqrt{d-1}}\right)^{2\ell+2}\right)   \frac{2\md(z_t)\msc^{6\ell}(z_t)}{(d-1)^{2\ell+3}}\times\\
      &\times \sum_{\al\neq\beta\in \sfA_i}\sum_{\bfi^+} \frac{1}{Z_{\cF^+}} \bE[\bm1(\cG\in \Omega)I(\cF^+, \cG)(G_{c_\al c_\beta}^{(b_\al b_\beta)})^2 (Q_t-Y_t)^{p-1}]
      +\OO(N^{-\fb/4} \bE[\Psi_p]).
\end{align*}
This finishes the proof of \eqref{e:track_error2}.


We recall $I_2$ from \eqref{e:case2_term}. Once again, conditioned on  $I(\cF^+,\cG)=1$, the expectation in \eqref{e:case2_term} does not depend on $\al$, and $|\sfA_i|=(d-1)^{\ell+1}$.
We denote $(i,o)=(b_\al, c_\al)$ and $(\cF, \bfi)=(\cF^+, \bfi^+)$, and rewrite $I_2$ from \eqref{e:case2_term} as
\begin{align}\begin{split}\label{e:dierxiang}
     &\phantom{{}={}}\frac{\msc^{2\ell}(z_t)}{Z_{\cF}}\sum_{\bfi}\bE\left[ I(\cF,\cG)\bm1(\cG\in \Omega)(G_{oo }^{(i)}-Q_t)(Q_t-Y_t)^{p-2}\right.\\
     &\times \left.\frac{1}{Nd}\sum_{u\sim v}\left((1-\del_1Y_\ell)\left(F_{vv}-\frac{2G_{uv}}{G_{uu}}F_{uv}+\frac{G^2_{uv}}{G^2_{uu}}F_{uu}\right)+(-t\del_2 Y_\ell) F_{vv} \right)\right],
\end{split}\end{align}
where 
\begin{align*}  F_{uv}=\left(\frac{1}{\sqrt{d-1}}+\frac{2\md(z_t) \msc(z_t)}{(d-1)\sqrt{d-1}}\right)( G^{\circ}_{uo} G^{\circ}_{vi}+ G^{\circ}_{ui} G^{\circ}_{vo})+\frac{2\md(z_t)}{d-1}\left( G^{\circ}_{ui} G^{\circ}_{vi}+ G^{\circ}_{uo} G^{\circ}_{vo}\right).
\end{align*}

The expression \eqref{e:dierxiang} is an $\OO(1)$-weighted sum of terms in the form \eqref{e:higher_case2}. From the proof of \eqref{e:higher_case2} in \Cref{p:general}, the errors from expanding \eqref{e:higher_case2} are either bounded by $\OO(N^{-\fb/4} \bE[\Psi_p])$, or given by substitutions as in \eqref{e:sub1}.

Next we show that after the substitutions as described in \eqref{e:sub1}, terms $F_{vv}$, $F_{uv}$ and $F_{uu}$ in \eqref{e:dierxiang} all vanishes. We check this for $F_{uv}$, the other two cases can be proven in the same way, so we omit their proofs. 
By the substitutions listed in \eqref{e:sub1}, $F_{uv}$ transforms to
\begin{align}\begin{split}\label{e:Fuv2}
    &\phantom{{}={}}\left( G^{\circ}_{c_\al u}+\frac{ G^{\circ}_{b_\al u}}{\sqrt{d-1}}\right)\left( G^{\circ}_{c_\beta v}+\frac{ G^{\circ}_{b_\beta v}}{\sqrt{d-1}}\right)\times \\
    &\times \left(\left(\frac{1}{\sqrt{d-1}}+\frac{2\md(z_t) \msc(z_t)}{(d-1)\sqrt{d-1}}\right)(L_{ol_\al}L_{il_\beta}+L_{il_\al}L_{o l_\beta})+\frac{2\md(z_t)}{d-1}\left(L_{i l_\al }L_{il_\beta}+L_{ol_\al}L_{ol_\beta}\right)\right).
\end{split}\end{align}
The key observation is that for $\al\in \sfA_i$, $L_{il_\al}=(-\msc/\sqrt{d-1})L_{ol_\al}$. The equivalent statement holds for $\beta\in \sfA_i$. Thus we can rewrite the second line of \eqref{e:Fuv2} as
\begin{align*}
  &\phantom{{}={}}\left(\frac{1}{\sqrt{d-1}}+\frac{2\md(z_t) \msc(z_t)}{(d-1)\sqrt{d-1}}\right)(L_{ol_\al}L_{il_\beta}+L_{il_\al}L_{o l_\beta})+\frac{2\md(z_t)}{d-1}\left(L_{i l_\al } L^{\circ}_{il_\beta}+L_{ol_\al}L_{ol_\beta}\right)\\
  &=
  -\frac{2\msc(z_t)}{\sqrt{d-1}}\left(\frac{1}{\sqrt{d-1}}+\frac{2\md(z_t) \msc(z_t)}{(d-1)\sqrt{d-1}}\right)L_{ol_\al}L_{ol_\beta}+\frac{2\md(z_t)}{d-1}\left(1+\frac{\msc^2(z_t)}{d-1}\right)L_{ol_\al}L_{ol_\beta}\\
  &=
  \left(\frac{2(\md(z_t)-\msc(z_t))}{d-1}-\frac{2\md(z_t) \msc^2(z_t)}{(d-1)^2}\right)L_{ol_\al}L_{ol_\beta},
\end{align*}
where the coefficient vanishes, as
\begin{align*}
    \md(z_t)=\frac{\msc(z_t)}{1-\msc^2(z_t)/(d-1)}, \quad \frac{2(\md(z_t)-\msc(z_t))}{d-1}-\frac{2\md(z_t) \msc^2(z_t)}{(d-1)^2}=0.
\end{align*}
We conclude that after the substitutions given in \eqref{e:sub1},  \eqref{e:dierxiang} vanishes.

\end{proof}

\subsection{Proofs from \Cref{s:first_step} and \Cref{s:general}}\label{sec:iterationproofs}

In this section, we prove \Cref{l:fyibu}, \Cref{p:tS-Sdiff} and \Cref{l:yibu}.
The proofs of \Cref{l:fyibu} and \Cref{l:yibu} are essentially the same, so we will only prove \Cref{l:fyibu}.
\begin{proof}[Proof of \Cref{l:fyibu}]
We notice that the error $\cE$ in \eqref{e:fhatRi} comes from
replacing $(\widetilde G_{oo}^{(i)}-Y_t)$ and $\widetilde \cW_j$ in \eqref{e:freplace0} by $\wh B_0$ and $\widehat \cW_j$ respectively. 
$\cE$ is an $\OO(1)$-weighted sum of terms in the following form 
\begin{align}\begin{split}\label{e:freplace2}
{Z^{-1}_{\cF^+}}\sum_{\bfi^+}\bE\left[I(\cF^+,\cG)\bm1(\cG,\tcG\in \Omega) \Delta B_0\prod_{j=1}^{p-1} \Delta \cW_j\right].
    \end{split}
\end{align}
where $\Delta B_0\in\{\wh B_0, \cZ_0, \cE_0\}, \Delta \cW_j\in\{\wh \cW_j, \widehat\cE_j\}$, where either $\Delta B_0\neq \wh B_0$ or $\Delta \cW_j\neq\wh \cW_j$ for some $1\leq j\leq p-1$.  With this notation, thanks to \eqref{eq:local_law}, \eqref{e:Werror2} and \eqref{e:Bsmall}, we have 
$|\Delta B_0|\lesssim N^{-\fb/2} $. It follows from \eqref{e:Bsmall} and \eqref{e:1_Sbound}, $|\Delta \cW_j|\lesssim |Q_t-Y_t|+(d-1)^{8\ell}\Upsilon \Phi$ for $1\leq j\leq p-1$. There are several cases.

If $\Delta \cW_{j'}=\wh\cE_j$ for some $1\leq j'\leq p-1$, noticing that $|\wh\cE_{j'}|\lesssim ((d-1)^{8\ell}\Phi)^2$, we conclude that
\begin{align*}
    |\eqref{e:freplace2}|&\lesssim {Z^{-1}_{\cF^+}}\sum_{\bfi^+}\bE\left[I(\cF^+,\cG)\bm1(\cG,\tcG\in \Omega) N^{-\fb/2} |\widehat \cE_{j'}|\prod_{j\in\qq{p-1}\setminus\{j'\}} |\Delta \cW_{j}|\right]\\
    &\lesssim N^{-\fb/2} \bE[((d-1)^{8\ell}\Phi)^2(|Q_t-Y_t|+(d-1)^{8\ell}\Upsilon \Phi)^{p-2}]
    \lesssim N^{-\fb/4}\bE[\Psi_p].
\end{align*}
In the remaining cases, $\Delta \cW_j=\wh \cW_j$ for all $1\leq j\leq p-1$.
If $\Delta B_0=\cE_0$ and $\Delta \cW_j=(Q_t-Y_t)$ for all $1\leq j\leq p-1$, then we can write
\begin{align}\begin{split}\label{e:error1}
\eqref{e:freplace2}
&={Z^{-1}_{\cF^+}}\sum_{\bfi^+}\bE\left[I(\cF^+,\cG)\bm1(\cG,\tcG\in \Omega)\cE_0 (Q_t-Y_t)^{p-1}\right].
\end{split}\end{align}
We recall $\cE_0$ from \eqref{e:defCE}. The right-hand side of \eqref{e:error1} is given by
\begin{align}\begin{split}\label{e:Eterm1} &\phantom{{}={}}\sum_{\bfi^+}\bE\left[\frac{I(\cF^+,\cG)\bm1(\cG,\tcG\in \Omega)}{Z_{\cF^+}}\left(\frac{\msc^{2\ell}(z_t)} {(d-1)^{\ell+1}}\sum_{\al,\beta\in\sfA_i}(\wt G^{(\bT)}_{c_\alpha c_\beta}-G^{(b_\alpha b_\beta)}_{c_\alpha c_\beta})\right) (Q_t-Y_t)^{p-1}\right]\\
&+\OO\left(\sum_{\bfi^+}\bE\left[\frac{I(\cF^+,\cG)\bm1(\cG,\tcG\in \Omega)}{Z_{\cF^+}}\left(\frac{1}{N^{\fb/2}} \sum_{\al,\beta\in \qq{\mu}}|\wt G^{(\bT)}_{c_\alpha c_\beta}-G^{(b_\alpha b_\beta)}_{c_\alpha c_\beta}|+\frac{1}{N^2}\right) (Q_t-Y_t)^{p-1}\right]\right).
\end{split}\end{align}
By \Cref{lem:task2} and \eqref{e:bPi1}, by taking $\Pi=\bm1(\cG\in \Omega)(|Q_t-Y_t|+(d-1)^{8\ell}\Upsilon \Phi)^{p-1}$, we can bound the two terms in \eqref{e:Eterm1} as $\OO((d-1)^{2\ell}\bE[\Psi_p])$ and $\OO(N^{-\fb/4}\bE[\Psi_p])$ respectively.


Otherwise, if $\Delta B_0=\cE_0$ and $\Delta \cW_j\neq (Q_t-Y_t)$ for some $1\leq j\leq p-1$, then $|\Delta \cW_j|\lesssim ((d-1)^{8\ell}\Upsilon \Phi)$, and 
\begin{align}\label{e:DeltaS_product}
 \prod_{j=1}^{p-1} |\Delta \cW_j|
    &\lesssim (d-1)^{8\ell}\Upsilon \Phi(|Q_t-Y_t|+(d-1)^{8\ell}\Upsilon \Phi)^{p-2}.
\end{align}
In this case, we can bound \eqref{e:freplace2} by taking $\Pi=\bm1(\cG\in \Omega) (d-1)^{8\ell}\Upsilon \Phi(|Q_t-Y_t|+(d-1)^{8\ell}\Upsilon \Phi)^{p-2}$ in \Cref{lem:task2}, and using \eqref{e:bPi2},
\begin{align*}
\eqref{e:freplace2}\lesssim {Z^{-1}_{\cF^+}}\sum_{\bfi^+}\bE\left[I(\cF^+,\cG)\bm1(\cG,\tcG\in \Omega)|\cE_0|\Pi\right]
\lesssim N^{-\fb/2}\bE[ \Psi_p].
\end{align*}

If $\Delta B_0=\cZ_0$ and $\cZ_0$ contains at least two terms of the form of \eqref{e:W-term}, then thanks to \eqref{e:Werror2}, with overwhelmingly high probability over $Z$
\begin{align*}
   |\cZ_0|\lesssim (d-1)^{3\ell}t N^{2\fo}\Phi.  
\end{align*}
Together with \eqref{e:DeltaS_product}, we conclude that \eqref{e:freplace2} is bounded by $N^{-\fb}\bE[\Psi_p]$. If $\cZ_0$ contains exact one term of the form of \eqref{e:W-term}, we recall from \eqref{e:W_decompose} that $\prod_{j=1}^{p-1} \widehat \cW_j$ decomposes as an $\OO(1)$-weighted sum of terms of the form $(d-1)^{(3h+3\sum_{j=1}^{p-1} h_{j1})\ell}\widehat R_{\bfi^+}$, where $R_{\bfi^+}\in \Adm(h,\bmh,\cF^+,\cG)$ with  $h\geq 0$, and $\bmh=[h_{jk}]_{1\leq j\leq p-1, 0\leq k\leq 2}$ satisfying \eqref{e:defr}. Thanks to \Cref{p:WRbound}, we have 
\begin{align*}
{Z^{-1}_{\cF^+}}\sum_{\bfi^+}\bE\left[I(\cF^+,\cG)\bm1(\cG,\tcG\in \Omega) \cZ_0\prod_{j=1}^{p-1} \widehat \cW_j\right] \lesssim N^{-\fb}\bE[\Psi_p].
\end{align*}
This finishes the proof of \Cref{l:fyibu}.
\end{proof}

\begin{proof}[Proof of \Cref{p:tS-Sdiff}]
The first and third statements follow from  \Cref{c:Q-Ylemma}. In the rest, we discuss the case that $r_{j0}=2$, and $\cW_j$ is of the following form
\begin{align*}
   \{1-\del_1 Y_\ell, -t\del_2 Y_\ell\}\times \frac{1}{dN}\sum_{u_j,v_j}A_{u_j v_j}F_{u_jv_j},
\end{align*} 
where $F_{u_jv_j}$ is a product of $r_{j1}$ factors of the form $\{ G^{\circ}_{su_j}, G^{\circ}_{sv_j}\}_{s\in \cK}$, or $\{(\Av G^{\circ})_{o'u_j}, (\Av G^{\circ})_{o'v_j}\}_{(i',o')\in \cC\setminus \cC^\circ}$, $r_{j2}$ factors of the form $\{L_{su_j}, L_{sv_j}\}_{s\in \cK}$, or $\{(\Av L)_{o'u_j}, (\Av L)_{o'v_j}\}_{(i',o')\in \cC\setminus \cC^\circ}$, and an arbitrary number of factors of the form $G_{u_jv_j}, 1/G_{u_ju_j}$.

In the following we compute $(d-1)^{3r_{j1}\ell}\wt \cW_j$:
\begin{align}\begin{split}\label{e:tS-Sdiff}
    (d-1)^{3r_{j1}\ell}\wt \cW_j
    &=\{1-\del_1  Y_\ell, -t\del_2  Y_\ell\}\times\frac{1}{dN}\sum_{u_j,v_j}\wt A_{u_j v_j}(d-1)^{3r_{j1}\ell}\wt F_{u_jv_j}\\
    &+\{-(\del_1 \wt Y_\ell-\del_1 Y_\ell), -t(\del_2 \wt Y_\ell-\del_2 Y_\ell)\}\times\frac{1}{dN}\sum_{u_j,v_j} \wt A_{u_j v_j} (d-1)^{3r_{j1}\ell}\wt F_{u_jv_j}.
\end{split}\end{align}
By \eqref{e:dtY-dYexp}, 
$|\del_1 \wt Y_\ell-\del_1 Y_\ell|,|\del_2 \wt Y_\ell-\del_2 Y_\ell|\lesssim (d-1)^{8\ell} \Phi$. Thus by \eqref{e:1_Sbound}, the second line in \eqref{e:tS-Sdiff} is bounded by $((d-1)^{8\ell} \Phi)^2$.

In the following we analyze the first term on the right-hand of \eqref{e:tS-Sdiff}; we first split it as
\begin{align}\begin{split}\label{e:tF-Fdiff2}
   \sum_{u_j,v_j\in\qq{N}}\wt A_{u_jv_j} (d-1)^{3r_{j1}\ell}\wt F_{u_jv_j}
&=\sum_{\{u_j,v_j\}\notin \{\{l_\al, a_\al\},\{b_\al, c_\al\}\}_{\al\in\qq{\mu}} }A_{u_jv_j}(d-1)^{3r_{j1}\ell}\wt F_{u_jv_j}\\
&+\sum_{\{u_j,v_j\}\in \{\{l_\al, c_\al\},\{a_\al, b_\al\}\}_{\al\in\qq{\mu}} }(d-1)^{3r_{j1}\ell}\wt F_{u_jv_j}.
\end{split}\end{align}

In the following, for any $h\geq 0$, we denote $R_h$ a product of $h$ factors of the form $\{ G^{\circ}_{ss'}\}_{s,s'\in \cK^+}$ (W-product from  \Cref{d:W-product}).
There are several cases for the factors of $(d-1)^{3r_{j1}\ell}F_{u_jv_j}$:
\begin{enumerate}
   \item \label{i:sa}
For $s\in \cK, \{i',o'\}\in \cC\setminus \cC^\circ, w\in \{u_j, v_j\}$, the terms $ G^{\circ}_{s w}$ and $(\Av G^\circ)_{o'w}$ each contributes one to $r_{j1}$. By the third and fourth statements in \Cref{lem:generalQlemma}, we have
\begin{align}\begin{split}\label{e:tG-Gexp1copy}
&\phantom{{}={}} (d-1)^{3\ell}\wt G^{\circ}_{sw}
=(d-1)^{3\ell}G^{\circ}_{sw}+\cU+\OO(N^{-2})\\
&+\bm1(s\in\{i,o\})\left(\frac{\sum_{\sfJ\in\{b,c\}}\sum_{\al\in \qq{\mu}} \fc(\sfJ, \bm1(\al\in \sfA_i))  (d-1)^{3\ell}G^{\circ}_{\sfJ_\al w}}{(d-1)^{\ell/2}}+(d-1)^{3\ell}(\Av G^{\circ})_{ow}\right),\\
&\phantom{{}={}} (d-1)^{3\ell}(\Av  \wt G^{\circ})_{o'w}=(d-1)^{3\ell}(\Av  G^\circ)_{o'w}+\cU+\OO(N^{-2}),
\end{split}\end{align}
where $\cU$ is an $\OO(1)$-weighted sum of terms in the form $(d-1)^{3h\ell}R_h \times (d-1)^{3\ell}G^{\circ}_{xw}$, or $(d-1)^{6h\ell}R_h L_{xw}$  with $x\in\cK^+, w\in\{u_j, v_j\}$ and $h\geq 1$. 

Also for $s\in \cK, \{i',o'\}\in \cC\setminus \cC^\circ, w\in \{u_j, v_j\}$, the terms  $L_{sw}$ and $ (\Av L)_{o'w}$ each contributes one to $r_{j2}$, and we have, by the third and fourth statements in \Cref{lem:generalQlemma},
\begin{align}\begin{split}\label{e:tL-Lexp1copy}
    &\wt L_{sw}=L_{sw}+\left(\frac{ \sum_{\sfJ\in \{b,c\}}\sum_{\al\in \qq{\mu}} \fc(\sfJ, \bm1(\al\in \sfA_i)) L_{\sfJ_\al w}}{(d-1)^{\ell/2}}+(\Av L)_{ow}\right)\bm1(s\in \{i,o\}),\\ 
    &(\Av \wt L)_{o'w}=(\Av L)_{o'w}.
\end{split}\end{align}

\item \label{i:sb}
For $G_{u_j v_j}, 1/G_{u_ju_j}$, by the second statement in \Cref{lem:generalQlemma},
we have 
\begin{align}\begin{split}\label{e:1/tG-1/Gexp1copy}
   &\tG_{u_jv_j}=G_{u_j v_j}-\md(z_t) \msc(z_t)(\wt H-H)_{u_j v_j}+\cU+\OO(N^{-2}),\\
    &1/\tG_{u_j u_j}=1/G_{u_j u_j}+\cU+\OO(N^{-2}),
\end{split}\end{align}
where $\cU$ is an $\OO(1)$-weighted sum of terms in the form
$(d-1)^{3h\ell}R_h (d-1)^{3h_1\ell}R'_{h_1h_2}$, with $h\geq 0$ and $ h_1+h_2\geq 2$ even; $R'_{h_1, h_2}$ contain $h_1$ factors of the form $ G^{\circ}_{xu_j},  G^{\circ}_{xv_j}$, $h_2$ factors of the form $L_{xu_j}, L_{xv_j}$, with $x\in\cK^+$ and an arbitrary number of factors $G_{u_jv_j}, 1/G_{u_ju_j}$;
\end{enumerate}
Based on this decomposition, we find that $(d-1)^{3r_{j1}\ell}\wt F_{u_jv_j}$ is an $\OO(1)$-weighted sum of terms of the form \eqref{e:yibufen}: 
\begin{align}\begin{split}\label{e:Fudecomp}
        & (d-1)^{6h_j'\ell}R_{h_j}'\times \frac{(d-1)^{3h_{j1}\ell}}{(d-1)^{(g_{j1}+g_{j2})\ell/2}}\sum_{\bm\theta_j} \frac{\{1-\del_1 Y_\ell, -t\del_2 Y_\ell\}}{Nd}\sum_{u_j\sim v_j\in \qq{N}}R''_{h_{j1}-g_{j1}, h_{j2}-g_{j2}}\prod_{m=1}^{g_{j1}+g_{j2}} D_{\theta_m}.
    \end{split}\end{align}
    Here $R_{h_j'}$ is the product of these $R_h$ from $\cU$ in \eqref{e:tG-Gexp1copy} and \eqref{e:1/tG-1/Gexp1copy}, with a total count $h_j$. 
    Here $D_{\theta_m}$ are from $ G^{\circ}_{\sfJ_\al w}$ in \eqref{e:tG-Gexp1copy} with count $g_{j1}$, or $L_{\sfJ_\al w}$ in \eqref{e:tL-Lexp1copy} with count $g_{j2}$. $R''_{h_{j1}-g_{j1}, h_{j2}-g_{j2}}$ contains the remaining $h_{j1}-g_{j1}$ factors of the following form: $\{ (\Av G^{\circ})_{o'w},   G^{\circ}_{xw}\}_{ x\in \cK^+, (i',o')\in \cC\setminus \cC^\circ, w\in \{u_j,v_j\}}$, 
$h_{j2}-g_{j2}$ factors of the form 
$\{(\Av L)_{o'w},  L_{xw}\}_{ x\in \cK^+, (i',o')\in \cC\setminus \cC^\circ, w\in \{u_j,v_j\}}$, and an arbitrary  number of $G_{u_jv_j}, 1/G_{u_j u_j}$; Then we have $h_{j1}+h_{j2}\geq r_{j1}+r_{j2}$ is even, and $ h_{j2}\geq r_{j2}$.

Next we show that \eqref{e:yibufen} with the summation over $\{u_j,v_j\}\in \{\{l_\al, a_\al\},\{b_\al, c_\al\}\}_{\al\in\qq{\mu}}$, and the last term in \eqref{e:tF-Fdiff2} are $\OO(1)$-weighted sum of terms in \eqref{e:erbufen}. We notice that in both cases, $\{u_j,v_j\}\in \{\{l_\al, a_\al\},\{b_\al, c_\al\}, \{l_\al, c_\al\},\{a_\al, b_\al\}\}_{\al\in\qq{\mu}}$.
Then for $s\in \cK$, $\{i',o'\}\in \cC\setminus \cC^\circ$, $w\in \{u_j, v_j\}$, 
$(d-1)^{3\ell}(\Av G^{\circ})_{o'w},   (d-1)^{3\ell} G^{\circ}_{sw}, (\Av L)_{o'w},  L_{sw}$ are of form $\OO((d-1)^{6h\ell})R_h$ with $h=0,1$.
\eqref{e:tG-Gexp1copy} implies that, up to error $\OO(N^{-2})$,
$(d-1)^{3\ell} \wt G^{\circ}_{s w}, (d-1)^{3\ell}(\Av  \wt G^{\circ})_{o'w}$ are $\OO(1)$-weighted sum of terms of the form $ 
(d-1)^{6h\ell}R_h$ with $h\geq 1$; \eqref{e:tL-Lexp1copy} implies that 
$\wt L_{sw}=L_{sw}+\OO(1), (\Av \wt L )_{ow}=(\Av L )_{ow}+\OO(1)$, which are bounded numbers. Moreover, \eqref{e:1/tG-1/Gexp1copy} implies that, up to error $\OO(N^{-2})$, $(\wt G_{u_j v_j}-G_{u_j v_j}), (1/\wt G_{u_j u_j}-1/G_{u_j u_j})$ are $\OO(1)$-weighted sums of terms of the form $ 
(d-1)^{3h\ell}R_h$ with $h\geq 0$. 
This yields \eqref{e:erbufen} with $h_j'\geq r_{j1}$, where $h_j'$ is given by the sum of the values of $h$ from the preceding discussion.

Next we prove  \eqref{e:S'bound2}. 
If $g_{j2}\geq 1$ in \eqref{e:S'bound2}, then $R'_{h_{j1}, h_{j2}}$ is nonzero only if $\theta_{g_{j1}+1}=\theta_{g_{j1}+2}=\cdots=\theta_{g_{j1}+g_{j2}}=\theta$. In this case $f_j\leq g_{j1}+1$, so we can bound \eqref{e:S'bound2} as
\begin{align*}
    &\phantom{{}={}}\frac{ \Upsilon}{dN}\frac{(d-1)^{(f_j+3h_{j1})\ell}}{(d-1)^{(g_{j1}+g_{j2})\ell/2}}\sum_{u_j\sim v_j}|R'_{h_{j1} h_{j2}}|
    \lesssim \frac{ \Upsilon}{dN}(d-1)^{(3h_{j1}+(2+g_{j1}-g_{j2})/2)\ell}\sum_{u_j\sim v_j}|R'_{h_{j1} h_{j2}}|\\
    &\lesssim \Upsilon(d-1)^{(3h_{j1} +(2+g_{j1}-g_{j2})/2)\ell} N^{-\max\{h_{j1}-1,0\}\fb} N^\fo\sqrt{\Phi/N}\lesssim (d-1)^{8\ell}\Upsilon\sqrt{\Phi/N},
\end{align*}
where in the first inequality we used $f_j\leq g_{j1}+1$; in the second inequality, we used \eqref{e:Sjbound}; in the third inequality, we used  that $g_{j2}\geq 1$ and $g_{j1}\leq h_{j1}$. 
If $g_{j2}=0$ and $h_{j2}=0$ in \eqref{e:hatFdef}, then $h_{j1}\geq 2$ and $f_j\leq g_{j1}$. By \eqref{e:Sjbound},  \eqref{e:S'bound2} is bounded by
\begin{align*}
    \Upsilon(d-1)^{(3h_{j1} +g_{j1}/2)\ell} N^{-(h_{j1}-2)\fb} N^\fo\Phi\lesssim (d-1)^{8\ell}\Upsilon \Phi,
\end{align*}
where we used that $g_{j1}\leq h_{j1}$. If $g_{j2}=0$ and $h_{j2}\geq 1$, similarly, we have \eqref{e:hatFdef} is bounded by $(d-1)^{8\ell}\Upsilon \sqrt{\Phi/N}$.

\end{proof}

\section{Refined Analysis}
\label{s:refined}
In this section, we will prove \Cref{lem:deletedalmostrandom} and \Cref{lem:task2} in \Cref{s:removeonevertex}, as well as \Cref{l:first_term} in \Cref{s:first_term}. Before proceeding, we collect some refined estimates on the Schur complement formula in \Cref{s:Schur_revisit}.

\subsection{Setting and notation}
\label{s:setting7}
In this section, let $d\geq 3$ and $\cG$ be a $d$-regular graph on $N$ vertices. Let $\cF = (\bfi, E)$ be a forest as in \eqref{e:cFtocF+}, with switching edges $\cK$, core edges $\cC$ and unused core edges $\cC^\circ$. We view $\cF$ as a subgraph of $\cG$. We construct $\cF^+ = (\bfi^+, E^+)$ (as in \eqref{e:cF++}) by performing a local resampling around $(i, o) \in \cC^\circ$ with resampling data ${\bf S}=\{(l_\al, a_\al), (b_\al, c_\al)\}_{\al\in \qq{\mu}}$ where $\mu=d(d-1)^{\ell}$. We denote $\cT=\cB_\ell(o,\cG)$ with vertex set $\bT$, $\bW=\{b_1, b_2, \cdots, b_\mu\}$ and  the switched graph $\widetilde \cG = T_\bfS(\cG)$. 
We recall the sets $\Omega$ and $\oOmega$ of $d$-regular graphs from \Cref{def:omegabar} and \Cref{thm:prevthm0}, and the indicator function $I(\cF,\cG)$ from \eqref{def:indicator}. In this section, we will typically consider the forests $\cF=\{i,o\}$, or $\cF=\{\{i,o\}, \{b,c\}\}$.

Fix $d\geq 3$. We recall the spectral domain $\bf D$ from \eqref{e:D}, and parameters $\fo\ll \ft\ll\fb\ll\fc\ll\fg$ from \eqref{e:parameters}. Fix time $t\leq N^{-1/3+\ft}$. We recall $\varrho_d(x,t)$, $\md(z,t)$ and $E_t$ from \eqref{e:defrhodt} and \eqref{e:edgeeqn2}. For any parameter $z\in \bf D$ we denote $\eta=\Im[z]$, $\kappa=\min\{|\Re[z]-E_t|, |\Re[z]+E_t|\}$, and $z_t=z+t\md(z,t)=z+t\md(z_t)$ (recall from \eqref{e:defw}). 

We recall the matrix $H(t)$, its Green's function $G(z,t)$, its Stieltjes transform $m_t(z)$, the quantities  $Q_t(z)$,  $Y_t(z)=Y_\ell(Q_t(z),z+tm_t(z))$, and   $X_t(z)=X_\ell(Q_t(z),z+tm_t(z))$ from \eqref{e:Ht} , \eqref{e:Gt}, \eqref{def_mtz}, \eqref{e:Qsum}, and  \eqref{e:defYt}. 
We recall the control parameters $\Phi(z), \Upsilon(z)$ and $\Psi_p(z)$ from \eqref{e:defPhi} and \eqref{eq:phidef}. We recall the local Green's function $L(z,t)=L(z,t,\cF^+,\cG)$ from \eqref{e:local_Green}, and $G^\circ(z,t)=G(z,t)-L(z,t)$ from \eqref{e:G-L}. We denote the corresponding quantities for the switched graph $\tcG$ as $\wt H(t)$, $\wt G(z,t)$, $\wt m_t(z)$, $\wt Q_t(z)$,  $\wt Y_t(z)$, $\wt X_t(z)$,$\wt \Phi(z)$, $\wt \Upsilon(z)$,$\wt \Psi_p(z)$, $\wt L(z,t)$ and $\wt G^\circ(z,t)$. If the context is clear, we may omit the dependence on $z$ and $t$.

\subsection{Schur complement formula revisit}
\label{s:Schur_revisit}
Adopt the notation of \Cref{s:setting7}. Then the normalized adjacency matrix $\widetilde H^{(\bT)}$ of $\tcG^{(\bT)}$ and $Z^{(\bT)}$ are in the block form
\begin{align*}
    \widetilde H^{(\bT)}=
    \left[
    \begin{array}{cc}
        \wt H^{(\bT)}_{\bW} & \wt B^\top\\
        \wt B & \widetilde H^{(\bT)}_{\bW^\complement}
    \end{array}
    \right],\quad
  Z^{(\bT)}=
    \left[
    \begin{array}{cc}
        Z^{(\bT)}_{\bW} & Z^{(\bT)}_{\bW\bW^\complement}\\
        Z^{(\bT)}_{\bW^\complement\bW} & Z^{(\bT)}_{\bW^\complement}
    \end{array}
    \right].
\end{align*}
We also denote the Green's function of $\cG^{(\bT)}$ and $\wt\cG^{(\bT)}$ as $ G^{(\bT)}$ and $\wt G^{(\bT)}$ respectively.

In this section, we investigate the error from replacing $\wt G^{(\bT)}_{c_\al c_\beta}$ with $ G_{c_\al c_\beta}^{(b_\al b_\beta)}$. 
 Then $ G_{c_\al c_\beta}^{(b_\al b_\beta)}$ can be obtained from $\wt G^{(\bT)}_{c_\al c_\beta}$ through the following steps. First, we remove $\bW=\{b_1, b_2, \cdots, b_\mu\}$, which gives $G_{c_\al c_\beta}^{(\bT \bW)}$; we then add $\bW\setminus \{b_\al, b_\beta\}$ back, which gives $G_{c_\al c_\beta}^{(\bT b_\al b_\beta)}$; finally we add $\bT$ back, which gives $G_{c_\al c_\beta}^{(b_\al b_\beta)}$. The errors from these replacements are explicit, thanks to the Schur complement formulas \eqref{e:Schur1}:
\begin{align}
\label{e:removeW}&\wt G^{(\bT)}_{c_\al c_\beta}-G^{(\bT\bW)}_{c_\al c_\beta}=(G^{(\bT\bW)}(\wt B+\sqrt{t} Z^{(\bT)}_{\bW^\complement \bW})  \wt G^{(\bT)}|_{\bW} (\wt B^\top+\sqrt{t} Z^{(\bT)}_{\bW\bW^\complement }) G^{(\bT\bW)})_{c_\al c_\beta},\\
 \label{e:addW}&G^{(\bT\bW)}_{c_\al c_\beta}-G^{(\bT b_\al b_\beta)}_{c_\al c_\beta}=-(G^{(\bT\bW)}(G^{(\bT b_\al b_\beta)}|_{\bW\backslash\{b_\alpha,b_\beta\}})^{-1}   G^{(\bT b_\al b_\beta)})_{c_\al c_\beta},\\
 \label{e:schur_removeT}
&G_{c_\al c_\beta}^{(b_\al b_\beta)}-G^{(\bT b_\al b_\beta)}_{c_\al c_\beta}=(G^{(b_\al b_\beta)}(G^{(b_\al b_\beta)}|_\bT)^{-1}G^{(b_\al b_\beta)})_{c_\al c_\beta}.
\end{align}

The following lemma provide leading order terms for  the replacement errors associated with the equations \eqref{e:removeW}, \eqref{e:addW} and \eqref{e:schur_removeT}. 
\begin{lemma}\label{l:diffG1}
Adopt the notation of \Cref{s:setting7} with $z\in \bf D$, $\cF=\{i,o\}$ and $\cT=\cB_\ell(o,\cG)$ having the vertex set $\bT$. We assume that $\cG, \tcG\in \Omega$ and $I(\cF^+,\cG)=1$. For any indices $\al,\beta\in\qq{\mu}$, the following holds with overwhelmingly high probability over $Z$: 
\begin{enumerate}
\item   
    The difference $\wt G^{(\bT)}_{c_\al c_\beta}-G^{(\bT\bW)}_{c_\al c_\beta}$, is given by
    \begin{align}\label{e:diffG1}
    \wt G^{(\bT)}_{c_\al c_\beta}-G^{(\bT\bW)}_{c_\al c_\beta}=\frac{\md(z_t)}{d-1}\sum_{\gamma\in\qq{\mu}}\left( \sum_{x\in \cN_\gamma}G_{c_\al x}^{(\bT\bW)}\right)\left(\sum_{x\in \cN_\gamma}G_{c_\beta x}^{(\bT\bW)}\right)+\cE,
    \end{align}
   where $\cN_\gamma=\{x\neq c_\gamma: x\sim b_\gamma \text{ in }\cG\}\cup \{a_\gamma\}$, which enumerates the adjacent vertices of $b_\gamma$ in $\tcG$, and
    \begin{align*}
     |\cE|\lesssim N^{-\fb/2}\sum_{\gamma\in\qq{\mu}}\sum_{x\in \cN_\gamma}(|G_{c_\al x}^{(\bT\bW)}|^2+|G_{c_\beta x }^{(\bT\bW)}|^2)+N^{-\fb }\Phi.
    \end{align*}
    Moreover,
    \begin{align}\label{e:total_diffG1}
    |\wt G^{(\bT)}_{c_\al c_\beta}-G^{(\bT\bW)}_{c_\al c_\beta}|\lesssim \sum_{\gamma\in\qq{\mu}} \sum_{x\in \cN_\gamma}(|G_{c_\al x}^{(\bT\bW)}|^2+|G_{c_\beta x}^{(\bT\bW)}|^2)+N^{-\fb }\Phi.
    \end{align}
\item
    The difference $G^{(\bT\bW)}_{c_\al c_\beta}-G^{(\bT b_\al b_\beta)}_{c_\al c_\beta}$ is given by
    \begin{align}\label{e:diffG2}
    G^{(\bT\bW)}_{c_\al c_\beta}-G^{(\bT b_\al b_\beta)}_{c_\al c_\beta}=-\frac{1}{\md(z_t)}\sum_{\gamma\in\qq{\mu}\setminus\{\al,\beta\}}G^{(\bT b_\al b_\beta )}_{c_\al b_\gamma}  G^{(\bT b_\al b_\beta)}_{b_{\gamma} c_\beta}+\cE,
    \end{align}
where
\begin{align*}
|\cE|
\lesssim N^{-\fb/2} \sum_{\gamma\in\qq{\mu}\setminus\{\al,\beta\}}|G^{(\bT b_\al b_\beta )}_{c_\al b_\gamma}|^2+|G^{(\bT b_\al b_\beta)}_{ c_\beta b_{\gamma}}|^2.
\end{align*}
Moreover,  
\begin{align}\label{e:total_diffG2}
    |G^{(\bT\bW)}_{c_\al c_\beta}-G^{(\bT b_\al b_\beta)}_{c_\al c_\beta}|\lesssim \sum_{\gamma\in\qq{\mu}\setminus\{\al,\beta\}}(|G^{(\bT b_\al b_\beta )}_{c_\al b_\gamma}|^2+ |G^{(\bT b_\al b_\beta)}_{c_\beta b_{\gamma} }|^2).
\end{align}
\item
The difference $G^{(\bT b_\al b_\beta)}_{c_\al c_\beta}-G_{c_\al c_\beta}^{(b_\al b_\beta)}$ is given by 
\begin{align}\label{e:diffG3}
G^{(\bT b_\al b_\beta)}_{c_\al c_\beta}-G_{c_\al c_\beta}^{(b_\al b_\beta)}=-\sum_{\dist(x,o)=\ell\atop  x\sim x'\in \bT}\left(\frac{1}{\sqrt{d-1}}G_{c_\al x}^{(b_\al b_\beta)}G^{(b_\al b_\beta)}_{x' c_\beta}
+\frac{1}{\msc(z_t)}G_{c_\al x}^{(b_\al b_\beta)}G^{(b_\al b_\beta)}_{x c_\beta}\right)+\cE,
\end{align}
where
    \begin{align*}
       |\cE|
       \lesssim N^{-\fb/2}\sum_{x\in \bT}(|G^{(b_\al b_\beta)}_{c_\al x}|^2+|G^{(b_\al b_\beta)}_{ c_\beta x}|^2+|(ZG^{(b_\al b_\beta)})_{x c_\beta}|^2 ).
    \end{align*}
Moreover,
\begin{align}\label{e:total_diffG3} 
|G^{(\bT b_\al b_\beta)}_{c_\al c_\beta}-G_{c_\al c_\beta}^{(b_\al b_\beta)}|\lesssim 
\sum_{x\in \bT}(|G_{c_\al x}^{(b_\al b_\beta)}|^2+|G^{(b_\al b_\beta)}_{c_\beta x}|^2).
\end{align}
\item The following holds 
\begin{align}\begin{split}\label{e:Greplace}
&|G^{(\bT \bW)}_{c_\al a_\beta}-G_{c_\al a_\beta}^{(b_\al l_\beta)}|\lesssim \sum_{\gamma\in\qq{\mu}\setminus\{\al\}}(|G^{(\bT b_\al )}_{c_\al b_\gamma}  |^2+|G^{(\bT b_\al )}_{ a_\beta b_{\gamma}}|^2)+\sum_{y\in \bT\setminus\{l_\beta\}}(|G_{c_\al y}^{(b_\al l_\beta)}|^2+|G^{(b_\al l_\beta)}_{ a_\beta y}|^2),
 \\
    &|G^{(\bT \bW)}_{c_\al x}-G_{c_\al x}^{(b_\al b_\beta)}|\lesssim \sum_{\gamma\in\qq{\mu}\setminus\{\al,\beta\}}(|G^{(\bT b_\al b_\beta )}_{c_\al b_\gamma}  |^2+|G^{(\bT b_\al b_\beta)}_{b_{\gamma} c_\beta}|^2)
    +\sum_{y\in \bT}(|G_{c_\al y}^{(b_\al b_\beta)}|^2+|G^{(b_\al b_\beta)}_{ x y}|^2),\text{ for } x\sim b_\beta,
\\
    &|G^{(\bT b_\al b_\beta)}_{c_\al b_\theta}-G_{c_\al b_\theta}^{(b_\al)}|\lesssim (|G^{(\bT b_\al )}_{c_\al b_\beta}  |^2+|G^{(\bT b_\al )}_{b_\theta b_\beta}|^2)
    +\sum_{y\in \bT}(|G_{c_\al y}^{(b_\al )}|^2+|G^{(b_\al)}_{ b_\theta y}|^2), \text{ for } \theta \in\qq{\mu}\setminus\{\al,\beta\}.
\end{split}\end{align}
\end{enumerate}
\end{lemma}
\begin{proof}[Proof of \Cref{l:diffG1}]
We can decompose the right-hand side of \eqref{e:removeW} as the sum of the following three terms $\eqref{e:removeW}=I+II+III$:
\begin{align*}
&I:=\frac{1}{d-1}\sum_{\gamma,\gamma'\in \qq{\mu}}\sum_{x\in \cN_\gamma\atop y\in \cN_{\gamma'}}G_{c_\al x}^{(\bT\bW)} \wt G^{(\bT)}_{b_\gamma b_{\gamma'}}   G^{(\bT\bW)}_{yc_\beta }\\
&II:=\frac{\sqrt t}{\sqrt{d-1}}\sum_{\gamma,\gamma'\in \qq{\mu}}\left(\sum_{y\in \cN_{\gamma'}}(G^{(\bT\bW)} Z)_{c_\al b_\gamma}\wt G^{(\bT)}_{b_\gamma b_{\gamma'}}  G^{(\bT\bW)}_{yc_\beta }+\sum_{x\in \cN_{\gamma}}G_{c_\al x}^{(\bT\bW)} \wt G^{(\bT)}_{b_\gamma b_{\gamma'}}  (ZG^{(\bT\bW)})_{b_{\gamma'}c_\beta }\right)
\\
&III:=t\sum_{\gamma,\gamma'\in \qq{\mu}} (G^{(\bT\bW)}Z)_{c_\al b_\gamma}\wt G^{(\bT)}_{b_\gamma b_{\gamma'}} (Z  G^{(\bT\bW)})_{b_{\gamma'} c_\beta},
\end{align*}
where $\cN_\gamma=\{x\neq c_\gamma: x\sim b_\gamma \text{ in }\cG\}\cup \{a_\gamma\}$.

Since $\tcG\in \Omega$ and $I(\cF^+,\cG)=1$, \eqref{eq:local_law} gives that for $\gamma\neq \gamma'$, $|\wt G^{(\bT)}_{b_\gamma b_{\gamma}}|\lesssim 1$ and $|\wt G^{(\bT)}_{b_\gamma b_{\gamma'}}|\lesssim N^{-\fb} $. The leading order term of $I$ is then for pairs $\gamma=\gamma'$, giving
\begin{align}\label{e:II_approx}
I= \frac{\md(z_t)}{d-1}\sum_{\gamma\in\qq{\mu}} \left(\sum_{x\in \cN_\gamma}G_{c_\al x}^{(\bT\bW)}\right)\left(\sum_{x\in \cN_\gamma}G_{c_\beta x}^{(\bT\bW)}\right)+\OO\left(N^{-\fb} \sum_{\gamma, \gamma'\in\qq{\mu}}\sum_{x\in \cN_\gamma\atop y\in \cN_{\gamma'}}|G_{c_\al x}^{(\bT\bW)}||G_{yc_\beta }^{(\bT\bW)}|\right).
\end{align}


For $II$ and $III$, thanks to \eqref{e:Werror} with overwhelmingly high probability over $Z$
\begin{align*}
    \left|(ZG^{(\bT\bW)})_{b_\gamma c_\al}\right|,\left|(ZG^{(\bT\bW)})_{b_{\gamma'} c_\beta}\right|\lesssim N^\fo \sqrt{\Phi}.
\end{align*}
It follows that 
\begin{align}\label{e:GWbound}
    |II|\lesssim N^\fo\sqrt t \Phi^{1/2}\sum_{x\in \cN_\gamma}(|G^{(\bT \bW)}_{c_\al x}|+|G^{(\bT \bW)}_{x c_\beta}|), \quad
    |III|\lesssim (d-1)^\ell N^{2\fo} t \Phi.
\end{align}
The claim \eqref{e:diffG1} follows from combining \eqref{e:II_approx} and \eqref{e:GWbound}.

To prove \eqref{e:diffG2}, we start with \eqref{e:addW}. Consider
\begin{align}\begin{split}\label{e:defcE1beta}
(G^{(\bT b_\al b_\beta)}(G^{(\bT b_\al b_\beta)}|_{\bW\backslash\{b_\alpha,b_\beta\}})^{-1}   G^{(\bT b_\al b_\beta)})_{c_\al c_\beta}
=\sum_{\gamma, \gamma'\in\qq{\mu}\setminus\{\al,\beta\}}
G^{(\bT b_\al b_\beta)}_{c_\al b_\gamma}(G^{(\bT b_\al b_\beta)}|_{\bW\backslash\{b_\alpha,b_\beta\}})^{-1}_{b_\gamma b_{\gamma'}}   G^{(\bT b_\al b_\beta)}_{b_{\gamma'} c_\beta}.
\end{split}\end{align}
The leading order term in \eqref{e:defcE1beta} is given by those with $\gamma=\gamma'$,
\begin{align}\begin{split}\label{e:IVerror}
&\eqref{e:defcE1beta}= \frac{1}{\md(z_t)}\sum_{\gamma\in\qq{\mu}\setminus\{\al,\beta\}}G^{(\bT b_\al b_\beta )}_{c_\al b_\gamma}  G^{(\bT b_\al b_\beta)}_{b_{\gamma} c_\beta}+\cE,\\
&\cE:=\sum_{\gamma,\gamma'\in\qq{\mu}\setminus\{\al,\beta\}}
G^{(\bT b_\al b_\beta )}_{c_\al b_\gamma}((G^{(\bT b_\al b_\beta)}|_{\bW\backslash\{b_\alpha,b_\beta\}})^{-1}_{b_\gamma b_{\gamma'}} -\delta_{\gamma\gamma'} /\md(z_t)) G^{(\bT b_\al b_\beta)}_{b_{\gamma'} c_\beta}.
\end{split}\end{align}
Thanks to \eqref{eq:local_law}, with overwhelmingly high probability over $Z$, $|G^{(\bT b_\al b_\beta)}_{b_\gamma b_{\gamma'}}-\md(z_t)\delta_{\gamma\gamma'}|\lesssim N^{-\fb}$. Thus 
$|(G^{(\bT b_\al b_\beta)}|_{\bW\backslash\{b_\alpha,b_\beta\}})^{-1}_{b_\gamma b_{\gamma'}} -\delta_{\gamma\gamma'} /\md(z_t)|\lesssim N^{-3\fb/4} $, and  the error $\cE$ in \eqref{e:IVerror} is bounded as
\begin{align}\label{e:IVerror2}
   |\cE|\lesssim N^{-3\fb/4} \sum_{\gamma,\gamma'\in\qq{\mu}\setminus\{\al,\beta\}}|G^{(\bT b_\al b_\beta )}_{c_\al b_\gamma}||G^{(\bT b_\al b_\beta)}_{b_{\gamma'} c_\beta}|\lesssim N^{-\fb/2} \sum_{\gamma\in\qq{\mu}\setminus\{\al,\beta\}}|G^{(\bT b_\al b_\beta )}_{c_\al b_\gamma}|^2+|G^{(\bT b_\al b_\beta)}_{ c_\beta b_{\gamma}}|^2.
\end{align}
The claim \eqref{e:diffG2} follows from combining \eqref{e:IVerror} and \eqref{e:IVerror2}.

To prove \eqref{e:diffG3}, we can rewrite the right-hand side of \eqref{e:schur_removeT} explicitly as
\begin{align}\label{e:schur_removeT2}
 -(G^{(b_\al b_\beta)}(G^{(b_\al b_\beta)}|_\bT)^{-1}G^{(b_\al b_\beta)})_{c_\al c_\beta}
    =-\sum_{x,y\in \bT} G^{(b_\al b_\beta)}_{c_\al x}(G^{(b_\al b_\beta)}|_\bT)^{-1}_{xy}G^{(b_\al b_\beta)}_{y c_\beta}.
\end{align}
Since $I(\cF^+;\cG)=1$, by \eqref{eq:local_law},  with overwhelmingly high probability over $Z$, for $x,y\in \bT$, $|G^{(b_\al b_\beta)}_{xy}-(H_\bT-z_t-\msc(z_t)\mathbb{I}^{\del})^{-1}_{xy}|\lesssim N^{-\fb}$, where $\mathbb{I}^\del_{xy}=\bm1(\dist_\cG(o,x)=\ell)\delta_{xy}$. Thus we have
\begin{align*}
|(G^{(b_\al b_\beta)}|_\bT)_{xy}^{-1}-(H^{(b_\al b_\beta)}-z_t-\msc(z_t)\mathbb{I}^{\del})_{xy}|\lesssim N^{-3\fb/4}, \text{ for }x,y\in \bT,
\end{align*}
and 
\begin{align}\label{e:diffG4}
    \eqref{e:schur_removeT2}=  -\sum_{x,y\in \bT}G^{(b_\al b_\beta)}_{c_\al x}(H-z_t-\msc(z_t)\mathbb{I}^\del)_{xy}G^{(b_\al b_\beta)}_{y c_\beta}+\OO\left(N^{-3\fb/4} \sum_{x,y\in \bT}|G^{(b_\al b_\beta)}_{c_\al x}||G^{(b_\al b_\beta)}_{y c_\beta}|\right).
\end{align}

For the summation over $x,y\in \bT$ in \eqref{e:diffG4},  if $x\in \bT$ but $\dist_\cG(x, o)<\ell$, then $\mathbb{I}^\del_{xy}=0$, and  we have
\begin{align}\begin{split}\label{e:HGexp}
&\phantom{{}={}}\sum_{y\in \bT}(H-z_t-\msc(z_t)\mathbb{I}^\del)_{xy}G^{(b_\al b_\beta)}_{y c_\beta}
=\sum_{y\in \bT}(H-z_t)_{xy}G^{(b_\al b_\beta)}_{y c_\beta}\\
&=\sum_{y\in \qq{N}\setminus\{b_\al, b_\beta\}}(H+\sqrt tZ -z)_{xy}G^{(b_\al b_\beta)}_{y c_\beta}
-\sqrt t(ZG^{(b_\al b_\beta)})_{x c_\beta}
-t\md(z,t) G^{(b_\al b_\beta)}_{x c_\beta}\\
&=-\sqrt t(ZG^{(b_\al b_\beta)})_{x c_\beta}
-t\md(z,t) G^{(b_\al b_\beta)}_{x c_\beta},
\end{split}\end{align} 
where for the second inequality, we used that $z_t=z+t\md(z,t)$; for the last equality, we used that by the definition of the Green's function, the first term on the second line is $\delta_{xc_\beta}=0$.
Thus by plugging \eqref{e:HGexp} into \eqref{e:diffG4}, it follows that 
\begin{align}\begin{split}\label{e:interior_term}
    &\phantom{{}={}}\left|\sum_{\dist(x,o)<\ell\atop y\in \bT}G_{c_\al x}^{(b_\al b_\beta)}(H-z_t-\msc(z_t)\mathbb{I}^\del)_{xy}G^{(b_\al b_\beta)}_{y c_\beta}\right|\\
    &=\left|\sum_{\dist(x,o)<\ell}G_{c_\al x}^{(b_\al b_\beta)}\left(\sqrt t(ZG^{(b_\al b_\beta)})_{x c_\beta}
+t\md(z,t) G^{(b_\al b_\beta)}_{x c_\beta}\right)\right|\\
&\lesssim \sum_{\dist(x,o)<\ell}|G_{c_\al x}^{(b_\al b_\beta)}(\sqrt{t}|(ZG^{(b_\al b_\beta)})_{x c_\beta}|+t|G^{(b_\al b_\beta)}_{x c_\beta}|)\\
&\lesssim \sqrt t \sum_{x\in \bT}(|G^{(b_\al b_\beta)}_{c_\al x}|^2+|G^{(b_\al b_\beta)}_{ c_\beta x}|^2+|(ZG^{(b_\al b_\beta)})_{x c_\beta}|^2 ).
\end{split}\end{align}

If $x\in \bT$ and $\dist(x, o)=\ell$, we denote the parent node of $x$ as $x'$. We then have, by the self-consistent equation of $\msc$,
\begin{align}\begin{split}\label{e:boundary_term}
&\phantom{{}={}}\sum_{y\in \bT}G_{c_\al x}^{(b_\al b_\beta)}(H-z_t-\msc(z_t)\mathbb{I}^\del)_{xy}G^{(b_\al b_\beta)}_{y c_\beta}\\
&=\frac{1}{\sqrt{d-1}}G_{c_\al x}^{(b_\al b_\beta)}G^{(b_\al b_\beta)}_{x' c_\beta}
-(z_t+\msc(z_t))G_{c_\al x}^{(b_\al b_\beta)}G^{(b_\al b_\beta)}_{x c_\beta}\\
&=\frac{1}{\sqrt{d-1}}G_{c_\al x}^{(b_\al b_\beta)}G^{(b_\al b_\beta)}_{x' c_\beta}
+\frac{1}{\msc(z_t)}G_{c_\al x}^{(b_\al b_\beta)}G^{(b_\al b_\beta)}_{x c_\beta}.
\end{split}\end{align} 
The claim \eqref{e:diffG3} follows from plugging \eqref{e:interior_term} and \eqref{e:boundary_term} into \eqref{e:diffG4}.

The claims \eqref{e:Greplace} follow from the same arguments as in \eqref{e:diffG2} and \eqref{e:diffG3}, so we omit the proof. 
\end{proof}
\subsection{Proof of \Cref{lem:deletedalmostrandom} and \Cref{lem:task2}}
\label{s:removeonevertex}
\begin{proof}[Proof of \Cref{lem:deletedalmostrandom}]
Take $\cF=\{i,o\}$, we can replace the indicator function $I_o$ by $I(\cF,\cG)$, and the error is negligible
\begin{align}\begin{split}\label{e:replace_indicator}
      \frac{1}{N}\sum_{o\in \qq{N}}\bE[|I_o-I(\cF,\cG)| |G^{(o)}_{ij}|^2 \Pi]&\lesssim \frac{1}{N}\sum_{o\in \qq{N}}\bE[|I_o-I(\cF,\cG)|N^{-2\fb}  \Pi]\\
      &\lesssim \bE[N^{-1+2\fc} N^{-2\fb}\Pi]\lesssim N^{-2\fb}\bE[\Phi\Pi],
\end{split}\end{align}
where in the first statement we used $|G^{(o)}_{ij}|\lesssim N^{-2\fb}$ from \eqref{eq:local_law}.
In the following we show 
\begin{align}\label{e:new_bound}
\frac{1}{Z_\cF}\sum_{\bfi}\bE\left[\bm1(\cG\in \Omega) I(\cF,\cG) |G^{(o)}_{ij}|^2 \Pi\right]\lesssim \bE[N^\fo\Phi\Pi],
\end{align}
and the claim \eqref{e:sameasdisconnect} follows from combining \eqref{e:replace_indicator} and \eqref{e:new_bound}.

We can perform a local resampling around the vertex $o$. We will first show that 
\begin{align}\begin{split}\label{e:switching}
&\phantom{{}={}}\frac{1}{Z_\cF}\sum_{\bfi}\bE\left[\bm1(\cG\in \Omega) I(\cF,\cG) |G^{(o)}_{ij}|^2 \Pi\right]\\
&=\frac{1}{Z_{\cF}}\sum_{\bfi}\bE\left[I(\cF,\cG)I(\cF,\wt \cG)\bm1(\cG,\wt \cG\in \Omega) | G^{(o)}_{ij}|^2  \Pi\right]+\OO(N^{-\fb}\bE[\Phi \Pi]).
\end{split}\end{align}

We split the left-hand side of \eqref{e:switching} into two terms
\begin{align}\begin{split}\label{e:switching2}
&\phantom{{}={}}\bE\left[I(\cF,\cG){\bm1(\cG\in \Omega)}|\wt G^{(o)}_{ij}|^2 \wt \Pi\right] \\ 
&=\bE\left[I(\cF,\cG){\bm1(\cG\in \Omega)}\left(I(\cF,\wt\cG)\bm1(\tcG\in \oOmega)+\bE_\bfS[1-I(\cF,\wt\cG)\bm1(\tcG\in \oOmega)]\right)| G^{(o)}_{ij}|^2  \Pi\right]\\
&=\bE\left[I(\cF,\cG)I(\cF,\wt\cG){\bm1(\cG\in \Omega)\bm1(\tcG\in \oOmega)}| G^{(o)}_{ij}|^2  \Pi\right]+\OO\left(N^{-1+2\fc}\bE\left[I(\cF,\cG){\bm1(\cG\in \Omega)}| G^{(o)}_{ij}|^2  \Pi\right] \right),
\end{split}\end{align}
where in the third line we used \Cref{lem:configuration}, that 
\begin{align}\label{e:changeF}
    \bE_\bfS[1-I(\cF,\tcG)\bm1(\tcG\in \oOmega) ]\leq  \bE_\bfS[1-\bm1(\tcG\in \oOmega) ]+\bE_\bfS[1-I(\cF,\tcG) ]\leq N^{-1+2\fc}.
\end{align} 
After averaging over $\bfi$, using that $|G_{ij}^{(o)}|\lesssim N^{-\fb}$ (from \eqref{eq:local_law}), the last term in \eqref{e:switching2} is sufficiently small. Namely,
\begin{align*}
    \frac{1}{Z_\cF}\sum_{\bfi}N^{-1+2\fc}\bE\left[I(\cF,\cG){\bm1(\cG\in \Omega)}| G^{(o)}_{ij}|^2 \Pi\right]=\OO(N^{-\fb}\bE[\Phi\Pi]).
\end{align*}

By \Cref{thm:prevthm0}, $\bP(\oOmega\setminus \Omega)\leq N^{-\fC'}$ for any $\fC'>0$, provided $N$ is large enough. Together with our assumption $\Pi=\bm1(\cG\in\Omega)\widehat \Pi$ and $N^{-\fC}\leq \widehat\Pi\leq N^{\fC}$, we get
\begin{align}\begin{split}\label{e:replace_oOmega}
&\phantom{{}={}}\frac{1}{Z_\cF}\sum_{\bfi}\bE\left[I(\cF,\cG)I(\cF,\wt\cG){\bm1(\cG\in\Omega)\bm1(\tcG\in \oOmega\setminus \Omega)} | G^{(o)}_{ij}|^2  \Pi\right]\lesssim N^{-\fb}\bE[\Phi\Pi].
\end{split}\end{align}
The estimates \eqref{e:switching2} and \eqref{e:replace_oOmega} give \eqref{e:switching}. 

Since $\tcG, \cG$ form an exchangeable pair, we can exchange $\cG$ and $\tcG$ in \eqref{e:switching}, and get 
\begin{align}\begin{split}\label{e:switching_Pi}
    &\phantom{{}={}}\frac{1}{Z_\cF}\sum_{\bfi}\bE\left[\bm1(\cG\in \Omega) I(\cF,\cG) |G^{(o)}_{ij}|^2 \Pi\right]\\
    &=\frac{1}{Z_\cF}\sum_{\bfi}\bE\left[I(\cF,\cG)I(\cF,\wt \cG)\bm1(\cG,\wt \cG\in \Omega) |\wt G^{(o)}_{ij}|^2  \wt \Pi\right]+\OO(N^{-\fb}\bE[\Phi\Pi] )\\
    &\lesssim\frac{1}{Z_{\cF^+}}\sum_{\bfi^+}\bE\left[\bm1(\cG, \wt \cG\in \Omega) I(\cF^+,\cG) |\wt G^{(o)}_{ij}|^2  \Pi\right]+\OO(N^{-\fb}\bE[\Phi \Pi]).
\end{split}\end{align}
where in the second statement, we used our assumption $ \bm1(\cG,\wt \cG\in \Omega) \wt \Pi
   \lesssim  \bm1(\cG,\wt \cG\in \Omega)\Pi$, replaced the indicator functions $I(\cF,\cG)I(\cF,\wt \cG)$ with $I(\cF^+,\cG)$ using \eqref{e:changeF}, and rewrote the expression as an average over all possible embeddings of $\cF^+$ in $\cG$.


We are now left to write the Green's function of the switched graph in terms of the original graph. Recall $\cT$ from \Cref{s:local_resampling}. $L:=P(\cT,z_t,\msc(z_t))$. Here $L$ is consistent with the local Green's function  (as defined in \eqref{e:local_Green}) on $\cT$, and we use the same symbols to represent them.  We notice that since $i,j$ are distinct neighbors of $o$, so $i,j$ are in different connected components of $\cT^{(o)}$. Thus $L^{(o)}_{ij}=0$, and $\wt G^{(o)}_{ij}=\wt G^{(o)}_{ij}-L^{(o)}_{ij}$. We will use the same argument as in the proof of \Cref{lem:diaglem}.   In the rest, we condition on that $\cG, \tcG\in \Omega$ and $I(\cF^+,\cG)=1$. Then by the same argument as for \eqref{eq:resolventexp}, we have
\begin{align}\label{e:GooY2}
    \wt G^{(o)}_{ij}&=\left(L^{(o)}\sum_{k=1}^\fp \left((\widetilde  B^\top(\tG^{(\bT)}-\msc(z_t))\widetilde  B+t(m_t-\md(z_t)\mathbb I)+\cD_2)  L^{(o)}\right)^k\right)_{ij}+\OO(N^{-2}),
\end{align}
where $\cD_2=-\sqrt tZ_\bT^{(o)}+ \sqrt t  Z_{\bT \bT^\complement}^{(o)}\widetilde G^{(\bT)}\widetilde  B+ \sqrt t \widetilde  B^\top\widetilde G^{(\bT)}Z^{(o)}_{\bT^\complement \bT}
+t (Z_{\bT \bT^\complement}^{(o)}\widetilde G^{(\bT)}Z_{\bT \bT^\complement}^{(o)}- m_t\mathbb I)$ is as defined in \eqref{e:defcE}. 

For any $1\leq k\leq \fp$, the $k$-th term in \eqref{e:GooY2} is an $\OO(1)$-weighted sum of terms of the following form
\begin{align}\label{e:PUP2}
   \sum_{x_1, x_2,\cdots, x_{2k}\in \bT}L^{(o)}_{i x_1} V_{x_1 x_2} L^{(o)}_{x_2 x_3} V_{x_3 x_4} 
   L^{(o)}_{x_4 x_5}\cdots V_{x_{2k-1} x_{2k}}L^{(o)}_{x_{2k} j}.
\end{align}
Here (with a slight abuse of notation) $V_{x_{2j-1}x_{2j}}$ can represent one of the following three terms  $t(m_t-\md(z_t)){\mathbb I})_{x_{2j-1}x_{2j}}$, $(\cD_2)_{x_{2j-1}x_{2j}}$, or $(\wt B^\top (\wt G^{(\bT)}-\msc(z_t))\wt B)_{x_{2j-1}x_{2j}}$.
Recall that for any $x,y\in \bT$,     
\Cref{p:WtGbound} implies that with overwhelmingly high probability over $Z$,
$|(\cD_2)_{xy}|\lesssim N^{\fo}\sqrt{t\Phi}$. Using this together with \eqref{eq:local_law} gives that  $|V_{xy}|\lesssim N^{-\fb} $ for any $x, y\in \bT$. 

We recall that $i,j$ are in different connected components of $\cT^{(o)}$. For the sequence of indices $x_0:=i, x_1, x_2,\cdots, x_{2k}, x_{2k+1}:=j$, there exists some pair of consecutively listed vertices that are in different connected components of $\cT^{(o)}$. If for some $0\leq m\leq k$, $x_{2m}, x_{2m+1}$ are in different connected components of $\cT^{(o)}$, then $L^{(o)}_{x_{2m}x_{2m+1}}=0$ and \eqref{e:PUP2} vanishes. Thus we only need to consider the case that for some $1\leq m\leq k$, $x_{2m-1}, x_{2m}$ are in different connected components of $\cT^{(o)}$. There are several cases:
\begin{enumerate}
   
    \item $V_{x_{2m-1}, x_{2m}}=(t(m_t-m_d(z_t))_{x_{2m-1}, x_{2m}}=0$, and \eqref{e:PUP2} vanishes.

    \item $V_{x_{2m-1}, x_{2m}}=(\cD_2)_{x_{2m-1}, x_{2m}}$, and we can bound \eqref{e:PUP2} as
\begin{align}\begin{split}\label{e:sumx123}
     &\phantom{{}={}}(N^\fo\sqrt{t \Phi})N^{-(k-1)\fb} \sum_{x_1, x_2,\cdots, x_{2k}\in \bT}|L^{(o)}_{i x_1}| |L^{(o)}_{x_2 x_3} |
   |L^{(o)}_{x_4 x_5}|\cdots |L^{(o)}_{x_{2k} j}|\\
   &\lesssim
   (N^\fo\sqrt{t \Phi})N^{-(k-1)\fb} 
   (d-1)^{\ell}(\ell(d-1)^\ell)^{k-1}\lesssim (d-1)^{\ell}N^\fo\sqrt{t \Phi},
\end{split}\end{align}
where we used \eqref{e:sum_Pbound}.
 \item $V_{x_{2m-1}, x_{2m}}=(\wt B^\top (\wt G^{(\bT)}-\msc(z_t))\wt B)_{x_{2m-1}, x_{2m}}
 =\sum_{\al, \beta: l_\al=x_{2m-1}, l_\beta=x_{2m}}\wt G^{(\bT)}_{c_\al c_\beta}
 $. If $k=1$, then $m=1$ and we can compute \eqref{e:PUP2} as
 \begin{align}\begin{split}\label{e:k=1Pxx}
     &\phantom{{}={}}\frac{1}{d-1}\sum_{\al\neq \beta\in \qq{\mu}}L^{(o)}_{i l_\al }\wt G^{(\bT)}_{c_\al c_\beta} L^{(o)}_{l_\beta j}=\frac{\msc^{2\ell}(z_t)}{(d-1)^{\ell}}\sum_{\dist_\cT(i, l_\al)=\ell-1\atop \dist_\cT(j, l_\beta)=\ell-1}\wt G^{(\bT)}_{c_\al c_\beta},
\end{split}\end{align}
where we used \eqref{e:sum_Pbound}, and $|\fc|=\OO(1)$.
For $k\geq 2$, similar to \eqref{e:sumx123}, we can bound \eqref{e:PUP2} as
\begin{align}\begin{split}\label{e:sumx1234}
     &\phantom{{}={}}N^{-(k-1)\fb} \sum_{\al\neq \beta}|\wt G^{(\bT)}_{c_\al c_\beta}|\sum_{x_1, x_2,\cdots, x_{2m-2}\in \bT\atop x_{2m+1}, x_{2m+2},\cdots, x_{2k}\in \bT}|L^{(o)}_{i x_1}| 
   \cdots |L^{(o)}_{x_{2m-2} l_\al}||L^{(o)}_{l_\beta x_{2m+1}}| \cdots |L^{(o)}_{x_{2k} j}|\\
   &\lesssim
   N^{-(k-1)\fb} \sum_{\al\neq \beta}|\wt G^{(\bT)}_{c_\al c_\beta}|
   (d-1)^{\ell}(\ell(d-1)^\ell)^{k-1}\lesssim 
    \frac{1}{(d-1)^{7(k-1)\fb/8}}\sum_{\al\neq \beta}|\wt G^{(\bT)}_{c_\al c_\beta}|,
\end{split}\end{align}
where we used $N^{\fb} \geq (d-1)^{20\ell}$.
\end{enumerate}
The estimates \eqref{e:sumx123}, \eqref{e:k=1Pxx} and \eqref{e:sumx1234} together lead to the following estimate for \eqref{e:PUP2} 
\begin{align}\label{e:firstbound}
   \eqref{e:PUP2}= \frac{\fc}{(d-1)^{\ell}}\sum_{\al\in \sfA_i\atop \beta\in \sfA_j}\wt G^{(\bT)}_{c_\al c_\beta}
    +\OO\left(\frac{1}{N^{7\fb/8}}\sum_{\al\neq \beta\in\qq{\mu}}|\wt G_{c_\al c_\beta}^{(\bT)}|+(d-1)^\ell N^\fo \sqrt{t\Phi}\right),
\end{align}
and by plugging \eqref{e:GooY2},\eqref{e:PUP2} and \eqref{e:firstbound} back into \eqref{e:switching_Pi} we conclude that
\begin{align}\begin{split}\label{e:IIGU2}
    &\frac{1}{Z_\cF}\sum_{\bfi}\bE\left[\bm1(\cG\in \Omega) I(\cF,\cG) |G^{(o)}_{ij}|^2 \Pi\right]\lesssim J_1+J_2+\OO(N^{-\fb}\bE[\Phi \Pi])\\
    &J_1:= \frac{1}{Z_{\cF^+}}\sum_{\bfi^+}\bE\left[\bm1(\cG, \wt \cG\in \Omega) I(\cF^+,\cG)  \left|\frac{1}{(d-1)^{\ell}}\sum_{\al\neq \beta\in\qq{\mu}}\wt G^{(\bT)}_{c_\al c_\beta} \right|^2 \Pi\right]\\
    &|J_2|\lesssim \frac{1}{Z_{\cF^+}}\sum_{\bfi^+}\bE\left[\bm1(\cG, \wt \cG\in \Omega) I(\cF^+,\cG) N^{-3\fb/4} \left(\sum_{\al\neq \beta\in\qq{\mu}}|\wt G^{(\bT)}_{c_\al c_\beta}|^2+ \Phi\right)\Pi\right].
\end{split}\end{align}

Next, we estimate $J_1$ and $J_2$ as in \eqref{e:IIGU2}. We need to express $\wt G_{c_\al c_\beta}^{(\bT)}$ in terms of the Green's function of the graph $\cG$. In this process, any term that can be bounded by $\OO(N^{-\oo(1)}\Phi)$ is considered negligible, since it contributes to an error $\bE\left[{\bm1(\cG\in \Omega)}N^{-\oo(1)}\Phi  \Pi\right]=\OO(N^{-\oo(1)}\bE[\Phi\Pi])$.

Thanks to \Cref{l:diffG1} (combining \eqref{e:total_diffG1}, \eqref{e:total_diffG2} and \eqref{e:total_diffG3}), we have
\begin{align}\label{e:replacealpha}
    \tG^{(\bT)}_{c_\al c_\beta}=G^{(b_\al b_\beta)}_{c_\al c_\beta}+\cE_{\al \beta},
\end{align}
where 
\begin{align}\begin{split}\label{e:cEbound}
    |\cE_{\al \beta}|&\lesssim  \sum_{\gamma\in\qq{\mu}, x\in \cN_\gamma}(|G_{c_\al x}^{(\bT\bW)}|^2+|G_{c_\beta x}^{(\bT\bW)}|^2)+\sum_{\gamma\in\qq{\mu}\setminus\{\al,\beta\}}(|G^{(\bT b_\al b_\beta )}_{c_\al b_\gamma}|^2+ |G^{(\bT b_\al b_\beta)}_{c_\beta b_{\gamma} }|^2)\\
    &+\sum_{x\in \bT}(|G_{c_\al x}^{(b_\al b_\beta)}|^2+|G^{(b_\al b_\beta)}_{c_\beta x}|^2)+N^{-\fb }\Phi,
\end{split}\end{align}
and $N_\gamma=\{x\neq c_\gamma: x\sim b_\gamma\}\cup\{a_\gamma\}$.

In the following we show that for $\cG\in \Omega$, and $\al\neq \beta, \al'\neq \beta'$
\begin{align}\begin{split}\label{e:smallterm}
   &\phantom{{}={}}\frac{1}{Z_{\cF^+}}\sum_{\bfi^+}I(\cF^+,\cG) (|\cE_{\al \beta}|| G^{(b_{\al'} b_{\beta'})}_{c_{\al'} c_{\beta'}}|+|G^{(b_\al b_\beta)}_{c_\al c_\beta}||\cE_{\al' \beta'}|
    +|\cE_{\al \beta}|^2)\\
    &\lesssim N^{-\fb} \sum_{\gamma\in\{\al, \beta, \al',\beta'\}\atop x\sim b_\gamma, 
x\neq c_\gamma}  \frac{1}{Z_{\cF^+}}\sum_{\bfi^+}I(\cF^+,\cG) |G^{(b_\gamma)}_{c_\gamma x}|^2+\frac{\Phi}{N^{3\fb/4 }},
\end{split}\end{align}
and thus
\begin{align}\begin{split}\label{e:GGGxy}
    &\phantom{{}={}}\frac{1}{Z_{\cF^+}}\sum_{\bfi^+}I(\cF^+,\cG)(G^{(b_\al b_\beta)}_{c_\al c_\beta}+\cE_{\al \beta})(\overline G^{(b_{\al'} b_{\beta'})}_{c_{\al'} c_{\beta'}}+\overline\cE_{\al' \beta'}) \\
    &=\frac{1}{Z_{\cF^+}}\sum_{\bfi^+}I(\cF^+,\cG)G^{(b_\al b_\beta)}_{c_\al c_\beta}\overline G^{(b_{\al'} b_{\beta'})}_{c_{\al'} c_{\beta'}}+\OO\left(\sum_{\gamma\in\{\al, \beta, \al',\beta'\}\atop x\sim b_\gamma, 
x\neq c_\gamma} \frac{N^{-\fb}}{ Z_{\cF^+}}\sum_{\bfi^+}I(\cF^+,\cG) |G^{(b_\gamma)}_{c_\gamma x}|^2+ \frac{\Phi}{N^{3\fb/4 }}\right).
\end{split}\end{align}

To prove \eqref{e:smallterm}, we start by plugging in the bound of $\cE_{\alpha \beta}$ from \eqref{e:cEbound} into the left-hand side of \eqref{e:smallterm}, after which each term contains three Green's function entries as factors. We can bound one of them by $N^{-\fb} $ using \eqref{eq:local_law}, and the remaining two can be bounded by terms in the form $\{|G_{c_\gamma x}^{(b_\gamma)}|^2\}$ with ${\gamma\in \{\al, \beta, \al', \beta'\}}, x\sim b_\gamma, x\neq c_\gamma$ or can be bounded by $N^\fo\Phi$ using \eqref{e:Gest} and \eqref{e:use_Ward}.  In the following we estimate the following term from \eqref{e:smallterm}, and the other terms can be bounded in the same way, so we omit arguments about them.
\begin{align}\label{e:yizhong}
\frac{1}{Z_{\cF^+}}\sum_{\bfi^+}I(\cF^+,\cG)\sum_{\gamma\in\qq{\mu}, x\in \cN_\gamma}|G_{c_\al x}^{(\bT\bW)}|^2|G^{(b_{\al'}b_{\beta'})}_{c_{\al'}c_{\beta'}}|\lesssim 
\frac{1}{Z_{\cF^+}}\sum_{\bfi^+}I(\cF^+,\cG)\sum_{\gamma\in\qq{\mu}, x\in \cN_\gamma}N^{-\fb}|G_{c_\al x}^{(\bT\bW)}|^2,
\end{align}
where we bound $|G^{(b_{\al'}b_{\beta'})}_{c_{\al'}c_{\beta'}}|\lesssim N^{-\fb} $ by \eqref{eq:local_law}. 
  We recall the definition of the indicator function $I(\cF^+,\cG)$ from \eqref{def:indicator}, and define the forest $\widehat \cF$ from $\cF^+$ by removing $\{(b_\al,c_\al)\}$: $
    \widehat \cF=(\widehat \bfi, \widehat E)=\cF^+\setminus  \{(b_\al,c_\al)\},\quad \widehat \bfi=\bfi^+\setminus\{b_\al, c_\al\}$.
Then 
\begin{align*}
 I(\cF^+,\cG)=I(\widehat \cF, \cG)A_{c_\al b_\al}I_{c_\al}, 
\end{align*}
where $I_{c_\al}=\bm1(c_\al\not\in \bX)$ and 
$\bX$ is the collection of vertices $v$ such that either there exists some $u\in \cB_\ell(v,\cG)$ such that $\cB_\fR(u,\cG)$ is not a tree; or $\dist(v,c)<3\fR$ for some $(b_\al, c_\al)\neq (b,c)\in \cC^+$. Thus $I_{c_\al}$ satisfies the requirements in \Cref{def:Ic}. Then we can rewrite the right-hand side of \eqref{e:yizhong} as
\begin{align}\begin{split}\label{e:yizhong1}
     \frac{1}{Z_{\cF^+}}\sum_{\widehat\bfi}I(\widehat\cF,\cG)\sum_{b_\al\sim c_\al } \sum_{\gamma\in\qq{\mu}, x\in \cN_\gamma}N^{-\fb}I_{c_\al} |G_{c_\al x}^{(\bT\bW)}|^2.
\end{split}\end{align}
If $\gamma\neq \al$, thanks to \eqref{e:use_Ward} and $\sum_{c_\al\in \qq{N}}|1-I_{c_\al}|\lesssim N^{3\fc/2}$, we have 
\begin{align}\label{e:yizhong2}
\eqref{e:yizhong1}
\lesssim \frac{1}{Z_{\cF^+}}\sum_{\widehat\bfi}I(\widehat\cF,\cG)\sum_{\gamma\in\qq{\mu}\setminus\{\al\}\atop x\in \cN_\gamma}N^{-\fb}\sum_{b_\al\sim c_\al}(|G_{c_\al x}^{(\bT\bW)}|^2+|1-I_{c_\al}|)
\lesssim N^{-\fb}N^\fo \Phi\lesssim N^{-3\fb/4}\Phi.
\end{align}

Thus we can reduce \eqref{e:yizhong} to the case $\gamma=\al$ 
\begin{align}
    \label{e:GtGG}
\frac{1}{Z_{\cF^+}}\sum_{\bfi^+}I(\cF^+,\cG)N^{-\fb}\left(|G_{c_\al a_\al}^{(\bT \bW)}|^2+\sum_{
x\sim b_\alpha\atop
x\neq c_\alpha} |G^{(\bT \bW)}_{c_\alpha x}|^2\right).
\end{align} 
The terms involving $|G_{c_\al a_\al}^{(\bT \bW)}|^2$ can be bounded by the same way as in \eqref{e:yizhong2}. 
Next, we show that we can replace $G_{c_\al x}^{(\bT \bW)}$ in \eqref{e:GtGG} by $G_{c_\al x}^{(b_\al)}$.
\begin{align}\begin{split}\label{e:yizhong3}
    \eqref{e:GtGG}&= \frac{1}{Z_{\cF^+}}\sum_{\bfi^+}I(\cF^+,\cG)N^{-\fb}\sum_{
x\sim b_\alpha\atop
x\neq c_\alpha} |G^{(b_\al)}_{c_\alpha x}|^2+\cE,\\
|\cE|&\lesssim \frac{1}{Z_{\cF^+}}\sum_{\bfi^+}I(\cF^+,\cG)N^{-\fb}\sum_{
x\sim b_\alpha\atop
x\neq c_\alpha} |G^{(\bT\bW)}_{c_\alpha x}-G^{(b_\al)}_{c_\alpha x}| +\OO(N^{-3\fb/4}\Phi),
\end{split}\end{align}
where for the bound of $\cE$, we used that $|G^{(\bT\bW)}_{c_\alpha x}|,|G^{(b_\al)}_{c_\alpha x}|\lesssim 1$ from \eqref{eq:local_law}.

Thanks to \eqref{e:Greplace}, we can bound the difference $|G_{c_\al x}^{(\bT \bW)}-G_{c_\al x}^{(b_\al)}|$ by
\begin{align}\label{e:yizhong4}
\left|G^{(\bT \bW)}_{c_\alpha x}-G^{(b_\alpha)}_{c_\alpha x}\right|
\leq \sum_{\gamma\in\qq{\mu}\setminus\{\al\}}(|G^{(\bT b_\al )}_{c_\al b_\gamma}|^2+|G^{(\bT b_\al)}_{x b_{\gamma}}|^2)+\sum_{y\in \bT}(|G^{(b_\al)}_{c_\al y}|^2+|G^{(b_\al)}_{xy}|^2).
\end{align}
By plugging \eqref{e:yizhong4} into \eqref{e:yizhong3}, by the same argument as in \eqref{e:yizhong2}, we can bound $\cE$ in \eqref{e:yizhong3} as
\begin{align}\label{e:yizhong5} \frac{N^{-\fb}}{Z_{\cF^+}}\sum_{\bfi^+}I(\cF^+,\cG)\left(\sum_{\gamma\in\qq{\mu}\setminus\{\al\}}(|G^{(\bT b_\al )}_{c_\al b_\gamma}|^2+|G^{(\bT b_\al)}_{x b_{\gamma}}|^2)+\sum_{y\in \bT}(|G^{(b_\al)}_{c_\al y}|^2+|G^{(b_\al)}_{xy}|^2)\right)
    \lesssim N^{-3\fb/4}\Phi.
\end{align}
The claim \eqref{e:smallterm} follows from plugging \eqref{e:yizhong2}, \eqref{e:yizhong3} and \eqref{e:yizhong5} into \eqref{e:yizhong}.

For the first term on the right-hand side of \eqref{e:GGGxy}, we recall that $\al \neq \beta$ and $\al'\neq \beta'$. Then either $\{\al, \beta\}=\{\al', \beta'\}$, or some indices, say $\al,\al'$, only appears once (namely, $\al\neq \al', \beta'$ and $\al'\neq \al, \beta$). Then we can sum over $(b_\al, c_\al)$ and $(b_{\al'}, c_{\al'})$ separately (as in \eqref{e:yizhong1}), using \eqref{e:Gccbb_youyige} and \eqref{e:Gest}
\begin{align}\begin{split}\label{e:fbound1}
    &\phantom{{}={}}\frac{1}{Z_{\cF^+}}\sum_{\bfi^+} I(\cF^+,\cG)G^{(b_\al b_\beta)}_{c_\al c_\beta} \overline G^{(b_{\al'} b_{\beta'})}_{c_{\al'} c_{\beta'}} \\
&=\frac{1}{Z_{\cF^+}}\sum_{\bfi^+\setminus \{b_\al, c_\al, b_{\al'}, c_{\al'}\}}\sum_{c_\al\sim b_\al}\sum_{c_{\al'}\sim b_{\al'}} I(\cF^+,\cG) G^{(b_\al b_\beta)}_{c_\al c_\beta} \overline G^{(b_{\al'} b_{\beta'})}_{c_{\al'} c_{\beta'}} 
\\
    &\lesssim  \frac{1}{Z_{\cF^+}}\sum_{\bfi^+} I(\cF^+,\cG)N^{-\fb/2}(|G^{( b_{\beta})}_{b_{\al} c_{\beta}}|+\Phi)|N^{-\fb/2}(|G^{( b_{\beta'})}_{b_{\al'} c_{\beta'}}|+\Phi)|\lesssim N^{-\fb} N^\fo\Phi.
\end{split}\end{align}
Otherwise if $\{\al, \beta\}=\{\al', \beta'\}$, by the same argument as in \eqref{e:yizhong1}, \eqref{e:use_Ward} gives
\begin{align}\label{e:fbound2}
    \frac{1}{Z_{\cF^+}}\sum_{\bfi^+}I(\cF^+,\cG) |G^{(b_\al b_\beta)}_{c_\al c_\beta}|^2\lesssim N^\fo\Phi.
\end{align}

We recall that  $J_1$ in \eqref{e:IIGU2} is obtained by averaging \eqref{e:GGGxy} over  $\alpha \neq \beta \in \qq{\mu}$and  $\alpha' \neq \beta' \in \qq{\mu}$. By substituting \eqref{e:GGGxy}, \eqref{e:fbound1}, and \eqref{e:fbound2}, we conclude that:
\begin{align}\begin{split}\label{e:J1bound}
    J_1&=\sum_{\al\neq\beta\in\qq{\mu}}\frac{1}{(d-1)^{2\ell}Z_{\cF^+}}\sum_{\bfi^+}\bE\left[\bm1(\cG, \wt \cG\in \Omega) I(\cF^+,\cG)  |G^{(b_\al b_\beta)}_{c_\al c_\beta}|^2 \Pi\right]\\
    &+\frac{1}{Z_{\cF^+}}\sum_{\bfi^+}\bE\left[\bm1(\cG,\wt\cG\in \Omega) I(\cF^+,\cG) \left(\frac{(d-1)^{2\ell}}{N^{\fb}}\sum_{\al\in\qq{\mu}}\sum_{
x\sim b_\alpha,
x\neq a_\alpha} |G^{(b_\al)}_{c_\al x}|^2 +\frac{\Phi}{N^{\fb/2}}\right)\Pi\right],\\
&\lesssim \frac{1}{Z_{\cF^+}}\sum_{\bfi^+}\bE\left[\bm1(\cG,\wt\cG\in \Omega) I(\cF^+,\cG) \left(\frac{1}{N^{3\fb/4}}\sum_{
x\sim b_\alpha,
x\neq a_\alpha} |G^{(b_\al)}_{c_\al x}|^2 +N^\fo \Phi\right)\Pi\right].
\end{split}\end{align}
where in the last statement we used \eqref{e:fbound2} and the permutation invariance of the vertices, so that the expectation does not depend on $\al$. By the same argument we can also bound $J_2$ in \eqref{e:IIGU2} as, 
\begin{align}\label{e:J2bound}
    J_2\lesssim \frac{1}{N^{3\fb/4}Z_{\cF^+}}\sum_{\bfi^+}\bE\left[\bm1(\cG,\wt\cG\in \Omega) I(\cF^+,\cG) \left(\frac{1}{N^{3\fb/4}}\sum_{
x\sim b_\alpha,
x\neq a_\alpha} |G^{(b_\al)}_{c_\al x}|^2 +(d-1)^{2\ell}N^\fo \Phi\right)\Pi\right].
\end{align}

By plugging \eqref{e:J1bound} and \eqref{e:J2bound} into \eqref{e:IIGU2}, we conclude
\begin{align}\begin{split}\label{e:IIGU3}
    \eqref{e:IIGU2}
    &\lesssim \frac{1}{Z_{\cF^+}}\sum_{\bfi^+}\bE\left[\bm1(\cG,\wt\cG\in \Omega) I(\cF^+,\cG) \left(\frac{1}{N^{3\fb/4}}\sum_{
x\sim b_\alpha,
x\neq a_\alpha} |G^{(b_\al)}_{c_\al x}|^2 +N^\fo \Phi\right)\Pi\right] \\
  &\lesssim 
   \frac{1}{N^{3\fb/4} Z_{\cF^+}}\sum_{\bfi^+}\bE\left[\bm1(\cG\in \Omega) I(\cF^+,\cG) \sum_{
x\sim b_\alpha,
x\neq a_\alpha} |G^{(b_\al)}_{c_\al x}|^2\Pi\right]+\bE[N^\fo\Phi\Pi]\\
&\lesssim \frac{1}{N^{3\fb/4}(Nd)} \sum_{b_\al\sim c_\al}  \bE\left[ I(\{c_\al,b_\al\},\cG)\sum_{
x\sim b_\alpha,
x\neq a_\alpha}  |G^{(b_\al)}_{c_\al x}|^2\Pi\right] +\bE[N^\fo\Phi\Pi]\\
&\lesssim \frac{1}{N^{1+3\fb/4}} \sum_{o\in\qq{N}}   \bE\left[ I(\{i,o\},\cG) |G^{(o)}_{ij}|^2\Pi\right] +\bE[N^\fo\Phi\Pi],
\end{split}\end{align}
where in the second statement, we dropped the indicator function $\bm1(\wt \cG\in \Omega)$; in the third statement sum over $\bfi^+\setminus\{b_\al,c_\al\}$; for the last statement, we used the permutation invariance of the vertices, so that $G_{ij}^{(o)}$ and $G_{c_\al x}^{(b_\al)}$ have the same distribution. 

Thus \eqref{e:IIGU2} and \eqref{e:IIGU3} together leads to the following bound 
    \begin{align*}
\frac{1}{N} \sum_{o\in\qq{N}} \bE[I(\{i,o\},\cG) |G_{ij}^{(o)}|^2\Pi]\lesssim  \frac{1}{N^{1+3\fb/4}} \sum_{o\in\qq{N}} \bE[I(\{i,o\},\cG) |G_{ij}^{(o)}|^2\Pi]+\bE[N^\fo\Phi\Pi],
\end{align*}
and the claim \eqref{e:sameasdisconnect} follows from rearranging.

To prove \eqref{e:GiGi}, we can proceed in the same way as \eqref{e:sameasdisconnect}. The same as in \eqref{e:switching}, we have
\begin{align}\begin{split}\label{e:step1}
 &\phantom{{}={}}\frac{1}{N}\sum_{o\in \qq{N}}\bE[I_o\bm1(\cG\in \Omega)(G_{i j}^{(o)})^2(Q_t-Y_t)^{p-1}]\\
    &=\frac{1}{Z_{\cF}}\sum_{\bfi}\bE\left[I(\cF,\cG)I(\cF,\wt \cG)\bm1(\cG,\wt \cG\in \Omega) | G^{(o)}_{ij}|^2 (Q_t-Y_t)^{p-1}\right]+\OO(N^{-\fb}\bE[\Psi_p])\\
     &=\frac{1}{Z_{\cF}}\sum_{\bfi}\bE\left[I(\cF,\cG)I(\cF,\wt \cG)\bm1(\cG,\wt \cG\in \Omega) | G^{(o)}_{ij}|^2 (\wt Q_t-\wt Y_t)^{p-1}\right]+\OO(N^{-\fb}\bE[\Psi_p])\\
     &+\OO\left(\frac{1}{Z_{\cF}}\sum_{\bfi}\bE\left[I(\cF,\cG)\bm1(\cG,\in \Omega) | G^{(o)}_{ij}|^2 (d-1)^{8\ell}\Upsilon \Phi
    (|Q_t-Y_t|+(d-1)^{8\ell}\Upsilon \Phi)^{p-2}\right]\right),
\end{split}\end{align}
where the last statement follows from replacing $(Q_t-Y_t)^{p-1}$ by $(\wt Q_t-\wt Y_t)^{p-1}$ and using \eqref{e:tilde_bound} that 
\begin{align*}
|(\wt Q_t-\wt Y_t)-(Q_t-Y_t)|(|\wt Q_t-\wt Y_t|+|Q_t-Y_t|)^{p-2}\lesssim (d-1)^{8\ell}\Upsilon \Phi
    (|Q_t-Y_t|+(d-1)^{8\ell}\Upsilon \Phi)^{p-2}.
\end{align*}
Using \eqref{e:sameasdisconnect} with $\Pi= \bm1(\cG\in \Omega)(d-1)^{8\ell}\Upsilon \Phi(|Q_t-Y_t|+(d-1)^{8\ell}\Upsilon \Phi)^{p-2}$, we can then bound the last term in \eqref{e:step1} by $\OO(N^{-\fb/2}\bE[\Psi_p])$. Then using the exchangeability of $\cG, \tcG$, by the same argument as in \eqref{e:switching_Pi}, \eqref{e:step1} leads to 
\begin{align}\begin{split}\label{e:changetilde}
 &\phantom{{}={}}\frac{1}{N}\sum_{o\in \qq{N}}\bE[I_o\bm1(\cG\in \Omega)(G_{i j}^{(o)})^2(Q_t-Y_t)^{p-1}]\\
        &=\frac{1}{Z_{\cF^+}}\sum_{\bfi^+}\bE[\bm1(\tcG,\cG\in \Omega)I(\cF^+,\cG) (\wt G_{i j}^{(o)})^2( Q_t- Y_t)^{p-1}]+\OO(N^{-\fb/2}\bE[\Psi_p]).
\end{split}\end{align}

Starting from \eqref{e:changetilde}, \eqref{e:k=1Pxx}, \eqref{e:IIGU2}, the first statement in \eqref{e:J1bound} , \eqref{e:J2bound} and \eqref{e:IIGU3} (recall that error $\bE[N^\fo \Phi \Pi]$ in \eqref{e:J1bound} is from \eqref{e:fbound2}) together give
\begin{align*}
    &\phantom{{}={}}\frac{1}{N}\sum_{o\in \qq{N}}\bE[I_o\bm1(\cG\in \Omega)(G_{i j}^{(o)})^2(Q_t-Y_t)^{p-1}]\\
    &=\frac{\msc^{4\ell}(z_t)}{(d-1)^{2\ell}}\sum_{\dist_\cT(i, l_\al)=\ell-1\atop \dist_\cT(j, l_\beta)=\ell-1}\frac{1}{Z_{\cF^+}}\sum_{\bfi^+}\bE[\bm1(\tcG,\cG\in \Omega)I(\cF^+,\cG) (G^{(b_\al b_\beta)}_{c_\al c_\beta})^2(Q_t-Y_t)^{p-1}]\\
    &+\OO\left(\frac{1}{N^{1+3\fb/4}}\sum_{o}\bE\left[\bm1(\cG\in\Omega) I(\{i,o\},\cG) |G^{(o)}_{ij}|^2 |Q_t-Y_t|^{p-1} \right] +N^{-\fb/2}\bE[\Phi |Q_t-Y_t|^{p-1}]\right)\\
    &=\frac{\msc^{4\ell}(z_t)}{Z_{\cF^+}}\sum_{\bfi^+}\bE[\bm1(\cG\in \Omega)I(\cF^+,\cG) (G^{(b_\al b_\beta)}_{c_\al c_\beta})^2(Q_t-Y_t)^{p-1}]+\OO\left(N^{-\fb/2}\bE[\Psi_p]\right).
\end{align*}
In the last statement, we used the permutation invariance of the vertices, ensuring that the expectation does not depend on $\al, \beta$ (as long as $\al\neq \beta$). Additionally, there are $(d-1)^\ell$ indices such that $\dist_\cT(i, l_\al)=\ell-1$ or $\dist_\cT(j, l_\beta)=\ell-1$.
We also bound the error term using
\eqref{e:sameasdisconnect} with $\Pi= \bm1(\cG\in \Omega)(|Q_t-Y_t|+(d-1)^{8\ell}\Upsilon \Phi)^{p-1}$.

The proofs of \eqref{e:GiGjGk} and \eqref{e:GiGj} follow from similar arguments as that of \eqref{e:sameasdisconnect}. We will only sketch the proof of the first statement in \eqref{e:GiGj} here. 

Take $\cF=\{\{i,o\}, \{b,c\}\}$. Similar to \eqref{e:replace_indicator}, we can replace the indicator function $I_{oc}$ by $I(\cF,\cG)$, and the error is negligible.
We perform a local resampling around $o$. The same argument as in \eqref{e:changetilde} using the $(\cG, \tcG)$ form an exchangeable pair, gives
\begin{align}\begin{split}\label{e:GGexp}
&\phantom{{}={}}\frac{1}{Z_{\cF}}\sum_{\bfi}\bE\left[\bm1(\cG\in \Omega) I(\cF,\cG)G_{i c}^{( b o)}G_{j c}^{( b o)}(Q_t-Y_t)^{p-1}\right]\\
&=\frac{1}{Z_{\cF^+}}\sum_{\bfi^+}\bE\left[I(\cF^+,\cG)\bm1(\cG,\wt \cG\in \Omega)\wt G_{i c}^{( b o)}\wt G_{j c}^{( b o)}( Q_t- Y_t)^{p-1}\right] +\OO( N^{-\fb/2}\bE[ \Psi_p]),
\end{split}\end{align}

Conditioned on that $\cG, \tcG\in \Omega$ and $I(\cF^+,\cG)=1$, by the Schur complement formula \eqref{e:Schur1}, 
\begin{align}\begin{split}\label{e:GGexp2}
   &\phantom{{}={}}\wt G_{i c}^{( b o)}\wt G_{j c}^{( b o)}
   =\sum_{x\in \bT\setminus\{o\}}\wt G_{ix}^{(ob)}((H+\sqrt{t}Z) \wt G^{(\bT b)})_{xc}\sum_{y\in \bT\setminus\{o\}}\wt G_{jy}^{(ob)}((H+\sqrt{t}Z) \wt G^{(\bT b)})_{yc}\\
   &=\frac{1}{(d-1)}\left(\sum_{\al\in \qq{\mu}}\wt G_{il_\al  }^{( b o)}\wt G_{c_\al c  }^{( b \bT)}+\OO ((d-1)^\ell N^\fo\sqrt{t\Phi})\right)\left(\sum_{\beta\in \qq{\mu}}\wt G_{j l_\beta }^{( b o)}\wt G_{c_\beta c}^{( b \bT)}+\OO ((d-1)^\ell N^\fo\sqrt{t\Phi})\right)\\
   &=
   \frac{1}{(d-1)}\sum_{\al,\beta\in \qq{\mu}}\wt G_{il_\al  }^{( b o)}\wt G_{c_\al c  }^{( b \bT)}\wt G_{j l_\beta }^{( b o)}\wt G_{c_\beta c}^{( b \bT)}
   +\OO \left(N^{-2\fb} \sum_{\al\in \qq{\mu}}(|\wt G^{(b\bT)}_{c_\al c}|^2+\Phi)\right),
\end{split}\end{align}
where in the second statement, we use for $x,y\in \bT$, $|(Z \wt G^{(\bT b)})_{xc}|, |(Z \wt G^{(\bT b)})_{yc}|\leq N^\fo\sqrt{\Phi}$ from \eqref{e:Werror2}; in the third statement, we use for $\tcG\in \Omega$,
$|\wt G_{c_\al c  }^{( b \bT)}|, |\wt G_{c_\beta c}^{( b \bT)}|\lesssim N^{-\fb}$ from \eqref{eq:local_law}. Similar to \eqref{e:yizhong2}, by \eqref{e:use_Ward2}, we have
\begin{align}\label{e:erzhong1}
    \frac{1}{Z_{\cF^+}}\sum_{\bfi^+}I(\cF^+,\cG)N^{-2\fb}\left(\Phi+\sum_{\al\in \qq{\mu}}|\wt G^{(b\bT)}_{c_\al c}|^2\right)\lesssim N^{-2\fb}(\Phi+N^\fo \Phi +N^{2\fc-1})\lesssim N^{-\fb/2}\Phi.
\end{align}
Thus by plugging \eqref{e:GGexp2} and \eqref{e:erzhong1} into \eqref{e:GGexp} and conditioning on $\cG, \tcG\in \Omega$, we get
\begin{align*}
\frac{1}{Z_{\cF^+}}\sum_{\bfi^+}I(\cF^+,\cG)\wt G_{i c}^{( b o)}\wt G_{j c}^{( b o)}=\sum_{\bfi^+}\sum_{\al, \beta\in \qq{\mu}}\frac{I(\cF^+,\cG)}{(d-1)Z_{\cF^+}}\wt G_{il_\al  }^{( b o)}\wt G_{c_\al c  }^{( b \bT)}\wt G_{j l_\beta }^{( b o)}\wt G_{c_\beta c}^{( b \bT)}+\OO\left( \frac{\Phi}{N^{\fb/2}}\right).
\end{align*}
We recall from \eqref{eq:local_law}, that $\wt G_{il_\al  }^{( b o)}=\wt L^{(o)}_{i l_\al}+\OO(N^{-\fb})= L^{(o)}_{i l_\al}+\OO(N^{-\fb})$ and $,\wt G_{jl_\beta  }^{( b o)}=\wt L^{(o)}_{j l_\beta}+\OO(N^{-\fb})=L^{(o)}_{j l_\beta}+\OO(N^{-\fb})$. Similar to \eqref{e:replacealpha}, we can replace $\wt G_{c_\al c}^{( b \bT)}, \wt G_{c_\beta c}^{( b \bT)}$ by $G_{c_\al c}^{(b_\al b )}, G_{c_\beta c}^{( b_\beta b)}$, as
\begin{align}\label{e:erzhong2}
\frac{1}{Z_{\cF^+}}\sum_{\bfi^+}I(\cF^+,\cG)\wt G_{i c}^{( b o)}\wt G_{j c}^{( b o)}
=\frac{1}{Z_{\cF^+}}\sum_{\bfi^+}\sum_{\al, \beta\in \qq{\mu}}\frac{ L^{(o)}_{i l_\al} L^{(o)}_{j l_\beta} }{d-1} I(\cF^+,\cG) G_{c_\al c}^{(b_\al b )} G_{c_\beta c}^{( b_\beta b)}+\OO\left( \frac{\Phi}{N^{\fb/2}}\right).
\end{align}
We recall that $i\neq j$, so $L^{(o)}_{i l_\al}L^{(o)}_{j l_\beta}$ is nonzero only if $\dist_\cG(i, l_\al)=\ell-1, \dist_\cG(j, l_\beta)=\ell-1$. In particular $\al\neq \beta$, and we can sum over $(b_\al, c_\al)$ and $(b_\beta, c_\beta)$ separately. Thanks to \eqref{e:Gccbb_youyige} and \eqref{e:use_Ward}
\begin{align}\begin{split}\label{e:erzhong3}
  &\phantom{{}={}}\sum_{\al\neq \beta\in \qq{\mu}}\frac{ L^{(o)}_{i l_\al} L^{(o)}_{j l_\beta} }{d-1}\frac{1}{Z_{\cF^+}}\sum_{\bfi^+} I(\cF^+,\cG) G_{c_\al c}^{(b_\al b )} G_{c_\beta c}^{( b_\beta b)}\\
  &\lesssim \sum_{\al\neq \beta\in \qq{\mu}}\frac{1 }{(d-1)^{\ell}}\frac{1}{Z_{\cF^+}}\sum_{\bfi^+} I(\cF^+,\cG) N^{-\fb/2}(|G_{b_\al c}^{( b )}|+\Phi) N^{-\fb/2}(|G_{b_\beta c}^{( b )}|+\Phi)
  =\OO(N^{-\fb/2} \Phi).
\end{split}\end{align}
The claim \eqref{e:GiGj} follows from plugging \eqref{e:erzhong2} and \eqref{e:erzhong3} into \eqref{e:GGexp}.
\end{proof}

\begin{proof}[Proof of Proposition \ref{lem:task2}]
We will only prove the first statement in \eqref{eq:task2} with $\mathbb X=\emptyset$ as the other statements are similar. Thanks to \Cref{l:diffG1}, we have
\begin{align*}
    \tG^{(\bT)}_{c_\al c_\beta}=G^{(b_\al b_\beta)}_{c_\al c_\beta}+\cE_{\al \beta},
\end{align*}
where $\cE_{\al \beta}$ as in \eqref{e:cEbound}. 
The statement follows from showing 
\begin{align}\begin{split}\label{e:G2bound}
     &\sum_{\gamma\in \qq{\mu}}\sum_{x\in \cN_\gamma}\frac{1}{Z_{\cF^+}}\sum_{\bfi^+}\bE[I(\cF^+,\cG)\bm1(\cG\in \Omega)|G_{c_\al x}^{(\bT\bW)}|^2\Pi]\lesssim (d-1)^\ell \bE[N^\fo\Pi].\\
     &\sum_{\gamma\in\qq{\mu}\setminus\{\al,\beta\}}\frac{1}{Z_{\cF^+}}\sum_{\bfi^+}\bE[I(\cF^+,\cG)\bm1(\cG\in \Omega)|G_{c_\al b_\gamma}^{(\bT b_\al b_\beta)}|^2\Pi]\lesssim (d-1)^\ell \bE[N^\fo\Pi],\\
     &\sum_{x\in\bT}\frac{1}{Z_{\cF^+}}\sum_{\bfi^+}\bE[I(\cF^+,\cG)\bm1(\cG\in \Omega)|G_{c_\al x}^{( b_\al b_\beta)}|^2\Pi]\lesssim (d-1)^\ell \bE[N^\fo\Pi],
\end{split}\end{align}
where  $N_\gamma=\{x\neq c_\gamma: x\sim b_\gamma\}\cup\{a_\gamma\}$.

In the following we prove the first statement in \eqref{e:G2bound}, the others are similar, so we omit their proofs.
If $x\in \cN_\gamma$ with $\gamma\neq \al$, or $x=a_\al$, we can first sum over $b_\al\sim c_\al$, and  \eqref{e:use_Ward} gives
\begin{align}\label{e:qinkuang1}
   \frac{1}{Z_{\cF^+}}\sum_{\bfi^+}I(\cF^+,\cG)\bm1(\cG\in \Omega)|G_{c_\al x}^{(\bT\bW)}|^2\lesssim N^\fo\Phi.
\end{align}
Otherwise $x\sim b_\al, x\neq c_\al$. We recall the upper bound on $|G^{(\bT \bW)}_{c_\alpha x}-G^{(b_\alpha)}_{c_\alpha x}|$ from \eqref{e:yizhong4} and \eqref{e:yizhong5}
\begin{align}\label{e:qinkuang2}
 \frac{1}{Z_{\cF^+}}\sum_{\bfi^+}I(\cF^+,\cG)\bm1(\cG\in \Omega)G^{(\bT \bW)}_{c_\alpha x}
 =\frac{1}{Z_{\cF^+}}\sum_{\bfi^+}I(\cF^+,\cG)\bm1(\cG\in \Omega)G^{(b_\alpha)}_{c_\alpha x}
 +\OO(N^{-\fb/2}\Phi).
\end{align}
By combining \eqref{e:qinkuang1} and \eqref{e:qinkuang2}, we conclude
    \begin{align}\begin{split}\label{e:decomp}
    &\phantom{{}={}}\sum_{\gamma\in \qq{\mu}}\sum_{x\in \cN_\gamma}\frac{1}{Z_{\cF^+}}\sum_{\bfi^+}\bE[I(\cF^+,\cG)\bm1(\cG\in \Omega)|G_{c_\al x}^{(\bT\bW)}|^2\Pi]\\
    &\lesssim 
    \sum_{x\sim b_\al, x\neq c_\al}\frac{1}{Z_{\cF^+}}\sum_{\bfi^+}\bE[I(\cF^+,\cG)\bm1(\cG\in \Omega)|G_{c_\al x}^{(b_\al)}|^2\Pi]
    +(d-1)^{\ell}\bE[N^\fo\Phi\Pi]
    \lesssim (d-1)^{\ell}\bE[N^\fo\Phi\Pi].
    \end{split}\end{align}
where in the last line we used \eqref{e:sameasdisconnect} to bound the first term; The first claim in \eqref{e:G2bound} follows from combining  \eqref{e:qinkuang1} and \eqref{e:decomp}.
\end{proof}

\subsection{Proof of \Cref{l:first_term}}
\label{s:first_term}

The following lemma collects some explicit computations related to the Green's functions.

\begin{lemma}\label{l:second_term}
Adopt the notation of \Cref{s:setting7} with $\cF=\{i,o\}$ and $z\in \bf D$ such that $|z-E_t|\leq N^{-\fg}$. 
We assume that $I(\cF^+,\cG)=1$, denote the Green's functions $L=P(\cT, z_t,\msc(z_t))$ and $L^{(i)}=P^{(i)}(\cT, z_t,\msc(z_t))$, and the set $\sfA_i := \{ \alpha \in \qq{\mu} : \dist_{\cT}(i, l_\al) = \ell+1 \}$.
 The following holds 

\begin{align}\begin{split}\label{e:schur_two1}
&\sum_{\al \in \sfA_i, \beta\in\qq{\mu}\atop\al\neq \beta}\frac{\msc^{2\ell}(z_t)L_{l_\beta l_\beta}^{(i)} }{(d-1)^{\ell+2}}
=\left(\frac{d+1}{d-2}-\frac{d(d-1)^{\ell}}{d-2}  \right)+\OO\left((d-1)^{-\ell}\right),
\\
&\sum_{\al \in \sfA_i, \beta\in\qq{\mu}\atop\al\neq \beta}\frac{\msc^{2\ell}(z_t)L_{l_\al l_\beta}^{(i)} }{(d-1)^{\ell+2}}
=-\left(\ell+1-\frac{1}{d-2}\right)+\OO\left((d-1)^{-\ell}\right),\\
&\sum_{\al\neq \beta\in \sfA_i}
  \frac{2\msc^{2\ell}(z_t)L_{l_\al l_\beta}}{(d-1)^{\ell+2}}=-2(\ell+1)
+\OO\left((d-1)^{-\ell}\right),\\
&\sum_{\al\neq\beta\in \sfA_i} \frac{2\md(z_t)\msc^{6\ell}(z_t)}{(d-1)^{2\ell+3}}=-\frac{2}{d-2}+\OO\left((d-1)^{-\ell}\right),
\end{split}\end{align}
and for the diagonal term, we have
\begin{align}\label{e:schur_two3}
\sum_{\al\in \sfA_i}\frac{\msc^{2\ell}(z_t)L_{ l_\al  l_\al}^{(i)}}{(d-1)^{\ell+2}} =-\frac{1}{d-2}+\OO((d-1)^{-\ell}).
\end{align}
If $|z+E_t|\leq N^{-\fg}$,  analogous statements hold after multiplying the right-hand sides of \eqref{e:schur_two1} and \eqref{e:schur_two3} by $-1$. 
\end{lemma}

When $I(\cF^+,\cG)=1$, vertex $o$ has radius $\fR$ tree neighborhood. Then $\mu=d(d-1)^{\ell}$, $|\sfA_i|=(d-1)^{\ell+1}$, $L^{(i)}_{l_\al l_\beta}$ and $L_{l_\al l_\beta}$ are given explicitly by the Green's function of trees from \Cref{greentree}. \Cref{l:second_term} follows from direct computation. Moreover, $\msc(z_t)=\pm 1+\OO(N^{-\fg/2})$ if $|z\pm E_t|\leq N^{-\fg}$. We postpone its proof to \Cref{s:tree_computation}.

The following lemma states that averages of Green's functions (with some vertices removed) can be written in terms of $m_t(z)$. Later, it will be used to prove \Cref{l:first_term}. 

\begin{lemma}\label{l:error_term}
Adopt the notation of \Cref{s:setting7}. Fix $z\in \bf D$ such that $\min\{|z-E_t|, |z+E_t|\}\leq N^{-\fg}$, and assume $\cG\in \Omega$. We recall the indicator function $I_{cc'}$ from \Cref{def:Ic}. The following holds with overwhelmingly high probability over $Z$:
\begin{align}\begin{split}\label{e:remove_indices}
&\frac{1}{(Nd)^2}\sum_{b\sim c \atop b'\sim c'}(G_{c c' }^{(b)})^2 I_{cc'}
=\frac{1}{\cA}\frac{\del_z m_t(z)}{N}+\OO(N^{-\fb/2} \Phi),\\
&\frac{1}{(Nd)^2}\sum_{b\sim c \atop b'\sim c'}(G_{c c' }^{(b b')})^2 I_{cc'}
=\frac{1}{\cA^2}\frac{\del_z m_t(z)}{N}+\OO(N^{-\fb/2} \Phi).
\end{split}\end{align}
Moreover, for dummy variables $x,x'\in \{b, c\}, y,y'\in \{b', c'\}, w,w'\in \{u, v\}$, with $\theta=\bm1(x\neq x')+\bm1(y\neq y')+\bm1(w\neq w')$ we have
\begin{align}\label{e:sumGGG}
&\phantom{{}={}}\frac{1}{(Nd)^3}\sum_{b\sim c \atop b'\sim c'}\sum_{u\sim v}
G_{xy'}G_{yw'}G_{wx'}= \left(\frac{2\sqrt{d-1}}{d}\right)^\theta\frac{\del^2_z m_t(z)}{2N^2}+\OO\left(\frac{N^{-\fb/2} \Phi}{N\eta}\right).
\end{align}
Moreover, take a forest $\cF=\{\{b,c\}, \{b',c'\}\}$ consisting of two unused core edges, and denote $I(\cF,\cG)=I_{cc'}$ (which satisfies the requirements in \eqref{def:Ic}), the following estimates also hold,
\begin{align}\begin{split}\label{e:GIGG1}
&\frac{1}{(Nd)^3}\sum_{u\sim v}\sum_{b\sim c\atop b'\sim c'} G_{c c'}^{(b b')} I_{cc'}
    \left( G^{\circ}_{v c}+\frac{\msc(z_t)  G^{\circ}_{v b}}{\sqrt{d-1}}-\frac{G_{uv}}{G_{uu}}\left( G^{\circ}_{u c}+\frac{\msc(z_t)  G^{\circ}_{u b}}{\sqrt{d-1}}\right)\right)\\
    &\phantom{\frac{1}{(Nd)^3}\sum_{u\sim v}\sum_{b, c\atop b', c'}}\times\left( G^{\circ}_{v c'}+\frac{\msc(z_t)  G^{\circ}_{v b'}}{\sqrt{d-1}}-\frac{G_{uv}}{G_{uu}}\left( G^{\circ}_{u c'}+\frac{\msc(z_t)  G^{\circ}_{u b'}}{\sqrt{d-1}}\right)\right)
= \frac{\del^2_z m_t(z)}{2\cA^3N^2}+\OO\left(\frac{N^{-\fb/2} \Phi}{N\eta}\right),\\
&\frac{1}{(Nd)^3}\sum_{u\sim v}\sum_{b\sim c\atop b'\sim c'} G_{c c'}^{(b b')} I_{cc'}\left( G^{\circ}_{u c}+\frac{\msc(z_t)  G^{\circ}_{u b}}{\sqrt{d-1}}\right)\left( G^{\circ}_{u c'}+\frac{\msc(z_t)  G^{\circ}_{u b'}}{\sqrt{d-1}}\right)=\frac{\del^2_z m_t(z)}{2\cA^2N^2}+\OO\left(\frac{N^{-\fb/2} \Phi}{N\eta}\right).
\end{split}\end{align}
\end{lemma}

\begin{proof}[Proof of \Cref{l:error_term}]

Without loss of generality, we assume that  $|z-E_t|\leq N^{-\fg}$. 
By \eqref{e:xi_behavior} and \eqref{e:relation_zt_z}, $z_t=z+t\md(z,t)$ satisfies
\begin{align}\label{e:z-2}
    \sqrt{z_t-2}=\OO(t+\sqrt{z-E_t})\lesssim N^{-\fg},\quad |z-2|\lesssim |E_t-2|+\OO(N^{-\fg})=\OO(N^{-\fg}),
\end{align}
where we used $t\leq N^{-1/3+\ft}$. We recall from \eqref{e:medge_behavior}
\begin{align}\begin{split}\label{e:mscmd}
\msc(z_t)&=-1+\OO(\sqrt{|z_t-2|})=-1+\OO(N^{-\fg/2}), \\
 \md(z_t)&=-\frac{d-1}{d-2}+\OO(\sqrt{|z_t-2|})=-\frac{d-1}{d-2}+\OO(N^{-\fg/2}).
\end{split}\end{align}
Conditioned on $A_{bc}A_{b'c'}I_{cc'}=1$, by \eqref{eq:infbound} and \eqref{e:mscmd}, the  terms
\begin{align}\begin{split}\label{e:centered_term}
   &G_{bc}-\frac{\sqrt{d-1}}{d-2},\quad  G_{b'c'}-\frac{\sqrt{d-1}}{d-2}, \quad 
   G_{c b'},\quad  G_{c b'},\quad  G_{c b'},\quad G_{c b'}\\
   &G_{bb}+\frac{d-1}{d-2},\quad  G_{c c}+\frac{d-1}{d-2}, \quad  G_{b'b'}+\frac{d-1}{d-2}, \quad 
 G_{c'c'}+\frac{d-1}{d-2},
\end{split}\end{align}
are all bounded by $\OO(N^{-\fb})$.
The Schur complement formula \eqref{e:Schurixj} and \eqref{e:Gbbcc} imply
\begin{align}\begin{split}\label{e:remove_one}
    &G_{c c'}^{(b)}
    =G_{c b'}-\frac{G_{c b}G_{b b'}}{G_{b b}}=
G_{c b'}-\frac{G_{b b'}}{\sqrt{d-1}}+\OO(N^{-\fb} |G_{b b'}|)\\
   &G_{c c'}^{(b b')}
    =G_{cc'}-\frac{G_{cb}G_{bc'}}{G_{bb}}-\frac{G_{cb'}G_{b'c'}}{G_{b'b'}} +\frac{G_{cb}G_{bb'}G_{b'c'}}{G_{bb}G_{b'b'}}+ \OO\left(|G_{bb'}|^2+|G_{b'b'}|^2+|G_{bc'}|^2+|G_{cb'}|^2+\frac{1}{N^2}\right) \\ 
&
    = G_{c c'}-\frac{ G_{b c'}}{\sqrt{d-1}}-\frac{ G_{c b'}}{\sqrt{d-1}}+\frac{ G_{b b'}}{d-1} +\cE,
\end{split}\end{align}
where $|\cE|\lesssim N^{-\fb} (|G_{c b'}|+|G_{c b'}|+|G_{c b'}|+|G_{c b'}|)+1/{N^2}$.

By plugging \eqref{e:remove_one} into the first statement in \eqref{e:remove_indices}, we get
\begin{align}\begin{split}\label{e:square1}
   &\phantom{{}={}}\frac{1}{(Nd)^2}\sum_{b\sim c \atop b'\sim c'}(G_{c c' }^{(b)})^2 I_{cc'} \\
   &=\frac{1}{(Nd)^2}\sum_{b\sim c \atop b'\sim c'}\left(G^2_{c b'}+\frac{G^2_{b b'}}{d-1}-\frac{2G_{c b'}G_{b b'}}{\sqrt{d-1}} +\OO\left(\frac{(|G_{bb'}|^2+|G_{cb'}|^2}{N^{\fb}}+\frac1{N^2}\right)\right)I_{cc'}\\
    &=\frac{1}{N}\frac{d \Tr[G^2]}{d-1}-\frac{1}{(Nd)^2}\sum_{c, b'\sim c'}\frac{2G_{c b'}\sum_{x\sim c}G_{x b'}}{d\sqrt{d-1}} +\OO(N^{-\fb/2}\Phi+N^{-1+3\fc/2}),
\end{split}\end{align}
where in the last statement we used \eqref{e:Gest} and $\sum_{cc'}|1-I_{cc'}|\lesssim N^{1+3\fc/2}$.
Noticing that $(HG)_{c b'}+\sqrt t (ZG)_{cb'}=(H(t)G)_{c b'}=zG_{c b'}+\delta_{c b'}=2G_{c b'}+((z-2)G_{c b'}+\delta_{c b'})$,  it follows that
\begin{align}\label{e:sum_neighbor}
    \frac{1}{\sqrt{d-1}}\sum_{x\sim c}G_{x b'}
    &=2G_{c b'}+\cE_{cb'},\quad \cE_{cb'}=\OO(N^{-\fb}|G_{c b'}|+\delta_{c b'}+\sqrt{t}|(ZG)_{c b'}|),
\end{align}
where we used $|z-2|\lesssim N^{-\fg}\leq N^{-\fb}$ from \eqref{e:z-2}.
Thus, using \eqref{e:sum_neighbor} we can rewrite the second term on the right-hand side of \eqref{e:square1} as
\begin{align}\begin{split}\label{e:square2}
    &\phantom{{}={}}-\frac{1}{(Nd)^2}\sum_{c, b'\sim c'}\frac{2G_{c b'}}{d}\frac{\sum_{x\sim c}G_{x b'}}{\sqrt{d-1}}\\
    &=-\frac{4\Tr[G^2]}{Nd}+\frac{1}{N^2}\sum_{c, b'\sim c'}\OO(|G_{c b'}|(\sqrt{t}|(ZG)_{c b'}|+N^{-\fb} |G_{c b'}|+\delta_{c b'})),
\end{split}\end{align}
where thanks to \eqref{e:Gest} and \eqref{e:Werror}, the second term on the right-hand side of \eqref{e:square2} is bounded by $\OO(N^{-\fb} N^\fo \Phi)=\OO(N^{-\fb/2}\Phi)$. The first statement in \eqref{e:remove_indices} follows from combining \eqref{e:square1} and \eqref{e:square2}, and noticing $\Tr[G^2]/N=\del_z m_t(z)$. The second statement in \eqref{e:remove_indices} can be proven in the same way, so we omit its proof. 

Next we prove \eqref{e:sumGGG}. If $\theta=0$ then \eqref{e:sumGGG} simplifies to 
\begin{align}\label{e:GGGtheta0}
    \frac{1}{N^3}\sum_{x,y,w}G_{xy}G_{yw}G_{wx}=\frac{\Tr G^3}{N^3}=\frac{\del_z^2 m_t(z)}{2N^2}.
\end{align}
Next we show that the cases $\theta\geq 1$ can be reduced to \eqref{e:GGGtheta0}. Assume $\theta\geq 1$ and $(w, w')=(u,v)$,  we can first sum over $v$ with $ v\sim u$, and \eqref{e:sum_neighbor} implies
\begin{align}\begin{split}\label{e:sumGGG2}
&\phantom{{}={}}\frac{1}{(Nd)^3}\sum_{b\sim c \atop b'\sim c'}\sum_{u\in\qq{N}}
G_{xy'}G_{ux'}\sum_{v:v\sim u}G_{yv}\\
&=\frac{2\sqrt{d-1}}{d}\frac{1}{N(Nd)^2}\sum_{b\sim c \atop b'\sim c'}\sum_{u\in\qq{N}}
G_{xy'}G_{ux'}G_{yu}+\frac{\sqrt{d-1}}{(Nd)^3}\sum_{b'\sim c'}\sum_{b\sim c, u\in\qq{N}}
G_{xy'}G_{ux'}
\cE_{yu}.
\end{split}\end{align}
For the second term on the right-hand side of \eqref{e:sumGGG2}, using that $\|G\|_{\rm spec}\leq 1/\eta$ and $\cE_{yu}=\OO(N^{-\fb}|G_{yu}|+\delta_{yu}+\sqrt{t}|(ZG)_{yu}|)$ from \eqref{e:sum_neighbor}, similar to \eqref{e:trianglesvcu0}, we can bound it as
\begin{align*}
    \frac{1}{(Nd)^3\eta}\sum_{b'\sim c'} \sqrt{\sum_{b\sim c}
|G_{xy'}|^2}
\sqrt{\sum_{u\in\qq{N}} \left(N^{-\fb}|G_{yu}|+\delta_{yu}+\sqrt{t}|(ZG)_{yu}|\right)^2}
\lesssim \frac{(N^{-\fb}+\sqrt t )N^{2\fo}\Phi}{N\eta},
\end{align*}
where we used \eqref{e:Gest} and \eqref{e:Werror}. Thus it is bounded by $N^{-\fb/2} \Phi/(N\eta)$. 

The first term on the right-hand side of \eqref{e:sumGGG2} has the same form as in \eqref{e:sumGGG}, but with $\theta$ reduced by $1$ and an additional factor $2\sqrt{d-1}/d$. By repeating this procedure, we will end at the expression \eqref{e:GGGtheta0} with a factor $(2\sqrt{d-1}/d)^\theta$.

Finally we prove \eqref{e:GIGG1}.
The expression in \eqref{e:GIGG1} decomposes into terms in the following form: $x\in \{b, c\}, y\in \{b', c'\}, w,w'\in \{u,v\}$ and $0\leq f\leq 2$.
\begin{align}\begin{split}\label{e:sumGGG3}
&\phantom{{}={}}\frac{1}{(Nd)^3}\sum_{b\sim c \atop b'\sim c'}\sum_{u\sim v}
G_{c c'}^{(b b')}I_{c c'} G^{\circ}_{xw} G^{\circ}_{yw'}\left(\frac{G_{uv}}{G_{uu}}\right)^{f}\\
&=\frac{1}{(Nd)^3}\sum_{b\sim c \atop b'\sim c'}\sum_{u\sim v}
\left(G_{c c'}-\frac{ G_{b c'}}{\sqrt{d-1}}-\frac{ G_{c b'}}{\sqrt{d-1}}+\frac{ G_{b b'}}{d-1} +\cE\right)I_{c c'} G^{\circ}_{xw} G^{\circ}_{yw'}\left(\frac{G_{uv}}{G_{uu}}\right)^{f},
\end{split}\end{align}
where the second line uses \eqref{e:remove_one}.
We begin by simplifying \eqref{e:sumGGG3}.
\begin{enumerate}

\item We notice that we can expand $ G^{\circ}_{xw} G^{\circ}_{yw'}=G_{xw}G_{yw'}-G_{xw}L_{yw'}-L_{xw}G_{yw'}+L_{xw}L_{yw'}$. First we show that up to negligible error, we can replace $ G^{\circ}_{xw} G^{\circ}_{yw'}$ in \eqref{e:sumGGG3} by $G_{xw}G_{yw'}$.  
Since $x\in \{b,c\}$ and $y\in \{b',c'\}$ belong to two different core edges, for $u\sim v$, $L_{xw} L_{yw'}=0$.
For $G_{xw}L_{yw'}$, using $\|G\|_{\rm spec}\leq 1/\eta$,\eqref{e:use_Ward}  and  \eqref{e:naive-Ward}
\begin{align}\begin{split}\label{e:triangle_bound0}
    &\phantom{{}={}}\frac{1}{(Nd)^3}\sum_{b'\sim c'}\left|\sum_{ b\sim c}\sum_{u\sim v}
G_{c c'}^{(b b')}I_{c c'}G_{xw}L_{yw'}\left(\frac{G_{uv}}{G_{uu}}\right)^{f}\right|\\
&\lesssim \frac{1}{(Nd)^3\eta}\sum_{b' \sim c'} 
\sqrt{\sum_{b\sim c}|G_{c c'}^{(b b')}|^2I_{c c'}}\sqrt{\sum_{u\sim v} |L_{yw'}|^2}\lesssim \frac{\sqrt{\fR N^\fo\Phi/N}}{N\eta}
\lesssim \frac{N^{-\fb} \Phi}{N\eta}.
\end{split}\end{align}
For $L_{xw}G_{yw'}$, we have exactly the same bound as in \eqref{e:triangle_bound0}, by first summing over $b'\sim c', u\sim v$. 
\item Next we show that we can remove the error term $\cE$ in \eqref{e:sumGGG3}. In fact, by the same argument as in \eqref{e:triangle_bound0}, we have
\begin{align}\begin{split}\label{e:triangle_bound}
    &\phantom{{}={}}\frac{1}{(Nd)^3}\sum_{b'\sim c'}\left|\sum_{ b\sim c}\sum_{u\sim v}
\cE I_{c c'}G_{xw}G_{yw'}\left(\frac{G_{uv}}{G_{uu}}\right)^{f}\right|\\
&\lesssim \frac{1}{(Nd)^3\eta}\sum_{b' \sim c'} 
\sqrt{\sum_{b\sim c}|\cE|^2I_{c c'}}\sqrt{\sum_{u\sim v} |G_{yw'}|^2}
\lesssim \frac{N^{-\fb/2} \Phi}{N\eta},
\end{split}\end{align}
where 
$|\cE|\lesssim N^{-\fb} (|G_{c b'}|+|G_{c b'}|+|G_{c b'}|+|G_{c b'}|)+1/{N^2}$.
 \item  Since $(Nd)^{-2}\sum_{b\sim c \atop b'\sim c'}|1-I_{cc'}|\leq N^{-1+3\fc/2}$, we can remove the indicator function $I_{cc'}$ in \eqref{e:sumGGG3}, and the error is bounded by
 \begin{align*}
     \frac{1}{(Nd)^3}\sum_{b\sim c \atop b'\sim c'}\sum_{u\sim v}
|1-I_{c c'}|| G_{xw} ||G_{yw'}|\lesssim \frac{N^\fo \Phi}{N^{1-3\fc/2}}\lesssim \frac{N^{-\fb} \Phi}{N\eta},
 \end{align*}
 where we used \eqref{e:Gest}, and $\eta\leq N^{-\fg}\leq N^{-2\fc}$.  
    \item Finally if $\{u,v\}$ has a radius $\fR$ tree neighborhood, then thanks to \eqref{eq:local_law} and \eqref{e:mscmd}, we will have that $|G_{uv}/G_{uu}-1/\sqrt{d-1}|\lesssim N^{-\fb} $.
    Moreover, since $\cG\in \Omega$, the number of edges which  do not have a radius $\fR$ tree neighborhood is bounded by $N^\fc$. Thus
    up to error $\OO(N^{-\fb} \Phi/N\eta)$, we can replace $G_{uv}/G_{uu}$ in \eqref{e:sumGGG3} by $1/\sqrt{d-1}$.
\end{enumerate}
After applying the above procedure, and ignoring negligible errors of size $\OO(N^{-\fb/2} \Phi/N\eta)$, we simplify \eqref{e:sumGGG3} as
\begin{align*}
     \frac{1}{(d-1)^{f/2}(Nd)^3}\sum_{b\sim c \atop b'\sim c'}\sum_{u\sim v}
\left(G_{c c'}-\frac{ G_{b c'}}{\sqrt{d-1}}-\frac{ G_{c b'}}{\sqrt{d-1}}+\frac{ G_{b b'}}{d-1}\right)G_{xw}G_{yw'}.
\end{align*}
This expression can then be computed using \eqref{e:sumGGG}. We can now compute the first statement in \eqref{e:GIGG1}
\begin{align*}
    &\phantom{{}={}}\frac{1}{(Nd)^3}
    \sum_{u\sim v}\sum_{b\sim c\atop b'\sim c'} \left(G_{c c'}-\frac{ G_{b c'}+G_{c b'}}{\sqrt{d-1}}+\frac{ G_{b b'}}{d-1}\right)\left(G_{v c}-\frac{ G_{v b}+G_{u c}}{\sqrt{d-1}}+\frac{ G_{u b}}{d-1}\right)\\
    &\times\left(G_{v c'}-\frac{ G_{v b'}+G_{u c'}}{\sqrt{d-1}}+\frac{G_{u b'}}{d-1}\right)= \frac{1}{\cA^3}\frac{\del_z m_t(z)}{2N^2}+\OO\left(\frac{N^{-\fb/2} \Phi}{N\eta}\right).
\end{align*}
Similarly, the second statement in \eqref{e:GIGG1} follows from 
\begin{align*}
&\phantom{{}={}}\frac{1}{(Nd)^3}\sum_{u\sim v}\sum_{b\sim c\atop b'\sim c'} \left(G_{c c'}-\frac{ G_{b c'}+G_{c b'}}{\sqrt{d-1}}+\frac{ G_{b b'}}{d-1}\right)\left(G_{u c}-\frac{G_{u b}}{\sqrt{d-1}}\right)\left(G_{u c'}-\frac{G_{u b'}}{\sqrt{d-1}}\right)\\
&= \frac{1}{\cA^2}\frac{\del_z m_t(z)}{2N^2}+\OO\left(\frac{N^{-\fb/2} \Phi}{N\eta}\right).
\end{align*}
\end{proof}

\begin{proof}[Proof of \Cref{l:first_term}]
Without loss of generality, we assume that  $|z-E_t|\leq N^{-\fg}$. The other case that $|z+E_t|\leq N^{-\fg}$ can be proven in the same way, so we omit its proof.

  We notice that $|\sfA_i|=(d-1)^{\ell+1}$, and the expectation in \eqref{e:first_term0} depends only on $\al=\beta$ or $\al\neq \beta$. We can decompose \eqref{e:first_term0} as 
    \begin{align}\label{e:decompose}
    \frac{1}{Z_{\cF^+}}\sum_{\bfi^+}\frac{\msc^{2\ell}(z_t)}{(d-1)^{\ell+1}}\sum_{\al,\beta\in \sfA_i}\bE\left[\bm1(\cF^+,\cG)I(\cF^+,\cG) (\wt G^{(\bT)}_{c_\al c_\beta}-G^{(b_\al b_\beta)}_{c_\al c_\beta})(Q_t-Y_t)^{p-1}\right]=:J_1+J_2+J_3,
    \end{align}
where $J_1, J_2, J_3$ are given by
    \begin{align}\begin{split}\label{e:defJ123}
       &J_1= \sum_{\beta\in \sfA_i}\frac{\msc^{2\ell}(z_t)}{Z_{\cF^+}}\sum_{\bfi^+}\bE\left[\bm1(\cF^+,\cG)I(\cF^+,\cG) (\wt G^{(\bT)}_{c_\al c_\beta}-G^{(\bT \bW)}_{c_\al c_\beta})(Q_t-Y_t)^{p-1}\right],\\
       &J_2= \sum_{\beta\in \sfA_i}\frac{\msc^{2\ell}(z_t)}{Z_{\cF^+}}\bE\left[\bm1(\cF^+,\cG)I(\cF^+,\cG) ( G^{(\bT \bW)}_{c_\al c_\beta}-G^{(\bT)}_{c_\al c_\beta})(Q_t-Y_t)^{p-1}\right],\\
       &J_3= \sum_{\beta\in \sfA_i}\frac{\msc^{2\ell}(z_t)}{Z_{\cF^+}}\sum_{\bfi^+}\bE\left[\bm1(\cF^+,\cG)I(\cF^+,\cG) (G^{(\bT)}_{c_\al c_\beta}-G^{(b_\al b_\beta)}_{c_\al c_\beta})(Q_t-Y_t)^{p-1}\right].
    \end{split}\end{align}

For $J_1$ in \eqref{e:defJ123}, thanks to \eqref{e:diffG1}, we have
\begin{align}\begin{split}\label{e:Ggamma}
&\phantom{{}={}}\frac{1}{Z_{\cF^+}}\sum_{\bfi^+}I(\cF^+,\cG) (\wt G^{(\bT)}_{c_\al c_\beta}-G^{(\bT \bW)}_{c_\al c_\beta})=\OO\left(\frac{\Phi}{N^{\fb/4}}\right)\\
&+ \sum_{\bfi^+} \sum_{\gamma\in\qq{\mu}} \frac{\md(z_t)I(\cF^+,\cG) }{(d-1)Z_{\cF^+}}\left(\sum_{c_\gamma\neq x\sim b_\gamma}G_{c_\al x}^{(\bT\bW)} +G^{(\bT\bW)}_{c_\al a_\gamma}\right)\left(\sum_{c_\gamma\neq x\sim b_\gamma}G_{c_\beta x}^{(\bT\bW)} +G^{(\bT\bW)}_{c_\beta a_\gamma}\right),
\end{split}\end{align}
where, the above error term is from 
\begin{align*}
    \frac{1}{Z_{\cF^+}}\sum_{\bfi^+}I(\cF^+,\cG) \left(N^{-\fb/2}\sum_{\gamma\in\qq{\mu}}\sum_{x\in \cN_\gamma}(|G_{c_\al x}^{(\bT\bW)}|^2+|G_{c_\beta x }^{(\bT\bW)}|^2)+N^{-\fb }\Phi\right)\lesssim \frac{\Phi}{N^{\fb/4}},
\end{align*}
which follows from \eqref{e:use_Ward}.
Thanks to \eqref{e:Greplace}, up to negligible error $\OO(N^{-\fb/2} \bE[\Psi_p])$, we can further replace $ G_{c_\al x}^{(\bT \bW)},G^{(\bT\bW)}_{c_\al a_\gamma}, G_{c_\beta x}^{(\bT \bW)},G^{(\bT\bW)}_{c_\beta a_\gamma}$ in \eqref{e:Ggamma} by $ G_{c_\al x}^{( b_\al b_\gamma)},G^{(b_\al l_\gamma)}_{c_\al a_\gamma}, G_{ c_\beta x}^{( b_\beta b_\gamma)},G^{(b_\beta l_\gamma)}_{c_\beta a_\gamma}$ respectively, giving
\begin{align}\label{e:Ggamma2}
 \sum_{\bfi^+} \sum_{\gamma\in\qq{\mu}} \frac{\md(z_t)I(\cF^+,\cG) }{(d-1)Z_{\cF^+}}\left(\sum_{c_\gamma\neq x\sim b_\gamma}G_{c_\al x}^{(b_\al b_\gamma)} +G^{(b_\al l_\gamma)}_{c_\al a_\gamma}\right)\left(\sum_{c_\gamma\neq x\sim b_\gamma}G_{c_\beta x}^{(b_\beta b_\gamma)} +G^{(b_\beta l_\gamma)}_{c_\beta a_\gamma}\right).
\end{align}

If $\al\neq \beta$, then either $\gamma\neq \al$ or $\gamma\neq \beta$. If $\gamma\neq \al$, we can first sum over $(b_\al, c_\al)$ in \eqref{e:Ggamma2}, and \eqref{e:Gccbb_youyige} implies that $\eqref{e:Ggamma2}=\OO(N^{-\fb/4} \Phi)$. The same estimate holds if $\gamma\neq \beta$. 
For the remaining terms $\al= \beta$, then 
\begin{align}\begin{split}\label{e:Ggamma3}
    &\phantom{{}={}}\sum_{\bfi^+}\frac{1}{Z_{\cF^+}}\bE\left[\bm1(\cG\in \Omega) I(\cF^+,\cG) \left(\sum_{c_\gamma\neq x\sim b_\gamma}G_{c_\al x}^{(b_\al b_\gamma)}+G^{(b_\al l_\gamma)}_{c_\al a_\gamma}\right)^2(Q_t-Y_t)^{p-1}\right]\\
    &=\sum_{\bfi^+}\frac{1}{Z_{\cF^+}}\bE\left[\bm1(\cG\in \Omega) I(\cF^+,\cG) \left(\sum_{c_\gamma\neq x\sim b_\gamma}(G_{c_\al x}^{(b_\al b_\gamma)})^2 +(G^{(b_\al l_\gamma)}_{c_\al a_\gamma})^2\right)(Q_t-Y_t)^{p-1}\right]+\OO\left(\frac{\bE[\Psi_p]}{N^{\fb/4}} \right)\\
    &=d\bE\left[\bm1(\cG\in \Omega) \frac{\del_z m_t(z)}{\cA^2 N}(Q_t-Y_t)^{p-1}\right]+\OO\left(\frac{\bE[\Psi_p]}{N^{\fb/4}} \right),
\end{split}
\end{align}
where in the first statement we use \eqref{e:GiGjGk} to bound cross terms $G_{c_\al x}^{(b_\al)}G_{c_\al x'}^{(b_\al)}$ from $\gamma=\al$, and \eqref{e:GiGj} to bound cross terms $G_{c_\al x}^{(b_\al b_\gamma)}G^{(b_\al l_\gamma)}_{c_\al a_\gamma}$, and 
$G_{c_\al x}^{(b_\al b_\gamma)}G_{c_\al x'}^{(b_\al b_\gamma)}$ with $\gamma\neq \al$; in the second statement we use \eqref{e:GiGi} to compute $(G_{c_\al x}^{(b_\al)})^2$ from $\gamma= \al$; and \eqref{e:remove_indices} to compute $(G_{c_\al a_\gamma }^{(b_\al l_\gamma)})^2$ and $(G_{c_\al x}^{(b_\al b_\gamma)})^2$ with $\gamma\neq \al$.
By plugging \eqref{e:Ggamma}, \eqref{e:Ggamma2} and \eqref{e:Ggamma3} into $J_1$ in \eqref{e:defJ123}, we conclude
\begin{align}\begin{split}\label{e:Gt1}
    J_1
    &=\frac{d\md(z_t)}{d-1}\sum_{\gamma\in\qq{\mu}}\bE\left[\bm1(\cG\in \Omega) \frac{\del_z m_t(z)}{\cA^2 N}(Q_t-Y_t)^{p-1}\right]+\OO\left(\frac{\bE[\Psi_p]}{(d-1)^\ell} \right)\\
    &=\md(z_t) d^2(d-1)^{\ell-1}\bE\left[\bm1(\cG\in \Omega) \frac{\del_z m_t(z)}{\cA^2 N}(Q_t-Y_t)^{p-1}\right]+\OO\left(\frac{\bE[\Psi_p]}{(d-1)^\ell} \right)\\
    &=- d(d-2)(d-1)^{\ell-1}\bE\left[\bm1(\cG\in \Omega) \frac{\del_z m_t(z)}{\cA N}(Q_t-Y_t)^{p-1}\right]+\OO\left(\frac{\bE[\Psi_p]}{(d-1)^\ell} \right),
\end{split}\end{align}
where in the second statement we used $\mu=d(d-1)^\ell$; in the third statement we used \eqref{e:mscmd} that $\md(z_t)=-(d-1)/(d-2)+\OO(N^{-\fg/2})$ and $\cA=d(d-1)/(d-2)^2$ from \eqref{e:edge_behavior}.

For $J_2$ in \eqref{e:defJ123}, thanks to \eqref{e:diffG2}, we have
 \begin{align}\label{e:diffG}
\frac{1}{Z_{\cF^+}}\sum_{\bfi^+}I(\cF^+,\cG) (G^{(\bT \bW)}_{c_\al c_\beta}-G^{(\bT)}_{c_\al c_\beta})= - \sum_{\gamma\in\qq{\mu}\setminus\{\al,\beta\}}\sum_{\bfi^+} \frac{I(\cF^+,\cG)}{\md(z_t) Z_{\cF^+}} G^{(\bT b_\al b_\beta )}_{c_\al b_\gamma}  G^{(\bT b_\al b_\beta)}_{b_{\gamma} c_\beta}+\OO\left(\frac{\Phi}{N^{\fb/4}}\right).
    \end{align}
Here, the error term above is from 
\begin{align*}
    \frac{1}{Z_{\cF^+}}\sum_{\bfi^+}I(\cF^+,\cG) \left(N^{-\fb/2}\sum_{\gamma\in\qq{\mu}\setminus\{\al,\beta\}}(|G_{c_\al b_\gamma}^{(\bT b_\al b_\beta)}|^2+|G_{c_\beta b_\gamma }^{(\bT b_\al b_\beta)}|^2)\right)\lesssim \frac{\Phi}{N^{\fb/4}},
\end{align*}
which follows from \eqref{e:use_Ward}.
By \eqref{e:Greplace}, up to negligible error $\OO(N^{-\fb/2} \bE[\Psi_p])$,
we can then replace $ G_{c_\al b_\gamma}^{(\bT b_\al b_\beta)}, G_{b_\gamma c_\beta}^{( \bT b_\al b_\beta)}$ in \eqref{e:diffG} by $ G_{c_\al b_\gamma}^{( b_\al)}, G_{b_\gamma c_\beta}^{( b_\beta)}$ respectively, giving
\begin{align}\label{e:diffG_term2}
    - \sum_{\gamma\in\qq{\mu}\setminus\{\al,\beta\}}\sum_{\bfi^+} \frac{I(\cF^+,\cG)}{\md(z_t) Z_{\cF^+}}\sum_{\gamma\in\qq{\mu}\setminus\{\al,\beta\}}G^{( b_\al  )}_{c_\al b_\gamma}  G^{(b_\beta )}_{b_{\gamma} c_\beta}.
    \end{align}

If $\al\neq \beta$, we can sum over $(b_\al, c_\al)$ using \eqref{e:Gccbb_youyige}, then $\eqref{e:diffG_term2}=\OO(N^{-\fb/4} \bE[\Psi_p])$. For the remaining terms we have $\al= \beta$. By plugging \eqref{e:diffG} and \eqref{e:diffG_term2} into the expression of $J_2$ in \eqref{e:defJ123},  we get
\begin{align}\begin{split}\label{e:Gt2}
    J_2
    &=-\sum_{\gamma\neq \al}\frac{\msc^{2\ell}(z_t)}{\md(z_t)}\sum_{\bfi^+}\frac{1}{Z_{\cF^+}}\bE\left[I(\cF^+,\cG) \bm1(\cG\in \Omega)(G^{(b_\al  )}_{c_\al b_\gamma})^2 (Q_t-Y_t)^{p-1}\right]+\OO\left(\frac{\bE[\Psi_p]}{(d-1)^\ell} \right)\\
    &=\frac{d-2}{d-1}(d(d-1)^\ell-1)\bE\left[\bm1(\cG\in \Omega) \frac{\del_z m_t(z)}{\cA N}(Q_t-Y_t)^{p-1}\right]+\OO\left(\frac{\bE[\Psi_p]}{(d-1)^\ell} \right),
    \end{split}\end{align}
 where in the second statement,  we use   \eqref{e:mscmd} and \eqref{e:remove_indices}.

For $J_3$ in \eqref{e:defJ123}, thanks to \eqref{e:diffG3}, we have
\begin{align}\begin{split}\label{e:disan}
&\phantom{{}={}}\frac{1}{Z_{\cF^+}}\sum_{\bfi^+}I(\cF^+,\cG) (G^{(\bT)}_{c_\al c_\beta}-G^{(b_\al b_\beta)}_{c_\al c_\beta})\\
&=-\frac{1}{Z_{\cF^+}}\sum_{\bfi^+}I(\cF^+,\cG) \sum_{\dist(x,o)=\ell\atop  x\sim x'\in \bT}\left(\frac{G_{c_\al x}^{(b_\al b_\beta)}G^{(b_\al b_\beta)}_{x' c_\beta}}{\sqrt{d-1}}
+\frac{G_{c_\al x}^{(b_\al b_\beta)}G^{(b_\al b_\beta)}_{x c_\beta}}{\msc(z_t)}\right)
+\OO\left(\frac{\Phi}{N^{\fb/4}}\right).
\end{split}\end{align}
Here, the error term above is from 
\begin{align*}
    \frac{1}{Z_{\cF^+}}\sum_{\bfi^+}I(\cF^+,\cG) \left(N^{-\fb/2}\sum_{x\in \bT}(|G_{c_\al x}^{(\bT b_\al b_\beta)}|^2+|G_{c_\beta x }^{(\bT b_\al b_\beta)}|^2+|(ZG^{(b_\al b_\beta)})_{x c_\beta}|^2)\right)\lesssim \frac{\Phi}{N^{\fb/4}}.
\end{align*}
which follows from \eqref{e:use_Ward} and \eqref{e:Werror}.
Thanks to the Schur complement formula \eqref{e:Schurixj}, up to error $\OO(N^{-\fb/2} \bE[\Psi_p])$, we can replace $G_{c_\al x}^{(b_\al b_\beta)}, G_{ x' c_\beta}^{(b_\al b_\beta)}, G^{(b_\al b_\beta)}_{x c_\beta}$ in \eqref{e:disan} with $G_{c_\al x}^{(b_\al )}, G_{x' c_\beta }^{(b_\beta )}, G^{( b_\beta)}_{x c_\beta}$ respectively,
\begin{align}\label{e:boundaryTGG}
&-\frac{1}{Z_{\cF^+}}\sum_{\bfi^+}I(\cF^+,\cG) \sum_{\dist(x,o)=\ell\atop  x\sim x'\in \bT}\left(\frac{G_{c_\al x}^{(b_\al)}G^{(b_\beta)}_{x' c_\beta}}{\sqrt{d-1}}
+\frac{G_{c_\al x}^{(b_\al )}G^{( b_\beta)}_{x c_\beta}}{\msc(z_t)}\right).
\end{align}
If $\al\neq \beta$, we can sum over $(b_\al, c_\al)$ using \eqref{e:Gccbb_youyige}, then  $\eqref{e:boundaryTGG}=\OO(N^{-\fb/4} \bE[\Psi_p])$. 

For the remaining cases $\al=\beta$, and \eqref{e:boundaryTGG} reduces to 
\begin{align}\label{e:Gt30}
-\frac{\msc^{2\ell}(z_t)}{Z_{\cF^+}}\sum_{\bfi^+}\bE\left[\bm1(\cG\in\Omega)I(\cF^+,\cG)\sum_{\dist(x,o)=\ell\atop  x\sim x'\in \bT}\left(\frac{G_{c_\al x}^{(b_\al)}G^{(b_\al)}_{x' c_\al}}{\sqrt{d-1}}
+\frac{(G_{c_\al x}^{(b_\al)})^2}{\msc(z_t)}\right)(Q_t-Y_t)^{p-1}\right].
\end{align}

For the first term in \eqref{e:Gt30}, we can first average over $x'\sim x$, and noticing that \begin{align*}
    zG^{(b_\al)}_{x c_\al}=(H(t)G^{(b_\al)})_{x c_\al}
=\sum_{x'\sim x}\frac{1}{\sqrt{d-1}}G^{(b_\al)}_{x' c_\al}+\sqrt t(ZG^{(b_\al)})_{x c_\al}.
\end{align*}
This gives
\begin{align}\begin{split}\label{e:boundaryT2}
&\phantom{{}={}}\frac{1}{Z_{\cF^+}}\sum_{\bfi^+}\bE\left[\bm1(\cG\in\Omega)I(\cF^+,\cG)\frac{1}{\sqrt{d-1}}G_{c_\al x}^{(b_\al)}G^{(b_\al)}_{x' c_\al}(Q_t-Y_t)^{p-1}\right]\\
&=\sum_{x'\sim x}\frac{1}{dZ_{\cF^+}}\sum_{\bfi^+}\bE\left[\bm1(\cG\in\Omega)I(\cF^+,\cG)\frac{1}{\sqrt{d-1}}G_{c_\al x}^{(b_\al)}G^{(b_\al)}_{x' c_\al}(Q_t-Y_t)^{p-1}\right]\\
&=\frac{1}{dZ_{\cF^+}}\sum_{\bfi^+}\bE\left[\bm1(\cG\in\Omega)I(\cF^+,\cG)G^{(b_\al)}_{c_\al x }(z G^{(b_\al)}_{x c_\al}-\sqrt{t}(ZG^{(b_\al)})_{x c_\al}))(Q_t-Y_t)^{p-1}\right]\\
&=\frac{z}{d }\frac{1}{Z_{\cF^+}}\sum_{\bfi^+}\bE\left[\bm1(\cG\in\Omega)I(\cF^+,\cG)(G_{c_\al x}^{(b_\al)})^2(Q_t-Y_t)^{p-1}\right]+\OO(N^{-\fb} \bE[\Psi_p]),
\end{split}\end{align}
where in the last statement, we use \eqref{e:use_Ward} and \eqref{e:Werror}.

By plugging \eqref{e:disan}, \eqref{e:boundaryTGG}, \eqref{e:Gt30} and \eqref{e:boundaryT2} into $J_3$ in \eqref{e:defJ123}, we get
\begin{align}\begin{split}\label{e:Gt3}
J_3&=-\sum_{ \dist(x,o)=\ell}\sum_{\bfi^+}\frac{\msc^{2\ell}(z_t)}{Z_{\cF^+}}\bE\left[\bm1(\cG\in \Omega) I(\cF^+,\cG) \left(\frac{z}{d}+\frac{1}{\msc(z_t)}\right)(G_{c_\al x}^{(b_\al)})^2(Q_t-Y_t)^{p-1}\right]+\OO\left(\frac{\bE[\Psi_p]}{(d-1)^\ell} \right)\\
&= (d-2)(d-1)^{\ell-1}\bE\left[ \bm1(\cG\in \Omega) \frac{\del_z m_t(z)}{\cA N}(Q_t-Y_t)^{p-1}\right]+\OO\left(\frac{\bE[\Psi_p]}{(d-1)^\ell} \right),
\end{split}\end{align}
where in the second statement,  we use  \eqref{e:mscmd} and \eqref{e:remove_indices}.

Plugging \eqref{e:Gt1}, \eqref{e:Gt2} and \eqref{e:Gt3} into \eqref{e:decompose}, we conclude 
\begin{align*}\begin{split}
       \eqref{e:first_term0}&=J_1+J_2+J_3\\
       &=\left((d-2)(d-1)^{\ell-1}-\frac{d-2}{d-1}\right)\bE\left[ \frac{\del_z m_t(z)}{\cA N}(Q_t-Y_t)^{p-1}\right]+\OO\left(\frac{\bE[\Psi_p]}{(d-1)^\ell} \right),\\
       &=\frac{1}{\cA^2}\left(\frac{d(d-1)^{\ell}}{d-2}-\frac{d}{d-2}\right)\bE\left[ \frac{\del_z m_t(z)}{N}(Q_t-Y_t)^{p-1}\right]+\OO\left(\frac{\bE[\Psi_p]}{(d-1)^\ell} \right).
    \end{split}\end{align*}
This finishes the proof of \eqref{e:first_term}.

    For \eqref{e:refine_Gccerror}, we notice that the expectation does not depend on $(\al, \beta)$. Thus we we can first sum over $\al\in \sfA_i, \al\neq  \beta$ using \eqref{e:schur_two1} to get
    \begin{align*}\begin{split}
   \eqref{e:second_term0} &=\left(\frac{d+2}{d-2}-\frac{d(d-1)^{\ell}}{d-2}-(\ell+1)+\OO((d-1)^{-\ell})\right)\times \\
   &\times \sum_{\bfi^+}\frac{1}{Z_{\cF^+}}\bE\left[I(\cF^+,\cG)\bm1(\cG\in \Omega) (G_{c_\al c_\beta}^{(b_\al b_\beta)})^2(Q_t-Y_t)^{p-1}\right]\\
    &=\left(\frac{d+2}{d-2}-\frac{d(d-1)^{\ell}}{d-2}-(\ell+1)\right)\frac{1}{\cA^2}\bE\left[\bm1(\cG\in \Omega)\frac{\del_z m_t(z)}{N}(Q_t-Y_t)^{p-1}\right]+\OO\left( \frac{\bE[\Psi_p]}{(d-1)^{\ell}}\right),
\end{split}\end{align*}
where the last line follows from \eqref{e:remove_indices}. This gives \eqref{e:refine_Gccerror}.

For \eqref{e:GIGG2}, we can similarly first sum over $\al\neq \beta\in \sfA_i$, which is computed in \eqref{e:schur_two1}. The remaining expectations are computed in \eqref{e:GIGG1}.
\begin{align}\label{e:third_term_copy}
        \begin{split}
\eqref{e:third_term0}
&=\frac{-2(p+1)(\ell+1) +\OO((d-1)^{-\ell})}{Z_{\cF^+}}
\sum_{\bfi^+}\bE\left[ I(\cF^+,\cG)\bm1(\cG\in \Omega)
 G_{c_\al c_\beta}^{(b_\al b_\beta)}(Q_t-Y_t)^{p-2}\right.\\
 &\times \frac{1}{Nd}\sum_{u\sim v}\left((-t\del_2 Y_\ell)\left( G^{\circ}_{u c_\al}+\frac{\msc(z_t)  G^{\circ}_{u b_\al}}{\sqrt{d-1}}\right)\left( G^{\circ}_{u c_\beta}+\frac{\msc(z_t)  G^{\circ}_{u b_\beta}}{\sqrt{d-1}}\right)
    \right.\\
    &+(1-\del_1 Y_\ell)\left( G^{\circ}_{v c_\al}+\frac{\msc(z_t)  G^{\circ}_{v b_\al}}{\sqrt{d-1}}-\frac{G_{uv}}{G_{uu}}\left( G^{\circ}_{u c_\al}+\frac{\msc(z_t)  G^{\circ}_{u b_\al}}{\sqrt{d-1}}\right)\right) \\
    &\times\left.\left.\left( G^{\circ}_{v c_\beta}+\frac{\msc(z_t)  G^{\circ}_{v b_\beta}}{\sqrt{d-1}}-\frac{G_{uv}}{G_{uu}}\left( G^{\circ}_{u c_\beta}+\frac{\msc(z_t)  G^{\circ}_{u b_\beta}}{\sqrt{d-1}}\right)\right)\right)\right]\\ 
    &=-   \frac{2(p-1)(\ell+1)}{\cA^2}\bE\left[\left(\frac{1-\del_1Y_\ell}{\cA}-t\del_2 Y_\ell\right)\bm1(\cG\in \Omega)\frac{\del^2_z m_t(z)}{2N^2} (Q_t-Y_t)^{p-2}\right]\\
    &+\OO\left((d-1)^{-\ell} \bE\left[\bm1(\cG\in \Omega)\frac{\Upsilon}{N\Im[z]}\Phi|Q_t-Y_t|^{p-2}\right]\right)=\OO\left( \frac{\bE[\Psi_p]}{(d-1)^{\ell}}\right).
\end{split}\end{align}
This gives \eqref{e:GIGG2}.

For \eqref{e:GIGG3}, we can first sum over $\al\neq \beta\in \sfA_i$, which is computed in \eqref{e:schur_two1}.
\begin{align*}
 \eqref{e:track_error2}&=\left(-\frac{2}{d-2}+\OO((d-1)^{-\ell})\right)\sum_{\bfi^+} \bE[\bm1(\cG\in \Omega)I(\cF^+, \cG)(G_{c_\al c_\beta}^{(b_\al b_\beta)})^2 (Q_t-Y_t)^{p-1}]\\
          &=-\frac{2}{d-2}\frac{1}{\cA^2}\bE\left[\bm1(\cG\in \Omega)\frac{\del_z m_t(z)}{N}(Q_t-Y_t)^{p-1}\right]+\OO\left( \frac{\bE[\Psi_p]}{(d-1)^{\ell}}\right),
\end{align*}
where the last line follows from \eqref{e:remove_indices}. This gives \eqref{e:GIGG3}.

\end{proof}

\appendix

\section{Properties of the Green's Functions}

Throughout this paper, we repeatedly use some (well-known) identities for Green's functions,
which we collect in this appendix.

\subsection{Resolvent identity}

The following well-known identity is referred as resolvent identity:
for two invertible matrices $A$ and $B$ of the same size, we have
\begin{equation} \label{e:resolv}
  A^{-1} - B^{-1} = A^{-1}(B-A)B^{-1}=B^{-1}(B-A)A^{-1}.
\end{equation}

\subsection{Schur complement formula}

Given an $N\times N$ matrix $H$ and an index set $\bT \subset \qq{N}$, recall that we denote by
$H|_\bT$ the $\bT \times \bT$-matrix obtained by restricting $H$ to $\bT$,
and that by $H^{(\bT)} = H|_{\bT^\complement}$ the matrix obtained by removing
the rows and columns corresponding to indices in $\bT$.
Thus, for any $\bT \subset \qq{N}$,
any symmetric matrix $H$ can be written (up to rearrangement of indices) in the block form
\begin{equation*}
  H = \begin{bmatrix} A& B^\top\\ B &D  \end{bmatrix},
\end{equation*}
with $A=H|_{\bT}$ and $D=H^{(\bT)}$.
The Schur complement formula asserts that, for any $z\in \bC^+$,
\begin{equation} \label{e:Schur}
 G=(H-z)^{-1}= \begin{bmatrix}
   (A-B^\top G^{(\bT)} B)^{-1} & -(A-B^\top G^{(\bT)} B)^{-1}B^\top G^{(\bT)} \\
   -G^{(\bT)} B(A-B^\top G^{(\bT)} B)^{-1} & G^{(\bT)}+G^{(\bT)} B(A-B^\top G^{(\bT)} B)^{-1}B^\top G^{(\bT)} 
 \end{bmatrix},
\end{equation}
where $G^{(\bT)}=(D-z)^{-1}$.
Throughout the paper, we often use the following special cases of \eqref{e:Schur}:
\begin{align} \begin{split}\label{e:Schur1}
  G|_{\bT} &= (A-B^\top G^{(\bT)} B)^{-1},\\
  G|_{\bT\bT^\complement}&=-G|_{\bT}B^\top G^{(\bT)},\\
   G|_{\bT^\complement}&=G^{(\bT)}+G|_{\bT^\complement\bT}(G|_{\bT})^{-1}G|_{\bT\bT^\complement}=G^{(\bT)}-G^{(\bT)}BG|_{\bT\bT^\complement},
  \end{split}
\end{align}
as well as the special case
\begin{equation} \label{e:Schurixj}
G_{ij}^{(k)} = G_{ij}-\frac{G_{ik}G_{kj}}{G_{kk}}
=G_{ij}+(G^{(k)}H)_{ik} G_{kj}.
\end{equation}

\subsection{Woodbury formula}
Let $A+UCV^\top$ be a rank $r$ perturbation of $A$. Namely, $U, V\in \bR^{N\times r}$ and $C\in \bR^{r\times r}$. Then, the Woodbury formula  gives us 
\begin{align}\label{e:woodbury}
(A+UCV^\top)^{-1}-A^{-1}=-A^{-1}U(C^{-1}+V^\top A^{-1} U)^{-1}V^\top A^{-1}.
\end{align}

\subsection{Ward identity}

For any symmetric $N\times N$ matrix $H$, its Green's function $G(z)=(H-z)^{-1}$ satisfies
the Ward identity
\begin{equation} \label{e:Ward}
  \sum_{j=1}^{N} |G_{ij}(z)|^2= \frac{\Im G_{jj}(z)}{\eta},
\end{equation}
where $\eta=\Im [z]$. This can be deduced from using \eqref{e:resolv} on $G-G^*$. This identity provides a bound for the sum $\sum_{j=1}^{N} |G_{ij}(z)|^2$
in terms of the diagonal entries of the Green's function.

\section{Proofs of Results from \Cref{s:preliminary}}
\label{app:Green}

\begin{proof}[Proof of \Cref{l:relation_zt_z}]
    We recall from \eqref{e:edgeeqn2} and \eqref{e:mdzt}, that $E _t=\xi _t-t\md(\xi _t)$ and $z=z_t-t\md(z,t)=z_t-t\md(z_t)$. 
    Recall from \eqref{e:md_equation} that in a small neighborhood of $2$, $w-tm_d(w)$ is an analytic function of $\sqrt{w-2}$. Thus, we can introduce an analytic  function 
    \begin{align*}
        F(\sqrt{w-2}):=w-t\md(w)
        =2-t\cA\sqrt{w-2}+(1+\OO(t))(\sqrt{w-2})^2+\OO(t)(\sqrt{w-2})^3\cdots.
    \end{align*}
    The defining relation of $\xi_t$ (as in \eqref{e:edgeeqn1}) states that $\sqrt{\xi_t-2}$ is a critical point of $F$, thus
    \begin{align}\begin{split}\label{e:z-E_texp}
        z-E_t&=(z_t-t\md(z_t))-(\xi _t-t\md(\xi _t))
        =F(\sqrt{z_t-2})-F(\sqrt{\xi_t-2})\\
        &=\frac{F''(\sqrt{\xi_t-2})}{2}\left(\sqrt{z_t-2}-\sqrt{\xi _t-2}\right)^2
        +\frac{F'''(\sqrt{\xi_t-2})}{3!}\left(\sqrt{z_t-2}-\sqrt{\xi _t-2}\right)^3+\cdots,\\
        &=:\frac{F''(\sqrt{\xi_t-2})}{2}\left((\sqrt{z_t-2}-\sqrt{\xi _t-2})G(\sqrt{z_t-2}-\sqrt{\xi _t-2})\right)^2,
    \end{split}\end{align}
    where $F''(\sqrt{\xi_t-2})/2=1+\OO(t)$, $F^{(k)}(\sqrt{\xi_t-2})=\OO(t)$ for $k\geq 3$,  and $G(w)$ is an analytic function with $G(w)=1+\OO(t|w|)$ when $|w|\ll 1$.

    By taking the square root of both sides of \eqref{e:z-E_texp}, we get
    \begin{align}\begin{split}\label{e:zjianEt0}
       \sqrt{ z-E_t}
       &= (1+\OO(t))(\sqrt{z_t-2}-\sqrt{\xi _t-2})G(\sqrt{z_t-2}-\sqrt{\xi _t-2})\\
       &=(\sqrt{z_t-2}-\sqrt{\xi _t-2}) + \OO(t\sqrt{|z_t-2|}+t^2),
    \end{split}\end{align}
   where we used that $\xi _t=2+\cA^2 t^2/4+\OO(t^3)$ from \eqref{e:xi_behavior}.
The above relation \eqref{e:zjianEt0} gives $\sqrt{|z_t-2|}\lesssim t+\sqrt{|z-E _t|}$, and it follows that
\begin{align*}
    \sqrt{z_t-2}
    =\sqrt{\xi_t-2}+\sqrt{z-E _t}+\OO\left(t\sqrt{|z-E _t|}+t^2\right).
\end{align*}
This finishes the proof of \eqref{e:relation_zt_z}, by noticing $\xi_t=2+\cA^2t^2/4+\OO(t^3)$ from \eqref{e:xi_behavior}.

We notice that $G(w)$ is a real analytic function, namely $G(x)\in \bR$ when $x\in \bR$. Then  $G(\Re[w])$ is real, and we can Taylor expand $G(w)$ around $\Re[w]$ to get $\Im[G(w)]=\Im[G(\Re[w])]+\OO(|w-\Re[w]|)=\OO(\Im[w])$. It follows that 
\begin{align}\label{e:imwG}
\Im[wG(w)]=\Im[w]\Re[G(w)]+\Re[w]\Im[G(w)]\lesssim \Im[w].
\end{align}
For \eqref{e:square_root_behavior}, we the take imaginary part on both sides of the first statement of \eqref{e:zjianEt0}, giving
\begin{align}\begin{split}\label{e:Imz-E}
\Im[\sqrt{ z-E_t}]
       &= (1+\OO(t))\Im[(\sqrt{z_t-2}-\sqrt{\xi _t-2})G(\sqrt{z_t-2}-\sqrt{\xi _t-2})]\\
       &\asymp \Im[(\sqrt{z_t-2}-\sqrt{\xi _t-2}]=\Im[\sqrt{z_t-2}],
\end{split}\end{align}
where the second statement follows from taking $w=\sqrt{z_t-2}-\sqrt{\xi_t-2}$ in \eqref{e:imwG}, and the last statement follows from noticing that $\xi_t-2>0$ is a real number.

Finally the square root behavior \eqref{eq:square_root_behave} of $\msc(z_t)$ and $m_d(z_t)$ follows from \eqref{e:Imz-E} and \eqref{e:medge_behavior}.

\end{proof}

\begin{proof}[Proof of Proposition \ref{p:recurbound}]
The proofs for $X_\ell$ and $Y_\ell$ are identical, so we will only provide the proof for $Y_\ell$.
We denote $\cH=\cB_\ell(o, \cY)$, which is the truncated $(d-1)$-ary tree at level $\ell$. We denote its vertex set as $\bH$ and normalized adjacency matrix as $H$. We denote $\bI$ and $\bI^\del$ the diagonal matrices, such that for $x,y\in \bH$,  $\bI_{xy}=\delta_{xy}$ and $\bI^\del_{xy}=\bm1(\dist_\cH(x,o)=\ell)\delta_{xy}$. Then \eqref{e:defP} gives that
\begin{align*}
Y_\ell( \Delta, w)=P(\cH,w,\Delta)
=\left(H-w-\Delta \mathbb I^\del\right)^{-1},\quad P=P(\cH,z,\msc(z))=(H-z-\msc(z)\bI)^{-1}.
\end{align*}
In the rest of the proof, we will simply write $\msc=\msc(z), \md=\md(z)$. We can compute the Green's function $P(\cH, w, \Delta)$ by a perturbation argument,
\begin{align}\begin{split}\label{e:expansionGP}
P(\cH, w, \Delta)
&=\left(H-w-\Delta \mathbb I^\del\right)^{-1}=\left(H-z-\msc\mathbb I^\del-((w-z)+(\Delta-\msc)\mathbb I^\del)\right)^{-1}\\
&=P+P\sum_{k\geq 1}(((w-z)+(\Delta-\msc)\mathbb I^\del)P)^k.
\end{split}\end{align}
With the explicit expression of $P$ as given in \eqref{e:Gtreemsc2}, we can compute
\begin{align}\label{e:PBPoo}
\left(P((w-z)+(\Delta-\msc)\mathbb I^\del)P\right)_{oo}
&=\msc^{2\ell+2}(\Delta-\msc)+(\msc^2+\msc^4+\cdots+\msc^{2\ell+2})(w-z).
\end{align}
Moreover, for $k\geq 2$ we will show the following two relations for $P$,
\begin{align}
   \label{e:Pboundary} &\left(P\mathbb I^\del P\mathbb I^\del P\right)_{oo}=\msc^{2\ell+2}\md\left(\frac{1-\msc^{2\ell+2}}{d-1}+\frac{d-2}{d-1}\frac{1-\msc^{2\ell+2}}{1-\msc^2}\right) ,\\
    \label{e:Ptotalsum}&(|P|^k)_{oo}\lesssim (C\ell)^{2k-3},
\end{align}
where for each $i,j\in \bH$, $|P|_{ij}\deq |P_{ij}|$.

The relation \eqref{e:Pboundary} follows from explicit computation using \eqref{e:Gtreemsc} and \eqref{e:Gtreemsc2}
\begin{align*}
&\phantom{{}={}}\sum_{l, l'}P_{ol}P_{ll'}P_{l'o}=\sum_{l,l'}\msc^2 \left(-\frac{\msc}{\sqrt{d-1}}\right)^{2\ell}\md\left(1-\left(-\frac{\msc}{\sqrt{d-1}}\right)^{2+2{\rm anc}(l, l')}\right)\left(-\frac{\msc}{\sqrt{d-1}}\right)^{\dist_{\cH}(l,l')}\\
&= \frac{\msc^{2\ell+2}}{(d-1)^{2\ell}}\sum_{l,l'} \md\left(1-\left(\frac{\msc}{\sqrt{d-1}}\right)^{2+2{\rm anc}(l, l')}\right)\left(-\frac{\msc}{\sqrt{d-1}}\right)^{\dist_{\cH}(l,l')}\\
&=\msc^{2\ell+2}\md\left(1-\left(\frac{\msc}{\sqrt{d-1}}\right)^{2+2\ell}
+\sum_{r=1}^\ell \left(1-\left(\frac{\msc}{\sqrt{d-1}}\right)^{2+2(\ell-r)}\right)\left(\frac{\msc}{\sqrt{d-1}}\right)^{2r}(d-2)(d-1)^{r-1}
\right)\\
&=\msc^{2\ell+2}\md\left(\frac{1-\msc^{2\ell+2}}{d-1}+\frac{d-2}{d-1}\frac{1-\msc^{2\ell+2}}{1-\msc^2}\right),
\end{align*}
where the summation in the first line is over $l,l'$ such that $\dist_{\cH}(l,o)=\dist_{\cH}(l',o)=\ell$; in the second to last line we used that for a given $l$, there are $(d-2)(d-1)^{r-1}$ values of $l'$ such that $\dist_{\cH}(l, l')=2r, {\rm anc}(l, l')=\ell-r$ for $1\leq r\leq \ell$. When $l=l'$, it holds $\dist_{\cH}(l, l')=0, {\rm anc}(l, l')=\ell$.

To prove \eqref{e:Ptotalsum}, we show by induction that for any vertex $i$ such that $\dist_{\cH}(o,i)\leq \ell$,
\begin{align}\label{e:inductionPk}
    (|P|^k)_{o i}\lesssim \frac{(C\ell)^{2k-2}}{(d-1)^{\dist_{\cH}(o,i)/2}}.
\end{align}
The statement for $k=1$ follows from \eqref{e:Gtreemsc}. Assume the statement \eqref{e:inductionPk} holds for $k-1$, we prove it for $k$,
\begin{align}\label{e:inductionPk2}
   (|P|^k)_{o i}= \sum_{j\in \bH} (|P|^{k-1})_{o j}|P|_{ji}
   \lesssim\sum_{j\in \bH} \frac{(C\ell)^{2k-4}}{(d-1)^{\dist_{\cH}(o,j)/2}}\frac{1}{(d-1)^{\dist_{\cH}(j,i)/2}}.
\end{align}
We denote the path from $o$ to $i$ as $\cP$, then $|\{j\in \bH: \dist_{\cH}(j,\cP)=r\}|\leq \ell(d-1)^{r}$, and for $\dist_{\cH}(j,\cP)=r$, we have
\begin{align*}
    \frac{1}{(d-1)^{\dist_{\cH}(o,j)/2}}\frac{1}{(d-1)^{\dist_{\cH}(j,i)/2}}\lesssim \frac{1}{(d-1)^r}\frac{1}{(d-1)^{\dist_{\cH}(o,i)/2}}.
\end{align*}
In this way, we can reorganize the sum over $j$ in \eqref{e:inductionPk2} according to its distance to $\cP$,
\begin{align*}
     (|P|^k)_{o i}&= \sum_{r=0}^\ell \sum_{j\in \bH:\dist_{\cH}(j,\cP)=r }\frac{(C\ell)^{2k-4}}{(d-1)^{\dist_{\cH}(o,j)/2}}\frac{1}{(d-1)^{\dist_{\cH}(j,i)/2}}\\
     &\lesssim \sum_{r=0}^\ell|\{j: \dist_{\cH}(j,\cP)=r\}|\frac{(C\ell)^{2k-4}}{(d-1)^r}\frac{1}{(d-1)^{\dist_{\cH}(o,i)/2}}\lesssim \frac{(C\ell)^{2k-2}}{(d-1)^{\dist_{\cH}(o,i)/2}},
\end{align*}
which shows \eqref{e:inductionPk}. The claim \eqref{e:Ptotalsum} is a consequence of \eqref{e:inductionPk}
\begin{align*}
    (|P|^k)_{oo}
    =\sum_{i\in \bH} (|P|^{k-1})_{oi}|P|_{io}
    \lesssim \sum_{i\in \bH} \frac{(C\ell)^{2k-4}}{(d-1)^{\dist_{\cH}(o,i)/2}}\frac{1}{(d-1)^{\dist_{\cH}(i,o)/2}}
    \lesssim (C\ell)^{2k-3}.
\end{align*}

The claim \eqref{e:Yl_derivative} follows from reading the coefficients in front of $(w-z)$ and $\Delta-\msc$ in \eqref{e:expansionGP}, and using the upper bounds \eqref{e:PBPoo}, \eqref{e:Pboundary} and \eqref{e:Ptotalsum}.

By \eqref{e:Pboundary} and \eqref{e:Ptotalsum}, we have
\begin{align}\begin{split}\label{e:PBPoo2}
&\phantom{{}={}}P(((w-z)+(\Delta-\msc)\mathbb I^\del)P)^2
=\left(P(\mathbb I^\del P)^2\right)_{oo}(\Delta-\msc)^2
+\OO(\ell^3(|w-z|^2 +|\Delta-\msc||w-z|))
\\
&=\msc^{2\ell+2}\md\left(\frac{1-\msc^{2\ell+2}}{d-1}+\frac{d-2}{d-1}\frac{1-\msc^{2\ell+2}}{1-\msc^2}\right)(\Delta-\msc)^2+\OO(\ell^3(|w-z|^2 +|\Delta-\msc||w-z|)).
\end{split}\end{align}
For $k\geq 3$
\begin{align}\label{e:PBPoo3}
\left(P(((w-z)+(\Delta-\msc)\mathbb I^\del)P)^k\right)_{oo}=\OO(\ell^{2k-1}(|w-z|^k +|\Delta-\msc|^k)).
\end{align}
The claim \eqref{e:recurbound} follows from plugging \eqref{e:PBPoo}, \eqref{e:PBPoo2} and \eqref{e:PBPoo3} into \eqref{e:expansionGP}, and use $\ell^2|\Delta-\msc|\ll1$.
\end{proof}

\begin{proof}[Proof of \Cref{thm:prevthm}]
For $t=0$, we notice that the first statement in \eqref{eq:infbound0} implies that for $\Im[z]\geq (\log N)^{300}/N$, $\Im[G_{ii}(z,0)]\lesssim 1$. It follows that eigenvectors are delocalized \begin{align*}
    \|\bmu_\al(0)\|^2_\infty\lesssim \max_{1\leq i\leq N}((\log N)^{300}/N)\Im[G_{ii}(\lambda_\alpha+\ri (\log N)^{300}/N,0)]\lesssim (\log N)^{300}/N\ll N^{\fo/2-1}.
\end{align*} 
Moreover, for $\Im[z]\geq N^{-1+\fg}$, thanks to \eqref{e:defepsilon}, \begin{align}\label{e:epsilon_bound}
    \varepsilon(z)\leq \sqrt{\varepsilon'(z)}\leq N^{-4\fb}.
\end{align}  
Thus \Cref{thm:prevthm} for $t=0$ follows from \Cref{thm:prevthm0}.

Next we prove \Cref{thm:prevthm} for $0< t\leq N^{-1/3+\ft}$. It was proven in \cite[Lemma 4.2]{bauerschmidt2017bulk} that if all
eigenvectors of $H(0)$ are uniformly delocalized, then with
overwhelmingly high probability over $Z$ they remain delocalized under the constrained Dyson Brownian motion \eqref{e:Ht}. The claim \eqref{e:eig_delocalization} follows from \cite[Lemma 4.2]{bauerschmidt2017bulk} (by taking $(B,\xi)=((\log N)^{200}, N^{\fo/4})$).

Recall that $m_t(z)=\Tr[(H(t)-z)^{-1}]/N=\Tr[(H+\sqrt t Z-z)^{-1}]/N$. We next we show that for $|z|\leq1/\fg, \Im[z]\geq N^{-1+\fg}$, with overwhelmingly high probability over $Z$
\begin{align}\label{e:mtdiff}
|m_t(z)-m_d(z,t)|\lesssim N^{-2\fb}.
\end{align} 
We denote the free convolution of the Stieltjes transform of the empirical eigenvalue distribution of $H$ with the semicircle distribution of variance $t$ as $\widehat m_t(z)$ (recall from \Cref{s:fc}). It follows from \cite[Theorem 3.1, Corollary 3.2]{huang2019rigidity} (with $M=(\log N)^4$) that for $|z|\leq 1/\fg,\Im[z]\geq N^{-1+\fg}$,
\begin{align}\label{e:rigidity}
|m_t(z)-\widehat m_t(z)|\leq \frac{(\log N)^4}{N\Im[z]}\leq \frac{1}{N^{4\fb} },
\end{align}
with overwhelmingly high probability over $Z$. Moreover, it follows from \eqref{e:Gfreeconvolution2} that
\begin{align*}
\widehat m_t(z)=m_0(z+t\widehat m_t(z)), \quad  m_d(z,t)=m_d(z+t m_d(z,t)).
\end{align*}
By taking the difference of the above two expressions, we get
\begin{align}\begin{split}\label{e:mtdiff2}
|\widehat m_t(z)-m_d(z,t)|
&\leq |m_0(z+t\widehat m_t(z))-m_d(z+t \widehat m_t(z))|+|m_d(z+t\widehat m_t(z))-m_d(z+t m_d(z,t))|\\
&\lesssim \frac{1}{N^{4\fb} }+\sqrt{t|\widehat m_t(z)-m_d(z,t)|},
\end{split}\end{align}
where in the second inequality we used the last statement in \eqref{eq:infbound0} (noticing that $m_0=m_N$), $|\varepsilon(z)|\lesssim N^{-4\fb}$ from \eqref{e:epsilon_bound}, and $m_d(z)$ is H\"{o}lder $1/2$. The claim \eqref{e:mtdiff} follows from rearranging \eqref{e:mtdiff2}, and noticing $t\leq N^{-1/3+\ft}\ll N^{-4\fb}$.

Next we prove the first statement in \eqref{eq:infbound}.
The entrywise estimates on Green's function of $H(t)$ with $t>0$ follow from an argument similar to the proof of \cite[Theorem 2.1]{bourgade2017eigenvector} (with $\psi=N^{\fo}$): For any 
\begin{align}\label{e:entrywise_law}
|G_{ij}(z,t)-G_{ij}(z+t\widehat m_t(z),0)|\leq \frac{N^{2\fo}}{\sqrt{N\Im[z]}}\leq \frac{1}{N^{4\fb} },\quad \Im[z]\geq N^{-1+\fg},
\end{align}
with overwhelmingly high probability over $Z$. 

We remark that in the statement of \cite[Theorem 2.1]{bourgade2017eigenvector}, it assumed that $\Im[m_0]$ is bounded from below and that $t\gg \eta_*$.
However, in the proof of \cite[Theorem 2.1]{bourgade2017eigenvector}, these assumptions are only used to show that $|m_t(z)-\widehat m_t(z)|$ is small.
With the required estimate of $|m_t(z)-\widehat m_t(z)|$ already established by \eqref{e:rigidity},
the remaining part of \cite[Theorem 2.1]{bourgade2017eigenvector} does not use that $\Im[m_0]$ is bounded from below or that $t\gg \eta_*$.
Therefore, with \eqref{e:rigidity} given,
the remaining proof of \cite[Theorem 2.1]{bourgade2017eigenvector} applies and gives the above result \eqref{e:entrywise_law} on the entrywise estimates of Green's function of $H(t)$.

Take $z_t= z+t \md(z,t)$ and $\widehat z_t= z+t\widehat m_t(z)$. Then  \eqref{e:rigidity} implies that $|z_t-\widehat z_t|\lesssim tN^{-4\fb} $, and 
\begin{align*}
&\phantom{{}={}}|G_{ij}(z,t)-P_{ij}(\cB_{\fR/100}(\{i,j\},\cG),z_t,\msc(z_t))|
\leq |G_{ij}(z,t)-G_{ij}(\widehat z_t,0)|\\
&+|P_{ij}(\cB_{\fR/100}(\{i,j\},\cG),z_t,\msc(z_t))-P_{ij}(\cB_{\fR/100}(\{i,j\},\cG),\widehat z_t,\msc(\widehat z_t))|\\
&+|G_{ij}(\widehat z_t,0)-P_{ij}(\cB_{\fR/100}(\{i,j\},\cG),\widehat z_t,\msc(\widehat z_t))|\lesssim \frac{1}{N^{4\fb} } +\sqrt{|z_t-\widehat z_t|}+\frac{1}{N^{4\fb} }\lesssim\frac{1}{N^{2\fb} },
\end{align*}
where in the last line we used \eqref{e:entrywise_law} for the first term,   the fact that  $P_{ij}(\cB_{\fR/100}(\{i,j\},\cG),z,\msc(z))$ is H\"{o}lder $1/2$ in $z$ for the second term, and the first statement in \eqref{eq:infbound0} for the third term.

Finally for the second statement in \eqref{eq:infbound}, using the first statement in \eqref{eq:infbound} we get
\begin{align}\begin{split}\label{e:Qtbound}
&\phantom{{}={}}Q_t(z)
=\frac{1}{Nd}\sum_{i\sim j}G^{(i)}_{jj}(z,t)
=\frac{1}{Nd}\sum_{i\sim j}\left(G_{jj}(z,t)-\frac{G^2_{ij}(z,t)}{G_{ii}(z,t)}\right)\\
&=\frac{1}{Nd}\sum_{i\sim j}\left(P_{jj}(\cB_{\fR/100}(j,\cG), z_t, \msc(z_t))-\frac{P^2_{ij}(\cB_{\fR/100}(\{i,j\},\cG), z_t, \msc(z_t))}{P_{ii}(\cB_{\fR/100}(i,\cG), z_t, \msc(z_t))}\right)+\OO(N^{-2\fb} ).
\end{split}\end{align}
If the edge $\{i,j\}$ has radius $\fR$ tree neighborhood, the expression above simplifies to \begin{align}\begin{split}\label{e:PP}
    &\phantom{{}={}}P_{jj}(\cB_{\fR/100}(j,\cG), z_t, \msc(z_t))-\frac{P^2_{ij}(\cB_{\fR/100}(\{i,j\},\cG), z_t, \msc(z_t))}{P_{ii}(\cB_{\fR/100}(i,\cG), w, \msc(w))}\\
    &=\md(z_t)-\frac{1}{\md(z_t)}\left(-\md(z_t)\frac{\msc(z_t)}{\sqrt{d-1}}\right)^2=\msc(z_t),
\end{split}\end{align}
where we used \eqref{e:Gtreemkm} for the first equality, and the relation \eqref{e:md_equation} between $\md$ and $\msc$ for the second equality. If the edge $\{i,j\}$ does not have a radius $\fR$ tree neighborhood, we can simply bound \eqref{e:PP} by $\OO(1)$. Since $\cG\in \overline{\Omega}$ (recall from \eqref{def:omegabar}), there are at most $\OO(N^\fc)$ such edges. Thus by plugging \eqref{e:PP} into \eqref{e:Qtbound}, we conclude that 
$\eqref{e:Qtbound}=\msc(z_t)+\OO\left(N^{-2\fb} \right)$.
\end{proof}

    \begin{proof}[Proof of \Cref{l:basicG}]

    The Schur complement formula \eqref{e:Schur1} establishes a relationship between 
$G(z,t)$ and $G^{(\bX)}(z,t)$. By analyzing this relationship and using the estimates in \eqref{eq:infbound}, one can show that $G^{(\bX)}(z,t)$ also satisfies the tree approximation in \eqref{eq:local_law}, albeit with a slightly larger error. The detailed proof is provided in \cite[Propositions 5.1 and 5.18]{huang2024spectrum}, where one can set 
$\varepsilon(z)=N^{-2\fb} $.

We will prove the statements in \eqref{e:Gest}, \eqref{e:Trace-change},\eqref{e:Ghao0} and \eqref{e:use_Ward},
only for $\bX=\bT$. The proofs for other cases follow similarly and are therefore omitted. We will simply write $m_t(z), \wt m_t(z),  G(z,t), \wt G(z,t)$ as $m_t, \wt m_t, G, \wt G$.

    For the first statement in \eqref{e:Gest}, we start with $\bX=\emptyset$, then  \eqref{e:eig_delocalization} implies
    \begin{align}\label{e:smallGij}
        |\Im[G_{xy}]|=\left|\sum_{\al=1}^N\frac{\eta (\bmu_\al(t) \bmu^\top_\al(t))_{xy}}{|\la_\al-z|^2}\right|
        \leq N^{\fo/2-1}\sum_{\al=1}^N\frac{\eta }{|\la_\al-z|^2}=N^{\fo/2} \Im[m_t].
    \end{align}
    From the Schur complement formula \eqref{e:Schur1}, we have
    \begin{align}\label{e:GT}
       G^{(\bT)}=G-G(G|_{\bT})^{-1}G.
    \end{align}
    By taking imaginary part on both sides of \eqref{e:GT}, we get
    \begin{align}\begin{split}\label{e:Gimaginary}
        \Im[G_{xy}^{(\bT)}]
        &=\Im[G_{xy}]-\Im[(G (G|_{\bT})^{-1}G)_{xy}]\lesssim \Im[G_{xy}]+\sum_{j\in \bT}\Im[G_{xj}] |(G|_{\bT})^{-1}G)_{jy}|\\
        &+\sum_{j\in \bT}|(G(G|_{\bT})^{-1})_{xj}||\Im[G_{jy}] |
        +\sum_{j,k\in \bT}|G_{xj}|\Im[(G|_{\bT})^{-1}_{jk}]|G_{ky}|.
    \end{split}\end{align}
In the following we show the following estimates
\begin{align}\label{e:sumG}
    \sum_{j\in \bT}(|G_{xj}|+|(G(G|_{\bT})^{-1})_{xj}|+|G_{jy}|+|(G(G|_{\bT})^{-1})_{jy}|)
    \lesssim \ell, \quad \max_{j,k\in \bT}\Im[(G|_{\bT})^{-1}_{jk}]\lesssim \ell^2 N^{\fo/2}\Im[m_t].
\end{align}
Then the first statement in \eqref{e:Gest} follows from plugging \eqref{e:smallGij} and \eqref{e:sumG} into \eqref{e:Gimaginary}.

We start with the first statement in \eqref{e:sumG}.
   We recall $\cT=\cB_\ell(o,\cG)$ and denote $P=P(\cT, z_t, \msc(z_t))$. Then \eqref{e:defP} gives $P^{-1}=H-z_t-\msc(z_t)\bI^\del$, where the diagonal matrix $I^\del_{ij}=\bm1(\dist_{\cT}(o,i)=\ell)\delta_{ij}$  for $i,j\in \bT$. Moreover, \eqref{eq:infbound} implies $i,j\in \bT$, $|G_{ij}-P_{ij}|\leq N^{-2\fb} $.
    Then we can perform a resolvent expansion according to \eqref{e:resolv} and rewrite $(G|_{\bT})^{-1}_{ij}$ as 
\begin{align*}\begin{split}
(G|_{\bT})^{-1}_{ij}-P^{-1}_{ij}=\left(
\sum_{k=1}^\fp P^{-1}\left((P-G|_{\bT})P^{-1}\right)^k
\right)_{ij}+\OO\left(\frac{1}{N}\right),
\end{split}\end{align*}
for some large enough $\fp\geq 1$. 
From the above expression, we get
\begin{align}\label{e:sum_Ginverse}
\sum_{j\in \bT}|(G|_\bT)^{-1}_{ij}| \lesssim \sum_{j\in \bT}|P^{-1}_{ij}| +\OO(N^{-\fb} )\lesssim 1.   
\end{align}
Since the vertex $o$ has radius $\fR/2$ tree neighborhood,  for any $x\not\in \bT$ and $0\leq r\leq 2\ell$, $|\{j\in \bT: \dist_\cG(x,j)=r\}|\lesssim (d-1)^{r/2}$. As a consequence, \eqref{eq:infbound} and \eqref{e:Gtreemkm} together imply that 
\begin{align}\label{e:sumGix}
    \sum_{j\in \bT} |G_{xj}|\lesssim 
    \sum_{0\leq r\leq 2\ell}\frac{|\{j\in \bT: \dist(x,j)=r\}|}{(d-1)^{r/2}} +\OO(1)\lesssim \ell.
\end{align}
The first statement in \eqref{e:sumG} follows from the estimates \eqref{e:sum_Ginverse} and \eqref{e:sumGix}
\begin{align*}
     \sum_{j\in \bT}(|G_{xj}|+|(G(G|_{\bT})^{-1})_{xj}|)
     \lesssim \sum_{j\in \bT}|G_{xj}|\left(1+\sum_{i\in \bT}|(G|_{\bT})^{-1}_{ji}|\right)\lesssim \ell.
\end{align*}

    Next we prove the second statement in \eqref{e:sumG}. For the imaginary part of a symmetric matrix, we have the following identity
    \begin{align*}
        \Im[A^{-1}]= -A^{-1}\Im[A] \overline{A^{-1}}.
    \end{align*}
    Thus 
    \begin{align}\label{e:Im_Ginverse}
        \Im[(G|_\bT)^{-1}_{xy}]
        =-((G|_\bT)^{-1}\Im[G|_\bT] \overline{(G|_\bT)^{-1}})_{xy}.
    \end{align}
The second statement in \eqref{e:sumG} follows from plugging \eqref{e:smallGij} and \eqref{e:sum_Ginverse} into \eqref{e:Im_Ginverse}.

The second statement in \eqref{e:Gest} follows from a Ward identity \eqref{e:Ward}, and the first statement in \eqref{e:Gest}:
\begin{align*}
    \frac{1}{N}\sum_{x\notin \bX}|G_{xy}^{(\bT)}|^2=\frac{\Im[G_{yy}^{(\bT)}]}{N\eta}\lesssim \frac{N^\fo\Im[m_t]}{N\eta}. 
\end{align*}

To prove \eqref{e:Trace-change}, we take the trace on both sides of \eqref{e:GT}
\begin{align*}
   &\phantom{{}={}}\left|\frac{1}{N}\Tr[G^{(\bT)}]-m_t\right|
   \leq \frac{1}{N}\sum_{x}\sum_{j,k\in \bT}|G_{xj}||(G|_{\bT})_{jk}^{-1}||G_{ky}|
   \\
   &\lesssim \frac{1}{N}\sum_{j,k\in \bT}|(G|_{\bT})_{jk}^{-1}|\sqrt{\sum_{x}|G_{xj}|^2\sum_{x}|G_{xk}|^2}
   \lesssim \frac{N^\fo\Im[m_t]}{N\eta}\sum_{j,k\in \bT}|(G|_{\bT})_{jk}^{-1}|
   \lesssim \frac{N^\fo(d-1)^\ell \Im[m_t]}{N\eta},
\end{align*}
where the second inequality follows from Cauchy-Schwarz inequality; the third inequality follows from \eqref{e:Gest}; the last inequality follows from \eqref{e:sum_Ginverse}.

The first relation in \eqref{e:Ghao0} follows from the fact that $(1,1,\cdots,1)$ is an eigenvector of $H(t)$ with eigenvalue $d/\sqrt{d-1}$. For the second relation, we can use \eqref{e:GT} and \eqref{e:sumG},
\begin{align*}
     \left|\sum_{x\not\in \bT} G_{xy}^{(\bT)}\right|=\left|\sum_{x\in \qq{N}}G_{xy}-\sum_{x\in \qq{N}}\sum_{j\in \bT}G_{xj}((G|_{\bT})^{-1}G)_{jy}\right|
     \lesssim 1+ \sum_{j\in \bT}|((G|_{\bT})^{-1}G)_{jy}|\lesssim \ell.
\end{align*}

For \eqref{e:use_Ward}, thanks to the Schur complement formula \eqref{e:Schurixj},
\begin{align}\label{e:Gbound}
    |G_{u y}^{(v \bX)}|=\left|G_{u y}^{( \bX)}-\frac{G_{u v}^{(\bX)}G_{v y}^{(\bX)}}{G_{vv}^{(\bX)}}\right|\lesssim |G_{u y}^{( \bX)}|+|G_{v y}^{(\bX)}|,
\end{align}
where in the second statement we used \eqref{eq:local_law} to bound $|G_{u v}^{(\bX)}|\lesssim 1, |G_{vv}^{(\bX)}|\gtrsim 1$. By plugging \eqref{e:Gbound} into \eqref{e:use_Ward} we get
\begin{align*}
    \frac{1}{Nd}\sum_{v\sim u\not\in \bX}|G_{u y}^{(v \bX)}|^2\lesssim  \frac{1}{Nd}\sum_{v\sim u\not\in \bX}|G_{u y}^{( \bX)}|^2+|G_{v y}^{(\bX)}|^2\lesssim \frac{ N^\fo \Im[m_t ]}{N\eta},
\end{align*}
where the second statment follows from \eqref{e:Gest}.

\end{proof}

\begin{proof}[Proof of \Cref{c:expectationbound}]
In the proof, we will simply write $m_t(z), \wt m_t(z),  G(z,t), \wt G(z,t)$ as $m_t, \wt m_t, G, \wt G$.
For \eqref{e:tmmdiff}, thanks to the resolvent identity \eqref{e:resolv}
        \begin{align*}
            \wt m_t =\frac{1}{N}\Tr \wt G=\frac{1}{N}\Tr  G +\frac{1}{N}\Tr[\wt G( H- \wt H)G]=m_t +\frac{1}{N}\sum_{j\in \qq{N}}\sum_{x,y\in \qq{N}}\wt G_{jx}(H- \wt H)_{xy}G_{yj}.
        \end{align*}
It follows by rearranging the above expression that
\begin{align}\begin{split}\label{e:tmmdiff2}
    |\wt m_t - m_t |
    &\lesssim \frac{1}{N}\sum_{x,y\in \qq{N}} |(H- \wt H)_{xy}\sum_{j\in \qq{N}}|\wt G_{yj}||G_{jx}|
    \lesssim 
    \frac{1}{N}\sum_{x,y\in \qq{N}} |(H- \wt H)_{xy}\sqrt{\sum_{j\in \qq{N}}|\wt G_{yj}|^2\sum_j |G_{jx}|^2}\\
    &\lesssim \sum_{x,y\in \qq{N}} |(H- \wt H)_{xy}|
    \frac{N^\fo \sqrt{\Im[\wt m_t]\Im[m_t]}}{N\eta}
    \lesssim 
    \frac{(d-1)^\ell N^\fo \sqrt{\Im[\wt m_t]\Im[m_t]}}{N\eta},
\end{split}\end{align}
where the second inequality follows from Cauchy-Schwarz inequality; the third inequality follows from \eqref{e:Gest}; the fourth inequality follows from that $\tcG$ and $\cG$ differ by $\OO((d-1)^\ell)$ edges, so the sum over $x,y$ is bounded by $\OO((d-1)^\ell)$.
Since $N\eta\geq N^\fg\gg (d-1)^\ell N^\fo$, the claim \eqref{e:tmmdiff} follows from rearranging \eqref{e:tmmdiff2}. 

For \eqref{e:use_Ward2}, since $\tcG\in \Omega$, \eqref{e:Gest}  gives
\begin{align*}\begin{split}
 \frac{1}{N}\sum_{x\notin\bX} |\wt G^{(\bX)}_{xy}|^2\lesssim \frac{ N^\fo \Im[\wt m_t ]}{N\eta} \lesssim \frac{ N^\fo \Im[ m_t ]}{N\eta},
\end{split}\end{align*}
where the last second inequality follows from \eqref{e:tmmdiff} by noticing $N\eta\gg (d-1)^\ell N^\fo$. The other statements in \eqref{e:use_Ward2} can be proven in the same way, by using \eqref{e:Trace-change} and \eqref{e:Ghao0}.
\end{proof}

\begin{proof}[Proof of \Cref{p:WtGbound}]
The upper bound $|Z_{xy}|\leq N^\fo/\sqrt{N}$ follows from that $Z_{xy}$ is Gaussian with variance $\OO(1/N)$. 

In the following we prove \eqref{e:Werror} for $\bX=\bT$. The other cases can be proven in the same way, so we omit its proof.   We recall from \eqref{e:GOEproject} that the constrained GOE matrix 
$Z$ is obtained by projecting the GOE matrix
$M$ onto the subspace of symmetric matrices with row and column sums equal to zero. 
  
  Fix the index $x\in \qq{N}$, we introduce the matrix
  $\widehat Z\in \bR^{N\times N}$ (depending on $x$), which is the projection 
of $Z$ onto the subspace of symmetric matrices with the $x$-th row and column equal to zero. For $i\in \qq{N}$, $\widehat Z_{xi}=\widehat Z_{ix}=0$; and for $i,j\neq x$
    \begin{align}\label{e:GOEproject2}
        \widehat Z_{ij}:=Z_{ij}-(z_i+z_j),  \quad z_i=-\frac{1}{N-1}\left( Z_{xi}+\frac{Z_{xx}}{2(N-1)}\right),
    \end{align}
    In this formulation, the row sums and column sums of $\widehat Z$ are all zero, which can also be obtained by projecting the GOE matrix $M$ onto the subspace of symmetric matrices, with row and column sums and the $x$-th row and column equal to zero. It follows that $\widehat Z$ is independent of the random variables $\{Z_{x1}, Z_{x2}, \cdots, Z_{xN}, g_1, \cdots, g_N\}$ and $\{M_{x1}, M_{x2}, \cdots, M_{xN}\}$. 

    We introduce the following Green's function 
    \begin{align*}
        \widehat G=\widehat G(z,t)=(H+\sqrt{t}\widehat Z-z)^{-1},
    \end{align*}
    and denote the vector $\bmz=(z_1, z_2, \cdots, z_{x-1}, 0, z_{x+1},\cdots, z_N)^\top \in \bR^N$. We recall that $x\in \bX=\bT$. Then thanks to \eqref{e:GOEproject2}, $Z^{(\bT)}-\widehat Z^{(\bT)}=\bmz^{(\bT)} (\bm1^{(\bT)})^\top+\bm1^{(\bT)}(\bmz^{(\bT)})^\top=:U V^\top$ where $U=[\sqrt N \bmz^{(\bT)} ,\bm1^{(\bT)}/\sqrt N]$ and $V=[\bm1^{(\bT)}/\sqrt N,\sqrt N\bmz^{(\bT)}]$. By the Woodbury formula \eqref{e:woodbury}
    \begin{align}\label{e:GW1}
        G^{(\bT)}-\widehat G^{(\bT)}=- G^{(\bT)}U (t^{-1/2}\bI_2-V^\top  G^{(\bT)} U)^{-1}V^\top  G^{(\bT)},\\
        \label{e:GW2}
        G^{(\bT)}-\widehat G^{(\bT)}=-\widehat G^{(\bT)}U (t^{-1/2}\bI_2+V^\top \widehat G^{(\bT)} U)^{-1}V^\top \widehat G^{(\bT)},
    \end{align}
    where $\bI_2$ is the $2\times 2$ identity matrix.

    Next we will first use \eqref{e:GW1} to get a coarse estimate for $\widehat G^{(\bT)}$.
We have the following estimates
\begin{align}\begin{split}\label{e:coarse}
    \sum_{x'\not\in \bT}G_{x'y}^{(\bT)}&\lesssim \ell,\quad \sum_{x'\not\in \bT}z_{x'} G^{(\bT)}_{x'y}
    \lesssim \ell\sqrt{\sum_{x'\not\in \bT}z^2_{x'} \sum_{x'\not\in \bT}|G^{(\bT)}_{x'y}|^2}
    \lesssim \frac{N^\fo}{\sqrt N}\sqrt{\frac{\Im[ m_t]}{N\eta}}, \\
     \sum_{x',y'\not\in \bT}z_{x'} G^{(\bT)}_{x'y'}
     &\lesssim \sum_{x'\not\in \bT}|z_{x'}| 
    \lesssim \frac{N^\fo}{\sqrt N},\\
    \sum_{x',y'\not\in \bT}z_{x'}G^{(\bT)}_{x'y'}z_{y'}
     &\lesssim \sqrt{\sum_{x',y'\not\in \bT} z_{x'}^2 z_{y'}^2 \sum_{x',y'\not\in \bT}|G^{(\bT)}_{x'y'}|^2}
     \lesssim \frac{N^\fo}{N} \sqrt{\frac{\Im[ m_t]}{N\eta}},
\end{split}\end{align}
where we used $\sum_{x'\not\in \bT}z_{x'}^2\lesssim N^\fo/N^2$ from the expression \eqref{e:GOEproject2}, as well as  \eqref{e:Gest} and \eqref{e:Ghao0}. By plugging \eqref{e:coarse} into \eqref{e:GW1}, we get for any $i,j\neq \bT$
\begin{align}\begin{split}\label{e:MatrixTerms}
&(G^{(\bT)}U)_{i\cdot}= \left[
    \begin{array}{cc}
      \OO (N^\fo\sqrt{\Im[m_t]/(N\eta)})  &   \OO(\ell/\sqrt N) 
    \end{array}
    \right],\\
&(V^\top G^{(\bT)})_{\cdot j}= \left[
    \begin{array}{cc}
    \OO(\ell/\sqrt N) \\
      \OO (N^\fo\sqrt{\Im[m_t]/(N\eta)})  
    \end{array}
    \right],\\
& V^\top G^{(\bT)}U=
    \left[
    \begin{array}{cc}
      \OO(N^\fo/\sqrt N)   &  \OO(\ell)\\
      \OO (N^\fo\sqrt{\Im[m_t]/(N\eta)}  &   \OO(N^\fo/\sqrt N) 
    \end{array}
    \right],\\
&(t^{-1/2}\bI_2-V^\top  G^{(\bT)} U)^{-1}
=
    \left[
    \begin{array}{cc}
      \OO(\sqrt t)   &  \OO(\ell t)\\
      \OO (N^\fo t\sqrt{\Im[m_t]/(N\eta)}  &   \OO(\sqrt t) 
    \end{array}
    \right].
\end{split}\end{align}
By plugging \eqref{e:MatrixTerms} into \eqref{e:GW1}, we conclude that for any $i,j\not\in \bT$
\begin{align}\label{e:GGhatdiff}
   |G^{(\bT)}_{ij}-\widehat G^{(\bT)}_{ij}|\lesssim N^{2\fo}\sqrt t\frac{\Im[m_t]}{N\eta}.
\end{align}
As a consequence of \eqref{e:GGhatdiff}, \eqref{eq:local_law} also holds for $\widehat G^{(\bT)}$, and \eqref{e:Trace-change} implies,
\begin{align}\label{e:change_trace2}
    \widehat m_t^{(\bT)}=\frac{\Tr[\widehat G^{(\bT)}]}{N}=\frac{\Tr[ G^{(\bT)}]}{N}+\OO\left(N^{2\fo}\sqrt t\frac{\Im[m_t]}{N\eta}\right)=m_t+\OO\left(\frac{N^{\fo}(d-1)^\ell\Im[m_t]}{N\eta}\right),
\end{align}
and in particularly, $\Im[\widehat m_t^{(\bT)}]\lesssim \Im[m_t]$.
Using \eqref{e:GGhatdiff} as input, by the same argument as in \eqref{e:Ghao0}, we also have
    \begin{align}\label{e:sum1}
        \left|\sum_{x'\not\in \bT}\widehat G^{(\bT)}_{x'y}\right|\lesssim \ell.
        \quad 
    \end{align}

     Next we use \eqref{e:GW2} to prove the first three statements of \eqref{e:Werror}. We first recall that  $\widehat G,\widehat Z$ and $\{Z_{x1}, Z_{x2}, \cdots, Z_{xN}\}$ are independent.
    Let $(h_1,h_2,\cdots, h_N)$ be a centered Gaussian vector with covariance $I_N/N$ independent with $\widehat Z$, and recall the covariance of $(Z_{x1}, Z_{x2}, \cdots, Z_{xN})$ from \eqref{W_ibp}: for $i,j\neq x$, 
    \begin{align*}
        \bE[Z_{xi}Z_{xj}]=\frac{N-1}{N^2}\left(\delta_{ij}-\frac{1}{N}\right).
    \end{align*}
    Then for $i\neq x$, we have the coupling between $Z_{xi}$ and $\{h_1, h_2,\cdots, h_N\}$,
    \begin{align}\begin{split}\label{e:GOEproject3}
       Z_{xi}&=\sqrt{\frac{N-1}{N}}\left(h_i-\frac{\sum_{j=1}^N h_j}{N}\right)=\sqrt{\frac{N-1}{N}}h_i +\OO\left(\frac{N^\fo}{N}\right),
       \\
       z_i&=-\sqrt{\frac{1}{(N-1)N}}h_{i} +\sqrt{\frac{1}{(N-1)N}}\frac{\sum_{j=1}^N h_{j}}{N} -\frac{Z_{xx}}{2(N-1)^2}=-\sqrt{\frac{1}{(N-1)N}}h_{i} +\OO\left(\frac{N^\fo}{N^2}\right).
    \end{split}\end{align}
    where the second statement follows from \eqref{e:GOEproject2}.
    With the coupling \eqref{e:GOEproject3}, we also have that $\widehat G, \widehat Z$ are independent from $\{h_1, h_2,\cdots, h_N\}$.
    By the concentration of measure for Gaussian vectors (see \cite[Theorem 7.7]{erdHos2017dynamical}), \eqref{e:change_trace2} and \eqref{e:sum1}, we have that with overwhelmingly high probability over $\{h_1, h_2,\cdots, h_N\}$
    \begin{align}\begin{split}\label{e:hG}
        &\sum_{x'\not\in \bT}h_{x'}\widehat G^{(\bT)}_{x'y}\lesssim N^\fo\sqrt{\frac{\Im[\widehat m_t^{(\bT)}]}{N\eta}}\lesssim N^\fo\sqrt{\frac{\Im[ m_t]}{N\eta}},\quad  \sum_{x',y'\not\in \bT}h_{x'}\widehat G^{(\bT)}_{x'y'}\lesssim N^{\fo/2},\\
        & \sum_{x',y'\not\in \bT}h_{x'}\widehat G^{(\bT)}_{x'y'}h_{y'}-\frac{\Tr[\widehat G^{(\bT)}]}{N}\lesssim N^\fo\sqrt{\frac{\Im[\widehat m_t^{(\bT)}]}{N\eta}}\lesssim N^\fo\sqrt{\frac{\Im[ m_t]}{N\eta}}.
    \end{split}\end{align}
Thanks to \eqref{e:change_trace2}, \eqref{e:GOEproject3} and \eqref{e:hG}
\begin{align}\begin{split}\label{e:refine}
    \sum_{x'\not\in \bT}z_{x'}\widehat G^{(\bT)}_{x'y}
    &=\frac{\OO(1)}{N} \sum_{x'\not\in \bT}h_{x'}\widehat G^{(\bT)}_{x'y} +\frac{\OO(N^{\fo/2})}{N^2}\sum_{x'\not\in \bT}\widehat G^{(\bT)}_{x'y}
    \lesssim \frac{N^\fo}{N}\sqrt{\frac{\Im[ m_t]}{N\eta}},\\
     \sum_{x',y'\not\in \bT}z_{x'}\widehat G^{(\bT)}_{x'y'}
     &=\frac{\OO(1)}{N} \sum_{x',y'\not\in \bT}h_{x'}\widehat G^{(\bT)}_{x'y'} +\frac{\OO(N^{\fo/2})}{N^2}\sum_{x',y'\not\in \bT}\widehat G^{(\bT)}_{x'y'}
    \lesssim \frac{N^\fo}{N},\\
    \sum_{x',y'\not\in \bT}z_{x'}\widehat G^{(\bT)}_{x'y'}z_{y'}
     &=\frac{1}{(N-1)N} \sum_{x',y'\not\in \bT}h_{x'}\widehat G^{(\bT)}_{x'y'}h_{y'} +\frac{\OO(N^{\fo/2})}{N^3} \sum_{x',y'\not\in \bT}h_{x'}\widehat G^{(\bT)}_{x'y'}+\frac{\OO(N^{\fo})}{N^4}\sum_{x',y'\not\in \bT}\widehat G^{(\bT)}_{x'y'}\\
    &=\frac{\widehat m_t^{(\bT)}}{N^2}+ \OO\left(\frac{N^\fo}{N^2}\sqrt{\frac{\Im[ m_t]}{N\eta}}\right)=\frac{ m_t}{N^2}+ \OO\left(\frac{N^\fo}{N^2}\sqrt{\frac{\Im[ m_t]}{N\eta}}\right).
\end{split}\end{align}
Similarly,
\begin{align}\begin{split}\label{e:refine2}
    &\sum_{x'\not\in \bT}Z_{xx'}\widehat G^{(\bT)}_{x'y}
    \lesssim {N^\fo}\sqrt{\frac{\Im[m_t]}{N\eta}},\quad 
    \sum_{x',y'\not\in \bT}Z_{xx'}\widehat G^{(\bT)}_{x'y'}z_{y'}
    =\frac{\widehat m_t^{(\bT)}}{N}+ \OO\left(\frac{N^\fo}{N}\sqrt{\frac{\Im[ m_t]}{N\eta}}\right),\\
    &\sum_{x',y'\not\in \bT}Z_{xx'}\widehat G^{(\bT)}_{x'y'}
    \lesssim {N^\fo},\quad  \sum_{x',y'\not\in \bT}Z_{xx'}\widehat G^{(\bT)}_{x'y'}Z_{y'x}
    = m_t+ \OO\left({N^\fo}\sqrt{\frac{\Im[ m_t]}{N\eta}}\right).
\end{split}\end{align}
Now we can plug \eqref{e:refine} and \eqref{e:refine2} into \eqref{e:GW2}, for $y\not\in \bT$
\begin{align*}
&(ZG^{(\bT)}U)_{x\cdot}= \left[
    \begin{array}{cc}
      \OO (1/\sqrt N)  &   \OO(N^\fo/\sqrt N) 
    \end{array}
    \right], \quad
   (V^\top G^{(\bT)})_{\cdot y}= \left[
    \begin{array}{c}
    \OO(\ell/\sqrt N) \\
      \OO ((N^\fo/\sqrt N)\sqrt{\Im[m_t]/(N\eta)})  
    \end{array}
    \right],\\
&V^\top G^{(\bT)}\bm1^{(\bT)}= \left[
    \begin{array}{c}
      \OO (\ell \sqrt N)  \\  \OO(N^\fo/\sqrt N) 
    \end{array}
    \right], \quad
   (V^\top G^{(\bT)}Z)_{\cdot x}= \left[
    \begin{array}{c}
    \OO(N^\fo/\sqrt N) \\
      \OO (1/\sqrt N)  
    \end{array}
    \right],\\
  &  V^\top \widehat G^{(\bT)}U=
    \left[
    \begin{array}{cc}
      \OO(N^\fo/ N)   &  \OO(\ell)\\
      \OO (1/N)  &   \OO(N^\fo/ N) 
    \end{array}
    \right],\quad (t^{-1/2}\bI_2-V^\top  \widehat G^{(\bT)} U)^{-1}
=
    \left[
    \begin{array}{cc}
      \OO(\sqrt t)   &  \OO(\ell t)\\
      \OO (t/N)  &   \OO(\sqrt t) 
    \end{array}
    \right],  
\end{align*}
and conclude
\begin{align*}
    \sum_{x'\not\in \bT}Z_{xx'} G^{(\bT)}_{x'y}
    &=\sum_{x'\not\in \bT}Z_{xx'} \widehat G^{(\bT)}_{x'y}
    - \sum_{x'\not\in \bT}Z_{xx'} (\widehat G^{(\bT)}U (t^{-1/2}\bI_2+V^\top \widehat G^{(\bT)} U)^{-1}V^\top \widehat G^{(\bT)})_{x'y}\\
    &\lesssim 
    {N^\fo}\sqrt{\frac{\Im[m_t]}{N\eta}}+\frac{\sqrt t N^{2\fo}}{N}\lesssim {N^\fo}\sqrt{\frac{\Im[m_t]}{N\eta}},\\
    \sum_{x',y'\not\in \bT}Z_{xx'} G^{(\bT)}_{x'y'}Z_{y'x}
    &=\sum_{x',y'\not\in \bT}Z_{xx'} \widehat G^{(\bT)}_{x'y'}Z_{y'x}
    - \sum_{x',y'\not\in \bT}Z_{xx'} (\widehat G^{(\bT)}U (t^{-1/2}\bI_2+V^\top \widehat G^{(\bT)} U)^{-1}V^\top \widehat G^{(\bT)})_{x'y'}Z_{y'x}\\
    &= m_t+ \OO\left({N^\fo}\sqrt{\frac{\Im[m_t]}{N\eta}}+\frac{\sqrt{t}N^{2\fo}}
    {N}\right)
    =m_t+ \OO\left({N^\fo}\sqrt{\frac{\Im[m_t]}{N\eta}}\right),\\
  \sum_{x',y'\not\in \bT}Z_{xx'} G^{(\bT)}_{x'y'}
    &=\sum_{x',y'\not\in \bT}Z_{xx'} \widehat G^{(\bT)}_{x'y'}
    - \sum_{x',y'\not\in \bT}Z_{xx'} (\widehat G^{(\bT)}U (t^{-1/2}\bI_2+V^\top \widehat G^{(\bT)} U)^{-1}V^\top \widehat G^{(\bT)})_{x'y'}\\
    &\lesssim N^\fo+N^\fo\sqrt t\lesssim N^\fo.
\end{align*}
This finishes the first three statements in \eqref{e:Werror}. 

For the last statement in \eqref{e:Werror}, if $x\in \bX$, the claim follows from the first statement in \eqref{e:Werror}. Otherwise, the Schur complement formula \eqref{e:Schurixj}, and the first statement in \eqref{e:Werror} imply
\begin{align}\begin{split}\label{e:ZG}
    &\phantom{{}={}}(Z G^{(\bX)})_{x y}
    =Z_{xx} G^{(\bX)}_{x y}+\sum_{x'}Z_{xx'} \left(G^{(x\bX)}_{x' y}-(G^{(x\bX)}(H+\sqrt t Z))_{x'x}G^{(\bX)}_{xy}\right)\\
    &=(Z_{xx}-\sqrt t (Z G^{(x\bX)}Z)_{xx}) G^{(\bX)}_{x y}
    +\OO(N^{\fo}\sqrt{\Im[m_t]/(N\eta)})
    \lesssim  |G^{(\bX)}_{x y}|+N^{\fo}\sqrt{\Im[m_t]/(N\eta)}.
\end{split}\end{align}
We conclude from \eqref{e:ZG} and \eqref{e:Gest} that
\begin{align*}
&\frac{1}{N}\sum_y|(Z G^{(\bX)})_{x y} |^2
\lesssim \frac{1}{N}\sum_y| G^{(\bX)}_{x y} |^2
+\frac{N^{2\fo} \Im[m_t]}{N\eta}
\lesssim \frac{N^{2\fo} \Im[m_t]}{N\eta}.
\end{align*}

Finally, the claims in \eqref{e:Werror2} follow from the same argument as those in \eqref{e:Werror}, by using \eqref{e:use_Ward2}. So we omit their proofs.
 
\end{proof}

\section{Proofs of \Cref{thm:eigrigidity} and  \Cref{c:rigidity} }
\label{s:proofrigidity}
In this section, we adopt the notation from \Cref{s:setting3}. 
For $z\in \bf D$, we denote $z_t=z+t\md(z,t)$. 
Let $\eta=\Im[z]$ and $\kappa:=\min\{|\Re[z]-E_t|, |\Re[z]+E_t|\}$.
We recall from \eqref{e:relation_zt_z}, if $|z-E_t|\ll 1$, then 
 \begin{align}\begin{split}\label{e:ztform}
    \sqrt{z_t-2}
    =\frac{\cA t}{2}+\sqrt{z-E _t}+\OO\left(t^2+t\sqrt{|z-E _t|}\right).
\end{split}\end{align}
The estimate \eqref{e:ztform} together with \eqref{e:medge_behavior} give us
\begin{align}\begin{split}\label{e:edge_behavior2}
&\msc(z_t)=-1+\frac{\cA t}{2}+\sqrt{z-E _t}+\OO(t^2+|z-E_t|)=-1+\OO(t+\sqrt{|z-E_t|}),\\
&
m_d(z_t)=-\frac{d-1}{d-2}+\cA\left(\frac{\cA t}{2}+\sqrt{z-E _t}\right)+\OO(t^2+|z-E_t|)=-\frac{d-1}{d-2}+\OO(t+\sqrt{|z-E_t|}),\\
&1-\msc(z_t)^2=2\sqrt{z-E_t}+\cA t+\OO(t^2+|z-E_t|).
\end{split}\end{align}
We recall the square root behavior of $\Im[\msc(z_t)],\Im[m_d(z_t)]$ from \eqref{eq:square_root_behave}. 
\begin{align}\begin{split}\label{eq:square_root_behave_copy}
&\Im[\msc(z_t)],\Im[m_d(z_t)]\asymp\left\{\begin{array}{ll}
\sqrt{\kappa+\eta}&\text{ for }  |\Re[z]|\leq E_t,\\
\eta/\sqrt{\kappa+\eta}&\text{ for }  |\Re[z]|\geq E_t.
\end{array}
\right. 
\end{split}
\end{align}

We observe that for any $z\in\bf D$, $|\msc(z_t)|\asymp 1$. We can chose $\ell$ satisfying $\ft\ll \ell/\log_{d-1}N \ll \fb$, and $|1+\msc^2(z_t) +\msc^4(z_t)+\cdots+\msc^{2\ell}(z_t) |\gtrsim 1$. This choice of $\ell$ will be used consistently throughout the remainder of this section.

We compile some estimates below that will be used later.
\begin{lemma}
For any $z\in \mathbf D$, we denote $\eta=\Im[z]$, $\kappa=\min\{|\Re[z]-E_t|, |\Re[z]+E_t|\}$ and $z_t=z+t\md(z,t)$. The following holds:
\begin{align}\begin{split}\label{e:1m}
    &1-\del_1 Y_\ell(\msc(z_t), z_t)=(1-\msc^{2}(z_t))(1+\msc^2(z_t)+\msc^4(z_t)+\cdots+\msc^{2\ell}(z_t)),\\
    &\del_2 Y_\ell(\msc(z_t), z_t)=\msc^2(z_t)(1+\msc^2(z_t)+\msc^4(z_t)+\cdots+\msc^{2\ell}(z_t)),
\end{split}\end{align}
and 
\begin{align}\begin{split}
\label{e:1m2}
&1-\msc^2(z_t)-t\cA\msc^2(z_t) \asymp \sqrt{\kappa+\eta},\\
    &1-\msc^2(z_t)-\frac{td\md^2(z_t) \msc^{2\ell+2}(z_t)}{d-1}\asymp \sqrt{\kappa+\eta}.
\end{split}\end{align}  
\end{lemma}
\begin{proof}
    The claim \eqref{e:1m} follows from \eqref{e:recurbound}. For the first statement in \eqref{e:1m2}, if $\kappa+\eta\asymp 1$, both sides are of order $1$, there is nothing to prove. Otherwise, we assume $|z-E_t|\ll1$ (the other case $|z+E_t|\ll1$ can be proven in the same way, so we omit). We have
    \begin{align*}
       1-\msc^2(z_t)-t\cA\msc^2(z_t) 
       &=2\sqrt{z-E_t}+\cA t+\OO(t^2+|z-E_t|)-t\cA(1+\OO(t+\sqrt{|z-E_t|}))\\
       &=2\sqrt{z-E_t}+\OO(t^2+t\sqrt{|z-E_t|}+|z-E_t|)\asymp \sqrt{\kappa+\eta},
    \end{align*}
   where in the first statement we used using \eqref{e:edge_behavior2}; in the second statement we used $t^2\leq N^{-2/3+2\ft}\ll N^{-(1-\fg)/2}\leq \sqrt{\kappa+\eta}$.

   The second statement in \eqref{e:1m2} follows from the same argument by noticing $d\md^2(z_t)/(d-1)=\cA(t+\sqrt{|z-E_t|})$.
\end{proof}




The following stability proposition states that if $|m_t(z)-X_t (z)|$ and $|Q_t(z)-Y_t (z)|$ are small, so are $|Q_t(z)-\msc(z_t)|$ and $|m_t(z)-\md(z_t)|$.

\begin{proposition}\label{p:stable}
For any $z\in \mathbf D$, let $z_t=z+t\md(z,t)$. We denote $\eta=\Im[z]$, $\kappa=\min\{|\Re[z]-E_t|, |\Re[z]+E_t|\}$, and
\begin{align}\label{e:defdelta}
|Q_t(z)-Y_t(z)|\lesssim \delta_1,\quad |m_t(z)-X_t(z)|\lesssim \delta_2.
\end{align}
We assume $\delta_1, \delta_2\leq \ell^{-10}$, and one of the following holds:
\begin{enumerate}
    \item either $|Q_t(z)-\msc(z_t)|\ll \sqrt{\kappa+\eta}$, and $\delta_1+\ell^3(t\delta_2+t^3)\ll \kappa+\eta$,
    \item or $\delta_1+\ell^3(t\delta_2+t^3)\gtrsim\kappa+\eta$.
\end{enumerate}
  Then the following holds
\begin{align}\label{e:Q-msc}
&|Q_t(z)-\msc(z_t)|\lesssim \frac{\delta_1+\ell^5(t\delta_2+t^3)}{\sqrt{\kappa+\eta+(\delta_1+\ell^5(t\delta_2+t^3))}},\\
&|m_t(z)-m_d(z_t)|\lesssim \delta_2 +\frac{\delta_1+\ell^5(t\delta_2+t^3)}{\sqrt{\kappa+\eta+(\delta_1+\ell^5(t\delta_2+t^3))}}.
\end{align}
\end{proposition}
\begin{proof}
In the proof, for the simplicity of  notation, we write $Q_t(z),m_t(z), Y_t(z),  X_t(z), \msc(z_t), \md(z_t)$ as $Q_t,m_t, Y_t,  X_t, \msc, \md$.   We have from \eqref{e:recurbound} that
\begin{align}\begin{split}\label{e:Q-YQ4}
&\phantom{{}={}}Q_t-Y_t=(Q_t-\msc)-(Y_t-\msc)\\
&=(1-\msc^{2\ell+2})(Q_t-\msc)-\msc^{2\ell+2}\md\left(\frac{1-\msc^{2\ell+2}}{d-1}+\frac{d-2}{d-1}\frac{1-\msc^{2\ell+2}}{1-\msc^2}\right)(Q_t-\msc)^2\\
&-\frac{\msc^2(1-\msc^{2\ell+2})}{1-\msc^2}t(m_t-\md)+\OO(\ell^5(t^2|m_t-\md|^2 +t|Q_t-\msc||m_t-\md|+|Q_t-\msc|^3)),
\end{split}
\end{align}
and  
\begin{align}\begin{split}\label{e:mt-Xt1}
m_t-X_t
&=(m_t-m_d)-\frac{d}{d-1}\md^2\msc^{2\ell}(Q_t-\msc)+\OO\left(\ell^3(t|m_t-\md|+|Q_t-\msc|^2)\right),\\
m_t-m_d
&=m_t-X_t+\frac{d}{d-1}\md^2\msc^{2\ell}(Q_t-\msc)+\OO\left(\ell^3(t|m_t-X_t|+t|Q_t-\msc|+|Q_t-\msc|^2)\right).
\end{split}\end{align}

By plugging \eqref{e:mt-Xt1} into \eqref{e:Q-YQ4}, we obtain that
\begin{align}\begin{split}\label{e:Q-YQ5}
Q_t-Y_t&=-\msc^{2\ell+2}\md\left(\frac{1-\msc^{2\ell+2}}{d-1}+\frac{d-2}{d-1}\frac{1-\msc^{2\ell+2}}{1-\msc^2}\right)(Q_t-\msc)^2\\
&+\left(1-\msc^{2\ell+2}-\frac{d}{d-1}\frac{t\md^2 \msc^{2\ell+2}(1-\msc^{2\ell+2})}{1-\msc^2}\right)(Q_t-\msc)\\
&+\OO(\ell^5(t|m_t-X_t| +t^2|Q_t-\msc|+t|Q_t-\msc|^2+|Q_t-\msc|^3)).
\end{split}
\end{align}

We consider the quadratic equation $ax^2+bx+c=0$, with
\begin{align}\begin{split}\label{e:coeff}
&a=-\msc^{2\ell+2}\md\left(\frac{1-\msc^{2\ell+2}}{d-1}+\frac{d-2}{d-1}\frac{1-\msc^{2\ell+2}}{1-\msc^2}\right)+\OO\left(\ell^5(|Q_t -\msc|+t)\right),\\
&b=1-\msc^{2\ell+2}-\frac{d}{d-1}\frac{t\md^2 \msc^{2\ell+2}(1-\msc^{2\ell+2})}{1-\msc^2}
=\frac{(1-\msc^{2\ell+2})}{1-\msc^2}\left(1-\msc^2-\frac{td\md^2 \msc^{2\ell+2}}{d-1}\right),\\
& 
c=Y_t-Q_t+\OO(\ell^5(t|m_t-X_t| +t^3)).
\end{split}\end{align}
Recall $\delta_1, \delta_2$ from \eqref{e:defdelta}, and we have  $|c|\lesssim \delta_1+\ell^5(t\delta_2+t^3)$.

We recall that by our choice of $\ell$, it holds $|1+\msc ^2+\msc ^4+\cdots+\msc ^{2\ell}|\gtrsim 1$. Also, we can rewrite $a$ from \eqref{e:coeff} as
\begin{align*}
    a=-\msc^{2\ell+2}\md\frac{d-1-\msc^{2}}{d-1}\frac{1-\msc^{2\ell+2}}{1-\msc^2}+\OO\left(\ell^5(|Q_t -\msc|+t)\right).
\end{align*}
It follows that $1\lesssim |a|\asymp |1+\msc ^2+\msc ^4+\cdots+\msc ^{2\ell}|\lesssim \ell$. 
Moreover, using \eqref{e:1m2}, we have
\begin{align}\label{e:abratio}
\left|\frac{b}{a}\right|\asymp \left|
    1-\msc^2-\frac{td\md^2 \msc^{2\ell+2}}{d-1}\right|\asymp \sqrt{\kappa+\eta},
\end{align}
and 
$\sqrt{\kappa+\eta}\lesssim |b|\lesssim\ell \sqrt{\kappa+\eta}$. 

We discuss the two assumptions separately,
\begin{enumerate}
    \item If $\delta_1+\ell^5(t\delta_2+t^3)\ll \kappa+\eta$, then  $|ac/b^2|\lesssim |c|/(\kappa+\eta)\ll1$.
 The two roots of $ax^2+bx+c=0$ satisfies:$|x_1|=|b/a|+\OO(|c/b|)\asymp |b/a|\gtrsim \sqrt{\kappa+\eta}$ and 
\begin{align}\label{e:cc1}
|x_2|=\OO(|c/b|)\lesssim \frac{|c|}{\sqrt{\kappa+\eta}}\lesssim \frac{\delta_1+\ell^5(t\delta_2+t^3)}{\sqrt{\kappa+\eta+(\delta_1+\ell^5(t\delta_2+t^3))}}.
\end{align}
In this case, our assumption states that $|Q_t-\msc|\ll \sqrt{\kappa+\eta}$, it is necessary that 
\begin{align*}
    |Q_t-\msc|=|x_2|\lesssim \frac{\delta_1+\ell^5(t\delta_2+t^3)}{\sqrt{\kappa+\eta+(\delta_1+\ell^5(t\delta_2+t^3))}}.
\end{align*}
    \item If $\delta_1+\ell^5(t\delta_2+t^3)\gtrsim \kappa+\eta$, then using \eqref{e:abratio}, we have
    \begin{align*}
        |x_1|, |x_2|=\OO(|b/a|+\sqrt{|c/a|})\lesssim \sqrt{\kappa+\eta}+\sqrt{|c|}\lesssim\sqrt{\delta_1+\ell^5(t\delta_2+t^3)} \lesssim \frac{\delta_1+\ell^5(t\delta_2+t^3)}{\sqrt{\kappa+\eta+(\delta_1+\ell^5(t\delta_2+t^3))}}.
    \end{align*}
\end{enumerate}

The first statement in \eqref{e:Q-msc} follows.
The second statement in \eqref{e:Q-msc} follows from plugging the first statement into \eqref{e:mt-Xt1}
\begin{align*}
\begin{split}
|m_t-m_d|&\lesssim |m_t-X_t|+|Q_t -\msc |\lesssim \delta_2 +\frac{\delta_1+\ell^5(t\delta_2+t^3)}{\sqrt{\kappa+\eta+(\delta_1+\ell^5(t\delta_2+t^3))}}.
\end{split}
\end{align*}
\end{proof}




An easy reformulation of \Cref{t:recursion} gives a useful bound, as stated in the following corollary. 
\begin{corollary}\label{cor:QYQbound}
Adopt the notation of Theorem \ref{t:recursion}, the following holds:
\begin{align}\begin{split}\label{e:QYQbound}
&\bE[\bm1(\GG\in \Omega)|Q_t-Y_t|^{2p}]
\lesssim N^{p\fc}\bE\left[\bm1(\GG\in \Omega)\left(\sqrt{\frac{\Upsilon \Phi}{N\eta}}+\Phi\right)^{2p}\right],\\
&\bE[\bm1(\GG\in \Omega)|m_t-X_t|^{2p}]
\lesssim N^{p\fc}\bE\left[\bm1(\GG\in \Omega)\left(\sqrt{\frac{ \Phi}{N\eta}}+\Phi\right)^{2p}\right].
\end{split}\end{align}
If we further assume that $\min\{|z-E_t|, |z+E_t|\}\leq N^{-\fg}$, we have the following improved estimate
\begin{align}\label{e:QYQbound2}
    &\bE[\bm1(\GG\in \Omega)|Q_t-Y_t|^{2p}]
\lesssim \bE\left[\bm1(\GG\in \Omega)\left((\ell \Phi)^{2p}+\left(\frac{\ell \widetilde \Upsilon \Phi}{N\eta}\right)^p+\left(\frac{\Upsilon \Phi}{(d-1)^{\ell/4}N\eta}\right)^{p}\right)\right],
\end{align}
where 
\begin{align*}
    \wt \Upsilon= |1-\del_1 Y_\ell(Q_t,z+tm_t)-\cA \del_2 Y_\ell(Q_t,z+tm_t)|.
\end{align*}
\end{corollary}
\begin{proof}[Proof of \Cref{cor:QYQbound}]
We will only prove the statement \eqref{e:QYQbound2}. The two weaker estimates in \eqref{e:QYQbound} can be proven in the same way using \eqref{e:QY} and \eqref{e:mz}, so we omit their proofs.

Assume $z=z_1=z_2=\cdots=z_p$ and $\overline{z}=z_{p+1}=z_{p+1}=\cdots z_{2p}$ in \eqref{e:Qrefined_bound}. Notice that
\begin{align*}
    \frac{|\del_z m_t(z)|}{N}\leq \frac{\Im[m_t(z)]}{N\eta}\leq \Phi(z), \quad \frac{1}{N^2}\left|\del_{z_j}\left(\frac{m_t(z)-m_t(z_j)}{z-z_j} \right)\right|\leq \frac{\Im[m_t(z)]}{N^2\eta^2}\leq \frac{\Phi(z)}{N\eta}.
\end{align*}
By substituting this into \eqref{e:Qrefined_bound}, we obtain
\begin{align}\begin{split}\label{e:QYbound1}
  & \phantom{{}={}}\bE\left[\bm1(\cG\in \Omega)|Q_t(z)-Y_t(z)|^{2p}\right]\lesssim 
\bE\left[\bm1(\cG\in \Omega)\ell \Phi(z) |Q_t(z)-Y_t(z)|^{2p-1}\right] \\
&+    \bE\left[\bm1(\cG\in \Omega)\frac{\ell \widetilde \Upsilon(z) \Phi(z)}{N\eta }|Q_t(z)-Y_t(z)|^{2p-2}\right]+(d-1)^{-\ell/2}\ell N^\fo\bE[\Psi_{2p}(z)].
\end{split}\end{align}

We recall Young's inequality: for any  $\lambda>0$ and $A, B\geq 0$, we can bound
\begin{align}\begin{split}\label{e:Young}
    A^{k} B^{2p-k}
    \leq \frac{2p-k}{2p}\lambda B^{2p}  + \frac{k}{2p}\lambda^{-\frac{2p-k}{k}} A^{2p}\leq \lambda B^{2p}  + \lambda^{-\frac{2p-k}{k}} A^{2p}.
\end{split}\end{align}
By taking $(A,B, k)=(\ell\Phi(z),|Q_t(z)-Y_t(z)|,1)$ and $(A,B,k)=((\ell \widetilde \Upsilon(z)\Phi(z)/(N\eta))^{1/2},|Q_t(z)-Y_t(z)|, 2)$, we can bound the first two terms on the right-hand side of \eqref{e:QYbound1} as
\begin{align}\begin{split}\label{e:QYbound2}
&\phantom{{}={}}\bE\left[\bm1(\cG\in \Omega)\ell \Phi(z) |Q_t(z)-Y_t(z)|^{2p-1}\right] +    \bE\left[\bm1(\cG\in \Omega)\frac{\ell \widetilde \Upsilon(z)\Phi(z)}{N\eta }|Q_t(z)-Y_t(z)|^{2p-2}\right]\\
&\leq 2\lambda \bE[\bm1(\cG\in \Omega)|Q_t(z)-Y_t(z)|^{2p}]
+\bE\left[\bm1(\cG\in \Omega)\left(\lambda^{1-2p}(\ell\Phi(z))^{2p}
+\lambda^{1-p}\left(\frac{\ell \widetilde \Upsilon(z)\Phi(z)}{N\eta}\right)^{p}\right)\right].
\end{split}\end{align}

For the last term on the right-hand side of \eqref{e:QYbound1}, we notice that
\begin{align}\begin{split}\label{e:Upsilon_bound}
    &\phantom{{}={}}\Upsilon(z)
    =|1-\del_1 Y_\ell(Q_t(z),z+tm_t(z))|+t|\del_2 Y_\ell(Q_t(z),z+tm_t(z))|\\
    &=|1-\del_1 Y_\ell(\msc(z_t),z_t)|+t|\del_2 Y_\ell(\msc(z_t),z_t)|
    +\OO(\ell^3(|Q_t(z)-\msc(z_t)|+t|m_t(z)-\md(z_t)|))\\
    &\lesssim \ell(\sqrt{\kappa+\eta}+t)
    +\ell^3(|Q_t(z)-\msc(z_t)|+t|m_t(z)-\md(z_t)|),
\end{split}\end{align}
where we used \eqref{e:Yl_derivative} in the first statement; and we used \eqref{e:edge_behavior2} and \eqref{e:1m} in the second statement.

For $\cG\in \Omega$ and $|z-E_t|\leq N^{-\fg}$, \eqref{eq:infbound} and \eqref{e:Upsilon_bound} imply that with overwhelmingly high probability over $Z$, we have $\Upsilon(z)\lesssim \ell^3 N^{-2\fb}$ and $N^\fb\Upsilon(z)\ll 1$. We can simplify  $\Psi_{2p}(z)$
from \eqref{eq:phidef} as
\begin{align}\label{e:Phi_bound}
    \Psi_{2p}(z)\lesssim \bm1(\cG\in \Omega)\Bigg[|Q_t(z)-Y_t(z)|^{2p}+\Phi^{2p}(z) +\frac{\Upsilon(z)\Phi(z)}{N\eta}(\Phi^{2p-2}(z)+|Q_t(z)-Y_t(z)|^{2p-2})
    \Bigg].
\end{align}
We can bound the last term in \eqref{e:Phi_bound} taking $(A,B,k)=(( \Upsilon(z)\Phi(z)/((d-1)^{\ell/4}N\eta))^{1/2},|Q_t(z)-Y_t(z)|, 2)$ in \eqref{e:Young}, and get
\begin{align}\begin{split}\label{e:QYbound3}
    \frac{\ell N^\fo}{(d-1)^{\ell/2}}\bE[\Psi_{2p}(z)]&\leq \bE\left[\bm1(\cG\in \Omega)\frac{|Q_t(z)-Y_t(z)|^{2p}}{(d-1)^{\ell/4}}+\Phi^{2p}(z) \right]+\\
    &+\lambda\bE[\bm1(\cG\in \Omega)|Q_t(z)-Y_t(z)|^{2p}]
+\bE\left[\bm1(\cG\in \Omega)\lambda^{1-p}\left(\frac{ \Upsilon(z)\Phi(z)}{(d-1)^{\ell/4}N\eta}\right)^{p}\right].
\end{split}\end{align}

Now, we substitute \eqref{e:QYbound2} and \eqref{e:QYbound3} into \eqref{e:QYbound1}, choosing 
$\lambda$ sufficiently small so that the total coefficient of $\bE[\bm1(\cG\in \Omega)|Q_t(z)-Y_t(z)|^{2p}]$  on the right-hand side becomes less than 
$1/2$. By rearranging, we obtain \eqref{e:QYQbound2}.
\end{proof}

In the rest of this section, we prove \Cref{thm:eigrigidity} and \Cref{c:rigidity} using  \Cref{cor:QYQbound} as input.

\begin{lemma}\label{p:iterate}
For any $z\in \mathbf D$, let $z_t=z+t\md(z,t)$ and we denote $\eta=\Im[z]$, $\kappa=\min\{|\Re[z]-E_t|, |\Re[z]+E_t|\}$. We assume that there exists some deterministic control parameter $\Lambda$ such that the following holds with overwhelmingly high probability
\begin{align}\label{e:defLambda}
\bm1(\cG\in \Omega)|m_t(z)-m_d(z_t)|, \bm1(\cG\in \Omega)|Q_t(z)-   \msc(z_t)|\leq\Lambda.
\end{align}
Then the following improved estimates hold with overwhelmingly high probability:
\begin{enumerate}
    \item 

If   $N\eta\sqrt{\kappa+\eta}\geq N^{6\fc}$ and $\Lambda \leq \sqrt{\kappa+\eta}/N^\fo$, then 
\begin{align}\label{e:weakbb1}
\bm1(\cG\in \Omega)|Q(z)-\msc(z_t)|, \bm1(\cG\in \Omega)|m(z)-m_d(z_t)|\lesssim  \frac{N^{2\fc}}{N \eta}.
\end{align}
\item 
If  $N\eta\sqrt{\kappa+\eta}\leq N^{8\fc}$ and $\Lambda \leq N^{10\fc}/N\eta$, then 
\begin{align}\label{e:weakbb2}
\bm1(\cG\in \Omega)|Q(z)-\msc(z_t)|, \bm1(\cG\in \Omega)|m(z)-m_d(z_t)|\lesssim  \frac{N^{8\fc}}{N \eta}.
\end{align}
\item 
If $\max\{|z-E_t|, |z+E_t|\}\leq N^{-\fg}$, $N\eta\sqrt{\kappa+\eta}\geq N^{6\fo}$, and $\Lambda \leq \sqrt{\kappa+\eta}/N^\fo$ then
\begin{align}\label{e:bb1}
\bm1(\cG\in \Omega)|Q(z)-\msc(z_t)|, \bm1(\cG\in \Omega)|m(z)-m_d(z_t)|\lesssim  \frac{N^{2\fo}}{N \eta}.
\end{align}
and if in addition $|\Re[z]|\geq E_t$, then
\begin{align}\label{e:bb2}
\bm1(\cG\in \Omega)|Q(z)-\msc(z_t)|, \bm1(\cG\in \Omega)|m(z)-m_d(z_t)|\lesssim \frac{N^{2\fo}}{\sqrt{\kappa+\eta}}\left(\frac{1}{N\sqrt{\eta}}+\frac{1}{(N\eta)^2}\right).
\end{align}

\item If $\max\{|z-E_t|, |z+E_t|\}\leq N^{-\fg}$, $N\eta\sqrt{\kappa+\eta}\leq N^{8\fo}$, and $\Lambda \leq N^{10\fo}/N\eta$ then
\begin{align}\label{e:bb3}
\bm1(\cG\in \Omega)|Q(z)-\msc(z_t)|, \bm1(\cG\in \Omega)|m(z)-m_d(z_t)|\lesssim  \frac{N^{8\fo}}{N \eta}.
\end{align}

\end{enumerate}

\end{lemma}

\begin{proof}

We recall the moment bounds of $Q_t-Y_t$ and $m_t-X_t$ from \Cref{cor:QYQbound},
\begin{align}\begin{split}\label{e:Q-YQ}
&\bE[\bm1(\GG\in \Omega)|Q_t-Y_t|^{2p}]
\lesssim N^{p\fc}\bE\left[\bm1(\GG\in \Omega)\left(\sqrt{\frac{\Upsilon \Phi}{N\eta}}+\Phi\right)^{2p}\right],\\
&\bE[\bm1(\GG\in \Omega)|m_t-X_t|^{2p}]
\lesssim N^{p\fc}\bE\left[\bm1(\GG\in \Omega)\left(\sqrt{\frac{ \Phi}{N\eta}}+\Phi\right)^{2p}\right],
\end{split}\end{align}
and for $\min\{|z-E_t|, |z+E_t|\}\leq N^{-\fg}$
\begin{align}\label{e:goodQ-YQ}
&\bE[\bm1(\GG\in \Omega)|Q_t-Y_t|^{2p}]
\lesssim \bE\left[\bm1(\GG\in \Omega)\left((\ell \Phi)^{2p}+\left(\frac{\ell \widetilde \Upsilon \Phi}{N\eta}\right)^p+\left(\frac{\Upsilon \Phi}{(d-1)^{\ell/4}N\eta}\right)^{p}\right)\right].
\end{align}

Using the control parameter $\Lambda$ from \eqref{e:defLambda}, we can estimate
\begin{align}\begin{split}\label{e:Immz}
&\Im[m_t(z)]\leq \Im[\md(z_t)]+|m_t(z)- m_d(z_t)|\leq \Im[\md(z_t)] +\Lambda ,\\
&\Phi(z)=\frac{\Im[m_t(z)]}{N\eta}+\frac{1}{N^{1-2\fc}}\leq \frac{\Im[m_d(z_t)]+\Lambda}{N\eta}+\frac{1}{N^{1-2\fc}}\lesssim \frac{\sqrt{\kappa+\eta}+\Lambda}{N\eta}+\frac{1}{N^{1-2\fc}},
\end{split}\end{align}
where we used \eqref{eq:square_root_behave_copy}, and
\begin{align}\begin{split}\label{e:Upz}
\Upsilon(z)&=
|1-\del_1 Y_\ell(\msc(z_t), z_t)|+t|\del_2 Y_\ell(\msc(z_t),z_t)|+\OO(\ell^3(|Q_t-\msc|+|m_t-\md|))+\\
&+(d-1)^{8\ell}\Phi\lesssim \ell(\sqrt{\kappa+\eta} +t)
+\ell^3\Lambda\lesssim \ell^3(\sqrt{\kappa+\eta} +t+\Lambda),\\
\wt \Upsilon(z)&= |1-\del_1 Y_\ell(\msc(z_t),z_t)-t\cA \del_2 Y_\ell(\msc(z_t),z_t)| +\OO(\ell^3(|Q_t-\msc|+|m_t-\md|))\\
&\lesssim \ell\sqrt{\kappa+\eta}+\ell^3\Lambda
\lesssim \ell^3(\sqrt{\kappa+\eta}+\Lambda),
\end{split}\end{align}
where we used \eqref{e:Yl_derivative}, \eqref{e:1m} and \eqref{e:1m2}.

Noticing that when $\min\{|z-E_t|, |z+E_t|\}\leq N^{-\fg}$, we have $\Phi(z)\lesssim (\Im[\md(z_t)]+\Lambda) /N\eta$. Plugging \eqref{e:Immz} and \eqref{e:Upz} into  \eqref{e:goodQ-YQ}, we get
\begin{align}\begin{split}\label{e:Q-YQ2}
&\phantom{{}={}}\bE[\bm1(\GG\in \Omega)|Q_t-Y_t|^{2p}]\\
&\lesssim \ell^{4p}\bE\left[\bm1(\GG\in \Omega)\left(\frac{\sqrt{(\Im[m_d]+\Lambda)(\sqrt{\kappa+\eta}+\Lambda)}}{N\eta}+\frac{\sqrt{(\Im[m_d]+\Lambda) t}}{(d-1)^{\ell/4}N\eta}\right)^{2p}\right].
\end{split}\end{align}
Thus a Markov inequality tells us the following holds with overwhelmingly high probability:
\begin{align}\begin{split}\label{e:Q-YQ3}
    \bm1(\GG\in \Omega)|Q_t-Y_t|
    &\lesssim N^{2\fo}\left(\frac{\sqrt{(\Im[m_d]+\Lambda)(\sqrt{\kappa+\eta}+\Lambda)}}{N\eta}+\frac{\sqrt{(\Im[m_d]+\Lambda) t}}{(d-1)^{\ell/4}N\eta}\right).
\end{split}\end{align}

In general, $\Phi(z)\lesssim N^{2\fc}(\Im[\md(z_t)]+\Lambda) /N\eta$, so, in the same way, \eqref{e:Q-YQ} gives 
\begin{align}\begin{split}\label{e:m-XQ2}
   &\bm1(\GG\in \Omega)|Q_t-Y_t|
    \lesssim N^{2\fc}\left(\frac{\sqrt{(\Im[m_d]+\Lambda)(\sqrt{\kappa+\eta}+\Lambda)}}{N\eta}\right), \\
    &\bm1(\GG\in \Omega)|m_t-X_t|
    \lesssim  N^{2\fc}\left(\frac{\sqrt{\Im[m_d]+\Lambda}}{N\eta}\right).
\end{split}\end{align}

In the following we prove \eqref{e:bb1}, \eqref{e:bb2} and \eqref{e:bb3}. The claims \eqref{e:weakbb1} and \eqref{e:weakbb2} can be proven in the same way by using \eqref{e:m-XQ2} as input, so we omit their proofs.

We start with \eqref{e:bb1}. We recall from \eqref{eq:square_root_behave_copy} that $\Im[m_d]\lesssim \sqrt{\kappa+\eta}$. This,  \eqref{e:Q-YQ3}, and \eqref{e:m-XQ2} simplify to the following statement: conditioned on $\cG\in \Omega$, with overwhelmingly high probability the following holds:
\begin{align}\label{e:diyibuaa}
     &|Q_t-Y_t|\lesssim\frac{N^{2\fo}(\sqrt{\kappa+\eta}+\Lambda+(d-1)^{-\ell/2}t)}{N\eta},\quad |m_t-X_t|
    \lesssim \frac{N^{2\fc}(\sqrt{\kappa+\eta}+\Lambda)^{1/2}}{N\eta}.
\end{align}

If $N\eta\sqrt{\kappa+\eta}\geq N^{6\fo}$, then $(d-1)^{-\ell/2}t\ll N^{-1/3}\leq \sqrt{\kappa+\eta}\leq N^{-\fg/2}$. Also, our assumption  $\Lambda\leq \sqrt{\kappa+\eta}/N^\fo$ verifies the assumption in \Cref{p:stable} with $\delta_1=N^{2\fo}\sqrt{\kappa+\eta}/N\eta\ll \kappa+\eta$ and $\delta_2=N^{2\fc}(\kappa+\eta)^{1/4}/N\eta$, and we have
\begin{align*}
\delta_2\lesssim \frac{N^{2\fc}(\kappa+\eta)^{1/4}}{N\eta}\lesssim \frac{1}{N\eta}, \quad \ell^5(t\delta_2+t^3)\lesssim \ell^5(N^{-1/3+\ft}\delta_2+N^{-1+3\ft})\ll \delta_1.
\end{align*}
Then \Cref{p:stable} implies 
\begin{align}\label{e:Qm_bound1}
     |Q_t-\msc|, |m_t-\md|\lesssim \delta_2+\frac{\delta_1+\ell^5(t\delta_2+t^3)}{\sqrt{\kappa+\eta}}
     \lesssim \frac{1}{N\eta}+\frac{\delta_1}{\sqrt{\kappa+\eta}}\lesssim \frac{N^{2\fo}}{N\eta}.
\end{align}

Next we show \eqref{e:bb3}. If $N\eta\sqrt{\kappa+\eta}\leq N^{8\fo}$, then $\eta\leq N^{-2/3+6\fo}$, and $\sqrt{\kappa+\eta}\leq N^{8\fo}/N\eta$. Moreover, our assumption gives $\Lambda\leq N^{10\fo}/N\eta$.  The assumptions in \Cref{p:stable} hold if we take 
$\delta_1=N^{16\fo}/(N\eta)^2$ and 
$\delta_2=N^{2\fc+6\fo}/(N\eta)^{3/2}$, and notice
\begin{align*}
    \delta_1=N^{16\fo}/(N\eta)^2\geq {\kappa+\eta},\quad \delta_2\leq 1/N\eta, \quad \ell^5(t\delta_2+t^3)\lesssim \ell^5(N^{-1/3+\ft}\delta_2+N^{-1+3\ft})\ll \delta_1.
\end{align*}
Then \Cref{p:stable} implies 
\begin{align}\label{e:Qm_bound2}
     |Q_t-\msc|, |m_t-\md|\lesssim \delta_2+\frac{\delta_1+\ell^5(t\delta_2+t^3)}{\sqrt{\kappa+\eta}}
     \lesssim \frac{1}{N\eta}+\sqrt{\delta_1}\lesssim \frac{N^{8\fo}}{N\eta}.
\end{align}

Finally, we show \eqref{e:bb2}. We recall that  $|\Re[z]|\geq E_t$, $N\eta\sqrt{\kappa+\eta}\geq N^{6\fo}$, and $\Lambda\leq \sqrt{\kappa+\eta}/N^\fo $. 
In this case $\Im[m_d]\asymp\eta/\sqrt{\kappa+\eta}$, $(d-1)^{-\ell/2}t\ll N^{-1/3}\leq \sqrt{\kappa+\eta}$ and  we can simplify \eqref{e:Q-YQ3} and \eqref{e:m-XQ2} to
\begin{align*} 
|Q_t-Y_t|\lesssim \frac{N^{2\fo}\sqrt{\eta+\sqrt{\kappa+\eta}\Lambda}}{N\eta}=:\delta_1,\quad
     |m_t-X_t|
    \lesssim\frac{N^{2\fc}\sqrt{\eta+\sqrt{\kappa+\eta}\Lambda}}{(\kappa+\eta)^{1/4}N\eta}=:\delta_2.
\end{align*}
Moreover, we have 
\begin{align*}
\delta_1\ll \kappa+\eta, \quad \delta_2\lesssim\delta_1/\sqrt{\kappa+\eta}, \quad \ell^5(t\delta_2+t^3)\lesssim \ell^5(N^{-1/3+\ft}\delta_2+N^{-1+3\ft})\ll \delta_1,
\end{align*}
and \Cref{p:stable} implies
\begin{align}\label{e:outS0}
     |Q_t-\msc|, |m_t-\md|\lesssim \delta_2+\frac{\delta_1+\ell^5(t\delta_2+t^3)}{\sqrt{\kappa+\eta}}
     \lesssim \frac{\delta_1}{\sqrt{\kappa+\eta}}\lesssim \frac{N^{2\fo}\sqrt{\eta+\sqrt{\kappa+\eta}\Lambda}}{N\eta\sqrt{\kappa+\eta}}.
\end{align}
We remark that using $N\eta\sqrt{\kappa+\eta}\geq N^{6\fo}$, the right-hand side of \eqref{e:outS0} is smaller than ${\sqrt{\kappa+\eta}}/{N^\fo}$.
We can take the new $\Lambda$ as the right-hand side of \eqref{e:outS0}, by iterating \eqref{e:outS0}, we get
\begin{align}\label{e:outS}
 \bm1(\GG\in \Omega)|Q_t-\msc|,  \bm1(\GG\in \Omega)|m_t-m_d|\lesssim  \frac{N^{2\fo}}{\sqrt{\kappa+\eta}}\left(\frac{1}{N\sqrt{\eta}}+\frac{1}{(N\eta)^2}\right).
\end{align}
This finishes the proof of \eqref{e:bb2}.
\end{proof}

\begin{proof}[Proof of  \Cref{c:rigidity}]
We will only prove \eqref{e:mbond} and \eqref{e:equation_est}, the other statements can be proven in the same way by using \Cref{p:iterate} as input, so we omit their proofs. 
We take a lattice grid
\begin{align*}
    \mathbf L=\{E+\ri\eta\in \mathbf D: E\in N^{-2}\bZ,\eta\in N^{-2}\bZ \}.
\end{align*}
For $\cG\in \Omega$ and $z\in \mathbf D$, $m_t(z), Q_t(z), \msc(z_t), \md(z_t)$ are all Lipschitz with Lipschitz constant at most $\OO(N)$.
Thus, if we can show that conditioned on $\cG\in\Omega$, with overwhelmingly high probability, for any $z\in \mathbf L$, 
\begin{align}\label{e:mbondcopy}
\bm1(\cG\in \Omega)|Q_t(z)-\msc(z_t)|, \bm1(\cG\in \Omega)|m_t(z)-m_d(z_t)|\lesssim \frac{N^{8\fc}}{N\eta},
\end{align}
 then \eqref{e:mbond} follows.

First for $z\in \mathbf L$ with $\Im[z]\asymp N^{-\fo}$, \eqref{eq:local_law} implies that
\begin{align*}
   \bm1(\cG\in \Omega)|m_t(z)-m_d(z_t)|,\quad \bm1(\cG\in \Omega)|Q_t(z)-   \msc(z_t)|\lesssim \frac{1}{N^{\fb}}\leq \frac{\sqrt{\kappa+\eta}}{N^{\fo}}.
\end{align*}
This verifies the assumptions of  \eqref{e:weakbb1} and \eqref{e:weakbb2}. We conclude that \eqref{e:mbondcopy} holds for $z\in \mathbf L$ with $\Im[z]\asymp N^{-\fo}$. Then inductively we can show that if \eqref{e:mbondcopy} holds for $z\in \mathbf L$ with $\Im[z]\geq (k+1)/N^2$, then it also holds for $z\in \mathbf L$ with $\Im[z]\geq k/N^2$. More precisely, for $z=E+\ri \eta\in \mathbf L$ with $\Im[z]\geq k/N^2$, we denote $z'=z+(\ri/N^2)\in \mathbf L$. By the inductive hypothesis and the fact that the Lipschitz constants of $m(z), \md(z_t)$ are at most $\OO(N)$ 
\begin{align*}
    \bm1(\cG\in \Omega)|m_t(z)-m_d(z_t)|
    =\bm1(\cG\in \Omega)|m_t(z')-m_d(z_t')| +\OO(1/N)
    \lesssim \frac{N^{8\fc}}{N\eta}\ll \frac{N^{10\fc}}{N\eta}.
\end{align*}
Thus \eqref{e:weakbb1} and \eqref{e:weakbb2} imply that \eqref{e:mbondcopy} holds for $z$. Finally the claim \eqref{e:equation_est} follows from plugging \eqref{e:mbondcopy} into the second statement of \eqref{e:diyibuaa} (taking $\Lambda=N^{8\fc}/N\eta$).

\end{proof}

\begin{proof}[Proof of  \Cref{thm:eigrigidity}]
Optimal rigidity \eqref{e:optimal_rigidity} follows from a standard argument using the Stieltjes transform estimates \Cref{c:rigidity} as an input, see \cite[Section 11]{erdHos2017dynamical}.  For didactic purposes, we show here how to show that, conditioned on $\Omega$, $\lambda_2\leq 2+N^{-2/3+10\fo}$ holds with overwhelmingly high probability. It has been proven in {\cite[Theorem 1.3]{huang2024spectrum}} that conditioned on $\Omega$, with overwhelmingly high probability, 
\begin{align}\label{e:la2bound}
    \lambda_2\leq 2+N^{-\oo(1)}.
\end{align}
In the following we show that with overwhelmingly high probability, there is no eigenvalue on the interval $[2+N^{-2/3+10\fo}, 2+N^{-\fg}]$. It, together with \eqref{e:la2bound}, implies that $\lambda_2\leq 2+N^{-2/3+10\fo}$.

In the following we take time $t=0$, then $z_0=z$ and $E_0=2$. We also denote the Stieltjes transform $m_N(z)=m_0(z)$.

We take $z=2+\kappa+\ri\eta$, with $ N^{-2/3+10\fo}\leq\kappa\leq N^{-\fg}$ and $\eta=N^{6\fo}/(N\sqrt\kappa)\leq N^{-8\fo}\kappa$. With this choice, one can check that $N\eta\sqrt{\kappa+\eta}\geq N^{6\fo}$, $z\in {\mathbf  D}$ and $|z-E_t|\lesssim N^{-\fg}$. We recall from \eqref{eq:square_root_behave_copy}, $\Im[\md(z)]\asymp \eta/\sqrt{\kappa+\eta}\ll 1/{N\eta}$.
In  this regime,   \eqref{e:mtimproved} implies that, conditioned on $\Omega$, the following holds with overwhelmingly high probability,
\begin{align*} \begin{split}
\left|m_N(z) - \md(z)\right|
&\leq N^{2\fo}\left(\frac{1}{N\sqrt{\eta(\kappa+\eta)}}+\frac{1}{(N\eta)^2\sqrt{\kappa+\eta}}\right)\lesssim \frac{N^{2\fo}}{N^{4\fo} (N\eta)}\ll \frac{1}{N\eta}.
\end{split}\end{align*}
Thus it follows that 
\begin{align*}
    \Im[m_N(z)]\leq |m_N(z)+\md(z)|+\Im[\md(z)]\ll \frac{1}{N\eta}.
\end{align*}
If there is an eigenvalue on the interval $\lambda_i\in [2+\kappa-\eta, 2+\kappa+\eta]$, then 
\begin{align*}
    \Im[m_N(z)]
    =\frac1N\sum_{i=1}^N \Im\left[\frac{1}{\lambda_i-z}\right]\geq\frac{1}{N}\frac{\eta}{|\lambda_i-(2+\kappa)|^2+\eta^2}\geq \frac{1}{2N\eta},
\end{align*}
which is impossible. It follows that conditioned on $\cG\in \Omega$, with overwhelmingly high probability, it is the case that $\lambda_2\leq 2+N^{-2/3+10\fo}$.
\end{proof}

\section{Proofs of \Cref{l:erbu} and \Cref{l:second_term}}\label{s:tree_computation}

\begin{proof}[Proof of \Cref{l:erbu}]

 We consider the cases in \Cref{l:erbu} one by one. 
 Let $f_0'$ denote the number of distinct values in $\bm\theta$. We recall from \eqref{e:S'bound3_main} that
 \begin{align}\label{e:S'bound3}
     \frac{(d-1)^{f_0'\ell+3\ell\sum_{j=1}^{p-1}h_{j1}}}{(d-1)^{k_4 \ell/2}}\prod_{j=1}^{p-1}|\cW_j'|\lesssim (|Q_t-Y_t|+(d-1)^{8\ell}\Upsilon \Phi)^{p-1}.
 \end{align}
By using \eqref{e:S'bound3} as input, the proof of  \Cref{l:erbu} is parallel to those of the statements \eqref{e:huanG-Q}, \eqref{e:huanG} and \eqref{e:huanGS}. 

 Assume $\widehat R_{\bfi^+}$ contains a term $A_\al=( G_{c_\al c_\al}^{(b_\al)}-Q_t)$ with $\al\in \mathsf I_{\rm single}$. Let $\widehat R_{\bfi^+}=( G_{c_\al c_\al}^{(b_\al)}-Q_t)\widehat R_{\bfi^+}'$. 
    We let  the forest $\widehat \cF$, the indicator function $I_{c_\al}$ as in \eqref{e:Idecompose}, and recall the estimate \eqref{e:huanGcb00}
\begin{align*}\begin{split}
     &\phantom{{}={}}\frac{\bm1(\cG\in \Omega)}{ Z_{\cF^+}}\sum_{b_\al, c_\alpha\in\qq{N}} I(\cF^+,\cG)(G_{c_\al c_\al}^{(b_\al)}-Q_t)\lesssim\frac{\bm1(\cG\in \Omega)}{N^{1-3\fc/2} Z_{\cF^+}}\sum_{b_\al, c_\alpha\in\qq{N}} I(\cF^+,\cG).
\end{split}\end{align*}
Then we can bound \eqref{e:oneterm1} as 
\begin{align*}\begin{split}
  \eqref{e:oneterm1} &\lesssim \frac{(d-1)^{f_0\ell}}{(d-1)^{(k_1 +(k_2+k_3+k_4)/2)\ell}}   \frac{(d-1)^{(6\widehat h+3\sum_{j=1}^{p-1} h_{j1})\ell}}{N^{1-3\fc/2}  Z_{\cF^+}}\sum_{ \bfi^+}\bE\left[I(\widehat\cF,\cG)\bm1(\cG\in \Omega)|\widehat R_{\bfi^+}|\right]\\
   &\lesssim  \frac{(d-1)^{(f_0-f_0')\ell}}{(d-1)^{(k_1 +(k_2+k_3)/2)\ell}}  \frac{(d-1)^{6\widehat h\ell}}{N^{1-3\fc/2}}\bE\left[\bm1(\cG\in \Omega) N^{-\widehat h\fb} (|Q_t-Y_t|+(d-1)^{8\ell}\Upsilon\Phi)^{p-1}\right]\\
    &\lesssim \frac{(d-1)^{(6\widehat h+(k_2+k_3)/2)\ell} N^{-\widehat h\fb} }{N^{1-3\fc/2}}\bE\left[\bm1(\cG\in \Omega)(|Q_t-Y_t|+(d-1)^{8\ell}\Upsilon\Phi)^{p-1}\right]\lesssim N^{-\fb/2}\bE[ \Psi_p],
    \end{split}\end{align*}
    where we used \eqref{e:S'bound3}, $\widehat R'_{\bfi^+}\in \Adm(\widehat h,\bmh,\cF^+,\cG)$ and \eqref{e:Bsmall} for the second statement, and $f_0-f_0'\leq k_1+k_2+k_3$ for the third statement.


If $\widehat R_{\bfi^+}$ contains $   G^{\circ}_{b_\beta s}$ or $G^{\circ}_{c_\beta s}$ for $s\in \cK\setminus\{o,i\}$, and the index $\beta\in \mathsf I_{\rm single}$, then the  same argument gives the desired bound.

We now consider if $\widehat R_{\bfi^+}$ contains a term in $ \{G^{\circ}_{b_\theta w},G^{\circ}_{c_\theta w},  L_{b_\theta w}, L_{c_\theta w}\}$ with $\theta\in \mathsf I_{\rm single}$ and $w$ belonging to a special edge $\{u_j,v_j\}$.
We assume that $\widehat R_{\bfi^+}$ contains a term $ G^{\circ}_{c_\theta w}$; other cases can be proven in the same way, so we omit their proofs. Then there is some $j'\in\qq{p-1}$, $h_{j'0}=2$, $h_{j'1}+h_{j'2}\geq 2$ even, 
such that 
\begin{align}\label{e:Sj'}
    \cW_{j'}'=\frac{\{1-\del_1 Y_\ell, \del_2 Y_\ell\}}{Nd}\sum_{u_j\sim v_j\in \qq{N}} G^{\circ}_{c_\theta w}
  R''_{h_{j'1}-1, h_{j'2}}.
\end{align}

There are two cases: either $h_{j'1}+h_{j'2}\geq 4$ or $h_{j'1}+h_{j'2}=2$. We discuss them separately.

\noindent\textbf{Case 1.} If $h_{j'1}+h_{j'2}\geq 4$, we consider  the forest $\widehat \cF$ and the indicator function $I_{c_\theta}$ as in \eqref{e:Idecompose} (taking $\al$ to be $\theta$). Then 
\begin{align}\begin{split}\label{e:huanGcb}
     &\phantom{{}={}}\left|\frac{(d-1)^{(f_0'+3\sum_{j=1}^{p-1}h_{j1})\ell}}{(d-1)^{k_4 \ell/2}}\frac{\bm1(\cG\in \Omega)}{ Z_{\cF^+}}\sum_{b_\theta, c_\theta\in\qq{N}} I(\cF^+,\cG)\prod_{j=1}^{p-1}\cW'_j\right|\\
     &\lesssim 
     \frac{(d-1)^{(f_0'+3\sum_{j=1}^{p-1}h_{j1})\ell}}{(d-1)^{k_4 \ell/2}} \frac{\bm1(\cG\in \Omega)I(\widehat \cF,\cG)}{Z_{\cF^+}}\left|\sum_{b_\theta\sim c_\theta}G^\circ_{c_\theta w}I_{c_\theta} \right|\frac{\Upsilon}{Nd}\sum_{u_j\sim v_j} |R''_{h_{j'1}-1, h_{j'2}}|\prod_{j\neq j'}|\cW'_{j}|\\
      &\lesssim  \frac{\bm1(\cG\in \Omega)I(\widehat \cF,\cG)}{Z_{\cF^+}}\frac{Nd}{N^{1-3\fc/2}} (d-1)^{10\ell} \Upsilon \Phi (|Q_t-Y_t|+(d-1)^{8\ell} \Upsilon \Phi )^{p-2}\\
      &\lesssim   \frac{\bm1(\cG\in \Omega)}{N^{1-3\fc/2}Z_{\cF^+}}(d-1)^{10\ell} \Upsilon \Phi (|Q_t-Y_t|+(d-1)^{8\ell} \Upsilon \Phi )^{p-2}\sum_{b_\theta\sim c_\theta\in\qq{N}} I_{c_\theta} I(\widehat\cF,\cG) \\
      &\lesssim \frac{(d-1)^{2\ell} } {N^{1-3\fc/2} Z_{\cF^+}}\sum_{b_\theta, c_\theta} \bm1(\cG\in \Omega)I(\cF^+,\cG)(|Q_t-Y_t|+(d-1)^{8\ell} \Upsilon \Phi )^{p-1},
\end{split}\end{align}
where the second statement follows from the decomposition \eqref{e:Sj'}; for the third statement we used \eqref{e:single_index_sum}, and \eqref{e:S'bound2} to bound the summation of $R''_{h_{j1}-1 h_{j2}}$, and $\cW'_{j'}$.
The last two statements follow from rearranging. Thanks to \eqref{e:huanGcb},  we can bound \eqref{e:oneterm1} as
\begin{align}\begin{split}\label{e:hala1}
   &\phantom{{}={}}\frac{(d-1)^{f_0\ell}}{(d-1)^{(k_1 +(k_2+k_3+k_4)/2)\ell}}   \frac{(d-1)^{(6\widehat h+3\sum_{j=1}^{p-1} h_{j1})\ell}}{Z_{\cF^+}}\sum_{ \bfi^+}\bE\left[I(\cF^+,\cG)\bm1(\cG\in \Omega)|\widehat R_{\bfi^+}|\right]\\
   &\lesssim\frac{(d-1)^{(f_0-f_0'+6\widehat h+2)\ell} N^{-(\widehat h+1)\fb}}{(d-1)^{(k_1 +(k_2+k_3)/2)\ell}N^{1-3\fc/2} Z_{ \cF^+}}\sum_{\bfi^+}\bE\left[\bm1(\cG\in \Omega)I(\cF^+,\cG) (|Q_t-Y_t|+(d-1)^{8\ell} \Upsilon \Phi )^{p-1}\right]\\
   &\lesssim\frac{(d-1)^{(6\widehat h+2+(k_2+k_3)/2)\ell }N^{-(\widehat h+1)\fb} }{N^{1-3\fc/2} }\bE\left[\bm1(\cG\in \Omega) (|Q_t-Y_t|+(d-1)^{8\ell} \Upsilon \Phi )^{p-1}\right]\lesssim \bE[N^{-\fb} \Psi_p],
\end{split}\end{align}
where to get the third statement, we used that $f_0-f_0'\leq k_1+k_2+k_3$; to get the last statement we used $\widehat h\geq k_1+k_2+k_3$.


 

\noindent 
\textbf{Case 2.}  If $h_{{j'}1}+h_{{j'}2}=2$, then
    \begin{align*}\begin{split}
    &R''_{h_{{j'}1}-1, h_{{j'}2}}
    \in \{ G^{\circ}_{sw'}, L_{sw'}, (\Av G^{\circ})_{o'w'},(\Av L)_{o'w'} \}_{w'\in\{u_{j'}, v_{j'}\},s\in \cK^+}\times U^{(u_{j'},v_{j'})},
    \end{split}\end{align*}
    where $U^{(u_{j'},v_{j'})}$ is a product of $G_{u_{j'} v_{j'}}, 1/G_{u_{j'} u_{j'}}$. 
    Without loss of generality, we assume that $ R'_{h_{{j'}1}-1, h_{{j'}2}}$ is of the form $ G^{\circ}_{sw'} U^{(u_{j'},v_{j'})}$ with $s\in \{b,c\}\in\cC^+$. The other cases can be proven in the same way, so we omit their proofs. Let $\widehat \cF=\{\{b,c\}, \{b_\theta, c_\theta\}\}$ consist of two unused core edges, and define
    \begin{align}\begin{split}\label{e:Rsbound}
 &\phantom{{}={}}V^{(b,c)}:=\frac{(d-1)^{f_0\ell} (d-1)^{(6\widehat h+3\sum_{j\neq j'} h_{j1})\ell}}{(d-1)^{(k_1 +(k_2+k_3+k_4)/2)\ell}Z_{\cF^+}}\sum_{ \bfi/\{b,c,b_\theta, c_\theta\}}\frac{(Nd) I(\cF^+,\cG)}{Z_{\cF^+}}R_{\widehat h+1}\prod_{j\neq j'}\cW'_{j}\\
  &\lesssim\frac{(d-1)^{(f_0-f_0'+2)\ell} (d-1)^{6\widehat h\ell}}{(d-1)^{(k_1 +(k_2+k_3)/2)\ell}}N^{-(\widehat h+1)\fb}(|Q_t-Y_t|+(d-1)^{8\ell}\Upsilon \Phi)^{p-2}\sum_{ \bfi/\{b,c,b_\theta, c_\theta\}}\frac{(Nd) I(\cF^+,\cG)}{Z_{\cF^+}}\\
  &\lesssim \frac{(|Q_t-Y_t|+(d-1)^{8\ell}\Upsilon \Phi)^{p-2}}{N}I(\widehat \cF,\cG),
\end{split}\end{align}
where the second statement follows from \eqref{e:S'bound3} by noticing that $\cW_{j'}$ contributes at most $2$ to $f_0'$; in the third statement we used that $f_0-f_0'\leq k_1+k_2+k_3$, $\widehat h\geq k_1+k_2+k_3$, and
$I(\cF^+,\cG)$ contains the factor $I(\wh\cF,\cG)$.  

By plugging in \eqref{e:Rsbound} and noticing $h_{j'1}=2$,  we can rewrite \eqref{e:oneterm1} as  
    \begin{align}\begin{split}\label{e:yigeJtheta}
      \eqref{e:oneterm1}
      =\frac{ (d-1)^{6\ell}}{Nd} \sum_{b,c, b_\theta, c_\theta\in\qq{N}}\bE\left[ \bm1(\cG\in \Omega)I(\wh \cF,\cG)V^{(b,c)} \cW_{j'}\right].
    \end{split}\end{align}
For $\cG\in \Omega$, the same as in \eqref{e:Wjsum}, the summation in \eqref{e:yigeJtheta} can be bounded by \eqref{e:single_index_term2}
\begin{align}\begin{split}\label{e:hala2}
 &\phantom{{}={}}\frac{ (d-1)^{6\ell}}{Nd} \sum_{b,c, b_\theta, c_\theta\in\qq{N}}I(\wh \cF,\cG)V^{(b,c)}\cW_{j'}\lesssim  N^{-\fb}\Psi_p.
\end{split}\end{align}

The two cases \eqref{e:hala1} and \eqref{e:hala2} both give that $|\eqref{e:oneterm1}|\lesssim N^{-\fb}\bE[\Psi_p]$ as claimed.

The four cases in \eqref{e:final_replace} can be proven in the same way, so we will only prove the last case.  We let  the forest $\widehat \cF$, the indicator function $I_{c_\gamma c_{\gamma'}}$ (by taking $(\al, \beta)$ as $(\gamma,\gamma')$) as in \eqref{e:Idecompose2}, and recall the estimate \eqref{e:huanGcc}
\begin{align}\begin{split}\label{e:huanGcccopy}
   &\phantom{{}={}}\frac{\bm1(\cG\in \Omega)}{Z_{\cF^+}}\sum_{b_\gamma,c_\gamma\atop b_{\gamma'},c_{\gamma'}}I(\cF^+,\cG) G_{c_\gamma c_{\gamma'}}^{(b_\gamma b_{\gamma'})}I_{c_\gamma c_{\gamma'}} \\
      &=\frac{\bm1(\cG\in \Omega)}{Z_{ \cF^+}}\sum_{b_\gamma, c_\gamma\atop b_{\gamma'}, c_{\gamma'}}I(\cF^+,\cG)\left(\frac{G_{b_\gamma b_{\gamma'}}(G_{c_\gamma c_\gamma}^{(b_\gamma)}-Q_t) (G_{c_{\gamma'} c_{\gamma'}}^{(b_{\gamma'})}-Q_t)}{d-1}+\OO(\cE_{\gamma\gamma'}) \right),
\end{split}\end{align}     
   where 
   \begin{align*}
       |\cE_{\gamma \gamma'}|\lesssim  N^{-\fb/2}  |G_{b_\gamma b_{\gamma'}}|\left(\sum_{x\sim b_\gamma, x\neq c_\gamma}|G_{c_\gamma x}^{(b_\gamma)}|+\sum_{x\sim b_{\gamma'}, x\neq c_{\gamma'}}|G_{c_{\gamma'} x}^{(b_{\gamma'})}|\right)+N^{-\fb/2} \Phi.
   \end{align*}

By plugging \eqref{e:huanGcccopy} back into \eqref{e:oneterm1}, the leading term in \eqref{e:huanGcccopy} gives first term on the right-hand side of \eqref{e:oneterm2}, where $R_{\bfi^+}'$ is obtained from $\widehat R_{\bfi^+}$ by replacing $G^{(b_\gamma b_\gamma')}_{c_\gamma c_{\gamma'}}$ by $G_{b_{\gamma} b_{{\gamma'}}}(G_{c_{\gamma} c_{\gamma}}^{(b_{\gamma})}-Q_t) (G_{c_{{\gamma'}} c_{{\gamma'}}}^{(b_{{\gamma'}})}-Q_t)$.
Thanks to \eqref{e:sameasdisconnect} and \eqref{e:bPi1}, the error term in \eqref{e:huanGcccopy} is negligible
\begin{align*}\begin{split}
       &\phantom{{}={}}\frac{(d-1)^{(f_0+6\widehat h+3\sum_{j=1}^{p-1} h_{j1})\ell}}{(d-1)^{(k_1 +(k_2+k_3+k_4)/2)\ell}Z_{\cF^+}}\sum_{\bfi^+}\bE\left[I(\cG\in \Omega)\bm1(\cF^+,\cG)||\cE_{\gamma\gamma'}
|N^{-\widehat h\fb} \prod_{j}|\cW_j'|\right]\\
&\lesssim 
    \frac{(d-1)^{(f_0-f_0'+6\widehat h)\ell} N^{-\widehat h\fb}}{(d-1)^{(k_1+(k_2+k_3)/2)\ell}}\bE\left[\bm1(\cG\in \Omega) |\cE_{\gamma\gamma'}|(|Q_t-Y_t|+\Upsilon(d-1)^{8\ell}\Phi)^{p-1}\right]
   \lesssim  N^{-\fb/4}\bE[\Psi_p].
\end{split}\end{align*}

We can repeat the above procedures for each index in $\mathsf I_{\rm single}$. Either $\eqref{e:oneterm1}=\OO(N^{-\fb/4}\bE[\Psi_p])$, or up to an error of size $\OO(N^{-\fb/4}\bE[\Psi_p])$, we can write \eqref{e:oneterm1} as 
\begin{align}\begin{split}\label{e:oneterm4}
\frac{(d-1)^{(6(\widehat h+f_1)+3\sum_{j=1}^{p-1} h_{j1})\ell}}{(d-1)^{(k_1 +(k_2+k_3+k_4)/2)\ell}} \frac{(d-1)^{f_0\ell} }{(d-1)^{6f_1\ell}}\sum_{\bfi^+}R'_{\bfi^+}=\frac{(d-1)^{(6(\widehat h+f_1)+3\sum_{j=1}^{p-1} h_{j1})\ell}}{(d-1)^{\fq^+\ell/2}}
\sum_{\bfi^+}R'_{\bfi^+},
\end{split}\end{align}
where $R'_{\bfi^+}$ contains $f_1=|\mathsf I_{\rm single}|$ additional factors compared to $\widehat R_{\bfi^+}$. Thus $R'_{\bfi^+}\in \Adm(\widehat h+f_1, \bmh,\cF^+,\cG)$. The claim \eqref{e:oneterm2} follows from rearranging \eqref{e:oneterm4}.

\end{proof}

\begin{proof}[Proof of \Cref{l:second_term}]
Without loss of generality, we assume that  $|z-E_t|\leq N^{-\fg}$. The other case $|z+E_t|\leq N^{-\fg}$ can be established in the same way, so we omit its proof. 

For the first statement in \eqref{e:schur_two1}, given $\alpha$, there are two cases for $\beta$: 
if $\beta\in \sfA_i$ and $\dist(l_\beta, i)=\ell+1$, there are $(d-1)^{\ell+1}-1$ such $\beta$, and $L_{l_\beta l_\beta}^{(i)}= \md(z_t) (1-\msc^{2\ell+2}(z_t)/(d-1)^{\ell+1})$;
if $\beta\notin \sfA_i$ and $\dist(l_\beta, i)=\ell-1$, there are $(d-1)^{\ell}$ such $\beta$, and
$
L_{l_\beta l_\beta}^{(i)}= \md(z_t) (1-\msc^{2\ell-2}(z_t)/(d-1)^{\ell-1}).
$
We get
\begin{align}\begin{split}\label{e:term2}
&\phantom{{}={}}\sum_{\al \in \sfA_i, \beta\in\qq{\mu}\atop\al\neq \beta}\frac{\msc^{2\ell}(z_t)L_{l_\beta l_\beta}^{(i)} }{(d-1)^{\ell+2}}
= 
\frac{\msc^{2\ell}(z_t)\md(z_t)}{(d-1)^{\ell+2}} \left(1-\frac{\msc^{2\ell+2}(z_t)}{(d-1)^{\ell+1}}\right)\times (d-1)^{\ell+1} ((d-1)^{\ell+1}-1)\\
&+\frac{\msc^{2\ell}(z_t)\md(z_t)}{(d-1)^{\ell+2}}\left(1-\frac{\msc^{2\ell-2}(z_t)}{(d-1)^{\ell-1}}\right)  (d-1)^{\ell} (d-1)^{\ell+1} \\
&=
- \frac{\msc^{2\ell}(z_t)\md(z_t)}{d-1} \left((1-\frac{\msc^{2\ell+2}(z_t)}{(d-1)^{\ell+1}})((d-1)^{\ell+1}-1)+(1-\frac{\msc^{2\ell-2}(z_t)}{(d-1)^{\ell-1}})(d-1)^{\ell}\right)\\
&=
- \frac{\msc^{2\ell}(z_t)\md(z_t)}{d-1} \left(d(d-1)^\ell-1-\msc^{2\ell+2}-(d-1)\msc^{2\ell-2}(z_t) +\frac{\msc^{2\ell+2}(z_t)}{(d-1)^{\ell+1}}\right)\\
&= 
- \frac{1}{d-2} \left(d(d-1)^{\ell}-1-1-(d-1)\right)+\OO\left((d-1)^{-\ell}\right)\\
&=-\left(\frac{d(d-1)^{\ell}}{d-2} -\frac{d+1}{(d-2)} \right) +\OO\left((d-1)^{-\ell}\right),
\end{split}\end{align}
where in the last two lines we used that 
$\msc(z_t)=-1+\OO(\sqrt{|z_t-2|})$ and $\md(z_t)=-(d-1)/(d-2)+\OO(\sqrt{|z_t-2|})$ from \eqref{e:mscmd}, and $|z_t-2|\lesssim N^{-2\fg}$ from \eqref{e:z-2}.

For the second statement in \eqref{e:schur_two1}, we notice that for a given $\al\in \sfA_i$, $L^{(i)}_{l_\al l_\beta}\neq 0$ only if $\beta\in \sfA_i$. Moreover, there are $(d-2)(d-1)^r$ values of $\beta$ such that $\dist(l_\al, l_\beta)=2r$ and ${\rm anc}(l_\al, l_\beta)=\ell-r$, for $0\leq r\leq \ell$. Thus, 
\begin{align*}
&\phantom{{}={}}\sum_{\al \in \sfA_i, \beta\in\qq{\mu}\atop\al\neq \beta}\frac{\msc^{2\ell}(z_t)L_{l_\al l_\beta}^{(i)} }{(d-1)^{\ell+2}}\\
&= \frac{\msc^{2\ell}(z_t)}{(d-1)^{\ell+2}}\sum_{\al \in \sfA_i, \beta\in\qq{\mu}\atop\al\neq \beta}\md(z_t)\left(1-\left(-\frac{\msc(z_t)}{\sqrt{d-1}}\right)^{2+2{\rm anc}( l_\al,  l_\beta)}\right)\left(-\frac{\msc(z_t)}{\sqrt{d-1}}\right)^{\dist( l_\al  l_\beta)}\\
&=\frac{\md(z_t)\msc^{2\ell}(z_t)}{(d-1)^{\ell+2}}(d-1)^{\ell+1}\sum_{r=0}^\ell \left(1-\left(-\frac{\msc(z_t)}{\sqrt{d-1}}\right)^{2+2(\ell-r)}\right)\left(-\frac{\msc(z_t)}{\sqrt{d-1}}\right)^{2r}(d-2)(d-1)^{r}\\
&=\frac{(d-2)\md(z_t)\msc^{2\ell}(z_t)}{d-1}\left(\sum_{r=0}^\ell \msc^{2r}(z_t) -\msc^{2+2\ell}(z_t)\frac{1-(d-1)^{-\ell-1}}{d-2} \right)\\
&=-\left(\ell+1-\frac{1}{d-2}\right)
+\OO\left((d-1)^{-\ell}\right).
\end{align*}

For the third statement in \eqref{e:schur_two1}, we notice that for a given $\al\in \sfA_i$, there are $(d-2)(d-1)^r$ values of $\beta$ such that $\dist(l_\al, l_\beta)=2r$, for $0\leq r\leq \ell$, giving
\begin{align*}
&\phantom{{}={}}\sum_{\al\neq \beta\in \sfA_i}
  \frac{2\msc^{2\ell}(z_t)L_{l_\al l_\beta}}{(d-1)^{\ell+2}}= \frac{2\msc^{2\ell}(z_t)}{(d-1)^{\ell+2}} \sum_{\al\neq \beta\in \sfA_i}
   \md(z_t) \left(-\frac{\msc(z_t)}{\sqrt{d-1}}\right)^{\dist(l_\al, l_\beta)}\\
  &=\frac{\md(z_t)\msc^{2\ell}(z_t)}{d-1}\left(\sum_{r=0}^\ell \left(-\frac{\msc(z_t)}{\sqrt{d-1}}\right)^{2r}(d-2)(d-1)^{r}\right)\\
&=\frac{(d-2)\md(z_t)\msc^{2\ell}(z_t)}{d-1}\sum_{r=0}^\ell \msc^{2r}=-2(\ell+1)
+\OO\left((d-1)^{-\ell}\right).
\end{align*}
The last statement of \eqref{e:schur_two1} follows by noticing that there are $(d-1)^{\ell+1}((d-1)^{\ell+1}-1)$ values of $\al\neq \beta\in \sfA_i$.

Finally for \eqref{e:schur_two3}, we similarly have 
\begin{align*}
&\phantom{{}={}}\frac{\msc^{2\ell}(z_t)}{(d-1)^{\ell+2}}\sum_{\al\in \sfA_i}L_{ l_\al  l_\al}^{(i)}
=\frac{\msc^{2\ell}(z_t)}{(d-1)^{\ell+2}}\sum_{\al \in \sfA_i}\md(z_t)\left(1-\left(-\frac{\msc(z_t)}{\sqrt{d-1}}\right)^{2+2{\rm anc}( l_\al,  l_\al)}\right)\\
&=\frac{\md(z_t)\msc^{2\ell}(z_t)}{d-1}\left(1-\frac{\msc^{2+2\ell}(z_t)}{(d-1)^{1+\ell}}\right)
=-\frac{1}{d-2}+\OO((d-1)^{-\ell}).
\end{align*}
\end{proof}

\bibliography{ref}{}
\bibliographystyle{abbrv}

\end{document}